\documentclass{amsart}
\usepackage{amsmath,amssymb,amsthm,bm,bbm}
\usepackage{graphicx,enumerate}
\usepackage{layout}
\usepackage{longtable}

\theoremstyle{plain}
\newtheorem{theorem}{Theorem}[section]
\newtheorem*{theorem-nn}{Theorem}
\newtheorem{lemma}[theorem]{Lemma}
\newtheorem{proposition}[theorem]{Proposition}
\newtheorem*{proposition-nn}{Proposition}
\newtheorem{corollary}[theorem]{Corollary}

\newtheorem{algorithmF}{Algorithm}
\newtheorem*{algorithmKS1}{Algorithm KS1}
\newtheorem*{algorithmKS2}{Algorithm KS2}
\newtheorem*{algorithmFC}{Algorithm FC}
\newtheorem*{algorithmN1T}{Algorithm N1T}
\newtheorem*{algorithmTC}{Algorithm TC}
\newtheorem*{algorithmIndmf}{Algorithm Indmf}

\theoremstyle{definition}
\newtheorem{definition}[theorem]{Definition}
\newtheorem{example}[theorem]{Example}
\newtheorem{remark}[theorem]{Remark}
\newtheorem*{acknowledgments}{Acknowledgments}

\theoremstyle{remark}

\newcommand{\bZ}{\mathbbm{Z}}\newcommand{\bQ}{\mathbbm{Q}}
\newcommand{\bC}{\mathbbm{C}}\newcommand{\bG}{\mathbbm{G}}
\newcommand{\bF}{\mathbbm{F}}
\newcommand{\bR}{\mathbbm{R}}
\newcommand{\cC}{\mathcal{C}}\newcommand{\cD}{\mathcal{D}}
\newcommand{\cH}{\mathcal{H}}\newcommand{\cS}{\mathcal{S}}
\newcommand{\GL}{{\rm GL}}\newcommand{\SL}{{\rm SL}}
\newcounter{sub}
{\begin{list}{(\arabic{sub})}{\usecounter{sub}%
\setlength{\leftmargin}{2em}}}{\end{list}}

\def\image{\mbox{Im }}
\def\ker{\mbox{Ker }}
\def\id{\mbox{id}}
\def\rank{\mbox{rank }}
\def\Hom{\mbox{Hom}}

\title[Rationality problem for algebraic tori]
{Rationality problem for algebraic tori}

\author[A. Hoshi]{Akinari Hoshi}
\address{Department of Mathematics, Niigata University, Niigata 950-2181, Japan}
\email{hoshi@math.sc.niigata-u.ac.jp}

\author[A. Yamasaki]{Aiichi Yamasaki}
\address{Department of Mathematics, Kyoto University, Kyoto 606-8502, Japan}
\email{aiichi.yamasaki@gmail.com}

\thanks{{\it Key words and phrases.} Rationality problem, 
algebraic tori, 
stably rational, retract rational, flabby resolution, Krull-Schmidt theorem, 
Bravais group, Tate cohomology.\\ 
This work was partially supported by JSPS KAKENHI Grant Numbers 22740028, 24540019, 25400027.
Some part of this work was done during the authors' visit to 
National Taiwan University, the National Center for Theoretic Sciences 
(Taipei Office), whose support is gratefully acknowledged.}

\subjclass[2010]{Primary 11E72, 12F20, 13A50, 14E08, 20C10, 20G15.}

\begin{document}
\begin{abstract}
We give the complete stably rational classification
of algebraic tori of dimensions $4$ and $5$ over a field $k$. 
In particular, the stably rational classification of 
norm one tori whose Chevalley modules are of rank $4$ and $5$ is given. 
We show that there exist exactly $487$ (resp. $7$, resp. $216$) stably rational 
(resp. not stably but retract rational, resp. not retract rational) 
algebraic tori of dimension $4$, 
and there exist exactly $3051$ (resp. $25$, resp. $3003$) stably rational 
(resp. not stably but retract rational, resp. not retract rational) 
algebraic tori of dimension $5$. 
We make a procedure to compute a flabby resolution of a $G$-lattice 
effectively by using the computer algebra system GAP. 
Some algorithms may determine whether the flabby class of a $G$-lattice 
is invertible (resp. zero) or not. 
Using the algorithms, we determine all the flabby and coflabby $G$-lattices 
of rank up to $6$ and verify that they are stably permutation. 
We also show that the Krull-Schmidt theorem for $G$-lattices 
holds when the rank $\leq 4$, and fails when the rank is $5$. 
Indeed, there exist exactly $11$ (resp. $131$) $G$-lattices 
of rank $5$  (resp. $6$) which are decomposable into two different ranks. 
Moreover, when the rank is $6$, 
there exist exactly $18$ $G$-lattices which are decomposable into 
the same ranks but the direct summands are not isomorphic. 
We confirm that $H^1(G,F)=0$ for any Bravais group $G$ of dimension $n\leq 6$ 
where $F$ is the flabby class of the corresponding $G$-lattice of rank $n$.  
In particular, $H^1(G,F)=0$ for any maximal finite subgroup $G\leq {\rm GL}(n,\bZ)$ 
where $n\leq 6$. 
As an application of the methods developed, 
some examples of not retract (stably) rational fields over $k$ are given. 
\end{abstract}

\maketitle

\tableofcontents


%
\section{Introduction}\label{seInt}

Let $k$ be a field and $K$ 
be a finitely generated field extension of $k$. 
A field $K$ is called {\it rational over $k$} 
(or {\it $k$-rational} for short) 
if $K$ is purely transcendental over $k$, 
i.e. $K$ is isomorphic to $k(x_1,\ldots,x_n)$, 
the rational function field over $k$ with $n$ variables $x_1,\ldots,x_n$ 
for some integer $n$. 
$K$ is called {\it stably $k$-rational} 
if $K(y_1,\ldots,y_m)$ is $k$-rational for some algebraically 
independent elements $y_1,\ldots,y_m$ over $K$. 
When $k$ is an infinite field, 
$K$ is called {\it retract $k$-rational} 
if there is a $k$-algebra $R$ contained in $K$ such that 
(i) $K$ is the quotient field of $R$, and (ii) 
the identity map $1_R : R\rightarrow R$ factors through a localized 
polynomial ring over $k$, i.e. there is an element $f\in k[x_1,\ldots,x_n]$, 
which is the polynomial ring over $k$, and there are $k$-algebra 
homomorphisms $\varphi : R\rightarrow k[x_1,\ldots,x_n][1/f]$ 
and $\psi : k[x_1,\ldots,x_n][1/f]\rightarrow R$ satisfying 
$\psi\circ\varphi=1_R$ (cf. \cite{Sal84a}). 
$K$ is called {\it $k$-unirational} 
if $k\subset K\subset k(x_1,\ldots,x_n)$ for some integer $n$. 
It is not difficult to see that 
``$k$-rational'' $\Rightarrow$ ``stably $k$-rational'' $\Rightarrow$ 
``retract $k$-rational'' $\Rightarrow$ ``$k$-unirational''. 

Let $L$ be a finite Galois extension of $k$ and $G={\rm Gal}(L/k)$ 
be the Galois group of the extension $L/k$. 
Let $M=\bigoplus_{1\leq i\leq n}\bZ\cdot u_i$ be a $G$-lattice with 
a $\bZ$-basis $\{u_1,\ldots,u_n\}$, 
i.e. finitely generated $\bZ[G]$-module 
which is $\bZ$-free as an abelian group. 
Let $G$ act on the rational function field $L(x_1,\ldots,x_n)$ 
over $L$ with $n$ variables $x_1,\ldots,x_n$ by 
\begin{align}
\sigma(x_i)=\prod_{j=1}^n x_j^{a_{i,j}},\quad 1\leq i\leq n\label{acts}
\end{align}
for any $\sigma\in G$, when $\sigma (u_i)=\sum_{j=1}^n a_{i,j} u_j$, 
$a_{i,j}\in\bZ$. 
The field $L(x_1,\ldots,x_n)$ with this action of $G$ will be denoted 
by $L(M)$.
There is the duality between the category of $G$-lattices 
and the category of algebraic $k$-tori which split over $L$ 
(see \cite[page 27, Example 6]{Vos98}). 
In fact, if $T$ is an algebraic $k$-torus, then the character 
group $X(T)={\rm Hom}(T,\bG_m)$ of $T$ may be regarded as a $G$-lattice. 
Conversely, for a given $G$-lattice $M$, there exists an algebraic $k$-torus 
$T$ which splits over $L$ such that $X(T)$ is isomorphic to $M$ as a $G$-lattice. 

The invariant field $L(M)^G$ of $L(M)$ under the action of $G$ 
may be identified with the function field of the algebraic $k$-torus $T$. 
Note that the field $L(M)^G$ is always $k$-unirational 
(see \cite[page 40, Example 21]{Vos98}).
Tori of dimension $n$ over $k$ correspond bijectively 
to the elements of the set $H^1(\mathcal{G},\GL(n,\bZ))$ 
where $\mathcal{G}={\rm Gal}(k_{\rm s}/k)$ since 
${\rm Aut}(\bG_m^n)=\GL(n,\bZ)$. 
The $k$-torus $T$ of dimension $n$ is determined uniquely by the integral 
representation $h : \mathcal{G}\rightarrow \GL(n,\bZ)$ up to conjugacy, 
and the group $h(\mathcal{G})$ is a finite subgroup of $\GL(n,\bZ)$ 
(see \cite[page 57, Section 4.9]{Vos98})). 

Let $K/k$ be a separable field extension of degree $n$ 
and $L/k$ be the Galois closure of $K/k$. 
Let $G={\rm Gal}(L/k)$ and $H={\rm Gal}(L/K)$. 
The Galois group $G$ may be regarded as a transitive subgroup of 
the symmetric group $S_n$ of degree $n$. 
Let $R^{(1)}_{K/k}(\bG_m)$ be the norm one torus of $K/k$,
i.e. the kernel of the norm map $R_{K/k}(\bG_m)\rightarrow \bG_m$ where 
$R_{K/k}$ is the Weil restriction (see \cite[page 37, Section 3.12]{Vos98}). 
The norm one torus $R^{(1)}_{K/k}(\bG_m)$ has the 
Chevalley module $J_{G/H}$ as its character module 
and the field $L(J_{G/H})^G$ as its function field 
where $J_{G/H}=(I_{G/H})^\circ={\rm Hom}_\bZ(I_{G/H},\bZ)$ 
is the dual lattice of $I_{G/H}={\rm Ker}\ \varepsilon$ and 
$\varepsilon : \bZ[G/H]\rightarrow \bZ$ is the augmentation map 
(see \cite[Section 4.8]{Vos98}). 
We have the exact sequence $0\rightarrow \bZ\rightarrow \bZ[G/H]
\rightarrow J_{G/H}\rightarrow 0$ and rank $J_{G/H}=n-1$. 
Write $J_{G/H}=\oplus_{1\leq i\leq n-1}\bZ x_i$. 
Then the action of $G$ on $L(J_{G/H})=L(x_1,\ldots,x_{n-1})$ is 
nothing but (\ref{acts}). 

The aim of this paper is to give the stably rational classification of 
algebraic $k$-tori of dimensions $4$ and $5$ 
(cf. \cite{Vos98}, \cite{Kun07} and the references therein). 
It is easy to see that all the $1$-dimensional algebraic $k$-tori $T$, 
i.e. the trivial torus $\bG_m$ and the norm one torus $R_{L/k}^{(1)}(\bG_m)$ 
with $[L:k]=2$, are $k$-rational. 
\begin{theorem}[{Voskresenskii \cite{Vos67}}]\label{thVo}
All the $2$-dimensional algebraic $k$-tori $T$ are $k$-rational. 
In particular, for any finite subgroups $G\leq \GL(2,\bZ)$, $L(x_1,x_2)^G$ 
is $k$-rational. 
\end{theorem}
There are $13$ non-conjugate finite subgroups of $\GL(2,\bZ)$. 
By Theorem \ref{thVo}, we see that 
for decomposable $3$-dimensional $k$-tori $T=T_1\times T_2$ 
with dim $T_1=1$ and dim $T_2=2$, 
the function fields $L(T)=L(M)^G$ are $k$-rational where 
$M=M_1\oplus M_2$ with rank $M_1=1$ and rank $M_2=2$. 

Let $S_n$ (resp. $A_n$, $D_n$, $C_n$) be the symmetric 
(resp. the alternating, the dihedral, the cyclic) group 
of degree $n$ of order $n!$ (resp. $n!/2$, $2n$, $n$). 
For $2\leq n\leq 4$, the GAP ID $(n,i,j,k)$ of a finite subgroup 
$G$ of $\GL(n,\bZ)$ stands for 
the $k$-th $\bZ$-class of the $j$-th $\bQ$-class of 
the $i$-th crystal system of dimension $n$ as in 
\cite[Table 1]{BBNWZ78} and \cite{GAP}. 
There are $73$ $\bZ$-classes forming $32$ $\bQ$-classes 
which are classified into $7$ crystal systems in $\GL(3,\bZ)$. 

The birational classification of 
$3$-dimensional $k$-tori is given by Kunyavskii \cite{Kun90}. 
\begin{theorem}[{Kunyavskii \cite{Kun90}}]\label{thKu}
Let $L/k$ be a Galois extension and $G\simeq 
{\rm Gal}(L/k)$ be a finite subgroup of $\GL(3,\bZ)$ 
which acts on $L(x_1,x_2,x_3)$ via $(\ref{acts})$. 
Then $L(x_1,x_2,x_3)^G$ is not $k$-rational if and only if 
$G$ is conjugate to one of the $15$ groups which are given 
as in {\rm Table} $1$.  
Moreover, if $L(x_1,x_2,x_3)^G$ is not $k$-rational, 
then it is not retract $k$-rational. 
\end{theorem}

\vspace*{2mm}
\begin{center}
Table $1$: $L(M)^G$ not retract $k$-rational, 
rank $M=3$, $M$: indecomposable ($15$ cases)\vspace*{2mm}\\
\fontsize{8pt}{10pt}\selectfont
\begin{tabular}{lll} 
${}^tG$ in \cite{Kun90} & GAP ID& $G$ \\\hline
$U_1$ & $(3,3,1,3)$ & $C_2^2$\\
$U_2$ & $(3,3,3,3)$ & $C_2^3$\\
$U_3$ & $(3,4,4,2)$ & $D_4$\\
$U_4$ & $(3,4,6,3)$ & $D_4$\\
$U_5$ & $(3,7,1,2)$ & $A_4$
\end{tabular}\quad
\begin{tabular}{lll}
${}^tG$ in \cite{Kun90} & GAP ID& $G$ \\\hline
$U_6$ & $(3,4,7,2)$ & $D_4\times C_2$\\
$U_7$ & $(3,7,2,2)$ & $A_4\times C_2$\\
$U_8$ & $(3,7,3,3)$ & $S_4$\\
$U_9$ & $(3,7,3,2)$ & $S_4$\\
$U_{10}$ & $(3,7,4,2)$ & $S_4$
\end{tabular}\quad
\begin{tabular}{lll} 
${}^tG$ in \cite{Kun90} & GAP ID& $G$ \\\hline
$U_{11}$ & $(3,7,5,3)$ & $S_4\times C_2$\\
$U_{12}$ & $(3,7,5,2)$ & $S_4\times C_2$\\
$W_1$ & $(3,4,3,2)$ & $C_4\times C_2$\\
$W_2$ & $(3,3,3,4)$ & $C_2^3$\\
$W_3$ & $(3,7,2,3)$ & $A_4\times C_2$
\end{tabular}
\end{center}
\vspace*{4mm}

For the last statement of Theorem \ref{thKu}, 
see \cite[page 25, the fifth paragraph]{Kan12}.

If we adopt the action of $G$ as in (\ref{acts}), 
we should take the transpose ${}^tG$ of the matrix group $G$ 
as in \cite{Kun90} (cf. Theorem \ref{thKu2} in Section \ref{seApp}).  
We will give an alternative proof of Theorem \ref{thKu} 
using the algorithms of this paper (see Example \ref{exKun}). 
For $n=4$, some birational invariants are computed by 
Popov \cite{Pop98}. 

Let $T=R^{(1)}_{K/k}(\bG_m)$ be the norm one torus defined by $K/k$. 
The rationality problem for norm one tori is investigated 
by \cite{EM74}, \cite{CTS77}, \cite{Hur84}, \cite{CTS87}, 
\cite{LeB95}, \cite{CK00}, \cite{LL00}, \cite{Flo} and \cite{End11}. 
%
\begin{theorem}[{Endo and Miyata \cite[Theorem 1.5]{EM74}, Saltman \cite[Theorem 3.14]{Sal84a}}]\label{th13-1}
Let $K/k$ be a finite Galois field extension and $G={\rm Gal}(K/k)$. 
The following conditions are equivalent:\\
{\rm (i)} $R^{(1)}_{K/k}(\bG_m)$ is retract $k$-rational;\\
{\rm (ii)} all the Sylow subgroups of $G$ are cyclic. 
\end{theorem}

\begin{theorem}[{Endo and Miyata \cite[Theorem 2.3]{EM74}, Colliot-Th\'{e}l\`{e}ne and Sansuc \cite[Proposition 3]{CTS77}}]\label{th13-2}
Let $K/k$ be a finite Galois field extension and $G={\rm Gal}(K/k)$. 
The following conditions are equivalent:\\
{\rm (i)} $R^{(1)}_{K/k}(\bG_m)$ is stably $k$-rational;\\
{\rm (ii)} all the Sylow subgroups of $G$ are cyclic and $H^4(G,\bZ)\simeq \widehat H^0(G,\bZ)$ 
where $\widehat H$ is the Tate cohomology $($see Section \ref{sePre}$)$;\\
{\rm (iii)} $G=C_m$ or $G=C_n\times \langle\sigma,\tau\mid\sigma^k=\tau^{2^d}=1,
\tau\sigma\tau^{-1}=\sigma^{-1}\rangle$ where $d\geq 1, k\geq 3$, 
$n,k$: odd, and ${\rm gcd}\{n,k\}=1$;\\
{\rm (iv)} $G=\langle s,t\mid s^m=t^{2^d}=1, tst^{-1}=s^r, m: odd,\ 
r^2\equiv 1\pmod{m}\rangle$.
\end{theorem}
\begin{theorem}[Endo {\cite[Theorem 2.1]{End11}}]
Let $K/k$ be a finite non-Galois, separable field extension 
and $L/k$ be the Galois closure of $K/k$. 
Assume that the Galois group of $L/k$ is nilpotent. 
Then the norm one torus $R^{(1)}_{K/k}(\bG_m)$ is not 
retract $k$-rational.
\end{theorem}
\begin{theorem}[Endo {\cite[Theorem 3.1]{End11}}]\label{th15}
Let $K/k$ be a finite non-Galois, separable field extension 
and $L/k$ be the Galois closure of $K/k$. 
Let $G={\rm Gal}(L/k)$ and $H={\rm Gal}(L/K)\leq G$. 
Assume that all the Sylow subgroups of $G$ are cyclic. 
Then the norm one torus $R^{(1)}_{K/k}(\bG_m)$ is retract $k$-rational, 
and the following conditions are equivalent:\\
{\rm (i)} 
$R^{(1)}_{K/k}(\bG_m)$ is stably $k$-rational;\\
{\rm (ii)} 
$G=D_n$ with $n$ odd $(n\geq 3)$ 
or $G=C_m\times D_n$ where $m,n$ are odd, 
$m,n\geq 3$, ${\rm gcd}\{m,n\}=1$, and $H\leq D_n$ is of order $2$;\\
{\rm (iii)} 
$H=C_2$ and $G\simeq C_r\rtimes H$, $r\geq 3$ odd, where 
$H$ acts non-trivially on $C_r$. 
\end{theorem}
\begin{theorem}[{Colliot-Th\'{e}l\`{e}ne and Sansuc \cite[Proposition 9.1]{CTS87}, 
\cite[Theorem 3.1]{LeB95}, 
\cite[Proposition 0.2]{CK00}, \cite{LL00}, 
Endo \cite[Theorem 4.1]{End11}, see also 
\cite[Remark 4.2 and Theorem 4.3]{End11}}]\label{thS}
Let $K/k$ be a non-Galois separable field extension 
of degree $n$ and $L/k$ be the Galois closure of $K/k$. 
Assume that ${\rm Gal}(L/k)=S_n$, $n\geq 3$, 
and ${\rm Gal}(L/K)=S_{n-1}$ is the stabilizer of one of the letters 
in $S_n$.\\
{\rm (i)}\ 
$R^{(1)}_{K/k}(\bG_m)$ is retract $k$-rational 
if and only if $n$ is a prime;\\
{\rm (ii)}\ 
$R^{(1)}_{K/k}(\bG_m)$ is $($stably$)$ $k$-rational 
if and only if $n=3$.
\end{theorem}

Let $[R^{(1)}_{K/k}(\bG_m)]^{(t)}$ be the product of $t$ copies of 
the norm one torus $R^{(1)}_{K/k}(\bG_m)$. 
\begin{theorem}[Endo {\cite[Theorem 4.4]{End11}}]\label{thA}
Let $K/k$ be a non-Galois separable field extension 
of degree $n$ and $L/k$ be the Galois closure of $K/k$. 
Assume that ${\rm Gal}(L/k)=A_n$, $n\geq 4$, 
and ${\rm Gal}(L/K)=A_{n-1}$ is the stabilizer of one of the letters 
in $A_n$.\\
{\rm (i)}\ 
$R^{(1)}_{K/k}(\bG_m)$ is retract $k$-rational 
if and only if $n$ is a prime.\\
{\rm (ii)}\ For some positive integer $t$, 
$[R^{(1)}_{K/k}(\bG_m)]^{(t)}$ is stably $k$-rational 
if and only if $n=5$.
\end{theorem}

The first main result of this paper is 
the stably rational classification of 
algebraic $k$-tori of dimension $4$. 
There are $710$ $\bZ$-classes forming $227$ $\bQ$-classes 
which are classified into $33$ crystal systems in $\GL(4,\bZ)$. 

Let $F_{20}\simeq C_5\rtimes C_4$ be the Frobenius group of order $20$. 
\begin{theorem}[
Stably rational classification of 
algebraic $k$-tori of dimension $4$]\label{th1}
Let $L/k$ be a Galois extension and $G\simeq 
{\rm Gal}(L/k)$ be a finite subgroup of $\GL(4,\bZ)$ 
which acts on $L(x_1,x_2,x_3,x_4)$ via $(\ref{acts})$. \\
{\rm (i)} 
$L(x_1,x_2,x_3,x_4)^G$ is stably $k$-rational 
if and only if 
$G$ is conjugate to one of the $487$ groups which are not in 
{\rm Tables} $2$, $3$ and $4$.\\
{\rm (ii)} 
$L(x_1,x_2,x_3,x_4)^G$ is not stably but retract $k$-rational 
if and only if $G$ is conjugate to one of the $7$ groups which are 
given as in {\rm Table} $2$.\\
{\rm (iii)} 
$L(x_1,x_2,x_3,x_4)^G$ is not retract $k$-rational if and only if 
$G$ is conjugate to one of the $216$ groups which are given as 
in {\rm Tables} $3$ and $4$.
\end{theorem}
\vspace*{2mm}
\begin{center}
Table $2$: $L(M)^G$ not stably but retract $k$-rational, 
rank $M=4$, $M$: indecomposable ($7$ cases)\vspace*{2mm}\\
{\scriptsize
\begin{tabular}{ll}
GAP ID & $G$ \\\hline
$(4,31,1,3)$ & $F_{20}$\\
$(4,31,1,4)$ & $F_{20}$
\end{tabular}
\begin{tabular}{ll}
GAP ID & $G$ \\\hline
$(4,31,2,2)$ & $C_2\times F_{20}$\\
$(4,31,4,2)$ & $S_5$
\end{tabular}
\begin{tabular}{ll}
GAP ID & $G$ \\\hline
$(4,31,5,2)$ & $S_5$\\
$(4,31,7,2)$ & $C_2\times S_5$
\end{tabular}
\begin{tabular}{ll}
GAP ID & $G$ \\\hline
$(4,33,2,1)$ & $C_3\rtimes C_8$\\
$$ & $$
\end{tabular}
}
\end{center}

\begin{center}
\vspace*{4mm}
Table $3$: $L(M)^G$ not retract $k$-rational, $M=M_1\oplus M_2$,\\ 
rank $M_1=3$, rank $M_2=1$, $M_2$: indecomposable ($64$ cases)\vspace*{2mm}\\
{\scriptsize
\begin{tabular}{l}
GAP ID  \\\hline
$(4,4,3,6)$\\
$(4,4,4,4)$\\
$(4,4,4,6)$\\
$(4,5,1,9)$\\
$(4,5,2,4)$\\
$(4,5,2,7)$\\
$(4,6,1,4)$\\
$(4,6,1,8)$
\end{tabular}
\begin{tabular}{l}
GAP ID  \\\hline
$(4,6,2,4)$\\
$(4,6,2,8)$\\
$(4,6,2,9)$\\
$(4,6,3,3)$\\
$(4,6,3,6)$\\
$(4,7,3,2)$\\
$(4,7,4,3)$\\
$(4,7,5,2)$
\end{tabular}
\begin{tabular}{l}
GAP ID  \\\hline
$(4,7,7,2)$\\
$(4,12,2,4)$\\
$(4,12,3,7)$\\
$(4,12,4,6)$\\
$(4,12,4,8)$\\
$(4,12,4,9)$\\
$(4,12,5,6)$\\
$(4,12,5,7)$
\end{tabular}
\begin{tabular}{l}
GAP ID  \\\hline
$(4,13,1,3)$\\
$(4,13,2,4)$\\
$(4,13,3,4)$\\
$(4,13,4,3)$\\
$(4,13,5,3)$\\
$(4,13,6,3)$\\
$(4,13,7,6)$\\
$(4,13,7,7)$
\end{tabular}
\begin{tabular}{l}
GAP ID  \\\hline
$(4,13,7,8)$\\
$(4,13,8,3)$\\
$(4,13,8,4)$\\
$(4,13,9,3)$\\
$(4,13,10,3)$\\
$(4,24,1,5)$\\
$(4,24,2,3)$\\
$(4,24,2,5)$
\end{tabular}
\begin{tabular}{l}
GAP ID  \\\hline
$(4,24,3,5)$\\
$(4,24,4,3)$\\
$(4,24,4,5)$\\
$(4,24,5,3)$\\
$(4,24,5,5)$\\
$(4,25,1,2)$\\
$(4,25,1,4)$\\
$(4,25,2,4)$
\end{tabular}
\begin{tabular}{l}
GAP ID  \\\hline
$(4,25,3,2)$\\
$(4,25,3,4)$\\
$(4,25,4,4)$\\
$(4,25,5,2)$\\
$(4,25,5,4)$\\
$(4,25,6,2)$\\
$(4,25,6,4)$\\
$(4,25,7,2)$
\end{tabular}
\begin{tabular}{l}
GAP ID  \\\hline
$(4,25,7,4)$\\
$(4,25,8,2)$\\
$(4,25,8,4)$\\
$(4,25,9,4)$\\
$(4,25,10,2)$\\
$(4,25,10,4)$\\
$(4,25,11,2)$\\
$(4,25,11,4)$\\
\end{tabular}
}
\end{center}

\vspace*{4mm}
\begin{center}
Table $4$: $L(M)^G$ not retract $k$-rational, 
rank $M=4$, $M$: indecomposable ($152$ cases)\vspace*{2mm}\\
{\scriptsize
\begin{tabular}{l} 
GAP ID \\\hline
$(4,5,1,12)$\\
$(4,5,2,5)$\\
$(4,5,2,8)$\\
$(4,5,2,9)$\\
$(4,6,1,6)$\\
$(4,6,1,11)$\\
$(4,6,2,6)$\\
$(4,6,2,10)$\\
$(4,6,2,12)$\\
$(4,6,3,4)$\\
$(4,6,3,7)$\\
$(4,6,3,8)$\\
$(4,12,2,5)$\\
$(4,12,2,6)$\\
$(4,12,3,11)$\\
$(4,12,4,10)$\\
$(4,12,4,11)$
\end{tabular}
\begin{tabular}{l} 
GAP ID \\\hline
$(4,12,4,12)$\\
$(4,12,5,8)$\\
$(4,12,5,9)$\\
$(4,12,5,10)$\\
$(4,12,5,11)$\\
$(4,13,1,5)$\\
$(4,13,2,5)$\\
$(4,13,3,5)$\\
$(4,13,4,5)$\\
$(4,13,5,4)$\\
$(4,13,5,5)$\\
$(4,13,6,5)$\\
$(4,13,7,9)$\\
$(4,13,7,10)$\\
$(4,13,7,11)$\\
$(4,13,8,5)$\\
$(4,13,8,6)$
\end{tabular}
\begin{tabular}{l} 
GAP ID \\\hline
$(4,13,9,4)$\\
$(4,13,9,5)$\\
$(4,13,10,4)$\\
$(4,13,10,5)$\\
$(4,18,1,3)$\\
$(4,18,2,4)$\\
$(4,18,2,5)$\\
$(4,18,3,5)$\\
$(4,18,3,6)$\\
$(4,18,3,7)$\\
$(4,18,4,4)$\\
$(4,18,4,5)$\\
$(4,18,5,5)$\\
$(4,18,5,6)$\\
$(4,18,5,7)$\\
$(4,19,1,2)$\\
$(4,19,2,2)$
\end{tabular}
\begin{tabular}{l} 
GAP ID \\\hline
$(4,19,3,2)$\\
$(4,19,4,3)$\\
$(4,19,4,4)$\\
$(4,19,5,2)$\\
$(4,19,6,2)$\\
$(4,22,1,1)$\\
$(4,22,2,1)$\\
$(4,22,3,1)$\\
$(4,22,4,1)$\\
$(4,22,5,1)$\\
$(4,22,5,2)$\\
$(4,22,6,1)$\\
$(4,22,7,1)$\\
$(4,22,8,1)$\\
$(4,22,9,1)$\\
$(4,22,10,1)$\\
$(4,22,11,1)$
\end{tabular}
\begin{tabular}{l} 
GAP ID \\\hline
$(4,24,2,4)$\\
$(4,24,2,6)$\\
$(4,24,4,4)$\\
$(4,24,5,4)$\\
$(4,24,5,6)$\\
$(4,25,1,3)$\\
$(4,25,2,3)$\\
$(4,25,2,5)$\\
$(4,25,3,3)$\\
$(4,25,4,3)$\\
$(4,25,5,3)$\\
$(4,25,5,5)$\\
$(4,25,6,3)$\\
$(4,25,6,5)$\\
$(4,25,7,3)$\\
$(4,25,8,3)$\\
$(4,25,9,3)$
\end{tabular}
\begin{tabular}{l} 
GAP ID \\\hline
$(4,25,9,5)$\\
$(4,25,10,3)$\\
$(4,25,10,5)$\\
$(4,25,11,3)$\\
$(4,25,11,5)$\\
$(4,29,1,1)$\\
$(4,29,1,2)$\\
$(4,29,2,1)$\\
$(4,29,3,1)$\\
$(4,29,3,2)$\\
$(4,29,3,3)$\\
$(4,29,4,1)$\\
$(4,29,4,2)$\\
$(4,29,5,1)$\\
$(4,29,6,1)$\\
$(4,29,7,1)$\\
$(4,29,7,2)$\\
\end{tabular}
\begin{tabular}{l} 
GAP ID \\\hline
$(4,29,8,1)$\\
$(4,29,8,2)$\\
$(4,29,9,1)$\\
$(4,32,1,2)$\\
$(4,32,2,2)$\\
$(4,32,3,2)$\\
$(4,32,4,2)$\\
$(4,32,5,2)$\\
$(4,32,5,3)$\\
$(4,32,6,2)$\\
$(4,32,7,2)$\\
$(4,32,8,2)$\\
$(4,32,9,4)$\\
$(4,32,9,5)$\\
$(4,32,10,2)$\\
$(4,32,11,2)$\\
$(4,32,11,3)$
\end{tabular}
\begin{tabular}{l} 
GAP ID \\\hline
$(4,32,12,2)$\\
$(4,32,13,3)$\\
$(4,32,13,4)$\\
$(4,32,14,3)$\\
$(4,32,14,4)$\\
$(4,32,15,2)$\\
$(4,32,16,2)$\\
$(4,32,16,3)$\\
$(4,32,17,2)$\\
$(4,32,18,2)$\\
$(4,32,18,3)$\\
$(4,32,19,2)$\\
$(4,32,19,3)$\\
$(4,32,20,2)$\\
$(4,32,20,3)$\\
$(4,32,21,2)$\\
$(4,32,21,3)$
\end{tabular}
\begin{tabular}{l} 
GAP ID \\\hline
$(4,33,1,1)$\\
$(4,33,3,1)$\\
$(4,33,4,1)$\\
$(4,33,5,1)$\\
$(4,33,6,1)$\\
$(4,33,7,1)$\\
$(4,33,8,1)$\\
$(4,33,9,1)$\\
$(4,33,10,1)$\\
$(4,33,11,1)$\\
$(4,33,12,1)$\\
$(4,33,13,1)$\\
$(4,33,14,1)$\\
$(4,33,14,2)$\\
$(4,33,15,1)$\\
$(4,33,16,1)$\\
\\
\end{tabular}
}
\end{center}
\vspace*{2mm}
%
\vspace*{2mm}

More detailed information of 
the stably rational classification of 
algebraic $k$-tori of dimension $4$ is given as in Table $7$. 
In Table $7$, $\#$ on the second column stands for the number 
of $\bZ$-classes in each $\bQ$-classes, and the list $[s,r,u]$ stands 
for the number $s$ (resp. $r$, $u$) of $\bZ$-classes 
whose invariant field $L(M)^G$ is stably $k$-rational 
(resp. not stably but retract $k$-rational, 
not retract $k$-rational) in each $\bQ$-classes $(4,i,j)$. 
For example, $[11,0,2]$ in the GAP ID $(4,5,1)$ 
in Table $7$ means that the $1$st $\bQ$-class of the $5$th 
crystal system of dimension $4$ consists of $13$ $\bZ$-classes 
and $L(M)^G$ is $k$-stably rational for $11$ $\bZ$-classes of them, 
and is not retract $k$-rational for $2$ $\bZ$-classes. 

Let $G(n,i)$ be the $i$-th group of order $n$ in GAP \cite{GAP}. 
Let $dTm$ be the $m$-th transitive subgroup of $S_d$ 
(cf. \cite{BM83} and \cite{GAP}). 

By Theorem \ref{th1}, we have the following theorem.
\begin{theorem}\label{th11}
Let $K/k$ be a separable field extension of degree $5$ 
and $L/k$ be the Galois closure of $K/k$. 
Assume that $G={\rm Gal}(L/k)$ is a transitive subgroup of $S_5$ 
which acts on $L(x_1,x_2,x_3,x_4)$ via $(\ref{acts})$, 
and $H={\rm Gal}(L/K)$ is the stabilizer of one 
of the letters in $G$. Then 
the stably rational classification of 
norm one tori $R_{K/k}^{(1)}(\bG_m)$ is given as in Table $5$.
\end{theorem}

\begin{center}
\vspace*{1mm}
Table $5$: \vspace*{2mm}\\

\begin{tabular}{lllll} 
 & & & GAP ID of the& $L(J_{G/H})^G$\\
$G$& & $G(n,i)$ & $G$-action on $J_{G/H}$ & $=L(x_1,x_2,x_3,x_4)^G$\\\hline
$5T1$&$C_5$&$G(5,1)$&$(4,27,1,1)$& stably $k$-rational\\
$5T2$&$D_5$&$G(10,1)$&$(4,27,3,2)$& stably $k$-rational\\
$5T3$&$F_{20}$&$G(20,3)$&$(4,31,1,3)$& 
not stably but retract $k$-rational\\
$5T4$&$A_5$&$G(60,5)$&$(4,31,3,2)$& stably $k$-rational\\
$5T5$&$S_5$&$G(120,34)$&$(4,31,4,2)$& 
not stably but retract $k$-rational
\end{tabular}
\vspace*{2mm}
\end{center}
Theorem \ref{th11} is already known except for the case of $A_5$ 
(see Theorems \ref{th13-1}, \ref{th13-2}, \ref{th15}, \ref{thS} and \ref{thA}). 
Stably $k$-rationality of $R_{K/k}^{(1)}(\bG_m)$ for the case $A_5$ 
is asked by S. Endo in \cite[Remark 4.6]{End11}.
By Theorems \ref{thA} and \ref{th11}, we get: 
\begin{corollary}\label{corA}
Let $K/k$ be a non-Galois separable field extension 
of degree $n$ and $L/k$ be the Galois closure of $K/k$. 
Assume that ${\rm Gal}(L/k)=A_n$, $n\geq 4$, 
and ${\rm Gal}(L/K)=A_{n-1}$ is the stabilizer of one of the letters 
in $A_n$. Then 
$R^{(1)}_{K/k}(\bG_m)$ is stably $k$-rational 
if and only if $n=5$.
\end{corollary}

For $n=5$, the CARAT ID $(n,i,j)$ of a finite subgroup 
$G$ of $\GL(5,\bZ)$ stands for 
the $j$-th $\bZ$-class of the $i$-th $\bQ$-class 
in $\GL(5,\bZ)$ in the CARAT package \cite{Carat} of 
GAP\footnote{The generators and some information about the groups 
$G\leq \GL(5,\bZ)$ for the CARAT ID $(5,i,j)$ are available 
at the second-named author's web page 
{\tt http://www.math.h.kyoto-u.ac.jp/\~{}yamasaki/Algorithm/}.} 
(see Section \ref{seCarat} for the details of the CARAT ID of dimension 
$n\leq 6$). 
There are $6079$ $\bZ$-classes forming $955$ $\bQ$-classes in $\GL(5,\bZ)$. 

The second main result of this paper gives 
the stably rational classification of 
algebraic $k$-tori of dimension $5$. 
We will display Tables $11$ to $15$ of Theorem \ref{th2} 
in Section \ref{tables}. 
\begin{theorem}[
Stably rational classification of 
algebraic $k$-tori of dimension $5$]\label{th2}
Let $L/k$ be a Galois extension and $G\simeq 
{\rm Gal}(L/k)$ be a finite subgroup of $\GL(5,\bZ)$ 
which acts on $L(x_1,x_2,x_3,x_4,x_5)$ via $(\ref{acts})$.\\
{\rm (i)} 
$L(x_1,x_2,x_3,x_4,x_5)^G$ is stably $k$-rational if and only if 
$G$ is conjugate to one of the $3051$ groups which are not in 
{\rm Tables} $11$, $12$, $13$, $14$ and $15$.\\
{\rm (ii)} 
$L(x_1,x_2,x_3,x_4,x_5)^G$ is not stably but retract $k$-rational 
if and only if $G$ is conjugate to one of the $25$ groups which are given as 
in {\rm Table} $11$.\\
{\rm (iii)} 
$L(x_1,x_2,x_3,x_4,x_5)^G$ is not retract $k$-rational if and only if 
$G$ is conjugate to one of the $3003$ groups which are given as 
in {\rm Tables} $12$, $13$, $14$ and $15$.
\end{theorem}
\begin{remark}\label{rem25}
For the $25$ groups $G$ as in Theorem \ref{th2} (ii), 
the corresponding $G$-lattices $M$ are decomposable $M\simeq M_1\oplus M_2$ 
where $M_1$ is a $G/N$-lattice of rank $4$, 
$N=\{\sigma\in G\mid\sigma(v)=v\ {\rm for\ any}\ v\in M_1\}$ 
and $G/N$ is one of the $7$ groups as in Theorem \ref{th1} (ii) (Table $2$) 
(see Example \ref{exN}). 
In particular, if $M$ is an indecomposable $G$-lattice of rank $5$, then 
$L(M)^G$ is either stably $k$-rational or not retract $k$-rational. 
\end{remark}

More detailed information of 
the stably rational classification of 
algebraic $k$-tori of dimension $5$ is given as 
in Table $16$ (see also the explanation of Table $7$ above). 
\begin{theorem}\label{th22}
Let $K/k$ be a separable field extension 
of degree $6$ and $L/k$ be the Galois closure of $K/k$. 
Assume that $G={\rm Gal}(L/k)$ is a transitive subgroup of $S_6$ 
which acts on $L(x_1,x_2,x_3,x_4,x_5)$ via $(\ref{acts})$, 
and $H={\rm Gal}(L/K)$ is the stabilizer of one 
of the letters in $G$. Then 
the stably rational classification of 
norm one tori $R_{K/k}^{(1)}(\bG_m)$ is given as in Table $6$.
\end{theorem}
\begin{center}
\vspace*{2mm}
Table $6$: \vspace*{2mm}\\
\begin{tabular}{lllll}
 & & & CARAT ID of the& $L(J_{G/H})^G$\\
$G$& & $G(n,i)$ & $G$-action on $J_{G/H}$ & $=L(x_1,x_2,x_3,x_4,x_5)^G$\\\hline
$6T1$&$C_6$&$G(6,2)$&$(5,461,4)$& stably $k$-rational\\
$6T2$&$S_3$&$G(6,1)$&$(5,173,4)$& stably $k$-rational\\
$6T3$&$D_6$&$G(12,4)$&$(5,391,4)$& stably $k$-rational\\
$6T4$&$A_4$&$G(12,3)$&$(5,580,2)$& not retract $k$-rational\\
$6T5$&$C_3\times S_3$&$G(18,3)$&$(5,823,4)$& not retract $k$-rational\\
$6T6$&$C_2\times A_4$&$G(24,13)$&$(5,606,2)$& not retract $k$-rational\\
$6T7$&$S_4$&$G(24,12)$&$(5,607,2)$& not retract $k$-rational\\
$6T8$&$S_4$&$G(24,12)$&$(5,608,2)$& not retract $k$-rational\\
$6T9$&$S_3^2$&$G(36,10)$&$(5,855,6)$& not retract $k$-rational\\
$6T{10}$&$C_3^2\rtimes C_4$&$G(36,9)$&$(5,853,5)$& not retract $k$-rational\\
$6T{11}$&$C_2\times S_4$&$G(48,48)$&$(5,623,2)$& not retract $k$-rational\\
$6T{12}$&$A_5$&$G(60,5)$&$(5,952,1)$& not retract $k$-rational\\
$6T{13}$&$S_3^2\rtimes C_2$&$G(72,40)$&$(5,892,4)$& not retract $k$-rational\\
$6T{14}$&$S_5$&$G(120,34)$&$(5,947,1)$& not retract $k$-rational\\
$6T{15}$&$A_6$&$G(360,118)$&$(5,951,1)$& not retract $k$-rational\\
$6T{16}$&$S_6$&$G(720,763)$&$(5,953,1)$& not retract $k$-rational
\end{tabular}
\vspace*{2mm}
\end{center}

In Theorems \ref{th1}, \ref{th11}, \ref{th2} and \ref{th22}, 
we do not know whether the field $L(M)^G$ is $k$-rational 
when the field is stably $k$-rational except for few cases 
(see \cite[Chapter 2]{Vos98}). 


\newpage
\begin{center}
\begin{table}[t]
Table $7$: 
stably rational classification of 
algebraic $k$-tori of dimension $4$
\vspace*{2mm}\\
\fontsize{8pt}{10pt}\selectfont
\begin{tabular}{lllll}
GAP ID & \# & $[s,r,u]$ & $G(n,i)$ & \\\hline
$(4,1,1)$&1&[1,0,0]&$G(1,1)$&$\{1\}$\\
$(4,1,2)$&1&[1,0,0]&$G(2,1)$&$C_2$\\
$(4,2,1)$&2&[2,0,0]&$G(2,1)$&$C_2$\\
$(4,2,2)$&2&[2,0,0]&$G(2,1)$&$C_2$\\
$(4,2,3)$&2&[2,0,0]&$G(4,2)$&$C_2^2$\\
$(4,3,1)$&3&[3,0,0]&$G(2,1)$&$C_2$\\
$(4,3,2)$&3&[3,0,0]&$G(4,2)$&$C_2^2$\\
$(4,4,1)$&6&[6,0,0]&$G(4,2)$&$C_2^2$\\
$(4,4,2)$&7&[7,0,0]&$G(4,2)$&$C_2^2$\\
$(4,4,3)$&6&[5,0,1]&$G(4,2)$&$C_2^2$\\
$(4,4,4)$&6&[4,0,2]&$G(8,5)$&$C_2^3$\\
$(4,5,1)$&13&[11,0,2]&$G(4,2)$&$C_2^2$\\
$(4,5,2)$&9&[4,0,5]&$G(8,5)$&$C_2^3$\\
$(4,6,1)$&12&[8,0,4]&$G(8,5)$&$C_2^3$\\
$(4,6,2)$&12&[6,0,6]&$G(8,5)$&$C_2^3$\\
$(4,6,3)$&8&[3,0,5]&$G(16,14)$&$C_2^4$\\
$(4,7,1)$&2&[2,0,0]&$G(4,1)$&$C_4$\\
$(4,7,2)$&2&[2,0,0]&$G(4,1)$&$C_4$\\
$(4,7,3)$&2&[1,0,1]&$G(8,2)$&$C_4\times C_2$\\
$(4,7,4)$&4&[3,0,1]&$G(8,3)$&$D_4$\\
$(4,7,5)$&2&[1,0,1]&$G(8,3)$&$D_4$\\
$(4,7,6)$&2&[2,0,0]&$G(8,3)$&$D_4$\\
$(4,7,7)$&2&[1,0,1]&$G(16,11)$&$C_2\times D_4$\\
$(4,8,1)$&2&[2,0,0]&$G(3,1)$&$C_3$\\
$(4,8,2)$&2&[2,0,0]&$G(6,2)$&$C_6$\\
$(4,8,3)$&3&[3,0,0]&$G(6,1)$&$S_3$\\
$(4,8,4)$&3&[3,0,0]&$G(6,1)$&$S_3$\\
$(4,8,5)$&3&[3,0,0]&$G(12,4)$&$D_6$\\
$(4,9,1)$&1&[1,0,0]&$G(6,2)$&$C_6$\\
$(4,9,2)$&1&[1,0,0]&$G(6,2)$&$C_6$\\
$(4,9,3)$&1&[1,0,0]&$G(12,5)$&$C_6\times C_2$\\
$(4,9,4)$&1&[1,0,0]&$G(12,4)$&$D_6$\\
$(4,9,5)$&1&[1,0,0]&$G(12,4)$&$D_6$\\
$(4,9,6)$&2&[2,0,0]&$G(12,4)$&$D_6$\\
$(4,9,7)$&1&[1,0,0]&$G(24,14)$&$C_2^2\times S_3$\\
$(4,10,1)$&1&[1,0,0]&$G(4,1)$&$C_4$\\
$(4,11,1)$&1&[1,0,0]&$G(3,1)$&$C_3$\\
$(4,11,2)$&1&[1,0,0]&$G(6,2)$&$C_6$\\
$(4,12,1)$&7&[7,0,0]&$G(4,1)$&$C_4$\\
$(4,12,2)$&6&[3,0,3]&$G(8,2)$&$C_4\times C_2$\\
$(4,12,3)$&13&[11,0,2]&$G(8,3)$&$D_4$\\
$(4,12,4)$&13&[7,0,6]&$G(8,3)$&$D_4$\\
$(4,12,5)$&11&[5,0,6]&$G(16,11)$&$C_2\times D_4$\\
$(4,13,1)$&6&[4,0,2]&$G(8,2)$&$C_4\times C_2$\\
$(4,13,2)$&6&[4,0,2]&$G(8,2)$&$C_4\times C_2$\\
$(4,13,3)$&6&[4,0,2]&$G(8,3)$&$D_4$\\
$(4,13,4)$&6&[4,0,2]&$G(8,3)$&$D_4$\\
$(4,13,5)$&5&[2,0,3]&$G(16,10)$&$C_4\times C_2^2$\\
$(4,13,6)$&6&[4,0,2]&$G(16,11)$&$C_2\times D_4$\\
$(4,13,7)$&12&[6,0,6]&$G(16,11)$&$C_2\times D_4$\\
$(4,13,8)$&6&[2,0,4]&$G(16,11)$&$C_2\times D_4$\\
$(4,13,9)$&5&[2,0,3]&$G(16,11)$&$C_2\times D_4$\\
$(4,13,10)$&5&[2,0,3]&$G(32,46)$&$C_2^2\times D_4$\\
$(4,14,1)$&4&[4,0,0]&$G(6,2)$&$C_6$\\
$(4,14,2)$&4&[4,0,0]&$G(6,2)$&$C_6$\\
$(4,14,3)$&8&[8,0,0]&$G(6,1)$&$S_3$\\
$(4,14,4)$&4&[4,0,0]&$G(12,5)$&$C_6\times C_2$\\
$(4,14,5)$&4&[4,0,0]&$G(12,4)$&$D_6$\\
$(4,14,6)$&6&[6,0,0]&$G(12,4)$&$D_6$\\
$(4,14,7)$&6&[6,0,0]&$G(12,4)$&$D_6$
\end{tabular}
\ \ 
\begin{tabular}{lllll}
GAP ID & \# & $[s,r,u]$ & $G(n,i)$ & \\\hline
$(4,14,8)$&6&[6,0,0]&$G(12,4)$&$D_6$\\
$(4,14,9)$&6&[6,0,0]&$G(12,4)$&$D_6$\\
$(4,14,10)$&6&[6,0,0]&$G(24,14)$&$C_2^2\times S_3$\\
$(4,15,1)$&2&[2,0,0]&$G(12,5)$&$C_6\times C_2$\\
$(4,15,2)$&2&[2,0,0]&$G(12,5)$&$C_6\times C_2$\\
$(4,15,3)$&2&[2,0,0]&$G(12,5)$&$C_6\times C_2$\\
$(4,15,4)$&2&[2,0,0]&$G(12,4)$&$D_6$\\
$(4,15,5)$&4&[4,0,0]&$G(12,4)$&$D_6$\\
$(4,15,6)$&2&[2,0,0]&$G(24,14)$&$C_2^2\times S_3$\\
$(4,15,7)$&4&[4,0,0]&$G(24,14)$&$C_2^2\times S_3$\\
$(4,15,8)$&2&[2,0,0]&$G(24,15)$&$C_6\times C_2^2$\\
$(4,15,9)$&2&[2,0,0]&$G(24,14)$&$C_2^2\times S_3$\\
$(4,15,10)$&4&[4,0,0]&$G(24,14)$&$C_2^2\times S_3$\\
$(4,15,11)$&2&[2,0,0]&$G(24,14)$&$C_2^2\times S_3$\\
$(4,15,12)$&2&[2,0,0]&$G(48,51)$&$C_2^3\times S_3$\\
$(4,16,1)$&3&[3,0,0]&$G(8,3)$&$D_4$\\
$(4,17,1)$&3&[3,0,0]&$G(6,1)$&$S_3$\\
$(4,17,2)$&2&[2,0,0]&$G(12,4)$&$D_6$\\
$(4,18,1)$&3&[2,0,1]&$G(8,2)$&$C_4\times C_2$\\
$(4,18,2)$&5&[3,0,2]&$G(16,3)$&$(C_4\times C_2)\rtimes C_2$\\
$(4,18,3)$&7&[4,0,3]&$G(16,3)$&$(C_4\times C_2)\rtimes C_2$\\
$(4,18,4)$&5&[3,0,2]&$G(16,11)$&$C_2\times D_4$\\
$(4,18,5)$&7&[4,0,3]&$G(32,27)$&$C_2^4\rtimes C_2$\\
$(4,19,1)$&2&[1,0,1]&$G(16,4)$&$C_4\rtimes C_4$\\
$(4,19,2)$&2&[1,0,1]&$G(16,2)$&$C_4^2$\\
$(4,19,3)$&2&[1,0,1]&$G(32,25)$&$C_4\times D_4$\\
$(4,19,4)$&4&[2,0,2]&$G(32,28)$&$(C_4\times C_2^2)\rtimes C_2$\\
$(4,19,5)$&2&[1,0,1]&$G(32,34)$&$C_4^2\rtimes C_2$\\
$(4,19,6)$&2&[1,0,1]&$G(64,226)$&$D_4^2$\\
$(4,20,1)$&1&[1,0,0]&$G(12,2)$&$C_{12}$\\
$(4,20,2)$&1&[1,0,0]&$G(12,2)$&$C_{12}$\\
$(4,20,3)$&2&[2,0,0]&$G(12,1)$&$C_3\rtimes C_4$\\
$(4,20,4)$&2&[2,0,0]&$G(24,5)$&$C_4\times S_3$\\
$(4,20,5)$&1&[1,0,0]&$G(24,9)$&$C_{12}\times C_2$\\
$(4,20,6)$&1&[1,0,0]&$G(24,10)$&$C_3\times D_4$\\
$(4,20,7)$&2&[2,0,0]&$G(24,6)$&$D_{12}$\\
$(4,20,8)$&1&[1,0,0]&$G(24,10)$&$C_3\times D_4$\\
$(4,20,9)$&2&[2,0,0]&$G(24,5)$&$C_4\times S_3$\\
$(4,20,10)$&2&[2,0,0]&$G(24,10)$&$C_3\times D_4$\\
$(4,20,11)$&2&[2,0,0]&$G(24,6)$&$D_{12}$\\
$(4,20,12)$&4&[4,0,0]&$G(24,8)$&$(C_6\times C_2)\rtimes C_2$\\
$(4,20,13)$&4&[4,0,0]&$G(24,8)$&$(C_6\times C_2)\rtimes C_2$\\
$(4,20,14)$&1&[1,0,0]&$G(24,7)$&$C_2\times (C_3\rtimes C_4)$\\
$(4,20,15)$&1&[1,0,0]&$G(48,35)$&$C_2\times C_4\times S_3$\\
$(4,20,16)$&2&[2,0,0]&$G(48,38)$&$D_4\times S_3$\\
$(4,20,17)$&2&[2,0,0]&$G(48,38)$&$D_4\times S_3$\\
$(4,20,18)$&1&[1,0,0]&$G(48,45)$&$C_6\times D_4$\\
$(4,20,19)$&1&[1,0,0]&$G(48,36)$&$C_2\times D_{12}$\\
$(4,20,20)$&4&[4,0,0]&$G(48,38)$&$D_4\times S_3$\\
$(4,20,21)$&2&[2,0,0]&$G(48,43)$&$C_2\times ((C_6\times C_2)\rtimes C_2)$\\
$(4,20,22)$&1&[1,0,0]&$G(96,209)$&$C_2\times S_3\times D_4$\\
$(4,21,1)$&2&[2,0,0]&$G(6,2)$&$C_6$\\
$(4,21,2)$&2&[2,0,0]&$G(12,5)$&$C_6\times C_2$\\
$(4,21,3)$&4&[4,0,0]&$G(12,4)$&$D_6$\\
$(4,21,4)$&2&[2,0,0]&$G(24,14)$&$C_2^2\times S_3$\\
$(4,22,1)$&2&[1,0,1]&$G(9,2)$&$C_3^2$\\
$(4,22,2)$&2&[1,0,1]&$G(18,5)$&$C_6\times C_3$\\
$(4,22,3)$&3&[2,0,1]&$G(18,3)$&$C_3\times S_3$\\
$(4,22,4)$&3&[2,0,1]&$G(18,3)$&$C_3\times S_3$\\
$(4,22,5)$&5&[3,0,2]&$G(18,4)$&$C_3^2\rtimes C_2$
\end{tabular}
\end{table}
\end{center}
%
\begin{center}
\begin{table}[t]
Table $7$ (continued): 
stably rational classification of 
algebraic $k$-tori of dimension $4$
\vspace*{2mm}\\\
\fontsize{8pt}{10pt}\selectfont
\begin{tabular}{lllll}
GAP ID & \# & $[s,r,u]$ & $G(n,i)$ & \\\hline
$(4,22,6)$&3&[2,0,1]&$G(36,12)$&$C_6\times S_3$\\
$(4,22,7)$&3&[2,0,1]&$G(36,13)$&$C_2\times (C_3^2\rtimes C_2)$\\
$(4,22,8)$&4&[3,0,1]&$G(36,10)$&$S_3^2$\\
$(4,22,9)$&5&[4,0,1]&$G(36,10)$&$S_3^2$\\
$(4,22,10)$&4&[3,0,1]&$G(36,10)$&$S_3^2$\\
$(4,22,11)$&4&[3,0,1]&$G(72,46)$&$C_2\times S_3^2$\\
$(4,23,1)$&1&[1,0,0]&$G(18,5)$&$C_6\times C_3$\\
$(4,23,2)$&1&[1,0,0]&$G(36,14)$&$C_6^2$\\
$(4,23,3)$&1&[1,0,0]&$G(36,12)$&$C_6\times S_3$\\
$(4,23,4)$&1&[1,0,0]&$G(36,12)$&$C_6\times S_3$\\
$(4,23,5)$&2&[2,0,0]&$G(36,13)$&$C_2\times (C_3^2\rtimes C_2)$\\
$(4,23,6)$&2&[2,0,0]&$G(36,12)$&$C_6\times S_3$\\
$(4,23,7)$&1&[1,0,0]&$G(72,48)$&$C_2\times C_6\times S_3$\\
$(4,23,8)$&1&[1,0,0]&$G(72,49)$&$C_2^2\times (C_3^2\rtimes C_2)$\\
$(4,23,9)$&2&[2,0,0]&$G(72,46)$&$C_2\times S_3^2$\\
$(4,23,10)$&2&[2,0,0]&$G(72,46)$&$C_2\times S_3^2$\\
$(4,23,11)$&1&[1,0,0]&$G(144,192)$&$C_2^2\times S_3^2$\\
$(4,24,1)$&6&[5,0,1]&$G(12,3)$&$A_4$\\
$(4,24,2)$&6&[2,0,4]&$G(24,13)$&$C_2\times A_4$\\
$(4,24,3)$&6&[5,0,1]&$G(24,12)$&$S_4$\\
$(4,24,4)$&6&[3,0,3]&$G(24,12)$&$S_4$\\
$(4,24,5)$&6&[2,0,4]&$G(48,48)$&$C_2\times S_4$\\
$(4,25,1)$&5&[2,0,3]&$G(24,13)$&$C_2\times A_4$\\
$(4,25,2)$&5&[2,0,3]&$G(24,13)$&$C_2\times A_4$\\
$(4,25,3)$&5&[2,0,3]&$G(24,12)$&$S_4$\\
$(4,25,4)$&5&[3,0,2]&$G(24,12)$&$S_4$\\
$(4,25,5)$&5&[1,0,4]&$G(48,49)$&$C_2^2\times A_4$\\
$(4,25,6)$&5&[1,0,4]&$G(48,48)$&$C_2\times S_4$\\
$(4,25,7)$&5&[2,0,3]&$G(48,48)$&$C_2\times S_4$\\
$(4,25,8)$&5&[2,0,3]&$G(48,48)$&$C_2\times S_4$\\
$(4,25,9)$&5&[2,0,3]&$G(48,48)$&$C_2\times S_4$\\
$(4,25,10)$&5&[1,0,4]&$G(48,48)$&$C_2\times S_4$\\
$(4,25,11)$&5&[1,0,4]&$G(96,226)$&$C_2^2\times S_4$\\
$(4,26,1)$&1&[1,0,0]&$G(8,1)$&$C_8$\\
$(4,26,2)$&1&[1,0,0]&$G(16,7)$&$D_8$\\
$(4,27,1)$&1&[1,0,0]&$G(5,1)$&$C_5$\\
$(4,27,2)$&1&[1,0,0]&$G(10,2)$&$C_{10}$\\
$(4,27,3)$&2&[2,0,0]&$G(10,1)$&$D_5$\\
$(4,27,4)$&1&[1,0,0]&$G(20,4)$&$D_{10}$\\
$(4,28,1)$&1&[1,0,0]&$G(12,2)$&$C_{12}$\\
$(4,28,2)$&1&[1,0,0]&$G(24,6)$&$D_{12}$\\
$(4,29,1)$&3&[1,0,2]&$G(18,3)$&$C_3\times S_3$\\
$(4,29,2)$&2&[1,0,1]&$G(36,12)$&$C_6\times S_3$\\
$(4,29,3)$&5&[2,0,3]&$G(36,10)$&$S_3^2$\\
$(4,29,4)$&4&[2,0,2]&$G(36,9)$&$C_3^2\rtimes C_4$\\
$(4,29,5)$&2&[1,0,1]&$G(72,46)$&$C_2\times S_3^2$\\
$(4,29,6)$&3&[2,0,1]&$G(72,45)$&$C_2\times (C_3^2\rtimes C_4)$\\
$(4,29,7)$&4&[2,0,2]&$G(72,40)$&$S_3^2\rtimes C_2$\\
$(4,29,8)$&4&[2,0,2]&$G(72,40)$&$S_3^2\rtimes C_2$\\
$(4,29,9)$&3&[2,0,1]&$G(144,186)$&$C_2\times (S_3^2\rtimes C_2)$\\
$(4,30,1)$&1&[1,0,0]&$G(12,1)$&$C_3\rtimes C_4$\\
$(4,30,2)$&1&[1,0,0]&$G(24,5)$&$C_4\times S_3$\\
$(4,30,3)$&1&[1,0,0]&$G(24,10)$&$C_3\times D_4$\\
$(4,30,4)$&2&[2,0,0]&$G(24,8)$&$(C_6\times C_2)\rtimes C_2$
\end{tabular}
\ \ 
\begin{tabular}{lllll}
GAP ID & \# & $[s,r,u]$ & $G(n,i)$ & \\\hline
$(4,30,5)$&1&[1,0,0]&$G(36,6)$&$C_3\times (C_3\rtimes C_4)$\\
$(4,30,6)$&1&[1,0,0]&$G(48,38)$&$D_4\times S_3$\\
$(4,30,7)$&1&[1,0,0]&$G(72,30)$&$C_3\times ((C_6\times C_2)\rtimes C_2)$\\
$(4,30,8)$&1&[1,0,0]&$G(72,23)$&$(C_6\times S_3)\rtimes C_2$\\
$(4,30,9)$&1&[1,0,0]&$G(72,21)$&$(C_3\times (C_3\rtimes C_4))\rtimes C_2$\\
$(4,30,10)$&1&[1,0,0]&$G(144,154)$&$(C_2\times S_3^2)\rtimes C_2$\\
$(4,30,11)$&1&[1,0,0]&$G(144,136)$&$(C_2\times (C_3^2\rtimes C_4))\rtimes C_2$\\
$(4,30,12)$&2&[2,0,0]&$G(144,115)$&$(C_2\times (C_3^2\rtimes C_4))\rtimes C_2$\\
$(4,30,13)$&1&[1,0,0]&$G(288,889)$&$(C_2^2\times S_3^2)\rtimes C_2$\\
$(4,31,1)$&4&[2,2,0]&$G(20,3)$&$C_5\rtimes C_4$\\
$(4,31,2)$&2&[1,1,0]&$G(40,12)$&$C_2\times (C_5\rtimes C_4)$\\
$(4,31,3)$&2&[2,0,0]&$G(60,5)$&$A_5$\\
$(4,31,4)$&2&[1,1,0]&$G(120,34)$&$S_5$\\
$(4,31,5)$&2&[1,1,0]&$G(120,34)$&$S_5$\\
$(4,31,6)$&2&[2,0,0]&$G(120,35)$&$C_2\times A_5$\\
$(4,31,7)$&2&[1,1,0]&$G(240,189)$&$C_2\times S_5$\\
$(4,32,1)$&2&[1,0,1]&$G(8,4)$&$Q_8$\\
$(4,32,2)$&2&[1,0,1]&$G(16,6)$&$C_8\rtimes C_2$\\
$(4,32,3)$&2&[1,0,1]&$G(16,8)$&$QD_8$\\
$(4,32,4)$&2&[1,0,1]&$G(16,13)$&$(C_4\times C_2)\rtimes C_2$\\
$(4,32,5)$&3&[1,0,2]&$G(24,3)$&$\SL(2,3)$\\
$(4,32,6)$&2&[1,0,1]&$G(32,43)$&$(C_2\times D_4)\rtimes C_2$\\
$(4,32,7)$&2&[1,0,1]&$G(32,7)$&$(C_8\rtimes C_2)\rtimes C_2$\\
$(4,32,8)$&2&[1,0,1]&$G(32,11)$&$(C_4\times C_4)\rtimes C_2$\\
$(4,32,9)$&5&[3,0,2]&$G(32,6)$&$((C_4\times C_2)\rtimes C_2)\rtimes C_2$\\
$(4,32,10)$&2&[1,0,1]&$G(32,49)$&$(C_2\times D_4)\rtimes C_2$\\
$(4,32,11)$&3&[1,0,2]&$G(48,29)$&$\GL(2,3)$\\
$(4,32,12)$&2&[1,0,1]&$G(64,134)$&$((C_2\times D_4)\rtimes C_2)\rtimes C_2$\\
$(4,32,13)$&4&[2,0,2]&$G(64,32)$&$((C_8\rtimes C_2)\rtimes C_2)\rtimes C_2$\\
$(4,32,14)$&4&[2,0,2]&$G(64,138)$&\\
$(4,32,15)$&2&[1,0,1]&$G(64,34)$&\\
$(4,32,16)$&3&[1,0,2]&$G(96,204)$&$((C_2\times D_4)\rtimes C_2)\rtimes C_3$\\
$(4,32,17)$&2&[1,0,1]&$G(128,928)$&$D_4^2\rtimes C_2$\\
$(4,32,18)$&3&[1,0,2]&$G(192,201)$&\\
$(4,32,19)$&3&[1,0,2]&$G(192,1493)$&\\
$(4,32,20)$&3&[1,0,2]&$G(192,1494)$&\\
$(4,32,21)$&3&[1,0,2]&$G(384,5602)$&\\
$(4,33,1)$&1&[0,0,1]&$G(24,11)$&$C_3\times Q_8$\\
$(4,33,2)$&1&[0,1,0]&$G(24,1)$&$C_3\rtimes C_8$\\
$(4,33,3)$&1&[0,0,1]&$G(24,3)$&$\SL(2,3)$\\
$(4,33,4)$&1&[0,0,1]&$G(48,17)$&$(C_3\times Q_8)\rtimes C_2$\\
$(4,33,5)$&1&[0,0,1]&$G(48,33)$&$\SL(2,3)\rtimes C_2$\\
$(4,33,6)$&1&[0,0,1]&$G(48,29)$&$\GL(2,3)$\\
$(4,33,7)$&1&[0,0,1]&$G(72,25)$&$C_3\times \SL(2,3)$\\
$(4,33,8)$&1&[0,0,1]&$G(96,201)$&$(\SL(2,3)\rtimes C_2)\rtimes C_2$\\
$(4,33,9)$&1&[0,0,1]&$G(96,193)$&$\GL(2,3)\rtimes C_2$\\
$(4,33,10)$&1&[0,0,1]&$G(96,67)$&$\SL(2,3)\rtimes C_4$\\
$(4,33,11)$&1&[0,0,1]&$G(144,125)$&$(C_3\times \SL(2,3))\rtimes C_2$\\
$(4,33,12)$&1&[0,0,1]&$G(192,988)$&$(\GL(2,3)\rtimes C_2)\rtimes C_2$\\
$(4,33,13)$&1&[0,0,1]&$G(288,860)$&\\
$(4,33,14)$&2&[0,0,2]&$G(576,8277)$&\\
$(4,33,15)$&1&[0,0,1]&$G(576,8282)$&\\
$(4,33,16)$&1&[0,0,1]&\multicolumn{2}{l}{$G(1152,157478)$}
\\ \\
\end{tabular}
\end{table}
\end{center}
\newpage

Let $M$ be a $G$-lattice. 
$M$ is called a {\it permutation} $G$-lattice if $M$ has a $\bZ$-basis
permuted by $G$, i.e. $M\simeq \oplus_{1\leq i\leq m}\bZ[G/H_i]$ 
for some subgroups $H_1,\ldots,H_m$ of $G$.
$M$ is called a {\it stably permutation} $G$-lattice if $M\oplus P\simeq P^\prime$ 
for some permutation $G$-lattices $P$ and $P^\prime$. 
$M$ is called {\it invertible} if it is a direct summand of a permutation $G$-lattice, 
i.e. $P\simeq M\oplus M^\prime$ for some permutation $G$-lattice 
$P$ and a $G$-lattice $M^\prime$. 
$M$ is called {\it coflabby} if $H^1(H,M)=0$
for any subgroup $H$ of $G$. 
$M$ is called {\it flabby} if $\widehat H^{-1}(H,M)=0$ 
for any subgroup $H$ of $G$ where $\widehat H$ is the Tate cohomology 
(see Section \ref{sePre}). 

\begin{definition}[{
see \cite[Section 1]{EM74}, \cite[Section 4.7]{Vos98}}]
Let $\cC(G)$ be the category of all $G$-lattices. 
Let $\cS(G)$ be the full subcategory of $\cC(G)$ of all permutation $G$-lattices 
and $\cD(G)$ be the full subcategory of $\cC(G)$ of all invertible $G$-lattices.
Let 
\begin{align*}
\cH^i(G)=\{M\in \cC(G)\mid \widehat H^i(H,M)=0\ {\rm for\ any}\ H\leq G\}\ (i=\pm 1)
\end{align*}
be the class of ``$\widehat H^i$-vanish'' $G$-lattices 
where $\widehat H^i$ is the Tate cohomology (see Section \ref{sePre}). 
Then one has the inclusions 
$\cS(G)\subset \cD(G)\subset \cH^i(G)\subset \cC(G)$ $(i=\pm 1)$. 
\end{definition}

\begin{definition}[The commutative monoid $\cC(G)/\cS(G)$]\label{defCM}
We say that two $G$-lattices $M_1$ and $M_2$ are {\it similar} 
if there exist permutation $G$-lattices $P_1$ and $P_2$ such that 
$M_1\oplus P_1\simeq M_2\oplus P_2$. 
We denote the similarity class of $M$ by $[M]$. 
The set of similarity classes $\cC(G)/\cS(G)$ becomes a 
commutative monoid 
(with respect to the sum $[M_1]+[M_2]:=[M_1\oplus M_2]$ 
and the zero $0=[P]$ where $P\in \cS(G)$). 
\end{definition}

For a $G$-lattice $M$, there exists a short exact sequence of $G$-lattices
$0 \rightarrow M \rightarrow P \rightarrow F \rightarrow 0$
where $P$ is permutation and $F$ is flabby which is called a 
{\it flasque resolution} of $M$ (see Theorem \ref{thEM}). 
The similarity class $[F]\in \cC(G)/\cS(G)$ of $F$ is determined uniquely and is called 
{\it the flabby class} of $M$. 
We denote the flabby class $[F]$ of $M$ by $[M]^{fl}$. 
We say that $[M]^{fl}$ is {\it invertible} if $[M]^{fl}=[E]$ for some 
invertible $G$-lattice $E$ (see Definition \ref{defFlabby}). 

The flabby class $[M]^{fl}$ plays crucial role in the rationality problem for 
$L(M)^G$ as follows (see Voskresenskii's fundamental book \cite[Section 4.6]{Vos98}, 
see also e.g. Swan \cite{Swa83}, Kunyavskii \cite[Section 2]{Kun90}, 
Lemire, Popov and Reichstein \cite[Section 2]{LPR06}, Kang \cite{Kan12}, 
Yamasaki \cite{Yam12}):  
\begin{theorem}[Endo and Miyata, Voskresenskii, Saltman]\label{thEM73}
Let $L/k$ be a finite Galois extension with Galois group $G={\rm Gal}(L/k)$ 
and $M$ and $M^\prime$ be $G$-lattices.\\
{\rm (i)} $(${\rm Endo and Miyata} \cite[Theorem 1.6]{EM73}$)$ 
$[M]^{fl}=0$ if and only if $L(M)^G$ is stably $k$-rational.\\
{\rm (ii)} $(${\rm Voskresenskii} \cite[Theorem 2]{Vos74}$)$ 
$[M]^{fl}=[M^\prime]^{fl}$ if and only if $L(M)^G$ and $L(M^\prime)^G$ 
are stably $k$-isomorphic, i.e. there exist algebraically independent 
elements $x_1,\ldots,x_m$ over $L(M)^G$ and 
$y_1,\ldots,y_n$ over $L(M^\prime)^G$ such that 
$L(M)^G(x_1,\ldots,x_m)\simeq L(M^\prime)^G(y_1,\ldots,y_n)$.\\
{\rm (iii)} $(${\rm Saltman} \cite[Theorem 3.14]{Sal84a}$)$ 
$[M]^{fl}$ is invertible if and only if $L(M)^G$ is 
retract $k$-rational.
\end{theorem}
\begin{theorem}[{Colliot-Th\'{e}l\`{e}ne and Sansuc \cite[Corollaire 1]{CTS77}}]
\label{thCTS77}
Let $G$ be a finite group. 
The following conditions are equivalent:\\
{\rm (i)} $[J_G]^{fl}$ is coflabby;\\
{\rm (ii)} any Sylow subgroup of $G$ is cyclic or generalized quaternion $Q_{4n}$ 
of order $4n$ $(n\geq 2)$;\\
{\rm (iii)} any abelian subgroup of $G$ is cyclic;\\
{\rm (iv)} $H^3(H,\bZ)=0$ for any subgroup $H$ of $G$.
\end{theorem}
\begin{remark} 
(1) 
It is known that 
each of the conditions (i)--(iv) of Theorem \ref{thCTS77} 
is equivalent to the condition that $G$ has periodic cohomology, 
i.e. there exists 
$q\neq 0$ and $u\in\widehat H^q(G,\bZ)$ such that the cup product map 
$u\cup -: \widehat H^n(G,\bZ)\rightarrow \widehat H^{n+q}(G,\bZ)$ 
is an isomorphism for any $n\in\bZ$ (see \cite[Theorem 11.6]{CE56}).\\
(2)  $H^3(H,\bZ)\simeq H^1(H,[J_G]^{fl})$ for any subgroup $H$ of $G$ 
(see \cite[Theorem 7]{Vos70} and \cite[Proposition 1]{CTS77}). 
\end{remark}
\begin{theorem}[{Endo and Miyata \cite[Theorem 2.1]{EM82}}]\label{thEM82}
Let $G$ be a finite group. 
The following conditions are equivalent:\\
{\rm (i)} $\cH^1(G)\cap \cH^{-1}(G)=\cD(G)$, i.e. 
any flabby and coflabby $G$-lattice is invertible;\\
{\rm (ii)} $[J_G\otimes_\bZ J_G]^{fl}=[[J_G]^{fl}]^{fl}$ is invertible;\\
{\rm (iii)} any p-Sylow subgroup of $G$ is cyclic for odd $p$ and 
cyclic or dihedral $($including Klein's four group$)$ for $p=2$.
\end{theorem}
Note that $[J_G]^{fl}=[J_G\otimes_\bZ J_G]$ (see \cite[Section 2]{EM82}). 

For $G$-lattice $M$, 
it is not difficult to see 
\begin{align*}
\textrm{permutation}\ \ 
\Rightarrow\ \ 
&\textrm{stably\ permutation}\ \ 
\Rightarrow\ \ 
\textrm{invertible}\ \ 
\Rightarrow\ \ 
\textrm{flabby\ and\ coflabby}\\
&\hspace*{8mm}\Downarrow\hspace*{34mm} \Downarrow\\
&\hspace*{7mm}[M]^{fl}=0\hspace*{10mm}\Rightarrow\hspace*{5mm}[M]^{fl}\ 
\textrm{is\ invertible}.
\end{align*}

The above implications in each step cannot be reversed. 
Swan \cite{Swa60} gave an example of $Q_8$-lattice 
$M$ of rank $8$ which is not permutation but stably 
permutation: $M\oplus \bZ\simeq \bZ[Q_8]\oplus \bZ$. 
This also indicates that the direct sum cancellation fails 
(see also Theorem \ref{thKS}). 
Colliot-Th\'{e}l\`{e}ne 
and Sansuc \cite[Remarque R1]{CTS77} 
and \cite[Remarque R4]{CTS77} 
gave examples of $S_3$-lattice $M$ of rank $4$ which 
is not permutation but stably permutation: 
$M\oplus \bZ\simeq \bZ[S_3/\langle\sigma\rangle]
\oplus\bZ[S_3/\langle\tau\rangle]$ 
where $S_3=\langle\sigma,\tau\rangle$ 
(see also Table $8$ of Theorem \ref{thCFSP}) 
and 
of $F_{20}$-lattice $[J_{F_{20}/C_4}]^{fl}$ 
of the Chevalley module $J_{F_{20}/C_4}$ of rank $4$ 
which is  not stably permutation but invertible 
(see also Theorem \ref{thE0} and Theorem \ref{th1M} (ii), (iv) and (v)). 
By Theorems \ref{th13-1}, \ref{thEM73} (ii) and \ref{thCTS77}, 
the flabby class $[J_{Q_8}]^{fl}$ of the Chevalley module $J_{Q_8}$ 
of rank $7$ is not invertible but flabby and coflabby 
(we may take $[J_{Q_8}]^{fl}$ of rank $9$, see Example \ref{exQ8}). 
The inverse direction of the vertical implication holds 
if $M$ is coflabby (see Lemma \ref{lemSL}).


By using the interpretation as in Theorem \ref{thEM73}, 
Theorems \ref{th13-1}, \ref{th13-2}, \ref{th15}, \ref{thS} and \ref{thA} 
may be obtained by the following theorems:
\begin{theorem}[{Endo and Miyata \cite[Theorem 1.5]{EM74}}]\label{thEM7415}
Let $G$ be a finite group. The following conditions are equivalent:\\
{\rm (i)} $[J_G]^{fl}$ is invertible;\\
{\rm (ii)} all the Sylow subgroups of $G$ are cyclic;\\
{\rm (iii)} $\cH^{-1}(G)=\cH^1(G)=\cD(G)$, i.e. any flabby $($resp. coflabby$)$ 
$G$-lattice is invertible.
\end{theorem}
\begin{theorem}[{Endo and Miyata \cite[Theorem 2.3]{EM74}, 
see also \cite[Proposition 3]{CTS77}}]\label{thEM74M}
Let $G$ be a finite group. The following conditions are equivalent:\\
{\rm (i)} $[J_G]^{fl}=0$;\\
{\rm (ii)} $[J_G]^{fl}$ is of finite order in $\cC(G)/\cS(G)$;\\
{\rm (iii)} all the Sylow subgroups of $G$ are cyclic and $H^4(G,\bZ)\simeq \widehat H^0(G,\bZ)$;\\
{\rm (iv)} $G=C_m$ or $G=C_n\times \langle\sigma,\tau\mid\sigma^k=\tau^{2^d}=1,
\tau\sigma\tau^{-1}=\sigma^{-1}\rangle$ where $d\geq 1, k\geq 3$, 
$n,k$: odd, and ${\rm gcd}\{n,k\}=1$;\\
{\rm (v)} $G=\langle s,t\mid s^m=t^{2^d}=1, tst^{-1}=s^r, m: odd,\ 
r^2\equiv 1\pmod{m}\rangle.$
\end{theorem}
\begin{theorem}[{Endo \cite[Theorem 3.1]{End11}}]\label{thE0}
Let $G$ be a non-abelian group whose Sylow subgroups are all cyclic.
Let $H$ be a non-normal subgroup of $G$ which contains no normal subgroup 
of $G$ except $\{1\}$. 
Then $[J_{G/H}]^{fl}$ is invertible, and the following conditions are equivalent:\\
{\rm (i)} $[J_{G/H}]^{fl}=0$;\\
{\rm (ii)} $[J_{G/H}]^{fl}$ is of finite order in $\cC(G)/\cS(G)$;\\
{\rm (iii)} 
$G=D_n$ with $n$ odd $(n\geq 3)$ 
or $G=C_m\times D_n$ where $m,n$ are odd, 
$m,n\geq 3$, ${\rm gcd}\{m,n\}=1$, and $H\leq D_n$ is of order $2$;\\
{\rm (iv)} 
$H=C_2$ and $G\simeq C_r\rtimes H$, $r\geq 3$ odd, where 
$H$ acts non-trivially on $C_r$. 
\end{theorem}
A partial result of Theorem \ref{thE0} was given by 
Colliot-Th\'{e}l\`{e}ne and Sansuc \cite[Remarque R4]{CTS77}. 
\begin{theorem}[{Colliot-Th\'{e}l\`{e}ne and Sansuc \cite[Proposition 9.1]{CTS87}, 
\cite[Theorem 3.1]{LeB95}, 
\cite[Proposition 0.2]{CK00}, \cite{LL00}, 
Endo \cite[Theorem 4.1]{End11}, see also 
\cite[Remark 4.2 and Theorem 4.3]{End11}}]
Let $n\geq 2$ be an integer.\\
{\rm (i)} $[J_{S_n/S_{n-1}}]^{fl}$ is invertible if and only if $n$ is a prime.\\
{\rm (ii)} $[J_{S_n/S_{n-1}}]^{fl}=0$ if and only if $n=2,3$.\\
{\rm (iii)} $[J_{S_n/S_{n-1}}]^{fl}$ is of finite order in $\cC(G)/\cS(G)$ 
if and only if $n=2,3$.
\end{theorem}
\begin{theorem}[{Endo \cite[Theorem 4.5]{End11}}]
Let $n\geq 3$ be an integer.\\
{\rm (i)} $[J_{A_n/A_{n-1}}]^{fl}$ is invertible if and only if $n$ is a prime.\\
{\rm (ii)} $[J_{A_n/A_{n-1}}]^{fl}$ is of finite order in $\cC(G)/\cS(G)$ 
if and only if $n=3,5$.
\end{theorem}
Note that $[J_{A_3/A_2}]^{fl}=[J_{C_3}]^{fl}=0$. 
By \cite[Corollary 3.3]{Dre75}, $[J_{A_5/A_4}]^{fl}$ is of finite order in 
$\cC(G)/\cS(G)$. 
Indeed, we get $[J_{A_5/A_4}]^{fl}=0$ 
(see Theorems \ref{th1} and \ref{th11}, Table $5$, Corollary \ref{corA} 
and Theorem \ref{th1M} (i) below). 


\begin{definition}[The $G$-lattice $M_G$ of a finite subgroup $G$ of 
$\GL(n,\bZ)$]\label{defMG} 
Let $G$ be a finite subgroup of $\GL(n,\bZ)$. 
The $G$-lattice $M_G$ of rank $n$ 
is defined to be the $G$-lattice with a $\bZ$-basis $\{u_1,\ldots,u_n\}$ 
on which $G$ acts by $\sigma(u_i)=\sum_{j=1}^n a_{i,j}u_j$ for any $
\sigma=[a_{i,j}]\in G$. 
\end{definition}

The stably rational classification of 
$k$-tori of dimensions $4$ and $5$ 
(Theorem \ref{th1} and Theorem \ref{th2}) 
may be obtained by the following two theorems respectively. 
\begin{theorem}\label{th1M}
Let $G$ be a finite subgroup of $\GL(4,\bZ)$ and 
$M_G$ be the $G$-lattice of rank $4$ as in Definition \ref{defMG}.\\
{\rm (i)} 
$[M_G]^{fl}=0$ if and only if 
$G$ is conjugate to one of the $487$ groups which are not in 
{\rm Tables} $2$, $3$ and $4$.\\
{\rm (ii)} 
$[M_G]^{fl}$ is not zero but invertible 
if and only if $G$ is conjugate to one of the $7$ groups which are 
given as in {\rm Table} $2$.\\
{\rm (iii)} 
$[M_G]^{fl}$ is not invertible if and only if 
$G$ is conjugate to one of the $216$ groups which are given as 
in {\rm Tables} $3$ and $4$.\\
{\rm (iv)} 
$[M_G]^{fl}=0$ if and only if $[M_G]^{fl}$ is of finite order in $\cC(G)/\cS(G)$.\\
{\rm (v)} For the group $G\simeq S_5$ of the {\rm GAP ID} $(4,31,5,2)$ in {\rm (ii)}, 
we have 
\begin{align*}
-[M_G]^{fl}=[J_{S_5/S_4}]^{fl}\neq 0.
\end{align*}
{\rm (vi)} For the group $G\simeq F_{20}$ of the {\rm GAP ID} $(4,31,1,4)$ in {\rm (ii)}, 
we have 
\begin{align*}
-[M_G]^{fl}=[J_{F_{20}/C_4}]^{fl}\neq 0.
\end{align*}

\end{theorem}

\begin{theorem}\label{th2M}
Let $G$ be a finite subgroup of $\GL(5,\bZ)$ and 
$M_G$ be the $G$-lattice of rank $5$ as in Definition \ref{defMG}.\\
{\rm (i)} 
$[M_G]^{fl}=0$ if and only if 
$G$ is conjugate to one of the $3051$ groups which are not in 
{\rm Tables} $11$, $12$, $13$, $14$ and $15$.\\
{\rm (ii)} 
$[M_G]^{fl}$ is not zero but invertible 
if and only if $G$ is conjugate to one of the $25$ groups which are given as 
in {\rm Table} $11$.\\
{\rm (iii)} 
$[M_G]^{fl}$ is not invertible if and only if 
$G$ is conjugate to one of the $3003$ groups which are given as 
in {\rm Tables} $12$, $13$, $14$ and $15$.\\
{\rm (iv)} 
$[M_G]^{fl}=0$ if and only if $[M_G]^{fl}$ is of finite order in $\cC(G)/\cS(G)$.
\end{theorem}


\begin{remark}
{\rm (i)} By the interpretation as in Theorem \ref{thEM73}, 
Theorem \ref{th1M} (v), (vi) claims that 
the corresponding two tori $T$ and $T'$ of dimension $4$ 
are not stably $k$-rational and are not stably $k$-isomorphic
but the torus $T\times T'$ of dimension $8$ 
is stably $k$-rational.\\
{\rm (ii)} When $[M]^{fl}$ is invertible, 
the inverse element of $[M]^{fl}$ is given by $-[M]^{fl}=[[M]^{fl}]^{fl}$ 
(see Lemma \ref{lemSwa}). 
Hence Theorem \ref{th1M} (v) also claims that 
$[[M_G]^{fl}]^{fl}=[J_{S_5/S_4}]^{fl}$ and 
$[[J_{S_5/S_4}]^{fl}]^{fl}=[M_G]^{fl}$. 
\end{remark}

We will give proofs of 
Theorem \ref{th1M} and Theorem \ref{th2M} in 
Section \ref{seProof1} and Section \ref{seProof2} respectively.\\

By using the algorithms in Section \ref{seAlg}, 
we show the following theorem: 

\begin{theorem-nn}[Theorem \ref{thfac} and Theorem \ref{thCFSP}]
Let $G$ be a finite subgroup of $\GL(n,\bZ)$ and 
$M_G$ be the $G$-lattice of rank $n$ as in Definition \ref{defMG}.\\
{\rm (i)}\ When $n \leq 3$, $M_G$ is flabby and coflabby if and only if
$M_G$ is permutation.\\
{\rm (ii)}\ When $n=4$, $M_G$ is flabby and coflabby if and only if
$M_G$ is permutation or the {\rm GAP ID} of $G$ is one of
$(4,14,2,2), (4,14,3,3), (4,14,3,4), (4,14,8,2)$.\\
$($There are $11$ conjugacy classes of subgroups of $S_4$ 
and hence $15$ flabby and coflabby $G$-lattices of rank $4$ in total.$)$
{\rm (iii)}\ When $n=5$, $M_G$ is flabby and coflabby if and only if
$M_G$ is permutation or the {\rm CARAT ID} of $G$ is one of
$(5,218,4), (5,911,4), (5,918,4), (5,931,4)$.\\
$($There are $19$ conjugacy classes of subgroups of $S_5$ 
and hence $23$ flabby and coflabby $G$-lattices of rank $5$ in total.$)$\\
{\rm (iv)}\ When $n=6$, $M_G$ is flabby and coflabby if and only if
$M_G$ is permutation or the {\rm CARAT ID} of $G$ is one of the $50$ triples 

\begin{center}
{\rm
\begin{tabular}{lllll}
$(6,159,14)$,&$(6,161,14)$,&$(6,161,28)$,&$(6,197,14)$,&$(6,226,14)$,\\
$(6,226,40)$,&$(6,231,39)$,&$(6,238,27)$,&$(6,246,21)$,&$(6,366,27)$,\\
$(6,1087,20)$,&$(6,1090,18)$,&$(6,1142,8)$,&$(6,2043,4)$,&$(6,2051,9)$,\\
$(6,2068,6)$,&$(6,2069,6)$,&$(6,2069,12)$,&$(6,2070,12)$,&$(6,2079,14)$,\\
$(6,2079,28)$,&$(6,2088,18)$,&$(6,2105,12)$,&$(6,2154,26)$,&$(6,2156,40)$,\\
$(6,2156,80)$,&$(6,2188,39)$,&$(6,2968,4)$,&$(6,2969,4)$,&$(6,2969,8)$,\\
$(6,2977,6)$,&$(6,3068,7)$,&$(6,3068,8)$,&$(6,3071,7)$,&$(6,3071,8)$,\\
$(6,3073,7)$,&$(6,3073,8)$,&$(6,3073,15)$,&$(6,3073,16)$,&$(6,3076,7)$,\\
$(6,3076,8)$,&$(6,3091,11)$,&$(6,3091,12)$,&$(6,5210,14)$,&$(6,5262,11)$,\\
$(6,5321,6)$,&$(6,5421,6)$,&$(6,5475,6)$,&$(6,5477,11)$,&$(6,5487,11)$.
\end{tabular}
}
\end{center}
\hfill\break
$($There are $56$ conjugacy classes of subgroups of $S_6$ 
and hence $106$ flabby and coflabby $G$-lattices of rank $6$ in total.$)$\\
{\rm (v)}\ 
When $n \leq 6$, $M$ is flabby and coflabby if and only if
$M$ is stably permutation. 
Indeed, flabby and coflabby $G$-lattices $M$ which are not permutation 
in {\rm (ii)}, {\rm (iii)}, {\rm (iv)} are stably permutation as in Table $8$.
\end{theorem-nn}
\begin{definition}[Decomposition type]
Let $G$ be a finite group and $M$ be a $G$-lattice. 
A $G$-lattice $M$ is said to be {\it decomposable} if there exist non-trivial 
$G$-lattices $U_1$ and $U_2$ such that $M\simeq U_1\oplus U_2$. 
A $G$-lattice is said to be {\it indecomposable} 
if it is not decomposable. 
When $M$ decomposes into indecomposable 
$G$-lattices $M\simeq U_1\oplus\cdots\oplus U_r$ of rank $n_1,\dots,n_r$, 
we say that a {\it decomposition type} ${\rm DT}(M)$ of $M$ is $(n_1,\dots,n_r)$.
\end{definition}
%
%
For $n\leq 6$, the number of $G$-lattices $M_G$ of rank $n$ 
for a given decomposition type ${\rm DT}(M_G)$ is as follows:
\vspace*{3mm}

\noindent\ 
\begin{tabular}{l|c|c}
${\rm DT}(M_G)$ & $(1)$ & Total\\\hline
$\# M_G$ & 2 & 2
\end{tabular}\quad 
\begin{tabular}{l|cc|c}
${\rm DT}(M_G)$ & $(1,1)$ & $(2)$& Total\\\hline
$\# M_G$ & 4 & 9 & 13
\end{tabular}\quad 
\begin{tabular}{l|ccc|c}
${\rm DT}(M_G)$ & $(1,1,1)$ & $(2,1)$ & $(3)$ & Total\\\hline
$\# M_G$ & 8 & 31 & 34 & 73
\end{tabular}\\

\noindent\ 
\begin{tabular}{l|ccccc|c}
${\rm DT}(M_G)$ & $(1,1,1,1)$ & $(2,1,1)$ & $(2,2)$ & $(3,1)$ & $(4)$ & Total\\\hline
$\# M_G$ & 16 & 96 & 175 & 128 & 295 & 710
\end{tabular}\\

\noindent\ 
\begin{tabular}{l|ccccccc|c}
${\rm DT}(M_G)$ & $(1^5)$ & $(2,1^3)$ & $(2^2,1)$ & $(3,1^2)$ & $(3,2)$ & $(4,1)$ & $(5)$ & Total\\\hline
$\# M_G$ & 32 & 280 & 1004 & 442 & {\bf 1480} & {\bf 1400} & 1452 & ${\bf 6090} \atop (6079)$
\end{tabular}\\

\noindent\ 
\begin{tabular}{l|ccccccccccc|c}
${\rm DT}(M_G)$ & $(1^6)$ & $(2,1^4)$ & $(2^2,1^2)$ & $(2^3)$ & 
$(3,1^3)$ & $(3,2,1)$ & $(3^2)$ & $(4,1^2)$ & $(4,2)$ & $(5,1)$ & $(6)$ & Total\\\hline
$\# M_G$ & 68 & 824 & 4862 & 6878 & 1466 & {\bf 10662} & {\bf 4235} & {\bf 5944} & 
21573 & {\bf 9931} & 18996 & ${\bf 85439} \atop (85308)$
\end{tabular}
\vspace*{4mm}

Let $G$ be a finite subgroup of $\GL(n,\bZ)$ and 
$M_G$ be the $G$-lattice of rank $n$ as in Definition \ref{defMG}. 
For $n\leq 4$, we will show that the Krull-Schmidt theorem holds for $M_G$, 
i.e. if $M_G\simeq M_1\oplus\cdots\oplus M_l\simeq N_1\oplus\cdots\oplus N_m$ 
for indecomposable $G$-lattices $M_i$ and $N_j$, 
then $l=m$ and, after a suitable renumbering of the $N_j$, $M_i\simeq N_i$ 
for any $1\leq i\leq m$. 
However, it turns out that the Krull-Schmidt theorem fails for $M_G$ 
when the rank of $M_G$ is $5$. 
We split the Krull-Schmidt theorem for $M_G$ into the following two parts: \\

(KS1) If $M_G\simeq M_1\oplus\cdots\oplus M_l\simeq N_1\oplus\cdots\oplus N_m$ 
for indecomposable $G$-lattices $M_i$ and $N_j$, 
then $l=m$ and, after a suitable renumbering of the $N_j$, 
rank $M_i=$ rank $N_i$ for any $1\leq i\leq m$;

(KS2) If $M_G\simeq M_1\oplus\cdots\oplus M_m\simeq N_1\oplus\cdots\oplus N_m$ 
for indecomposable $G$-lattices $M_i$ and $N_i$ 
with \rank $M_i=$ \rank $N_i$ for any $1\leq i\leq m$, 
then after a suitable renumbering of the $N_i$, 
$M_i\simeq N_i$ for any $1\leq i\leq m$.\\

Krull-Schmidt theorem holds for $M_G$ if and only if the conditions 
(KS1) and (KS2) hold for $M_G$. 

\begin{theorem-nn}[{Theorem \ref{thKS}}]
Let $G$ be a finite subgroup of $\GL(n,\bZ)$ and 
$M_G$ be the $G$-lattice of rank $n$ as in Definition \ref{defMG}.\\
{\rm (i)} When $n\leq 4$, the Krull-Schmidt theorem holds for $M_G$, 
i.e. if $M_G\simeq M_1\oplus\cdots\oplus M_l\simeq N_1\oplus\cdots\oplus N_m$ 
for indecomposable $G$-lattices $M_i$ and $N_j$, 
then $l=m$ and, after a suitable renumbering of the $N_j$, $M_i\simeq N_i$ 
for any $1\leq i\leq m$.\\
{\rm (ii)}When $n=5$, {\rm (KS2)} holds for $M_G$, and 
the Krull-Schmidt theorem fails for $M_G$ if and only if {\rm (KS1)} fails for $M_G$ if and only if the {\rm CARAT ID} of $G$ is one of the $11$ triples
\vspace*{-3mm}

{\footnotesize 
\begin{align*}
(5,188,4),(5,189,4),(5,190,6),(5,191,6),(5,192,6),(5,193,4),
(5,205,6),(5,218,8),(5,219,8),(5,220,4),(5,221,4). 
\end{align*}
}

\vspace*{-3mm}
\noindent
For the exceptional $11$ cases, 
the decomposition types of $M_G$ are $(3,2)$ and $(4,1)$ 
and $G$ is a subgroup of the group $C_2\times D_6$ of the {\rm CARAT ID} $(5,205,6)$.\\
{\rm (iii)} When $n=6$, {\rm (KS1)} fails for $M_G$ if and only if 
the {\rm CARAT ID} of $G$ is one of the $131$ triples 
\vspace*{-3mm}

{\footnotesize 
\begin{align*}
&(6,2013,8),(6,2018,4),(6,2023,6),(6,2024,6),(6,2025,6),
(6,2026,6),(6,2033,6),(6,2042,8),(6,2043,8),(6,2044,4),\\
&(6,2045,4),(6,2048,5),(6,2049,8),(6,2050,8),(6,2051,8),
(6,2052,8),(6,2058,5),(6,2059,5),(6,2067,5),(6,2068,5),\\
&(6,2069,5),(6,2069,11),(6,2070,9),(6,2071,9),(6,2072,10),
(6,2072,11),(6,2076,24),(6,2076,25),(6,2077,24),(6,2077,25),\\
&(6,2078,24),(6,2078,25),(6,2079,24),(6,2079,25),(6,2087,15),
(6,2088,15),(6,2089,17),(6,2089,18),(6,2094,9),(6,2102,24),\ \,\\
&(6,2102,25),(6,2105,9),(6,2106,9),(6,2107,10),(6,2107,11),
(6,2108,15),(6,2109,15),(6,2110,17),(6,2110,18),(6,2111,15),\\
&(6,2139,9),
\end{align*}
}
\vspace*{-3mm}
\vspace*{-3mm}
{\footnotesize 
\begin{align*}
&(6,40,4),(6,41,4),(6,44,6),(6,45,6),(6,47,4),
(6,53,4),(6,54,4),(6,54,8),(6,55,4),(6,63,4),\\
&(6,64,6),(6,65,4),(6,66,6),(6,67,6),(6,75,4),
(6,75,8),(6,76,8),(6,76,12),(6,77,8),(6,77,12),\\
&(6,78,4),(6,78,8),(6,79,6),(6,80,4),(6,81,8),
(6,81,12),(6,90,4),(6,99,4),(6,108,4),(6,108,8),\\
&(6,109,8),(6,109,12),(6,110,4),(6,111,6),(6,112,8),
(6,112,12),(6,113,4),(6,114,6),(6,115,6),(6,145,4),\\
&(6,2070,10),(6,2070,11),(6,2071,10),(6,2071,11),(6,2072,12),
(6,2072,13),(6,2076,26),(6,2076,27),(6,2077,26),(6,2077,27),\\
&(6,2078,26),(6,2078,27),(6,2079,26),(6,2079,27),(6,2087,16),
(6,2087,17),(6,2088,16),(6,2088,17),(6,2089,19),(6,2089,20),\\
&(6,2094,10),(6,2094,11),(6,2102,26),(6,2102,27),(6,2105,10),
(6,2105,11),(6,2106,10),(6,2106,11),(6,2107,12),(6,2107,13),\\
&(6,2108,16),(6,2108,17),(6,2109,16),(6,2109,17),(6,2110,19),
(6,2110,20),(6,2111,16),(6,2111,17),(6,2139,10),(6,2139,11).
\end{align*}
}

\vspace*{-3mm}
\noindent
For the former $51$ cases $($resp. the latter $80$ cases$)$, 
the decomposition types of $M_G$ are $(3,2,1)$ and $(4,1,1)$ 
$($resp. $(3,3)$ and $(5,1)$$)$ and $G$ is a subgroup of the group 
$C_2^2\times D_6$ of the {\rm CARAT ID} $(6,2139,9)$ 
$($resp. $D_6\times D_4$ of the {\rm CARAT ID} $(6,145,4)$$)$.\\
{\rm (iv)} When $n=6$, {\rm (KS2)} fails for $M_G$ if and only if 
the {\rm CARAT ID} of $G$ is one of the $18$ triples 
\vspace*{-3mm}

{\footnotesize 
\begin{align*}
&(6,2072,14),(6,2076,28),(6,2077,28),(6,2078,28),(6,2079,28),
(6,2089,21),(6,2102,28),(6,2107,14),(6,2110,21),(6,2295,2),\\
&(6,3045,3),(6,3046,3),(6,3047,3),(6,3052,5),
(6,3053,5),(6,3054,3),(6,3061,5),(6,3066,3).
\end{align*}
}

\vspace*{-5mm}
\noindent
For the former $10$ cases, 
the decomposition type of $M_G$ is $(4,2)$ 
and $G$ is the group $D_6$ of the {\rm CARAT ID} $(6,2295,2)$ 
or a subgroup of the $3$ groups $C_2\times D_6$ of 
the {\rm CARAT ID}s $(6,2102,28)$, $(6,2107,14)$ and $(6,2110,21)$. 
For the latter $8$ cases, the decomposition type of $M_G$ 
is $(5,1)$ and $G$ is a subgroup of the group 
$C_2\times S_5$ of the {\rm CARAT ID} $(6,3054,3)$.
\end{theorem-nn}

\bigskip

The following flow chart presents the structure of the GAP algorithms 
which will be given in Section \ref{seAlg}. 

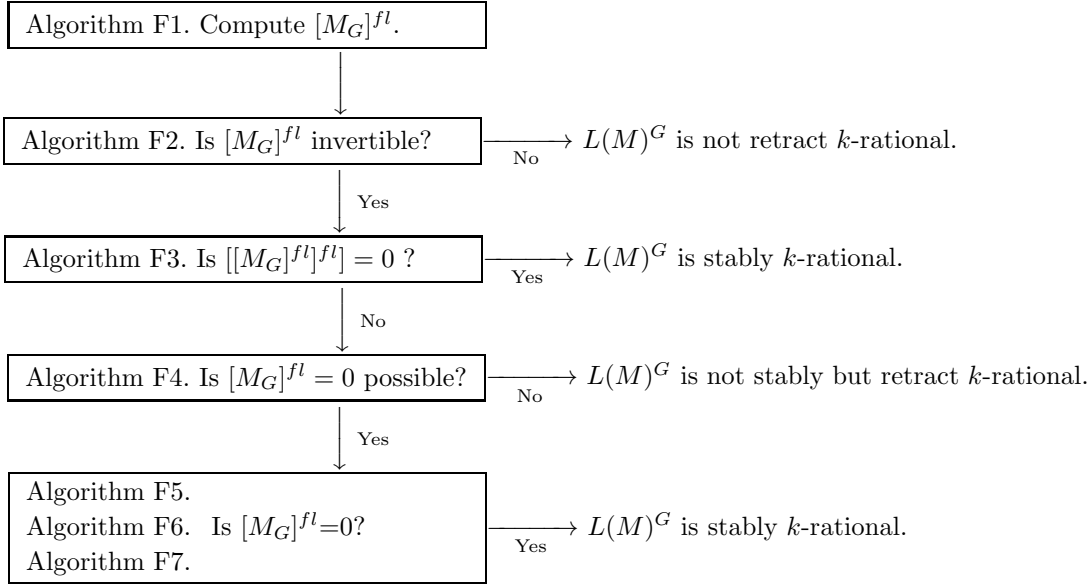
\vspace*{1cm}

\begin{center}
\fbox{
\begin{minipage}{6cm}
Algorithm F1.\ 
Compute $[M_G]^{fl}$.
\end{minipage}}\hspace*{80mm}

$\bigg\downarrow$ \hspace*{4mm}\hspace*{50mm}

\fbox{
\begin{minipage}{6cm}
Algorithm F2.\ 
Is $[M_G]^{fl}$ invertible?
\end{minipage}}$\xrightarrow[\textrm{No}]{\hspace*{1cm}}$ 
$L(M)^G$ is not retract $k$-rational.\hspace*{18mm}

$\bigg\downarrow$ {\scriptsize Yes}\hspace*{50mm}

\fbox{
\begin{minipage}{6cm}
Algorithm F3.\ 
Is $[[M_G]^{fl}]^{fl}]=0$ ?
\end{minipage}}$\xrightarrow[\textrm{Yes}]{\hspace*{1cm}}$
$L(M)^G$ is stably $k$-rational.\hspace*{25mm}

$\bigg\downarrow$ {\scriptsize No}\hspace*{50mm}

\fbox{
\begin{minipage}{6cm}
Algorithm F4.\ 
Is $[M_G]^{fl}=0$ possible?
\end{minipage}}$\xrightarrow[\textrm{No}]{\hspace*{1cm}}$
$L(M)^G$ is not stably but retract $k$-rational.

$\bigg\downarrow$ {\scriptsize Yes}\hspace*{50mm}

\fbox{
\begin{minipage}{6cm}
Algorithm F5.\ \\
Algorithm F6. \ 
Is $[M_G]^{fl}$=0?\\
Algorithm F7.\ 
\end{minipage}}$\xrightarrow[\textrm{Yes}]{\hspace*{1cm}}$
$L(M)^G$ is stably $k$-rational.\hspace*{24mm}
\end{center}

\vspace*{1cm}

Using the algorithms as in Section \ref{seAlg} and Section \ref{seFC}, 
we may verify the following isomorphism 
which gives the smallest example exhibiting the failure of 
the Krull-Schmidt theorem for permutation $G$-lattices 
(see Section \ref{seKSfail} and Dress's paper \cite[Proposition 9.6]{Dre73}): 
\begin{proposition-nn}[The Krull-Schmidt theorem fails for permutation $D_6$-lattices, see Proposition \ref{propKSfailD6}]
Let $D_6$ be the dihedral group of order $12$ and 
$\{1\}$, $C_2^{(1)}$, $C_2^{(2)}$, $C_2^{(3)}$, $C_3$, $C_2^2$, 
$C_6$, $S_3^{(1)}$, $S_3^{(2)}$ and $D_6$ be the 
conjugacy classes of subgroups of $D_6$. 
Then the following isomorphism of permutation $D_6$-lattices holds: 
\begin{align*}
& ~{} \bZ[D_6]\oplus\bZ[D_6/C_2^2]^{\oplus 2}\oplus\bZ[D_6/C_6]
\oplus\bZ[D_6/S_3^{(1)}]\oplus\bZ[D_6/S_3^{(2)}]\\
\simeq & ~{} \bZ[D_6/C_2^{(1)}]\oplus\bZ[D_6/C_2^{(2)}]
\oplus\bZ[D_6/C_2^{(3)}]\oplus\bZ[D_6/C_3]\oplus\bZ^{\oplus 2}.
\end{align*}
\end{proposition-nn}

\bigskip

By the classification of Bravais groups of dimension $n\leq 6$ 
as in Subsection \ref{ssBravais} and the algorithms in Section \ref{seAlg} 
and Section \ref{seFC}, 
we will show the following theorem. 

\begin{theorem-nn}[Theorem \ref{thBravais6}]
If $G$ is a Bravais group of dimension $n\leq 6$, 
then $H^1(G,[M_G]^{fl})=0$. 
In particular, if $G$ is a maximal finite subgroup 
$G\leq \GL(n,\bZ)$ where $n\leq 6$, then $H^1(G,[M_G]^{fl})=0$. 
\end{theorem-nn}

\bigskip

As an application of the method of this paper, 
we give examples of not retract $k$-rational fields 
which are related to the rationality problem under the 
finite group action, e.g. Noether's problem 
(see \cite{HKY11}, \cite{Yam12}, \cite{HKK14}). 

Let $k$ be a field of {\rm char} $k\neq 2$ and $k(x,y,z)$ be 
the rational function field over $k$ with variables $x,y,z$. 
We consider the $k$-involution (i.e. $k$-automorphism of order $2$) 
\begin{align*}
\sigma_{a,b,c,d} : x\mapsto -x,\quad y\mapsto \frac{-ax^2+b}{y},\quad  z\mapsto 
\frac{-cx^2+d}{z}
\quad (a,b,c,d\in k^{\times})
\end{align*}
on $k(x,y,z)$
and the rationality problem of $k(x,y,z)^{\langle\sigma_{a,b,c,d}\rangle}$ over $k$, 
namely whether the fixed field $k(x,y,z)^{\langle\sigma_{a,b,c,d}\rangle}$ 
is $k$-rational. 
We see that the fixed field $k(x,y,z)^{\langle\sigma_{a,b,c,d}\rangle}$ 
is $k$-isomorphic to 
$k(x,y,z)^{\langle\sigma_{\tau(a),\tau(b),\tau(c),\tau(d)}\rangle}$ 
for $\tau\in D_4$ where $D_4=\langle (abdc),(ab)(cd)\rangle$ 
is the permutation group on the set $\{a,b,c,d\}$ which is isomorphic to 
the dihedral group of order $8$. 
Let $m=[k(\sqrt{a},\sqrt{b},\sqrt{c},\sqrt{d}):k]$. 
Hence the rationality problem is determined by the following $22$ cases:\\

\noindent
{\rm (C1)} $m=1$;\\
{\rm (C2)} $m=2$, 
{\rm (1)} $a,b,c\in k^{\times 2}$;
{\rm (2)} $a,b,d\in k^{\times 2}$;
{\rm (3)} $a,c,d\in k^{\times 2}$;
{\rm (4)} $b,c,d\in k^{\times 2}$;\\
{\rm (C3)} $m=2$, 
{\rm (1)} $a,b,cd\in k^{\times 2}$;
{\rm (2)} $b,d,ac\in k^{\times 2}$;
{\rm (3)} $d,c,ab\in k^{\times 2}$;
{\rm (4)} $c,a,bd\in k^{\times 2}$;\\
{\rm (C4)} $m=2$, 
{\rm (1)} $a, d, bc\in k^{\times 2}$;
{\rm (2)} $b, c, ad\in k^{\times 2}$;\\
{\rm (C5)} $m=2$, 
{\rm (1)} $a,bd,cd\in k^{\times 2}$;
{\rm (2)} $b,cd,ac\in k^{\times 2}$;
{\rm (3)} $d,ac,ab\in k^{\times 2}$;
{\rm (4)} $c,ab,bd\in k^{\times 2}$;\\
{\rm (C6)} $m=2$, $ab,ac,ad\in k^{\times 2}$;\\
{\rm (C7)} $m=4$, 
{\rm (1)} $a, b\in k^{\times 2}$;
{\rm (2)} $b, d\in k^{\times 2}$;
{\rm (3)} $d, c\in k^{\times 2}$;
{\rm (4)} $c, a\in k^{\times 2}$;\\
{\rm (C8)} $m=4$, 
{\rm (1)} $a,d\in k^{\times 2}$;
{\rm (2)} $b, c\in k^{\times 2}$;\\
{\rm (C9)} $m=4$, 
{\rm (1)} $a, bc\in k^{\times 2}$;
{\rm (2)} $b, ad\in k^{\times 2}$;
{\rm (3)} $d, bc\in k^{\times 2}$;
{\rm (4)} $c, ad\in k^{\times 2}$;\\
{\rm (C10)} $m=4$, 
{\rm (1)} $a, bd\in k^{\times 2}$;
{\rm (2)} $b, dc\in k^{\times 2}$;
{\rm (3)} $d, ac\in k^{\times 2}$;
{\rm (4)} $c, ab\in k^{\times 2}$;\\
\hspace*{20.6mm}
{\rm (5)} $a, cd\in k^{\times 2}$;
{\rm (6)} $b, ac\in k^{\times 2}$;
{\rm (7)} $d, ab\in k^{\times 2}$;
{\rm (8)} $c, bd\in k^{\times 2}$;\\
{\rm (C11)} $m=4$, 
{\rm (1)}  $a,bcd\in k^{\times 2}$;
{\rm (2)} $b,acd\in k^{\times 2}$;
{\rm (3)} $d,abc\in k^{\times 2}$;
{\rm (4)} $c,abd\in k^{\times 2}$;\\
{\rm (C12)} $m=4$, 
{\rm (1)} $ab,cd\in k^{\times 2}$;
{\rm (2)} $bd,ac\in k^{\times 2}$;\\
{\rm (C13)} $m=4$, 
{\rm (1)}  $ab,ac\in k^{\times 2}$;
{\rm (2)} $bd,ab\in k^{\times 2}$;
{\rm (3)} $cd,bd\in k^{\times 2}$;
{\rm (4)} $ac,cd\in k^{\times 2}$;\\
{\rm (C14)} $m=4$, $ad,bc\in k^{\times 2}$;\\
{\rm (C15)} $m=4$, 
{\rm (1)} $ab, acd\in k^{\times 2}$;
{\rm (2)} $bd, abc\in k^{\times 2}$;
{\rm (3)} $cd, abd\in k^{\times 2}$;
{\rm (4)} $ac, bcd\in k^{\times 2}$;\\
{\rm (C16)} $m=4$, 
{\rm (1)} $ad, abc\in k^{\times 2}$;
{\rm (2)} $bc, abd\in k^{\times 2}$;\\
{\rm (C17)} $m=8$, 
{\rm (1)} $a\in k^{\times 2}$;
{\rm (2)} $b\in k^{\times 2}$;
{\rm (3)} $d\in k^{\times 2}$;
{\rm (4)} $c\in k^{\times 2}$;\\
{\rm (C18)} $m=8$, 
{\rm (1)} $ab\in k^{\times 2}$;
{\rm (2)} $ac\in k^{\times 2}$;
{\rm (3)} $bd\in k^{\times 2}$;
{\rm (4)} $cd\in k^{\times 2}$;\\
{\rm (C19)} $m=8$, 
{\rm (1)}  $ad\in k^{\times 2}$;
{\rm (2)} $bc\in k^{\times 2}$;\\
{\rm (C20)} $m=8$, 
{\rm (1)} $abc\in k^{\times 2}$;
{\rm (2)} $bcd\in k^{\times 2}$;
{\rm (3)} $abd\in k^{\times 2}$;
{\rm (4)} $acd\in k^{\times 2}$;\\
{\rm (C21)} $m=8$, $abcd\in k^{\times 2}$;\\
{\rm (C22)} $m=16$.

\bigskip

We see that if one of the conditions 
(C1), (C2), (C3), (C5), (C6), (C7), (C10), (C12), (C13) 
holds, then $k(x,y,z)^{\langle\sigma_{a,b,c,d}\rangle}$ 
is $k$-rational (see Lemma \ref{lem3}).

\begin{theorem-nn}[Theorem \ref{th3}]
Let $k$ be a field of {\rm char} $k\neq 2$ and $k(x,y,z)$ be 
the rational function field over $k$ with variables $x,y,z$. 
Let $\sigma_{a,b,c,d}$ be a $k$-involution on $k(x,y,z)$ defined by 
\begin{align*}
\sigma_{a,b,c,d} : x\mapsto -x,\quad y\mapsto \frac{-ax^2+b}{y},\quad  z\mapsto 
\frac{-cx^2+d}{z}
\quad (a,b,c,d\in k^{\times})
\end{align*}
and $m=[k(\sqrt{a},\sqrt{b},\sqrt{c},\sqrt{d}):k]$.\\ 
{\rm (i)} $k(x,y,z)^{\langle\sigma_{a,b,c,d}\rangle}=k(t_1,t_2,t_3,t_4)$ 
where $t_1,t_2,t_3,t_4$ satisfy the relation 
\begin{align*}
(t_1^2-a)(t_4^2-d)=(t_2^2-b)(t_3^2-c).
\end{align*}
{\rm (ii)} $k(x,y,z)^{\langle\sigma_{a,b,c,d}\rangle}$ 
is $k$-isomorphic to 
$k(x,y,z)^{\langle\sigma_{\tau(a),\tau(b),\tau(c),\tau(d)}\rangle}$ 
for $\tau\in D_4$ where $D_4=\langle (abdc),(ab)(cd)\rangle$ 
is the permutation group on the set $\{a,b,c,d\}$ which is isomorphic to 
the dihedral group of order $8$.\\
{\rm (iii)} If one of the following conditions holds, then 
$k(x,y,z)^{\langle\sigma_{a,b,c,d}\rangle}$ 
is not retract $k$-rational:\\
{\rm (C15)} $m=4$, 
{\rm (1)} $ab, acd\in k^{\times 2}$;
{\rm (2)} $bd, abc\in k^{\times 2}$;
{\rm (3)} $cd, abd\in k^{\times 2}$;
{\rm (4)} $ac, bcd\in k^{\times 2}$;\\
{\rm (C16)} $m=4$, 
{\rm (1)} $ad, abc\in k^{\times 2}$;
{\rm (2)} $bc, abd\in k^{\times 2}$;\\
{\rm (C18)} $m=8$, 
{\rm (1)} $ab\in k^{\times 2}$;
{\rm (2)} $ac\in k^{\times 2}$;
{\rm (3)} $bd\in k^{\times 2}$;
{\rm (4)} $cd\in k^{\times 2}$;\\
{\rm (C19)} $m=8$, 
{\rm (1)}  $ad\in k^{\times 2}$;
{\rm (2)} $bc\in k^{\times 2}$;\\
{\rm (C20)} $m=8$, 
{\rm (1)} $abc\in k^{\times 2}$;
{\rm (2)} $bcd\in k^{\times 2}$;
{\rm (3)} $abd\in k^{\times 2}$;
{\rm (4)} $acd\in k^{\times 2}$;\\
{\rm (C21)} $m=8$, $abcd\in k^{\times 2}$;\\
{\rm (C22)} $m=16$.
\end{theorem-nn}

We do not know whether the field 
$k(x,y,z)^{\langle\sigma_{a,b,c,d}\rangle}$ is $k$-rational 
for the cases 
(C4), (C8), (C9), (C11), (C14), (C17).\\

We organize this paper as follows. 
In Section \ref{sePre}, we recall known results 
and prepare some basic materials, 
e.g. Galois cohomology, Tate cohomology, flabby resolution 
of a $G$-lattice. 
We explain the relationship between these materials and 
the stably rational classification of algebraic $k$-tori. 
In Section \ref{seCarat}, 
we explain how to access the GAP ID 
and the CARAT ID of a finite subgroup $G$ of $\GL(n,\bZ)$ $(n\leq 6)$ 
of this paper. 
In Section \ref{seKSfail}, 
we show that 
the Krull-Schmidt theorem for $G$-lattices 
holds when the rank $\leq 4$, and fails 
when the rank is $5$ and $6$ by using the GAP ID and the CARAT ID. 
The classification of maximal finite groups $G\leq \GL(n,\bZ)$ 
and Bravais groups of dimension $n$ where $n\leq 6$ will be also given 
in Subsections \ref{ssMax} and \ref{ssBravais} respectively. 
In Section \ref{seAlg}, we give some algorithms to 
compute a flabby resolution of a $G$-lattice effectively in GAP. 
We also give some algorithms which 
may determine whether the flabby class of the $G$-lattice 
is invertible (resp. zero) or not. 
These algorithms enable us to classify the function fields of 
algebraic $k$-tori up to stably equivalence. 
In Section \ref{seFC}, 
the classification of all the flabby and coflabby $G$-lattices 
of rank up to $6$ will be given. 
In Section \ref{seBravais}, we confirm that $H^1(G,[M_G]^{fl})=0$ 
for any Bravais group $G$ of dimension $n\leq 6$. 
In particular, we get that $H^1(G,[M_G]^{fl})=0$ for maximal finite 
subgroups $G\leq \GL(n,\bZ)$ where $n\leq 6$. 
In Section \ref{seNorm1}, we obtain 
the stably rational classification 
of 
norm one tori of dimensions $4$ and $5$. 
The same classification for dimensions $6$ and $10$ will also be given.
In Section \ref{seTate}, we give some GAP algorithms 
for computing the Tate cohomologies. 
In Sections \ref{seProof1} and \ref{seProof2}, we will prove 
Theorem \ref{th1M} and Theorem \ref{th2M} respectively by using 
the algorithms which are given in Sections \ref{seCarat}, 
\ref{seKSfail}, \ref{seAlg} and \ref{seTate}. 
Using the algorithms in Section \ref{seAlg}, 
we show Theorem \ref{th3} which provides 
some examples of not retract $k$-rational fields 
in Section \ref{seProof3}. 
In Section \ref{seApp}, 
an application of Theorem \ref{th3} which 
is related to Noether's problem will be given. 
Tables $11$ to $15$ in Theorems \ref{th2} and \ref{th2M} are located in
Section \ref{tables}. 
We also give detailed information of 
the stably rational classification of 
algebraic $k$-tori of dimension $5$ as in Table $16$ 
in Section \ref{tables}. 

\begin{acknowledgments}
The authors would like to thank 
Ming-chang Kang for giving them useful and valuable comments. 
They also would like to thank Shizuo Endo 
for valuable comments and for fruitful discussions, 
in particular, about Section \ref{seKSfail}. 
The first-named author wishes to gratefully thank his teacher 
Yumiko Hironaka for continuous encouragement. 
The authors also would like to thank the referee 
who draws their attention to Bravais groups 
and gives them useful suggestions and comments. 
Subsections \ref{ssMax}, \ref{ssBravais} and Section \ref{seBravais} 
are added according to referee's valuable advice. 
\end{acknowledgments}


%
\section{Preliminaries: Tate cohomology and flabby resolutions}\label{sePre}

First we recall the definitions of Galois cohomology and Tate cohomology.
See for details Cartan and Eilenberg \cite[ChapterXII]{CE56} 
and Brown \cite[Chapter VI]{Bro82}.

\begin{definition}[$n$-cochains, coboundary homomorphisms]
Let $G$ be a group and $M$ be an additive $G$-module.
Let $n \geq 0$ be an integer and $C^n(G,M)$ be the additive group
of all maps from $G^n$ to $M$ $(G^0=1)$.
The elements of $C^n(G,M)$ are called {\it the $n$-cochains}.
The {\it coboundary homomorphisms}
$$d^n: C^n(G,M) \rightarrow C^{n+1}(G,M)$$
are defined as
$$\begin{array}{rl}
(d^n\varphi)(g_1,\dots,g_{n+1})=& g_1 \cdot \varphi(g_2,\dots,g_{n+1}) \\
&+\sum_{i=1}^n(-1)^i\varphi(g_1,\dots,g_{i-1},g_ig_{i+1},g_{i+2},\dots,g_{n+1}) \\
&+(-1)^{n+1}\varphi(g_1,\dots,g_n).
\end{array}$$
\end{definition}

\begin{lemma}
$d^{n+1} \circ d^n=0$ holds,
and $\big(C(G,M),d\big)$ becomes a cochain complex.
\end{lemma}

\begin{definition}[Group cohomology]
Define {\it the group of $n$-cocycles} as
$$Z^n(G,M)=\ker(d^n),$$
{\it the group of $n$-coboundaries} as
$$\begin{cases}
B^0(G,M)=0, & \\
B^n(G,M)=\image(d^{n-1}) & (n \geq 1), \\
\end{cases}$$
and {\it the $n$-th cohomology group} as
$$H^n(G,M)=Z^n(G,M)/B^n(G,M)\ \ (n\geq 0).$$
\end{definition}

\begin{definition}[Tate cohomology]\label{defTate}
Let $G$ be a finite group and $M$ be an additive $G$-module.
Define the trace map $T_G:M \rightarrow M$ as
$$T_G(m)=\sum_{g \in G} g \cdot m,$$
{\it the group of $0$-cocycles} and {\it the group of $(-1)$-cocycles} as
$$\begin{cases}
\widehat Z^0(G,M)=M^G=H^0(G,M), & \\
\widehat Z^{-1}(G,M)=\ker(T_G), & \\
\end{cases}$$
{\it the group of $0$-coboundaries} and {\it the group of $(-1)$-coboundaries} as
$$\begin{cases}
\widehat B^0(G,M)=\image(T_G), & \\
\widehat B^{-1}(G,M)=\sum_{g \in G} \image (g-\id_M), & \\
\end{cases}$$
and {\it the $n$-th Tate cohomology group} as
$$\widehat H^n(G,M)=
\begin{cases}
H^n(G,M) & (n \geq 1),\\
\widehat Z^0(G,M)/\widehat B^0(G,M) & (n=0),\\
\widehat Z^{-1}(G,M)/\widehat B^{-1}(G,M) & (n=-1),\\
H_{-n-1}(G,M) & (n \leq -2)
\end{cases}$$
where $H_i$ is the $i$-th homology group. 
\end{definition}

Assume that $G$ is a finite group and $M$ is a $G$-lattice, 
i.e. finitely generated $\bZ[G]$-module which is $\bZ$-free as an abelian group. 
Both $\widehat Z^n(G,M)$ and $\widehat B^n(G,M)$ are free $\bZ$-modules 
of finite rank, and it turns out that the groups $\widehat H^n(G,M)$ 
have exponent dividing $\# G$ and hence finite for any $n\in\bZ$. 
We have that $\widehat H^n(G,\bZ[G])=0$ for $n=\pm 1$ (see Lemma \ref{lemSL})
and that $\widehat H^0(G,\bZ[G])=0$.

We will give some GAP algorithms for computing the Tate cohomology 
$\widehat H^n(G,M_G)$ in Section \ref{seTate}.\\

In order to construct a flabby resolution of $M$, 
we use the following long exact sequence:

\begin{lemma}[A long exact sequence]\label{longexact}
If $$0 \rightarrow A \xrightarrow{f} B \xrightarrow{g} C \rightarrow 0$$
is an exact sequence of $G$-lattices,
then there exists a long exact sequence of abelian groups 
\[
\cdots \rightarrow \widehat H^{k-1}(G,C)
\xrightarrow{d^\ast} \widehat H^k(G,A)
\xrightarrow{f^\ast} \widehat H^k(G,B)
\xrightarrow{g^\ast} \widehat H^k(G,C)
\xrightarrow{d^\ast} \widehat H^{k+1}(G,A)
\rightarrow \cdots
\]
where $f^\ast$ and $g^\ast$ are the maps in cohomology
induced from the cochain maps $f$ and $g$,
and $d^\ast$ is the connecting homomorphism
obtained by using the snake lemma.
\end{lemma}

\begin{definition}[Dual $G$-lattice]
Let $M$ be a $G$-lattice. 
The $G$-lattice 
$M^\circ=\Hom_\bZ(M,\bZ)$ with $G$-action
$$ (g\cdot f)(m)=f(g^{-1}\cdot m) \quad (f \in M^\circ, m \in M, g \in G)$$
is called {\it the dual $G$-lattice} of $M$.
\end{definition}

Note that $(M^\circ)^\circ\simeq M$ and 
if $A(g)$ is the matrix representation of the action of $g \in G$
on the $G$-lattice $M$ with a fixed $\bZ$-basis, then the matrix representation of 
the action of $g$ on $M^\circ$ is given by ${}^tA(g^{-1})$. 
In particular, if $M$ is permutation $G$-lattice, then $M\simeq M^\circ$. 

\begin{lemma}[{see \cite[Theorem 2.2]{Arn84}}]
$H^n(G,M) \simeq \widehat H^{-n}(G,M^\circ)$.
\end{lemma}

\begin{definition}[Dual homomorphism]
For $f \in \Hom_\bZ(M_1,M_2)$,
we define {\it the dual homomorphism} $f^\circ \in \Hom_\bZ(M_2^\circ,M_1^\circ)$ by 
\[
f^\circ(l)(m)=l(f(m)) \quad (l \in M_2^\circ, m \in M_1).
\]
\end{definition}

\begin{lemma}
If $A \xrightarrow{f} B \xrightarrow{g} C$
is an exact sequence of $G$-lattices, then
$C^\circ \xrightarrow{g^\circ} B^\circ
\xrightarrow{f^\circ} A^\circ$
is also an exact sequence of $G$-lattices.
\end{lemma}

\bigskip

We recall some basic facts of the theory of flabby (flasque) $G$-lattices
(see \cite{CTS77}, \cite{Swa83}, \cite[Chapter 2]{Vos98}, \cite[Chapter 2]{Lor05}, \cite{Swa10}).

\begin{definition}[Permutation, stably permutation, invertible, flabby and coflabby $G$-lattices]
Let $G$ be a finite group and $M$ be a $G$-lattice 
(i.e. finitely generated $\bZ[G]$-module which is $\bZ$-free 
as an abelian group). \\
{\rm (i)} $M$ is called a {\it permutation} $G$-lattice 
if $M$ has a $\bZ$-basis permuted by $G$, 
i.e. $M\simeq \oplus_{1\leq i\leq m}\bZ[G/H_i]$ 
for some subgroups $H_1,\ldots,H_m$ of $G$.\\
{\rm (ii)} $M$ is called a {\it stably permutation} $G$-lattice 
if $M\oplus P\simeq P^\prime$ 
for some permutation $G$-lattices $P$ and $P^\prime$.\\
{\rm (iii)} $M$ is called {\it invertible} (or {\it permutation projective}) 
if it is a direct summand of a permutation $G$-lattice, 
i.e. $P\simeq M\oplus M^\prime$ for some permutation $G$-lattice 
$P$ and a $G$-lattice $M^\prime$.\\
{\rm (iv)} $M$ is called {\it flabby} (or {\it flasque}) if $\widehat H^{-1}(H,M)=0$ 
for any subgroup $H$ of $G$ where $\widehat H$ is the Tate cohomology.\\
{\rm (v)} $M$ is called {\it coflabby} (or {\it coflasque}) if $H^1(H,M)=0$
for any subgroup $H$ of $G$.
\end{definition}

\begin{lemma}[{\cite[Propositions 1.1 and 1.2]{Len74}, see also 
\cite[Section 8]{Swa83}}]\label{lemSL}
Let $E$ be an invertible $G$-lattice. Then\\
{\rm (i)} $E$ is flabby and coflabby.\\
{\rm (ii)} If $C$ is a coflabby $G$-lattice, then any short exact sequence
$0 \rightarrow C \rightarrow N \rightarrow E \rightarrow 0$ splits.
\end{lemma}

Let $\cC(G)$ be the category of all $G$-lattices and 
$\cS(G)$ be the full subcategory of $\cC(G)$ of all permutation $G$-lattices. 
Recall that two $G$-lattices $M_1$ and $M_2$ are {\it similar} 
if there exist permutation $G$-lattices $P_1$ and $P_2$ such that 
$M_1\oplus P_1\simeq M_2\oplus P_2$. 
The {\it similarity class} of $M$ is denoted by $[M]$. 
Then the set of similarity classes $\cC(G)/\cS(G)$ becomes a 
commutative monoid 
(with respect to the sum $[M_1]+[M_2]:=[M_1\oplus M_2]$ 
and the zero $0=[P]$ where $P\in \cS(G)$) (see Definition \ref{defCM}). 


\begin{theorem}[{Endo and Miyata \cite[Lemma 1.1]{EM74}, 
Colliot-Th\'el\`ene and Sansuc \cite[Lemma 3]{CTS77}, 
see also \cite[Lemma 8.5]{Swa83}, \cite[Lemma 2.6.1]{Lor05}}]\label{thEM}
For any $G$-lattice $M$,
there is a short exact sequence of $G$-lattices
$0 \rightarrow M \rightarrow P \rightarrow F \rightarrow 0$
where $P$ is permutation and $F$ is flabby.
\end{theorem}

\begin{definition}[Flabby resolution]\label{defFlabby}
The exact sequence $0 \rightarrow M \rightarrow P \rightarrow F \rightarrow 0$ 
as in Theorem \ref{thEM} is called a {\it flabby resolution} of the $G$-lattice $M$.
$\rho_G(M)=[F] \in \cC(G)/\cS(G)$ is called {\it the flabby class} of $M$,
denoted by $[M]^{fl}=[F]$.
Note that $[M]^{fl}$ is well-defined: 
if $[M]=[M^\prime]$, $[M]^{fl}=[F]$ and $[M^\prime]^{fl}=[F^\prime]$
then $F \oplus P_1 \simeq F^\prime \oplus P_2$
for some permutation $G$-lattices $P_1$ and $P_2$,
and therefore $[F]=[F^\prime]$ (cf. \cite[Lemma 8.7]{Swa83}). 
We say that $[M]^{fl}$ is {\it invertible} if 
$[M]^{fl}=[E]$ for some invertible $G$-lattice $E$. 
\end{definition}

For $G$-lattice $M$, 
it is not difficult to see 
\begin{align*}
\textrm{permutation}\ \ 
\Rightarrow\ \ 
&\textrm{stably\ permutation}\ \ 
\Rightarrow\ \ 
\textrm{invertible}\ \ 
\Rightarrow\ \ 
\textrm{flabby\ and\ coflabby}\\
&\hspace*{8mm}\Downarrow\hspace*{34mm} \Downarrow\\
&\hspace*{7mm}[M]^{fl}=0\hspace*{10mm}\Rightarrow\hspace*{5mm}[M]^{fl}\ 
\textrm{is\ invertible}
\end{align*}
(see Section \ref{seInt}). 
Let $L/k$ be a finite Galois extension with Galois group $G={\rm Gal}(L/k)$ 
and $M$ be a $G$-lattice. 
The flabby class $[M]^{fl}$ plays crucial role in the rationality problem for $L(M)^G$. 
In particular, by Theorem \ref{thEM73} 
(Endo and Miyata \cite{EM73}, Voskresenskii \cite{Vos74}
and Saltman \cite{Sal84a}), we have the following fundamentals: 

{\rm (i)} $[M]^{fl}=0$ if and only if $L(M)^G$ is stably $k$-rational;

{\rm (ii)} $[M]^{fl}=[M^\prime]^{fl}$ if and only if $L(M)^G$ and $L(M^\prime)^G$ 
are stably $k$-isomorphic, i.e. there exist algebraically independent 
elements $x_1,\ldots,x_m$ over $L(M)^G$ and 
$y_1,\ldots,y_n$ over $L(M^\prime)^G$ such that 
$L(M)^G(x_1,\ldots,x_m)\simeq L(M^\prime)^G(y_1,\ldots,y_n)$;

{\rm (iii)} $[M]^{fl}$ is invertible if and only if $L(M)^G$ is 
retract $k$-rational.\\

Let $G\simeq {\rm Gal}(L/k)$ be a finite group and 
$M\simeq M_1\oplus M_2$ be a decomposable $G$-lattice. 
Let $N_i=\{\sigma\in G\mid\sigma(v)=v\ {\rm for\ any}\ v\in M_i\}$ 
be the kernel of the action of $G$ on $M_i$ $(i=1,2)$. 
Then $L(M)^G$ is the function field of an algebraic torus $T$ 
and is $k$-isomorphic to the free composite of 
$L(M_1)^G$ and $L(M_2)^G$ over $k$ where 
$L(M_i)^G=(L^{N_i})(M_i^{N_i})^{G/N_i}$ 
is the function field of some torus $T_i$ ($i=1,2$) with 
$T=T_1\times T_2$ 
and $M_i$ may be regarded as a $G/N_i$-lattice. 

\begin{lemma} \label{lemp1}
Let $G$ be a finite group and
$M \simeq M_1 \oplus M_2$ be a decomposable $G$-lattice with 
the flabby class $\rho_G(M)=[M]^{fl}$. 
Let $N_1$ be a normal subgroup of $G$ which acts on $M_1$ trivially. 
The $G$-lattice $M_1$ may be regarded as a $G/N_1$-lattice with the 
flabby class $\rho_{G/N_1}(M_1)$ as $G/N_1$-lattice. Then\\
{\rm (i)} $\rho_G(M)=\rho_G(M_1)+\rho_G(M_2)$.\\
{\rm (ii)} $\rho_{G}(M_1)=0$ if and only if $\rho_{G/N_1}(M_1)=0$.\\
{\rm (iii)} $\rho_G(M_1)$ is invertible if and only if 
$\rho_{G/N_1}(M_1)$ is invertible.
\end{lemma}
\begin{proof}
(i) Let $0 \rightarrow M_i \rightarrow P_i \rightarrow F_i
\rightarrow 0$ be flabby resolutions of $M_i$ as $G$-lattices $(i=1,2)$. 
Then 
$0 \rightarrow M \rightarrow P_1 \oplus P_2
\rightarrow F_1 \oplus F_2 \rightarrow 0$
is a flabby resolution of $M$. 
Hence $\rho_G(M)=[F_1 \oplus F_2]=\rho_G(M_1)+\rho_G(M_2)$. 
(ii), (iii) See \cite[Lemme 2]{CTS77} and \cite[Lemma 4.1]{Kan09}.
\end{proof}
\begin{lemma}[{Swan \cite[Lemma 3.1]{Swa10}}]\label{lemSwa}
Let $0\rightarrow M_1\rightarrow M\rightarrow M_2\rightarrow 0$ 
be a short exact sequence of $G$-lattices with $M_2$ invertible. 
Then $\rho_G(M)=\rho_G(M_1)+\rho_G(M_2)$. 
In particular, if $\rho_G(M_1)$ is invertible, 
then $-\rho_G(M_1)=\rho_G(\rho_G(M_1))$. 
\end{lemma}

\begin{proof}
See \cite[Lemma 3.1]{Swa10}. 
\end{proof}
\begin{lemma}\label{lemp2}
Let $G\simeq {\rm Gal}(L/k)$ be a finite group and 
$M\simeq M_1\oplus M_2$ be a decomposable $G$-lattice.\\
{\rm (i)} $L(M)^G$ is retract $k$-rational 
if and only if both of $L(M_i)^G$ $(i=1,2)$ are retract $k$-rational.\\
{\rm (ii)} If $L(M_1)^G$ and $L(M_2)^G$ are stably $k$-rational, 
then $K(M)^G$ is stably $k$-rational.\\
{\rm (iii)} When \rank $M_i\leq 3$ $(i=1,2)$, 
$L(M)^G$ is stably $k$-rational if and only if 
both of $L(M_i)^G$ $(i=1,2)$ are stably $k$-rational.
\end{lemma}
\begin{proof}
See, for example, \cite[Theorem 6.5]{HKK14}.
\end{proof}

Let $H$ be a subgroup of $G$.
For a $G$-lattice $M$,
it can be regarded as a $H$-lattice
by restricting the action of $G$ to $H$.
We write this $H$-lattice as $M|_H$.

\begin{lemma}\label{lemp3}
Let $G$ be a finite subgroup of $\GL(n,\bZ)$ 
and $M_G$ be the corresponding $G$-lattice 
as in Definition \ref{defMG}. 
Let $H$ be a subgroup of $G$.\\
{\rm (i)} If $\rho_G(M_G)=0$, then $\rho_H(M_H)=0$.\\
{\rm (ii)} If $\rho_G(M_G)$ is invertible, then $\rho_H(M_H)$ is invertible.
\end{lemma}

\begin{proof}
Let $0 \rightarrow M_G
\rightarrow P \rightarrow F \rightarrow 0$
be a flabby resolution of $M_G$ as a $G$-lattice. 
Then $0 \rightarrow M_G|_H
\rightarrow P|_H \rightarrow F|_H
\rightarrow 0$ is a flabby resolution of $M_G|_H=M_H$ as a $H$-lattice,
because $P|_H$ is a permutation $H$-lattice 
and $F|_H$ is a flabby $H$-lattice.
This shows that $\rho_H(M_H)=[F|_H]$
as a $H$-lattice.
If $G$-lattice $F$ is stably permutation (resp. invertible), 
then $F|_H$ is stably permutation (resp. invertible) as a $H$-lattice.
\end{proof}

\bigskip
 
\section{CARAT ID of the $\bZ$-classes in dimensions $5$ and $6$}\label{seCarat}

In this section, we will explain how to access the GAP ID 
and the CARAT ID of a finite subgroup $G$ of $\GL(n,\bZ)$ $(n\leq 6)$. 
We need the GAP (\cite{GAP}) packages CrystCat and CARAT to do 
the computations below.

The CrystCat package of GAP provides a catalog of $\bQ$-classes 
and $\bZ$-classes (conjugacy classes) of finite subgroups $G$ 
of $\GL(n,\bQ)$ and $\GL(n,\bZ)$ $(2\leq n\leq 4)$. 
For $2\leq n\leq 4$, the GAP ID $(n,i,j,k)$ of a finite subgroup 
$G$ of $\GL(n,\bZ)$ means that $G$ belongs to 
the $k$-th $\bZ$-class of the $j$-th $\bQ$-class of 
the $i$-th crystal system of dimension $n$ in GAP 
(see also \cite[Table 1]{BBNWZ78}).

The CARAT\footnote{CARAT works on Linux or Mac OS X, but not on Windows.} 
(\cite{Carat}) package of GAP provides all conjugacy classes of finite 
subgroups of $\GL(n,\bQ)$ ($n \leq 6$) (see \cite{PS00}). 
There exist exactly $2$ (reps. $13$, $73$, $710$, $6079$, $85308$) 
$\bZ$-classes forming $2$ (resp. $10$, $32$, $227$, $955$, $7103$)
$\bQ$-classes in dimension $n=1$ (resp. $2$, $3$, $4$, $5$, $6$)
\footnote{In the old version of CARAT, 
the number of $\bQ$-classes of $\GL(6,\bQ)$, 7104, 
was not correct because of the overlapping two same $\bQ$-classes 
(cf. \cite{PS00}). 
The second-named author detected it and reported the correct number, 7103, 
to the CARAT group (see also \cite{Carat}). 
It also turned out that 
the correct number of $\bZ$-classes of $\GL(6,\bQ)$ is 85308 (in the old version, 
85311 was wrong). 
This has been fixed in the current version 2.1b1 of CARAT.}.

After unpacking the CARAT, we get the $\bQ$-catalog file
{\tt carat-2.1b1/tables/qcatalog.tar.gz}.
Unpacking this file,
we get lists of $\bQ$-classes of $\GL(n,\bQ)$ 
$(n=1,\ldots,6)$ in ${\tt qcatalog/data1}, \ldots, {\tt qcatalog/data6}$.
Generators of each group are in individual files under the folders
${\tt qcatalog/dim1/}, \ldots, {\tt qcatalog/dim6/}$.

The second-named author wrote the perl script {\tt crystlst.pl} to 
collect these generators into a single file. 
Files ${\tt cryst1.gap}, \ldots, {\tt cryst6.gap}$ are
lists of representatives of $\bQ$-classes
of $\GL(1,\bQ), \ldots, \GL(6,\bQ)$ respectively.
These files are available from  
{\tt http://www.math.h.kyoto-u.ac.jp/\~{}yamasaki/Algorithm/} as {\tt GLnQ.zip}.

Let $G$ be a finite subgroup of $\GL(n,\bZ)$. 
CARAT has a command {\tt ZClassRepsQClass(G)}
to compute the complete $\bZ$-class representatives of the 
$\bQ$-class of $G$. 
We split the $\bQ$-class of $G$ into $\bZ$-classes by the 
command {\tt ZClassRepsQClass(G)}. 
For the $l$-th group $\widetilde{G}$ in the list of $\bZ$-classes 
obtained by {\tt ZClassRepsQClass(}$G${\tt )} 
where $G$ is the $m$-th group ($\bQ$-class) in {\tt qcatalog/data$n$}, 
we say that the CARAT ID of $\widetilde{G}$ is $(n,m,l)$.

The second-named author wrote a GAP program to determine
the $\bQ$-class and the $\bZ$-class of a group $G$.
The files {\tt crystcat.gap} and {\tt caratnumber.gap}
contain programs related to the GAP ID and the CARAT ID respectively.
The file {\tt caratnumber.gap} uses other files which are packed in the 
{\tt crystdat.zip}.
This zip file should be packed at the current directory.

All the files above are available from 
{\tt http://www.math.h.kyoto-u.ac.jp/\~{}yamasaki/Algorithm/}.

\bigskip

\noindent
{\tt MatGroupZClass(n,i,j,k)} (build-in function of GAP) returns 
the group $G\leq \GL(n,\bZ)$ of the GAP ID $(n,i,j,k)$ when $2\leq n\leq 4$.\\
{\tt CaratMatGroupZClass(n,i,j)} returns the group $G\leq \GL(n,\bZ)$ 
of the CARAT ID $(n,i,j)$ when $1\leq n\leq 6$.\\
{\tt CrystCatZClass(G)} returns the GAP ID $(n,i,j,k)$ of $G\leq \GL(n,\bZ)$ 
when $1\leq n\leq 4$.\\
{\tt CaratZClass(G)} returns the CARAT ID $(n,i,j)$ of 
$G\leq \GL(n,\bZ)$ when $1\leq n\leq 6$.\\
{\tt NrQClasses(n)} returns the number of $\bQ$-classes in dimension $n$ when 
$1\leq n\leq 6$.\\
{\tt NrZClasses(n,m)} returns the number of $\bZ$-classes 
in the $m$-th $\bQ$-class in dimension $n$ 
when $1\leq n\leq 6$.\\
{\tt CrystCat2Carat(l)} returns the CARAT ID of the group $G$ of the GAP ID $l$.\\
{\tt Carat2CrystCat(l)} returns the GAP ID of the group $G$ of the CARAT ID $l$.

\bigskip

\begin{example}[Functions in {\tt crystcat.gap} and {\tt caratnumber.gap}]
We give some examples of the functions in {\tt crystcat.gap} and 
{\tt caratnumber.gap}. 
Note that {\tt caratnumber.gap} needs the CARAT package in GAP. 

Let $C_n$ be the cyclic group of order $n$ and 
$J_n=J_{C_n}$ be the Chevalley module of rank $n-1$ 
which is the dual of $I_n={\rm Ker}\ \varepsilon$ where 
$\varepsilon : \bZ[C_n]\rightarrow \bZ$ is the 
augmentation map (see Section \ref{seInt}). 
Then $K(J_n)^{C_n}$ is the function field of the norm 
one torus $R_{K/k}^{(1)}(\bG_m)$ where $K$ is a cyclic Galois 
extension of $k$ of degree $n$. 

\bigskip

\begin{verbatim}
gap> Read("crystcat.gap");
gap> Read("caratnumber.gap");

gap> List([1..6],n->NrQClasses(n)); # # of Q-classes in dimension n
[ 2, 10, 32, 227, 955, 7103 ]
gap> List([1..6],n->Sum([1..NrQClasses(n)],i->NrZClasses(n,i))); # # of Z-classes
[ 2, 13, 73, 710, 6079, 85308 ]

gap> J5:=Group([ [ 
> [ 0, 1, 0, 0 ], 
> [ 0, 0, 1, 0 ], 
> [ 0, 0, 0, 1 ], 
> [ -1, -1, -1, -1 ] ] ]);
gap> CrystCatZClass(J5);
[ 4, 27, 1, 1 ]
gap> G:=MatGroupZClass(4,27,1,1); # G=C5
MatGroupZClass( 4, 27, 1, 1 )
gap> CrystCat2Carat([4,27,1,1]);
[ 4, 227, 1 ]
gap> Carat2CrystCat([4,227,1]);
[ 4, 27, 1, 1 ]
gap> GeneratorsOfGroup(G);   
[ [ [ 0, -1, 1, 0 ], 
    [ 0, -1, 0, 1 ], 
    [ 0, -1, 0, 0 ], 
    [ 1, -1, 0, 0 ] ] ]
gap> P:=RepresentativeAction(GL(4,Integers),J5,G);
[ [ 0, -1, 0, 1 ], 
  [ 0, 1, 0, 0 ], 
  [ 0, 0, -1, 0 ], 
  [ -1, 0, 1, 0 ] ]
gap> J5^P=G; #checking P^-1*J5*P=G
true

gap> J6:=Group([ [ 
> [ 0, 1, 0, 0, 0 ], 
> [ 0, 0, 1, 0, 0 ], 
> [ 0, 0, 0, 1, 0 ], 
> [ 0, 0, 0, 0, 1 ], 
> [ -1, -1, -1, -1, -1 ] ] ]);
<matrix group with 1 generators>
gap> CaratZClass(J6);
[ 5, 461, 4 ]
gap> G:=CaratMatGroupZClass(5,461,4); # G=C6
<matrix group with 1 generators>
gap> GeneratorsOfGroup(G);
[ [ [ -1, 0, 1, 0, 0 ], 
    [ 1, 0, 0, 0, 0 ], 
    [ -1, 0, 0, 0, -1 ], 
    [ 1, 1, 0, 0, 0 ], 
    [ 1, 0, 0, 1, 0 ] ] ]
gap> P:=RepresentativeAction(GL(5,Integers),J6,G);
[ [ 0, 0, 0, -1, 1 ], 
  [ 0, -1, 0, 1, 0 ], 
  [ 0, 1, 0, 0, 0 ], 
  [ 1, 0, 0, 0, 0 ], 
  [ -1, 0, 1, 0, 0 ] ]
gap> J6^P=G; #checking P^-1*J6*P=G
true
\end{verbatim}
\end{example}

Some programs related to a flabby resolution are available from\\
{\tt http://www.math.h.kyoto-u.ac.jp/\~{}yamasaki/Algorithm/}.

\bigskip

%
\section{Krull-Schmidt theorem fails for dimension $5$}\label{seKSfail}

Let $G$ be a finite group and $\bZ[G]$ be the integral group ring of $G$. 
\begin{definition}[Decomposable and reducible $G$-lattices]
A $G$-lattice $M$ is said to be {\it reducible} if 
there exists a non-trivial $G$-invariant subspace of $M$. 
A $G$-lattice is said to be {\it irreducible} if it is not reducible. 
A $G$-lattice $M$ is said to be {\it decomposable} if there exist non-trivial 
$G$-lattices $U_1$ and $U_2$ such that $M\simeq U_1\oplus U_2$. 
A $G$-lattice is said to be {\it indecomposable} 
if it is not decomposable.
\end{definition}
If a $G$-lattice $M$ is decomposable, then it is reducible. 
By Maschke's theorem, the converse holds for $\bQ[G]$-modules, but not for 
$G$-lattices (i.e. finitely generated $\bZ$-free $\bZ[G]$-module). 

We say that {\it the direct sum cancellation holds for $G$-lattices} if 
$M_1\oplus N\simeq M_2\oplus N$ implies $M_1\simeq M_2$ for 
$G$-lattices $M_1$, $M_2$ and $N$.
We say that {\it the Krull-Schmidt theorem holds for $G$-lattices} 
if $M_1\oplus\cdots\oplus M_l\simeq N_1\oplus\cdots\oplus N_m$ 
for indecomposable $G$-lattices $M_i$ and $N_j$, 
then $l=m$ and, after a suitable renumbering of the $N_j$, $M_i\simeq N_i$ 
for any $1\leq i\leq m$ (see \cite{Fac03}). 
Clearly, the Krull-Schmidt theorem holds for $G$-lattices implies 
that the direct sum cancellation holds for $G$-lattices. 

By Krull-Schmidt-Azumaya theorem (\cite[Theorem 6.12]{CR81}, see also \cite{Azu50}), 
the Krull-Schmidt theorem holds 
for any $\bZ_p[G]$-lattice where 
$\bZ_p$ is the ring of $p$-adic integers. 
By \cite[Theorem 2]{Jon65} (see \cite[Theorem 36.1]{CR81}), 
for $p$-group $G$ where $p$ is odd prime, 
the Krull-Schmidt theorem holds for $\bZ_{(p)}[G]$-lattices 
where $\bZ_{(p)}$ is the localization of $\bZ$ at the prime ideal $(p)$. 
Note that the $G$-lattice $\bZ[G/H]$ is indecomposable 
for any $H\leq G$ (see \cite[Theorem 32.14]{CR87}).


Endo and Hironaka \cite[Theorem, page 161]{EH79} 
(see \cite[Theorem 50.29]{CR87}) showed that 
if the direct sum cancellation holds for $G$-lattices, 
then $G$ is abelian, dihedral, $A_4$, $S_4$ or $A_5$ via 
the projective class group 
(see also \cite[Corollary 1.3 and Section 7]{Swa88}). 
The question whether the Krull-Schmidt theorem holds 
for $G$-lattices is solved except for the case where $G$ is 
the dihedral group $D_8$ of order $16$ 
(see also \cite[Theorem 9.1]{Fac03}). 
%
%
\begin{theorem}[{Hindman, Klingler and Odenthal \cite[Theorem 1.6]{HKO98}}]
Let $G$ be a finite group which is not the dihedral group 
$D_8$ of order $16$. 
Then the Krull-Schmidt theorem holds for $G$-lattices 
if and only if one of the following conditions holds: 
{\rm (i)} $G=C_p$ for prime $p\leq 19$; 
{\rm (ii)} $G=C_n$ for $n=1,4,8$ or $9$; 
{\rm (iii)} $G=C_2\times C_2$;
{\rm (iv)} $G$ is the dihedral group $D_4$ of order $8$.
\end{theorem}
Two $G$-lattices $M$, $N$ are placed in the same genus ($M\approx N$)
if $M_{(p)}\simeq N_{(p)}$ for any prime ideal $(p)$. 
We say that {\it the generalized Krull-Schmidt theorem holds for $G$-lattices} 
if $M_1\oplus\cdots\oplus M_l\simeq N_1\oplus\cdots\oplus N_m$ 
for indecomposable $G$-lattices $M_i$ and $N_j$, 
then $l=m$ and, after a suitable renumbering of the $N_j$, $M_i\approx N_i$ 
for any $1\leq i\leq m$.
Clearly, the Krull-Schmidt theorem holds for $G$-lattices 
implies that the generalized Krull-Schmidt theorem holds for $G$-lattices.
The following theorem was pointed out to the authors by S. Endo:
\begin{theorem}[{Dress \cite[Section 1]{Dre70}}]
Let $p$ be a prime number. 
The following conditions are equivalent:\\
{\rm (i)} $G$ is a $p$-group;\\
{\rm (ii)} the generalized Krull-Schmidt theorem holds 
for invertible $G$-lattices.
\end{theorem}
\begin{proof} (This proof is due to S. Endo \cite{End12}.) 
{\rm (i)} $\Rightarrow$ {\rm (ii)}. 
If $p$ is odd prime, then the Krull-Schmidt theorem holds 
for $\bZ_{(p)}[G]$-lattices. 
Assume that $G$ is a $2$-group and $M$ is an invertible $G$-lattice. 
Since the Krull-Schmidt theorem holds 
for $\bZ_2[G]$-lattices, $M_2\simeq \bigoplus \bZ_2[G/H_i]$ 
where $M_2$ is the $2$-adic completion of $M$ and $H_i\leq G$. 
It follows from Maranda's theorem \cite[Theorem 30.14]{CR81} 
(see also \cite[Proposition 30.17]{CR81}) that 
$M_{(2)}\simeq \bigoplus \bZ_{(2)}[G/H_i]$ 
where $M_{(2)}$ is the localization of $M$ at prime $(2)$. 
The generalized Krull-Schmidt theorem holds 
for invertible $G$-lattices because $\bZ_{(p)}[G/H]$ is indecomposable 
for any prime $p$ and any subgroup $H\leq G$. 

{\rm (ii)} $\Rightarrow$ {\rm (i)}. 
Assume that $G$ is not a $p$-group and $H\leq G$ is of order $p^lq^m$ 
where $p$ and $q$ are different primes and $l,m\geq 1$. 
Let $Sy_p(H)$ be a $p$-Sylow subgroup of $H$. 
Then by \cite[Section 1]{Dre70} there exists $G$-lattice $M$ such that 
$\bZ[H/Sy_p(H)]\oplus\bZ[H/Sy_q(H)]\simeq M\oplus\bZ$. 
By taking the tensor product $\bZ[G]\otimes_{\bZ[H]}$ of both sides, 
we get $\bZ[G/Sy_p(H)]\oplus\bZ[G/Sy_q(H)]\simeq 
\bZ[G]\otimes_{\bZ[H]} M\oplus\bZ[G/H]$ as $G$-lattices. 
This contradicts that the generalized Krull-Schmidt theorem holds. 
\end{proof}
%

A $G$-set is a finite set with left $G$-action. 
The disjoint union $X\coprod X^\prime$ of $G$-sets $X$ and $X^\prime$ 
is also $G$-set.
Two $G$-sets are isomorphic if there exists a bijection between them 
which preserves the action of $G$. 
A $G$-set $X$ may be written uniquely up to isomorphism as 
$X\simeq \coprod_{H} a_H(X)G/H$ where $H$ runs through 
a set of representatives of conjugacy classes of subgroups of $G$ 
(see \cite[Chapter 11]{CR87}, \cite[Chapter 5]{Ben91}, \cite{GW93}, 
\cite{Bou00} for related materials, e.g. Burnside ring). 
For $G$-sets $X$ and $X^\prime$, 
the direct sum of permutation $G$-lattices $\bZ[X]$ 
and $\bZ[X^\prime]$ is also permutation: 
$\bZ[X]\oplus \bZ[X^\prime]\simeq \bZ[X\coprod X^\prime]$. 
A finite group $G$ is called {\it cyclic mod $p$} (or {\it $p$-hypoelementary}) 
if the quotient group $G/O_p(G)$ is cyclic where $O_p(G)$ 
is the largest normal $p$-subgroup of $G$. 

\begin{theorem}[{Dress \cite[Proposition 9.6]{Dre73}}]
Let $G$ be a finite group. The following conditions are equivalent:\\
{\rm (i)} $G$ is cyclic mod $p$ for some $p$;\\
{\rm (ii)} for any two $G$-sets $S$, $T$, 
$\bZ[S]\simeq \bZ[T]$ if and only if $S\simeq T$.\\
In particular, the Krull-Schmidt theorem holds for 
permutation $G$-lattices if and only if $G$ is cyclic mod $p$ 
for some $p$.
\end{theorem}
Using the algorithms in Sections \ref{seAlg} and \ref{seFC}, 
we will show the following proposition by constructing explicit isomorphism 
in Section \ref{seFC} (see Example \ref{exKSfailD6}). 
We remark that the dihedral group $D_6$ of order $12$ 
is the smallest group which is not cyclic mod $p$ for any $p$.
\begin{proposition-nn}[The Krull-Schmidt theorem fails for permutation $D_6$-lattices, see Proposition \ref{propKSfailD6}]
Let $D_6$ be the dihedral group of order $12$ and 
$\{1\}$, $C_2^{(1)}$, $C_2^{(2)}$, $C_2^{(3)}$, $C_3$, $C_2^2$, 
$C_6$, $S_3^{(1)}$, $S_3^{(2)}$ and $D_6$ be the 
conjugacy classes of subgroups of $D_6$. 
Then the following isomorphism of permutation $D_6$-lattices holds: 
\begin{align*}
& ~{} \bZ[D_6]\oplus\bZ[D_6/C_2^2]^{\oplus 2}\oplus\bZ[D_6/C_6]
\oplus\bZ[D_6/S_3^{(1)}]\oplus\bZ[D_6/S_3^{(2)}]\\
\simeq & ~{} \bZ[D_6/C_2^{(1)}]\oplus\bZ[D_6/C_2^{(2)}]
\oplus\bZ[D_6/C_2^{(3)}]\oplus\bZ[D_6/C_3]\oplus\bZ^{\oplus 2}.
\end{align*}
\end{proposition-nn}

\bigskip

\begin{definition}[Decomposition type]
Let $G$ be a finite group and $M$ be a $G$-lattice. 
When $M$ decomposes into indecomposable 
$G$-lattices $M\simeq U_1\oplus\cdots\oplus U_r$ of rank $n_1,\dots,n_r$, 
we say that a {\it decomposition type} ${\rm DT}(M)$ of $M$ is $(n_1,\dots,n_r)$. 
(This may not be unique.)
\end{definition}
Let $G$ be a finite subgroup of $\GL(n,\bZ)$ 
and $M_G$ be the corresponding $G$-lattice of rank $n$ 
as in Definition \ref{defMG}. 
The number of $G$-lattices $M_G$ for a given decomposition type 
${\rm DT}(M_G)$ is as follows (see Example \ref{exKS1} below):
\vspace*{3mm}

\noindent\ 
\begin{tabular}{l|c|c}
${\rm DT}(M_G)$ & $(1)$ & Total\\\hline
$\# M_G$ & 2 & 2
\end{tabular}\quad 
\begin{tabular}{l|cc|c}
${\rm DT}(M_G)$ & $(1,1)$ & $(2)$& Total\\\hline
$\# M_G$ & 4 & 9 & 13
\end{tabular}\quad 
\begin{tabular}{l|ccc|c}
${\rm DT}(M_G)$ & $(1,1,1)$ & $(2,1)$ & $(3)$ & Total\\\hline
$\# M_G$ & 8 & 31 & 34 & 73
\end{tabular}\\

\noindent\ 
\begin{tabular}{l|ccccc|c}
${\rm DT}(M_G)$ & $(1,1,1,1)$ & $(2,1,1)$ & $(2,2)$ & $(3,1)$ & $(4)$ & Total\\\hline
$\# M_G$ & 16 & 96 & 175 & 128 & 295 & 710
\end{tabular}\\

\noindent\ 
\begin{tabular}{l|ccccccc|c}
${\rm DT}(M_G)$ & $(1^5)$ & $(2,1^3)$ & $(2^2,1)$ & $(3,1^2)$ & $(3,2)$ & $(4,1)$ & $(5)$ & Total\\\hline
$\# M_G$ & 32 & 280 & 1004 & 442 & {\bf 1480} & {\bf 1400} & 1452 & ${\bf 6090} \atop (6079)$
\end{tabular}\\

\noindent\ 
\begin{tabular}{l|ccccccccccc|c}
${\rm DT}(M_G)$ & $(1^6)$ & $(2,1^4)$ & $(2^2,1^2)$ & $(2^3)$ & 
$(3,1^3)$ & $(3,2,1)$ & $(3^2)$ & $(4,1^2)$ & $(4,2)$ & $(5,1)$ & $(6)$ & Total\\\hline
$\# M_G$ & 68 & 824 & 4862 & 6878 & 1466 & {\bf 10662} & {\bf 4235} & {\bf 5944} & 
21573 & {\bf 9931} & 18996 & ${\bf 85439} \atop (85308)$
\end{tabular}
\vspace*{4mm}

For $n\leq 4$, we see that the Krull-Schmidt theorem holds for $M_G$ of rank $n$, 
i.e. if $M_G\simeq M_1\oplus\cdots\oplus M_l\simeq N_1\oplus\cdots\oplus N_m$ 
for indecomposable $G$-lattices $M_i$ and $N_j$, 
then $l=m$ and, after a suitable renumbering of the $N_j$, $M_i\simeq N_i$ 
for any $1\leq i\leq m$. 
However, it turns out that the Krull-Schmidt theorem fails for $M_G$ when the rank $n$ 
of $M_G$ is $5$. 
We split the Krull-Schmidt theorem for $M_G$ into the following two parts: \\

(KS1) If $M_G\simeq M_1\oplus\cdots\oplus M_l\simeq N_1\oplus\cdots\oplus N_m$ 
for indecomposable $G$-lattices $M_i$ and $N_j$, 
then $l=m$ and, after a suitable renumbering of the $N_j$, 
rank $M_i=$ rank $N_i$ for any $1\leq i\leq m$;

(KS2) If $M_G\simeq M_1\oplus\cdots\oplus M_m\simeq N_1\oplus\cdots\oplus N_m$ 
for indecomposable $G$-lattices $M_i$ and $N_i$ 
with \rank $M_i=$ \rank $N_i$ for any $1\leq i\leq m$, 
then after a suitable renumbering of the $N_i$, 
$M_i\simeq N_i$ for any $1\leq i\leq m$.\\

The Krull-Schmidt theorem holds for $M_G$ 
if and only if the conditions (KS1) and (KS2) hold for $M_G$. 

\begin{theorem}\label{thKS}
Let $G$ be a finite subgroup of $\GL(n,\bZ)$ and 
$M_G$ be the $G$-lattice as in Definition \ref{defMG}.\\
{\rm (i)} When $n\leq 4$, the Krull-Schmidt theorem holds for $M_G$, 
i.e. if $M_G\simeq M_1\oplus\cdots\oplus M_l\simeq N_1\oplus\cdots\oplus N_m$ 
for indecomposable $G$-lattices $M_i$ and $N_j$, 
then $l=m$ and, after a suitable renumbering of the $N_j$, $M_i\simeq N_i$ 
for any $1\leq i\leq m$.\\
{\rm (ii)}When $n=5$, {\rm (KS2)} holds for $M_G$, and 
the Krull-Schmidt theorem fails for $M_G$ if and only if {\rm (KS1)} fails for $M_G$ 
if and only if the {\rm CARAT ID} of $G$ is one of the $11$ triples
\vspace*{-3mm}

{\footnotesize 
\begin{align*}
(5,188,4),(5,189,4),(5,190,6),(5,191,6),(5,192,6),(5,193,4),
(5,205,6),(5,218,8),(5,219,8),(5,220,4),(5,221,4). 
\end{align*}
}

\vspace*{-3mm}
\noindent
For the exceptional $11$ cases, 
the decomposition types of $M_G$ are $(3,2)$ and $(4,1)$ 
and $G$ is a subgroup of the group $C_2\times D_6$ of the {\rm CARAT ID} $(5,205,6)$.\\
{\rm (iii)} When $n=6$, {\rm (KS1)} fails for $M_G$ 
if and only if 
the {\rm CARAT ID} of $G$ is one of the $131$ triples 
\vspace*{-3mm}

{\footnotesize 
\begin{align*}
&(6,2013,8),(6,2018,4),(6,2023,6),(6,2024,6),(6,2025,6),
(6,2026,6),(6,2033,6),(6,2042,8),(6,2043,8),(6,2044,4),\\
&(6,2045,4),(6,2048,5),(6,2049,8),(6,2050,8),(6,2051,8),
(6,2052,8),(6,2058,5),(6,2059,5),(6,2067,5),(6,2068,5),\\
&(6,2069,5),(6,2069,11),(6,2070,9),(6,2071,9),(6,2072,10),
(6,2072,11),(6,2076,24),(6,2076,25),(6,2077,24),(6,2077,25),\\
&(6,2078,24),(6,2078,25),(6,2079,24),(6,2079,25),(6,2087,15),
(6,2088,15),(6,2089,17),(6,2089,18),(6,2094,9),(6,2102,24),\ \,\\
&(6,2102,25),(6,2105,9),(6,2106,9),(6,2107,10),(6,2107,11),
(6,2108,15),(6,2109,15),(6,2110,17),(6,2110,18),(6,2111,15),\\
&(6,2139,9),
\end{align*}
}
\vspace*{-3mm}
\vspace*{-3mm}
{\footnotesize 
\begin{align*}
&(6,40,4),(6,41,4),(6,44,6),(6,45,6),(6,47,4),
(6,53,4),(6,54,4),(6,54,8),(6,55,4),(6,63,4),\\
&(6,64,6),(6,65,4),(6,66,6),(6,67,6),(6,75,4),
(6,75,8),(6,76,8),(6,76,12),(6,77,8),(6,77,12),\\
&(6,78,4),(6,78,8),(6,79,6),(6,80,4),(6,81,8),
(6,81,12),(6,90,4),(6,99,4),(6,108,4),(6,108,8),\\
&(6,109,8),(6,109,12),(6,110,4),(6,111,6),(6,112,8),
(6,112,12),(6,113,4),(6,114,6),(6,115,6),(6,145,4),\\
&(6,2070,10),(6,2070,11),(6,2071,10),(6,2071,11),(6,2072,12),
(6,2072,13),(6,2076,26),(6,2076,27),(6,2077,26),(6,2077,27),\\
&(6,2078,26),(6,2078,27),(6,2079,26),(6,2079,27),(6,2087,16),
(6,2087,17),(6,2088,16),(6,2088,17),(6,2089,19),(6,2089,20),\\
&(6,2094,10),(6,2094,11),(6,2102,26),(6,2102,27),(6,2105,10),
(6,2105,11),(6,2106,10),(6,2106,11),(6,2107,12),(6,2107,13),\\
&(6,2108,16),(6,2108,17),(6,2109,16),(6,2109,17),(6,2110,19),
(6,2110,20),(6,2111,16),(6,2111,17),(6,2139,10),(6,2139,11).
\end{align*}
}

\vspace*{-3mm}
\noindent
For the former $51$ cases $($resp. the latter $80$ cases$)$, 
the decomposition types of $M_G$ are $(3,2,1)$ and $(4,1,1)$ 
$($resp. $(3,3)$ and $(5,1)$$)$ and $G$ is a subgroup of the group 
$C_2^2\times D_6$ of the {\rm CARAT ID} $(6,2139,9)$ 
$($resp. $D_6\times D_4$ of the {\rm CARAT ID} $(6,145,4)$$)$.\\
{\rm (iv)} When $n=6$, {\rm (KS2)} fails for $M_G$ 
if and only if 
the {\rm CARAT ID} of $G$ is one of the $18$ triples 
\vspace*{-3mm}

{\footnotesize 
\begin{align*}
&(6,2072,14),(6,2076,28),(6,2077,28),(6,2078,28),(6,2079,28),
(6,2089,21),(6,2102,28),(6,2107,14),(6,2110,21),(6,2295,2),\\
&(6,3045,3),(6,3046,3),(6,3047,3),(6,3052,5),
(6,3053,5),(6,3054,3),(6,3061,5),(6,3066,3).
\end{align*}
}

\vspace*{-5mm}
\noindent
For the former $10$ cases, 
the decomposition type of $M_G$ is $(4,2)$ 
and $G$ is the group $D_6$ of the {\rm CARAT ID} $(6,2295,2)$ 
or a subgroup of the $3$ groups $C_2\times D_6$ of 
the {\rm CARAT ID}s $(6,2102,28)$, $(6,2107,14)$ and $(6,2110,21)$. 
For the latter $8$ cases, the decomposition type of $M_G$ 
is $(5,1)$ and $G$ is a subgroup of the group 
$C_2\times S_5$ of the {\rm CARAT ID} $(6,3054,3)$.
\end{theorem}

\bigskip

\setcounter{subsection}{-1}
\subsection{Classification of indecomposable maximal finite groups $G\leq \GL(n,\bZ)$ of dimension $n\leq 6$}
\label{ssIndmf}

Let $G\leq \GL(n,\bZ)$ be a finite matrix group. 
$G$ is called {\it reducible} (resp. {\it irreducible, decomposable, indecomposable}) 
if $M_G$ is reducible (resp. irreducible, decomposable, indecomposable) 
where $M_G$ is the corresponding $G$-lattice as in Definition \ref{defMG}.
Let ${\rm Imf}(n,i,j)\leq \GL(n,\bZ)$ be the 
$j$-th $\bZ$-class of the $i$-th $\bQ$-class of the irreducible 
maximal finite group of dimension $n$ 
which corresponds to 
the build-in function {\tt ImfMatrixGroup(n,i,j)} of GAP. 
For $n\leq 10$ and $n=p\leq 23$; prime, 
the irreducible maximal finite groups ${\rm Imf}(n,i,j)$ 
is determined by Plesken and Pohst \cite{PP77} ($n\leq 7$), 
\cite{PP80} ($n=8$, $9$), Plesken \cite{Ple85} ($n=p\leq 23$; prime) 
and Souvignier \cite{Sou94} ($n=10$).

For $n=2$, there exist exactly $2$ irreducible maximal finite groups ($\bZ$-classes) 
${\rm Imf}(2,1,1)\simeq D_4$ and ${\rm Imf}(2,2,1)\simeq D_6$ of order $8$ and $12$ 
of the GAP IDs $(2,3,2,1)$ and $(2,4,4,1)$. 

For $n=3$, there exist exactly $3$ irreducible maximal finite groups 
${\rm Imf}(3,1,1)\simeq {\rm Imf}(3,1,2)\simeq {\rm Imf}(3,1,3)\simeq C_2\times S_4$ 
of order $48$ of the GAP IDs $(3,7,5,1)$, $(3,7,5,2)$ and $(3,7,5,3)$. 

For $n=4$, there exist exactly $6$ irreducible maximal finite groups 
${\rm Imf}(4,1,1)$, ${\rm Imf}(4,2,1)\simeq D_6^2\rtimes C_2$, 
${\rm Imf}(4,3,1)\simeq {\rm Imf}(4,3,2)\simeq C_2\times S_5$, 
${\rm Imf}(4,4,1)\simeq C_2^4\rtimes S_4$ and 
${\rm Imf}(4,5,1)\simeq C_2\times(S_3^2\rtimes C_2)$ 
of order $1152$, $288$, $240$, $240$, $384$ and $144$ 
of the GAP IDs 
$(4,33,16,1)$, $(4,30,13,1)$, $(4,31,7,1)$, $(4,31,7,2)$, 
$(4,32,21,1)$ and $(4,29,9,1)$ respectively. 
Note that the first one is isomorphic to 
the Wyle group $W(F_4)$ of type $F_4$ of order $1152$. 

For $n=5$, there exist exactly $7$ irreducible maximal finite groups 
${\rm Imf}(5,1,1)\simeq {\rm Imf}(5,1,2)\simeq {\rm Imf}(5,1,3)\simeq 
C_2^5\rtimes S_5$ 
of order $3840$ of the {\rm CARAT ID}s 
$(5,942,1)$, $(5,942,2)$, $(5,942,3)$ and 
${\rm Imf}(5,2,1)\simeq {\rm Imf}(5,2,2)\simeq {\rm Imf}(5,2,3)\simeq {\rm Imf}(5,2,4)\simeq 
C_2\times S_6$ of order $1440$ of the {\rm CARAT ID}s $(5,949,1)$, 
$(5,949,4)$, $(5,949,2)$, $(5,949,3)$.

For $n=6$, there exist exactly $17$ irreducible maximal finite groups 
${\rm Imf}(6,1,1)\simeq {\rm Imf}(6,1,2)\simeq {\rm Imf}(6,1,3)\simeq C_2^6\rtimes S_6$, 
${\rm Imf}(6,2,1)\simeq D_6^3\rtimes S_3$, 
${\rm Imf}(6,3,1)\simeq {\rm Imf}(6,3,2)$, 
${\rm Imf}(6,4,1)\simeq {\rm Imf}(6,4,2)\simeq C_2\times S_7$, 
${\rm Imf}(6,5,1)\simeq C_2\times {\rm PGL}(2,\bF_7)$, 
${\rm Imf}(6,6,1)\simeq {\rm Imf}(6,6,2)\simeq {\rm Imf}(6,6,3)\simeq C_2\times S_5$, 
${\rm Imf}(6,7,1)\simeq {\rm Imf}(6,7,2)\simeq (C_2\times S_4)^2\rtimes C_2$, 
${\rm Imf}(6,8,1)\simeq (C_2^5 \rtimes A_6) \rtimes C_2$, 
${\rm Imf}(6,9,1)\simeq {\rm Imf}(6,9,2)\simeq D_6\times S_4$ 
of order $46080$, $46080$, $46080$, $10368$, $103680$, $103680$, $10080$, $10080$, 
$672$, $240$, $240$, $240$, $4608$, $4608$, $23040$, $288$ and $288$ 
of the {\rm CARAT ID}s 
$(6,2773,1)$, $(6,2773,3)$, $(6,2773,2)$, $(6,2803,1)$, $(6,2804,2)$,
$(6,2804,1)$, $(6,2932,1)$, $(6,2932,2)$, $(6,2945,1)$, $(6,2952,1)$,
$(6,2952,3)$, $(6,2952,2)$, $(6,2772,2)$, $(6,2772,5)$, $(6,2750,4)$,
$(6,2866,2)$ and $(6,2866,3)$ respectively.

Let ${\rm Indmf}(n,i,j)\leq \GL(n,\bZ)$ be the 
$j$-th $\bZ$-class of the $i$-th $\bQ$-class of the indecomposable 
maximal finite group of dimension $n$. 
We see that all the groups ${\rm Indmf}(n,i,j)$ coincide with 
${\rm Imf}(n,i,j)$ for $n\leq 5$. 
However, this is not true for $n=6$. 
This phenomenon is suggested by \cite[Section V]{Ple78} 
(see also \cite[Section V]{PH84}). 
Indeed, it turns out that 
we need the one additional group 
${\rm Indmf}(6,10,1)\simeq (C_2\times S_4)^2$ of order $2304$ 
of the {\rm CARAT ID} $(6,5517,4)$ in order to get all 
the indecomposable maximal finite groups. 
Namely, there exist exactly $18$ indecomposable maximal finite groups 
of dimension $6$. 

In summary, there exist exactly $2$ (resp. $3$, $6$, $7$, $17$) 
irreducible maximal finite groups $G\leq \GL(n,\bZ)$ ($\bZ$-classes)
and $2$ (resp. $3$, $6$, $7$, $18$) 
indecomposable maximal finite groups $G\leq \GL(n,\bZ)$ ($\bZ$-classes) 
of dimension $2$ 
(resp. $3$, $4$, $5$, $6$).

We will check this in the next subsection (see 
Example \ref{exKS1} in Subsection \ref{ssKS1} and Example \ref{ex6101}). 
The algorithms given in this section are available from 
{\tt http://math.h.kyoto-u.ac.jp/\~{}yamasaki/Algorithm/} 
as {\tt KS.gap}.\\

\noindent
{\tt IndmfMatrixGroup(n,i,j)} returns ${\rm Indmf}(n,i,j)$ of dimension $n$ 
(this works only for $n\leq 6$).\\
{\tt IndmfNumberQClasses(n)} returns the number of $\bQ$-classes of all 
the indecomposable maximal finite groups of dimension $n$ 
(this works only for $n\leq 6$).\\
{\tt IndmfNumberZClasses(n,i)} returns the number of $\bZ$-classes in 
the $i$-th $\bQ$-class of
the indecomposable maximal finite groups ${\rm Imf}(n,i,j)$ of dimension $n$ 
(this works only for $n\leq 6$).\\
{\tt AllImfMatrixGroups(n)} returns all the irreducible maximal 
finite groups of dimension $n$.\\
{\tt AllIndmfMatrixGroups(n)} returns all the indecomposable maximal 
finite groups of dimension $n$ 
(this works only for $n\leq 6$).

\bigskip

\begin{algorithmIndmf}[Constructing all the indecomposable maximal 
finite groups (Indmf) of dimension $n \leq 6$]
\hfill\break
$($The following algorithm needs the CARAT package of GAP and {\tt caratnumber.gap}$)$. \\
\begin{verbatim}
IndmfMatrixGroup:= function(d,q,z)
    local ans;
    if d=6 and q=10 and z=1 then
        ans:=CaratMatGroupZClass(6,5517,4);
        SetName(ans,"IndmfMatrixGroup(6,10,1)");
        return ans;
    else
        return ImfMatrixGroup(d,q,z);
    fi;
end;

IndmfNumberQClasses:= function(d)
    if d=6 then
        return 10;
    else
        return ImfNumberQClasses(d);
    fi;
end;

IndmfNumberZClasses:= function(d,q)
    if d=6 and q=10 then
        return 1;
    else
        return ImfNumberZClasses(d,q);
    fi;
end;

AllImfMatrixGroups:= function(n)
    local l;
    l:=List([1..ImfNumberQClasses(n)],
      x->List([1..ImfNumberZClasses(n,x)],y->ImfMatrixGroup(n,x,y)));
    return Concatenation(l);
end;

AllIndmfMatrixGroups:= function(n)
    local l;
    l:=List([1..ImfNumberQClasses(n)],
      x->List([1..ImfNumberZClasses(n,x)],y->ImfMatrixGroup(n,x,y)));
    l:=Concatenation(l);
    if n=6 then
        l[18]:=IndmfMatrixGroup(6,10,1);
    fi;
    return l;
end;
\end{verbatim}
\end{algorithmIndmf}

\bigskip

\begin{example}[All the indecomposable maximal finite groups of dimension $n\leq 6$]
{}~{}\\
\begin{verbatim}
gap> Imf2:=AllImfMatrixGroups(2); # Imf2=Indmf2: 2 groups
[ ImfMatrixGroup(2,1,1), ImfMatrixGroup(2,2,1) ]
gap> Length(Imf2);
2
gap> List(Imf2,CrystCatZClass);
[ [ 2, 3, 2, 1 ], [ 2, 4, 4, 1 ] ]
gap> List(Imf2,Size);      
[ 8, 12 ]
gap> List([1..ImfNumberQClasses(2)],                                               
> x->List([1..ImfNumberZClasses(2,x)],y->ImfInvariants(2,x,y).isomorphismType));
[ [ "C2 wr C2 = D8" ], [ "C2 x S3 = C2 x W(A2) = D12" ] ]

gap> Imf3:=AllImfMatrixGroups(3); # Imf3=Indmf3: 3 groups
[ ImfMatrixGroup(3,1,1), ImfMatrixGroup(3,1,2), ImfMatrixGroup(3,1,3) ]
gap> Length(Imf3);                
3
gap> List(Imf3,CrystCatZClass);
[ [ 3, 7, 5, 1 ], [ 3, 7, 5, 2 ], [ 3, 7, 5, 3 ] ]
gap> List(Imf3,Size);      
[ 48, 48, 48 ]
gap> List([1..ImfNumberQClasses(3)],                                               
> x->List([1..ImfNumberZClasses(3,x)],y->ImfInvariants(3,x,y).isomorphismType));
[ [ "C2 wr S3 = C2 x S4 = W(B3)", "C2 wr S3 = C2 x S4 = C2 x W(A3)", 
    "C2 wr S3 = C2 x S4 = C2 x W(A3)" ] ]

gap> Imf4:=AllImfMatrixGroups(4); # Imf4=Indmf4: 6 groups
[ ImfMatrixGroup(4,1,1), ImfMatrixGroup(4,2,1), ImfMatrixGroup(4,3,1), 
  ImfMatrixGroup(4,3,2), ImfMatrixGroup(4,4,1), ImfMatrixGroup(4,5,1) ]
gap> Length(Imf4);
6
gap> List(Imf4,CrystCatZClass);
[ [ 4, 33, 16, 1 ], [ 4, 30, 13, 1 ], [ 4, 31, 7, 1 ], 
  [ 4, 31, 7, 2 ], [ 4, 32, 21, 1 ], [ 4, 29, 9, 1 ] ]
gap> List(Imf4,Size); # the first one is the Wyle group W(F4) of type F4 
[ 1152, 288, 240, 240, 384, 144 ]
gap> List([1..ImfNumberQClasses(4)],                                               
> x->List([1..ImfNumberZClasses(4,x)],y->ImfInvariants(4,x,y).isomorphismType));
[ [ "W(F4)" ], [ "D12 wr C2 = (C2 x W(A2)) wr C2" ], 
  [ "C2 x S5 = C2 x W(A4)", "C2 x S5 = C2 x W(A4)" ], [ "C2 wr S4 = W(B4)" ], 
  [ "(D12 Y D12):C2" ] ]

gap> Imf5:=AllImfMatrixGroups(5); # Imf5=Indmf5: 7 groups
[ ImfMatrixGroup(5,1,1), ImfMatrixGroup(5,1,2), ImfMatrixGroup(5,1,3), 
  ImfMatrixGroup(5,2,1), ImfMatrixGroup(5,2,2), ImfMatrixGroup(5,2,3), 
  ImfMatrixGroup(5,2,4) ]
gap> Length(Imf5);
7
gap> List(Imf5,CaratZClass);     
[ [ 5, 942, 1 ], [ 5, 942, 2 ], [ 5, 942, 3 ], 
  [ 5, 949, 1 ], [ 5, 949, 4 ], [ 5, 949, 2 ], 
  [ 5, 949, 3 ] ]
gap> List(Imf5,Size);
[ 3840, 3840, 3840, 1440, 1440, 1440, 1440 ]
gap> List([1..ImfNumberQClasses(5)],                                               
> x->List([1..ImfNumberZClasses(5,x)],y->ImfInvariants(5,x,y).isomorphismType));
[ [ "C2 wr S5 = W(B5)", "C2 wr S5 = C2 x W(D5)", "C2 wr S5 = C2 x W(D5)" ], 
  [ "C2 x S6", "C2 x S6", "C2 x S6", "C2 x S6" ] ]

gap> Imf6:=AllImfMatrixGroups(6); # Imf6: 17 groups
[ ImfMatrixGroup(6,1,1), ImfMatrixGroup(6,1,2), ImfMatrixGroup(6,1,3), 
  ImfMatrixGroup(6,2,1), ImfMatrixGroup(6,3,1), ImfMatrixGroup(6,3,2), 
  ImfMatrixGroup(6,4,1), ImfMatrixGroup(6,4,2), ImfMatrixGroup(6,5,1), 
  ImfMatrixGroup(6,6,1), ImfMatrixGroup(6,6,2), ImfMatrixGroup(6,6,3), 
  ImfMatrixGroup(6,7,1), ImfMatrixGroup(6,7,2), ImfMatrixGroup(6,8,1), 
  ImfMatrixGroup(6,9,1), ImfMatrixGroup(6,9,2) ]
gap> Length(Imf6);
17
gap> List(Imf6,CaratZClass);     
[ [ 6, 2773, 1 ], [ 6, 2773, 3 ], [ 6, 2773, 2 ], 
  [ 6, 2803, 1 ], [ 6, 2804, 2 ], [ 6, 2804, 1 ], 
  [ 6, 2932, 1 ], [ 6, 2932, 2 ], [ 6, 2945, 1 ], 
  [ 6, 2952, 1 ], [ 6, 2952, 3 ], [ 6, 2952, 2 ], 
  [ 6, 2772, 2 ], [ 6, 2772, 5 ], [ 6, 2750, 4 ], 
  [ 6, 2866, 2 ], [ 6, 2866, 3 ] ]
gap> List(Imf6,Size);
[ 46080, 46080, 46080, 10368, 103680, 103680, 10080, 10080, 672, 240, 
  240, 240, 4608, 4608, 23040, 288, 288 ]
gap> List([1..ImfNumberQClasses(6)],                                               
> x->List([1..ImfNumberZClasses(6,x)],y->ImfInvariants(6,x,y).isomorphismType));
[ [ "C2 wr S6 = W(B6)", "C2 wr S6 = C2 x W(D6)", "C2 wr S6 = C2 x W(D6)" ], 
  [ "(C2 x S3) wr S3 = (C2 x W(A2)) wr S3 = D12 wr S3" ], [ "C2 x W(E6)", "C2 x W(E6)" ], 
  [ "C2 x S7 = C2 x W(A6)", "C2 x S7 = C2 x W(A6)" ], [ "C2 x PGL(2,7)" ], 
  [ "C2 x S5", "C2 x S5", "C2 x S5" ], 
  [ "(C2 x S4) wr C2 = (C2 x W(A3)) wr C2", "(C2 x S4) wr C2 = (C2 x W(A3)) wr C2" ], 
  [ "subgroup of index 2 of C2 wr S6" ], 
  [ "C2 x S3 x S4 = D12 x S4 = C2 x W(A2) x W(A3)", 
    "C2 x S3 x S4 = D12 x S4 = C2 x W(A2) x W(A3)" ] ]

gap> Indmf2:=AllIndmfMatrixGroups(2);;
gap> Indmf3:=AllIndmfMatrixGroups(3);;
gap> Indmf4:=AllIndmfMatrixGroups(4);;
gap> Indmf5:=AllIndmfMatrixGroups(5);;
gap> Indmf6:=AllIndmfMatrixGroups(6);;
gap> List([Indmf2,Indmf3,Indmf4,Indmf5,Indmf6],Length);
[ 2, 3, 6, 7, 18 ]
gap> [Imf2=Indmf2,Imf3=Indmf3,Imf4=Indmf4,Imf5=Indmf5,Imf6=Indmf6];
[ true, true, true, true, false ]

gap> Indmf6; # Indmf6: 18 (=17+1) groups
[ ImfMatrixGroup(6,1,1), ImfMatrixGroup(6,1,2), ImfMatrixGroup(6,1,3), 
  ImfMatrixGroup(6,2,1), ImfMatrixGroup(6,3,1), ImfMatrixGroup(6,3,2), 
  ImfMatrixGroup(6,4,1), ImfMatrixGroup(6,4,2), ImfMatrixGroup(6,5,1), 
  ImfMatrixGroup(6,6,1), ImfMatrixGroup(6,6,2), ImfMatrixGroup(6,6,3), 
  ImfMatrixGroup(6,7,1), ImfMatrixGroup(6,7,2), ImfMatrixGroup(6,8,1), 
  ImfMatrixGroup(6,9,1), ImfMatrixGroup(6,9,2), IndmfMatrixGroup(6,10,1) ]
gap> CaratZClass(IndmfMatrixGroup(6,10,1));
[ 6, 5517, 4 ]
gap> Size(IndmfMatrixGroup(6,10,1));
2304
gap> StructureDescription(IndmfMatrixGroup(6,10,1)); 
"C2 x C2 x S4 x S4"
\end{verbatim}
\end{example}


\bigskip

\subsection{Krull-Schmidt theorem (1)}\label{ssKS1}

We will determine all the possible decompositions of $M_G$ into indecomposable 
ones for all finite subgroups $G$ of $\GL(n,\bZ)$ with $n \leq 6$ 
(see Examples \ref{exKS1} and Examples \ref{exKS2} below). 
Note that if a $G$-lattice $M$ splits into indecomposable $G$-lattices $M_1$ 
and $M_2$ of rank $i$ and $j$,
then $G$ is a subgroup of $G_1\times G_2$
where $G_1$ and $G_2$ are some indecomposable maximal finite groups
of $\GL(i,\bZ)$ and $\GL(j,\bZ)$ respectively. 

\bigskip

\noindent
{\tt LatticeDecompositions(n)} returns the list $\mathcal{L}=\{l_1,\ldots,l_s\}$ 
of the lists $l_i$ of the GAP IDs whose $i$-th list $l_i$ contains all the 
GAP IDs of the groups $G$ whose corresponding $G$-lattice $M_G$ 
is decomposable into the indecomposable $G$-lattices 
$M\simeq M_1\oplus\cdots\oplus M_m$ with rank $M_j=r_j$ 
where $(r_1,\ldots,r_m)$ corresponds to the $i$-th partitions 
{\tt Partitions(n)[i]} of the integer $2\leq n\leq 4$.\\
{\tt LatticeDecompositions(n:Carat)} returns the same as 
{\tt LatticeDecompositions(n)} but using the CARAT ID instead of the GAP ID. 
This algorithm is valid for $1\leq n\leq 6$.\\
{\tt LatticeDecompositions(n:Carat,FromPerm)} returns the same as 
{\tt LatticeDecompositions(n:Carat)} but using 
{\tt ConjugacyClassesSubgroupsFromPerm(G)} instead of 
{\tt ConjugacyClassesSubgroups2(G)} (see below).

\bigskip

In order to construct conjugacy classes of subgroups 
of a group $G$, we use the following GAP function 
{\tt ConjugacyClassesSubgroups2(G)} because the ordering of 
the conjugacy classes of subgroups of $G$ by the built-in function 
{\tt ConjugacyClassesSubgroups(G)} is not fixed for some groups. 

\bigskip

\begin{verbatim}
ConjugacyClassesSubgroups2:= function(g)
    Reset(GlobalMersenneTwister);
    Reset(GlobalRandomSource);
    return ConjugacyClassesSubgroups(g);
end;
\end{verbatim}

\bigskip

If a group $G$ is too big, {\tt ConjugacyClassesSubgroups2(G)} 
may not work well. 
For example, we should use 
{\tt ConjugacyClassesSubgroupsFromPerm(G)} instead of 
{\tt ConjugacyClassesSubgroups2(G)} 
for the 2nd $\bZ$-class of 
the 1st $\bQ$-class of the irreducible maximal 
finite group ${\rm Imf}(6,1,2)\simeq C_2^6\rtimes S_6$ of dimension $6$ 
of order $46080$. 

\bigskip

\begin{verbatim}
ConjugacyClassesSubgroupsFromPerm:= function(g)
    local iso,h,i;
    Reset(GlobalMersenneTwister);
    Reset(GlobalRandomSource);
    iso:=IsomorphismPermGroup(g);
    h:=ConjugacyClassesSubgroups2(Range(iso));
    h:=List(h,Representative);
    h:=List(h,x->PreImage(iso,x));
    return h;
end;
\end{verbatim}

\bigskip

\begin{algorithmKS1}[All the decomposable $G$-lattices $M_G$ of rank $n$]
\hfill\break
$($The following algorithm needs the CARAT package of GAP and {\tt caratnumber.gap}$)$. \\
\begin{verbatim}
DirectSumMatrixGroup:= function(l)
    local gg,gg1;
    gg:=List(l,GeneratorsOfGroup);
    if Length(Set(gg,Length))>1 then
        return fail;
    else
        gg1:=List([1..Length(gg[1])],x->DirectSumMat(List(gg,y->y[x])));
    fi;
    return Group(gg1,DirectSumMat(List(l,Identity)));
end;

DirectProductMatrixGroup:= function(l)
    local gg,gg1,o,o1,i,j,gx;
    gg:=List(l,GeneratorsOfGroup);
    gg1:=[];
    for i in [1..Length(l)] do
        o:=List(l,Identity);
        for j in gg[i] do
            o[i]:=j;
            Add(gg1,DirectSumMat(o));
        od;
    od;
    return Group(gg1,DirectSumMat(List(l,Identity)));
end;

AllDirectProductIndmfMatrixGroups:= function(l)
    local li;
    li:=List(Collected(l),
      x->UnorderedTuples(AllIndmfMatrixGroups(x[1]),x[2]));
    return List(Cartesian(li),
      x->DirectProductMatrixGroup(Concatenation(x)));
end;

PartialMatrixGroup:= function(G,l)
    local gg,gp;
    gg:=GeneratorsOfGroup(G);
    gp:=List(gg,x->x{l}{l});
    return Group(gp,IdentityMat(Length(l)));
end;

LatticeDecompositions:= function(n)
    local d,ind,pp,p1,subgr,ld;
    ind:=[];
    for d in [1..n] do
        pp:=Partitions(d);
        p1:=List(pp,SortedList);
        if ValueOption("fromperm")=true or ValueOption("FromPerm")=true then
            subgr:=List(p1,x->
              Concatenation(List(AllDirectProductIndmfMatrixGroups(x),
              y->ConjugacyClassesSubgroupsFromPerm(y))));
        else
            subgr:=List(p1,x->
              Concatenation(List(AllDirectProductIndmfMatrixGroups(x),
              y->List(ConjugacyClassesSubgroups2(y),Representative))));
        fi;
        subgr:=List([1..NrPartitions(d)],
          x->Filtered(subgr[x],y->ForAll([1..Length(p1[x])],z->
          p1[x][z]=1 or p1[x][z]=d or
          CaratZClass(PartialMatrixGroup(y,
          [Sum([1..z-1],w->p1[x][w])+1..Sum([1..z],w->p1[x][w])]))
          in ind[p1[x][z]])));
        ld:=List(subgr,x->Set(x,CaratZClass));
        ld[NrPartitions(d)]:=Difference(ld[NrPartitions(d)],
          Union(List([1..NrPartitions(d)-1],x->ld[x])));
        ind[d]:=ld[NrPartitions(d)];
    od;
    if ValueOption("carat")=true or ValueOption("Carat")=true then
        return ld;
    else
        return List(ld,x->Set(x,Carat2CrystCat));
    fi;
end;
\end{verbatim}
\end{algorithmKS1}

\bigskip

\begin{example}[Functions {\tt DirectSumMatrixGroup(l)} and {\tt DirectProductMatrixGroup(l)}]
{}~{}\\

\noindent
{\tt DirectSumMatrixGroup(l)} returns the direct sum of the groups $G_1,\ldots,G_n$ 
for the list $l=[G_1,\ldots,G_n]$.\\
{\tt DirectProductMatrixGroup(l)} returns the direct product of the groups $G_1,\ldots,G_n$ 
for the list $l=[G_1,\ldots,G_n]$.\\

\begin{verbatim}
gap> Read("caratnumber.gap");
gap> Read("KS.gap");

gap> G:=ImfMatrixGroup(1,1,1); # G=C2 of dimension 1
ImfMatrixGroup(1,1,1)
gap> GeneratorsOfGroup(G);
[ [ [ -1 ] ] ]
gap> Gs3:=DirectSumMatrixGroup([G,G,G]); # Gs3=C2 of dimension 3
Group([ [ [ -1, 0, 0 ], [ 0, -1, 0 ], [ 0, 0, -1 ] ] ])
gap> Order(Gs3);
2
gap> Gp3:=DirectProductMatrixGroup([G,G,G]); # Gp3=C2^3 of dimension 3
<matrix group with 3 generators>
gap> GeneratorsOfGroup(Gp3);                                                       
[ [ [ -1, 0, 0 ], 
    [ 0, 1, 0 ], 
    [ 0, 0, 1 ] ], 
  [ [ 1, 0, 0 ], 
    [ 0, -1, 0 ], 
    [ 0, 0, 1 ] ], 
  [ [ 1, 0, 0 ], 
    [ 0, 1, 0 ], 
    [ 0, 0, -1 ] ] ]
gap> Size(Gp3);
8
\end{verbatim}
\end{example}

\bigskip

\begin{example}[Algorithm KS1]\label{exKS1}
By Algorithm KS1, we may check that the condition 
(KS1) holds when $\rank M_G=n\leq 4$ and fails for 
the $11$ cases when $\rank M_G=5$ and 
the $131$ cases when $\rank M_G=6$ as in Theorem \ref{thKS}. 
The decomposition types of the $11$ $G$-lattices $M_G$ of rank $5$ are 
$(3,2)$ and $(4,1)$ and $G$ is a subgroup of the 
group $C_2\times D_6$ of the CARAT ID $(5,205,6)$.
The decomposition types of the former $51$ cases 
(resp. the latter $80$ cases) out of the $131$ cases 
are $(3,2,1)$ and $(4,1,1)$ 
$($resp. $(3,3)$ and $(5,1)$$)$ and $G$ is a subgroup of the group 
$C_2^2\times D_6$ of the CARAT ID $(6,2139,9)$ 
$($resp. $D_6\times D_4$ of the CARAT ID $(6,145,4)$$)$. 
For the groups $G_1$ and $G_2$ of the CARAT IDs 
$(6,2139,9)$ and $(5,205,6)$, we see that $M_{G_1}\simeq M_{G_2}\oplus N$ 
for some $G$-lattice $N$ of rank $1$. 

The computation in this example also confirms that the $2$ (resp. $3$, $6$, $7$) 
irreducible maximal finite groups ${\rm Imf}(n,i,j)$ of dimension $2$ (resp. $3$, $4$, $5$) 
become the indecomposable maximal finite groups, 
and the $18$ groups given in Subsection \ref{ssIndmf} are 
the indecomposable maximal finite groups of dimension $6$ (see also Example \ref{ex6101}). 
Indeed, we get the $18996$ indecomposable (conjugacy classes of) 
subgroups of $\GL(6,\bZ)$ as the subgroups of the $18$ indecomposable maximal finite 
groups although we obtain only $14348$ subgroups if we use the $17$ irreducible ones. 

\bigskip

\begin{verbatim}
gap> Read("caratnumber.gap");
gap> Read("KS.gap");

gap> ld1:=LatticeDecompositions(1:Carat);
[ [ [ 1, 1, 1 ], [ 1, 2, 1 ] ] ]

gap> ld2:=LatticeDecompositions(2);      
[ [ [ 2, 1, 1, 1 ], [ 2, 1, 2, 1 ], [ 2, 2, 1, 1 ], [ 2, 2, 2, 1 ] ], 
  [ [ 2, 2, 1, 2 ], [ 2, 2, 2, 2 ], [ 2, 3, 1, 1 ], [ 2, 3, 2, 1 ], [ 2, 4, 1, 1 ], 
    [ 2, 4, 2, 1 ], [ 2, 4, 2, 2 ], [ 2, 4, 3, 1 ], [ 2, 4, 4, 1 ] ] ]
gap> Partitions(2);
[ [ 1, 1 ], [ 2 ] ]
gap> List(ld2,Length);                   
[ 4, 9 ]
gap> [Length(Union(ld2)),Sum(ld2,Length)];
[ 13, 13 ]

gap> ld3:=LatticeDecompositions(3);                         
[ [ [ 3, 1, 1, 1 ], [ 3, 1, 2, 1 ], [ 3, 2, 1, 1 ], [ 3, 2, 2, 1 ], [ 3, 2, 3, 1 ], 
    [ 3, 3, 1, 1 ], [ 3, 3, 2, 1 ], [ 3, 3, 3, 1 ] ], 
  [ [ 3, 2, 1, 2 ], [ 3, 2, 2, 2 ], [ 3, 2, 3, 2 ], [ 3, 3, 1, 2 ], [ 3, 3, 2, 2 ], 
    [ 3, 3, 2, 3 ], [ 3, 3, 3, 2 ], [ 3, 4, 1, 1 ], [ 3, 4, 2, 1 ], [ 3, 4, 3, 1 ], 
    [ 3, 4, 4, 1 ], [ 3, 4, 5, 1 ], [ 3, 4, 6, 1 ], [ 3, 4, 6, 2 ], [ 3, 4, 7, 1 ], 
    [ 3, 5, 1, 2 ], [ 3, 5, 2, 2 ], [ 3, 5, 3, 2 ], [ 3, 5, 3, 3 ], [ 3, 5, 4, 2 ], 
    [ 3, 5, 4, 3 ], [ 3, 5, 5, 2 ], [ 3, 5, 5, 3 ], [ 3, 6, 1, 1 ], [ 3, 6, 2, 1 ], 
    [ 3, 6, 3, 1 ], [ 3, 6, 4, 1 ], [ 3, 6, 5, 1 ], [ 3, 6, 6, 1 ], [ 3, 6, 6, 2 ], 
    [ 3, 6, 7, 1 ] ], 
  [ [ 3, 3, 1, 3 ], [ 3, 3, 1, 4 ], [ 3, 3, 2, 4 ], [ 3, 3, 2, 5 ], [ 3, 3, 3, 3 ], 
    [ 3, 3, 3, 4 ], [ 3, 4, 1, 2 ], [ 3, 4, 2, 2 ], [ 3, 4, 3, 2 ], [ 3, 4, 4, 2 ], 
    [ 3, 4, 5, 2 ], [ 3, 4, 6, 3 ], [ 3, 4, 6, 4 ], [ 3, 4, 7, 2 ], [ 3, 5, 1, 1 ], 
    [ 3, 5, 2, 1 ], [ 3, 5, 3, 1 ], [ 3, 5, 4, 1 ], [ 3, 5, 5, 1 ], [ 3, 7, 1, 1 ], 
    [ 3, 7, 1, 2 ], [ 3, 7, 1, 3 ], [ 3, 7, 2, 1 ], [ 3, 7, 2, 2 ], [ 3, 7, 2, 3 ], 
    [ 3, 7, 3, 1 ], [ 3, 7, 3, 2 ], [ 3, 7, 3, 3 ], [ 3, 7, 4, 1 ], [ 3, 7, 4, 2 ], 
    [ 3, 7, 4, 3 ], [ 3, 7, 5, 1 ], [ 3, 7, 5, 2 ], [ 3, 7, 5, 3 ] ] ]
gap> Partitions(3);   
[ [ 1, 1, 1 ], [ 2, 1 ], [ 3 ] ]
gap> List(ld3,Length);             
[ 8, 31, 34 ]
gap> [Length(Union(ld3)),Sum(ld3,Length)];
[ 73, 73 ]

gap> ld4:=LatticeDecompositions(4);;
gap> Partitions(4);
[ [ 1, 1, 1, 1 ], [ 2, 1, 1 ], [ 2, 2 ], [ 3, 1 ], [ 4 ] ]
gap> List(ld4,Length);
[ 16, 96, 175, 128, 295 ]
gap> [Length(Union(ld4)),Sum(ld4,Length)];
[ 710, 710 ]

gap> ld5:=LatticeDecompositions(5:Carat);;
gap> Partitions(5);
[ [ 1, 1, 1, 1, 1 ], [ 2, 1, 1, 1 ], [ 2, 2, 1 ], [ 3, 1, 1 ], [ 3, 2 ], [ 4, 1 ], [ 5 ] ]
gap> List(ld5,Length);
[ 32, 280, 1004, 442, 1480, 1400, 1452 ]
gap> [Length(Union(ld5)),Sum(ld5,Length)];
[ 6079, 6090 ]
gap> last[2]-last[1];
11
gap> f32_41:=Intersection(ld5[5],ld5[6]); # type [ 3, 2 ] and [ 4, 1 ]
[ [ 5, 188, 4 ], [ 5, 189, 4 ], [ 5, 190, 6 ], [ 5, 191, 6 ], [ 5, 192, 6 ], 
  [ 5, 193, 4 ], [ 5, 205, 6 ], [ 5, 218, 8 ], [ 5, 219, 8 ], [ 5, 220, 4 ], 
  [ 5, 221, 4 ] ]
gap> Length(f32_41);
11

gap> f32_41g:=List(f32_41,x->CaratMatGroupZClass(x[1],x[2],x[3]));;
gap> List(f32_41g,Order);
[ 12, 12, 12, 12, 12, 12, 24, 6, 6, 6, 6 ]
gap> List(f32_41g,StructureDescription);
[ "D12", "C6 x C2", "D12", "D12", "D12", "D12", "C2 x C2 x S3", "S3", "S3", "C6", "C6" ]
gap> last[7]; # the 7th group of the CARAT ID [ 5, 205, 6 ] is maximal
"C2 x C2 x S3"
gap> f32_41gsub:=Set(ConjugacyClassesSubgroups2(f32_41g[7]),
> x->CaratZClass(Representative(x)));;
gap> Difference(f32_41,f32_41gsub);
[  ]
gap> f32_41[7];
[ 5, 205, 6 ]

gap> ld6:=LatticeDecompositions(6:Carat,FromPerm);;
gap> Partitions(6);
[ [ 1, 1, 1, 1, 1, 1 ], [ 2, 1, 1, 1, 1 ], [ 2, 2, 1, 1 ], [ 2, 2, 2 ], 
  [ 3, 1, 1, 1 ], [ 3, 2, 1 ], [ 3, 3 ], [ 4, 1, 1 ], [ 4, 2 ], [ 5, 1 ], [ 6 ] ]
gap> List(ld6,Length);
[ 68, 824, 4862, 6878, 1466, 10662, 4235, 5944, 21573, 9931, 18996 ]
gap> [Length(Union(ld6)),Sum(ld6,Length)];
[ 85308, 85439 ]
gap> last[2]-last[1];
131
gap> f321_411:=Intersection(ld6[6],ld6[8]); # type [ 3, 2, 1 ] and [ 4, 1, 1 ]
[ [ 6, 2013, 8 ], [ 6, 2018, 4 ], [ 6, 2023, 6 ], [ 6, 2024, 6 ], [ 6, 2025, 6 ], 
  [ 6, 2026, 6 ], [ 6, 2033, 6 ], [ 6, 2042, 8 ], [ 6, 2043, 8 ], [ 6, 2044, 4 ], 
  [ 6, 2045, 4 ], [ 6, 2048, 5 ], [ 6, 2049, 8 ], [ 6, 2050, 8 ], [ 6, 2051, 8 ], 
  [ 6, 2052, 8 ], [ 6, 2058, 5 ], [ 6, 2059, 5 ], [ 6, 2067, 5 ], [ 6, 2068, 5 ], 
  [ 6, 2069, 5 ], [ 6, 2069, 11 ], [ 6, 2070, 9 ], [ 6, 2071, 9 ], [ 6, 2072, 10 ], 
  [ 6, 2072, 11 ], [ 6, 2076, 24 ], [ 6, 2076, 25 ], [ 6, 2077, 24 ], [ 6, 2077, 25 ], 
  [ 6, 2078, 24 ], [ 6, 2078, 25 ], [ 6, 2079, 24 ], [ 6, 2079, 25 ], [ 6, 2087, 15 ], 
  [ 6, 2088, 15 ], [ 6, 2089, 17 ], [ 6, 2089, 18 ], [ 6, 2094, 9 ], [ 6, 2102, 24 ], 
  [ 6, 2102, 25 ], [ 6, 2105, 9 ], [ 6, 2106, 9 ], [ 6, 2107, 10 ], [ 6, 2107, 11 ], 
  [ 6, 2108, 15 ], [ 6, 2109, 15 ], [ 6, 2110, 17 ], [ 6, 2110, 18 ], [ 6, 2111, 15 ], 
  [ 6, 2139, 9 ] ]
gap> Length(f321_411);
51

gap> f321_411g:=List(f321_411,x->CaratMatGroupZClass(x[1],x[2],x[3]));;
gap> List(f321_411g,Order);
[ 12, 12, 12, 12, 12, 12, 24, 6, 6, 6, 6, 12, 12, 12, 12, 12, 12, 24, 6, 6, 6, 6, 12, 12, 
  12, 12, 12, 12, 12, 12, 12, 12, 12, 12, 12, 12, 12, 12, 24, 24, 24, 24, 24, 24, 24, 24, 
  24, 24, 24, 24, 48 ]
gap> List(f321_411g,StructureDescription);
[ "D12", "C6 x C2", "D12", "D12", "D12", "D12", "C2 x C2 x S3", "S3", "S3", "C6", "C6", 
  "C6 x C2", "D12", "D12", "D12", "D12", "D12", "C2 x C2 x S3", "C6", "C6", "S3", "S3", 
  "C6 x C2", "C6 x C2", "C6 x C2", "C6 x C2", "D12", "D12", "D12", "D12", "D12", "D12", 
  "D12", "D12", "D12", "D12", "D12", "D12", "C6 x C2 x C2", "C2 x C2 x S3", "C2 x C2 x S3", 
  "C2 x C2 x S3", "C2 x C2 x S3", "C2 x C2 x S3", "C2 x C2 x S3", "C2 x C2 x S3", 
  "C2 x C2 x S3", "C2 x C2 x S3", "C2 x C2 x S3", "C2 x C2 x S3", "C2 x C2 x C2 x S3" ]
gap> last[51]; # the 51th group of the CARAT ID [ 6, 2139, 9 ] is maximal
"C2 x C2 x C2 x S3"
gap> f321_411gsub:=Set(ConjugacyClassesSubgroups2(f321_411g[51]),
> x->CaratZClass(Representative(x)));;
gap> Difference(f321_411,f321_411gsub);
[  ]
gap> f321_411[51];
[ 6, 2139, 9 ]

gap> gg:=GeneratorsOfGroup(CaratMatGroupZClass(5,205,6));;
gap> gg:=List(gg,x->DirectSumMat([[1]],x));;
gap> Add(gg,-IdentityMat(6));
gap> CaratZClass(Group(gg));
[ 6, 2139, 9 ]

gap> f33_51:=Intersection(ld6[7],ld6[10]); # type [ 3, 3 ] and [ 5, 1 ]
[ [ 6, 40, 4 ], [ 6, 41, 4 ], [ 6, 44, 6 ], [ 6, 45, 6 ], [ 6, 47, 4 ],
  [ 6, 53, 4 ], [ 6, 54, 4 ], [ 6, 54, 8 ], [ 6, 55, 4 ], [ 6, 63, 4 ],
  [ 6, 64, 6 ], [ 6, 65, 4 ], [ 6, 66, 6 ], [ 6, 67, 6 ], [ 6, 75, 4 ],
  [ 6, 75, 8 ], [ 6, 76, 8 ], [ 6, 76, 12 ], [ 6, 77, 8 ], [ 6, 77, 12 ],
  [ 6, 78, 4 ], [ 6, 78, 8 ], [ 6, 79, 6 ], [ 6, 80, 4 ], [ 6, 81, 8 ],
  [ 6, 81, 12 ], [ 6, 90, 4 ], [ 6, 99, 4 ], [ 6, 108, 4 ], [ 6, 108, 8 ],
  [ 6, 109, 8 ], [ 6, 109, 12 ], [ 6, 110, 4 ], [ 6, 111, 6 ], [ 6, 112, 8 ],
  [ 6, 112, 12 ], [ 6, 113, 4 ], [ 6, 114, 6 ], [ 6, 115, 6 ], [ 6, 145, 4 ],
  [ 6, 2070, 10 ], [ 6, 2070, 11 ], [ 6, 2071, 10 ], [ 6, 2071, 11 ], [ 6, 2072, 12 ], 
  [ 6, 2072, 13 ], [ 6, 2076, 26 ], [ 6, 2076, 27 ], [ 6, 2077, 26 ], [ 6, 2077, 27 ], 
  [ 6, 2078, 26 ], [ 6, 2078, 27 ], [ 6, 2079, 26 ], [ 6, 2079, 27 ], [ 6, 2087, 16 ], 
  [ 6, 2087, 17 ], [ 6, 2088, 16 ], [ 6, 2088, 17 ], [ 6, 2089, 19 ], [ 6, 2089, 20 ], 
  [ 6, 2094, 10 ], [ 6, 2094, 11 ], [ 6, 2102, 26 ], [ 6, 2102, 27 ], [ 6, 2105, 10 ], 
  [ 6, 2105, 11 ], [ 6, 2106, 10 ], [ 6, 2106, 11 ], [ 6, 2107, 12 ], [ 6, 2107, 13 ], 
  [ 6, 2108, 16 ], [ 6, 2108, 17 ], [ 6, 2109, 16 ], [ 6, 2109, 17 ], [ 6, 2110, 19 ], 
  [ 6, 2110, 20 ], [ 6, 2111, 16 ], [ 6, 2111, 17 ], [ 6, 2139, 10 ], [ 6, 2139, 11 ] ]
gap> Length(f33_51);
80

gap> f33_51g:=List(f33_51,x->CaratMatGroupZClass(x[1],x[2],x[3]));;
gap> List(f33_51g,Order);
[ 12, 12, 12, 12, 24, 24, 24, 24, 24, 24, 24, 24, 24, 24, 24, 24, 24, 24, 24, 24, 24, 24, 
  24, 24, 24, 24, 48, 48, 48, 48, 48, 48, 48, 48, 48, 48, 48, 48, 48, 96, 12, 12, 12, 12, 
  12, 12, 12, 12, 12, 12, 12, 12, 12, 12, 12, 12, 12, 12, 12, 12, 24, 24, 24, 24, 24, 24, 
  24, 24, 24, 24, 24, 24, 24, 24, 24, 24, 24, 24, 48, 48 ]
gap> List(f33_51g,StructureDescription);
[ "C12", "C12", "C3 : C4", "C3 : C4", "C12 x C2", "C3 x D8", "C3 x D8", "C3 x D8", 
  "C3 x D8", "C4 x S3", "C4 x S3", "C4 x S3", "C4 x S3", "C2 x (C3 : C4)", "(C6 x C2) : C2", 
  "(C6 x C2) : C2", "(C6 x C2) : C2", "(C6 x C2) : C2", "(C6 x C2) : C2", "(C6 x C2) : C2", 
  "(C6 x C2) : C2", "(C6 x C2) : C2", "D24", "D24", "D24", "D24", "C6 x D8", "C2 x C4 x S3", 
  "D8 x S3", "D8 x S3", "C2 x ((C6 x C2) : C2)", "C2 x ((C6 x C2) : C2)", "D8 x S3", 
  "D8 x S3", "D8 x S3", "D8 x S3", "D8 x S3", "D8 x S3", "C2 x D24", "C2 x S3 x D8", 
  "C6 x C2", "C6 x C2", "C6 x C2", "C6 x C2", "C6 x C2", "C6 x C2", "D12", "D12", "D12", 
  "D12", "D12", "D12", "D12", "D12", "D12", "D12", "D12", "D12", "D12", "D12", 
  "C6 x C2 x C2", "C6 x C2 x C2", "C2 x C2 x S3", "C2 x C2 x S3", "C2 x C2 x S3", 
  "C2 x C2 x S3", "C2 x C2 x S3", "C2 x C2 x S3", "C2 x C2 x S3", "C2 x C2 x S3", 
  "C2 x C2 x S3", "C2 x C2 x S3", "C2 x C2 x S3", "C2 x C2 x S3", "C2 x C2 x S3", 
  "C2 x C2 x S3", "C2 x C2 x S3", "C2 x C2 x S3", "C2 x C2 x C2 x S3", "C2 x C2 x C2 x S3" ]
gap> last[40]; # the 40th group of the CARAT ID [ 6, 145, 4 ] is maximal
"C2 x S3 x D8"
gap> f33_51gsub:=Set(ConjugacyClassesSubgroups2(f33_51g[40]),
> x->CaratZClass(Representative(x)));;
gap> Difference(f33_51,f33_51gsub);
[  ]
gap> f33_51[40];
[ 6, 145, 4 ]
\end{verbatim}
\end{example}

\bigskip

\begin{example}[The generators of the exceptional $11$ groups $G\leq \GL(5,\bZ)$ 
whose corresponding $G$-lattices $M_G\simeq M_1\oplus M_2\simeq N_1\oplus N_2$ 
with rank $M_1=4$, rank $M_2=1$, rank $N_1=3$ and rank $N_2=2$]
Let $I$ be the identity matrix of rank $5$. 
Let
\[
X=
{\scriptsize 
\left(
\begin{array}{cccc|c}
 0 & 1 & 0 & 1 & 0 \\
 1 & 0 & 0 & 1 & 0 \\
 0 & 0 & 0 & 1 & 0 \\
 0 & 0 &-1 &-1 & 0 \\\hline
 0 & 0 & 0 & 0 & 1
\end{array}
\right)},\ \ 
Y=
{\scriptsize 
\left(
\begin{array}{cccc|c}
 1 & 0 & 0 & 0 & 0 \\
 0 & 1 & 0 & 0 & 0 \\
 0 & 0 & 0 &-1 & 0 \\
 0 & 0 &-1 & 0 & 0 \\\hline
 0 & 0 & 0 & 0 & 1
\end{array}
\right)}.
\]
Then the CARAT ID and the generators of the exceptional 
$11$ groups $G\leq \GL(5,\bZ)$ are given as 

%
\begin{center}
\begin{tabular}{lll}
$G$ & CARAT ID & Generators\\\hline
$S_3\times C_2\simeq D_6$ & $(5,188,4)$ & $\langle X^2, XY, -I\rangle$\\
$C_2\times C_6$ & $(5,189,4)$ & $\langle X, -I\rangle$\\
$D_6$ & $(5,190,6)$ & $\langle -X, Y\rangle$\\
$D_6$ & $(5,191,6)$ & $\langle -X, XY\rangle$\\
$D_6$ & $(5,192,6)$ & $\langle X, Y\rangle$\\
$D_6$ & $(5,193,4)$ & $\langle X, -Y\rangle$\\
$D_6\times C_2$ & $(5,205,6)$ & $\langle X, Y, -I\rangle$\\
$S_3$ & $(5,218,8)$ & $\langle X^2, XY\rangle$\\
$S_3$ & $(5,219,8)$ & $\langle X^2, -XY\rangle$\\
$C_6$ & $(5,220,4)$ & $\langle X\rangle$\\
$C_6$ & $(5,221,4)$ & $\langle -X\rangle$
\end{tabular}
\end{center}
The group $\langle X, Y \rangle\leq \GL(5,\bZ)$ may be regarded as the 
group embedded directly by the group $G\leq \GL(4,\bZ)$ 
of the GAP ID $(4,14,8,2)$. 
We also see that
\begin{align*} 
P^{-1}XP=
{\scriptsize 
\left(
\begin{array}{ccc|cc}
 0 & 1 & 0 & 0 & 0 \\
 0 & 0 & 1 & 0 & 0 \\
 1 & 0 & 0 & 0 & 0 \\\hline
 0 & 0 & 0 & 0 & 1 \\
 0 & 0 & 0 & 1 & 0
\end{array}
\right)},\ \ 
P^{-1}YP=
{\scriptsize 
\left(
\begin{array}{ccc|cc}
 0 & 1 & 0 & 0 & 0 \\
 1 & 0 & 0 & 0 & 0 \\
 0 & 0 & 1 & 0 & 0 \\\hline
 0 & 0 & 0 & 1 & 0 \\
 0 & 0 & 0 & 0 & 1
\end{array}
\right)}
\ \ {\rm where}\ \ 
P={\scriptsize \begin{pmatrix}
 0 & 0 & 1 & 0 & 1 \\
 0 & 0 & 1 & 1 & 0 \\
 0 &-1 & 1 & 0 & 0 \\
 1 & 0 &-1 & 0 & 0 \\
 1 & 1 & 1 & 1 & 1
\end{pmatrix}}.
\end{align*}
This shows that 
$M_G\simeq M_1\oplus M_2\simeq N_1\oplus N_2$ 
where $M_1$, $M_2$, $N_1$ and $N_2$ are indecomposable 
$G$-lattices with 
rank $M_1=4$, rank $M_2=1$, rank $N_1=3$ and rank $N_2=2$. 
\end{example}

\bigskip

\subsection{Krull-Schmidt theorem (2)}
We will determine whether the condition {\rm (KS2)} holds.

\bigskip

\begin{algorithmKS2}[Determination whether the condition {\rm (KS2)} holds]
\label{AlgKS2}
\hfill\break
$($The following algorithm needs the CARAT package of GAP and {\tt caratnumber.gap}$)$. \\
\begin{verbatim}
InverseProjection:= function(l)
    local lc,lcg,lg,ll,G,N,i,j,k,gn,gn1,gn2,o,gN,h,h1,ans,ans1;
    lc:=Collected(l);
    if ValueOption("carat")=true or ValueOption("Carat")=true then
        lcg:=List(lc,x->
          [CaratMatGroupZClass(x[1][1],x[1][2],x[1][3]),x[2]]);
    else
        lcg:=List(lc,x->
          [MatGroupZClass(x[1][1],x[1][2],x[1][3],x[1][4]),x[2]]);
    fi;
    lg:=Concatenation(List(lcg,x->List([1..x[2]],y->x[1])));
    ll:=Concatenation(List(lc,x->List([1..x[2]],y->x[1][1])));
    G:=DirectProductMatrixGroup(lg);
    gn:=[];
    for i in [1..Length(lc)] do
        if lc[i][1][1]=1 then
            gn1:=[[[-1]]];
        else
            gn1:=GeneratorsOfGroup(Normalizer(GL(lc[i][1][1],Integers),
              lcg[i][1]));
        fi;
        gn2:=[];
        for j in [1..lc[i][2]] do
            o:=List([1..lc[i][2]],x->IdentityMat(lc[i][1][1]));
            for k in gn1 do
                o[j]:=k;
                Add(gn2,DirectSumMat(o));
            od;
        od;
        gn1:=GeneratorsOfGroup(SymmetricGroup(lc[i][2]));
        gn2:=Concatenation(gn2,List(gn1,
          x->KroneckerProduct(PermutationMat(x,lc[i][2]),
          IdentityMat(lc[i][1][1]))));
        Add(gn,Group(gn2));
    od;
    N:=DirectProductMatrixGroup(gn);
    ans1:=List(ConjugacyClassesSubgroups2(G),Representative);
    ans1:=Filtered(ans1,x->
      ForAll([1..Length(lg)],y->PartialMatrixGroup(x,
      [Sum([1..y-1],z->ll[z])+1..Sum([1..y],z->ll[z])])=lg[y]));
    ans:=[];
    gN:=GeneratorsOfGroup(N);
    while ans1<>[] do
        h:=[];
        h1:=[ans1[1]];
        while h1<>[] do
            h:=Union(h,h1);
            h1:=Difference(Concatenation(List(h1,x->List(gN,y->x^y))),h);
        od;
        Add(ans,ans1[1]);
        ans1:=Difference(ans1,h);
    od;
    return ans;
end;
\end{verbatim}
\end{algorithmKS2}

\bigskip

\begin{example}[{{\tt InverseProjection([l1,l2])}}] 
Let $G_1\simeq C_4\leq \GL(2,\bZ)$ (resp. $G_2\simeq D_6\leq \GL(2,\bZ)$) 
be a cyclic group of order $4$ 
(resp. dihedral group of order $12$) of the GAP ID $l_1=(2,3,1,1)$ 
(resp. $l_2=(2,4,4,1)$). \\

\noindent
{\tt InverseProjection([l1,l2])} returns the list of all groups 
$G\leq \GL(4,\bZ)$ such that $M_G\simeq M_{G_1}\oplus M_{G_2}$ and 
the GAP ID of $G_1$ (resp. $G_2$) is $l_1$ (resp. $l_2$). \\
{\tt InverseProjection([l1,l2]:Carat)} returns the same as 
{\tt InverseProjection([l1,l2])} but with respect to the CARAT ID 
$l_1$ and $l_2$ instead of the GAP ID.\\

\begin{verbatim}
Read("crystcat.gap");
Read("caratnumber.gap");
Read("KS.gap");

gap> pgs:=InverseProjection([[2,3,1,1],[2,4,4,1]]);
[ <matrix group of size 24 with 4 generators>, 
  <matrix group of size 48 with 5 generators>, 
  <matrix group of size 24 with 4 generators>, 
  <matrix group of size 24 with 4 generators> ]
gap> List(pgs,StructureDescription);
[ "C2 x (C3 : C4)", "C2 x C4 x S3", "C4 x S3", "C4 x S3" ]
gap> List(pgs,CrystCatZClass);
[ [ 4, 20, 14, 1 ], [ 4, 20, 15, 1 ], [ 4, 20, 9, 2 ], [ 4, 20, 9, 1 ] ]
\end{verbatim}
\end{example}

\bigskip

\begin{example}[{Verification of $[M_G]^{fl}\neq 0$ and that $[M_G]^{fl}$ is not invertible 
where $M_G\simeq M_{G_1}\oplus M_{G_2}$ is a decomposable $G$-lattice}]\label{exN}
By using {\tt InverseProjection} in Algorithm KS2, 
we may obtain Table $3$ of Theorems \ref{th1} and \ref{th1M} and 
Tables $11$ to $14$ in Theorems \ref{th2} and \ref{th2M}. 

Let $G\leq \GL(4,\bZ)$ where $M_G\simeq M_{G_1}\oplus M_{G_2}$ is a decomposable 
$G$-lattice of rank $4$ with $G_1\leq\GL(3,\bZ)$ and $G_2\leq \GL(1,\bZ)$. 
Let $\mathcal{N}_3$ be the set of the $15$ GAP IDs as in Table $1$. 
By Lemma \ref{lemp1}, we have $[M_G]^{fl}=[M_1]^{fl}$. 
Hence, by Theorem \ref{thKu} and Theorem \ref{thEM}, 
$L(M_G)^G$ is not retract $k$-rational 
if and only if $L(M_G)^G$ is not stably $k$-rational 
if and only if the GAP ID of $G_1$ is $l\in\mathcal{N}_3$. 

All the GAP IDs $\mathcal{N}_{31}$, as in Table $3$ of Theorem \ref{th1}, 
of such groups $G$ which satisfy 
$M_G\simeq M_{G_1}\oplus M_{G_2}$, $G_1\leq\GL(3,\bZ)$ of the GAP ID 
$l \in\mathcal{N}_3$ and $G_2\leq \GL(1,\bZ)$ may be obtained as follows. 

Similarly, by using the result $\mathcal{I}_4$ and $\mathcal{N}_4$ as 
in Tables $2$ and $4$ of Theorems \ref{th1} and \ref{th1M}, 
we may obtain the CARAT IDs 
$\mathcal{I}_{41}$, $\mathcal{N}_{311}$, $\mathcal{N}_{32}$ and 
$\mathcal{N}_{41}$ as in Tables $11$ to $14$ of 
Theorems \ref{th2} and \ref{th2M}.

\bigskip

\begin{verbatim}
Read("crystcat.gap");
Read("caratnumber.gap");
Read("KS.gap");

gap> ind1:=LatticeDecompositions(1:Carat)[NrPartitions(1)];
[ [ 1, 1, 1 ], [ 1, 2, 1 ] ]
gap> ind2:=LatticeDecompositions(2:Carat)[NrPartitions(2)];
[ [ 2, 3, 2 ], [ 2, 4, 2 ], [ 2, 5, 1 ], [ 2, 6, 1 ], [ 2, 7, 1 ], 
  [ 2, 8, 1 ], [ 2, 9, 1 ], [ 2, 10, 1 ], [ 2, 10, 2 ] ]
gap> ind3:=LatticeDecompositions(3:Carat)[NrPartitions(3)];
[ [ 3, 6, 3 ], [ 3, 6, 4 ], [ 3, 7, 4 ], [ 3, 7, 5 ], [ 3, 8, 3 ], [ 3, 8, 4 ], 
  [ 3, 9, 2 ], [ 3, 10, 2 ], [ 3, 11, 2 ], [ 3, 12, 2 ], [ 3, 13, 2 ], [ 3, 14, 2 ], 
  [ 3, 14, 4 ], [ 3, 15, 2 ], [ 3, 17, 2 ], [ 3, 22, 2 ], [ 3, 25, 2 ], [ 3, 26, 2 ], 
  [ 3, 27, 2 ], [ 3, 28, 1 ], [ 3, 28, 2 ], [ 3, 28, 3 ], [ 3, 29, 1 ], [ 3, 29, 2 ], 
  [ 3, 29, 3 ], [ 3, 30, 1 ], [ 3, 30, 2 ], [ 3, 30, 3 ], [ 3, 31, 1 ], [ 3, 31, 2 ], 
  [ 3, 31, 3 ], [ 3, 32, 1 ], [ 3, 32, 2 ], [ 3, 32, 3 ] ]
gap> ind4:=LatticeDecompositions(4:Carat)[NrPartitions(4)];;
gap> ind5:=LatticeDecompositions(5:Carat)[NrPartitions(5)];;
gap> List([ind1,ind2,ind3,ind4,ind5],Length);
[ 2, 9, 34, 295, 1452 ]

gap> N3:=[ [ 3, 3, 1, 3 ], [ 3, 3, 3, 3 ], [ 3, 3, 3, 4 ], [ 3, 4, 3, 2 ], [ 3, 4, 4, 2 ], 
>  [ 3, 4, 6, 3 ], [ 3, 4, 7, 2 ], [ 3, 7, 1, 2 ], [ 3, 7, 2, 2 ], [ 3, 7, 2, 3 ], 
>  [ 3, 7, 3, 2 ], [ 3, 7, 3, 3 ], [ 3, 7, 4, 2 ], [ 3, 7, 5, 2 ], [ 3, 7, 5, 3 ] ];;
gap> Length(N3);
15
gap> N3c:=List(N3,CrystCat2Carat);
[ [ 3, 6, 3 ], [ 3, 8, 3 ], [ 3, 8, 4 ], [ 3, 15, 2 ], [ 3, 13, 2 ], 
  [ 3, 14, 4 ], [ 3, 9, 2 ], [ 3, 28, 1 ], [ 3, 29, 1 ], [ 3, 29, 3 ], 
  [ 3, 31, 1 ], [ 3, 31, 3 ], [ 3, 30, 1 ], [ 3, 32, 1 ], [ 3, 32, 3 ] ]
gap> N31g:=List(Cartesian([N3c,ind1]),x->InverseProjection(x:Carat));;
gap> N31:=Set(Concatenation(N31g),CrystCatZClass);
[ [ 4, 4, 3, 6 ], [ 4, 4, 4, 4 ], [ 4, 4, 4, 6 ], [ 4, 5, 1, 9 ], [ 4, 5, 2, 4 ], 
  [ 4, 5, 2, 7 ], [ 4, 6, 1, 4 ], [ 4, 6, 1, 8 ], [ 4, 6, 2, 4 ], [ 4, 6, 2, 8 ], 
  [ 4, 6, 2, 9 ], [ 4, 6, 3, 3 ], [ 4, 6, 3, 6 ], [ 4, 7, 3, 2 ], [ 4, 7, 4, 3 ], 
  [ 4, 7, 5, 2 ], [ 4, 7, 7, 2 ], [ 4, 12, 2, 4 ], [ 4, 12, 3, 7 ], [ 4, 12, 4, 6 ], 
  [ 4, 12, 4, 8 ], [ 4, 12, 4, 9 ], [ 4, 12, 5, 6 ], [ 4, 12, 5, 7 ], [ 4, 13, 1, 3 ], 
  [ 4, 13, 2, 4 ], [ 4, 13, 3, 4 ], [ 4, 13, 4, 3 ], [ 4, 13, 5, 3 ], [ 4, 13, 6, 3 ], 
  [ 4, 13, 7, 6 ], [ 4, 13, 7, 7 ], [ 4, 13, 7, 8 ], [ 4, 13, 8, 3 ], [ 4, 13, 8, 4 ], 
  [ 4, 13, 9, 3 ], [ 4, 13, 10, 3 ], [ 4, 24, 1, 5 ], [ 4, 24, 2, 3 ], [ 4, 24, 2, 5 ], 
  [ 4, 24, 3, 5 ], [ 4, 24, 4, 3 ], [ 4, 24, 4, 5 ], [ 4, 24, 5, 3 ], [ 4, 24, 5, 5 ], 
  [ 4, 25, 1, 2 ], [ 4, 25, 1, 4 ], [ 4, 25, 2, 4 ], [ 4, 25, 3, 2 ], [ 4, 25, 3, 4 ], 
  [ 4, 25, 4, 4 ], [ 4, 25, 5, 2 ], [ 4, 25, 5, 4 ], [ 4, 25, 6, 2 ], [ 4, 25, 6, 4 ], 
  [ 4, 25, 7, 2 ], [ 4, 25, 7, 4 ], [ 4, 25, 8, 2 ], [ 4, 25, 8, 4 ], [ 4, 25, 9, 4 ], 
  [ 4, 25, 10, 2 ], [ 4, 25, 10, 4 ], [ 4, 25, 11, 2 ], [ 4, 25, 11, 4 ] ]
gap> Length(N31);
64

gap> I4:=[ [ 4, 31, 1, 3 ], [ 4, 31, 1, 4 ], [ 4, 31, 2, 2 ], [ 4, 31, 4, 2 ], 
>  [ 4, 31, 5, 2 ],  [ 4, 31, 7, 2 ], [ 4, 33, 2, 1 ] ];;
gap> Length(I4);
7
gap> N4:=[ [ 4, 5, 1, 12 ], [ 4, 5, 2, 5 ], [ 4, 5, 2, 8 ], [ 4, 5, 2, 9 ], [ 4, 6, 1, 6 ], 
>  [ 4, 6, 1, 11 ], [ 4, 6, 2, 6 ], [ 4, 6, 2, 10 ], [ 4, 6, 2, 12 ], [ 4, 6, 3, 4 ], 
>  [ 4, 6, 3, 7 ], [ 4, 6, 3, 8 ], [ 4, 12, 2, 5 ], [ 4, 12, 2, 6 ], [ 4, 12, 3, 11 ], 
>  [ 4, 12, 4, 10 ], [ 4, 12, 4, 11 ], [ 4, 12, 4, 12 ], [ 4, 12, 5, 8 ], [ 4, 12, 5, 9 ], 
>  [ 4, 12, 5, 10 ], [ 4, 12, 5, 11 ], [ 4, 13, 1, 5 ], [ 4, 13, 2, 5 ], [ 4, 13, 3, 5 ], 
>  [ 4, 13, 4, 5 ], [ 4, 13, 5, 4 ], [ 4, 13, 5, 5 ], [ 4, 13, 6, 5 ], [ 4, 13, 7, 9 ], 
>  [ 4, 13, 7, 10 ], [ 4, 13, 7, 11 ], [ 4, 13, 8, 5 ], [ 4, 13, 8, 6 ], [ 4, 13, 9, 4 ], 
>  [ 4, 13, 9, 5 ], [ 4, 13, 10, 4 ], [ 4, 13, 10, 5 ], [ 4, 18, 1, 3 ], [ 4, 18, 2, 4 ], 
>  [ 4, 18, 2, 5 ], [ 4, 18, 3, 5 ], [ 4, 18, 3, 6 ], [ 4, 18, 3, 7 ], [ 4, 18, 4, 4 ], 
>  [ 4, 18, 4, 5 ], [ 4, 18, 5, 5 ], [ 4, 18, 5, 6 ], [ 4, 18, 5, 7 ], [ 4, 19, 1, 2 ], 
>  [ 4, 19, 2, 2 ], [ 4, 19, 3, 2 ], [ 4, 19, 4, 3 ], [ 4, 19, 4, 4 ], [ 4, 19, 5, 2 ], 
>  [ 4, 19, 6, 2 ], [ 4, 22, 1, 1 ], [ 4, 22, 2, 1 ], [ 4, 22, 3, 1 ], [ 4, 22, 4, 1 ], 
>  [ 4, 22, 5, 1 ], [ 4, 22, 5, 2 ], [ 4, 22, 6, 1 ], [ 4, 22, 7, 1 ], [ 4, 22, 8, 1 ], 
>  [ 4, 22, 9, 1 ], [ 4, 22, 10, 1 ], [ 4, 22, 11, 1 ], [ 4, 24, 2, 4 ], [ 4, 24, 2, 6 ], 
>  [ 4, 24, 4, 4 ], [ 4, 24, 5, 4 ], [ 4, 24, 5, 6 ], [ 4, 25, 1, 3 ], [ 4, 25, 2, 3 ], 
>  [ 4, 25, 2, 5 ], [ 4, 25, 3, 3 ], [ 4, 25, 4, 3 ], [ 4, 25, 5, 3 ], [ 4, 25, 5, 5 ], 
>  [ 4, 25, 6, 3 ], [ 4, 25, 6, 5 ], [ 4, 25, 7, 3 ], [ 4, 25, 8, 3 ], [ 4, 25, 9, 3 ], 
>  [ 4, 25, 9, 5 ], [ 4, 25, 10, 3 ], [ 4, 25, 10, 5 ], [ 4, 25, 11, 3 ], [ 4, 25, 11, 5 ], 
>  [ 4, 29, 1, 1 ], [ 4, 29, 1, 2 ], [ 4, 29, 2, 1 ], [ 4, 29, 3, 1 ], [ 4, 29, 3, 2 ], 
>  [ 4, 29, 3, 3 ], [ 4, 29, 4, 1 ], [ 4, 29, 4, 2 ], [ 4, 29, 5, 1 ], [ 4, 29, 6, 1 ], 
>  [ 4, 29, 7, 1 ], [ 4, 29, 7, 2 ], [ 4, 29, 8, 1 ], [ 4, 29, 8, 2 ], [ 4, 29, 9, 1 ], 
>  [ 4, 32, 1, 2 ], [ 4, 32, 2, 2 ], [ 4, 32, 3, 2 ], [ 4, 32, 4, 2 ], [ 4, 32, 5, 2 ], 
>  [ 4, 32, 5, 3 ], [ 4, 32, 6, 2 ], [ 4, 32, 7, 2 ], [ 4, 32, 8, 2 ], [ 4, 32, 9, 4 ], 
>  [ 4, 32, 9, 5 ], [ 4, 32, 10, 2 ], [ 4, 32, 11, 2 ], [ 4, 32, 11, 3 ], [ 4, 32, 12, 2 ], 
>  [ 4, 32, 13, 3 ], [ 4, 32, 13, 4 ], [ 4, 32, 14, 3 ], [ 4, 32, 14, 4 ], [ 4, 32, 15, 2 ], 
>  [ 4, 32, 16, 2 ], [ 4, 32, 16, 3 ], [ 4, 32, 17, 2 ], [ 4, 32, 18, 2 ], [ 4, 32, 18, 3 ], 
>  [ 4, 32, 19, 2 ], [ 4, 32, 19, 3 ], [ 4, 32, 20, 2 ], [ 4, 32, 20, 3 ], [ 4, 32, 21, 2 ], 
>  [ 4, 32, 21, 3 ], [ 4, 33, 1, 1 ], [ 4, 33, 3, 1 ], [ 4, 33, 4, 1 ], [ 4, 33, 5, 1 ], 
>  [ 4, 33, 6, 1 ], [ 4, 33, 7, 1 ], [ 4, 33, 8, 1 ], [ 4, 33, 9, 1 ], [ 4, 33, 10, 1 ], 
>  [ 4, 33, 11, 1 ], [ 4, 33, 12, 1 ], [ 4, 33, 13, 1 ], [ 4, 33, 14, 1 ], [ 4, 33, 14, 2 ], 
>  [ 4, 33, 15, 1 ], [ 4, 33, 16, 1 ] ];;
gap> Length(N4);
152

gap> I4c:=List(I4,x->CaratZClass(MatGroupZClass(x[1],x[2],x[3],x[4])));;
gap> I41g:=List(Cartesian([I4c,ind1]),x->InverseProjection(x:Carat));;
gap> I41:=Set(Concatenation(I41g),CaratZClass);;                   
[ [ 5, 692, 1 ], [ 5, 693, 1 ], [ 5, 736, 1 ], [ 5, 911, 1 ], [ 5, 912, 1 ], 
  [ 5, 914, 1 ], [ 5, 916, 1 ], [ 5, 917, 1 ], [ 5, 917, 5 ], [ 5, 918, 1 ], 
  [ 5, 918, 5 ], [ 5, 919, 1 ], [ 5, 921, 1 ], [ 5, 922, 1 ], [ 5, 923, 1 ], 
  [ 5, 924, 1 ], [ 5, 925, 1 ], [ 5, 926, 1 ], [ 5, 926, 3 ], [ 5, 927, 1 ], 
  [ 5, 928, 1 ], [ 5, 928, 3 ], [ 5, 929, 1 ], [ 5, 930, 1 ], [ 5, 932, 1 ] ]
gap> Length(I41);
25
gap> N311g:=List(Cartesian([UnorderedTuples(N3c,1),UnorderedTuples(ind1,2)]),
> x->InverseProjection(Concatenation(x):Carat));;
gap> N311:=Set(Concatenation(N311g),CaratZClass);;
gap> Length(N311);
245
gap> N32g:=List(Cartesian([N3c,ind2]),x->InverseProjection(x:Carat));;
gap> N32:=Set(Concatenation(N32g),CaratZClass);;
gap> Length(N32);
849
gap> N4c:=List(N4,x->CaratZClass(MatGroupZClass(x[1],x[2],x[3],x[4])));;
gap> N41g:=List(Cartesian([N4c,ind1]),x->InverseProjection(x:Carat));;
gap> N41:=Set(Concatenation(N41g),CaratZClass);;
Length(N41);
768
\end{verbatim}
\end{example}

\bigskip

\begin{example}[Verification of the condition (KS2)]\label{exKS2}
By using Algorithm KS2, we may check that the condition (KS2) holds 
for $n\leq 5$. 
For $n=6$, there exist exactly $10$ (resp. $8$) 
$G$-lattices of decomposition type $(4,2)$ (resp. $(5,1)$) 
which are decomposable in two different ways 
$M_G\simeq M_1\oplus M_2\simeq N_1\oplus N_2$ as follows:
\vspace*{2mm} 

\begin{center}
{\small 
\noindent
\begin{tabular}{ll|ll|ll|ll|ll} 
\multicolumn{2}{l|}{$M_G\simeq M_1\oplus M_2\simeq N_1\oplus N_2$} & 
\multicolumn{2}{l|}{$M_1$} & 
\multicolumn{2}{l|}{$M_2$} & 
\multicolumn{2}{l|}{$N_1$} & 
\multicolumn{2}{l}{$N_2$}\\\hline
$(6,2072,14)$ & $C_2\times C_6$ & 
$(2,2,1,2)$ & $C_2$ & $(4,14,4,2)$ & $C_2\times C_6$ & 
$(2,2,2,2)$ & $C_2^2$ & $(4,14,1,2)$ & $C_6$\\
$(6,2076,28)$ & $D_6$ & 
$(2,2,1,2)$ & $C_2$ & $(4,14,9,2)$ & $D_6$ & 
$(2,2,2,2)$ & $C_2^2$ & $(4,14,3,4)$ & $S_3$\\
$(6,2077,28)$ & $D_6$ & 
$(2,2,2,2)$ & $C_2^2$ & $(4,14,5,2)$ & $D_6$ & 
$(2,2,2,2)$ & $C_2^2$ & $(4,14,7,2)$ & $D_6$\\
$(6,2078,28)$ & $D_6$ & 
$(2,2,1,2)$ & $C_2$ & $(4,14,5,2)$ & $D_6$ & 
$(2,2,2,2)$ & $C_2^2$ & $(4,14,6,2)$ & $D_6$\\
$(6,2079,28)$ & $D_6$ & 
$(2,2,1,2)$ & $C_2$ & $(4,14,8,2)$ & $D_6$ & 
$(2,2,1,2)$ & $C_2$ & $(4,14,3,3)$ & $S_3$\\
$(6,2089,21)$ & $D_6$ & 
$(2,2,1,2)$ & $C_2$ & $(4,14,7,2)$ & $D_6$ & 
$(2,2,2,2)$ & $C_2^2$ & $(4,14,6,2)$ & $D_6$\\
$(6,2102,28)$ & $C_2\times D_6$ & 
$(2,2,2,2)$ & $C_2^2$ & $(4,14,5,2)$ & $D_6$ & 
$(2,2,2,2)$ & $C_2^2$ & $(4,14,10,2)$ & $C_2\times D_6$\\
$(6,2107,14)$ & $C_2\times D_6$ & 
$(2,2,2,2)$ & $C_2^2$ & $(4,14,6,2)$ & $D_6$ & 
$(2,2,1,2)$ & $C_2$ & $(4,14,10,2)$ & $C_2\times D_6$\\
$(6,2110,21)$ & $C_2\times D_6$ & 
$(2,2,2,2)$ & $C_2^2$ & $(4,14,7,2)$ & $D_6$ & 
$(2,2,2,2)$ & $C_2^2$ & $(4,14,10,2)$ & $C_2\times D_6$\\
$(6,2295,2)$ & $D_6$ & 
$(2,4,2,2)$ & $S_3$ & $(4,21,3,1)$ & $D_6$ & 
$(2,4,2,1)$ & $S_3$ & $(4,21,3,2)$ & $D_6$
\end{tabular}
}
\end{center}

\bigskip

\begin{center}
{\small
\begin{tabular}{ll|ll|ll|ll|ll} 
\multicolumn{2}{l|}{$M_G\simeq M_1\oplus M_2\simeq N_1\oplus N_2$} & 
\multicolumn{2}{l|}{$M_1$} & 
\multicolumn{2}{l|}{$M_2$} & 
\multicolumn{2}{l|}{$N_1$} & 
\multicolumn{2}{l}{$N_2$}\\\hline
$(6,3045,3)$ & $C_2\times A_5$ & 
$(1,2,1)$ & $C_2$ & $(5,910,3)$ & $C_2\times A_5$ & 
$(1,2,1)$ & $C_2$ & $(5,910,4)$ & $C_2\times A_5$\\
$(6,3046,3)$ & $S_5$ & 
$(1,1,1)$ & $\{1\}$ & $(5,911,3)$ & $S_5$ & 
$(1,1,1)$ & $\{1\}$ & $(5,911,4)$ & $S_5$\\
$(6,3047,3)$ & $S_5$ & 
$(1,2,1)$ & $C_2$ & $(5,912,3)$ & $S_5$ & 
$(1,2,1)$ & $C_2$ & $(5,912,4)$ & $S_5$\\
$(6,3052,5)$ & $F_{20}$ & 
$(1,2,1)$ & $C_2$ & $(5,917,3)$ & $F_{20}$ & 
$(1,2,1)$ & $C_2$ & $(5,917,4)$ & $F_{20}$\\
$(6,3053,5)$ & $F_{20}$ & 
$(1,1,1)$ & $\{1\}$ & $(5,918,3)$ & $F_{20}$ & 
$(1,1,1)$ & $\{1\}$ & $(5,918,4)$ & $F_{20}$\\
$(6,3054,3)$ & $C_2\times S_5$ & 
$(1,2,1)$ & $C_2$ & $(5,919,3)$ & $C_2\times S_5$ & 
$(1,2,1)$ & $C_2$ & $(5,919,4)$ & $C_2\times S_5$\\
$(6,3061,5)$ & $C_2\times F_{20}$ & 
$(1,2,1)$ & $C_2$ & $(5,926,5)$ & $C_2\times F_{20}$ & 
$(1,2,1)$ & $C_2$ & $(5,926,6)$ & $C_2\times F_{20}$\\
$(6,3066,3)$ & $A_5$ & 
$(1,1,1)$ & $\{1\}$ & $(5,931,3)$ & $A_5$ & 
$(1,1,1)$ & $\{1\}$ & $(5,931,4)$ & $A_5$
\end{tabular}
}
\end{center}

\bigskip

In particular, we see that the $3$ 
$G$-lattices $M_G$ for the groups 
$S_5$, $F_{20}$ and $A_5$ of the CARAT IDs 
$(5,911,4)$, $(5,918,4)$ and $(5,931,4)$ respectively 
are not permutation but stably permutation 
because the $G$-lattices $M_G$ for the groups 
$S_5$, $F_{20}$ and $A_5$ of the CARAT IDs 
$(5,911,3)$, $(5,918,3)$ and $(5,931,3)$ are permutation. 
We will study this $3$ cases in 
Section \ref{seFC} again (see Theorem \ref{thfac} (iii), Table $8$ 
and Example \ref{exMI2}). 

\bigskip

\begin{verbatim}
Read("crystcat.gap");
Read("caratnumber.gap");
Read("KS.gap");

gap> ind1:=LatticeDecompositions(1:Carat)[NrPartitions(1)];;
gap> ind2:=LatticeDecompositions(2:Carat)[NrPartitions(2)];;
gap> ind3:=LatticeDecompositions(3:Carat)[NrPartitions(3)];;
gap> ind4:=LatticeDecompositions(4:Carat)[NrPartitions(4)];;
gap> ind5:=LatticeDecompositions(5:Carat)[NrPartitions(5)];;
gap> List([ind1,ind2,ind3,ind4,ind5],Length);
[ 2, 9, 34, 295, 1452 ]

gap> ips21:=List(Cartesian([ind2,ind1]),x->InverseProjection(x:Carat));;
gap> List([Flat(last),Set(List(Flat(last),CaratZClass))],Length);
[ 31, 31 ]

gap> ips211:=List(Cartesian([UnorderedTuples(ind2,1),UnorderedTuples(ind1,2)]),
> x->InverseProjection(Concatenation(x):Carat));;
gap> List([Flat(last),Set(List(Flat(last),CaratZClass))],Length);
[ 96, 96 ]
gap> ips22:=List(UnorderedTuples(ind2,2),x->InverseProjection(x:Carat));;
gap> List([Flat(last),Set(List(Flat(last),CaratZClass))],Length);
[ 175, 175 ]
gap> ips31:=List(Cartesian([ind3,ind1]),x->InverseProjection(x:Carat));;
gap> List([Flat(last),Set(List(Flat(last),CaratZClass))],Length);
[ 128, 128 ]

gap> ips2111:=List(Cartesian([UnorderedTuples(ind2,1),UnorderedTuples(ind1,3)]),
> x->InverseProjection(Concatenation(x):Carat));;
gap> List([Flat(last),Set(List(Flat(last),CaratZClass))],Length);
[ 280, 280 ]
gap> ips221:=List(Cartesian([UnorderedTuples(ind2,2),UnorderedTuples(ind1,1)]),
> x->InverseProjection(Concatenation(x):Carat));;
gap> List([Flat(last),Set(List(Flat(last),CaratZClass))],Length);
[ 1004, 1004 ]
gap> ips311:=List(Cartesian([UnorderedTuples(ind3,1),UnorderedTuples(ind1,2)]),
> x->InverseProjection(Concatenation(x):Carat));;
gap> List([Flat(last),Set(List(Flat(last),CaratZClass))],Length);
[ 442, 442 ]
gap> ips32:=List(Cartesian([ind3,ind2]),x->InverseProjection(x:Carat));;
gap> List([Flat(last),Set(List(Flat(last),CaratZClass))],Length);
[ 1480, 1480 ]
gap> ips41:=List(Cartesian([ind4,ind1]),x->InverseProjection(x:Carat));;
gap> List([Flat(last),Set(List(Flat(last),CaratZClass))],Length);
[ 1400, 1400 ]

gap> ips21111:= List(Cartesian([UnorderedTuples(ind2,1),UnorderedTuples(ind1,4)]),
> x->InverseProjection(Concatenation(x):Carat));;
gap> List([Flat(last),Set(List(Flat(last),CaratZClass))],Length);
[ 824, 824 ]
gap> ips2211:= List(Cartesian([UnorderedTuples(ind2,2),UnorderedTuples(ind1,2)]),
> x->InverseProjection(Concatenation(x):Carat));;
gap> List([Flat(last),Set(List(Flat(last),CaratZClass))],Length);
[ 4862, 4862 ]
gap> ips222:= List(UnorderedTuples(ind2,3),x->InverseProjection(x:Carat));;
gap> List([Flat(last),Set(List(Flat(last),CaratZClass))],Length);
[ 6878, 6878 ]
gap> ips3111:= List(Cartesian([UnorderedTuples(ind3,1),UnorderedTuples(ind1,3)]),
> x->InverseProjection(Concatenation(x):Carat));;
gap> List([Flat(last),Set(List(Flat(last),CaratZClass))],Length);
[ 1466, 1466 ]
gap> ips321:= List(Cartesian([ind3,ind2,ind1]),x->InverseProjection(x:Carat));;
gap> List([Flat(last),Set(List(Flat(last),CaratZClass))],Length);
[ 10662, 10662 ]
gap> ips33:= List(UnorderedTuples(ind3,2),x->InverseProjection(x:Carat));;
gap> List([Flat(last),Set(List(Flat(last),CaratZClass))],Length);
[ 4235, 4235 ]
gap> ips411:= List(Cartesian([UnorderedTuples(ind4,1),UnorderedTuples(ind1,2)]),
> x->InverseProjection(Concatenation(x):Carat));;
gap> List([Flat(last),Set(List(Flat(last),CaratZClass))],Length);
[ 5944, 5944 ]
gap> ips42:= List(Cartesian([ind4,ind2]),x->InverseProjection(x:Carat));;
gap> List([Flat(last),Set(List(Flat(last),CaratZClass))],Length);
[ 21583, 21573 ]
gap> last[1]-last[2];
10
gap> ips51:= List(Cartesian([ind5,ind1]),x->InverseProjection(x:Carat));;
gap> List([Flat(last),Set(List(Flat(last),CaratZClass))],Length);
[ 9939, 9931 ]
gap> last[1]-last[2];
8

gap> Filtered(Collected(List(Flat(ips42),CaratZClass)),x->x[2]>1); # type [4,2]
[ [ [ 6, 2072, 14 ], 2 ], [ [ 6, 2076, 28 ], 2 ], [ [ 6, 2077, 28 ], 2 ], 
  [ [ 6, 2078, 28 ], 2 ], [ [ 6, 2079, 28 ], 2 ], [ [ 6, 2089, 21 ], 2 ], 
  [ [ 6, 2102, 28 ], 2 ], [ [ 6, 2107, 14 ], 2 ], [ [ 6, 2110, 21 ], 2 ], 
  [ [ 6, 2295, 2 ], 2 ] ]
gap> KSfail42:= List(last,x->x[1]);
[ [ 6, 2072, 14 ], [ 6, 2076, 28 ], [ 6, 2077, 28 ], [ 6, 2078, 28 ], [ 6, 2079, 28 ], 
  [ 6, 2089, 21 ], [ 6, 2102, 28 ], [ 6, 2107, 14 ], [ 6, 2110, 21 ], [ 6, 2295, 2 ] ]
gap> Length(KSfail42);
10
gap> KSfail42g:=List(KSfail42,x->CaratMatGroupZClass(x[1],x[2],x[3]));;
gap> List(KSfail42g,Order);
[ 12, 12, 12, 12, 12, 12, 24, 24, 24, 12 ]
gap> List(KSfail42g,StructureDescription);
[ "C6 x C2", "D12", "D12", "D12", "D12", "D12", "C2 x C2 x S3", "C2 x C2 x S3", 
  "C2 x C2 x S3", "D12" ]
gap> [last[7],last[8],last[9],last[10]]; # 7th, 8th, 9th, 10th groups are maximal
[ "C2 x C2 x S3", "C2 x C2 x S3", "C2 x C2 x S3", "D12" ]
gap> KSfail42gsub1:=Set(ConjugacyClassesSubgroups2(KSfail42g[7]),
> x->CaratZClass(Representative(x)));;
gap> KSfail42gsub2:=Set(ConjugacyClassesSubgroups2(KSfail42g[8]),
> x->CaratZClass(Representative(x)));;
gap> KSfail42gsub3:=Set(ConjugacyClassesSubgroups2(KSfail42g[9]),
> x->CaratZClass(Representative(x)));;
gap> List(KSfail42,x->x in KSfail42gsub1); 
[ false, true, true, true, true, false, true, false, false, false ]
gap> List(KSfail42,x->x in KSfail42gsub2);
[ true, true, false, true, false, true, false, true, false, false ]
gap> List(KSfail42,x->x in KSfail42gsub3);
[ true, false, true, false, true, true, false, false, true, false ]
gap> ips42f:=Flat(ips42);;
gap> ips42fc:=List(ips42f,CaratZClass);;
gap> KSfail42i:=List(KSfail42,x->Filtered([1..21583],y->ips42fc[y]=x));
[ [ 9822, 10444 ], [ 10162, 10496 ], [ 9901, 9960 ], [ 9890, 10028 ], [ 10095, 10518 ], 
  [ 9956, 10033 ], [ 9895, 10257 ], [ 10030, 10230 ], [ 9963, 10248 ], [ 10647, 10720 ] ]
gap> KSfail42gg:=List(KSfail42i,x->[ips42f[x[1]],ips42f[x[2]]]);;
gap> List(KSfail42gg,x->List(x,y->List([[1,2],[3..6]],
> z->CrystCatZClass(PartialMatrixGroup(y,z)))));
[ [ [ [ 2, 2, 1, 2 ], [ 4, 14, 4, 2 ] ], [ [ 2, 2, 2, 2 ], [ 4, 14, 1, 2 ] ] ], 
  [ [ [ 2, 2, 1, 2 ], [ 4, 14, 9, 2 ] ], [ [ 2, 2, 2, 2 ], [ 4, 14, 3, 4 ] ] ], 
  [ [ [ 2, 2, 2, 2 ], [ 4, 14, 5, 2 ] ], [ [ 2, 2, 2, 2 ], [ 4, 14, 7, 2 ] ] ], 
  [ [ [ 2, 2, 1, 2 ], [ 4, 14, 5, 2 ] ], [ [ 2, 2, 2, 2 ], [ 4, 14, 6, 2 ] ] ], 
  [ [ [ 2, 2, 1, 2 ], [ 4, 14, 8, 2 ] ], [ [ 2, 2, 1, 2 ], [ 4, 14, 3, 3 ] ] ], 
  [ [ [ 2, 2, 1, 2 ], [ 4, 14, 7, 2 ] ], [ [ 2, 2, 2, 2 ], [ 4, 14, 6, 2 ] ] ], 
  [ [ [ 2, 2, 2, 2 ], [ 4, 14, 5, 2 ] ], [ [ 2, 2, 2, 2 ], [ 4, 14, 10, 2 ] ] ], 
  [ [ [ 2, 2, 2, 2 ], [ 4, 14, 6, 2 ] ], [ [ 2, 2, 1, 2 ], [ 4, 14, 10, 2 ] ] ], 
  [ [ [ 2, 2, 2, 2 ], [ 4, 14, 7, 2 ] ], [ [ 2, 2, 2, 2 ], [ 4, 14, 10, 2 ] ] ], 
  [ [ [ 2, 4, 2, 2 ], [ 4, 21, 3, 1 ] ], [ [ 2, 4, 2, 1 ], [ 4, 21, 3, 2 ] ] ] ]
gap> List(last,y->List([y[1][1],y[1][2],y[2][1],y[2][2]],
> x->StructureDescription(MatGroupZClass(x[1],x[2],x[3],x[4]))));
[ [ "C2", "C6 x C2", "C2 x C2", "C6" ], 
  [ "C2", "D12", "C2 x C2", "S3" ], 
  [ "C2 x C2", "D12", "C2 x C2", "D12" ], 
  [ "C2", "D12", "C2 x C2", "D12" ], 
  [ "C2", "D12", "C2", "S3" ], 
  [ "C2", "D12", "C2 x C2", "D12" ], 
  [ "C2 x C2", "D12", "C2 x C2", "C2 x C2 x S3" ], 
  [ "C2 x C2", "D12", "C2", "C2 x C2 x S3" ], 
  [ "C2 x C2", "D12", "C2 x C2", "C2 x C2 x S3" ], 
  [ "S3", "D12", "S3", "D12" ] ]

gap> Filtered(Collected(List(Flat(ips51),CaratZClass)),x->x[2]>1); # type [5,1]
[ [ [ 6, 3045, 3 ], 2 ], [ [ 6, 3046, 3 ], 2 ], [ [ 6, 3047, 3 ], 2 ],
  [ [ 6, 3052, 5 ], 2 ], [ [ 6, 3053, 5 ], 2 ], [ [ 6, 3054, 3 ], 2 ],
  [ [ 6, 3061, 5 ], 2 ], [ [ 6, 3066, 3 ], 2 ] ]
gap> KSfail51:= List(last,x->x[1]);
[ [ 6, 3045, 3 ], [ 6, 3046, 3 ], [ 6, 3047, 3 ], [ 6, 3052, 5 ],
  [ 6, 3053, 5 ], [ 6, 3054, 3 ], [ 6, 3061, 5 ], [ 6, 3066, 3 ] ]
gap> Length(KSfail51);
8
gap> KSfail51g:=List(KSfail51,x->CaratMatGroupZClass(x[1],x[2],x[3]));;
gap> List(KSfail51g,Order);
[ 120, 120, 120, 20, 20, 240, 40, 60 ]
gap> List(KSfail51g,StructureDescription);
[ "C2 x A5", "S5", "S5", "C5 : C4", "C5 : C4", "C2 x S5", "C2 x (C5 : C4)", "A5" ]
gap> last[6]; # the 6th group of the CARAT ID [ 6, 3054, 3 ] is maximal
"C2 x S5"
gap> KSfail51gsub:=Set(ConjugacyClassesSubgroups2(KSfail51g[6]),
> x->CaratZClass(Representative(x)));;
gap> Difference(KSfail51,KSfail51gsub);
[  ]
gap> ips51f:=Flat(ips51);;
gap> ips51fc:=List(ips51f,CaratZClass);;
gap> KSfail51i:=List(KSfail51,x->Filtered([1..9939],y->ips51fc[y]=x));
[ [ 9630, 9633 ], [ 9635, 9638 ], [ 9642, 9645 ], [ 9648, 9651 ], 
  [ 9653, 9656 ], [ 9660, 9665 ], [ 9670, 9675 ], [ 9679, 9681 ] ]
gap> KSfail51gg:=List(KSfail51i,x->[ips51f[x[1]],ips51f[x[2]]]);;
gap> List(KSfail51gg,x->List(x,y->List([[1],[2..6]],
> z->CaratZClass(PartialMatrixGroup(y,z)))));
[ [ [ [ 1, 2, 1 ], [ 5, 910, 3 ] ], [ [ 1, 2, 1 ], [ 5, 910, 4 ] ] ], 
  [ [ [ 1, 1, 1 ], [ 5, 911, 3 ] ], [ [ 1, 1, 1 ], [ 5, 911, 4 ] ] ], 
  [ [ [ 1, 2, 1 ], [ 5, 912, 3 ] ], [ [ 1, 2, 1 ], [ 5, 912, 4 ] ] ], 
  [ [ [ 1, 2, 1 ], [ 5, 917, 3 ] ], [ [ 1, 2, 1 ], [ 5, 917, 4 ] ] ], 
  [ [ [ 1, 1, 1 ], [ 5, 918, 3 ] ], [ [ 1, 1, 1 ], [ 5, 918, 4 ] ] ], 
  [ [ [ 1, 2, 1 ], [ 5, 919, 3 ] ], [ [ 1, 2, 1 ], [ 5, 919, 4 ] ] ], 
  [ [ [ 1, 2, 1 ], [ 5, 926, 5 ] ], [ [ 1, 2, 1 ], [ 5, 926, 6 ] ] ], 
  [ [ [ 1, 1, 1 ], [ 5, 931, 3 ] ], [ [ 1, 1, 1 ], [ 5, 931, 4 ] ] ] ]
gap> List(last,y->List([y[1][1],y[1][2],y[2][1],y[2][2]],
> x->StructureDescription(CaratMatGroupZClass(x[1],x[2],x[3]))));
[ [ "C2", "C2 x A5", "C2", "C2 x A5" ], 
  [ "1", "S5", "1", "S5" ], 
  [ "C2", "S5", "C2", "S5" ],
  [ "C2", "C5 : C4", "C2", "C5 : C4" ], 
  [ "1", "C5 : C4", "1", "C5 : C4" ], 
  [ "C2", "C2 x S5", "C2", "C2 x S5" ], 
  [ "C2", "C2 x (C5 : C4)", "C2", "C2 x (C5 : C4)" ], 
  [ "1", "A5", "1", "A5" ] ]
\end{verbatim}
\end{example}

%

\bigskip

\subsection{Maximal finite groups $G\leq \GL(n,\bZ)$ of dimension $n\leq 6$}\label{ssMax}

In this subsection, we consider not only maximal irreducible (resp. indecomposable) 
groups $G\leq \GL(n,\bZ)$ but also reducible (resp. decomposable) ones, 
i.e. we consider all the maximal groups $G\leq \GL(n,\bZ)$.

Let $L$ be a set of $\bZ$-classes of finite groups $G$ in $\GL(n,\bZ)$. 
The following algorithm enable us to determine which $\bZ$-classes 
are maximal ones in $L$.

\bigskip

\noindent
{\tt MaximalGroupsID(L)}
returns the list of the GAP IDs of the maximal $\bZ$-classes 
in the groups of the GAP IDs $L$. 

\noindent
{\tt MaximalGroupsID(L:Carat)} returns the same as
{\tt MaximalGroupsID(L)} but using the CARAT ID instead of the GAP ID.

\noindent
{\tt MaximalGroupsID(L:FromPerm)}
(resp. {\tt MaximalGroupsID(L:Carat,FromPerm)}) returns the same as\\
{\tt MaximalGroupsID(L)} (resp. {\tt MaximalGroupsID(L:Carat)})
but using 
{\tt ConjugacyClassesSubgroupsFromPerm(G)} instead of 
{\tt ConjugacyClassesSubgroups(G)}.

\bigskip

\begin{verbatim}
MaximalGroupsID:= function(L)
    local G,O,m,m0,r,ri,ro,i,j,id,sg;
    if ValueOption("carat")=true or ValueOption("Carat")=true then
        G:=List(L,x->CaratMatGroupZClass(x[1],x[2],x[3]));
    else
        G:=List(L,x->MatGroupZClass(x[1],x[2],x[3],x[4]));
    fi;
    O:=List(G,Order);
    r:=L;
    ri:=List(r,x->Position(L,x));
    ro:=Set(ri,x->O[x]);
    m:=[];
    m0:=Filtered(ri,x->Number(ro,y->y mod O[x]=0)=1);
    while m0<>[] do
        for i in m0 do
            if ValueOption("fromperm")=true or ValueOption("FromPerm")=true then
                sg:=ConjugacyClassesSubgroupsFromPerm(G[i]);
            else
                sg:=List(ConjugacyClassesSubgroups(G[i]),Representative);
            fi;
            for j in sg do
                if Order(j) in List(ri,x->O[x]) then
                    if ValueOption("carat")=true or ValueOption("Carat")=true then
                        id:=CaratZClass(j);
                    else
                        id:=CrystCatZClass(j);
                    fi;
                    r:=Difference(r,[id]);
                fi;
            od;
        od;
        ri:=List(r,x->Position(L,x));
        ro:=Set(ri,x->O[x]);
        m:=Concatenation(m,m0);
        m0:=Filtered(ri,x->Number(ro,y->y mod O[x]=0)=1);
    od;
    return List(SortedList(m),x->L[x]);
end;
\end{verbatim}

\bigskip

\begin{example}[Maximal finite groups $G\leq \GL(n,\bZ)$ of dimension $n\leq 6$]
Using {\tt MaximalGroupsID(L)} for all $\bZ$-classes $L$ of finite subgroups 
$G\leq \GL(n,\bZ)$, we can determine all the maximal finite groups 
$G\leq \GL(n,\bZ)$ of dimension $n\leq 6$. 
There exist $2$ (resp. $4$, $9$, $17$, $39$) maximal groups $G\leq \GL(n,\bZ)$ 
($\bZ$-classes) of dimension $n=2$ (resp. $3$, $4$, $5$, $6$) 
(see \cite{Dad65} for $n=4$, 
\cite{Rys72a}, \cite{Rys72b}, \cite{RL80} for $n=5$, 
\cite[Section 1.10]{Lor05} for $n=6$).

\bigskip

\begin{verbatim}
gap> Read("caratnumber.gap");
gap> Read("KS.gap");

gap> GL2Zcr:=Concatenation(List([1..NrCrystalSystems(2)],
> x->Concatenation(List([1..NrQClassesCrystalSystem(2,x)],
> y->List([1..NrZClassesQClass(2,x,y)],z->[2,x,y,z])))));;
gap> GL3Zcr:=Concatenation(List([1..NrCrystalSystems(3)],
> x->Concatenation(List([1..NrQClassesCrystalSystem(3,x)],
> y->List([1..NrZClassesQClass(3,x,y)],z->[3,x,y,z])))));;
gap> GL4Zcr:=Concatenation(List([1..NrCrystalSystems(4)],
> x->Concatenation(List([1..NrQClassesCrystalSystem(4,x)],
> y->List([1..NrZClassesQClass(4,x,y)],z->[4,x,y,z])))));;
gap> GL5Zca:=Concatenation(List([1..CaratNrQClasses(5)],
> x->List([1..CaratNrZClasses(5,x)],y->[5,x,y])));;
gap> GL6Zca:=Concatenation(List([1..CaratNrQClasses(6)],
> x->List([1..CaratNrZClasses(6,x)],y->[6,x,y])));;
gap> MaximalGroupsID(GL2Zcr);
[ [ 2, 3, 2, 1 ], [ 2, 4, 4, 1 ] ]
gap> Length(last); # there exist 2 maximal groups in dimension 2
2
gap> MaximalGroupsID(GL3Zcr);
[ [ 3, 6, 7, 1 ], [ 3, 7, 5, 1 ], [ 3, 7, 5, 2 ], [ 3, 7, 5, 3 ] ]
gap> Length(last); # there exist 4 maximal groups in dimension 3
4
gap> MaximalGroupsID(GL4Zcr);
[ [ 4, 20, 22, 1 ], [ 4, 25, 11, 2 ], [ 4, 25, 11, 4 ], [ 4, 29, 9, 1 ], 
  [ 4, 30, 13, 1 ], [ 4, 31, 7, 1 ], [ 4, 31, 7, 2 ], [ 4, 32, 21, 1 ], 
  [ 4, 33, 16, 1 ] ]
gap> Length(last); # there exist 9 maximal groups in dimension 4
9
gap> MaximalGroupsID(GL5Zca:Carat);
[ [ 5, 559, 3 ], [ 5, 559, 4 ], [ 5, 626, 1 ], [ 5, 626, 2 ], [ 5, 626, 3 ], 
  [ 5, 690, 1 ], [ 5, 836, 2 ], [ 5, 866, 1 ], [ 5, 930, 1 ], [ 5, 930, 2 ], 
  [ 5, 942, 1 ], [ 5, 942, 2 ], [ 5, 942, 3 ], [ 5, 949, 1 ], [ 5, 949, 2 ], 
  [ 5, 949, 3 ], [ 5, 949, 4 ] ]
gap> Length(last); # there exist 17 maximal groups in dimension 5
17
gap> MaximalGroupsID(GL6Zca:Carat,FromPerm);
[ [ 6, 2750, 4 ], [ 6, 2772, 2 ], [ 6, 2772, 5 ], [ 6, 2773, 1 ], [ 6, 2773, 2 ], 
  [ 6, 2773, 3 ], [ 6, 2803, 1 ], [ 6, 2804, 1 ], [ 6, 2804, 2 ], [ 6, 2866, 2 ], 
  [ 6, 2866, 3 ], [ 6, 2932, 1 ], [ 6, 2932, 2 ], [ 6, 2945, 1 ], [ 6, 2952, 1 ], 
  [ 6, 2952, 2 ], [ 6, 2952, 3 ], [ 6, 3120, 1 ], [ 6, 3120, 2 ], [ 6, 3193, 1 ], 
  [ 6, 3193, 2 ], [ 6, 3272, 3 ], [ 6, 3272, 5 ], [ 6, 3296, 1 ], [ 6, 3296, 2 ], 
  [ 6, 3296, 4 ], [ 6, 3296, 5 ], [ 6, 3795, 3 ], [ 6, 3891, 1 ], [ 6, 4232, 2 ], 
  [ 6, 5209, 1 ], [ 6, 5209, 5 ], [ 6, 5517, 3 ], [ 6, 5517, 4 ], [ 6, 5517, 5 ], 
  [ 6, 5517, 8 ], [ 6, 6509, 1 ], [ 6, 6568, 1 ], [ 6, 6878, 2 ] ]
gap> Length(last); # there exist 39 maximal groups in dimension 6
39
\end{verbatim}
\end{example}

\bigskip

\subsection{Bravais groups of dimension $n\leq 6$ and corresponding quadratic forms}
\label{ssBravais}

\bigskip

\begin{definition}[Bravais group]
Let $G\leq \GL(n,\bZ)$ be finite subgroup and 
$\bR_{\rm sym}^{n\times n}$ be the set of $n\times n$ symmetric matrices 
(which correspond to quadratic forms) 
whose entries are in $\bR$.\\
{\rm (i)} $\mathcal{F}(G)=\{F\in\bR_{\rm sym}^{n\times n}\mid g\,F\, {}^tg=F\ 
({\rm for\ any}\ g\in G)\}$ is called {\it the space of invariant forms} 
%
%
of $G$;\\
{\rm (ii)} For $\mathcal{F}\leq \bR_{\rm sym}^{n\times n}$, 
$\mathcal{B}(\mathcal{F})=\{g\in \GL(n,\bZ)\mid g\, F\, {}^tg=F\ 
({\rm for\ any}\ F\in \mathcal{F})\}$ 
is called {\it the Bravais group of $\mathcal{F}$};\\ 
{\rm (iii)} $B(G):=\mathcal{B}(\mathcal{F}(G))$ is 
called {\it the Bravais group of $G$};\\
{\rm (iv)} $G\leq \GL(n,\bZ)$ is simply called a {\it Bravais group} 
of dimension $n$ if $B(G)=G$. 
\end{definition}

Note that if $\mathcal{F}$ contains a positive definite matrix, 
then $\mathcal{B}(\mathcal{F})$ becomes a finite set. 
By definitions, 
$\mathcal{F}(G)$ contains a positive definite matrix, 
and hence $G\leq B(G)\leq \GL(n,\bZ)$ is finite. 
In particular, all the maximal finite groups $G\leq \GL(n,\bZ)$ (as in the last subsection) 
are Bravais groups. 

\begin{example}[Bravais groups $G\leq \GL(n,\bZ)$ where $n\leq 6$]
By using  CARAT (\cite{Carat}), 
e.g. the build-in function {\tt BravaisGroupsCrystalFamily}, 
we may obtain the GAP ID and the Carat ID of 
all the $\bZ$-classes of Bravais groups $G\leq\GL(n,\bZ)$ for $n\leq 6$ as follows. 
There exists $1$ (resp. $5$, $14$, $64$, $189$, $841$)\footnote{In \cite{PH84}, 
the number, $826$, of Bravais groups of dimension $6$ was not correct. 
This is modified by \cite{OPS98} with the correct number $841$.} Bravais group of 
dimension $n=1$ (resp. $2$, $3$, $4$, $5$, $6$) 
(see \cite[Section 1.3]{BBNWZ78}, \cite{PP77}, \cite{Ple81}, \cite{PH84}, \cite{OPS98}). 

\bigskip

\begin{verbatim}
gap> Read("caratnumber.gap");

gap> b1:=Flat(List(CaratCrystalFamilies[1],BravaisGroupsCrystalFamily));;
gap> b2:=Flat(List(CaratCrystalFamilies[2],BravaisGroupsCrystalFamily));;
gap> b3:=Flat(List(CaratCrystalFamilies[3],BravaisGroupsCrystalFamily));;
gap> b4:=Flat(List(CaratCrystalFamilies[4],BravaisGroupsCrystalFamily));;
gap> b5:=Flat(List(CaratCrystalFamilies[5],BravaisGroupsCrystalFamily));;
gap> b6:=Flat(List(CaratCrystalFamilies[6],BravaisGroupsCrystalFamily));;
gap> List([b1,b2,b3,b4,b5,b6],Length);
[ 1, 5, 14, 64, 189, 841 ]

gap> b1ca:=List(b1,CaratZClass); 
[ [ 1, 2, 1 ] ]  
gap> Length(last); # there exists 1 Bravais group of dimension 1
1
gap> b2cr:=List(b2,CrystCatZClass);
[ [ 2, 1, 2, 1 ], [ 2, 2, 2, 1 ], [ 2, 2, 2, 2 ], [ 2, 3, 2, 1 ], [ 2, 4, 4, 1 ] ]
gap> Length(last); # there exist 5 Bravais groups of dimension 2
5
gap> b3cr:=List(b3,CrystCatZClass);
[ [ 3, 1, 2, 1 ], [ 3, 2, 3, 1 ], [ 3, 2, 3, 2 ], [ 3, 3, 3, 1 ], [ 3, 3, 3, 2 ], 
  [ 3, 3, 3, 3 ], [ 3, 3, 3, 4 ], [ 3, 4, 7, 1 ], [ 3, 4, 7, 2 ], [ 3, 6, 7, 1 ], 
  [ 3, 5, 5, 1 ], [ 3, 7, 5, 3 ], [ 3, 7, 5, 1 ], [ 3, 7, 5, 2 ] ]
gap> Length(last); # there exist 14 Bravais groups of dimension 3
14
gap> b4cr:=List(b4,CrystCatZClass);
[ [ 4, 1, 2, 1 ], [ 4, 2, 3, 1 ], [ 4, 2, 3, 2 ], [ 4, 3, 2, 1 ], [ 4, 3, 2, 2 ], 
  [ 4, 3, 2, 3 ], [ 4, 4, 4, 1 ], [ 4, 4, 4, 3 ], [ 4, 4, 4, 2 ], [ 4, 4, 4, 6 ], 
  [ 4, 4, 4, 5 ], [ 4, 4, 4, 4 ], [ 4, 6, 3, 1 ], [ 4, 6, 3, 2 ], [ 4, 6, 3, 6 ], 
  [ 4, 6, 3, 8 ], [ 4, 6, 3, 5 ], [ 4, 6, 3, 3 ], [ 4, 6, 3, 7 ], [ 4, 6, 3, 4 ], 
  [ 4, 5, 2, 1 ], [ 4, 10, 1, 1 ], [ 4, 16, 1, 1 ], [ 4, 16, 1, 2 ], [ 4, 16, 1, 3 ], 
  [ 4, 7, 7, 1 ], [ 4, 7, 7, 2 ], [ 4, 13, 10, 1 ], [ 4, 13, 10, 3 ], [ 4, 13, 10, 2 ], 
  [ 4, 13, 10, 5 ], [ 4, 13, 10, 4 ], [ 4, 12, 5, 1 ], [ 4, 19, 6, 1 ], [ 4, 19, 6, 2 ], 
  [ 4, 18, 5, 1 ], [ 4, 20, 22, 1 ], [ 4, 11, 2, 1 ], [ 4, 17, 2, 1 ], [ 4, 17, 2, 2 ], 
  [ 4, 9, 7, 1 ], [ 4, 8, 5, 1 ], [ 4, 15, 12, 1 ], [ 4, 15, 12, 2 ], [ 4, 14, 10, 1 ], 
  [ 4, 14, 10, 2 ], [ 4, 23, 11, 1 ], [ 4, 22, 11, 1 ], [ 4, 21, 4, 1 ], [ 4, 25, 11, 2 ], 
  [ 4, 25, 11, 3 ], [ 4, 24, 5, 1 ], [ 4, 25, 11, 1 ], [ 4, 25, 11, 5 ], [ 4, 25, 11, 4 ], 
  [ 4, 32, 21, 1 ], [ 4, 33, 16, 1 ], [ 4, 26, 2, 1 ], [ 4, 30, 13, 1 ], [ 4, 29, 9, 1 ], 
  [ 4, 28, 2, 1 ], [ 4, 31, 7, 1 ], [ 4, 31, 7, 2 ], [ 4, 27, 4, 1 ] ]
gap> Length(last); # there exist 64 Bravais groups of dimension 4
64
gap> b5ca:=List(b5,CaratZClass);
[ [ 5, 2, 1 ], [ 5, 5, 1 ], [ 5, 5, 2 ], [ 5, 8, 1 ], [ 5, 8, 2 ], 
  [ 5, 8, 3 ], [ 5, 12, 1 ], [ 5, 12, 2 ], [ 5, 12, 3 ], [ 5, 12, 4 ], 
  [ 5, 12, 5 ], [ 5, 12, 6 ], [ 5, 16, 1 ], [ 5, 16, 2 ], [ 5, 16, 4 ], 
  [ 5, 16, 3 ], [ 5, 16, 5 ], [ 5, 16, 6 ], [ 5, 16, 8 ], [ 5, 16, 7 ], 
  [ 5, 16, 9 ], [ 5, 17, 1 ], [ 5, 17, 2 ], [ 5, 17, 3 ], [ 5, 17, 4 ], 
  [ 5, 17, 5 ], [ 5, 17, 6 ], [ 5, 17, 7 ], [ 5, 17, 10 ], [ 5, 17, 8 ], 
  [ 5, 17, 11 ], [ 5, 17, 9 ], [ 5, 17, 12 ], [ 5, 17, 13 ], [ 5, 17, 14 ], 
  [ 5, 17, 17 ], [ 5, 17, 16 ], [ 5, 17, 15 ], [ 5, 20, 18 ], [ 5, 29, 1 ], 
  [ 5, 29, 2 ], [ 5, 29, 3 ], [ 5, 29, 4 ], [ 5, 29, 5 ], [ 5, 29, 6 ], 
  [ 5, 29, 8 ], [ 5, 29, 10 ], [ 5, 29, 7 ], [ 5, 29, 9 ], [ 5, 29, 11 ], 
  [ 5, 29, 12 ], [ 5, 29, 14 ], [ 5, 29, 13 ], [ 5, 29, 15 ], [ 5, 29, 16 ], 
  [ 5, 25, 31 ], [ 5, 25, 32 ], [ 5, 25, 33 ], [ 5, 35, 1 ], [ 5, 35, 2 ], 
  [ 5, 36, 4 ], [ 5, 36, 5 ], [ 5, 36, 1 ], [ 5, 36, 2 ], [ 5, 36, 3 ], 
  [ 5, 36, 6 ], [ 5, 36, 7 ], [ 5, 40, 1 ], [ 5, 40, 2 ], [ 5, 56, 1 ], 
  [ 5, 56, 2 ], [ 5, 56, 4 ], [ 5, 56, 5 ], [ 5, 56, 3 ], [ 5, 56, 6 ], 
  [ 5, 56, 7 ], [ 5, 48, 15 ], [ 5, 97, 1 ], [ 5, 97, 2 ], [ 5, 97, 3 ], 
  [ 5, 97, 4 ], [ 5, 97, 5 ], [ 5, 97, 6 ], [ 5, 97, 7 ], [ 5, 97, 9 ], 
  [ 5, 97, 8 ], [ 5, 97, 10 ], [ 5, 97, 11 ], [ 5, 97, 12 ], [ 5, 88, 39 ], 
  [ 5, 88, 40 ], [ 5, 88, 41 ], [ 5, 104, 1 ], [ 5, 104, 2 ], [ 5, 104, 4 ], 
  [ 5, 104, 3 ], [ 5, 104, 5 ], [ 5, 151, 17 ], [ 5, 151, 18 ], [ 5, 151, 19 ], 
  [ 5, 273, 1 ], [ 5, 273, 2 ], [ 5, 372, 5 ], [ 5, 372, 6 ], [ 5, 163, 1 ], 
  [ 5, 167, 2 ], [ 5, 172, 2 ], [ 5, 171, 2 ], [ 5, 172, 1 ], [ 5, 171, 4 ], 
  [ 5, 180, 1 ], [ 5, 176, 3 ], [ 5, 215, 1 ], [ 5, 215, 2 ], [ 5, 207, 5 ], 
  [ 5, 207, 6 ], [ 5, 205, 5 ], [ 5, 205, 6 ], [ 5, 266, 1 ], [ 5, 266, 2 ], 
  [ 5, 266, 3 ], [ 5, 257, 10 ], [ 5, 266, 4 ], [ 5, 257, 12 ], [ 5, 257, 14 ], 
  [ 5, 257, 13 ], [ 5, 257, 15 ], [ 5, 421, 1 ], [ 5, 394, 2 ], [ 5, 465, 6 ], 
  [ 5, 395, 2 ], [ 5, 465, 7 ], [ 5, 460, 2 ], [ 5, 416, 4 ], [ 5, 517, 4 ], 
  [ 5, 517, 5 ], [ 5, 511, 6 ], [ 5, 517, 2 ], [ 5, 517, 3 ], [ 5, 517, 1 ], 
  [ 5, 518, 8 ], [ 5, 518, 9 ], [ 5, 518, 10 ], [ 5, 518, 11 ], [ 5, 540, 15 ], 
  [ 5, 540, 16 ], [ 5, 518, 12 ], [ 5, 540, 17 ], [ 5, 518, 3 ], [ 5, 518, 4 ], 
  [ 5, 518, 5 ], [ 5, 518, 6 ], [ 5, 518, 7 ], [ 5, 518, 1 ], [ 5, 518, 2 ], 
  [ 5, 559, 4 ], [ 5, 559, 5 ], [ 5, 559, 1 ], [ 5, 559, 2 ], [ 5, 559, 3 ], 
  [ 5, 626, 3 ], [ 5, 626, 2 ], [ 5, 638, 4 ], [ 5, 626, 1 ], [ 5, 638, 2 ], 
  [ 5, 643, 1 ], [ 5, 643, 2 ], [ 5, 772, 4 ], [ 5, 772, 5 ], [ 5, 690, 1 ], 
  [ 5, 772, 3 ], [ 5, 793, 1 ], [ 5, 866, 1 ], [ 5, 801, 3 ], [ 5, 836, 2 ], 
  [ 5, 801, 6 ], [ 5, 908, 1 ], [ 5, 904, 3 ], [ 5, 930, 2 ], [ 5, 919, 3 ], 
  [ 5, 919, 4 ], [ 5, 930, 1 ], [ 5, 942, 1 ], [ 5, 942, 3 ], [ 5, 942, 2 ], 
  [ 5, 949, 4 ], [ 5, 949, 1 ], [ 5, 949, 3 ], [ 5, 949, 2 ] ]
gap> Length(last); # there exist 189 Bravais groups of dimension 5
189
gap> b6ca:=List(b6,CaratZClass);
[ [ 6, 2710, 1 ], [ 6, 3, 1 ], [ 6, 3, 2 ], [ 6, 6, 1 ], [ 6, 6, 2 ], 
  [ 6, 6, 3 ], [ 6, 10, 1 ], [ 6, 10, 2 ], [ 6, 10, 3 ], [ 6, 10, 4 ], 
  [ 6, 10, 6 ], [ 6, 10, 5 ], [ 6, 12, 1 ], [ 6, 12, 2 ], [ 6, 12, 3 ], 
  [ 6, 12, 4 ], [ 6, 17, 1 ], [ 6, 17, 2 ], [ 6, 17, 3 ], [ 6, 17, 4 ], 
  [ 6, 17, 5 ], [ 6, 17, 6 ], [ 6, 17, 7 ], [ 6, 17, 8 ], [ 6, 17, 9 ], 
  [ 6, 17, 10 ], [ 6, 17, 11 ], [ 6, 17, 12 ], [ 6, 4617, 1 ], [ 6, 4617, 2 ], 
  [ 6, 4617, 4 ], [ 6, 4617, 6 ], [ 6, 4617, 3 ], [ 6, 4617, 5 ], [ 6, 4617, 8 ], 
  [ 6, 4617, 11 ], [ 6, 4617, 14 ], [ 6, 4617, 9 ], [ 6, 4617, 7 ], [ 6, 4617, 13 ], 
  [ 6, 4617, 10 ], [ 6, 4617, 12 ], [ 6, 4617, 18 ], [ 6, 4617, 16 ], [ 6, 4617, 17 ], 
  [ 6, 4617, 15 ], [ 6, 4620, 19 ], [ 6, 4627, 1 ], [ 6, 4627, 2 ], [ 6, 4627, 3 ], 
  [ 6, 4627, 6 ], [ 6, 4627, 4 ], [ 6, 4627, 7 ], [ 6, 4627, 8 ], [ 6, 4627, 5 ], 
  [ 6, 4627, 9 ], [ 6, 4627, 11 ], [ 6, 4627, 10 ], [ 6, 4627, 12 ], [ 6, 4628, 1 ], 
  [ 6, 4628, 2 ], [ 6, 4628, 6 ], [ 6, 4628, 3 ], [ 6, 4628, 4 ], [ 6, 4628, 7 ], 
  [ 6, 4628, 5 ], [ 6, 4628, 12 ], [ 6, 4628, 9 ], [ 6, 4628, 14 ], [ 6, 4628, 8 ], 
  [ 6, 4628, 13 ], [ 6, 4628, 19 ], [ 6, 4628, 23 ], [ 6, 4628, 16 ], [ 6, 4628, 11 ], 
  [ 6, 4628, 21 ], [ 6, 4628, 10 ], [ 6, 4628, 15 ], [ 6, 4628, 18 ], [ 6, 4628, 22 ], 
  [ 6, 4628, 17 ], [ 6, 4628, 20 ], [ 6, 4628, 26 ], [ 6, 4628, 30 ], [ 6, 4628, 25 ], 
  [ 6, 4628, 27 ], [ 6, 4628, 31 ], [ 6, 4628, 29 ], [ 6, 4628, 33 ], [ 6, 4628, 35 ], 
  [ 6, 4628, 24 ], [ 6, 4628, 28 ], [ 6, 4628, 32 ], [ 6, 4628, 34 ], [ 6, 4631, 36 ], 
  [ 6, 4628, 37 ], [ 6, 4628, 36 ], [ 6, 4631, 39 ], [ 6, 4643, 1 ], [ 6, 4643, 2 ], 
  [ 6, 4643, 3 ], [ 6, 4643, 4 ], [ 6, 4643, 5 ], [ 6, 4643, 6 ], [ 6, 4643, 7 ], 
  [ 6, 4643, 8 ], [ 6, 4643, 11 ], [ 6, 4643, 14 ], [ 6, 4643, 10 ], [ 6, 4643, 13 ], 
  [ 6, 4643, 16 ], [ 6, 4643, 12 ], [ 6, 4643, 15 ], [ 6, 4643, 20 ], [ 6, 4643, 18 ], 
  [ 6, 4643, 22 ], [ 6, 4643, 24 ], [ 6, 4643, 25 ], [ 6, 4643, 26 ], [ 6, 4643, 23 ], 
  [ 6, 4643, 27 ], [ 6, 4643, 9 ], [ 6, 4643, 19 ], [ 6, 4643, 17 ], [ 6, 4643, 21 ], 
  [ 6, 4643, 34 ], [ 6, 4643, 39 ], [ 6, 4643, 35 ], [ 6, 4643, 36 ], [ 6, 4643, 40 ], 
  [ 6, 4643, 44 ], [ 6, 4643, 45 ], [ 6, 4643, 29 ], [ 6, 4643, 33 ], [ 6, 4643, 32 ], 
  [ 6, 4643, 37 ], [ 6, 4643, 38 ], [ 6, 4643, 30 ], [ 6, 4643, 43 ], [ 6, 4643, 31 ], 
  [ 6, 4643, 42 ], [ 6, 4643, 28 ], [ 6, 4643, 41 ], [ 6, 4637, 98 ], [ 6, 4637, 99 ], 
  [ 6, 4636, 50 ], [ 6, 4636, 51 ], [ 6, 4643, 48 ], [ 6, 4643, 47 ], [ 6, 4643, 49 ], 
  [ 6, 4643, 46 ], [ 6, 4637, 101 ], [ 6, 4637, 100 ], [ 6, 4636, 52 ], [ 6, 4636, 53 ], 
  [ 6, 4663, 1 ], [ 6, 4663, 2 ], [ 6, 4663, 3 ], [ 6, 4663, 4 ], [ 6, 4663, 5 ], 
  [ 6, 4663, 6 ], [ 6, 4663, 8 ], [ 6, 4663, 10 ], [ 6, 4663, 12 ], [ 6, 4663, 9 ], 
  [ 6, 4663, 11 ], [ 6, 4663, 13 ], [ 6, 4663, 15 ], [ 6, 4663, 14 ], [ 6, 4663, 16 ], 
  [ 6, 4663, 17 ], [ 6, 4663, 7 ], [ 6, 4663, 24 ], [ 6, 4663, 25 ], [ 6, 4663, 26 ], 
  [ 6, 4663, 28 ], [ 6, 4663, 27 ], [ 6, 4663, 29 ], [ 6, 4663, 19 ], [ 6, 4663, 23 ], 
  [ 6, 4663, 21 ], [ 6, 4663, 20 ], [ 6, 4663, 22 ], [ 6, 4663, 18 ], [ 6, 4659, 107 ], 
  [ 6, 4659, 109 ], [ 6, 4659, 112 ], [ 6, 4663, 33 ], [ 6, 4663, 34 ], [ 6, 4663, 31 ], 
  [ 6, 4663, 35 ], [ 6, 4663, 32 ], [ 6, 4663, 30 ], [ 6, 4659, 127 ], [ 6, 4659, 131 ], 
  [ 6, 4659, 128 ], [ 6, 4659, 132 ], [ 6, 4657, 107 ], [ 6, 4659, 129 ], [ 6, 4659, 133 ], 
  [ 6, 4663, 36 ], [ 6, 4658, 37 ], [ 6, 4659, 134 ], [ 6, 4657, 109 ], [ 6, 4666, 1 ], 
  [ 6, 4669, 1 ], [ 6, 4669, 2 ], [ 6, 4669, 3 ], [ 6, 4670, 1 ], [ 6, 4670, 2 ], 
  [ 6, 4670, 3 ], [ 6, 4670, 4 ], [ 6, 4670, 6 ], [ 6, 4670, 5 ], [ 6, 4672, 7 ], 
  [ 6, 4678, 1 ], [ 6, 4678, 2 ], [ 6, 4677, 3 ], [ 6, 25, 1 ], [ 6, 4680, 2 ], 
  [ 6, 4680, 1 ], [ 6, 4680, 3 ], [ 6, 4681, 1 ], [ 6, 4681, 2 ], [ 6, 4681, 3 ], 
  [ 6, 4681, 4 ], [ 6, 4681, 5 ], [ 6, 4681, 6 ], [ 6, 4681, 7 ], [ 6, 4681, 8 ], 
  [ 6, 4681, 9 ], [ 6, 4690, 1 ], [ 6, 4690, 2 ], [ 6, 4690, 3 ], [ 6, 4690, 4 ], 
  [ 6, 4690, 6 ], [ 6, 4690, 5 ], [ 6, 4689, 34 ], [ 6, 4689, 35 ], [ 6, 4690, 7 ], 
  [ 6, 4690, 8 ], [ 6, 4690, 9 ], [ 6, 4690, 10 ], [ 6, 4690, 11 ], [ 6, 4690, 12 ],
  [ 6, 4690, 14 ], [ 6, 4690, 13 ], [ 6, 4690, 15 ], [ 6, 4690, 16 ], [ 6, 4689, 21 ], 
  [ 6, 4690, 17 ], [ 6, 4690, 18 ], [ 6, 4690, 19 ], [ 6, 4690, 20 ], [ 6, 4690, 21 ], 
  [ 6, 4689, 27 ], [ 6, 4688, 22 ], [ 6, 4704, 1 ], [ 6, 4704, 2 ], [ 6, 4700, 3 ], 
  [ 6, 4700, 4 ], [ 6, 4704, 3 ], [ 6, 4704, 4 ], [ 6, 4704, 5 ], [ 6, 4700, 11 ], 
  [ 6, 4704, 6 ], [ 6, 4704, 7 ], [ 6, 4700, 18 ], [ 6, 4701, 13 ], [ 6, 39, 2 ], 
  [ 6, 39, 1 ], [ 6, 39, 3 ], [ 6, 4705, 1 ], [ 6, 4705, 2 ], [ 6, 4721, 1 ], 
  [ 6, 4721, 3 ], [ 6, 4721, 2 ], [ 6, 4721, 4 ], [ 6, 4721, 5 ], [ 6, 4721, 7 ], 
  [ 6, 4721, 6 ], [ 6, 4719, 8 ], [ 6, 4742, 1 ], [ 6, 4742, 2 ], [ 6, 4742, 3 ], 
  [ 6, 4742, 4 ], [ 6, 4742, 6 ], [ 6, 4742, 5 ], [ 6, 4742, 7 ], [ 6, 4742, 8 ], 
  [ 6, 4737, 9 ], [ 6, 4742, 9 ], [ 6, 4737, 11 ], [ 6, 4800, 1 ], [ 6, 4800, 5 ], 
  [ 6, 4800, 2 ], [ 6, 4800, 6 ], [ 6, 4800, 7 ], [ 6, 4800, 3 ], [ 6, 4800, 4 ], 
  [ 6, 4800, 8 ], [ 6, 4800, 9 ], [ 6, 4800, 19 ], [ 6, 4800, 14 ], [ 6, 4800, 20 ], 
  [ 6, 4800, 11 ], [ 6, 4800, 17 ], [ 6, 4800, 12 ], [ 6, 4800, 10 ], [ 6, 4800, 13 ], 
  [ 6, 4800, 21 ], [ 6, 4800, 22 ], [ 6, 4800, 15 ], [ 6, 4800, 18 ], [ 6, 4800, 23 ], 
  [ 6, 4800, 16 ], [ 6, 4787, 34 ], [ 6, 4787, 35 ], [ 6, 4798, 24 ], [ 6, 4798, 26 ], 
  [ 6, 4800, 25 ], [ 6, 4800, 26 ], [ 6, 4800, 27 ], [ 6, 4800, 24 ], [ 6, 4787, 41 ], 
  [ 6, 4787, 40 ], [ 6, 4798, 30 ], [ 6, 4798, 31 ], [ 6, 4815, 1 ], [ 6, 4815, 2 ], 
  [ 6, 4815, 3 ], [ 6, 4815, 4 ], [ 6, 4815, 5 ], [ 6, 4815, 6 ], [ 6, 4815, 7 ], 
  [ 6, 4815, 8 ], [ 6, 4815, 10 ], [ 6, 4815, 12 ], [ 6, 4815, 9 ], [ 6, 4815, 11 ], 
  [ 6, 4815, 13 ], [ 6, 4815, 16 ], [ 6, 4815, 17 ], [ 6, 4815, 14 ], [ 6, 4815, 18 ], 
  [ 6, 4815, 15 ], [ 6, 4815, 19 ], [ 6, 4815, 20 ], [ 6, 4815, 21 ], [ 6, 4880, 54 ], 
  [ 6, 4880, 56 ], [ 6, 4880, 59 ], [ 6, 4815, 23 ], [ 6, 4815, 25 ], [ 6, 4815, 22 ], 
  [ 6, 4815, 26 ], [ 6, 4815, 24 ], [ 6, 4815, 27 ], [ 6, 4815, 29 ], [ 6, 4815, 28 ], 
  [ 6, 4880, 76 ], [ 6, 4880, 80 ], [ 6, 4880, 77 ], [ 6, 4880, 81 ], [ 6, 4880, 82 ], 
  [ 6, 4880, 78 ], [ 6, 4881, 31 ], [ 6, 4881, 32 ], [ 6, 4875, 58 ], [ 6, 4815, 30 ], 
  [ 6, 4880, 83 ], [ 6, 4881, 33 ], [ 6, 4882, 31 ], [ 6, 4875, 61 ], [ 6, 2319, 1 ], 
  [ 6, 2319, 2 ], [ 6, 2319, 3 ], [ 6, 2319, 4 ], [ 6, 2319, 5 ], [ 6, 2366, 5 ], 
  [ 6, 2366, 7 ], [ 6, 2319, 6 ], [ 6, 2366, 9 ], [ 6, 2366, 10 ], [ 6, 2425, 1 ], 
  [ 6, 2425, 5 ], [ 6, 2425, 2 ], [ 6, 2425, 6 ], [ 6, 2425, 3 ], [ 6, 2425, 4 ], 
  [ 6, 2425, 7 ], [ 6, 2425, 12 ], [ 6, 2425, 9 ], [ 6, 2425, 14 ], [ 6, 2379, 15 ], 
  [ 6, 2379, 16 ], [ 6, 2379, 17 ], [ 6, 2425, 10 ], [ 6, 2425, 8 ], [ 6, 2425, 11 ], 
  [ 6, 2425, 13 ], [ 6, 2382, 19 ], [ 6, 2382, 21 ], [ 6, 2538, 25 ], [ 6, 2379, 21 ], 
  [ 6, 2425, 15 ], [ 6, 2379, 22 ], [ 6, 2379, 19 ], [ 6, 2379, 20 ], [ 6, 2379, 23 ], 
  [ 6, 2379, 24 ], [ 6, 2382, 22 ], [ 6, 2384, 50 ], [ 6, 2379, 25 ], [ 6, 2538, 30 ], 
  [ 6, 2436, 35 ], [ 6, 2665, 1 ], [ 6, 2665, 2 ], [ 6, 2665, 3 ], [ 6, 2665, 4 ], 
  [ 6, 2636, 5 ], [ 6, 2636, 7 ], [ 6, 2604, 15 ], [ 6, 2637, 5 ], [ 6, 2636, 8 ], 
  [ 6, 2700, 12 ], [ 6, 2604, 17 ], [ 6, 748, 1 ], [ 6, 748, 2 ], [ 6, 684, 7 ], 
  [ 6, 1985, 1 ], [ 6, 2293, 2 ], [ 6, 2293, 1 ], [ 6, 46, 1 ], [ 6, 46, 2 ], 
  [ 6, 145, 3 ], [ 6, 145, 4 ], [ 6, 251, 1 ], [ 6, 251, 3 ], [ 6, 251, 2 ], 
  [ 6, 251, 4 ], [ 6, 164, 12 ], [ 6, 164, 16 ], [ 6, 164, 15 ], [ 6, 164, 14 ], 
  [ 6, 164, 17 ], [ 6, 251, 5 ], [ 6, 184, 6 ], [ 6, 164, 18 ], [ 6, 498, 27 ], 
  [ 6, 1343, 1 ], [ 6, 1616, 4 ], [ 6, 1487, 2 ], [ 6, 1966, 1 ], [ 6, 1967, 1 ], 
  [ 6, 1971, 2 ], [ 6, 1971, 3 ], [ 6, 1976, 1 ], [ 6, 1976, 2 ], [ 6, 1972, 3 ], 
  [ 6, 1972, 4 ], [ 6, 1999, 1 ], [ 6, 1992, 2 ], [ 6, 1986, 2 ], [ 6, 2255, 2 ], 
  [ 6, 2255, 1 ], [ 6, 2261, 1 ], [ 6, 2257, 5 ], [ 6, 2257, 6 ], [ 6, 2261, 2 ], 
  [ 6, 2257, 2 ], [ 6, 2277, 3 ], [ 6, 2277, 4 ], [ 6, 2271, 9 ], [ 6, 2271, 10 ], 
  [ 6, 2277, 1 ], [ 6, 2277, 2 ], [ 6, 2271, 3 ], [ 6, 2271, 4 ], [ 6, 2270, 5 ], 
  [ 6, 2270, 6 ], [ 6, 2296, 2 ], [ 6, 2313, 8 ], [ 6, 2300, 4 ], [ 6, 2296, 1 ], 
  [ 6, 2313, 3 ], [ 6, 2300, 2 ], [ 6, 2006, 1 ], [ 6, 2002, 3 ], [ 6, 2041, 1 ], 
  [ 6, 2041, 2 ], [ 6, 2033, 5 ], [ 6, 2031, 3 ], [ 6, 2033, 6 ], [ 6, 2031, 4 ], 
  [ 6, 2066, 1 ], [ 6, 2066, 2 ], [ 6, 2059, 3 ], [ 6, 2059, 5 ], [ 6, 2066, 3 ], 
  [ 6, 2059, 6 ], [ 6, 2153, 1 ], [ 6, 2153, 3 ], [ 6, 2153, 2 ], [ 6, 2153, 4 ], 
  [ 6, 2139, 5 ], [ 6, 2152, 13 ], [ 6, 2139, 9 ], [ 6, 2139, 7 ], [ 6, 2139, 10 ], 
  [ 6, 2152, 18 ],[ 6, 2152, 17 ], [ 6, 2152, 15 ], [ 6, 2152, 19 ], [ 6, 2153, 6 ], 
  [ 6, 2153, 5 ], [ 6, 2139, 12 ], [ 6, 2139, 11 ], [ 6, 2152, 21 ], [ 6, 2152, 20 ], 
  [ 6, 2158, 1 ], [ 6, 2158, 2 ], [ 6, 2158, 3 ], [ 6, 2158, 4 ], [ 6, 2243, 19 ], 
  [ 6, 2243, 23 ], [ 6, 2243, 26 ], [ 6, 2243, 27 ], [ 6, 2243, 28 ], [ 6, 2243, 29 ], 
  [ 6, 2158, 6 ], [ 6, 2158, 5 ], [ 6, 2158, 7 ], [ 6, 2243, 32 ], [ 6, 2243, 31 ], 
  [ 6, 2243, 33 ], [ 6, 2243, 34 ], [ 6, 2243, 35 ], [ 6, 2158, 8 ], [ 6, 2254, 9 ], 
  [ 6, 2243, 36 ], [ 6, 2235, 26 ], [ 6, 904, 1 ], [ 6, 878, 4 ], [ 6, 876, 2 ], 
  [ 6, 951, 7 ], [ 6, 943, 2 ], [ 6, 899, 4 ], [ 6, 951, 8 ], [ 6, 951, 9 ], 
  [ 6, 1198, 1 ], [ 6, 1198, 2 ], [ 6, 1112, 3 ], [ 6, 1113, 7 ], [ 6, 998, 5 ], 
  [ 6, 1113, 8 ], [ 6, 1112, 4 ], [ 6, 998, 8 ], [ 6, 1342, 2 ], [ 6, 1190, 6 ], 
  [ 6, 998, 9 ], [ 6, 1001, 9 ], [ 6, 998, 10 ], [ 6, 1001, 10 ], [ 6, 1098, 10 ], 
  [ 6, 1342, 4 ], [ 6, 1190, 8 ], [ 6, 1098, 12 ], [ 6, 1758, 1 ], [ 6, 1956, 2 ], 
  [ 6, 1872, 3 ], [ 6, 1846, 2 ], [ 6, 1915, 3 ], [ 6, 1748, 4 ], [ 6, 1872, 4 ], 
  [ 6, 1870, 5 ], [ 6, 1915, 4 ], [ 6, 4897, 4 ], [ 6, 4897, 2 ], [ 6, 4897, 5 ], 
  [ 6, 4897, 1 ], [ 6, 4897, 3 ], [ 6, 4897, 6 ], [ 6, 4894, 3 ], [ 6, 4894, 7 ], 
  [ 6, 4894, 9 ], [ 6, 4913, 1 ], [ 6, 4913, 2 ], [ 6, 4912, 3 ], [ 6, 4913, 3 ], 
  [ 6, 4913, 4 ], [ 6, 4913, 5 ], [ 6, 4914, 10 ], [ 6, 4914, 12 ], [ 6, 4914, 11 ], 
  [ 6, 4914, 13 ], [ 6, 4914, 14 ], [ 6, 4952, 7 ], [ 6, 4952, 8 ], [ 6, 4914, 15 ], 
  [ 6, 4914, 16 ], [ 6, 4952, 10 ], [ 6, 4952, 11 ], [ 6, 4964, 6 ], [ 6, 4964, 8 ], 
  [ 6, 4964, 10 ], [ 6, 4964, 11 ], [ 6, 4914, 1 ], [ 6, 4914, 3 ], [ 6, 4914, 2 ], 
  [ 6, 4914, 4 ], [ 6, 4914, 5 ], [ 6, 4914, 7 ], [ 6, 4914, 6 ], [ 6, 4914, 8 ], 
  [ 6, 4914, 9 ], [ 6, 4979, 13 ], [ 6, 4979, 14 ], [ 6, 4979, 15 ], [ 6, 4979, 16 ], 
  [ 6, 4979, 17 ], [ 6, 4979, 18 ], [ 6, 4966, 10 ], [ 6, 4966, 12 ], [ 6, 4966, 15 ], 
  [ 6, 4979, 19 ], [ 6, 4979, 20 ], [ 6, 4979, 21 ], [ 6, 4979, 22 ], [ 6, 4979, 23 ], 
  [ 6, 4966, 22 ], [ 6, 4966, 26 ], [ 6, 4966, 23 ], [ 6, 4966, 27 ], [ 6, 4966, 24 ], 
  [ 6, 4966, 28 ], [ 6, 4979, 24 ], [ 6, 4966, 29 ], [ 6, 4978, 25 ], [ 6, 4979, 1 ], 
  [ 6, 4979, 2 ], [ 6, 4979, 3 ], [ 6, 4979, 4 ], [ 6, 4979, 5 ], [ 6, 4979, 6 ], 
  [ 6, 4979, 8 ], [ 6, 4979, 7 ], [ 6, 4979, 9 ], [ 6, 4979, 10 ], [ 6, 4979, 11 ], 
  [ 6, 4979, 12 ], [ 6, 4979, 25 ], [ 6, 4979, 26 ], [ 6, 4979, 27 ], [ 6, 4979, 28 ], 
  [ 6, 5144, 9 ], [ 6, 5144, 11 ], [ 6, 5144, 10 ], [ 6, 5144, 12 ], [ 6, 5144, 13 ], 
  [ 6, 5116, 6 ], [ 6, 5116, 8 ], [ 6, 5144, 14 ], [ 6, 5118, 21 ], [ 6, 5116, 9 ], 
  [ 6, 5144, 1 ], [ 6, 5144, 3 ], [ 6, 5144, 2 ], [ 6, 5144, 4 ], [ 6, 5144, 5 ], 
  [ 6, 5144, 6 ], [ 6, 5144, 7 ], [ 6, 5144, 8 ], [ 6, 5209, 1 ], [ 6, 5209, 2 ], 
  [ 6, 5448, 3 ], [ 6, 5447, 3 ], [ 6, 5448, 4 ], [ 6, 5347, 11 ], [ 6, 5209, 3 ], 
  [ 6, 5310, 2 ], [ 6, 5209, 4 ], [ 6, 5448, 7 ], [ 6, 5310, 4 ], [ 6, 5491, 14 ], 
  [ 6, 5448, 8 ], [ 6, 5491, 16 ], [ 6, 5209, 5 ], [ 6, 5310, 8 ], [ 6, 5448, 12 ], 
  [ 6, 5491, 24 ], [ 6, 5517, 6 ], [ 6, 5517, 7 ], [ 6, 5503, 8 ], [ 6, 5557, 11 ], 
  [ 6, 5517, 3 ], [ 6, 5517, 4 ], [ 6, 5517, 8 ], [ 6, 5517, 1 ], [ 6, 5517, 2 ], 
  [ 6, 5529, 2 ], [ 6, 5529, 4 ], [ 6, 5533, 5 ], [ 6, 5517, 5 ], [ 6, 5529, 8 ], 
  [ 6, 5517, 9 ], [ 6, 5529, 13 ], [ 6, 5533, 17 ], [ 6, 5557, 15 ], [ 6, 5533, 18 ], 
  [ 6, 5566, 1 ], [ 6, 5566, 2 ], [ 6, 5582, 1 ], [ 6, 5582, 2 ], [ 6, 5582, 3 ], 
  [ 6, 5582, 4 ], [ 6, 5582, 5 ], [ 6, 5578, 11 ], [ 6, 5584, 1 ], [ 6, 5584, 2 ], 
  [ 6, 5599, 5 ], [ 6, 5601, 1 ], [ 6, 5749, 3 ], [ 6, 5749, 4 ], [ 6, 5667, 1 ], 
  [ 6, 5749, 2 ], [ 6, 5700, 3 ], [ 6, 5830, 8 ], [ 6, 5830, 9 ], [ 6, 5830, 10 ], 
  [ 6, 5830, 11 ], [ 6, 5830, 12 ], [ 6, 5889, 11 ], [ 6, 5993, 1 ], [ 6, 5830, 4 ], 
  [ 6, 5993, 2 ], [ 6, 5830, 6 ], [ 6, 6090, 6 ], [ 6, 5830, 7 ], [ 6, 6098, 8 ], 
  [ 6, 6322, 1 ], [ 6, 6322, 2 ], [ 6, 6404, 5 ], [ 6, 6509, 1 ], [ 6, 6322, 5 ], 
  [ 6, 6213, 11 ], [ 6, 6878, 2 ], [ 6, 6568, 1 ], [ 6, 6722, 2 ], [ 6, 3313, 1 ], 
  [ 6, 3328, 1 ], [ 6, 3328, 2 ], [ 6, 3330, 1 ], [ 6, 3332, 2 ], [ 6, 3358, 1 ], 
  [ 6, 3362, 2 ], [ 6, 3446, 1 ], [ 6, 3376, 3 ], [ 6, 3416, 3 ], [ 6, 3376, 6 ], 
  [ 6, 3478, 1 ], [ 6, 3478, 2 ], [ 6, 3635, 5 ], [ 6, 3635, 6 ], [ 6, 3513, 7 ], 
  [ 6, 3513, 8 ], [ 6, 3719, 5 ], [ 6, 3719, 6 ], [ 6, 3635, 11 ], [ 6, 3635, 12 ], 
  [ 6, 3891, 1 ], [ 6, 3988, 2 ], [ 6, 3988, 4 ], [ 6, 3795, 3 ], [ 6, 4349, 1 ], 
  [ 6, 4586, 4 ], [ 6, 4616, 2 ], [ 6, 4232, 2 ], [ 6, 4586, 8 ], [ 6, 4528, 7 ], 
  [ 6, 2965, 1 ], [ 6, 2961, 3 ], [ 6, 2988, 1 ], [ 6, 2988, 2 ], [ 6, 2982, 5 ], 
  [ 6, 2982, 6 ], [ 6, 2989, 1 ], [ 6, 3020, 1 ], [ 6, 3065, 2 ], [ 6, 3054, 3 ], 
  [ 6, 3065, 1 ], [ 6, 3113, 3 ], [ 6, 3113, 4 ], [ 6, 3097, 5 ], [ 6, 3097, 6 ], 
  [ 6, 3097, 7 ], [ 6, 3097, 8 ], [ 6, 3113, 1 ], [ 6, 3113, 2 ], [ 6, 3120, 2 ], 
  [ 6, 3120, 1 ], [ 6, 3127, 2 ], [ 6, 3193, 2 ], [ 6, 3193, 1 ], [ 6, 3272, 1 ], 
  [ 6, 3272, 2 ], [ 6, 3272, 5 ], [ 6, 3272, 3 ], [ 6, 3272, 4 ], [ 6, 3258, 5 ], 
  [ 6, 3296, 5 ], [ 6, 3296, 6 ], [ 6, 3282, 8 ], [ 6, 3282, 9 ], [ 6, 3296, 1 ], 
  [ 6, 3296, 4 ], [ 6, 3282, 5 ], [ 6, 3296, 2 ], [ 6, 3296, 3 ], [ 6, 2773, 1 ], 
  [ 6, 2773, 2 ], [ 6, 2773, 3 ], [ 6, 2772, 5 ], [ 6, 2772, 2 ], [ 6, 2750, 4 ], 
  [ 6, 2803, 1 ], [ 6, 2804, 1 ], [ 6, 2804, 2 ], [ 6, 2866, 3 ], [ 6, 2866, 2 ], 
  [ 6, 2930, 1 ], [ 6, 2932, 2 ], [ 6, 2932, 1 ], [ 6, 2945, 1 ], [ 6, 2949, 1 ], 
  [ 6, 2952, 1 ], [ 6, 2952, 2 ], [ 6, 2952, 3 ], [ 6, 2953, 2 ], [ 6, 2953, 3 ], 
  [ 6, 2953, 1 ] ]
gap> Length(last); # there exist 841 Bravais groups of dimension 6
841
\end{verbatim}
\end{example}

\bigskip

\begin{remark}
We can get all the maximal group $G\leq \GL(n,\bZ)$ $(n\leq 6)$
by the function {\tt MaximalGroupsID(L)} as in Subsection \ref{ssMax} 
applying to only the Bravais groups $L$ (cf. Subsection \ref{ssMax}). 
For example, we get the following for $n=6$ which agrees 
the result in Subsection \ref{ssMax}: 

\bigskip

\begin{verbatim}
gap> Read("caratnumber.gap");
gap> Read("KS.gap");

gap> MaximalGroupsID(b6ca:Carat,FromPerm);
[ [ 6, 5209, 1 ], [ 6, 5209, 5 ], [ 6, 5517, 3 ], [ 6, 5517, 4 ], [ 6, 5517, 8 ], 
  [ 6, 5517, 5 ], [ 6, 6509, 1 ], [ 6, 6878, 2 ], [ 6, 6568, 1 ], [ 6, 3891, 1 ], 
  [ 6, 3795, 3 ], [ 6, 4232, 2 ], [ 6, 3120, 2 ], [ 6, 3120, 1 ], [ 6, 3193, 2 ], 
  [ 6, 3193, 1 ], [ 6, 3272, 5 ], [ 6, 3272, 3 ], [ 6, 3296, 5 ], [ 6, 3296, 1 ], 
  [ 6, 3296, 4 ], [ 6, 3296, 2 ], [ 6, 2773, 1 ], [ 6, 2773, 2 ], [ 6, 2773, 3 ], 
  [ 6, 2772, 5 ], [ 6, 2772, 2 ], [ 6, 2750, 4 ], [ 6, 2803, 1 ], [ 6, 2804, 1 ], 
  [ 6, 2804, 2 ], [ 6, 2866, 3 ], [ 6, 2866, 2 ], [ 6, 2932, 2 ], [ 6, 2932, 1 ], 
  [ 6, 2945, 1 ], [ 6, 2952, 1 ], [ 6, 2952, 2 ], [ 6, 2952, 3 ] ]
gap> Length(last); # there exist 39 maximal groups in dimension 6
39
\end{verbatim}
\end{remark}

\bigskip


Voskresenskii (see \cite{Vos83}, \cite[Section 8]{Vos98}) 
investigated the rationality problem for algebraic tori $T$ 
which correspond to Bravais groups $G\leq\GL(n,\bZ)$. 
Using the classification of Bravais groups above, 
we will show that 
$H^1(G,[M_G]^{fl})=0$ for any Bravais group $G$ of dimension $n\leq 6$ 
in Section \ref{seBravais} (see Theorem \ref{thBravais6}).


By using the algorithms below, one can obtain (positive definite) 
invariant quadratic forms $f$ (or corresponding invariant forms $F\in\bR_{\rm sym}^{n\times n}$) under the action of Bravais group $G\leq\GL(n,\bZ)$. 
%
We will present the result only for the maximal groups $G\leq\GL(n,\bZ)$ 
and for $n\leq 6$ in the end of this section. 
For the maximal cases, 
see also \cite{Dad65} for $n=4$, \cite{Rys72a}, \cite{Rys72b}, 
\cite{RL80} for $n=5$, \cite{PP77} for $n=6$ (irreducible cases). 
%
The following algorithms 
are available from 
{\tt http://math.h.kyoto-u.ac.jp/\~{}yamasaki/Algorithm/} 
as {\tt Bravais.gap}.

\bigskip

\noindent
{\tt InvariantFormsMatrix(G)} returns 
a basis of the space of invariant forms 
$\mathcal{F}(G)=\{F\in\bR_{\rm sym}^{n\times n}\mid g\,F\, {}^tg=F\ 
({\rm for\ any}\ g\in G)\}$ for a finite group $G\leq\GL(n,\bZ)$. 

\noindent
{\tt InvariantQuadraticForms(G)}
returns a 
basis of the set of all invariant quadratic forms
under the action of $G\leq\GL(n,\bZ)$.

\noindent
{\tt AutomorphismGroupOfMatrix(F)}
returns the automorphism group $G=\{g\in\GL(n,\bZ)\mid g\,F\, {}^tg=F\}$ of a positive 
definite symmetric matrix $F\in\bR_{\rm sym}^{n\times n}$.

\noindent
{\tt QuadraticFormToMatrix(f)} 
returns a symmetric matrix $F$ which corresponds to quadratic form $f$.

\noindent
{\tt QuadraticFormToMatrix(f,R)},
(resp. {\tt QuadraticFormToMatrix(f,[}$x_1,\dots,x_n${\tt ])}) 
returns the same as\\ {\tt QuadraticFormToMatrix(f)} but 
with respect to the base ring $R$ (resp. variables $x_1,\dots,x_n$). 

\noindent
{\tt MatrixToQuadraticForm(F)} returns a quadratic form $f$ 
which corresponds to a symmetric matrix $F$.

\noindent
{\tt MatrixToQuadraticForm(F,R)},
(resp. {\tt MatrixToQuadraticForm(F,[}$x_1,\dots,x_n${\tt ])}) 
returns the same as\\ {\tt MatrixToQuadraticForm(F)} 
but with respect to the base ring $R$ (resp. variables $x_1,\dots,x_n$).

\noindent
{\tt IsPositiveDefinite(F)}
returns whether a symmetric matrix $F$ is positive definite or not.

\bigskip

\begin{verbatim}
InvariantFormsMatrix:= function(G)
    local d,gg,ut,bs,b,i,m,l,g,p;
    d:=Length(Identity(G));
    gg:=GeneratorsOfGroup(G);
    ut:=UnorderedTuples([1..d],2);
    bs:=[];
    for i in ut do
        b:=NullMat(d,d);
        b[i[1]][i[2]]:=1;
        b[i[2]][i[1]]:=1;
        Add(bs,b);
    od;
    if gg=[] then
        return bs;
    fi;
    m:=[];
    for b in bs do
        l:=[];
        for g in gg do
            p:=g*b*TransposedMat(g)-b;
            l:=Concatenation(l,List(ut,x->p[x[1]][x[2]]));
        od;
        Add(m,l);
    od;
    return List(NullspaceIntMat(m),x->x*bs);
end;

InvariantQuadraticForms:= function(G)
    local d,R,xx,gg,ut,bp,bm,b,i,bx,m,l,g,p;
    d:=Length(Identity(G));
    R:=PolynomialRing(Integers,d);
    xx:=IndeterminatesOfPolynomialRing(R);
    gg:=GeneratorsOfGroup(G);
    ut:=UnorderedTuples([1..d],2);
    bp:=List(ut,x->xx[x[1]]*xx[x[2]]);
    bm:=[];
    for i in ut do
        b:=NullMat(d,d);
        if i[1]=i[2] then
            b[i[1]][i[2]]:=1;
        else
            b[i[1]][i[2]]:=1/2;
            b[i[2]][i[1]]:=1/2;
        fi;
        Add(bm,b);
    od;
    if gg=[] then
        bx:=bp;
    else
        m:=[];
        for b in bm do
            l:=[];
            for g in gg do
                p:=g*b*TransposedMat(g)-b;
                l:=Concatenation(l,List(ut,x->p[x[1]][x[2]]*Length(Set(x))));
            od;
            Add(m,l);
        od;
        bx:=List(NullspaceIntMat(m),x->x*bp);
    fi;
    return bx;
end;

AutomorphismGroupOfMatrix:= function(M)
    local d,LM,M0,P,nb,m,vv,cb,ii,sp,gg,g;
    d:=Length(M);
    LM:=LLLReducedGramMat(M);
    M0:=LM.remainder;
    P:=LM.transformation;
    nb:=List([1..d],x->M0[x][x]);
    m:=Maximum(nb);
    vv:=ShortestVectors(M0,m);
    cb:=List(nb,x->Filtered(vv.vectors,y->y*M0*y=x));
    cb:=List(cb,x->Concatenation(List(x,y->[y,-y])));
    ii:=List([1..d],Zero);
    ii[1]:=1;
    sp:=1;
    gg:=[];
    while sp>0 do
        if ii[sp]>Length(cb[sp]) then
            sp:=sp-1;
            if sp>0 then
                ii[sp]:=ii[sp]+1;
            fi;
        elif ForAll([1..sp-1],x->cb[x][ii[x]]*M0*cb[sp][ii[sp]]=M0[x][sp]) then
            if sp<d then
                sp:=sp+1;
                ii[sp]:=1;
            else
                g:=List([1..d],x->cb[x][ii[x]])^P;
                if not g in Group(gg,IdentityMat(d)) then
                    Add(gg,g);
                fi;
                ii[sp]:=ii[sp]+1;
            fi;
        else
            ii[sp]:=ii[sp]+1;
        fi;
    od;
    return Group(gg);
end;

QuadraticFormToMatrix:= function(arg)
    local f,R,xx,d,M;
    f:=arg[1];
    if Length(arg)=1 then
        R:=DefaultRing(f);
        xx:=IndeterminatesOfPolynomialRing(R);
    elif IsRing(arg[2]) then
        R:=arg[2];
        xx:=IndeterminatesOfPolynomialRing(R);
    else
        xx:=arg[2];
    fi;
    d:=Length(xx);
    M:=List([1..d],i->List([1..d],j->Derivative(Derivative(f,xx[i]),xx[j])));
    M:=List([1..d],i->List([1..d],j->Value(M[i][j],xx,List(xx,x->0))));
    return M/2;
end;

MatrixToQuadraticForm:= function(arg)
    local M,d,R,xx;
    M:=arg[1];
    d:=Length(M);
    if Length(arg)=1 then
        R:=PolynomialRing(Integers,d);
        xx:=IndeterminatesOfPolynomialRing(R);
    elif IsRing(arg[2]) then
        R:=arg[2];
        xx:=IndeterminatesOfPolynomialRing(R);
    else
        xx:=arg[2];
    fi;
    return xx*M*xx;
end;

IsPositiveDefinite:= function(M)
    local d;
    d:=Length(M);
    return ForAll([1..d],x->Determinant(M{[1..x]}{[1..x]})>0);
end;
\end{verbatim}

\bigskip

\begin{example}[Functions related to Bravais groups $G\leq\GL(n,\bZ)$ in {\tt Bravais.gap}]
{}~{}\\
\begin{verbatim}
gap> Read("Bravais.gap");
gap> Read("caratnumber.gap");
gap> Read("KS.gap");

gap> InvariantFormsMatrix(ImfMatrixGroup(4,1,1));
[ [ [ 2, 1, 0, 0 ], 
    [ 1, 2, 1, 1 ], 
    [ 0, 1, 2, 0 ], 
    [ 0, 1, 0, 2 ] ] ]
gap> InvariantFormsMatrix(IndmfMatrixGroup(6,10,1));
[ [ [ 1, 0, 0, 0, 0, 0 ], 
    [ 0, 1, 0, 0, 0, 0 ], 
    [ 0, 0, 0, 0, 0, 0 ], 
    [ 0, 0, 0, 0, 0, 0 ], 
    [ 0, 0, 0, 0, 1, 0 ], 
    [ 0, 0, 0, 0, 0, 0 ] ], 
  [ [ 0, 3, -2, -4, -3, -6 ], 
    [ 3, 0, -2, -4, -3, -6 ], 
    [ -2, -2, 4, 4, 2, 4 ], 
    [ -4, -4, 4, 8, 4, 8 ], 
    [ -3, -3, 2, 4, 0, 6 ],
    [ -6, -6, 4, 8, 6, 12 ] ] ]
gap> InvariantQuadraticForms(ImfMatrixGroup(4,1,1));
[ x_1^2+x_1*x_2+x_2^2+x_2*x_3+x_2*x_4+x_3^2+x_4^2 ]
gap> InvariantQuadraticForms(IndmfMatrixGroup(6,10,1));
[ x_1^2+x_2^2+x_5^2, 
  3*x_1*x_2-2*x_1*x_3-4*x_1*x_4-3*x_1*x_5-6*x_1*x_6-2*x_2*x_3
  -4*x_2*x_4-3*x_2*x_5-6*x_2*x_6+2*x_3^2+4*x_3*x_4+2*x_3*x_5
  +4*x_3*x_6+4*x_4^2+4*x_4*x_5+8*x_4*x_6+6*x_5*x_6+6*x_6^2 ]

gap> AutomorphismGroupOfMatrix([[1,0],[0,1]]);
Group([ [ [ 0, 1 ], [ 1, 0 ] ], [ [ 0, 1 ], [ -1, 0 ] ] ])
gap> StructureDescription(last); # the dihedral group of order 8
"D8"
gap> AutomorphismGroupOfMatrix([[2,1],[1,2]]);
Group([ [ [ 0, 1 ], [ -1, 1 ] ], [ [ 0, 1 ], [ 1, 0 ] ] ])
gap> StructureDescription(last); # the dihedral group of order 12
"D12"

gap> R:=PolynomialRing(Integers,4);
Integers[x_1,x_2,x_3,x_4]
gap> [R.1,R.2,R.3,R.4];
[ x_1, x_2, x_3, x_4 ]
gap> QuadraticFormToMatrix(R.1^2+R.1*R.2+R.2^2+R.3^2+R.3*R.4+R.4^2);
[ [ 1, 1/2, 0, 0 ], [ 1/2, 1, 0, 0 ], [ 0, 0, 1, 1/2 ], [ 0, 0, 1/2, 1 ] ]
gap> QuadraticFormToMatrix(R.1^2+R.1*R.2+R.2^2+R.3^2+R.3*R.4+R.4^2,R);
[ [ 1, 1/2, 0, 0 ], [ 1/2, 1, 0, 0 ], [ 0, 0, 1, 1/2 ], [ 0, 0, 1/2, 1 ] ]
gap> QuadraticFormToMatrix(R.1^2+R.1*R.2+R.2^2+R.3^2+R.3*R.4+R.4^2,[R.1,R.2,R.3,R.4]);
[ [ 1, 1/2, 0, 0 ], [ 1/2, 1, 0, 0 ], [ 0, 0, 1, 1/2 ], [ 0, 0, 1/2, 1 ] ]
gap> QuadraticFormToMatrix(R.1^2+R.1*R.2+R.2^2+R.3^2+R.3*R.4+R.4^2,[R.1,R.3,R.2,R.4]);
[ [ 1, 0, 1/2, 0 ], [ 0, 1, 0, 1/2 ], [ 1/2, 0, 1, 0 ], [ 0, 1/2, 0, 1 ] ]
gap> MatrixToQuadraticForm(IdentityMat(4));
x_1^2+x_2^2+x_3^2+x_4^2
gap> MatrixToQuadraticForm(IdentityMat(4),R);
x_1^2+x_2^2+x_3^2+x_4^2
gap> MatrixToQuadraticForm(IdentityMat(4),[R.1,R.2,R.3,R.4]);
x_1^2+x_2^2+x_3^2+x_4^2

gap> IsPositiveDefinite([[2,1],[1,2]]);
true
gap> IsPositiveDefinite([[2,2],[2,2]]);
false
\end{verbatim}
\end{example}

\newpage

{\small
\begin{longtable}{lll}
GAP ID/\\
CARAT ID & Imf/Indmf & Invariant (positive definite) quadratic form of maximal group $G\leq\GL(n,\bZ)$\\
\hline
(1,2,1) & Imf(1,1,1) & $a_1x_1^2 \ (a_1>0)$ \\
\hline
(2,3,2,1) & Imf(2,1,1) & $a_1 (x_1^2+x_2^2) \ (a_1>0)$ \\

(2,4,4,1) & Imf(2,2,1) & $a_1 (x_1^2-x_1x_2+x_2^2) \ (a_1>0)$ \\
\hline
(3,6,7,1) & Imf(2,2,1)$\times$Imf(1,1,1) & $a_1(x_1^2-x_1x_2+x_2^2)+a_2x_3^2 \ (a_1,a_2>0)$ \\

(3,7,5,1) & Imf(3,1,1) & $a_1(x_1^2+x_2^2+x_3^2) \ (a_1>0)$ \\

(3,7,5,2) & Imf(3,1,2) & $a_1(3x_1^2-2x_1x_2-2x_1x_3+3x_2^2-2x_2x_3+3x_3^2) \ (a_1>0)$\\

(3,7,5,3) & Imf(3,1,3) & $a_1(x_1^2+x_1x_2+x_1x_3+x_2^2+x_2x_3+x_3^2) \ (a_1>0)$\\
\hline
(4,20,22,1) & Imf(2,1,1)$\times$Imf(2,2,1) &
$a_1(x_1^2+x_2^2)+a_2(x_3^2-x_3x_4+x_4^2) \ (a_1,a_2>0)$\\

(4,25,11,2) & Imf(3,1,3)$\times$Imf(1,1,1) &
$a_1(x_1^2+x_1x_2+x_1x_3+x_2^2+x_2x_3+x_3^2)+a_2x_4^2 \ (a_1,a_2>0)$\\

(4,25,11,4) & Imf(3,1,2)$\times$Imf(1,1,1) &
 $a_1(3x_1^2-2x_1x_2-2x_1x_3+3x_2^2-2x_2x_3+3x_3^2)+a_2x_4^2 \ (a_1,a_2>0)$\\

(4,29,9,1) & Imf(4,5,1) &
$a_1(2x_1^2-2x_1x_2-2x_1x_3+x_1x_4+2x_2^2+x_2x_3-2x_2x_4+2x_3^2-2x_3x_4+2x_4^2) \ (a_1>0)$\\

(4,30,13,1) & Imf(4,2,1) &
$a_1(x_1^2-x_1x_2+x_2^2+x_3^2-x_3x_4+x_4^2) \ (a_1>0)$\\

(4,31,7,1) & Imf(4,3,1) &
$a_1(x_1^2+x_1x_2+x_1x_3+x_1x_4+x_2^2+x_2x_3+x_2x_4+x_3^2+x_3x_4+x_4^2) \ (a_1>0)$\\

(4,31,7,2) & Imf(4,3,2) &
$a_1(2x_1^2-x_1x_2-x_1x_3-x_1x_4+2x_2^2-x_2x_3-x_2x_4+2x_3^2-x_3x_4+2x_4^2) \ (a_1>0)$\\

(4,32,21,1) & Imf(4,4,1) & $a_1(x_1^2+x_2^2+x_3^2+x_4^2) \ (a_1>0)$ \\

(4,33,16,1) & Imf(4,1,1) &
$a_1(x_1^2+x_1x_2+x_2^2+x_2x_3+x_2x_4+x_3^2+x_4^2) \ (a_1>0)$\\
\hline
(5,559,3) & Imf(3,1,2)$\times$Imf(2,1,1) &
$a_1(3x_1^2-2x_1x_2-2x_1x_3+3x_2^2-2x_2x_3+3x_3^2)+a_2(x_4^2+x_5^2) \ (a_1,a_2>0)$\\

(5,559,4) & Imf(3,1,3)$\times$Imf(2,1,1) &
$a_1(x_1^2+x_1x_2+x_1x_3+x_2^2+x_2x_3+x_3^2)+a_2(x_4^2+x_5^2) \ (a_1,a_2>0)$\\

(5,626,1) & Imf(3,1,2)$\times$Imf(2,2,1) &
$a_1(3x_1^2-2x_1x_2-2x_1x_3+3x_2^2-2x_2x_3+3x_3^2)+a_2(x_4^2-x_4x_5+x_5^2) \ (a_1,a_2>0)$\\

(5,626,2) & Imf(3,1,1)$\times$Imf(2,2,1) &
$a_1(x_1^2+x_2^2+x_3^2)+a_2(x_4^2-x_4x_5+x_5^2) \ (a_1,a_2>0)$\\

(5,626,3) & Imf(3,1,3)$\times$Imf(2,2,1) &
$a_1(x_1^2+x_1x_2+x_1x_3+x_2^2+x_2x_3+x_3^2)+a_2(x_4^2-x_4x_5+x_5^2) \ (a_1,a_2>0)$\\

(5,690,1) & Imf(4,1,1)$\times$Imf(1,1,1) &
$a_1(x_1^2+x_1x_2+x_2^2+x_2x_3+x_2x_4+x_3^2+x_4^2)+a_2x_5^2 \ (a_1,a_2>0)$\\

(5,836,2) & Imf(4,5,1)$\times$Imf(1,1,1) &
$a_1(2x_1^2-2x_1x_2-2x_1x_3+x_1x_4+2x_2^2+x_2x_3-2x_2x_4+2x_3^2-2x_3x_4+2x_4^2)$\\
&&$+a_2x_5^2 \ (a_1,a_2>0)$\\

(5,866,1) & Imf(4,2,1)$\times$Imf(1,1,1) &
$a_1(x_1^2-x_1x_2+x_2^2+x_3^2-x_3x_4+x_4^2)+a_2x_5^2 \ (a_1,a_2>0)$\\

(5,930,1) & Imf(4,3,2)$\times$Imf(1,1,1) &
$a_1(2x_1^2-x_1x_2-x_1x_3-x_1x_4+2x_2^2-x_2x_3-x_2x_4+2x_3^2-x_3x_4+2x_4^2)$\\
&&$+a_2x_5^2 \ (a_1,a_2>0)$\\

(5,930,2) & Imf(4,3,1)$\times$Imf(1,1,1) &
$a_1(x_1^2+x_1x_2+x_1x_3+x_1x_4+x_2^2+x_2x_3+x_2x_4+x_3^2+x_3x_4+x_4^2)+a_2x_5^2$\\
&&$\ (a_1,a_2>0)$\\

(5,942,1) & Imf(5,1,1) &
$a_1(x_1^2+x_2^2+x_3^2+x_4^2+x_5^2) \ (a_1>0)$ \\

(5,942,2) & Imf(5,1,2) &
$a_1(x_1^2+x_1x_2+x_2^2+x_2x_3+x_3^2+x_3x_4+x_3x_5+x_4^2+x_5^2) \ (a_1>0)$\\

(5,942,3) & Imf(5,1,3) &
$a_1(4x_1^2+4x_1x_5+4x_2^2+4x_2x_5+4x_3^2+4x_3x_5+4x_4^2+4x_4x_5+5x_5^2) \ (a_1>0)$\\

(5,949,1) & Imf(5,2,1) &
$a_1(5x_1^2-2x_1x_2-2x_1x_3-2x_1x_4-2x_1x_5+5x_2^2-2x_2x_3-2x_2x_4-2x_2x_5$\\
&&$+5x_3^2-2x_3x_4-2x_3x_5+5x_4^2-2x_4x_5+5x_5^2) \ (a_1>0)$\\

(5,949,2) & Imf(5,2,3) &
$a_1(2x_1^2+x_1x_2-2x_1x_3-2x_1x_4-2x_1x_5+2x_2^2+x_2x_3-2x_2x_4+x_2x_5$\\
&&$+2x_3^2+x_3x_4+x_3x_5+2x_4^2+x_4x_5+2x_5^2) \ (a_1>0)$\\

(5,949,3) & Imf(5,2,4) &
$a_1(3x_1^2+2x_1x_2-2x_1x_3-2x_1x_4-2x_1x_5+3x_2^2+2x_2x_3-2x_2x_4+2x_2x_5$\\
&&$+3x_3^2+2x_3x_4+2x_3x_5+3x_4^2+2x_4x_5+3x_5^2) \ (a_1>0)$\\

(5,949,4) & Imf(5,2,2) &
$a_1(x_1^2+x_1x_2+x_1x_3+x_1x_4+x_1x_5+x_2^2+x_2x_3+x_2x_4+x_2x_5$\\
&&$+x_3^2+x_3x_4+x_3x_5+x_4^2+x_4x_5+x_5^2) \ (a_1>0)$\\
\hline
(6,5209,1) & Imf(3,1,3)$\times$Imf(2,2,1)& 
$a_1(x_1^2+x_1x_2+x_1x_3+x_2^2+x_2x_3+x_3^2)+a_2(x_4^2-x_4x_5+x_5^2)+a_3x_6^2$
\\
&$\times$Imf(1,1,1) & 
$\ (a_1,a_2,a_3>0)$
\\

(6,5209,5) & Imf(3,1,2)$\times$Imf(2,2,1)& 
$a_1(3x_1^2-2x_1x_2-2x_1x_3+3x_2^2-2x_2x_3+3x_3^2)+a_2(x_4^2-x_4x_5+x_5^2)+a_3x_6^2$\\
&$\times$Imf(1,1,1) &
$\ (a_1,a_2>0)$\\

(6,5517,3) & Imf(3,1,1)$\times$Imf(3,1,3) &
$a_1(x_1^2+x_2^2+x_3^2)+a_2(x_4^2+x_4x_5+x_4x_6+x_5^2+x_5x_6+x_6^2) \ (a_1,a_2>0)$\\

(6,5517,4) & Indmf(6,10,1) &
$a_1(x_1^2+x_2^2+x_5^2)+a_2(3x_1x_2-2x_1x_3-4x_1x_4-3x_1x_5-6x_1x_6$\\
&&$-2x_2x_3-4x_2x_4-3x_2x_5-6x_2x_6+2x_3^2+4x_3x_4+2x_3x_5+4x_3x_6$\\
&&$+4x_4^2+4x_4x_5+8x_4x_6+6x_5x_6+6x_6^2) \ (\frac{2}{3}a_1>a_2>0)$\\

(6,5517,8) & Imf(3,1,2)$\times$Imf(3,1,3) &
$a_1(3x_1^2-2x_1x_2-2x_1x_3+3x_2^2-2_2x_3+3x_3^2)$\\
&&$+a_2(x_4^2+x_4x_5+x_4x_6+x_5^2+x_5x_6+x_6^2) \ (a_1,a_2>0)$\\

(6,5517,5) & Imf(3,1,1)$\times$Imf(3,1,2) &
$a_1(x_1^2+x_2^2+x_3^2)+a_2(3x_4^2-2x_4x_5-2x_4x_6+3x_5^2-2x_5x_6+3x_6^2) \ (a_1,a_2>0)$\\

(6,6509,1) & Imf(4,1,1)$\times$Imf(2,1,1) &
$a_1(x_1^2+x_1x_2+x_2^2+x_2x_3+x_2x_4+x_3^2+x_4^2)+a_2(x_5^2+x_6^2) \ (a_1,a_2>0)$\\

(6,6878,2) & Imf(4,4,1)$\times$Imf(2,2,1) &
$a_1(x_1^2+x_2^2+x_3^2+x_4^2)+a_2(x_5^2-x_5x_6+x_6^2) \ (a_1,a_2>0)$\\

(6,6568,1) & Imf(4,1,1)$\times$Imf(2,2,1) &
$a_1(x_1^2+x_1x_2+x_2^2+x_2x_3+x_2x_4+x_3^2+x_4^2)+a_2(x_5^2-x_5x_6+x_6^2) \ (a_1,a_2>0)$\\

(6,3891,1) & Imf(4,2,1)$\times$Imf(2,1,1) &
$a_1(x_1^2-x_1x_2+x_2^2+x_3^2-x_3x_4+x_4^2)+a_2(x_5^2+x_6^2) \ (a_1,a_2>0)$\\

(6,3795,3) & Imf(4,5,1)$\times$Imf(2,1,1) &
$a_1(2x_1^2-2x_1x_2-2x_1x_3+x_1x_4+2x_2^2+x_2x_3-2x_2x_4$\\
&&$+2x_3^2-2x_3x_4+2x_4^2)+a_2(x_5^2+x_6^2) \ (a_1,a_2>0)$\\

(6,4232,2) & Imf(4,5,1)$\times$Imf(2,2,1) &
$a_1(2x_1^2-2x_1x_2-2x_1x_3+x_1x_4+2x_2^2+x_2x_3-2x_2x_4+2x_3^2-2x_3x_4+2x_4^2)$\\
&&$+a_2(x_5^2-x_5x_6+x_6^2) \ (a_1,a_2>0)$\\

(6,3120,2) & Imf(4.3,1)$\times$Imf(2,1,1) &
$a_1(x_1^2+x_1x_2+x_1x_3+x_1x_4+x_2^2+x_2x_3+x_2x_4+x_3^2+x_3x_4+x_4^2)$\\
&&$+a_2(x_5^2+x_6^2) \ (a_1,a_2>0)$\\

(6,3120,1) & Imf(4,3,2)$\times$Imf(2,1,1) &
$a_1(2x_1^2-x_1x_2-x_1x_3-x_1x_4+2x_2^2-x_2x_3-x_2x_4+2x_3^2-x_3x_4+2x_4^2)$\\
&&$+a_2(x_5^2+x_6^2) \ (a_1,a_2>0)$\\

(6,3193,2) & Imf(4,3,1)$\times$Imf(2,2,1)&
$a_1(x_1^2+x_1x_2+x_1x_3+x_1x_4+x_2^2+x_2x_3+x_2x_4+x_3^2+x_3x_4+x_4^2)$\\
&&$+a_2(x_5^2-x_5x_6+x_6^2) \ (a_1,a_2>0)$\\

(6,3193,1) & Imf(4,3,2)$\times$Imf(2,2,1) &
$a_1(2x_1^2-x_1x_2-x_1x_3-x_1x_4+2x_2^2-x_2x_3-x_2x_4+2x_3^2-x_3x_4+2x_4^2)$\\
&&$+a_2(x_5^2-x_5x_6+x_6^2) \ (a_1,a_2>0)$\\

(6,3272,5) & Imf(5,1,3)$\times$Imf(1,1,1) &
$a_1(4x_1^2+4x_1x_5+4x_2^2+4x_2x_5+4x_3^2+4x_3x_5+4x_4^2+4x_4x_5+5x_5^2)$\\
&&$+a_2x_6^2 \ (a_1,a_2>0)$\\

(6,3272,3) & Imf(5,1,2)$\times$Imf(1,1,1) &
$a_1(x_1^2+x_1x_2+x_2^2+x_2x_3+x_3^2+x_3x_4+x_3x_5+x_4^2+x_5^2)+a_2x_6^2 \ (a_1,a_2>0)$\\

(6,3296,5) & Imf(5,2,2)$\times$Imf(1,1,1)&
$a_1(x_1^2+x_1x_2+x_1x_3+x_1x_4+x_1x_5+x_2^2+x_2x_3+x_2x_4+x_2x_5$\\
&&$+x_3^2+x_3x_4+x_3x_5+x_4^2+x_4x_5+x_5^2)+a_2x_6^2 \ (a_1,a_2>0)$\\

(6,3296,1) & Imf(5,2,1)$\times$Imf(1,1,1) &
$a_1(5x_1^2-2x_1x_2-2x_1x_3-2x_1x_4-2x_1x_5+5x_2^2-2x_2x_3-2x_2x_4-2x_2x_5$\\
&&$+5x_3^2-2x_3x_4-2x_3x_5+5x_4^2-2x_4x_5+5x_5^2)+a_2x_6^2 \ (a_1,a_2>0)$\\

(6,3296,4) & Imf(5,2,4)$\times$Imf(1,1,1) &
$a_1(3x_1^2+2x_1x_2-2x_1x_3-2x_1x_4-2x_1x_5+3x_2^2+2x_2x_3-2x_2x_4+2x_2x_5$\\
&&$+3x_3^2+2x_3x_4+2x_3x_5+3x_4^2+2x_4x_5+3x_5^2)+a_2x_6^2 \ (a_1,a_2>0)$\\

(6,3296,2) & Imf(5,2,3)$\times$Imf(1,1,1) &
$a_1(2x_1^2+x_1x_2-2x_1x_3-2x_1x_4-2x_1x_5+2x_2^2+x_2x_3-2x_2x_4+x_2x_5$\\
&&$+2x_3^2+x_3x_4+x_3x_5+2x_4^2+x_4x_5+2x_5^2)+a_2x_6^2 \ (a_1,a_2>0)$\\

(6,2773,1) & Imf(6,1,1) &
$a_1(x_1^2+x_2^2+x_3^2+x_4^2+x_5^2+x_6^2) \ (a_1>0)$\\

(6,2773,2) & Imf(6,1,3) &
$a_1(2x_1^2+2x_1x_6+2x_2^2+2x_2x_6+2x_3^2+2x_3x_6+2x_4^2+2x_4x_6$\\
&&$+2x_5^2+2x_5x_6+3x_6^2) \ (a_1>0)$\\

(6,2773,3) & Imf(6,1,2) &
$a_1(x_1^2+x_1x_2+x_2^2+x_2x_3+x_3^2+x_3x_4+x_4^2+x_4x_5+x_4x_6+x_5^2+x_6^2) \ (a_1>0)$\\

(6,2772,5) & Imf(6,7,2) &
$a_1(3x_1^2-2x_1x_2-2x_1x_3+3x_2^2-2x_2x_3+3x_3^2+3x_4^2-2x_4x_5-2x_4x_6$\\
&&$+3x_5^2-2x_5x_6+3x_6^2) \ (a_1>0)$\\

(6,2772,2) & Imf(6,7,1) &
$a_1(x_1^2+x_1x_2+x_1x_3+x_2^2+x_2x_3+x_3^2+x_4^2+x_4x_5+x_4x_6$\\
&&$+x_5^2+x_5x_6+x_6^2) \ (a_1>0)$\\

(6,2750,4) & Imf(6,8,1) &
$a_1(3x_1^2+2x_1x_2+2x_1x_3+2x_1x_4+2x_1x_5+2x_1x_6+3x_2^2+2x_2x_3$\\
&&$+2x_2x_4+2x_2x_5+2x_2x_6+3x_3^2+2x_3x_4+2x_3x_5-2x_3x_6+3x_4^2$\\
&&$+2x_4x_5-2x_4x_6+3x_5^2+2x_5x_6+3x_6^2) \ (a_1>0)$\\

(6,2803,1) & Imf(6,2,1) &
$a_1(x_1^2-x_1x_2+x_2^2+x_3^2-x_3x_4+x_4^2+x_5^2-x_5x_6+x_6^2) \ (a_1>0)$\\

(6,2804,1) & Imf(6,3,2) &
$a_1(2x_1^2+x_1x_2-2x_1x_3-2x_1x_4+x_1x_5+x_1x_6+2x_2^2+x_2x_3-2x_2x_4$\\
&&$-2x_2x_5+x_2x_6+2x_3^2+x_3x_4-2x_3x_5-2x_3x_6$\\
&&$+2x_4^2+x_4x_5-2x_4x_6+2x_5^2+x_5x_6+2x_6^2) \ (a_1>0)$\\

(6,2804,2) & Imf(6,3,1) &
$a_1(x_1^2-x_1x_3+x_2^2-x_2x_4+x_3^2-x_3x_4+x_4^2-x_4x_5+x_5^2-x_5x_6+x_6^2) \ (a_1>0)$\\

(6,2866,3) & Imf(6,9,2) &
$a_1(3x_1^2-2x_1x_2-2x_1x_3-3x_1x_4+x_1x_5+x_1x_6+3x_2^2-2x_2x_3+x_2x_4$\\
&&$-3x_2x_5+x_2x_6+3x_3^2+x_3x_4+x_3x_5-3x_3x_6$\\
&&$+3x_4^2-2x_4x_5-2x_4x_6+3x_5^2-2x_5x_6+3x_6^2) \ (a_1>0)$\\

(6,2866,2) & Imf(6,9,1) &
$a_1(2x_1^2+2x_1x_2+2x_1x_3-2x_1x_4-x_1x_5-x_1x_6+2x_2^2+2x_2x_3-x_2x_4$\\
&&$-2x_2x_5-x_2x_6+2x_3^2-x_3x_4-x_3x_5-2x_3x_6+2x_4^2+2x_4x_5+2x_4x_6$\\
&&$+2x_5^2+2x_5x_6+2x_6^2) \ (a_1>0)$\\

(6,2932,2) & Imf(6,4,2) &
$a_1(x_1^2+x_1x_2+x_1x_3+x_1x_4+x_1x_5+x_1x_6+x_2^2+x_2x_3+x_2x_4+x_2x_5+x_2x_6$\\
&&$+x_3^2+x_3x_4+x_3x_5+x_3x_6+x_4^2+x_4x_5+x_4x_6+x_5^2+x_5x_6+x_6^2)\ (a_1>0)$\\

(6,2932,1) & Imf(6,4,1) &
$a_1(3x_1^2-x_1x_2-x_1x_3-x_1x_4-x_1x_5-x_1x_6+3x_2^2-x_2x_3-x_2x_4-x_2x_5$\\
&&$-x_2x_6+3x_3^2-x_3x_4-x_3x_5-x_3x_6+3x_4^2-x_4x_5-x_4x_6+3x_5^2-x_5x_6$\\
&&$+3x_6^2) \ (a_1>0)$\\

(6,2945,1) & Imf(6,5,1) &
$a_1(2x_1^2-x_1x_2-2x_1x_3+x_1x_4+x_1x_5-2x_1x_6+2x_2^2-x_2x_3-2x_2x_4+$\\
&&$x_2x_5+x_2x_6+2x_3^2-x_3x_4-2x_3x_5+x_3x_6+2x_4^2-x_4x_5-2x_4x_6+2x_5^2$\\
&&$-x_5x_6+2x_6^2) \ (a_1>0)$\\

(6,2952,1) & Imf(6,6,1) &
$a_1(3x_1^2-2x_1x_2-2x_1x_3+2x_1x_4+2x_1x_5+3x_2^2-2x_2x_3-2x_2x_4+2x_2x_6$\\
&&$+3x_3^2-2x_3x_5-2x_3x_6+3x_4^2-2x_4x_5+2x_4x_6+3x_5^2-2x_5x_6+3x_6^2) \ (a_1>0)$\\

(6,2952,2) & Imf(6,6,3) &
$a_1(5x_1^2+2x_1x_2-2x_1x_3-2x_1x_4+2x_1x_5+4x_1x_6+5x_2^2+2x_2x_3-2x_2x_4$\\
&&$-2x_2x_5+4x_2x_6+5x_3^2+2x_3x_4-2x_3x_5+4x_3x_6+5x_4^2+2x_4x_5+4x_4x_6$\\
&&$+5x_5^2+4x_5x_6+5x_6^2) \ (a_1>0)$\\

(6,2952,3) & Imf(6,6,2) &
$a_1(2x_1^2+x_1x_2+2x_1x_3+2x_1x_4+x_1x_5+2x_2^2+x_2x_3+2x_2x_4+2x_2x_5+x_2x_6$\\
&&$+2x_3^2+x_3x_4+2x_3x_5-x_3x_6+2x_4^2+x_4x_5-x_4x_6+2x_5^2+x_5x_6+2x_6^2) \ (a_1>0)$\\
\hline
\end{longtable}
}

\bigskip

%
%
\section{GAP algorithms: the flabby class $[M_G]^{fl}$}\label{seAlg}

Let $G$ be a finite subgroup of $\GL(n,\bZ)$ and 
$M=M_G$ be the $G$-lattice as in Definition \ref{defMG}. 
When $G\simeq {\rm Gal}(L/k)$ for a Galois extension 
$L/k$, the field $L(M_G)^G$ may be regarded as the function field of 
an algebraic $k$-torus $T$ which splits $L$ via (\ref{acts}) (see Section \ref{seInt}).

In this section, we provide some GAP algorithms for 
computing the flabby class $[M]^{fl}$ of $M$ 
and for verifying whether $[M]^{fl}$ is invertible and $[M]^{fl}=0$. 
The following flow chart shows the structure of the GAP algorithms: 
~\\
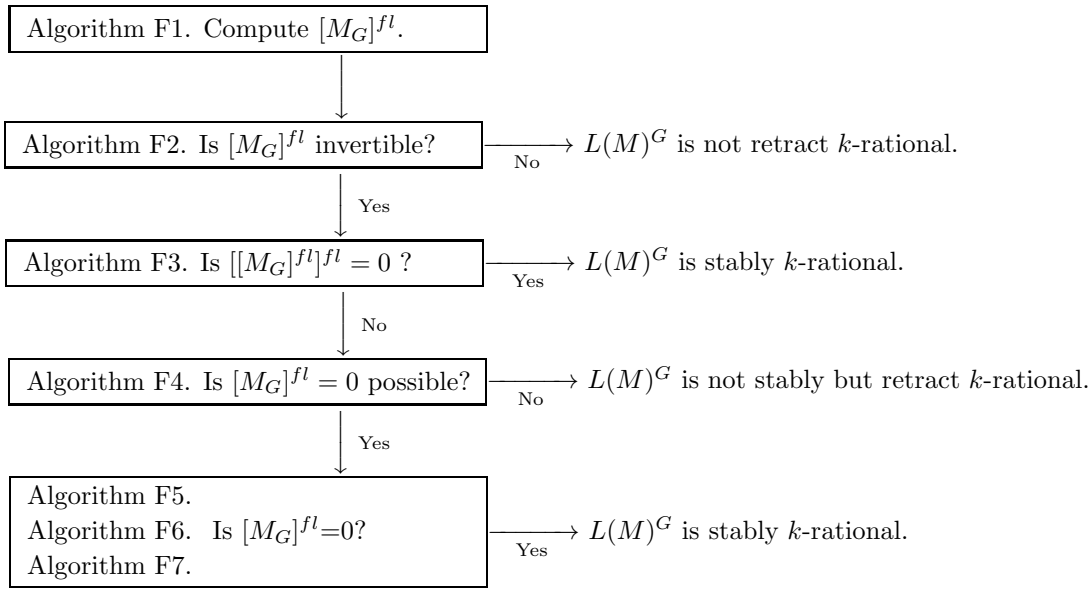
\begin{center}
\fbox{
\begin{minipage}{6cm}
Algorithm F1.
Compute $[M_G]^{fl}$.
\end{minipage}}\hspace*{80mm}

$\bigg\downarrow$ \hspace*{4mm}\hspace*{50mm}

\fbox{
\begin{minipage}{6cm}
Algorithm F2.
Is $[M_G]^{fl}$ invertible?
\end{minipage}}$\xrightarrow[\textrm{No}]{\hspace*{1cm}}$ 
$L(M)^G$ is not retract $k$-rational.\hspace*{18mm}

$\bigg\downarrow$ {\scriptsize Yes}\hspace*{50mm}

\fbox{
\begin{minipage}{6cm}
Algorithm F3.
Is $[[M_G]^{fl}]^{fl}=0$ ?
\end{minipage}}$\xrightarrow[\textrm{Yes}]{\hspace*{1cm}}$
$L(M)^G$ is stably $k$-rational.\hspace*{25mm}

$\bigg\downarrow$ {\scriptsize No}\hspace*{50mm}

\fbox{
\begin{minipage}{6cm}
Algorithm F4.
Is $[M_G]^{fl}=0$ possible?
\end{minipage}}$\xrightarrow[\textrm{No}]{\hspace*{1cm}}$
$L(M)^G$ is not stably but retract $k$-rational.

$\bigg\downarrow$ {\scriptsize Yes}\hspace*{50mm}

\fbox{
\begin{minipage}{6cm}
Algorithm F5.\ \\
Algorithm F6. \ 
Is $[M_G]^{fl}$=0?\\
Algorithm F7.\ 
\end{minipage}}$\xrightarrow[\textrm{Yes}]{\hspace*{1cm}}$
$L(M)^G$ is stably $k$-rational.\hspace*{24mm}
\end{center}

\bigskip

\setcounter{subsection}{-1}
\subsection{Determination whether $M_G$ is flabby (coflabby)}\label{ss50}~\\

Let $G$ be a finite subgroup of $\GL(n,\bZ)$ 
and $M=M_G$ be the corresponding $G$-lattice of rank $n$ as in Definition \ref{defMG}. 
We provide GAP some algorithms for computing $\widehat H^{-1}(G,M_G)$ and $H^1(G,M_G)$ 
and verifying whether the $G$-lattice $M_G$ is flabby (resp. coflabby).\\

\noindent
{\tt Hminus1(G)} returns the Tate cohomology group $\widehat H^{-1}(G,M_G)$.\\
{\tt H1(G)} returns the cohomology group $H^1(G,M_G)$.\\
{\tt IsFlabby(G)} returns whether $G$-lattice $M_G$ is flabby or not.\\
{\tt IsCoflabby(G)} returns whether $G$-lattice $M_G$ is coflabby or not.\\

Note that in the algorithms below 
we use the formulas $(\widehat B^{-1}(G,M) \otimes_\bZ \bQ) \cap M=\widehat Z^{-1}(G,M)$,
$(\widehat B^0(G,M) \otimes_\bZ \bQ) \cap M=\widehat Z^0(G,M)$ and 
$(B^1(G,M) \otimes_\bZ \bQ) \cap C^1(G,M)=Z^1(G,M)$ 
to compute $\widehat H^{-1}(G,M)$, $\widehat H^0(G,M)$ and $H^1(G,M)$ respectively.

\bigskip

\begin{algorithmF}[Determination whether $M_G$ is flabby (coflabby)]
{}~\\
\begin{verbatim}
Hminus1:= function(g)
    local m,gg,i,s,r;
    m:=[];
    gg:=GeneratorsOfGroup(g);
    if gg=[] then
        return [];
    else
        for i in gg do
            m:=Concatenation(m,i-Identity(g));
        od;
        s:=SmithNormalFormIntegerMat(m);
        r:=Rank(s);
        return List([1..r],x->s[x][x]);
    fi;
end;

H1:= function(g)
    local m,gg,i,s,r;
    m:=[];
    gg:=GeneratorsOfGroup(g);
    if gg=[] then
        return [];
    else
        for i in gg do
            m:=Concatenation(m,TransposedMat(i)-Identity(g));
        od;
        m:=TransposedMat(m);
        s:=SmithNormalFormIntegerMat(m);
        r:=Rank(s);
        return List([1..r],x->s[x][x]);
    fi;
end;

IsFlabby:= function(g)
    local h;
    h:=List(ConjugacyClassesSubgroups2(g),Representative);
    return ForAll(h,x->Product(Hminus1(x))=1);
end;

IsCoflabby:= function(g)
    local h;
    h:=List(ConjugacyClassesSubgroups2(g),Representative);
    return ForAll(h,x->Product(H1(x))=1);
end;
\end{verbatim}
\end{algorithmF}

\bigskip
\subsection{Construction of the flabby class $[M_G]^{fl}$ of the $G$-lattice $M_G$}
\label{ss51}~\\

Let $G$ be a finite subgroup of $\GL(n,\bZ)$ 
and $M=M_G$ be the corresponding $G$-lattice of rank $n$ as in Definition \ref{defMG}. 
We want to construct a flabby resolution
$0 \rightarrow M \rightarrow P \rightarrow F \rightarrow 0$
with $\rank F$ not too large.
If $0 \rightarrow Q \rightarrow R \rightarrow M^\circ \rightarrow 0$
is a coflabby resolution of $M^\circ=\Hom_\bZ(M,\bZ)$
where $R$ is permutation and $Q$ is coflasque,
we can get a flasque resolution by taking a dual exact sequence
$0 \rightarrow M \rightarrow \Hom_\bZ(R,\bZ) \rightarrow \Hom_\bZ(Q,\bZ)
\rightarrow 0$.

We construct a coflabby resolution of $M^\circ$ first. 
Let $P^\circ$ be a permutation $G$-lattice and $P^\circ \xrightarrow{f} M^\circ$ be a 
$G$-homomorphism.
For a subgroup $H$ of $G$, suppose that $f$ maps $(P^\circ)^H$ surjectively to 
$(M^\circ)^{H}$, i.e.
\begin{equation} \label{coflabby}
(P^\circ)^H \xrightarrow{f|_{(P^\circ)^H}} (M^\circ)^H \rightarrow 0 \text{ is exact.} 
\end{equation}
Then $\widehat H^0(H,P^\circ) \rightarrow \widehat H^0(H,M^\circ)$ is surjective,
so that $\widehat H^1(H,Q)=0$ where $Q=\ker f$ by Lemma \ref{longexact}.
In order to get a coflabby resolution of $M^\circ$,
it is enough to construct a permutation $G$-lattice $P^\circ$ and a $G$-homomorphism $f$
such that (\ref{coflabby}) is satisfied
for all subgroups $H$ of $G$.

Let $\mathcal{P}^\circ$ be a finite subset of $M^\circ$ 
which is closed under the action of $G$.
Let $P^\circ$ be a free $\bZ$-module over $\mathcal{P}^\circ$, i.e. 
$P^\circ=\bZ[\mathcal{P}^\circ]$. 
The $G$-lattice $P^\circ$ is permutation naturally, and for 
$p \in \mathcal{P}^\circ, \, P^\circ \ni p \mapsto p \in 
\mathcal{P}^\circ \subset M^\circ$
defines a $G$-homomorphism $f: P^\circ \rightarrow M^\circ$.
If $\mathcal{P}^\circ$ contains a $\bZ$-basis of $(M^\circ)^H$ for all subgroups $H$ of $G$, 
then the condition (\ref{coflabby}) is satisfied for all subgroups $H\subset G$.

The condition (\ref{coflabby}) is consistent with the conjugation.
Namely, (\ref{coflabby}) is satisfied for $H$ if and only if
it is satisfied for $H^g=g^{-1}Hg$.
Thus it is enough to consider only the subgroups $H$ not mutually conjugate.

Let $\mathcal{H}=\{H_i\}_{i \in I}$ be the set of all conjugacy classes of subgroups of $G$.
Let $\mathcal{Q}_i$ be a free $\bZ$-basis of $(M^\circ)^{H_i}$.
Then 
\begin{align*}
\mathcal{P}^\circ=\bigcup_{r \in\mathcal{R}} {\rm Orb}_G (r),\ 
\mathcal{R}=\bigcup_{i\in I} \mathcal{Q}_i
\end{align*}
provides $P^\circ=\bZ[\mathcal{P}^\circ]$ which satisfies (\ref{coflabby}) 
for all $H_i \subset G$. 
Therefore we may obtain a coflabby resolution of $M$:
\begin{align}
0\rightarrow {\rm Ker} f\rightarrow P^\circ\xrightarrow{f} M^\circ\rightarrow 0.
\label{cr}
\end{align}
The following algorithm (Algorithm \ref{Alg1}) tries to find some ``reduced'' subset 
$\mathcal{R}^{\prime} \subset \mathcal{R}$ such that 
$P^\circ=\bZ[\mathcal{P}^\circ]$ satisfies the condition (\ref{coflabby}) where 
\begin{align}
\mathcal{P}^\circ=\bigcup_{r \in\mathcal{R}^{\prime}} {\rm Orb}_G (r).\label{eqPP}
\end{align}
This trial is performed in a computer calculation as follows:\\

The following algorithms construct a ``reduced'' flabby resolution of $M_G$: 
$0\rightarrow M_G\xrightarrow{\iota} P\xrightarrow{\phi} F\rightarrow 0$.\\ 

\noindent
{\tt Z0lattice(G)} returns a $\bZ$-basis of $(M_G)^G$.
\hfill\break
{\tt FindCoflabbyResolutionBase(G,HH)} returns a $\bZ$-basis of a 
smaller permutation 
lattice $P^\circ$ satisfying (\ref{coflabby}) for any $H\in\mathcal{H}$.
\hfill\break
{\tt FlabbyResolution(G)} returns a flabby resolution of $M_G$ as follows: 
\hfill\break
{\tt FlabbyResolution(G).actionP} returns the matrix representation of 
the action of $G$ on $P$;
\hfill\break
{\tt FlabbyResolution(G).actionF} returns the matrix representation of 
the action of $G$ on $F$;
\hfill\break
{\tt FlabbyResolution(G).injection} returns the matrix which corresponds to 
the injection $\iota : M_G\rightarrow P$;
\hfill\break
{\tt FlabbyResolution(G).surjection} returns the matrix which corresponds to 
the surjection $\phi : P\rightarrow F$.

\bigskip

\begin{algorithmF}[{Construction of the flabby class $[M_G]^{fl}$ of the $G$-lattice $M_G$}]
\label{Alg1}
{}~\\
\begin{verbatim}
Z0lattice:= function(g)
    local gg,m,i;
    gg:=GeneratorsOfGroup(g);
    if gg=[] then
        return Identity(g);
    else
        m:=[];
        for i in gg do
            m:=Concatenation(m,TransposedMat(i)-Identity(g));
        od;
        m:=TransposedMat(m);
        return NullspaceIntMat(m);
    fi;
end;

ReduceCoflabbyResolutionBase:= function(g,hh,mi)
    local o,oo,z0,mi2;
    oo:=Orbits(g,mi);
    z0:=List(hh,Z0lattice);
    for o in oo do
       mi2:=Filtered(mi,x->(x in o)=fail);
       if ForAll([1..Length(hh)],i->LatticeBasis(List(Orbits(hh[i],mi2),Sum))=z0[i]) then
           mi:=mi2;
       fi;
    od;
    return mi;
end;

FindCoflabbyResolutionBase:= function(g,hh)
    local d,mi,h,z0,ll,i,o;
    d:=Length(Identity(g));
    mi:=[];
    for h in hh do
        z0:=Z0lattice(h);
        ll:=LatticeBasis(List(Orbits(h,mi),Sum));
        for i in z0 do
            if LatticeBasis(Concatenation(ll,[i]))<>ll then
                o:=Orbit(g,i);
                mi:=Concatenation(mi,o);
                o:=List(Orbits(h,o),Sum);
                ll:=LatticeBasis(Concatenation(ll,o));
            fi;
        od;
    od;
    return    ReduceCoflabbyResolutionBase(g,hh,mi);
end;

FlabbyResolution:= function(g)
    local tg,gg,d,th,mi,ms,o,r,gg1,gg2,v1,mg,img;
    tg:=TransposedMatrixGroup(g);
    gg:=GeneratorsOfGroup(tg);
    d:=Length(Identity(g));
    th:=List(ConjugacyClassesSubgroups2(tg),Representative);
    mi:=FindCoflabbyResolutionBase(tg,th);
    r:=Length(mi);
    o:=IdentityMat(r);
    gg1:=List(gg,x->PermutationMat(Permutation(x,mi),r));
    if r=d then
        return rec(injection:=TransposedMat(mi),
                   surjection:=NullMat(r,0),
                   actionP:=TransposedMatrixGroup(Group(gg1,o))
        );
    else
        ms:=NullspaceIntMat(mi);
        v1:=NullspaceIntMat(TransposedMat(ms));
        mg:=Concatenation(v1,ms);
        img:=mg^-1;
        gg2:=List(gg1,x->mg*x*img);
        gg2:=List(gg2,x->x{[d+1..r]}{[d+1..r]});
        return rec(injection:=TransposedMat(mi),
                   surjection:=TransposedMat(ms),
                   actionP:=TransposedMatrixGroup(Group(gg1)),
                   actionF:=TransposedMatrixGroup(Group(gg2))
        );
    fi;
end;
\end{verbatim}
\end{algorithmF}

\bigskip

\begin{example}[Algorithm \ref{Alg1}]
The $G$-lattice $M_G$ has the trivial flabby class $[M_G]^{fl}=0$ for 
$G=
{\tiny \left\langle 
\begin{pmatrix} 0\!\! &\!\!1 \\ -1\!\! &\!\! -1 \\ \end{pmatrix}
\right\rangle}\simeq C_3$.\\ 

\begin{verbatim}
Read("FlabbyResolution.gap");

gap> G:=Group([[[0,1],[-1,-1]]]);
Group([ [ [ 0, 1 ], [ -1, -1 ] ] ])
gap> FlabbyResolution(G);
rec( injection := [ [ 1, 0, -1 ], [ 0, -1, 1 ] ], 
  surjection := [ [ 1 ], [ 1 ], [ 1 ] ], 
  actionP := Group([ [ [ 0, 0, 1 ], [ 1, 0, 0 ], [ 0, 1, 0 ] ] ]), 
  actionF := Group([ [ [ 1 ] ] ]) )
gap> FlabbyResolution(G).actionF; # F is trivial of rank 1
Group([ [ [ 1 ] ] ]) )
\end{verbatim}
\end{example}

\bigskip

\subsection{Determination whether $[M_G]^{fl}$ is invertible}\label{ss52}~\\

Let $G$ be a finite subgroup of $\GL(n,\bZ)$ 
and $M=M_G$ be the corresponding $G$-lattice of rank $n$ as in Definition \ref{defMG}. 
We provide the algorithm (Algorithm F2) for the determination whether $[M_G]^{fl}$ is invertible. 
First, take a flabby resolution of $M$: 
\begin{align*}
0 \rightarrow M\rightarrow  P \rightarrow F \rightarrow 0.
\end{align*}
If $F$ is not coflabby, then $[M]^{fl}$ is not invertible. 
If $F$ is coflabby, then take a flabby resolution
\begin{align}
0 \rightarrow F \xrightarrow{\iota} Q \rightarrow E \rightarrow 0.\label{ex1}
\end{align}
By Lemma \ref{lemSL}, we find that 
$[M]^{fl}$ is invertible $\iff$ 
$F$ is invertible $\iff$ 
$E$ is invertible $\iff$ $(\ref{ex1})$ splits. 
Thus it remains to check whether (\ref{ex1}) splits, 
i.e. whether there exists $\pi : Q\rightarrow F$ such that $\pi\circ\iota=\id_F$. 

We divide the standard basis of $Q$ into $G$-orbits, 
and take a complete representative system $\{x_\lambda\}$ of the $G$-orbits.
Let $H_\lambda$ be the stabilizer of $x_\lambda$ in $G$. 
Then $\pi(x_\lambda)$ in $F^{H_\lambda}$.
Conversely if we fix an element of $F^H$ as $\pi(x_\lambda)$, 
then $\pi$ is determined by the representatives $\{x_\lambda\}$.

We may have the necessary and the sufficient condition for $\pi \circ \iota=\id_M$, 
and this becomes a system of linear equations. 
Hence this system of linear equations has an integer solution $\iff$ 
the section $\pi: Q\rightarrow F$ exists $\iff$ $[M]^{fl}$ is invertible.
This computation is performed in GAP as follows:\\

{\tt IsInvertibleF(G)} returns whether $[M_G]^{fl}$ is invertible.

\bigskip

\begin{algorithmF}[{Determination whether $[M_G]^{fl}$ is invertible}]\label{Alg2}
{}~\\
\begin{verbatim}
IsInvertibleF:= function(g)
    local tg,gg,d,th,mi,mi2,ms,z0,ll,h,r,i,j,k,gg1,gg2,g1,g2,oo,iso,ker,
      tg2,th2,h2,l1,l2,v1,m1,m2;
    tg:=TransposedMatrixGroup(g);
    gg:=GeneratorsOfGroup(tg);
    d:=Length(Identity(g));
    th:=List(ConjugacyClassesSubgroups2(tg),Representative);
    mi:=FindCoflabbyResolutionBase(tg,th);
    r:=Length(mi);
    gg1:=List(gg,x->PermutationMat(Permutation(x,mi),r));
    if r=d then
        return true;
    else
        ms:=NullspaceIntMat(mi);
        v1:=NullspaceIntMat(TransposedMat(ms));
        m1:=Concatenation(v1,ms);
        m2:=m1^-1;
        gg2:=List(gg1,x->m1*x*m2);
        gg2:=List(gg2,x->x{[d+1..r]}{[d+1..r]});
        tg2:=Group(gg2);
        iso:=GroupHomomorphismByImages(tg,tg2,gg,gg2);
        ker:=Kernel(iso);
        th2:=List(Filtered(th,x->IsSubgroup(x,ker)),x->Image(iso,x));
        for h in th2 do
            if Product(Hminus1(h))>1 then
                return false;
            fi;
        od;
        for h in th do
            if Product(Hminus1(h))>1 then
                d:=r-d;
                mi:=FindCoflabbyResolutionBase(tg2,th2);
                r:=Length(mi);
                gg1:=List(gg2,x->Permutation(x,mi));
                gg2:=List(gg2,TransposedMat);
                g2:=Group(gg2);
                gg1:=List(gg1,x->x^-1);
                g1:=Group(gg1);
                mi:=TransposedMat(mi);
                iso:=GroupHomomorphismByImagesNC(g1,g2,gg1,gg2);
                oo:=OrbitsDomain(g1,[1..r]);
                m1:=[];
                for i in oo do
                    h:=Stabilizer(g1,i[1]);
                    ll:=List(RightCosetsNC(g1,h),Representative);
                    l1:=List(ll,x->i[1]^x);
                    l2:=List(ll,x->Image(iso,x));
                    z0:=Z0lattice(Image(iso,h));
                    for j in z0 do
                        m2:=NullMat(r,d);
                        for k in [1..Length(ll)] do
                            m2[l1[k]]:=j*l2[k];
                        od;
                        Add(m1,Flat(mi*m2));
                    od;
                od;
                m2:=Concatenation(m1,[Flat(IdentityMat(d))]);
                m1:=NullspaceIntMat(m2);
                return Gcd(TransposedMat(m1)[Length(m2)])=1;
            fi;
        od;
        gg1:=List(gg,x->Permutation(x,mi));
        gg:=GeneratorsOfGroup(g);
        gg1:=List(gg1,x->x^-1);
        g1:=Group(gg1);
        mi:=TransposedMat(mi);
        iso:=GroupHomomorphismByImagesNC(g1,g,gg1,gg);
        oo:=OrbitsDomain(g1,[1..r]);
        m1:=[];
        for i in oo do
            h:=Stabilizer(g1,i[1]);
            ll:=List(RightCosetsNC(g1,h),Representative);
            l1:=List(ll,x->i[1]^x);
            l2:=List(ll,x->Image(iso,x));
            z0:=Z0lattice(Image(iso,h));
            for j in z0 do
                m2:=NullMat(r,d);
                for k in [1..Length(ll)] do
                    m2[l1[k]]:=j*l2[k];
                od;
                Add(m1,Flat(mi*m2));
            od;
        od;
        m2:=Concatenation(m1,[Flat(IdentityMat(d))]);
        m1:=NullspaceIntMat(m2);
        return Gcd(TransposedMat(m1)[Length(m2)])=1;
    fi;
end;
\end{verbatim}
\end{algorithmF}

\bigskip

\subsection{Computation of $E$ with $[[M_G]^{fl}]^{fl}=[E]$}\label{ss53}~\\

Let $G$ be a finite subgroup of $\GL(n,\bZ)$ 
and $M=M_G$ be the corresponding $G$-lattice of rank $n$ as in Definition \ref{defMG}. 
By a result of Section \ref{ss52} (Algorithm \ref{Alg2}), we assume that 
$[M]^{fl}$ is invertible. 
The next step is to determine whether $[M]^{fl}=0$. 
First, we give a sufficient condition for $[M]^{fl}=0$. 
Let $0 \rightarrow M \rightarrow P \rightarrow F \rightarrow 0$ be a flabby 
resolution of $M$. 
By Algorithm \ref{Alg2}, we can compute $F$, and by the assumption $F$ is invertible. 
If $F$ is still complicated, we take a flabby resolution of $F$: 
$0 \rightarrow F \rightarrow Q \rightarrow E \rightarrow 0$.  
Then $E=0\Longrightarrow F=0$ $\Longrightarrow$ $[M]^{fl}=0$. 
By the same way, we define $[M]^{{fl}^n}:=[[M]^{{fl}^{n-1}}]^{fl}$ inductively. 
Then we obtain a sufficient condition for $[M]^{fl}=0$, 
namely, $[M]^{{fl}^n}=0\Longrightarrow [M]^{fl}=0$. 
The following algorithm may compute $E$ with $[E]=[[M]^{fl}]^{fl}$ effectively. \\

\noindent
{\tt flfl(G)} returns the $G$-lattice $E$ with $[[M_G]^{fl}]^{fl}=[E]$.

\bigskip

\begin{algorithmF}[{Computation of $E$ with $[[M_G]^{fl}]^{fl}=[E]$}]\label{Alg3}
{}~\\
\begin{verbatim}
flfl:= function(g)
    local tg,gg,d,th,mi,ms,r,gg1,gg2,v1,mg,img,tg2,iso,ker;
    tg:=TransposedMatrixGroup(g);
    gg:=GeneratorsOfGroup(tg);
    d:=Length(Identity(g));
    th:=List(ConjugacyClassesSubgroups2(tg),Representative);
    mi:=FindCoflabbyResolutionBase(tg,th);
    r:=Length(mi);
    gg1:=List(gg,x->PermutationMat(Permutation(x,mi),r));
    if r=d then
        return [];
    else
        ms:=NullspaceIntMat(mi);
        v1:=NullspaceIntMat(TransposedMat(ms));
        mg:=Concatenation(v1,ms);
        img:=mg^-1;
        gg2:=List(gg1,x->mg*x*img);
        gg2:=List(gg2,x->x{[d+1..r]}{[d+1..r]});
        tg2:=Group(gg2);
        iso:=GroupHomomorphismByImages(tg,tg2,gg,gg2);
        ker:=Kernel(iso);
        th:=List(Filtered(th,x->IsSubset(x,ker)),x->Image(iso,x));
        tg:=tg2;
        gg:=gg2;
        d:=r-d;
        mi:=FindCoflabbyResolutionBase(tg,th);
        r:=Length(mi);
        gg1:=List(gg,x->PermutationMat(Permutation(x,mi),r));
        if r=d then
            return [];
        else
            ms:=NullspaceIntMat(mi);
            v1:=NullspaceIntMat(TransposedMat(ms));
            mg:=Concatenation(v1,ms);
            img:=mg^-1;
            gg2:=List(gg1,x->mg*x*img);
            gg2:=List(gg2,x->x{[d+1..r]}{[d+1..r]});
            return TransposedMatrixGroup(Group(gg2));
        fi;
    fi;
end;
\end{verbatim}
\end{algorithmF}

\bigskip

\begin{example}[Algorithm \ref{Alg3}]\label{ex3}
Let ${\rm Imf}(n,i,j)\leq \GL(n,\bZ)$ be the 
$j$-th $\bZ$-class of the $i$-th $\bQ$-class of the irreducible 
maximal finite group of dimension $n$ (which corresponds to 
the GAP command {\tt ImfMatrixGroup(n,i,j)}).  
Note that the maximal irreducible finite groups ${\rm Imf}(n,i,j)$
coincide with the maximal indecomposable finite groups ${\rm Indmf}(n,i,j)$ when $n\leq 5$ 
(see Subsection \ref{ssIndmf}). 

By using the algorithm {\tt flfl}, 
we may confirm that $[M_G]^{fl}=0$ for $G={\rm Imf}(4,4,1)$ and ${\rm Imf}(5,1,1)$. 
By Lemma \ref{lemp3}, we obtain $[M_H]^{fl}=0$ for any subgroups $H$ of $G$. 
There are $193$ (resp. $953$) conjugacy classes of subgroups 
of ${\rm Imf}(4,4,1)$ (resp. ${\rm Imf}(5,1,1)$).

\bigskip

\begin{verbatim}
Read("crystcat.gap");
Read("caratnumber.gap");
Read("FlabbyResolution.gap");

gap> G:=ImfMatrixGroup(4,4,1);
ImfMatrixGroup(4,4,1)
gap> Size(G);
384
gap> CrystCatZClass(G);                                 
[ 4, 32, 21, 1 ]
gap> flfl(G);
[  ]
gap> Length(List(ConjugacyClassesSubgroups2(G),Representative)); # # of conjugacy classes
193

gap> G:=ImfMatrixGroup(5,1,1);                                 
ImfMatrixGroup(5,1,1)
gap> Size(G);
3840
gap> CaratZClass(G);
[ 5, 942, 1 ]
gap> flfl(G);
[  ]
gap> Length(List(ConjugacyClassesSubgroups2(G),Representative)); # # of conjugacy classes
953
\end{verbatim}
\end{example}

\bigskip

\begin{example}[Kunyavskii's birational classification of 
algebraic $k$-tori of dimension $3$ (Theorem \ref{thKu})]\label{exKun}
By Voskresenskii's theorem (Theorem \ref{thVo}), 
$[M_G]^{fl}=0$ for all $G$-lattices $M_G$ of rank $\leq 2$. 
Using the algorithm {\tt flfl}, 
we may confirm Kunyavskii's theorem (Theorem \ref{thKu}). 
There exist $39$ (resp. $34$) decomposable (resp. indecomposable) 
$G$-lattices $M_G$ of rank $3$ (see Example \ref{exKS1}). 

By Voskresenskii's theorem and Lemma \ref{lemp1}, 
$[M_G]^{fl}=0$ for $39$ decomposable $G$-lattices $M_G$. 

Using {\tt flfl} and by Lemma \ref{lemp3}, 
we may see that $[M_G]^{fl}=0$ for any subgroups $G$ of 
${\rm Imf(3,1,1)}\simeq C_2\times S_4$, the group $G_1\simeq D_4$ 
of the GAP ID $(3,4,5,2)$ and the group $G_2\simeq S_4$ 
of the GAP ID $(3,7,4,3)$. 
Namely, $L(M_G)^G$ is stably $k$-rational. 
Note that Kunyavskii's theorem claims not only the stably $k$-rationality 
but also the $k$-rationality, and we could not confirm the $k$-rationality 
by this method. 

There exist exactly $15$ groups which are not subgroups of {\rm Imf(3,1,1)}, 
$G_1$ and $G_2$ (see Table $2$ in Theorem \ref{thVo}). 
Indeed, using the algorithm {\tt IsInvertibleF}, we may confirm that $[M_G]^{fl}$ 
is not invertible for all the $15$ groups $G$. 
Namely, $L(M_G)^G$ is not retract $k$-rational. 

\bigskip

\begin{verbatim}
Read("crystcat.gap");
Read("FlabbyResolution.gap");

gap> G:=ImfMatrixGroup(3,1,1); # G=C2xS4
ImfMatrixGroup(3,1,1)
gap> CrystCatZClass(G);
[ 3, 7, 5, 1 ]
gap> flfl(G);
[  ]
gap> flfl(MatGroupZClass(3,4,5,2)); # G=D4
[  ]
gap> flfl(MatGroupZClass(3,7,4,3)); # G=S4
[  ]

gap> ld3:=LatticeDecompositions(3);;
gap> Partitions(3);
[ [ 1, 1, 1 ], [ 2, 1 ], [ 3 ] ]
gap> List(ld3,Length);
[ 8, 31, 34 ]
gap> ind3:=ld3[3];;
gap> imf311sub:=Set(ConjugacyClassesSubgroups2(G),
> x->CrystCatZClass(Representative(x)));;
gap> G3452sub:=Set(ConjugacyClassesSubgroups2(MatGroupZClass(3,4,5,2)),
> x->CrystCatZClass(Representative(x)));;
gap> G3743sub:=Set(ConjugacyClassesSubgroups2(MatGroupZClass(3,7,4,3)),
> x->CrystCatZClass(Representative(x)));;
gap> N3:=Difference(ind3,Union(imf311sub,G3452sub,G3743sub));
[ [ 3, 3, 1, 3 ], [ 3, 3, 3, 3 ], [ 3, 3, 3, 4 ], [ 3, 4, 3, 2 ], [ 3, 4, 4, 2 ], 
  [ 3, 4, 6, 3 ], [ 3, 4, 7, 2 ], [ 3, 7, 1, 2 ], [ 3, 7, 2, 2 ], [ 3, 7, 2, 3 ], 
  [ 3, 7, 3, 2 ], [ 3, 7, 3, 3 ], [ 3, 7, 4, 2 ], [ 3, 7, 5, 2 ], [ 3, 7, 5, 3 ] ]
gap> Length(N3);                                             
15
gap> Filtered(N3,x->IsInvertibleF(MatGroupZClass(x[1],x[2],x[3],x[4]))=true);
[  ]
\end{verbatim}
\end{example}

\bigskip

\subsection{Possibility for $[M_G]^{fl}=0$}\label{ss54}~\\

We will devise an algorithm which, in a favorable situation, 
will enable us to show that some $G$-lattices are not stably permutation. 

Let $G$ be a finite subgroup of $\GL(n,\bZ)$ 
and $M=M_G$ be the corresponding $G$-lattice of rank $n$ as in Definition \ref{defMG}. 
By a result of Section \ref{ss52} (Algorithm \ref{Alg2}), we assume that 
$[M]^{fl}$ is invertible. 

Each isomorphism class of irreducible permutation $G$-lattices 
corresponds to a conjugacy class of subgroup $H$ of $G$ by
$H \leftrightarrow \bZ[G/H]$. 
Let $H_1,\dots,H_r$ be conjugacy classes of subgroups of $G$ 
whose ordering corresponds to the GAP function {\tt ConjugacyClassesSubgroups2(G)} 
(see Section \ref{seKSfail}). 
Let $F$ be the flabby class of $M_G$. 

We assume that $F$ is stably permutation, i.e. for $x_{r+1}=\pm 1$, 
\begin{align*}
\left(\bigoplus_{i=1}^r \bZ[G/H_i]^{\oplus x_i}\right)\oplus F^{\oplus x_{r+1}}
\ \simeq\ \bigoplus_{i=1}^r \bZ[G/H_i]^{\oplus y_i}.
\end{align*}
Define $a_i=x_i-y_i$ and $b_1=x_{r+1}$. Then we have for $b_1=\pm 1$, 
\begin{align}
\bigoplus_{i=1}^r \bZ[G/H_i]^{\oplus a_i}\ \simeq\ F^{\oplus(-b_1)}.\label{eqpos2}
\end{align}

By computing some $\bZ$-class invariants, 
we will give a necessary condition for $[M_G]^{fl}=0$. 

Let $\{c_1,\ldots,c_r\}$ be a set of complete representatives of 
the conjugacy classes of $G$. 
Let $A_i(c_j)$ be the matrix representation of the factor coset 
action of $c_j\in G$ on $\bZ[G/H_i]$ and 
$B(c_j)$ be the matrix representation of the action of $c_j\in G$ on $F$. 
By (\ref{eqpos2}), for each $c_j\in G$, we have 
\begin{align}
\sum_{i=1}^r a_i\, {\rm tr} A_i(c_j)+ b_1\, {\rm tr} B(c_j)=0\label{eqp1}
\end{align}
where tr\,$A$ is the trace of the matrix $A$. 
Similarly, we consider the rank of $H^0=\widehat Z^0$.  
For each $H_j$, we get
\begin{align}
\sum_{i=1}^r a_i\, \rank \widehat Z^0(H_j,\bZ[G/H_i])
+b_1\, \rank \widehat Z^0(H_j,F)=0.\label{eqp2}
\end{align}
Finally, we compute $\widehat H^0$.
Let $Sy_p(A)$ be a $p$-Sylow subgroup of an abelian group $A$.
$Sy_p(A)$ can be written as a direct product of cyclic groups uniquely.
Let $n_{p,e}(Sy_p(A))$ be the number of direct summands 
of cyclic groups of order $p^e$. 
For each $H_j,p,e$, we get 
\begin{align}
\sum_{i=1}^r a_i\, n_{p,e}(Sy_p(\widehat H^0(H_j,\bZ[G/H_i])))
+b_1\, n_{p,e}(Sy_p(\widehat H^0(H_j,F)))=0.\label{eqp3}
\end{align}
By the equalities (\ref{eqp1}), (\ref{eqp2}) and (\ref{eqp3}), we may get a 
system of linear equations in $a_1,\dots,a_r,b_1$ over $\bZ$.
Namely, we have that 
$[M_G]^{fl}=0\Longrightarrow$ there exist $a_1,\ldots,a_r\in\bZ$ and $b_1=\pm 1$ 
which satisfy (\ref{eqpos2}) $\Longrightarrow$ 
this system of linear equations has a integer solution in $a_1,\ldots,a_r$ with $b_1=\pm 1$. 
In particular, if this system of linear equations has no integer solutions, 
then we conclude that $[M_G]^{fl}\neq 0$. \\

\noindent 
{\tt H0(G)} returns the Tate cohomology group $\widehat H^0(G,M_G)$.\\
{\tt PossibilityOfStablyPermutationF(G)} returns a basis 
$\mathcal{L}=\{l_1,\dots,l_s\}$ of the solution space 
of the system of linear equations which is obtained by the equalities 
(\ref{eqp1}), (\ref{eqp2}) and (\ref{eqp3}).\\
{\tt PossibilityOfStablyPermutationM(G)} returns the same as 
{\tt PossibilityOfStablyPermutationF(G)} but with respect to $M_G$ 
instead of $F$. (We will use this in Section \ref{seFC}.)

\bigskip

\begin{algorithmF}[{Possibility for $[M_G]^{fl}=0$}]\label{Alg4}
{}~\\
\begin{verbatim}
H0:= function(g)
    local m,s,r;
    m:=Sum(g);
    s:=SmithNormalFormIntegerMat(m);
    r:=Rank(s);
    return List([1..r],x->s[x][x]);
end;

PossibilityOfStablyPermutationH:= function(g,hh)
    local gg,hg,hgg,hom,c,h,m,m1,m2,v,h0,og,oh,p,e;
    gg:=GeneratorsOfGroup(g);
    hg:=List(hh,x->RightCosets(g,x));
    hgg:=List(hg,x->List(gg,y->Permutation(y,x,OnRight)));
    hom:=List(hgg,x->GroupHomomorphismByImages(g,Group(x,()),gg,x));
    c:=List(ConjugacyClasses(g),Representative);
    og:=Order(g);
    m:=List(c,x->List([1..Length(hh)],y->og/Order(hh[y])-NrMovedPoints(Image(hom[y],x))));
    v:=List(c,Trace);
    for h in hh do
        h0:=H0(h);
        oh:=Order(h);
        m1:=List([1..Length(hh)],
          x->List(OrbitLengths(Image(hom[x],h),[1..og/Order(hh[x])]),y->oh/y));
        Add(m,List(m1,Length));
        Add(v,Length(h0));
        if oh>1 then
            for p in Set(FactorsInt(oh)) do
                for e in [1..PadicValuation(oh,p)] do
                    Add(m,List(m1,x->Number(x,y->PadicValuation(y,p)=e)));
                    Add(v,Number(h0,x->PadicValuation(x,p)=e));
                od;
            od;
        fi;
    od;
    m:=TransposedMat(m);
    m:=Concatenation(m,[v]);
    return NullspaceIntMat(m);
end;

PossibilityOfStablyPermutationF:= function(g)
    local tg,gg,d,th,mi,ms,r,gg1,gg2,g2,iso,ker,tg2,th2,h2,m1,m2,v;
    tg:=TransposedMatrixGroup(g);
    gg:=GeneratorsOfGroup(tg);
    d:=Length(Identity(g));
    th:=List(ConjugacyClassesSubgroups2(tg),Representative);
    mi:=FindCoflabbyResolutionBase(tg,th);
    r:=Length(mi);
    gg1:=List(gg,x->PermutationMat(Permutation(x,mi),r));
    if r=d then
        return [0,1];
    else
        ms:=NullspaceIntMat(mi);
        v:=NullspaceIntMat(TransposedMat(ms));
        m1:=Concatenation(v,ms);
        m2:=m1^-1;
        gg2:=List(gg1,x->m1*x*m2);
        gg2:=List(gg2,x->x{[d+1..r]}{[d+1..r]});
        tg2:=Group(gg2);
        iso:=GroupHomomorphismByImages(tg,tg2,gg,gg2);
        ker:=Kernel(iso);
        th2:=List(Filtered(th,x->IsSubgroup(x,ker)),x->Image(iso,x));
        g2:=TransposedMatrixGroup(tg2);
        h2:=List(th2,TransposedMatrixGroup);
        return PossibilityOfStablyPermutationH(g2,h2);
    fi;
end;

PossibilityOfStablyPermutationM:= function(g)
    local hh;
    hh:=List(ConjugacyClassesSubgroups2(g),Representative);
    return PossibilityOfStablyPermutationH(g,hh);
end;
\end{verbatim}
\end{algorithmF}

\bigskip

\begin{example}[{Algorithm \ref{Alg4}: Possibility for $[M_G]^{fl}=0$}]
\label{exal4} 
Let $G\simeq F_{20}$ be the Frobenius group of order $20$ 
of the GAP ID $(4,31,1,3)$ or $(4,31,1,4)$ as in Table $2$. 
By Algorithm \ref{Alg2}, we may check that $[M_G]^{fl}$ is invertible. 
We will show that $[M_G]^{fl}\neq 0$. 
There exist $6$ conjugacy classes of subgroups
$\{1\}$, $H_2\simeq C_2$, $H_3\simeq C_4$, $H_4\simeq C_5$, $H_5\simeq D_5$ 
and $G$ of 
$G$ of order $1,2,4,5,10$ and $20$ respectively.
The corresponding permutation $G$-lattices 
are $\bZ[G]$, $\bZ[G/H_2]$, $\bZ[G/H_3]$, $\bZ[G/H_4]$, 
$\bZ[G/H_5]$ and $\bZ$ of rank $20$, $10$, $5$, $4$, $2$ and $1$ 
(this ordering is determined by the GAP function
{\tt ConjugacyClassesSubgroups2(G)} (see Section \ref{seKSfail})).
By Algorithm \ref{Alg1}, we may obtain the flabby $G$-lattice $F$ 
of rank $16$ via the function {\tt FlabbyResolution(G).actionF} where 
$0\rightarrow M_G\rightarrow P\rightarrow F$ 
is a flabby resolution of $M_G$. 

Let $\mathcal{L}=\{l_1,\dots,l_s\}$ be a list of lists obtained by
the GAP function 
{\tt PossibilityOfStablyPermutationF(G)}. 
Put $\mathcal{U}=\bZ l_1+\dots+\bZ l_s$
i.e. the set of all linear combinations of $\mathcal{L}$ over $\bZ$.
When $[a_1,\dots,a_r,b_1] \in \mathcal{U}$,
there is a possibility that
$$\bigoplus_{i=1}^r \bZ[G/H_i]^{\oplus a_i}
\simeq F^{\oplus (-b_1)}.$$
By $\mathcal{U}=\langle[1,1,0,1,-1,0,-2]\rangle$, there is a possibility that
\[
\bZ[G] \oplus \bZ[G/H_2] \oplus \bZ[G/H_4] \simeq \bZ[G/H_5] \oplus F^{\oplus 2}.
\]
However, $b_1$ should be $\pm 1$. 
This implies that $F$ is not stably permutation, hence $[M_G]^{fl}\neq 0$. 
In other words, $L(M_G)^G$ is not stably but retract $k$-rational. 

By $H^2(G,\bZ[G])=0$, $H^2(G,\bZ[G/H_2])=\bZ/2\bZ$ 
and $H^2(G,\bZ[G/H_4])=\bZ/5\bZ$ while 
$H^2(G,\bZ[G/H_5])=\bZ/2\bZ$ and $H^2(G,F)=0$, 
we also see that $F^{\oplus k}$ is not stably permutation for all $k\geq 1$.

\bigskip

\begin{verbatim}
Read("FlabbyResolution.gap");

gap> G:=MatGroupZClass(4,31,1,3);; # G=F20
gap> Rank(FlabbyResolution(G).actionF.1); # F is of rank 16
16
gap> IsInvertibleF(G);
true
gap> List(ConjugacyClassesSubgroups2(G),x->StructureDescription(Representative(x)));
[ "1", "C2", "C4", "C5", "D10", "C5 : C4" ]
gap> PossibilityOfStablyPermutationF(G); # checking [M]^{fl}: non-zero
[ [ 1, 1, 0, 1, -1, 0, -2 ] ]

gap> G:=MatGroupZClass(4,31,1,4);; # G=F20
gap> Rank(FlabbyResolution(G).actionF.1); # F is of rank 16
16
gap> IsInvertibleF(G);                  
true
gap> List(ConjugacyClassesSubgroups2(G),x->StructureDescription(Representative(x)));
[ "1", "C2", "C4", "C5", "D10", "C5 : C4" ]
gap> PossibilityOfStablyPermutationF(G); # checking [M]^{fl}: non-zero
[ [ 1, 1, 0, 1, -1, 0, -2 ] ]
\end{verbatim}
\end{example}

\bigskip

\subsection{Verification of $[M_G]^{fl}=0$: Method I}\label{ss55}~\\

Let $G$ be a finite subgroup of $\GL(n,\bZ)$ 
and $M=M_G$ be the corresponding $G$-lattice of rank $n$ as in Definition \ref{defMG}. 
Next we try to check whether the possibility of the isomorphism as 
in (\ref{eqpos2}) actually holds. 
The following algorithm tries to find the isomorphism 
using the GAP function {\tt RepresentativeAction(GL(n,Integers),G1,G2)}.\\

\noindent
{\tt Nlist(l)} returns the negative part of the list $l$.\\
{\tt Plist(l)} returns the positive part of the list $l$.\\
{\tt StablyPermutationFCheck(G,L1,L2)} returns 
the matrix $P$ which satisfies $G_1P=PG_2$ 
where $G_1$ (resp. $G_2$) is the matrix representation group of the action of $G$ 
on $(\oplus_{i=1}^r \bZ[G/H_i]^{\oplus a_i})\oplus F^{\oplus b_1}$ 
(resp. $(\oplus_{i=1}^r \bZ[G/H_i]^{\oplus a_i^{\prime}})\oplus F^{\oplus b_1^{\prime}}$) 
with the isomorphism 
\begin{align}
\left(\bigoplus_{i=1}^r \bZ[G/H_i]^{\oplus a_i}\right)\oplus F^{\oplus b_1}
\simeq
\left(\bigoplus_{i=1}^r \bZ[G/H_i]^{\oplus a_i^{\prime}}\right)\oplus F^{\oplus b_1^{\prime}}
\label{eqisom}
\end{align}
for lists $L_1=[a_1,\ldots,a_r,b_1]$ and 
$L_2=[a_1^{\prime},\ldots,a_r^{\prime},b_1^{\prime}]$, if $P$ exists. 
If such $P$ does not exist, this returns {\tt false}.\\
{\tt StablyPermutationMCheck(G,L1,L2)} returns 
the same as {\tt StablyPermutationFCheck(G,L1,L2)} 
but with respect to $M_G$ 
instead of $F$. (We will use this in Section \ref{seFC}.)\\

If the rank of $F$ is small enough, 
{\tt StablyPermutationFCheck} works well. 
However, if the rank of $F$ is not small, 
{\tt StablyPermutationFCheck} does not return the answer in a suitable time, 
and then we need to make more efforts. 
We will explain this in the next section (Section \ref{ss56}). 

\bigskip

\begin{algorithmF}[{Verification of $[M_G]^{fl}=0$: Method I}]\label{Alg5}
\hfill\break
$($The following algorithm needs the CARAT package, in particular 
for the command {\tt RepresentativeAction}.$)$\\

\begin{verbatim}
Nlist:= function(l)
    return List(l,x->Maximum([-x,0]));
end;

Plist:= function(l)
    return List(l,x->Maximum([x,0]));
end;

CosetRepresentation:= function(g,h)
    local gg,hg,og;
    gg:=GeneratorsOfGroup(g);
    hg:=SortedList(RightCosets(g,h));
    og:=Length(hg);
    return List(gg,x->PermutationMat(Permutation(x,hg,OnRight),og));
end;

StablyPermutationCheckH:= function(g,hh,c1,c2)
    local gg,g1,g2,dx,m,i,j,d;
    gg:=List(hh,x->CosetRepresentation(g,x));
    Add(gg,GeneratorsOfGroup(g));
    g1:=[];
    g2:=[];
    for i in [1..Length(gg[1])] do
        m:=[];
        for j in [1..Length(gg)] do
            m:=Concatenation(m,List([1..c1[j]],x->gg[j][i]));
        od;
        Add(g1,DirectSumMat(m));
        m:=[];
        for j in [1..Length(gg)] do
            m:=Concatenation(m,List([1..c2[j]],x->gg[j][i]));
        od;
        Add(g2,DirectSumMat(m));
    od;
    d:=Length(g1[1]);
    if d<>Length(g2[1]) then
        return fail;
    else
        return RepresentativeAction(GL(d,Integers),Group(g1),Group(g2));
    fi;
end;

StablyPermutationMCheck:= function(g,c1,c2)
    local h;
    h:=List(ConjugacyClassesSubgroups2(g),Representative);
    return StablyPermutationCheckH(g,h,c1,c2);
end;

StablyPermutationFCheck:= function(g,c1,c2)
    local tg,gg,d,th,mi,ms,o,h,r,gg1,gg2,g2,iso,ker,tg2,th2,h2,m1,m2,v;
    tg:=TransposedMatrixGroup(g);
    gg:=GeneratorsOfGroup(tg);
    d:=Length(Identity(g));
    th:=List(ConjugacyClassesSubgroups2(tg),Representative);
    mi:=FindCoflabbyResolutionBase(tg,th);
    r:=Length(mi);
    gg1:=List(gg,x->PermutationMat(Permutation(x,mi),r));
    if r=d then
        return true;
    else
        ms:=NullspaceIntMat(mi);
        v:=NullspaceIntMat(TransposedMat(ms));
        m1:=Concatenation(v,ms);
        m2:=m1^-1;
        gg2:=List(gg1,x->m1*x*m2);
        gg2:=List(gg2,x->x{[d+1..r]}{[d+1..r]});
        tg2:=Group(gg2);
        iso:=GroupHomomorphismByImages(tg,tg2,gg,gg2);
        ker:=Kernel(iso);
        th2:=List(Filtered(th,x->IsSubgroup(x,ker)),x->Image(iso,x));
        g2:=TransposedMatrixGroup(tg2);
        h2:=List(th2,TransposedMatrixGroup);
        return StablyPermutationCheckH(g2,h2,c1,c2);
    fi;
end;
\end{verbatim}
\end{algorithmF}

\bigskip

\begin{example}[Algorithm \ref{Alg5}: Method I (1)]\label{ex5m11}
Let $G={\rm Imf}(4,3,1)\simeq S_5\times C_2$ be the group of order $240$ 
of the GAP ID $(4,31,7,1)$. 
By Algorithm \ref{Alg1}, 
the rank of the flabby class $F$ of $G$ is $6$. 
By Algorithm \ref{Alg5}, 
the following possibility of the isomorphism exists: 
$\bZ[G/H_{52}]\oplus\bZ[G/H_{54}]\simeq\bZ\oplus F$ 
on which the matrix representation groups of the 
action of $G$ both sides are $G_1$ and $G_2$ respectively 
where $H_{52}\simeq S_4\times C_2$ and $H_{54}\simeq S_5$. 
We may confirm the isomorphism 
via {\tt StablyPermutationFCheck(G,Nlist(l),Plist(l))} which returns the 
matrix $P$ with $G_1P=PG_2$. 
This implies $[M_G]^{fl}=0$, and hence $L(M_G)^G$ is stably $k$-rational. 

\bigskip

\begin{verbatim}
Read("crystcat.gap");
Read("FlabbyResolution.gap");

gap> G:=ImfMatrixGroup(4,3,1); # G=C2xS5
ImfMatrixGroup(4,3,1)
gap> CrystCatZClass(G);
[ 4, 31, 7, 1 ]
gap> Rank(FlabbyResolution(G).actionF.1); # F is of rank 6
6
gap> ll:=PossibilityOfStablyPermutationF(G);;
gap> Length(ll);
18
gap> l:=ll[Length(ll)];                              
[ 0, 0, 0, 0, 0, 0, 0, 0, 0, 0, 0, 0, 0, 0, 0, 0, 0, 0, 0, 0, 0, 0, 0, 0, 0, 
  0, 0, 0, 0, 0, 0, 0, 0, 0, 0, 0, 0, 0, 0, 0, 0, 0, 0, 0, 0, 0, 0, 0, 0, 0, 
  0, 1, 0, 1, 0, 0, -1, -1 ]
gap> Length(l);
58
gap> ss:=List(ConjugacyClassesSubgroups2(G),x->StructureDescription(Representative(x)));
[ "1", "C2", "C2", "C2", "C2", "C2", "C3", "C2 x C2", "C2 x C2", "C2 x C2", "C4", 
  "C2 x C2", "C2 x C2", "C4", "C2 x C2", "C2 x C2", "C5", "C6", "S3", "C6", "C6", 
  "S3", "S3", "S3", "C2 x C2 x C2", "D8", "C4 x C2", "C2 x C2 x C2", "D8", "D8", "D8", 
  "D10", "C10", "D10", "A4", "C6 x C2", "D12", "D12", "D12", "D12", "D12", "D12", 
  "C2 x D8", "D20", "C5 : C4", "C5 : C4", "C2 x A4", "S4", "S4", "C2 x C2 x S3", 
  "C2 x (C5 : C4)", "C2 x S4", "A5", "S5", "C2 x A5", "S5", "C2 x S5" ]
gap> Length(ss);
57
gap> Nlist(l);
[ 0, 0, 0, 0, 0, 0, 0, 0, 0, 0, 0, 0, 0, 0, 0, 0, 0, 0, 0, 0, 0, 0, 0, 0, 0, 
  0, 0, 0, 0, 0, 0, 0, 0, 0, 0, 0, 0, 0, 0, 0, 0, 0, 0, 0, 0, 0, 0, 0, 0, 0, 
  0, 0, 0, 0, 0, 0, 1, 1 ]
gap> Plist(l);
[ 0, 0, 0, 0, 0, 0, 0, 0, 0, 0, 0, 0, 0, 0, 0, 0, 0, 0, 0, 0, 0, 0, 0, 0, 0, 
  0, 0, 0, 0, 0, 0, 0, 0, 0, 0, 0, 0, 0, 0, 0, 0, 0, 0, 0, 0, 0, 0, 0, 0, 0, 
  0, 1, 0, 1, 0, 0, 0, 0 ]
gap> [ss[52], ss[54], ss[57]];       
[ "C2 x S4", "S5", "C2 x S5" ]
gap> P:=StablyPermutationFCheck(G,Nlist(l),Plist(l));
[ [ 1, 1, 1, 1, 1, 2, 2 ], 
  [ 0, 0, 0, -1, 0, 0, -1 ], 
  [ 0, 0, 0, 0, -1, -1, 0 ], 
  [ 0, 0, 0, 0, -1, 0, -1 ], 
  [ -1, 0, 0, 0, 0, 0, -1 ], 
  [ 0, -1, 0, 0, 0, -1, 0 ], 
  [ 0, 0, -1, 0, 0, 0, -1 ] ]
\end{verbatim}
\end{example}

\bigskip

\begin{example}[Algorithm \ref{Alg5}: Method I (2)]\label{ex52}
It may be needed to add some more $G$-lattices 
to make the both hand side of $G$-lattices isomorphic 
before applying {\tt StablePermutationFCheck} as in Method I. 
The following example shows that 
$F\oplus \bZ\simeq \bZ[S_5/S_4]\oplus \bZ$ 
for the group $G\simeq S_5$ of the CARAT ID $(5,946,2)$ 
and $F$ is of rank $5$ which satisfies $[F]=[M_G]^{fl}$ 
although $F\not\simeq \bZ[S_5/S_4]$. 

\bigskip

\begin{verbatim}
Read("caratnumber.gap");
Read("FlabbyResolution.gap");

gap> G:=CaratMatGroupZClass(5,946,2);; # G=S5
gap> Rank(FlabbyResolution(G).actionF.1); # F is of rank 5
5
gap> CaratZClass(FlabbyResolution(G).actionF);  
[ 5, 911, 4 ]
gap> ll:=PossibilityOfStablyPermutationF(G);                 
[ [ 1, 0, 0, -1, 0, 0, -4, 0, -2, 1, 2, 0, -1, -1, 0, 4, 0, 1, -4, 4 ], 
  [ 0, 1, 0, 0, 0, -1, -1, 0, -1, 0, 0, 0, 0, 0, 1, 1, 0, 0, -1, 1 ], 
  [ 0, 0, 1, 0, 0, 0, -2, 0, -1, 0, 1, 0, -1, -1, 0, 2, 0, 1, -2, 2 ], 
  [ 0, 0, 0, 0, 1, 2, -2, 0, -2, 1, 2, -2, -1, -2, -2, 2, 0, 1, -2, 4 ], 
  [ 0, 0, 0, 0, 0, 0, 0, 0, 0, 0, 0, 0, 0, 0, 0, 0, 1, 0, 0, -1 ] ]
gap> l:=ll[Length(ll)];
[ 0, 0, 0, 0, 0, 0, 0, 0, 0, 0, 0, 0, 0, 0, 0, 0, 1, 0, 0, -1 ]
gap> Length(l);
20
gap> ss:=List(ConjugacyClassesSubgroups2(G),x->StructureDescription(Representative(x)));
[ "1", "C2", "C2", "C3", "C2 x C2", "C2 x C2", "C4", "C5", "C6", 
  "S3", "S3", "D8", "D10", "A4", "D12", "C5 : C4", "S4", "A5", "S5" ]
gap> ss[17];
"S4"
gap> StablyPermutationFCheck(G,Nlist(l),Plist(l));
fail
gap> l2:=IdentityMat(Length(l))[Length(l)-1];
[ 0, 0, 0, 0, 0, 0, 0, 0, 0, 0, 0, 0, 0, 0, 0, 0, 0, 0, 1, 0 ]
gap> StablyPermutationFCheck(G,Nlist(l)+l2,Plist(l)+l2);
[ [ 2, 2, 2, 2, 2, 3 ], 
  [ 0, 0, -1, -1, -1, -1 ], 
  [ -1, -1, -1, 0, 0, -1 ], 
  [ -1, 0, 0, -1, -1, -1 ], 
  [ 0, -1, -1, 0, -1, -1 ], 
  [ 0, -1, 0, -1, -1, -1 ] ]
\end{verbatim}
\end{example}

\bigskip

\subsection{Verification of $[M_G]^{fl}=0$: Method II}\label{ss56}~\\

Let $G$ be a finite subgroup of $\GL(n,\bZ)$ and $M=M_G$ be 
the corresponding $G$-lattice of rank $n$ as in Definition \ref{defMG}. 
The function 
{\tt RepresentativeAction(GL($n$,Integers),$G_1$,$G_2$)} 
in Algorithm \ref{Alg5} may not work well when $n$ is not small. 
We will provide another method in order to confirm the isomorphism (\ref{eqisom}), 
i.e. 
\setcounter{equation}{9}
\begin{align}
\left(\bigoplus_{i=1}^r \bZ[G/H_i]^{\oplus a_i}\right)\oplus F^{\oplus b_1}
\simeq
\left(\bigoplus_{i=1}^r \bZ[G/H_i]^{\oplus a_i^{\prime}}\right)\oplus F^{\oplus b_1^{\prime}}, 
\label{eqiso2}
\end{align} 
although it is needed by trial and error.

Our aim is to find the matrix $P$ which satisfies $G_1 P=P G_2$ rapidly.
If we can choose a matrix with determinant det $P=\pm 1$, 
$G_1$ and $G_2$ are $\GL(n,\bZ)$-conjugate, 
and hence the isomorphism (\ref{eqisom}) established. 
This implies that the flabby class $[M_G]^{fl}=0$.
\hfill\break

\noindent
{\tt StablyPermutationFCheckP(G,L1,L2)} returns 
a basis $\mathcal{P}=\{P_1,\dots,P_m\}$ 
of the solution space of $G_1P=PG_2$ 
where $G_1$ (resp. $G_2$) is the matrix representation group of the action of $G$ 
on $(\oplus_{i=1}^r \bZ[G/H_i]^{\oplus a_i})\oplus F^{\oplus b_1}$ (resp. 
$(\oplus_{i=1}^r \bZ[G/H_i]^{\oplus a_i^{\prime}})\oplus F^{\oplus b_1^{\prime}}$) 
with the isomorphism (\ref{eqisom}) 
for lists $L_1=[a_1,\ldots,a_r,b_1]$ and 
$L_2=[a_1^{\prime},\ldots,a_r^{\prime},b_1^{\prime}]$, if $P$ exists. 
If such $P$ does not exist, this returns {\tt [ ]}.\\
{\tt StablyPermutationMCheckP(G,L1,L2)} 
returns the same as 
{\tt StablyPermutationFCheckP(G,L1,L2)}  
but with respect to $M_G$ 
instead of $F$. (We will use these in Section \ref{seFC}.)\\

\noindent
{\tt StablyPermutationFCheckMat(G,L1,L2,P)} returns 
{\tt true} if $G_1P=PG_2$ and det $P=\pm 1$ where $G_1$ (resp. $G_2$) 
is the matrix representation group of the action of $G$ 
on $(\oplus_{i=1}^r \bZ[G/H_i]^{\oplus a_i})\oplus F^{\oplus b_1}$ (resp. 
$(\oplus_{i=1}^r \bZ[G/H_i]^{\oplus a_i^{\prime}})\oplus F^{\oplus b_1^{\prime}}$) 
with the isomorphism (\ref{eqiso2})
for lists $L_1=[a_1,\ldots,a_r,b_1]$ and 
$L_2=[a_1^{\prime},\ldots,a_r^{\prime},b_1^{\prime}]$. 
If not, this returns {\tt false}.\\
{\tt StablyPermutationMCheckMat(G,L1,L2,P)} 
returns the same as 
{\tt StablyPermutationFCheckMat(G,L1,L2,P)} 
but with respect to $M_G$ 
instead of $F$. (We will use these in Section \ref{seFC}.)\\

\noindent
{\tt StablyPermutationFCheckGen(G,L1,L2)} returns 
the list $[\mathcal{M}_1,\mathcal{M}_2]$ where 
$\mathcal{M}_1=[g_1,\ldots,g_t]$ (resp. $\mathcal{M}_2=[g_1^{\prime},\ldots,g_t^{\prime}]$) 
is a list of the generators of $G_1$ (resp. $G_2$) 
which is the matrix representation group of the action of $G$ 
on $(\oplus_{i=1}^r \bZ[G/H_i]^{\oplus a_i})\oplus F^{\oplus b_1}$ (resp. 
$(\oplus_{i=1}^r \bZ[G/H_i]^{\oplus a_i^{\prime}})\oplus F^{\oplus b_1^{\prime}}$) 
with the isomorphism (\ref{eqiso2})
for lists $L_1=[a_1,\ldots,a_r,b_1]$ and 
$L_2=[a_1^{\prime},\ldots,a_r^{\prime},b_1^{\prime}]$.\\
{\tt StablyPermutationMCheckGen(G,L1,L2)} 
returns the same as  
{\tt StablyPermutationFCheckGen(G,L1,L2)} 
but with respect to $M_G$ 
instead of $F$. (We will use these in Section \ref{seFC}.)

\bigskip

\begin{algorithmF}[{Verification of $[M_G]^{fl}=0$: Method II}]\label{Alg6}
{}~\\
\begin{verbatim}
TransformationMat:= function(l1,l2)
    local d1,d2,l,m,p,i,j;
    d1:=Length(l1[1]);
    d2:=Length(l2[1]);
    l:=Length(l1);
    m:=[];
    for i in [1..d1] do
        for j in [1..d2] do
            p:=NullMat(d1,d2);
            p[i][j]:=1;
            Add(m,Flat(List([1..l],x->l1[x]*p-p*l2[x])));
        od;
    od;
    p:=NullspaceIntMat(m);
    return List(p,x->List([1..d1],y->x{[(y-1)*d2+1..y*d2]}));
end;

StablyPermutationCheckHP:= function(g,hh,c1,c2)
    local gg,l,m,m1,m2,tm,d1,d2,s1,s2,i,j,k;
    gg:=List(hh,x->CosetRepresentation(g,x));
    Add(gg,GeneratorsOfGroup(g));
    l:=List([1..Length(gg)],x->Length(gg[x][1]));
    d1:=Sum([1..Length(gg)],x->c1[x]*l[x]);
    d2:=Sum([1..Length(gg)],x->c2[x]*l[x]);
    m:=[];
    s1:=0;
    for i in [1..Length(gg)] do
        if c1[i]>0 then
            m1:=[];
            s2:=0;
            for j in [1..Length(gg)] do
                if c2[j]>0 then
                    tm:=TransformationMat(gg[i],gg[j]);
                    for k in [1..c2[j]] do
                        m2:=List(tm,x->TransposedMat(Concatenation(
                          [NullMat(s2,l[i]),TransposedMat(x),
                          NullMat(d2-s2-l[j],l[i])])));
                        m1:=Concatenation(m1,m2);
                        s2:=s2+l[j];
                    od;
                fi;
            od;
            m1:=LatticeBasis(List(m1,Flat));
            m1:=List(m1,x->List([1..l[i]],y->x{[(y-1)*s2+1..y*s2]}));
            for k in [1..c1[i]] do
                m:=Concatenation(m,List(m1,x->Concatenation(
                  [NullMat(s1,d2),x,NullMat(d1-s1-l[i],d2)])));
                s1:=s1+l[i];
            od;
        fi;
    od;
    return m;
end;

StablyPermutationMCheckP:= function(g,c1,c2)
    local h;
    h:=List(ConjugacyClassesSubgroups2(g),Representative);
    return StablyPermutationCheckHP(g,h,c1,c2);
end;

StablyPermutationFCheckP:= function(g,c1,c2)
    local tg,gg,d,th,mi,ms,o,h,r,gg1,gg2,g2,iso,ker,tg2,th2,h2,m1,m2,v;
    tg:=TransposedMatrixGroup(g);
    gg:=GeneratorsOfGroup(tg);
    d:=Length(Identity(g));
    th:=List(ConjugacyClassesSubgroups2(tg),Representative);
    mi:=FindCoflabbyResolutionBase(tg,th);
    r:=Length(mi);
    gg1:=List(gg,x->PermutationMat(Permutation(x,mi),r));
    if r=d then
        return true;
    else
        ms:=NullspaceIntMat(mi);
        v:=NullspaceIntMat(TransposedMat(ms));
        m1:=Concatenation(v,ms);
        m2:=m1^-1;
        gg2:=List(gg1,x->m1*x*m2);
        gg2:=List(gg2,x->x{[d+1..r]}{[d+1..r]});
        tg2:=Group(gg2);
        iso:=GroupHomomorphismByImages(tg,tg2,gg,gg2);
        ker:=Kernel(iso);
        th2:=List(Filtered(th,x->IsSubgroup(x,ker)),x->Image(iso,x));
        g2:=TransposedMatrixGroup(tg2);
        h2:=List(th2,TransposedMatrixGroup);
        return StablyPermutationCheckHP(g2,h2,c1,c2);
    fi;
end;

StablyPermutationCheckHMat:= function(g,hh,c1,c2,p)
    local gg,g1,g2,dx,m,i,j,d;
    gg:=List(hh,x->CosetRepresentation(g,x));
    Add(gg,GeneratorsOfGroup(g));
    g1:=[];
    g2:=[];
    for i in [1..Length(gg[1])] do
        m:=[];
        for j in [1..Length(gg)] do
            m:=Concatenation(m,List([1..c1[j]],x->gg[j][i]));
        od;
        Add(g1,DirectSumMat(m));
        m:=[];
        for j in [1..Length(gg)] do
            m:=Concatenation(m,List([1..c2[j]],x->gg[j][i]));
        od;
        Add(g2,DirectSumMat(m));
    od;
    d:=Length(g1[1]);
    if d<>Length(g2[1]) or d<>Length(p) or DeterminantMat(p)^2<>1 then
        return fail;
    else
        return List(g1,x->x^p)=g2;
    fi;
end;

StablyPermutationMCheckMat:= function(g,c1,c2,p)
    local h;
    h:=List(ConjugacyClassesSubgroups2(g),Representative);
    return StablyPermutationCheckHMat(g,h,c1,c2,p);
end;

StablyPermutationFCheckMat:= function(g,c1,c2,p)
    local tg,gg,d,th,mi,ms,o,h,r,gg1,gg2,g2,iso,ker,tg2,th2,h2,m1,m2,v;
    tg:=TransposedMatrixGroup(g);
    gg:=GeneratorsOfGroup(tg);
    d:=Length(Identity(g));
    th:=List(ConjugacyClassesSubgroups2(tg),Representative);
    mi:=FindCoflabbyResolutionBase(tg,th);
    r:=Length(mi);
    gg1:=List(gg,x->PermutationMat(Permutation(x,mi),r));
    if r=d then
        return true;
    else
        ms:=NullspaceIntMat(mi);
        v:=NullspaceIntMat(TransposedMat(ms));
        m1:=Concatenation(v,ms);
        m2:=m1^-1;
        gg2:=List(gg1,x->m1*x*m2);
        gg2:=List(gg2,x->x{[d+1..r]}{[d+1..r]});
        tg2:=Group(gg2);
        iso:=GroupHomomorphismByImages(tg,tg2,gg,gg2);
        ker:=Kernel(iso);
        th2:=List(Filtered(th,x->IsSubgroup(x,ker)),x->Image(iso,x));
        g2:=TransposedMatrixGroup(tg2);
        h2:=List(th2,TransposedMatrixGroup);
        return StablyPermutationCheckHMat(g2,h2,c1,c2,p);
    fi;
end;

StablyPermutationCheckHGen:= function(g,hh,c1,c2)
    local gg,g1,g2,dx,m,i,j;
    gg:=List(hh,x->CosetRepresentation(g,x));
    Add(gg,GeneratorsOfGroup(g));
    g1:=[];
    g2:=[];
    for i in [1..Length(gg[1])] do
        m:=[];
        for j in [1..Length(gg)] do
            m:=Concatenation(m,List([1..c1[j]],x->gg[j][i]));
        od;
        Add(g1,DirectSumMat(m));
        m:=[];
        for j in [1..Length(gg)] do
            m:=Concatenation(m,List([1..c2[j]],x->gg[j][i]));
        od;
        Add(g2,DirectSumMat(m));
    od;
    return [g1,g2];
end;

StablyPermutationMCheckGen:= function(g,c1,c2)
    local h;
    h:=List(ConjugacyClassesSubgroups2(g),Representative);
    return StablyPermutationCheckHGen(g,h,c1,c2);
end;

StablyPermutationFCheckGen:= function(g,c1,c2)
    local tg,gg,d,th,mi,ms,o,h,r,gg1,gg2,g2,iso,ker,tg2,th2,h2,m1,m2,v;
    tg:=TransposedMatrixGroup(g);
    gg:=GeneratorsOfGroup(tg);
    d:=Length(Identity(g));
    th:=List(ConjugacyClassesSubgroups2(tg),Representative);
    mi:=FindCoflabbyResolutionBase(tg,th);
    r:=Length(mi);
    gg1:=List(gg,x->PermutationMat(Permutation(x,mi),r));
    if r=d then
        return true;
    else
        ms:=NullspaceIntMat(mi);
        v:=NullspaceIntMat(TransposedMat(ms));
        m1:=Concatenation(v,ms);
        m2:=m1^-1;
        gg2:=List(gg1,x->m1*x*m2);
        gg2:=List(gg2,x->x{[d+1..r]}{[d+1..r]});
        tg2:=Group(gg2);
        iso:=GroupHomomorphismByImages(tg,tg2,gg,gg2);
        ker:=Kernel(iso);
        th2:=List(Filtered(th,x->IsSubgroup(x,ker)),x->Image(iso,x));
        g2:=TransposedMatrixGroup(tg2);
        h2:=List(th2,TransposedMatrixGroup);
        return StablyPermutationCheckHGen(g2,h2,c1,c2);
    fi;
end;
\end{verbatim}
\end{algorithmF}

\bigskip

\begin{example}[Algorithm \ref{Alg6}: Method II]\label{ex421}
Let $M_G$ be the $G$-lattice where $G={\rm Imf}(4,2,1)
\simeq D_6^2\rtimes C_2$ is the group of order $288$ of 
the GAP ID $(4,30,13,1)$. 
We will show that $[M]^{fl}=[F]=0$. 
Indeed, we can verify that $F$ is of rank $8$ and 
\begin{align*}
\bZ[G/H_{196}]\oplus\bZ[G/H_{212}]\simeq F\oplus\bZ[G/H_{217}]
\end{align*}
where $H_{196}\simeq C_2^2\times D_6$, $H_{212}\simeq C_2\times S_3^2$ and 
$H_{217}\simeq D_6^2$ (the rank of the both sides is 
$6+4=8+2=10$). 
By comparing with Algorithm \ref{Alg5}: Method I, we may 
obtain the matrix representations 
of $G_1$ and $G_2$ which corresponds to the action of $G$ on 
$\bZ[G/H_{196}]\oplus\bZ[G/H_{212}]$ and $F\oplus\bZ[G/H_{217}]$ 
respectively and the matrix $P$ which satisfies $\sigma_1 P=P \sigma_2$ 
for any $\sigma_1\in G_1$ and $\sigma_2\in G_2$. 

\bigskip

\begin{verbatim}
gap> Read("crystcat.gap");
gap> Read("FlabbyResolution.gap");

gap> G:=ImfMatrixGroup(4,2,1);
ImfMatrixGroup(4,2,1)
gap> StructureDescription(G);
"(C2 x C2 x S3 x S3) : C2"
gap> CrystCatZClass(G);
[ 4, 30, 13, 1 ]
gap> GeneratorsOfGroup(FlabbyResolution(G).actionF); # F is of rank 8
[ [ [ 0, 1, 0, 0, 0, 0, 0, 0 ], 
    [ 1, 0, 0, 0, 0, 0, 0, 0 ], 
    [ 0, 0, 1, 0, 0, 0, 0, 0 ], 
    [ 0, 0, 0, 0, 0, 1, 0, 0 ], 
    [ 0, 0, 0, 0, 1, 0, 0, 0 ], 
    [ 0, 0, 0, 1, 0, 0, 0, 0 ], 
    [ 0, 0, 0, 0, 0, 0, 1, 0 ], 
    [ 0, 0, 0, 0, 0, 0, 0, 1 ] ], 
  [ [ 0, 0, 0, 0, 0, 1, 0, 0 ], 
    [ 1, 0, 0, 0, 0, 0, 0, 0 ], 
    [ 0, 0, 1, 0, 0, 0, 0, 0 ], 
    [ 0, 1, 0, 0, 0, 0, 0, 0 ], 
    [ 0, 0, 0, 0, 1, 0, 0, 0 ], 
    [ 0, 1, 0, 1, 0, -1, 0, 0 ], 
    [ 0, 0, 0, 0, 0, 0, 1, 0 ], 
    [ 0, 0, 0, 0, 0, 0, 0, 1 ] ], 
  [ [ 0, 0, 1, 0, 0, 0, 0, 0 ], 
    [ 0, 0, 0, 0, 1, 0, 0, 0 ], 
    [ 1, 0, 0, 0, 0, 0, 0, 0 ], 
    [ 0, 0, 0, 0, 0, 0, 1, 0 ], 
    [ 0, 1, 0, 0, 0, 0, 0, 0 ], 
    [ 0, 0, 0, 0, 0, 0, 0, 1 ], 
    [ 0, 0, 0, 1, 0, 0, 0, 0 ], 
    [ 0, 0, 0, 0, 0, 1, 0, 0 ] ] ]
gap> ll:=PossibilityOfStablyPermutationF(G);;
gap> l:=ll[Length(ll)];
[ 0, 0, 0, 0, 0, 0, 0, 0, 0, 0, 0, 0, 0, 0, 0, 0, 0, 0, 0, 0, 0, 0, 0, 0, 0, 
  0, 0, 0, 0, 0, 0, 0, 0, 0, 0, 0, 0, 0, 0, 0, 0, 0, 0, 0, 0, 0, 0, 0, 0, 0, 
  0, 0, 0, 0, 0, 0, 0, 0, 0, 0, 0, 0, 0, 0, 0, 0, 0, 0, 0, 0, 0, 0, 0, 0, 0, 
  0, 0, 0, 0, 0, 0, 0, 0, 0, 0, 0, 0, 0, 0, 0, 0, 0, 0, 0, 0, 0, 0, 0, 0, 0, 
  0, 0, 0, 0, 0, 0, 0, 0, 0, 0, 0, 0, 0, 0, 0, 0, 0, 0, 0, 0, 0, 0, 0, 0, 0, 
  0, 0, 0, 0, 0, 0, 0, 0, 0, 0, 0, 0, 0, 0, 0, 0, 0, 0, 0, 0, 0, 0, 0, 0, 0, 
  0, 0, 0, 0, 0, 0, 0, 0, 0, 0, 0, 0, 0, 0, 0, 0, 0, 0, 0, 0, 0, 0, 0, 0, 0, 
  0, 0, 0, 0, 0, 0, 0, 0, 0, 0, 0, 0, 0, 0, 0, 0, 0, 0, 0, 0, 1, 0, 0, 0, 0, 
  0, 0, 0, 0, 0, 0, 0, 0, 0, 0, 0, 1, 0, 0, 0, 0, -1, 0, 0, 0, 0, 0, 0, 0, -1 ]
gap> Length(l);
225
gap> [l[196],l[212],l[217],l[225]];
[ 1, 1, -1, -1 ]
gap> ss:=List(ConjugacyClassesSubgroups2(G),x->StructureDescription(Representative(x)));;
gap> [ss[196],ss[212],ss[217]];
[ "C2 x C2 x C2 x S3", "(S3 x S3) : C2", "C2 x C2 x S3 x S3" ]
gap> bp:=StablyPermutationFCheckP(G,Nlist(l),Plist(l));;
gap> Length(bp); 
10
gap> Length(bp[1]); # rank of the both sides of (10) is 10
10
gap> cc:=Filtered(Tuples([0,1],10),x->Determinant(x*bp)^2=1);
[ [ 0, 1, 1, 0, 0, 0, 1, 0, 1, 0 ], [ 0, 1, 1, 0, 0, 0, 1, 1, 0, 0 ], 
  [ 0, 1, 1, 0, 0, 1, 0, 0, 1, 0 ], [ 0, 1, 1, 0, 0, 1, 0, 1, 0, 0 ], 
  [ 0, 1, 1, 0, 1, 0, 1, 0, 1, 1 ], [ 0, 1, 1, 0, 1, 0, 1, 1, 0, 1 ], 
  [ 0, 1, 1, 0, 1, 1, 0, 0, 1, 1 ], [ 0, 1, 1, 0, 1, 1, 0, 1, 0, 1 ], 
  [ 1, 0, 0, 1, 0, 0, 1, 0, 1, 0 ], [ 1, 0, 0, 1, 0, 0, 1, 1, 0, 0 ], 
  [ 1, 0, 0, 1, 0, 1, 0, 0, 1, 0 ], [ 1, 0, 0, 1, 0, 1, 0, 1, 0, 0 ], 
  [ 1, 0, 0, 1, 1, 0, 1, 0, 1, 1 ], [ 1, 0, 0, 1, 1, 0, 1, 1, 0, 1 ], 
  [ 1, 0, 0, 1, 1, 1, 0, 0, 1, 1 ], [ 1, 0, 0, 1, 1, 1, 0, 1, 0, 1 ] ]
gap> p:=cc[1]*bp;
[ [ 0, 1, 0, 0, 1, 1, 1, 1, 0, 0 ], 
  [ 1, 0, 1, 1, 0, 0, 0, 0, 1, 1 ], 
  [ 0, 0, 0, 0, 0, 1, 0, 1, 0, 0 ], 
  [ 0, 0, 0, 0, 1, 0, 1, 0, 0, 0 ], 
  [ 0, 0, 0, 1, 0, 0, 0, 0, 0, 1 ], 
  [ 0, 1, 0, 0, 0, 0, 0, 1, 0, 0 ], 
  [ 0, 0, 1, 0, 0, 0, 0, 0, 1, 0 ], 
  [ 0, 1, 0, 0, 0, 0, 1, 0, 0, 0 ], 
  [ 1, 0, 0, 0, 0, 0, 0, 0, 0, 1 ], 
  [ 1, 0, 0, 0, 0, 0, 0, 0, 1, 0 ] ]
gap> Determinant(p);
1
gap> gg:=StablyPermutationFCheckGen(G,Nlist(l),Plist(l));
[ [ [ [ 1, 0, 0, 0, 0, 0, 0, 0, 0, 0 ], [ 0, 1, 0, 0, 0, 0, 0, 0, 0, 0 ], 
      [ 0, 0, 0, 1, 0, 0, 0, 0, 0, 0 ], [ 0, 0, 1, 0, 0, 0, 0, 0, 0, 0 ], 
      [ 0, 0, 0, 0, 1, 0, 0, 0, 0, 0 ], [ 0, 0, 0, 0, 0, 0, 0, 1, 0, 0 ], 
      [ 0, 0, 0, 0, 0, 0, 1, 0, 0, 0 ], [ 0, 0, 0, 0, 0, 1, 0, 0, 0, 0 ], 
      [ 0, 0, 0, 0, 0, 0, 0, 0, 1, 0 ], [ 0, 0, 0, 0, 0, 0, 0, 0, 0, 1 ] ], 
    [ [ 1, 0, 0, 0, 0, 0, 0, 0, 0, 0 ], [ 0, 1, 0, 0, 0, 0, 0, 0, 0, 0 ], 
      [ 0, 0, 0, 0, 0, 0, 0, 1, 0, 0 ], [ 0, 0, 1, 0, 0, 0, 0, 0, 0, 0 ], 
      [ 0, 0, 0, 0, 1, 0, 0, 0, 0, 0 ], [ 0, 0, 0, 1, 0, 0, 0, 0, 0, 0 ], 
      [ 0, 0, 0, 0, 0, 0, 1, 0, 0, 0 ], [ 0, 0, 0, 1, 0, 1, 0, -1, 0, 0 ], 
      [ 0, 0, 0, 0, 0, 0, 0, 0, 1, 0 ], [ 0, 0, 0, 0, 0, 0, 0, 0, 0, 1 ] ], 
    [ [ 0, 1, 0, 0, 0, 0, 0, 0, 0, 0 ], [ 1, 0, 0, 0, 0, 0, 0, 0, 0, 0 ], 
      [ 0, 0, 0, 0, 1, 0, 0, 0, 0, 0 ], [ 0, 0, 0, 0, 0, 0, 1, 0, 0, 0 ], 
      [ 0, 0, 1, 0, 0, 0, 0, 0, 0, 0 ], [ 0, 0, 0, 0, 0, 0, 0, 0, 1, 0 ], 
      [ 0, 0, 0, 1, 0, 0, 0, 0, 0, 0 ], [ 0, 0, 0, 0, 0, 0, 0, 0, 0, 1 ], 
      [ 0, 0, 0, 0, 0, 1, 0, 0, 0, 0 ], [ 0, 0, 0, 0, 0, 0, 0, 1, 0, 0 ] ] ], 
  [ [ [ 1, 0, 0, 0, 0, 0, 0, 0, 0, 0 ], [ 0, 1, 0, 0, 0, 0, 0, 0, 0, 0 ], 
      [ 0, 0, 1, 0, 0, 0, 0, 0, 0, 0 ], [ 0, 0, 0, 1, 0, 0, 0, 0, 0, 0 ], 
      [ 0, 0, 0, 0, 0, 1, 0, 0, 0, 0 ], [ 0, 0, 0, 0, 1, 0, 0, 0, 0, 0 ], 
      [ 0, 0, 0, 0, 0, 0, 0, 1, 0, 0 ], [ 0, 0, 0, 0, 0, 0, 1, 0, 0, 0 ], 
      [ 0, 0, 0, 0, 0, 0, 0, 0, 1, 0 ], [ 0, 0, 0, 0, 0, 0, 0, 0, 0, 1 ] ], 
    [ [ 1, 0, 0, 0, 0, 0, 0, 0, 0, 0 ], [ 0, 0, 0, 0, 1, 0, 0, 0, 0, 0 ], 
      [ 0, 0, 1, 0, 0, 0, 0, 0, 0, 0 ], [ 0, 0, 0, 1, 0, 0, 0, 0, 0, 0 ], 
      [ 0, 0, 0, 0, 0, 1, 0, 0, 0, 0 ], [ 0, 1, 0, 0, 0, 0, 0, 0, 0, 0 ], 
      [ 0, 0, 0, 0, 0, 0, 0, 1, 0, 0 ], [ 0, 0, 0, 0, 0, 0, 1, 0, 0, 0 ], 
      [ 0, 0, 0, 0, 0, 0, 0, 0, 1, 0 ], [ 0, 0, 0, 0, 0, 0, 0, 0, 0, 1 ] ], 
    [ [ 0, 1, 0, 0, 0, 0, 0, 0, 0, 0 ], [ 1, 0, 0, 0, 0, 0, 0, 0, 0, 0 ], 
      [ 0, 0, 0, 0, 1, 0, 0, 0, 0, 0 ], [ 0, 0, 0, 0, 0, 1, 0, 0, 0, 0 ], 
      [ 0, 0, 1, 0, 0, 0, 0, 0, 0, 0 ], [ 0, 0, 0, 1, 0, 0, 0, 0, 0, 0 ], 
      [ 0, 0, 0, 0, 0, 0, 0, 0, 1, 0 ], [ 0, 0, 0, 0, 0, 0, 0, 0, 0, 1 ], 
      [ 0, 0, 0, 0, 0, 0, 1, 0, 0, 0 ], [ 0, 0, 0, 0, 0, 0, 0, 1, 0, 0 ] ] ] ]
gap> List(gg[1],x->p^-1*x*p)=gg[2];
true
gap> StablyPermutationFCheckMat(G,Nlist(l),Plist(l),p);
true
\end{verbatim}
\end{example}

\bigskip

\subsection{Verification of $[M_G]^{fl}=0$: Method III}\label{ss57}~\\

Let $G$ be a finite subgroup of $\GL(n,\bZ)$ and $M=M_G$ be 
the corresponding $G$-lattice of rank $n$ as in Definition \ref{defMG}. 
In order to confirm that $[M]^{fl}=0$, Method I and Method II 
in the previous two section gave how to find the explicit isomorphism 
\setcounter{equation}{9}
\begin{align}
\left(\bigoplus_{i=1}^r \bZ[G/H_i]^{\oplus a_i}\right)\oplus F^{\oplus b_1}
\simeq
\left(\bigoplus_{i=1}^r \bZ[G/H_i]^{\oplus a_i^{\prime}}\right)\oplus F^{\oplus b_1^{\prime}}.
\label{eqiso3}
\end{align} 
However, if the rank of the $G$-lattices in the both sides of (\ref{eqiso3}) 
is large, the algorithms in Method I and Method II do not work. 
In this case, we should try to reduce the rank of the $G$-lattices in (\ref{eqiso3}). 
The following algorithm (Algorithm \ref{Alg7}) can search suitable 
base change of the permutation $G$-lattice $P^\circ$ in a coflabby resolution 
$0\rightarrow {\rm Ker}\, f\rightarrow P^\circ\xrightarrow{f} M^\circ\rightarrow 0$ 
of $M^\circ$ as in (\ref{cr}) 
in order to reduce the rank of ${\rm Ker}\, f$. 
Then we may get a ``reduced'' flabby resolution of $M$: 
$0\rightarrow M\rightarrow P\rightarrow ({\rm Ker}\, f)^\circ\rightarrow 0$.

\bigskip

\noindent
{\tt SearchCoflabbyResolutionBase(G,b)} returns a list 
$\mathcal{L}=\{m_1,\dots,m_s\}$ where the $m_i$'s are all of the 
lists which satisfy that 
$m_i=\mathcal{P}^\circ$, $P^\circ=\bZ[\mathcal{P}^\circ]$ and 
$P^\circ$ satisfies (\ref{eqPP}) with $\#\mathcal{R}^{\prime}\leq b$. \\

\noindent
{\tt FlabbyResolutionFromBase(G,mi)},\\
{\tt PossibilityOfStablyPermutationFFromBase(G,mi)},\\ 
{\tt StablyPermutationFCheckFromBase(G,mi,L1,L2)},\\
{\tt StablyPermutationFCheckPFromBase(G,mi,L1,L2)},\\
{\tt StablyPermutationFCheckMatFromBase(G,mi,L1,L2,P)}\\

\noindent
return the same as\\ 

\noindent
{\tt FlabbyResolution(G)} (in Algorithm \ref{Alg1}),\\
{\tt PossibilityOfStablyPermutationF(G)} (in Algorithm \ref{Alg4}),\\
{\tt StablyPermutationFCheck(G,L1,L2)} (in Algorithm \ref{Alg5}: Method I),\\
{\tt StablyPermutationFCheckP(G,L1,L2)} (in Algorithm \ref{Alg6}: Method II),\\
{\tt StablyPermutationFCheckMat(G,L1,L2,P)} (in Algorithm \ref{Alg6}: Method II)\\

\noindent
respectively 
but with respect to $m_i=\mathcal{P}^\circ$ 
instead of the original $\mathcal{P}^\circ$ as in (\ref{eqPP}).

\bigskip

\begin{algorithmF}[Base change of a flabby resolution of $M_G$: 
Method III]\label{Alg7}
{}~\\
\begin{verbatim}
CheckCoflabbyResolutionBaseH:= function(g,hh,mi)
    local z0;
    z0:=List(hh,Z0lattice);
    return ForAll([1..Length(hh)],i->LatticeBasis(List(OrbitsDomain(hh[i],mi),Sum))=z0[i]);
end;

CheckCoflabbyResolutionBase:= function(g,mi)
    local hh,z0;
    hh:=List(ConjugacyClassesSubgroups2(g),Representative);
    return CheckCoflabbyResolutionBaseH(g,hh,mi);
end;

SearchCoflabbyResolutionBaseH:= function(g,hh,b)
    local z0,orbs,imgs,i,j,mi,mis;
    mis:=[];
    z0:=List(hh,Z0lattice);
    orbs:=Set(Union(z0),x->Orbit(g,x));
    imgs:=List(orbs,x->List(hh,y->LatticeBasis(List(OrbitsDomain(y,x),Sum))));
    if b=0 then
        for i in [1..Length(orbs)] do
            for j in Combinations([1..Length(orbs)],i) do
                if ForAll([1..Length(hh)],x->LatticeBasis(
                  Union(List(j,y->imgs[y][x])))=z0[x]) then
                    mi:=Union(List(j,x->orbs[x]));
                    if mis=[] or Length(mi)<Length(mis) then
                        mis:=mi;
                    fi;
                fi;
            od;
            if mis<>[] then
                return mis;
            fi;
       od;
    else
        for i in [1..b] do
            for j in Combinations([1..Length(orbs)],i) do
                if ForAll([1..Length(hh)],x->LatticeBasis(
                  Union(List(j,y->imgs[y][x])))=z0[x]) then
                    mi:=Union(List(j,x->orbs[x]));
                    Add(mis,mi);
                fi;
            od;
        od;
        return Set(mis);
    fi;
end;

SearchCoflabbyResolutionBase:= function(g,b)
    local hh;
    hh:=List(ConjugacyClassesSubgroups2(g),Representative);
    return SearchCoflabbyResolutionBaseH(g,hh,b);
end;

FlabbyResolutionFromBase:= function(g,mi)
    local tg,gg,d,th,ms,o,r,gg1,gg2,v1,mg,img;
    tg:=TransposedMatrixGroup(g);
    gg:=GeneratorsOfGroup(tg);
    d:=Length(Identity(g));
    th:=List(ConjugacyClassesSubgroups2(tg),Representative);
    r:=Length(mi);
    o:=IdentityMat(r);
    gg1:=List(gg,x->PermutationMat(Permutation(x,mi),r));
    if r=d then
        return rec(injection:=TransposedMat(mi),
                   surjection:=NullMat(r,0),
                   actionP:=TransposedMatrixGroup(Group(gg1,o))
        );
    else
        ms:=NullspaceIntMat(mi);
        v1:=NullspaceIntMat(TransposedMat(ms));
        mg:=Concatenation(v1,ms);
        img:=mg^-1;
        gg2:=List(gg1,x->mg*x*img);
        gg2:=List(gg2,x->x{[d+1..r]}{[d+1..r]});
        return rec(injection:=TransposedMat(mi),
                   surjection:=TransposedMat(ms),
                   actionP:=TransposedMatrixGroup(Group(gg1)),
                   actionF:=TransposedMatrixGroup(Group(gg2))
        );
    fi;
end;

PossibilityOfStablyPermutationFFromBase:= function(g,mi)
    local tg,gg,d,th,ms,r,gg1,gg2,g2,iso,ker,tg2,th2,h2,m1,m2,v;
    tg:=TransposedMatrixGroup(g);
    gg:=GeneratorsOfGroup(tg);
    d:=Length(Identity(g));
    th:=List(ConjugacyClassesSubgroups2(tg),Representative);
    r:=Length(mi);
    gg1:=List(gg,x->PermutationMat(Permutation(x,mi),r));
    if r=d then
        return [0,1];
    else
        ms:=NullspaceIntMat(mi);
        v:=NullspaceIntMat(TransposedMat(ms));
        m1:=Concatenation(v,ms);
        m2:=m1^-1;
        gg2:=List(gg1,x->m1*x*m2);
        gg2:=List(gg2,x->x{[d+1..r]}{[d+1..r]});
        tg2:=Group(gg2);
        iso:=GroupHomomorphismByImages(tg,tg2,gg,gg2);
        ker:=Kernel(iso);
        th2:=List(Filtered(th,x->IsSubgroup(x,ker)),x->Image(iso,x));
        g2:=TransposedMatrixGroup(tg2);
        h2:=List(th2,TransposedMatrixGroup);
        return PossibilityOfStablyPermutationH(g2,h2);
    fi;
end;

StablyPermutationFCheckFromBase:= function(g,mi,c1,c2)
    local tg,gg,d,th,ms,o,h,r,gg1,gg2,g2,iso,ker,tg2,th2,h2,m1,m2,v;
    tg:=TransposedMatrixGroup(g);
    gg:=GeneratorsOfGroup(tg);
    d:=Length(Identity(g));
    th:=List(ConjugacyClassesSubgroups2(tg),Representative);
    r:=Length(mi);
    gg1:=List(gg,x->PermutationMat(Permutation(x,mi),r));
    if r=d then
        return true;
    else
        ms:=NullspaceIntMat(mi);
        v:=NullspaceIntMat(TransposedMat(ms));
        m1:=Concatenation(v,ms);
        m2:=m1^-1;
        gg2:=List(gg1,x->m1*x*m2);
        gg2:=List(gg2,x->x{[d+1..r]}{[d+1..r]});
        tg2:=Group(gg2);
        iso:=GroupHomomorphismByImages(tg,tg2,gg,gg2);
        ker:=Kernel(iso);
        th2:=List(Filtered(th,x->IsSubgroup(x,ker)),x->Image(iso,x));
        g2:=TransposedMatrixGroup(tg2);
        h2:=List(th2,TransposedMatrixGroup);
        return StablyPermutationCheckH(g2,h2,c1,c2);
    fi;
end;

StablyPermutationFCheckPFromBase:= function(g,mi,c1,c2)
    local tg,gg,d,th,ms,o,h,r,gg1,gg2,g2,iso,ker,tg2,th2,h2,m1,m2,v;
    tg:=TransposedMatrixGroup(g);
    gg:=GeneratorsOfGroup(tg);
    d:=Length(Identity(g));
    th:=List(ConjugacyClassesSubgroups2(tg),Representative);
    r:=Length(mi);
    gg1:=List(gg,x->PermutationMat(Permutation(x,mi),r));
    if r=d then
        return true;
    else
        ms:=NullspaceIntMat(mi);
        v:=NullspaceIntMat(TransposedMat(ms));
        m1:=Concatenation(v,ms);
        m2:=m1^-1;
        gg2:=List(gg1,x->m1*x*m2);
        gg2:=List(gg2,x->x{[d+1..r]}{[d+1..r]});
        tg2:=Group(gg2);
        iso:=GroupHomomorphismByImages(tg,tg2,gg,gg2);
        ker:=Kernel(iso);
        th2:=List(Filtered(th,x->IsSubgroup(x,ker)),x->Image(iso,x));
        g2:=TransposedMatrixGroup(tg2);
        h2:=List(th2,TransposedMatrixGroup);
        return StablyPermutationCheckHP(g2,h2,c1,c2);
    fi;
end;

StablyPermutationFCheckMatFromBase:= function(g,mi,c1,c2,p)
    local tg,gg,d,th,ms,o,h,r,gg1,gg2,g2,iso,ker,tg2,th2,h2,m1,m2,v;
    tg:=TransposedMatrixGroup(g);
    gg:=GeneratorsOfGroup(tg);
    d:=Length(Identity(g));
    th:=List(ConjugacyClassesSubgroups2(tg),Representative);
    r:=Length(mi);
    gg1:=List(gg,x->PermutationMat(Permutation(x,mi),r));
    if r=d then
        return true;
    else
        ms:=NullspaceIntMat(mi);
        v:=NullspaceIntMat(TransposedMat(ms));
        m1:=Concatenation(v,ms);
        m2:=m1^-1;
        gg2:=List(gg1,x->m1*x*m2);
        gg2:=List(gg2,x->x{[d+1..r]}{[d+1..r]});
        tg2:=Group(gg2);
        iso:=GroupHomomorphismByImages(tg,tg2,gg,gg2);
        ker:=Kernel(iso);
        th2:=List(Filtered(th,x->IsSubgroup(x,ker)),x->Image(iso,x));
        g2:=TransposedMatrixGroup(tg2);
        h2:=List(th2,TransposedMatrixGroup);
        return StablyPermutationCheckHMat(g2,h2,c1,c2,p);
    fi;
end;

StablyPermutationFCheckGenFromBase:= function(g,mi,c1,c2)
    local tg,gg,d,th,ms,o,h,r,gg1,gg2,g2,iso,ker,tg2,th2,h2,m1,m2,v;
    tg:=TransposedMatrixGroup(g);
    gg:=GeneratorsOfGroup(tg);
    d:=Length(Identity(g));
    th:=List(ConjugacyClassesSubgroups2(tg),Representative);
    r:=Length(mi);
    gg1:=List(gg,x->PermutationMat(Permutation(x,mi),r));
    if r=d then
        return true;
    else
        ms:=NullspaceIntMat(mi);
        v:=NullspaceIntMat(TransposedMat(ms));
        m1:=Concatenation(v,ms);
        m2:=m1^-1;
        gg2:=List(gg1,x->m1*x*m2);
        gg2:=List(gg2,x->x{[d+1..r]}{[d+1..r]});
        tg2:=Group(gg2);
        iso:=GroupHomomorphismByImages(tg,tg2,gg,gg2);
        ker:=Kernel(iso);
        th2:=List(Filtered(th,x->IsSubgroup(x,ker)),x->Image(iso,x));
        g2:=TransposedMatrixGroup(tg2);
        h2:=List(th2,TransposedMatrixGroup);
        return StablyPermutationCheckHGen(g2,h2,c1,c2);
    fi;
end;
\end{verbatim}
\end{algorithmF}

\bigskip

\begin{example}[Algorithm \ref{Alg7}: Method III]
Let $M_G$ be the $G$-lattice where $G\simeq C_2\times S_4$ 
is the group of order $48$ of the CARAT ID $(5,533,8)$. 
By using {\tt FlabbyResolution(G).actionF} as in Algorithm \ref{Alg1}, 
we obtain a flabby resolution 
$0\rightarrow M_G\rightarrow P\rightarrow F\rightarrow 0$ of $M_G$. 
However, the flabby $G$-lattice $F$ is of rank $44$. 

By {\tt FlabbyResolutionFromBase(G,mi).actionF}, 
we get a suitable flabby resolution of $M_G$: 
$0\rightarrow M_G\rightarrow P^\prime\rightarrow F^\prime\rightarrow 0$ 
where the flabby $G$-lattice $F^\prime$ is of rank $10$. 

By using {\tt PossibilityOfStablyPermutationF(G)}, 
it may be possible that the isomorphism 
$\bZ[G/H_{20}]\oplus\bZ[G/H_{22}]\oplus\bZ\simeq
\bZ[G/H_{29}]\oplus F^\prime$ occurs where 
$H_{20}\simeq C_2^3$, $H_{22}\simeq D_4$ and $H_{29}\simeq C_2\times D_4$. 

In order to confirm the isomorphism, 
we will use Mersenne Twister (cf. \cite{MN98}) to search an appropriate 
coefficients $c_i$ to get a transformation matrix $P=\sum_i c_i P_i$ 
which satisfies $G_1P=PG_2$ as in (\ref{eqisom})
(one can use the classical global random generator via 
{\tt IsGlobalRandomSource} which is given in 
\cite[Algorithm A in 3.2.2 with lag 30]{Knu98}). 
We should choose suitable integers $n_1$, $n_2$, $n_3$ in
\begin{center}
{\tt rr:=List([1..n1],x->List([1..20],y->Random(rs,[n2..n3])))}
\end{center}
to get the desired coefficients.

\bigskip

\begin{verbatim}
Read("caratnumber.gap");
Read("FlabbyResolution.gap");

gap> G:=CaratMatGroupZClass(5,533,8);; # G=C2xS4
gap> Rank(FlabbyResolution(G).actionF.1); # F is of rank 44
44
gap> mis:=SearchCoflabbyResolutionBase(TransposedMatrixGroup(G),3);; # Method III
gap> List(mis,Length);
[ 29, 29, 15, 15, 15, 15 ]
gap> mi:=mis[3]; 
[ [ -1, -1, 0, -1, 0 ], [ -1, -1, 0, 0, -1 ], [ 0, -1, 0, -1, 0 ], 
  [ 0, -1, 0, 0, -1 ], [ 0, -1, 1, -1, 0 ], [ 0, -1, 1, 0, -1 ], 
  [ 0, 0, 0, -1, 0 ], [ 0, 0, 0, -1, 1 ], [ 0, 0, 0, 0, -1 ], 
  [ 0, 0, 0, 1, -1 ], [ 0, 0, 1, -1, 0 ], [ 0, 0, 1, 0, -1 ], 
  [ 0, 1, -1, 1, 1 ], [ 1, 0, 1, -1, 0 ], [ 1, 0, 1, 0, -1 ] ]
gap> FF:=FlabbyResolutionFromBase(G,mi).actionF;
<matrix group with 2 generators>
gap> Rank(FF.1); # FF is of rank 10 (=15-5)
10
gap> ll:=PossibilityOfStablyPermutationFFromBase(G,mi);
[ [ 1, 0, 0, 0, 0, -2, -1, 0, 0, 0, 0, 0, 0, 0, 0, 0, 0, 0, 2, 0, 0, 1, 0, 0, 
      1, -1, 0, 0, 1, 1, -1, 1, -1, -1 ], 
  [ 0, 1, 0, 0, 0, 0, 0, 0, 0, 0, 0, 0, 0, 0, 0, -2, 0, -1, 0, 0, 0, 1, 0, 0, 
      0, 0, 0, 2, 1, 1, 0, 0, -1, -1 ], 
  [ 0, 0, 0, 0, 1, -1, 0, 0, 0, 0, 0, 0, 0, 0, 0, 0, -1, 0, 1, 0, 0, 0, 0, 0, 
      1, -1, 0, 0, 0, 0, -1, 1, 0, 0 ], 
  [ 0, 0, 0, 0, 0, 0, 0, 1, 0, 0, 0, 0, 0, 0, 0, 0, 0, 0, 0, 0, 0, 1, 0, 0, 
      -1, -1, -1, 0, 1, 1, 1, 1, -1, -1 ], 
  [ 0, 0, 0, 0, 0, 0, 0, 0, 0, 0, 0, 0, 0, 0, 0, 0, 0, 0, 0, 1, 0, 1, 0, 0, 
      0, 0, 0, 0, -1, 0, 0, 0, 1, -1 ] ]
gap> l:=ll[Length(ll)];
[ 0, 0, 0, 0, 0, 0, 0, 0, 0, 0, 0, 0, 0, 0, 0, 0, 0, 0, 0, 1, 
  0, 1, 0, 0, 0, 0, 0, 0, -1, 0, 0, 0, 1, -1 ]
gap> Length(l);
34
gap> [l[20],l[22],l[33],l[29],l[34]];
[ 1, 1, 1, -1, -1 ]
gap> ss:=List(ConjugacyClassesSubgroups2(G),x->StructureDescription(Representative(x)));
[ "1", "C2", "C2", "C2", "C2", "C2", "C3", "C2 x C2", "C2 x C2", "C2 x C2", "C2 x C2", 
  "C2 x C2", "C4", "C4", "C2 x C2", "C2 x C2", "C6", "S3", "S3", "C2 x C2 x C2", "D8", 
  "D8", "C2 x C2 x C2", "C4 x C2", "D8", "D8", "A4", "D12", "C2 x D8", "C2 x A4", 
  "S4", "S4", "C2 x S4" ]
gap> [ss[20],ss[22],ss[33],ss[29]];
[ "C2 x C2 x C2", "D8", "C2 x S4", "C2 x D8" ]
gap> bp:=StablyPermutationFCheckPFromBase(G,mi,Nlist(l),Plist(l));;
gap> Length(bp); 
20
gap> Length(bp[1]); # rank of the both sides of (10) is 13
13
gap> rs:=RandomSource(IsMersenneTwister);
<RandomSource in IsMersenneTwister>
gap> rr:=List([1..1000],x->List([1..20],y->Random(rs,[0..1])));;
gap> Filtered(rr,x->Determinant(x*bp)^2=1); # MT found 3 solutions
[ [ 1, 1, 0, 1, 1, 1, 1, 1, 0, 1, 1, 1, 1, 1, 0, 1, 1, 0, 1, 1 ], 
  [ 0, 1, 1, 0, 0, 1, 1, 0, 1, 0, 0, 1, 1, 1, 0, 1, 1, 0, 1, 1 ], 
  [ 0, 1, 1, 1, 1, 1, 1, 1, 1, 1, 1, 0, 1, 0, 1, 1, 1, 1, 1, 1 ] ]
gap> p:=last[1]*bp;                                            
[ [ 1, 1, 1, 1, 0, 0, 1, 1, 1, 1, 1, 1, 1 ], 
  [ 0, 1, 1, 0, 1, 1, 1, 1, 1, 1, 1, 1, 1 ], 
  [ 1, 0, 0, 1, 1, 1, 1, 1, 1, 1, 1, 1, 1 ], 
  [ 1, 1, 0, 1, 1, 1, 1, 1, 1, 0, 1, 1, 1 ], 
  [ 0, 1, 1, 1, 1, 1, 1, 1, 0, 1, 1, 1, 1 ], 
  [ 1, 1, 1, 1, 0, 1, 1, 1, 1, 1, 1, 0, 1 ], 
  [ 1, 0, 1, 1, 1, 1, 1, 1, 1, 1, 0, 1, 1 ], 
  [ 1, 1, 1, 0, 1, 1, 1, 0, 1, 1, 1, 1, 1 ], 
  [ 1, 1, 1, 1, 1, 0, 0, 1, 1, 1, 1, 1, 1 ], 
  [ 1, 1, 1, 0, 1, 1, 0, 1, 1, 1, 1, 1, 1 ], 
  [ 0, 0, 1, 1, 1, 0, 1, 1, 1, 1, 1, 1, 1 ], 
  [ 1, 1, 0, 0, 0, 1, 1, 1, 1, 1, 1, 1, 1 ], 
  [ 1, 1, 1, 1, 0, 1, 1, 1, 1, 1, 0, 1, 1 ] ]
gap> Determinant(p);
1
gap> StablyPermutationFCheckMatFromBase(G,mi,Nlist(l),Plist(l),p);
true

gap> rs:=RandomSource(IsGlobalRandomSource); # alternatively 
<RandomSource in IsGlobalRandomSource>
gap> rr:=List([1..1000],x->List([1..20],y->Random(rs,[0..1])));;
gap> Filtered(rr,x->Determinant(x*bp)^2=1); # found 1 solution
[ [ 0, 1, 0, 1, 0, 1, 0, 0, 1, 0, 0, 0, 0, 1, 0, 0, 1, 1, 1, 1 ] ]
\end{verbatim}
\end{example}

\bigskip

%
\section{Flabby and coflabby $G$-lattices}\label{seFC}

In this section, we will determine all the flabby and coflabby 
$G$-lattices $M$ of rank $M\leq 6$ by using the algorithms 
{\tt IsFlabby} and {\tt IsCoflabby} which are given in Section \ref{seAlg}. 

First we make a list of all $\bZ$-class of $\GL(n,\bZ)$ and 
we filter off the groups $G$ such that $H^{-1}(G) \not= 0$ 
or $H^1(G) \not= 0$ form the list.
Next we use the derived subgroup $D(G)=[G,G]$,
the center $Z(G)$ and a 2-Sylow subgroup $Sy_2(G)$ of $G$ 
to filter the groups off. 
Finally we filter off the groups which are not flabby or not coflabby. 
Because {\tt IsFlabby(G)} and {\tt IsCoflabby(G)} are much slower 
than {\tt Hminus1(G)} and {\tt H1(G)}, 
we use {\tt Hminus(G)} and {\tt H1(G)} to the specified subgroups 
of $G$ as above. 
The following algorithms are available from 
{\tt http://math.h.kyoto-u.ac.jp/\~{}yamasaki/Algorithm/} 
as {\tt KS.gap}.\\

\noindent
{\tt AllFlabbyCoflabbyZClasses(n)} returns the list of all the GAP IDs
of $G$ such that $M_G$ is a flabby and coflabby $G$-lattice
of rank $n$ when $2\leq n\leq 4$.\\
{\tt AllFlabbyCoflabbyZClasses(n:Carat)} returns the same as 
{\tt AllFlabbyCoflabbyZClasses(n)} but using the CARAT ID instead 
of the GAP ID when $1\leq n\leq 6$. \\

\noindent
{\tt AllPermutationZClasses(n)} returns the list of all the GAP IDs
of $G$ such that $M_G$ is a permutation $G$-lattice when $2\leq n\leq 4$.\\
{\tt AllPermutationZClasses(n:Carat)} returns the same as 
{\tt AllPermutation(n)} but using the CARAT ID instead of the GAP ID 
when $1\leq n\leq 6$. 

\bigskip

\begin{algorithmFC}[Flabby and coflabby $G$-lattices]\label{algFC}
{}~\\
\begin{verbatim}
AllFlabbyCoflabbyZClasses:= function(n)
    local glnz,listg;
    glnz:=Concatenation(List([1..Length(cryst[n])],
      x->List([1..Length(cryst[n][x])],y->[n,x,y])));
    listg:=List(glnz,x->CaratMatGroupZClass(x[1],x[2],x[3]));
    listg:=Filtered(listg,
      x->Product(Hminus1(x))=1 and Product(H1(x))=1);
    listg:=Filtered(listg,
      x->ForAll([DerivedSubgroup(x),Centre(x),SylowSubgroup(x,2)],
      y->Product(Hminus1(y))=1 and Product(H1(y))=1));
    listg:=Filtered(listg,
      x->ForAll(List(ConjugacyClassesSubgroups2(x),Representative),
      y->Product(Hminus1(y))=1 and Product(H1(y))=1));
    if ValueOption("carat")=true or ValueOption("Carat")=true then
        return Set(listg,CaratZClass);
    else
        return Set(listg,CrystCatZClass);
    fi;
end;

AllPermutationZClasses:= function(n)
    local Sn,Snsub;
    Sn:=Group(List(GeneratorsOfGroup(SymmetricGroup(n)),
      x->PermutationMat(x,n)));
    Snsub:=List(ConjugacyClassesSubgroups2(Sn),Representative);
    if ValueOption("carat")=true or ValueOption("Carat")=true then
        return Set(Snsub,CaratZClass);
    else
        return Set(Snsub,CrystCatZClass);
    fi;
end;
\end{verbatim}
\end{algorithmFC}

\bigskip

\begin{example}[All flabby and coflabby $G$-lattices of rank $n\leq 6$ 
which are not permutation]
We may compute all the permutation $G$-lattices of rank $n\leq 6$ 
and all the flabby and coflabby $G$-lattices by using Algorithm \ref{algFC}. 

There exist $2$ (resp. $4$, $11$, $19$, $56$) permutation $G$-lattices of 
rank $2$ (resp. $3$, $4$, $5$, $6$). 
All the permutation $G$-lattices of rank $2\leq n\leq 5$ are given as follows:

\bigskip

\begin{tabular}{l|ll}
$G$ : permutation & $\{1\}$ & $C_2$\\\hline
GAP ID & $(2,1,1,1)$ & $(2,2,1,2)$
\end{tabular}\\

\begin{tabular}{l|llll}
$G$ : permutation & $\{1\}$ & $C_2$ & $C_3$ & $S_3$\\\hline
GAP ID & $(3,1,1,1)$ & $(3,2,2,2)$ & $(3,5,1,1)$ & $(3,5,4,1)$
\end{tabular}\\

\begin{tabular}{l|llllll}
$G$ : permutation & $\{1\}$ & $C_2$ & $C_2$ & $C_2^2$ & $C_2^2$ & $C_3$ \\\hline
GAP ID & $(4,1,1,1)$ & $(4,2,1,2)$ & $(4,3,1,3)$ & $(4,4,1,5)$ & $(4,5,1,1)$ & $(4,8,1,1)$ 
\end{tabular}\\

\begin{tabular}{l|lllll}
$G$ : permutation & $S_3$ & $C_4$ & $D_4$ & $A_4$ & $S_4$\\\hline
GAP ID & $(4,8,3,1)$ & $(4,12,1,1)$ & $(4,12,3,1)$ & $(4,24,1,1)$ & $(4,24,3,1)$ 
\end{tabular}\\

\begin{tabular}{l|lllllll}
$G$ : permutation & $\{1\}$ & $C_2$ & $C_2$ & $C_2^2$ & $C_2^2$ & $C_4$ & $D_4$\\\hline
CARAT ID & $(5,1,1)$ & $(5,4,2)$ & $(5,7,3)$ & $(5,9,6)$ & $(5,19,18)$ & 
$(5,58,8)$ & $(5,62,8)$
\end{tabular}\\

\begin{tabular}{l|llllll}
$G$ : permutation & $C_3$ & $S_3$ & $D_6$ & $S_3$ & $C_6$ & $A_4$\\\hline
CARAT ID & $(5,181,2)$ & $(5,186,2)$ & $(5,192,6)$ & $(5,218,8)$ & $(5,220,4)$ 
& $(5,502,3)$ 
\end{tabular}\\

\begin{tabular}{l|llllll}
$G$ : permutation & $C_4$ & $D_5$ & $C_5$ & $S_5$ & $F_{20}$ & $A_5$\\\hline
CARAT ID & $(5,506,6)$ & $(5,901,3)$ & $(5,909,2)$ & $(5,911,3)$ & $(5,918,3)$ & $(5,931,3)$ 
\end{tabular}

\bigskip

\begin{verbatim}
Read("caratnumber.gap");
Read("KS.gap");

gap> l2f:=AllFlabbyCoflabbyZClasses(2);
[ [ 2, 1, 1, 1 ], [ 2, 2, 1, 2 ] ]
gap> l2p:=AllPermutationZClasses(2);
[ [ 2, 1, 1, 1 ], [ 2, 2, 1, 2 ] ]

gap> l3f:=AllFlabbyCoflabbyZClasses(3);
[ [ 3, 1, 1, 1 ], [ 3, 2, 2, 2 ], [ 3, 5, 1, 1 ], [ 3, 5, 4, 1 ] ]
gap> l3p:=AllPermutationZClasses(3);   
[ [ 3, 1, 1, 1 ], [ 3, 2, 2, 2 ], [ 3, 5, 1, 1 ], [ 3, 5, 4, 1 ] ]

gap> l4f:=AllFlabbyCoflabbyZClasses(4);
[ [ 4, 1, 1, 1 ], [ 4, 2, 1, 2 ], [ 4, 3, 1, 3 ], [ 4, 4, 1, 5 ], [ 4, 5, 1, 1 ], 
  [ 4, 8, 1, 1 ], [ 4, 8, 3, 1 ], [ 4, 12, 1, 1 ], [ 4, 12, 3, 1 ], [ 4, 14, 2, 2 ], 
  [ 4, 14, 3, 3 ], [ 4, 14, 3, 4 ], [ 4, 14, 8, 2 ], [ 4, 24, 1, 1 ], [ 4, 24, 3, 1 ] ]
gap> Length(l4f);
15
gap> l4p:=AllPermutationZClasses(4);;   
gap> Length(l4p);
11
gap> Difference(l4f,l4p);              
[ [ 4, 14, 2, 2 ], [ 4, 14, 3, 3 ], [ 4, 14, 3, 4 ], [ 4, 14, 8, 2 ] ]
gap> Length(last);
4

gap> l5f:=AllFlabbyCoflabbyZClasses(5:Carat);
[ [ 5, 1, 1 ], [ 5, 4, 2 ], [ 5, 7, 3 ], [ 5, 9, 6 ], [ 5, 19, 18 ], 
  [ 5, 58, 8 ], [ 5, 62, 8 ], [ 5, 181, 2 ], [ 5, 186, 2 ], [ 5, 192, 6 ], 
  [ 5, 218, 4 ], [ 5, 218, 8 ], [ 5, 220, 4 ], [ 5, 502, 3 ], [ 5, 506, 6 ], 
  [ 5, 901, 3 ], [ 5, 909, 2 ], [ 5, 911, 3 ], [ 5, 911, 4 ], [ 5, 918, 3 ], 
  [ 5, 918, 4 ], [ 5, 931, 3 ], [ 5, 931, 4 ] ]
gap> Length(l5f);
23
gap> l5p:=AllPermutationZClasses(5:Carat);;   
gap> Length(l5p);
19
gap> Difference(l5f,l5p);                    
[ [ 5, 218, 4 ], [ 5, 911, 4 ], [ 5, 918, 4 ], [ 5, 931, 4 ] ]
gap> Length(last);
4

gap> l6f:=AllFlabbyCoflabbyZClasses(6:Carat);
[ [ 6, 2, 2 ], [ 6, 4, 3 ], [ 6, 8, 6 ], [ 6, 11, 4 ], [ 6, 15, 12 ], 
  [ 6, 159, 14 ], [ 6, 161, 14 ], [ 6, 161, 28 ], [ 6, 197, 14 ], [ 6, 226, 14 ], 
  [ 6, 226, 40 ], [ 6, 231, 39 ], [ 6, 238, 27 ], [ 6, 246, 21 ], [ 6, 366, 27 ], 
  [ 6, 891, 7 ], [ 6, 894, 6 ], [ 6, 927, 9 ], [ 6, 984, 6 ], [ 6, 993, 16 ], 
  [ 6, 1087, 20 ], [ 6, 1090, 18 ], [ 6, 1142, 8 ], [ 6, 1199, 16 ], [ 6, 1968, 3 ], 
  [ 6, 2007, 2 ], [ 6, 2010, 3 ], [ 6, 2026, 6 ], [ 6, 2043, 4 ], [ 6, 2043, 8 ], 
  [ 6, 2044, 4 ], [ 6, 2051, 9 ], [ 6, 2068, 6 ], [ 6, 2069, 6 ], [ 6, 2069, 12 ], 
  [ 6, 2070, 12 ], [ 6, 2079, 14 ], [ 6, 2079, 28 ], [ 6, 2088, 18 ], [ 6, 2105, 12 ], 
  [ 6, 2154, 26 ], [ 6, 2156, 40 ], [ 6, 2156, 80 ], [ 6, 2188, 39 ], [ 6, 2263, 6 ], 
  [ 6, 2278, 8 ], [ 6, 2709, 1 ], [ 6, 2958, 3 ], [ 6, 2966, 2 ], [ 6, 2968, 4 ], 
  [ 6, 2969, 4 ], [ 6, 2969, 8 ], [ 6, 2977, 6 ], [ 6, 3046, 3 ], [ 6, 3053, 5 ], 
  [ 6, 3066, 3 ], [ 6, 3068, 7 ], [ 6, 3068, 8 ], [ 6, 3071, 7 ], [ 6, 3071, 8 ], 
  [ 6, 3073, 7 ], [ 6, 3073, 8 ], [ 6, 3073, 15 ], [ 6, 3073, 16 ], [ 6, 3076, 7 ], 
  [ 6, 3076, 8 ], [ 6, 3091, 11 ], [ 6, 3091, 12 ], [ 6, 3276, 9 ], [ 6, 3297, 9 ], 
  [ 6, 3299, 9 ],  [ 6, 3302, 9 ], [ 6, 3575, 8 ], [ 6, 3662, 8 ], [ 6, 3663, 12 ], 
  [ 6, 3749, 10 ], [ 6, 4618, 18 ], [ 6, 4618, 19 ], [ 6, 4621, 18 ], [ 6, 4630, 52 ], 
  [ 6, 4647, 101 ], [ 6, 4722, 8 ], [ 6, 4733, 8 ], [ 6, 4743, 13 ], [ 6, 4750, 13 ], 
  [ 6, 4762, 41 ], [ 6, 4807, 41 ], [ 6, 4811, 41 ], [ 6, 4814, 82 ], [ 6, 4898, 3 ], 
  [ 6, 4904, 3 ], [ 6, 4915, 20 ], [ 6, 4919, 20 ], [ 6, 4929, 11 ], [ 6, 5210, 8 ], 
  [ 6, 5210, 14 ], [ 6, 5262, 11 ], [ 6, 5311, 16 ], [ 6, 5318, 8 ], [ 6, 5321, 6 ], 
  [ 6, 5321, 14 ], [ 6, 5421, 6 ], [ 6, 5424, 16 ], [ 6, 5475, 6 ], [ 6, 5477, 11 ], 
  [ 6, 5487, 11 ] ]
gap> Length(l6f);
106
gap> l6p:=AllPermutationZClasses(6:Carat);
gap> Length(l6p);                            
56
gap> Difference(l6f,l6p);
[ [ 6, 159, 14 ], [ 6, 161, 14 ], [ 6, 161, 28 ], [ 6, 197, 14 ], [ 6, 226, 14 ], 
  [ 6, 226, 40 ], [ 6, 231, 39 ], [ 6, 238, 27 ], [ 6, 246, 21 ], [ 6, 366, 27 ], 
  [ 6, 1087, 20 ], [ 6, 1090, 18 ], [ 6, 1142, 8 ], [ 6, 2043, 4 ], [ 6, 2051, 9 ], 
  [ 6, 2068, 6 ], [ 6, 2069, 6 ], [ 6, 2069, 12 ], [ 6, 2070, 12 ], [ 6, 2079, 14 ], 
  [ 6, 2079, 28 ], [ 6, 2088, 18 ], [ 6, 2105, 12 ], [ 6, 2154, 26 ], [ 6, 2156, 40 ], 
  [ 6, 2156, 80 ], [ 6, 2188, 39 ], [ 6, 2968, 4 ], [ 6, 2969, 4 ], [ 6, 2969, 8 ], 
  [ 6, 2977, 6 ], [ 6, 3068, 7 ], [ 6, 3068, 8 ], [ 6, 3071, 7 ], [ 6, 3071, 8 ], 
  [ 6, 3073, 7 ], [ 6, 3073, 8 ], [ 6, 3073, 15 ], [ 6, 3073, 16 ], [ 6, 3076, 7 ], 
  [ 6, 3076, 8 ], [ 6, 3091, 11 ], [ 6, 3091, 12 ], [ 6, 5210, 14 ], [ 6, 5262, 11 ], 
  [ 6, 5321, 6 ], [ 6, 5421, 6 ], [ 6, 5475, 6 ], [ 6, 5477, 11 ], [ 6, 5487, 11 ] ]
gap> Length(last);
50
\end{verbatim}
\end{example}

\bigskip


\begin{theorem}\label{thfac}
Let $G$ be a finite subgroup of $\GL(n,\bZ)$ and 
$M_G$ be the $G$-lattice as in Definition \ref{defMG}.\\
{\rm (i)}\ When $n \leq 3$, $M_G$ is flabby and coflabby if and only if
$M_G$ is permutation.\\
{\rm (ii)}\ When $n=4$, $M_G$ is flabby and coflabby if and only if
$M_G$ is permutation or the {\rm GAP ID} of $G$ is one of
$(4,14,2,2), (4,14,3,3), (4,14,3,4), (4,14,8,2)$.\\
$($There are $11$ conjugacy classes of subgroups of $S_4$ 
and hence $15$ flabby and coflabby $G$-lattices of rank $4$ in total.$)$
{\rm (iii)}\ When $n=5$, $M_G$ is flabby and coflabby if and only if
$M_G$ is permutation or the {\rm CARAT ID} of $G$ is one of
$(5,218,4), (5,911,4), (5,918,4), (5,931,4)$.\\
$($There are $19$ conjugacy classes of subgroups of $S_5$ 
and hence $23$ flabby and coflabby $G$-lattices of rank $5$ in total.$)$\\
{\rm (iv)}\ When $n=6$, $M_G$ is flabby and coflabby if and only if
$M_G$ is permutation or the {\rm CARAT ID} of $G$ is one of the $50$ triples 
\begin{center}
{\rm
\begin{tabular}{lllll}
$(6,159,14)$,&$(6,161,14)$,&$(6,161,28)$,&$(6,197,14)$,&$(6,226,14)$,\\
$(6,226,40)$,&$(6,231,39)$,&$(6,238,27)$,&$(6,246,21)$,&$(6,366,27)$,\\
$(6,1087,20)$,&$(6,1090,18)$,&$(6,1142,8)$,&$(6,2043,4)$,&$(6,2051,9)$,\\
$(6,2068,6)$,&$(6,2069,6)$,&$(6,2069,12)$,&$(6,2070,12)$,&$(6,2079,14)$,\\
$(6,2079,28)$,&$(6,2088,18)$,&$(6,2105,12)$,&$(6,2154,26)$,&$(6,2156,40)$,\\
$(6,2156,80)$,&$(6,2188,39)$,&$(6,2968,4)$,&$(6,2969,4)$,&$(6,2969,8)$,\\
$(6,2977,6)$,&$(6,3068,7)$,&$(6,3068,8)$,&$(6,3071,7)$,&$(6,3071,8)$,\\
$(6,3073,7)$,&$(6,3073,8)$,&$(6,3073,15)$,&$(6,3073,16)$,&$(6,3076,7)$,\\
$(6,3076,8)$,&$(6,3091,11)$,&$(6,3091,12)$,&$(6,5210,14)$,&$(6,5262,11)$,\\
$(6,5321,6)$,&$(6,5421,6)$,&$(6,5475,6)$,&$(6,5477,11)$,&$(6,5487,11)$.
\end{tabular}
}
\end{center}
\hfill\break
$($There are $56$ conjugacy classes of subgroups of $S_6$ 
and hence $106$ flabby and coflabby $G$-lattices of rank $6$ in total.$)$
\end{theorem}

\bigskip

By Theorem \ref{thEM82}, when 
any $p$-Sylow subgroup of $G$ is cyclic for odd $p$ and 
cyclic or dihedral $($including Klein's four group$)$ for $p=2$, 
$G$-lattice $M$ is flabby and coflabby if and only if $M$ is invertible. 
Moreover, for rank $M\leq 6$, $M$ is flabby and coflabby 
if and only if $M$ is stably permutation. 

\begin{theorem}\label{thCFSP}
Let $G$ be a finite subgroup of $\GL(n,\bZ)$ and 
$M_G$ be the $G$-lattice as in Definition \ref{defMG}. 
When $n \leq 6$, $M_G$ is flabby and coflabby if and only if
$M_G$ is stably permutation. 
Indeed, flabby and coflabby $G$-lattices $M_G$ which are not permutation 
as in Theorem \ref{thfac} are stably permutation as in Table $8$.
\end{theorem}

\newpage
\begin{center}
Table $8$: flabby and coflabby $G$-lattices $M=M_G$ of rank $\leq 6$ 
which are not permutation 
\vspace*{2mm}\\
{\small
\begin{tabular}{llrcl}
GAP/CARAT ID& $G$ & \multicolumn{3}{c}{$M$ is stably permutation, 
$H_i$ is the $i$th conjugacy class of subgroups of $G$}\\\hline
(4,14,2,2) & $C_6$ & $M \oplus \mathbb{Z}$ & $\simeq$ & $\mathbb{Z}[G/H_2] \oplus \mathbb{Z}[G/H_3]$ \\
(4,14,3,3) & $S_3$ & $M \oplus \mathbb{Z}$ & $\simeq$ & $\mathbb{Z}[G/H_2] \oplus \mathbb{Z}[G/H_3]$ \\
(4,14,3,4) & $S_3$ & $M \oplus \mathbb{Z}[G/H_2]$ & $\simeq$ & $\mathbb{Z}[G] \oplus \mathbb{Z}$ \\
(4,14,8,2) & $D_6$ & $M \oplus \mathbb{Z}$ & $\simeq$ & $\mathbb{Z}[G/H_6] \oplus 
\mathbb{Z}[G/H_9]$\\\hline
(5,218,4) & $S_3$ & $M \oplus \mathbb{Z}[G/H_2]$ & $\simeq$ & $\mathbb{Z}[G] \oplus 
\mathbb{Z}^{\oplus 2}$ \\
(5,911,4) & $S_5$ & $M \oplus \mathbb{Z}$ & $\simeq$ & $\mathbb{Z}[G/H_{17}] \oplus 
\mathbb{Z}$ \\
(5,918,4) & $F_{20}$ & $M \oplus \mathbb{Z}$ & $\simeq$ & $\mathbb{Z}[G/H_3] \oplus 
\mathbb{Z}$ \\
(5,931,4) & $A_5$ & $M \oplus \mathbb{Z}$ & $\simeq$ & $\mathbb{Z}[G/H_8] \oplus 
\mathbb{Z}$\\\hline
(6,159,14) & $C_{12}$ & $M \oplus \mathbb{Z}$ & $\simeq$ & $\mathbb{Z}[G/H_3] \oplus 
\mathbb{Z}[G/H_4]$ \\
(6,161,14) & $Q_{12}$ & $M \oplus \mathbb{Z}[G/H_4] \oplus \mathbb{Z}[G/H_5]$ & $\simeq$ & 
$\mathbb{Z}[G/H_2] \oplus \mathbb{Z}[G/H_3] \oplus \mathbb{Z}$ \\
(6,161,28) & $Q_{12}$ & $M \oplus \mathbb{Z}$ & $\simeq$ & $\mathbb{Z}[G/H_3] \oplus 
\mathbb{Z}[G/H_4]$ \\
(6,197,14) & $D_4 \times C_3$ & $M \oplus \mathbb{Z}$ & $\simeq$ & $\mathbb{Z}[G/H_{10}] 
\oplus \mathbb{Z}[G/H_{12}]$ \\
(6,226,14) & $C_3 \rtimes D_4$ & $M \oplus \mathbb{Z}[G/H_{12}] \oplus \mathbb{Z}[G/H_{13}]$ 
& $\simeq$ & $\mathbb{Z}[G/H_6] \oplus \mathbb{Z}[G/H_{10}] \oplus \mathbb{Z}$ \\
(6,226,40) & $C_3 \rtimes D_4$ & $M \oplus \mathbb{Z}$ & $\simeq$ & $\mathbb{Z}[G/H_{10}] 
\oplus \mathbb{Z}[G/H_{12}]$ \\
(6,231,39) & $D_{12}$ & $M \oplus \mathbb{Z}$ & $\simeq$ & $\mathbb{Z}[G/H_{10}] \oplus 
\mathbb{Z}[G/H_{12}]$ \\
(6,238,27) & $C_3 \rtimes D_4$ & $M \oplus \mathbb{Z}$ & $\simeq$ & $\mathbb{Z}[G/H_{11}] 
\oplus \mathbb{Z}[G/H_{12}]$ \\
(6,246,21) & $S_3 \times C_4$ & $M \oplus \mathbb{Z}$ & $\simeq$ & $\mathbb{Z}[G/H_{11}] 
\oplus \mathbb{Z}[G/H_{12}]$ \\
(6,366,27) & $D_4 \times S_3$ & $M \oplus \mathbb{Z}$ & $\simeq$ & $\mathbb{Z}[G/H_{43}] 
\oplus \mathbb{Z}[G/H_{46}]$ \\
(6,1087,20) & $S_3 \times C_3$ & $M \oplus \mathbb{Z}[G/H_7]$ & $\simeq$ & $\mathbb{Z}[G/H_4] 
\oplus \mathbb{Z}[G/H_6]$ \\
(6,1090,18) & $C_3^2 \rtimes C_2$ & $M \oplus \mathbb{Z}[G/H_9]$ & $\simeq$ 
& $\mathbb{Z}[G/H_5] \oplus \mathbb{Z}[G/H_{10}]$ \\
(6,1142,8) & $S_3 \times S_3$ & $M \oplus \mathbb{Z}[G/H_{17}]$ & $\simeq$ 
& $\mathbb{Z}[G/H_{10}] \oplus \mathbb{Z}[G/H_{18}]$ \\
(6,2043,4) & $S_3$ & $M \oplus \mathbb{Z}[G/H_2]$ & $\simeq$ & $\mathbb{Z}[G] \oplus 
\mathbb{Z}^{\oplus 3}$ \\
(6,2051,9) & $D_6$ & $M \oplus \mathbb{Z}$ & $\simeq$ & $\mathbb{Z}[G/H_6] \oplus 
\mathbb{Z}[G/H_9]^{\oplus 2}$ \\
(6,2068,6) & $C_6$ & $M \oplus \mathbb{Z}$ & $\simeq$ & $\mathbb{Z}[G/H_2] \oplus 
\mathbb{Z}[G/H_3]^{\oplus 2}$ \\
(6,2069,6) & $S_3$ & $M \oplus \mathbb{Z}[G/H_2]$ & $\simeq$ & $\mathbb{Z}[G] \oplus 
\mathbb{Z}[G/H_3] \oplus \mathbb{Z}$ \\
(6,2069,12) & $S_3$ & $M \oplus \mathbb{Z}$ & $\simeq$ & $\mathbb{Z}[G/H_2] \oplus 
\mathbb{Z}[G/H_3]^{\oplus 2}$ \\
(6,2070,12) & $C_6 \times C_2$ & $M \oplus \mathbb{Z}$ & $\simeq$ & $\mathbb{Z}[G/H_6] 
\oplus \mathbb{Z}[G/H_8] \oplus \mathbb{Z}[G/H_9]$ \\
(6,2079,14) & $D_6$ & $M \oplus \mathbb{Z}[G/H_6]$ & $\simeq$ & $\mathbb{Z}[G/H_2] 
\oplus \mathbb{Z}[G/H_9] \oplus \mathbb{Z}$ \\
(6,2079,28) & $D_6$ & $M \oplus \mathbb{Z}$ & $\simeq$ & $\mathbb{Z}[G/H_6] \oplus 
\mathbb{Z}[G/H_7] \oplus \mathbb{Z}[G/H_9]$ \\
(6,2088,18) & $D_6$ & $M \oplus \mathbb{Z}$ & $\simeq$ & $\mathbb{Z}[G/H_6] \oplus 
\mathbb{Z}[G/H_8] \oplus \mathbb{Z}[G/H_9]$ \\
(6,2105,12) & $D_6 \times C_2$ & $M \oplus \mathbb{Z}$ & $\simeq$ & $\mathbb{Z}[G/H_{24}] 
\oplus \mathbb{Z}[G/H_{29}] \oplus \mathbb{Z}[G/H_{31}]$ \\
(6,2154,26) & $C_6 \times C_2$ & $M \oplus \mathbb{Z}$ & $\simeq$ & $\mathbb{Z}[G/H_5] 
\oplus \mathbb{Z}[G/H_6]$ \\
(6,2156,40) & $D_6$ & $M \oplus \mathbb{Z}[G/H_6] \oplus \mathbb{Z}[G/H_9]$ & $\simeq$ 
& $\mathbb{Z}[G/H_2] \oplus \mathbb{Z}[G/H_5] \oplus \mathbb{Z}$ \\
(6,2156,80) & $D_6$ & $M \oplus \mathbb{Z}$ & $\simeq$ & $\mathbb{Z}[G/H_5] \oplus 
\mathbb{Z}[G/H_6]$ \\
(6,2188,39) & $D_6 \times C_2$ & $M \oplus \mathbb{Z}$ & $\simeq$ & $\mathbb{Z}[G/H_{19}] 
\oplus \mathbb{Z}[G/H_{24}]$ \\
(6,2968,4) & $C_{10}$ & $M \oplus \mathbb{Z}$ & $\simeq$ & $\mathbb{Z}[G/H_2] \oplus 
\mathbb{Z}[G/H_3]$ \\
(6,2969,4) & $D_5$ & $M \oplus \mathbb{Z}$ & $\simeq$ & $\mathbb{Z}[G/H_2] \oplus 
\mathbb{Z}[G/H_3]$ \\
(6,2969,8) & $D_5$ & $M \oplus \mathbb{Z}[G/H_2]$ & $\simeq$ & $\mathbb{Z}[G] \oplus 
\mathbb{Z}$ \\
(6,2977,6) & $D_{10}$ & $M \oplus \mathbb{Z}$ & $\simeq$ & $\mathbb{Z}[G/H_5] \oplus 
\mathbb{Z}[G/H_8]$ \\
(6,3068,7) & $A_5 \times C_2$ & $M \oplus \mathbb{Z}$ & $\simeq$ & $\mathbb{Z}[G/H_{20}] 
\oplus \mathbb{Z}[G/H_{21}]$ \\
(6,3068,8) & $A_5 \times C_2$ & $M \oplus \mathbb{Z}$ & $\simeq$ & $\mathbb{Z}[G/H_{20}] 
\oplus \mathbb{Z}[G/H_{21}]$ \\
(6,3071,7) & $S_5$ & $M \oplus \mathbb{Z}$ & $\simeq$ & $\mathbb{Z}[G/H_{17}] \oplus 
\mathbb{Z}[G/H_{18}]$ \\
(6,3071,8) & $S_5$ & $M \oplus \mathbb{Z}$ & $\simeq$ & $\mathbb{Z}[G/H_{17}] \oplus 
\mathbb{Z}[G/H_{18}]$ \\
(6,3073,7) & $F_{20}$ & $M \oplus \mathbb{Z}$ & $\simeq$ & $\mathbb{Z}[G/H_3] \oplus 
\mathbb{Z}[G/H_5]$ \\
(6,3073,8) & $F_{20}$ & $M \oplus \mathbb{Z}$ & $\simeq$ & $\mathbb{Z}[G/H_3] \oplus 
\mathbb{Z}[G/H_5]$ \\
(6,3073,15) & $F_{20}$ & $M \oplus \mathbb{Z}[G/H_3]$ & $\simeq$ & $\mathbb{Z}[G/H_2] \oplus 
\mathbb{Z}$ \\
(6,3073,16) & $F_{20}$ & $M \oplus \mathbb{Z}[G/H_3]$ & $\simeq$ & $\mathbb{Z}[G/H_2] \oplus 
\mathbb{Z}$ \\
(6,3076,7) & $S_5 \times C_2$ & $M \oplus \mathbb{Z}$ & $\simeq$ & $\mathbb{Z}[G/H_{52}] \oplus 
\mathbb{Z}[G/H_{56}]$ \\
(6,3076,8) & $S_5 \times C_2$ & $M \oplus \mathbb{Z}$ & $\simeq$ & $\mathbb{Z}[G/H_{52}] \oplus 
\mathbb{Z}[G/H_{56}]$ \\
(6,3091,11) & $F_{20}\times C_2$ & $M \oplus \mathbb{Z}$ & $\simeq$ & $\mathbb{Z}[G/H_9] \oplus 
\mathbb{Z}[G/H_{15}]$ \\
(6,3091,12) & $F_{20} \times C_2$ & $M \oplus \mathbb{Z}$ & $\simeq$ & $\mathbb{Z}[G/H_9] \oplus 
\mathbb{Z}[G/H_{15}]$ \\
(6,5210,14) & $A_4$ & $M \oplus \mathbb{Z}$ & $\simeq$ & $\mathbb{Z}[G/H_3] \oplus 
\mathbb{Z}[G/H_4]$ \\
(6,5262,11) & $S_4 \times S_3$ & $M \oplus \mathbb{Z}$ & $\simeq$ & $\mathbb{Z}[G/H_{64}] 
\oplus \mathbb{Z}[G/H_{65}]$ \\
(6,5321,6) & $S_4$ & $M \oplus \mathbb{Z}$ & $\simeq$ & $\mathbb{Z}[G/H_8] \oplus 
\mathbb{Z}[G/H_9]$ \\
(6,5421,6) & $A_4 \times C_3$ & $M \oplus \mathbb{Z}$ & $\simeq$ & $\mathbb{Z}[G/H_9] 
\oplus \mathbb{Z}[G/H_{13}]$ \\
(6,5475,6) & $S_4 \times C_3$ & $M \oplus \mathbb{Z}$ & $\simeq$ & $\mathbb{Z}[G/H_{20}] 
\oplus \mathbb{Z}[G/H_{21}]$ \\
(6,5477,11) & $A_4 \times S_3$ & $M \oplus \mathbb{Z}$ & $\simeq$ & $\mathbb{Z}[G/H_{22}] 
\oplus \mathbb{Z}[G/H_{24}]$ \\
(6,5487,11) & $(A_4 \times A_3) \rtimes C_2$ & $M \oplus \mathbb{Z}$ & $\simeq$ 
& $\mathbb{Z}[G/H_{24}] \oplus \mathbb{Z}[G/H_{28}]$\\\hline
\end{tabular}
}
\end{center}

Note that in Table $8$, 
$H_i$ is the $i$-th conjugacy class of subgroups of $G$ 
which is determined by the function {\tt ConjugacyClassesSubgroups2} 
in GAP (see Section \ref{seKSfail}). 

All cases can be done by using 
Method I in Section \ref{ss55} and Method II in Section \ref{ss56}. 
We give the following $3$ typical examples instead of the full proof 
of Theorem \ref{thCFSP}. 

\begin{example}[{\bf Method I} (1)]\label{exM11}
We use {\tt PossibilityOfStablyPermutationM} as in 
Algorithm \ref{Alg4} 
to get a possibility of the isomorphism.
Then we use {\tt StablyPermutationMCheck} 
as in Algorithm \ref{Alg5}: Method I to get the actual isomorphism. 

Let $G\simeq C_6$ be the group of the GAP ID $(4,14,2,2)$ 
which is generated by the matrix
$\sigma=
{\scriptsize 
\begin{pmatrix}
0 & 1 & 0 & 1 \\
1 & 0 & 0 & 1 \\
0 & 0 & 0 & 1 \\
0 & 0 &-1 &-1
\end{pmatrix}}$.
There exist $4$ conjugacy classes of subgroups $\{1\}$, $H_2$, $H_3$ and $G$ 
of $G$ which are isomorphic to cyclic groups of order $1,2,3$ and $6$.
Corresponding $G$-lattices $M_G$ are isomorphic to $\bZ[G]$,
$\bZ[G/\langle \sigma^3 \rangle]$, $\bZ[G/\langle \sigma^2 \rangle]$ and 
$\bZ$ of rank $6,3,2$ and $1$ respectively. 

Let $G_1$ (resp. $G_2$) be the matrix representation group of the action of $G$ 
on $\bZ\oplus M_G$ (resp. $\bZ[G/\langle \sigma^3 \rangle] \oplus
\bZ[G/\langle \sigma^2 \rangle]$. 
{\tt StablyPermutaionMCheck(G,Nlist(l),Plist(l))} shows 
that $G_1P=PG_2$ with the isomorphism 
\[
\bZ \oplus M_G \simeq
\bZ[G/\langle \sigma^3 \rangle] \oplus
\bZ[G/\langle \sigma^2 \rangle]
\]
where 
\[
{\scriptsize 
P=
\begin{pmatrix}
 1 & 1 & 1 & 1 & 1 \\
 0 & 0 &-1 & 0 &-1 \\
 0 & 0 &-1 &-1 & 0 \\
 1 & 0 &-1 & 0 & 0 \\
 0 &-1 & 1 & 0 & 0
\end{pmatrix}}.
\]

Similar examples for the groups $S_3$, $S_3$ and 
$D_6$ of the CARAT IDs $(4,14,3,3)$, $(4,14,3,4)$ and $(4,14,8,2)$ 
are given below.

\bigskip

\begin{verbatim}
Read("FlabbyResolution.gap");

gap> G:=MatGroupZClass(4,14,2,2);; # G=C6
gap> ll:=PossibilityOfStablyPermutationM(G);
[ [ 0, 1, 1, -1, -1 ] ]
gap> l:=ll[1];
[ 0, 1, 1, -1, -1 ]
gap> List(ConjugacyClassesSubgroups2(G),x->StructureDescription(Representative(x)));
[ "1", "C2", "C3", "C6" ]
gap> StablyPermutationMCheck(G,Nlist(l),Plist(l));
[ [ 1, 1, 1, 1, 1 ], 
  [ 0, 0, -1, 0, -1 ], 
  [ 0, 0, -1, -1, 0 ], 
  [ 1, 0, -1, 0, 0 ], 
  [ 0, -1, 1, 0, 0 ] ]

gap> G:=MatGroupZClass(4,14,3,3);; # G=S3
gap> ll:=PossibilityOfStablyPermutationM(G);
[ [ 0, 1, 1, -1, -1 ] ]
gap> l:=ll[1];
[ 0, 1, 1, -1, -1 ]
gap> List(ConjugacyClassesSubgroups2(G),x->StructureDescription(Representative(x)));
[ "1", "C2", "C3", "S3" ]
gap> StablyPermutationMCheck(G,Nlist(l),Plist(l));
[ [ 1, 1, 1, 1, 1 ], 
  [ 0, 0, -1, 0, -1 ], 
  [ 0, 0, -1, -1, 0 ], 
  [ 1, 0, -1, 0, 0 ], 
  [ 0, -1, 1, 0, 0 ] ]

gap> G:=MatGroupZClass(4,14,3,4);; # G=S3
gap> ll:=PossibilityOfStablyPermutationM(G);
[ [ 1, -1, 0, 1, -1 ] ]
gap> l:=ll[1];
[ 1, -1, 0, 1, -1 ]
gap> List(ConjugacyClassesSubgroups2(G),x->StructureDescription(Representative(x)));
[ "1", "C2", "C3", "S3" ]
gap> StablyPermutationMCheck(G,Nlist(l),Plist(l));                      
[ [ 0, 0, -1, 0, 0, -1, -1 ], 
  [ -1, 0, 0, 0, -1, 0, -1 ], 
  [ 0, -1, 0, -1, 0, 0, -1 ], 
  [ 0, 0, -1, -1, 0, -1, -1 ], 
  [ 1, 1, 0, 0, 1, 0, 1 ], 
  [ 1, 0, -1, -1, 1, 0, 0 ], 
  [ 0, -1, 1, 0, -1, 1, 0 ] ]

gap> G:=MatGroupZClass(4,14,8,2);; # G=D6
gap> ll:=PossibilityOfStablyPermutationM(G);
[ [ 1, -1, -1, -1, -1, 0, 1, 1, -1, 0, 2 ], 
  [ 0, 0, 0, 0, 0, 1, 0, 0, 1, -1, -1 ] ]
gap> l:=ll[Length(ll)];
[ 0, 0, 0, 0, 0, 1, 0, 0, 1, -1, -1 ]
gap> List(ConjugacyClassesSubgroups2(G),x->StructureDescription(Representative(x)));
[ "1", "C2", "C2", "C2", "C3", "C2 x C2", "C6", "S3", "S3", "D12" ]
gap> StablyPermutationMCheck(G,Nlist(l),Plist(l));
[ [ 1, 1, 1, 1, 1 ], 
  [ 0, 0, -1, 0, -1 ], 
  [ 0, 0, -1, -1, 0 ], 
  [ 1, 0, -1, 0, 0 ], 
  [ 0, -1, 1, 0, 0 ] ]
\end{verbatim}
\end{example}

\bigskip

\begin{example}[{\bf Method I} (2)]\label{exMI2}
Before applying {\tt StablePermutationMCheck} in Method I, 
in some cases, we have to add some more $G$-lattices 
to make both hand side of $G$-lattices isomorphic. 
We will show that for the group $G\simeq S_5$ 
of the CARAT ID $(5,911,4)$, $M_G\not\simeq \bZ[S_5/S_4]$ but 
$M_G \oplus \bZ \simeq\bZ[S_5/S_4] \oplus \bZ$.

\bigskip

\begin{verbatim}
gap> Read("caratnumber.gap");
gap> Read("FlabbyResolution.gap");

gap> G:=CaratMatGroupZClass(5,911,4);; # G=S5
gap> ll:=PossibilityOfStablyPermutationM(G);
[ [ 1, 0, 0, -1, 0, 0, -4, 0, -2, 1, 2, 0, -1, -1, 0, 4, 0, 1, -4, 4 ], 
  [ 0, 1, 0, 0, 0, -1, -1, 0, -1, 0, 0, 0, 0, 0, 1, 1, 0, 0, -1, 1 ], 
  [ 0, 0, 1, 0, 0, 0, -2, 0, -1, 0, 1, 0, -1, -1, 0, 2, 0, 1, -2, 2 ], 
  [ 0, 0, 0, 0, 1, 2, -2, 0, -2, 1, 2, -2, -1, -2, -2, 2, 0, 1, -2, 4 ], 
  [ 0, 0, 0, 0, 0, 0, 0, 0, 0, 0, 0, 0, 0, 0, 0, 0, 1, 0, 0, -1 ] ]
gap> l:=ll[Length(ll)];
[ 0, 0, 0, 0, 0, 0, 0, 0, 0, 0, 0, 0, 0, 0, 0, 0, 1, 0, 0, -1 ]
gap> List(ConjugacyClassesSubgroups2(G),x->StructureDescription(Representative(x)));
[ "1", "C2", "C2", "C3", "C2 x C2", "C2 x C2", "C4", "C5", "S3", "S3", "C6", "D8", 
  "D10", "A4", "D12", "C5 : C4", "S4", "A5", "S5" ]
gap> StablyPermutationMCheck(G,Nlist(l),Plist(l));
fail
gap> l2:=IdentityMat(Length(l))[Length(l)-1];
[ 0, 0, 0, 0, 0, 0, 0, 0, 0, 0, 0, 0, 0, 0, 0, 0, 0, 0, 1, 0 ]
gap> StablyPermutationMCheck(G,Nlist(l)+l2,Plist(l)+l2);
[ [ 2, 2, 2, 2, 2, 3 ], 
  [ 0, -1, 0, -1, -1, -1 ], 
  [ 0, 1, -1, 0, 0, 0 ], 
  [ 1, 1, 0, 0, 1, 1 ], 
  [ 1, 0, 1, 1, 0, 1 ], 
  [ -1, -2, -1, -1, -1, -2 ] ]
\end{verbatim}
\end{example}

\bigskip

\begin{example}[{\bf Method II}]
When {\tt StablyPermutationMCheck} does not return any result
in an appropriate time, we may use the command 
{\tt StablyPermutatinoCheakP}. 

\bigskip

\begin{verbatim}
gap> Read("caratnumber.gap");
gap> Read("FlabbyResolution.gap");

gap> G:=CaratMatGroupZClass(6,161,14);; # G=Q12
gap> ll:=PossibilityOfStablyPermutationM(G);
[ [ 0, 1, 1, -1, -1, 1, -1 ] ]
gap> l:=ll[1];
[ 0, 1, 1, -1, -1, 1, -1 ]
gap> List(ConjugacyClassesSubgroups2(G),x->StructureDescription(Representative(x)));
[ "1", "C2", "C3", "C4", "C6", "C3 : C4" ]
gap> gg:=StablyPermutationMCheckGen(G,Nlist(l),Plist(l));
[ [ [ [ 1, 0, 0, 0, 0, 0, 0, 0, 0, 0, 0 ], 
      [ 0, 0, 1, 0, 0, 0, 0, 0, 0, 0, 0 ], 
      [ 0, 1, 0, 0, 0, 0, 0, 0, 0, 0, 0 ], 
      [ 0, 0, 0, 0, 1, 0, 0, 0, 0, 0, 0 ], 
      [ 0, 0, 0, 1, 0, 0, 0, 0, 0, 0, 0 ], 
      [ 0, 0, 0, 0, 0, 0, 1, 1, 1, 0, -1 ], 
      [ 0, 0, 0, 0, 0, -1, 0, 1, 0, 0, 0 ], 
      [ 0, 0, 0, 0, 0, 0, 0, 1, 0, 0, 0 ], 
      [ 0, 0, 0, 0, 0, 0, 0, -1, 0, 1, 0 ], 
      [ 0, 0, 0, 0, 0, 0, 0, 1, 1, 0, 0 ], 
      [ 0, 0, 0, 0, 0, 0, 0, 1, 0, 0, -1 ] ], 
    [ [ 0, 0, 1, 0, 0, 0, 0, 0, 0, 0, 0 ], 
      [ 1, 0, 0, 0, 0, 0, 0, 0, 0, 0, 0 ], 
      [ 0, 1, 0, 0, 0, 0, 0, 0, 0, 0, 0 ], 
      [ 0, 0, 0, 1, 0, 0, 0, 0, 0, 0, 0 ], 
      [ 0, 0, 0, 0, 1, 0, 0, 0, 0, 0, 0 ], 
      [ 0, 0, 0, 0, 0, 1, 0, -1, 0, 0, 0 ], 
      [ 0, 0, 0, 0, 0, 0, 1, -1, 0, 0, 0 ], 
      [ 0, 0, 0, 0, 0, 0, 0, -2, -1, 1, 1 ], 
      [ 0, 0, 0, 0, 0, 0, 0, 1, 1, 0, 0 ], 
      [ 0, 0, 0, 0, 0, 0, 0, -1, 0, 1, 0 ], 
      [ 0, 0, 0, 0, 0, 0, 0, -1, 0, 0, 1 ] ] ], 
  [ [ [ 0, 0, 0, 0, 1, 0, 0, 0, 0, 0, 0 ], 
      [ 0, 0, 0, 0, 0, 1, 0, 0, 0, 0, 0 ], 
      [ 0, 0, 0, 1, 0, 0, 0, 0, 0, 0, 0 ], 
      [ 0, 0, 1, 0, 0, 0, 0, 0, 0, 0, 0 ], 
      [ 1, 0, 0, 0, 0, 0, 0, 0, 0, 0, 0 ], 
      [ 0, 1, 0, 0, 0, 0, 0, 0, 0, 0, 0 ], 
      [ 0, 0, 0, 0, 0, 0, 0, 1, 0, 0, 0 ], 
      [ 0, 0, 0, 0, 0, 0, 0, 0, 0, 1, 0 ], 
      [ 0, 0, 0, 0, 0, 0, 1, 0, 0, 0, 0 ], 
      [ 0, 0, 0, 0, 0, 0, 0, 0, 1, 0, 0 ], 
      [ 0, 0, 0, 0, 0, 0, 0, 0, 0, 0, 1 ] ], 
    [ [ 0, 0, 1, 0, 0, 0, 0, 0, 0, 0, 0 ], 
      [ 1, 0, 0, 0, 0, 0, 0, 0, 0, 0, 0 ], 
      [ 0, 1, 0, 0, 0, 0, 0, 0, 0, 0, 0 ], 
      [ 0, 0, 0, 0, 1, 0, 0, 0, 0, 0, 0 ], 
      [ 0, 0, 0, 0, 0, 1, 0, 0, 0, 0, 0 ], 
      [ 0, 0, 0, 1, 0, 0, 0, 0, 0, 0, 0 ], 
      [ 0, 0, 0, 0, 0, 0, 1, 0, 0, 0, 0 ], 
      [ 0, 0, 0, 0, 0, 0, 0, 1, 0, 0, 0 ], 
      [ 0, 0, 0, 0, 0, 0, 0, 0, 1, 0, 0 ], 
      [ 0, 0, 0, 0, 0, 0, 0, 0, 0, 1, 0 ], 
      [ 0, 0, 0, 0, 0, 0, 0, 0, 0, 0, 1 ] ] ] ]
gap> bp:=StablyPermutationMCheckP(G,Nlist(l),Plist(l));;
gap> Length(bp);
19
gap> Length(bp[1]); # rank of the both sides is 11
11
gap> rs:=RandomSource(IsMersenneTwister);
<RandomSource in IsMersenneTwister>
gap> rr:=List([1..1000],x->List([1..19],y->Random(rs,[0,1])));;
gap> Filtered(rr,x->Determinant(x*bp)^2=1);
[ [ 0, 1, 0, 0, 1, 1, 0, 1, 0, 0, 0, 1, 0, 1, 0, 0, 1, 0, 1 ], 
  [ 0, 1, 1, 0, 1, 1, 0, 0, 1, 1, 1, 0, 1, 0, 0, 0, 1, 0, 1 ] ]
gap> p:=last[1]*bp;
[ [ 0, 1, 0, 0, 0, 1, 0, 0, 0, 0, 1 ], 
  [ 0, 0, 1, 0, 1, 0, 0, 0, 0, 0, 1 ], 
  [ 1, 0, 0, 1, 0, 0, 0, 0, 0, 0, 1 ], 
  [ 1, 1, 1, 0, 0, 0, 1, 0, 0, 1, 0 ], 
  [ 0, 0, 0, 1, 1, 1, 0, 1, 1, 0, 0 ], 
  [ 0, 1, 0, 1, 0, 1, 0, 0, 1, 0, 1 ], 
  [ -1, 0, -1, 0, -1, 0, 0, 0, 0, -1, -1 ], 
  [ -1, 1, 0, 0, -1, 1, 0, 0, 0, 0, 0 ], 
  [ 1, 0, 1, 1, 2, 1, 0, 1, 1, 0, 2 ], 
  [ 1, 2, 1, 1, 0, 1, 1, 0, 0, 1, 2 ], 
  [ -1, 0, -1, 1, 0, 1, -1, 1, 1, -1, 0 ] ]
gap> Determinant(p);
1
gap> List(gg[1],x->p^-1*x*p)=gg[2]; 
true
gap> StablyPermutationMCheckMat(G,Nlist(l),Plist(l),p);
true
\end{verbatim}
\end{example}

\bigskip

As an application of Example \ref{exM11}, we see that 
Krull-Schmidt theorem fails for permutation $D_6$-lattices 
by constructing explicit isomorphism. 
Let $M_G$ be the $G$-lattice where $G\simeq D_6\leq \GL(4,\bZ)$ 
is the group of the GAP ID $(4,14,8,2)$. 
Example \ref{exM11} confirms that the isomorphism
\begin{align}
M_G\oplus\bZ\simeq \bZ[G/H_6]\oplus\bZ[G/H_9]\label{eqM1}
\end{align}
holds. 
Note that $\bZ[G/H]$ is indecomposable for any $H\leq G$ (\cite[Theorem 32.4]{CR87}). 
We also see by Example \ref{exM11} that it is possible that 
\begin{align}
M_G^{\oplus 2}\oplus\bZ[G]\oplus\bZ[G/H_7]\oplus\bZ[G/H_8]
\simeq \bZ[G/H_2]\oplus\bZ[G/H_3]\oplus\bZ[G/H_4]
\oplus\bZ[G/H_5]\oplus\bZ[G/H_9].\label{eqM2}
\end{align}
If the isomorphism (\ref{eqM2}) holds, then it follows from (\ref{eqM1}) and 
(\ref{eqM2}) that the Krull-Schmidt theorem fails for the 
permutation $D_6$-lattices (the rank of the both sides is 
$12+2\times 3+2+2+2=6+6+6+4+2\times 1=24$):
\begin{align*}
\bZ[G]\oplus\bZ[G/H_6]^{\oplus 2}\oplus\bZ[G/H_7]\oplus\bZ[G/H_8]\oplus\bZ[G/H_9]
\simeq \bZ[G/H_2]\oplus\bZ[G/H_3]\oplus\bZ[G/H_4]
\oplus\bZ[G/H_5]\oplus\bZ^{\oplus 2}.
\end{align*}
We may check that this isomorphism actually holds 
(see Example \ref{exKSfailD6} below), namely we have:
\begin{proposition}[The Krull-Schmidt theorem fails for 
permutation $D_6$-lattices]\label{propKSfailD6}
Let $D_6$ be the dihedral group of order $12$ and 
$\{1\}$, $C_2^{(1)}$, $C_2^{(2)}$, $C_2^{(3)}$, $C_3$, $C_2^2$, 
$C_6$, $S_3^{(1)}$, $S_3^{(2)}$ and $D_6$ be the 
conjugacy classes of subgroups of $D_6$. 
Then the following isomorphism holds: 
\begin{align*}
& ~{} \bZ[D_6]\oplus\bZ[D_6/C_2^2]^{\oplus 2}\oplus\bZ[D_6/C_6]
\oplus\bZ[D_6/S_3^{(1)}]\oplus\bZ[D_6/S_3^{(2)}]\\
\simeq & ~{} \bZ[D_6/C_2^{(1)}]\oplus\bZ[D_6/C_2^{(2)}]
\oplus\bZ[D_6/C_2^{(3)}]\oplus\bZ[D_6/C_3]\oplus\bZ^{\oplus 2}.
\end{align*}
\end{proposition}

%

\bigskip

\begin{example}[Verification of Proposition \ref{propKSfailD6}: the Krull-Schmidt theorem fails for 
permutation $D_6$-lattices]\label{exKSfailD6}
{}~\\
\begin{verbatim}
Read("FlabbyResolution.gap");

gap> G:=MatGroupZClass(4,14,8,2);; # G=D6
gap> ll:=PossibilityOfStablyPermutationM(G);
[ [ 1, -1, -1, -1, -1, 0, 1, 1, -1, 0, 2 ], 
  [ 0, 0, 0, 0, 0, 1, 0, 0, 1, -1, -1 ] ]
gap> l:=ll[1]+2*ll[2];
[ 1, -1, -1, -1, -1, 2, 1, 1, 1, -2, 0 ]
gap> List(ConjugacyClassesSubgroups2(G),x->StructureDescription(Representative(x)));
[ "1", "C2", "C2", "C2", "C3", "C2 x C2", "C6", "S3", "S3", "D12" ]
gap> bp:=StablyPermutationMCheckP(G,Nlist(l),Plist(l));; 
gap> Length(bp);
68
gap> Length(bp[1]); # rank of the both sides is 24
24

# after some efforts we may get
gap> n:=[ 
> 1, -1, -1, -1, -1, -1, -1, -1, 0, -1, 0, 1, 0, 1, 1, 0, 0, 0, -1, 1, 1, -1, 1, 0, -1, 
> 1, -1, 0, -1, 0, 1, 1, 0, 1, 0, 0, -1, 1, 1, -1, -1, 0, -1, 0, -1, 1, 0, -1, -1, 1, 
> 1, 0, 0, 1, 0, 1, 1, 0, 1, 0, 1, 1, 1, 1, 0, -1, -1, 0 ];;
gap> p:=n*bp;
[ [ 1, -1, -1, -1, -1, -1, 1, -1, -1, -1, -1, -1, -1, -1, 0, -1, 0, 1, 0, 1, 1, 1, 0, 0 ], 
  [ -1, 1, -1, -1, -1, -1, -1, 1, -1, -1, -1, -1, -1, -1, 0, 0, -1, 1, 1, 0, 1, 1, 0, 0 ], 
  [ -1, -1, 1, -1, -1, -1, -1, -1, 1, -1, -1, -1, 0, -1, -1, 1, 0, -1, 1, 0, 1, 1, 0, 0 ], 
  [ -1, -1, -1, 1, -1, -1, -1, -1, -1, 1, -1, -1, 0, -1, -1, 1, -1, 0, 0, 1, 1, 1, 0, 0 ], 
  [ -1, -1, -1, -1, 1, -1, -1, -1, -1, -1, 1, -1, -1, 0, -1, 0, 1, -1, 0, 1, 1, 1, 0, 0 ], 
  [ -1, -1, -1, -1, -1, 1, -1, -1, -1, -1, -1, 1, -1, 0, -1, -1, 1, 0, 1, 0, 1, 1, 0, 0 ], 
  [ 0, 0, -1, 1, 1, -1, -1, 1, 1, -1, 0, 0, 1, 0, 0, -1, 1, 1, -1, -1, 0, -1, 0, 0 ], 
  [ 0, 0, 1, -1, -1, 1, 1, -1, 0, 0, 1, -1, 0, 1, 0, 1, -1, 1, -1, -1, -1, 0, 0, 0 ], 
  [ -1, 1, 0, 0, -1, 1, 1, -1, -1, 1, 0, 0, 0, 0, 1, 1, 1, -1, -1, -1, -1, 0, 0, 0 ], 
  [ 1, -1, 0, 0, 1, -1, 0, 0, 1, -1, -1, 1, 0, 1, 0, 1, -1, 1, -1, -1, 0, -1, 0, 0 ], 
  [ 1, -1, -1, 1, 0, 0, -1, 1, 0, 0, -1, 1, 0, 0, 1, 1, 1, -1, -1, -1, 0, -1, 0, 0 ], 
  [ -1, 1, 1, -1, 0, 0, 0, 0, -1, 1, 1, -1, 1, 0, 0, -1, 1, 1, -1, -1, -1, 0, 0, 0 ], 
  [ 1, 1, 0, 0, 1, 1, 1, 0, 0, 0, 0, 1, -1, 1, 1, 1, -1, -1, -1, -1, 0, 0, -1, 0 ], 
  [ 1, 1, 1, 1, 0, 0, 0, 1, 0, 1, 0, 0, 1, -1, 1, -1, 1, -1, -1, -1, 0, 0, -1, 0 ], 
  [ 0, 0, 1, 1, 1, 1, 0, 0, 1, 0, 1, 0, 1, 1, -1, -1, -1, 1, -1, -1, 0, 0, -1, 0 ], 
  [ 1, 0, 0, 0, 0, 1, 1, 1, 0, 0, 1, 1, -1, 1, 1, 1, -1, -1, -1, -1, 0, 0, 0, -1 ], 
  [ 0, 1, 0, 1, 0, 0, 1, 1, 1, 1, 0, 0, 1, -1, 1, -1, 1, -1, -1, -1, 0, 0, 0, -1 ], 
  [ 0, 0, 1, 0, 1, 0, 0, 0, 1, 1, 1, 1, 1, 1, -1, -1, -1, 1, -1, -1, 0, 0, 0, -1 ], 
  [ -1, 1, 1, -1, -1, 1, 0, -1, -1, 0, 0, -1, -1, -1, -1, 1, 1, 1, 1, 0, 0, 1, 0, 1 ], 
  [ 1, -1, -1, 1, 1, -1, -1, 0, 0, -1, -1, 0, -1, -1, -1, 1, 1, 1, 0, 1, 1, 0, 0, 1 ], 
  [ 0, -1, -1, 0, 0, -1, -1, 1, 1, -1, -1, 1, -1, -1, -1, 1, 1, 1, 1, 0, 1, 0, 1, 0 ], 
  [ -1, 0, 0, -1, -1, 0, 1, -1, -1, 1, 1, -1, -1, -1, -1, 1, 1, 1, 0, 1, 0, 1, 1, 0 ], 
  [ 1, 1, 1, 1, 1, 1, 1, 1, 1, 1, 1, 1, 0, 0, 0, 1, 1, 1, 0, 0, 1, 1, 1, 1 ], 
  [ 1, 1, 1, 1, 1, 1, 1, 1, 1, 1, 1, 1, 1, 1, 1, 0, 0, 0, -1, -1, -1, -1, 0, 0 ] ]
gap> Determinant(p);
1
gap> StablyPermutationMCheckMat(G,Nlist(l),Plist(l),p);
true
\end{verbatim}
\end{example}

\bigskip

%
\section{$H^1(G,[M_G]^{fl})=0$ for any Bravais group $G$ of dimension $n\leq 6$}\label{seBravais}
\bigskip

Let $G$ be a finite subgroup of $\GL(n,\bZ)$ and 
$M_G$ be the $G$-lattice as in Definition \ref{defMG}. 
The rationality problem for algebraic tori $T$ which correspond to 
Bravais groups $G\leq\GL(n,\bZ)$ 
was investigated by Voskresenskii (\cite{Vos83}, \cite[Section 8]{Vos98}). 
The referee of the paper pointed out that Voskresenskii asked a question 
whether $H^1(G,[M_G]^{fl})=0$ for any Bravais group $G$ 
(or, at least maximal finite subgroup $G\leq \GL(n,\bZ)$). 

By the classification of Bravais group $G$ of dimension $n\leq 6$ 
(see Subsection \ref{ssBravais}) and the computation of the flabby class $[M_G]^{fl}$ 
(see Subsection \ref{ss51}), we obtain the following: 

\begin{theorem}\label{thBravais6}
If $G$ is a Bravais group of dimension $n\leq 6$, 
then $H^1(G,[M_G]^{fl})=0$. 
In particular, if $G$ is a maximal finite subgroup 
$G\leq \GL(n,\bZ)$ where $n\leq 6$, then $H^1(G,[M_G]^{fl})=0$. 
\end{theorem}

In dimension $n=6$, 
we should use the following {\tt FlabbyResolutionFromPerm(G)} 
instead of {\tt FlabbyResolution(G)} 
for the 821st and the 822nd Bravais group of dimension $6$ (see Example \ref{exBravais6} below):\\

\noindent
{\tt FlabbyResolutionFromPerm(G)} returns the same as 
{\tt FlabbyResolution(G)} but using\\ 
{\tt ConjugacyClassesSubgroupsFromPerm(G)} instead of 
{\tt ConjugacyClassesSubgroups2(G)}.

\bigskip

\begin{verbatim}
FlabbyResolutionFromPerm:= function(g)
    local tg,gg,d,th,mi,ms,o,r,gg1,gg2,v1,mg,img;
    tg:=TransposedMatrixGroup(g);
    gg:=GeneratorsOfGroup(tg);
    d:=Length(Identity(g));
    th:=ConjugacyClassesSubgroupsFromPerm(tg);
    mi:=FindCoflabbyResolutionBase(tg,th);
    r:=Length(mi);
    o:=IdentityMat(r);
    gg1:=List(gg,x->PermutationMat(Permutation(x,mi),r));
    if r=d then
        return rec(injection:=TransposedMat(mi),
                   surjection:=NullMat(r,0),
                   actionP:=TransposedMatrixGroup(Group(gg1,o))
        );
    else
        ms:=NullspaceIntMat(mi);
        v1:=NullspaceIntMat(TransposedMat(ms));
        mg:=Concatenation(v1,ms);
        img:=mg^-1;
        gg2:=List(gg1,x->mg*x*img);
        gg2:=List(gg2,x->x{[d+1..r]}{[d+1..r]});
        return rec(injection:=TransposedMat(mi),
                   surjection:=TransposedMat(ms),
                   actionP:=TransposedMatrixGroup(Group(gg1)),
                   actionF:=TransposedMatrixGroup(Group(gg2))
        );
    fi;
end;
\end{verbatim}

\bigskip

\begin{example}[{Verification of $H^1(G,[M_G]^{fl})=0$ for any Bravais group of dimension $n\leq 6$}]\label{exBravais6}
{}~\\
\begin{verbatim}
gap> Read("caratnumber.gap");
gap> Read("FlabbyResolution.gap");

gap> b1:=Flat(List(CaratCrystalFamilies[1],BravaisGroupsCrystalFamily));;
gap> b2:=Flat(List(CaratCrystalFamilies[2],BravaisGroupsCrystalFamily));;
gap> b3:=Flat(List(CaratCrystalFamilies[3],BravaisGroupsCrystalFamily));;
gap> b4:=Flat(List(CaratCrystalFamilies[4],BravaisGroupsCrystalFamily));;
gap> b5:=Flat(List(CaratCrystalFamilies[5],BravaisGroupsCrystalFamily));;
gap> b6:=Flat(List(CaratCrystalFamilies[6],BravaisGroupsCrystalFamily));;

gap> List(b1,x->Product(H1(FlabbyResolution(x).actionF)));
[ 1 ]
gap> Length(last);                                        
1
gap> List(b2,x->Product(H1(FlabbyResolution(x).actionF)));
[ 1, 1, 1, 1, 1 ]
gap> Length(last); # there exist 5 Bravais groups of dimension 2 
5
gap> List(b3,x->Product(H1(FlabbyResolution(x).actionF)));
[ 1, 1, 1, 1, 1, 1, 1, 1, 1, 1, 1, 1, 1, 1 ]
gap> Length(last); # there exist 14 Bravais groups of dimension 3 
14
gap> List(b4,x->Product(H1(FlabbyResolution(x).actionF)));
[ 1, 1, 1, 1, 1, 1, 1, 1, 1, 1, 1, 1, 1, 1, 1, 1, 1, 1, 1, 1, 1, 1, 1, 1, 1, 
  1, 1, 1, 1, 1, 1, 1, 1, 1, 1, 1, 1, 1, 1, 1, 1, 1, 1, 1, 1, 1, 1, 1, 1, 1, 
  1, 1, 1, 1, 1, 1, 1, 1, 1, 1, 1, 1, 1, 1 ]
gap> Length(last); # there exist 64 Bravais groups of dimension 4 
64
gap> List(b5,x->Product(H1(FlabbyResolution(x).actionF)));
[ 1, 1, 1, 1, 1, 1, 1, 1, 1, 1, 1, 1, 1, 1, 1, 1, 1, 1, 1, 1, 1, 1, 1, 1, 1, 
  1, 1, 1, 1, 1, 1, 1, 1, 1, 1, 1, 1, 1, 1, 1, 1, 1, 1, 1, 1, 1, 1, 1, 1, 1, 
  1, 1, 1, 1, 1, 1, 1, 1, 1, 1, 1, 1, 1, 1, 1, 1, 1, 1, 1, 1, 1, 1, 1, 1, 1, 
  1, 1, 1, 1, 1, 1, 1, 1, 1, 1, 1, 1, 1, 1, 1, 1, 1, 1, 1, 1, 1, 1, 1, 1, 1, 
  1, 1, 1, 1, 1, 1, 1, 1, 1, 1, 1, 1, 1, 1, 1, 1, 1, 1, 1, 1, 1, 1, 1, 1, 1, 
  1, 1, 1, 1, 1, 1, 1, 1, 1, 1, 1, 1, 1, 1, 1, 1, 1, 1, 1, 1, 1, 1, 1, 1, 1, 
  1, 1, 1, 1, 1, 1, 1, 1, 1, 1, 1, 1, 1, 1, 1, 1, 1, 1, 1, 1, 1, 1, 1, 1, 1, 
  1, 1, 1, 1, 1, 1, 1, 1, 1, 1, 1, 1, 1, 1 ]
gap> Length(last); # there exist 189 Bravais groups of dimension 5 
189

gap> Length(b6); # there exist 841 Bravais groups of dimension 6 
841
gap> b6fl:=[1..841];
gap> for i in [1..820] do
> b6fl[i]:=FlabbyResolution(b6[i]);
> od;
gap> for i in [821..822] do # we need FlabbyResolutionFromPerm for i=821 and 822
> b6fl[i]:=FlabbyResolutionFromPerm(b6[i]);
> od;
gap> for i in [823..841] do
> b6fl[i]:=FlabbyResolution(b6[i]);
> od;
gap> List(b6fl,x->Product(H1(x.actionF)));
[ 1, 1, 1, 1, 1, 1, 1, 1, 1, 1, 1, 1, 1, 1, 1, 1, 1, 1, 1, 1, 1, 1, 1, 1, 1, 
  1, 1, 1, 1, 1, 1, 1, 1, 1, 1, 1, 1, 1, 1, 1, 1, 1, 1, 1, 1, 1, 1, 1, 1, 1, 
  1, 1, 1, 1, 1, 1, 1, 1, 1, 1, 1, 1, 1, 1, 1, 1, 1, 1, 1, 1, 1, 1, 1, 1, 1, 
  1, 1, 1, 1, 1, 1, 1, 1, 1, 1, 1, 1, 1, 1, 1, 1, 1, 1, 1, 1, 1, 1, 1, 1, 1, 
  1, 1, 1, 1, 1, 1, 1, 1, 1, 1, 1, 1, 1, 1, 1, 1, 1, 1, 1, 1, 1, 1, 1, 1, 1, 
  1, 1, 1, 1, 1, 1, 1, 1, 1, 1, 1, 1, 1, 1, 1, 1, 1, 1, 1, 1, 1, 1, 1, 1, 1, 
  1, 1, 1, 1, 1, 1, 1, 1, 1, 1, 1, 1, 1, 1, 1, 1, 1, 1, 1, 1, 1, 1, 1, 1, 1, 
  1, 1, 1, 1, 1, 1, 1, 1, 1, 1, 1, 1, 1, 1, 1, 1, 1, 1, 1, 1, 1, 1, 1, 1, 1, 
  1, 1, 1, 1, 1, 1, 1, 1, 1, 1, 1, 1, 1, 1, 1, 1, 1, 1, 1, 1, 1, 1, 1, 1, 1, 
  1, 1, 1, 1, 1, 1, 1, 1, 1, 1, 1, 1, 1, 1, 1, 1, 1, 1, 1, 1, 1, 1, 1, 1, 1, 
  1, 1, 1, 1, 1, 1, 1, 1, 1, 1, 1, 1, 1, 1, 1, 1, 1, 1, 1, 1, 1, 1, 1, 1, 1, 
  1, 1, 1, 1, 1, 1, 1, 1, 1, 1, 1, 1, 1, 1, 1, 1, 1, 1, 1, 1, 1, 1, 1, 1, 1, 
  1, 1, 1, 1, 1, 1, 1, 1, 1, 1, 1, 1, 1, 1, 1, 1, 1, 1, 1, 1, 1, 1, 1, 1, 1, 
  1, 1, 1, 1, 1, 1, 1, 1, 1, 1, 1, 1, 1, 1, 1, 1, 1, 1, 1, 1, 1, 1, 1, 1, 1, 
  1, 1, 1, 1, 1, 1, 1, 1, 1, 1, 1, 1, 1, 1, 1, 1, 1, 1, 1, 1, 1, 1, 1, 1, 1, 
  1, 1, 1, 1, 1, 1, 1, 1, 1, 1, 1, 1, 1, 1, 1, 1, 1, 1, 1, 1, 1, 1, 1, 1, 1, 
  1, 1, 1, 1, 1, 1, 1, 1, 1, 1, 1, 1, 1, 1, 1, 1, 1, 1, 1, 1, 1, 1, 1, 1, 1, 
  1, 1, 1, 1, 1, 1, 1, 1, 1, 1, 1, 1, 1, 1, 1, 1, 1, 1, 1, 1, 1, 1, 1, 1, 1, 
  1, 1, 1, 1, 1, 1, 1, 1, 1, 1, 1, 1, 1, 1, 1, 1, 1, 1, 1, 1, 1, 1, 1, 1, 1, 
  1, 1, 1, 1, 1, 1, 1, 1, 1, 1, 1, 1, 1, 1, 1, 1, 1, 1, 1, 1, 1, 1, 1, 1, 1, 
  1, 1, 1, 1, 1, 1, 1, 1, 1, 1, 1, 1, 1, 1, 1, 1, 1, 1, 1, 1, 1, 1, 1, 1, 1, 
  1, 1, 1, 1, 1, 1, 1, 1, 1, 1, 1, 1, 1, 1, 1, 1, 1, 1, 1, 1, 1, 1, 1, 1, 1, 
  1, 1, 1, 1, 1, 1, 1, 1, 1, 1, 1, 1, 1, 1, 1, 1, 1, 1, 1, 1, 1, 1, 1, 1, 1, 
  1, 1, 1, 1, 1, 1, 1, 1, 1, 1, 1, 1, 1, 1, 1, 1, 1, 1, 1, 1, 1, 1, 1, 1, 1, 
  1, 1, 1, 1, 1, 1, 1, 1, 1, 1, 1, 1, 1, 1, 1, 1, 1, 1, 1, 1, 1, 1, 1, 1, 1, 
  1, 1, 1, 1, 1, 1, 1, 1, 1, 1, 1, 1, 1, 1, 1, 1, 1, 1, 1, 1, 1, 1, 1, 1, 1, 
  1, 1, 1, 1, 1, 1, 1, 1, 1, 1, 1, 1, 1, 1, 1, 1, 1, 1, 1, 1, 1, 1, 1, 1, 1, 
  1, 1, 1, 1, 1, 1, 1, 1, 1, 1, 1, 1, 1, 1, 1, 1, 1, 1, 1, 1, 1, 1, 1, 1, 1, 
  1, 1, 1, 1, 1, 1, 1, 1, 1, 1, 1, 1, 1, 1, 1, 1, 1, 1, 1, 1, 1, 1, 1, 1, 1, 
  1, 1, 1, 1, 1, 1, 1, 1, 1, 1, 1, 1, 1, 1, 1, 1, 1, 1, 1, 1, 1, 1, 1, 1, 1, 
  1, 1, 1, 1, 1, 1, 1, 1, 1, 1, 1, 1, 1, 1, 1, 1, 1, 1, 1, 1, 1, 1, 1, 1, 1, 
  1, 1, 1, 1, 1, 1, 1, 1, 1, 1, 1, 1, 1, 1, 1, 1, 1, 1, 1, 1, 1, 1, 1, 1, 1, 
  1, 1, 1, 1, 1, 1, 1, 1, 1, 1, 1, 1, 1, 1, 1, 1, 1, 1, 1, 1, 1, 1, 1, 1, 1, 
  1, 1, 1, 1, 1, 1, 1, 1, 1, 1, 1, 1, 1, 1, 1, 1 ]
gap> Length(last); 
841
\end{verbatim}
\end{example}

\bigskip

%
\section{Norm one tori}\label{seNorm1}

Let $K/k$ be a separable field extension of degree $n$ 
and $L/k$ be the Galois closure of $K/k$. 
Let $G={\rm Gal}(L/k)$ and $H={\rm Gal}(L/K)$. 
The Galois group $G$ may be regarded as a transitive subgroup of 
the symmetric group $S_n$ of degree $n$. 
Let $R^{(1)}_{K/k}(\bG_m)$ be the norm one torus of $K/k$,
i.e. the kernel of the norm map $R_{K/k}(\bG_m)\rightarrow \bG_m$. 
The norm one torus $R^{(1)}_{K/k}(\bG_m)$ has the 
Chevalley module $J_{G/H}$ as its character module 
and the field $L(J_{G/H})^G$ as its function field 
(see Section \ref{seInt}). 
The following algorithm is available from 
{\tt http://math.h.kyoto-u.ac.jp/\~{}yamasaki/Algorithm/} 
as {\tt FlabbyResolution.gap}.\\

\noindent
{\tt Norm1TorusJ(d,m)} returns the 
Chevalley module $J_{G/H}$ for the $m$-th transitive subgroup 
$G=dTm\leq S_d$ of degree $d$ where $H$ is the stabilizer of 
one of the letters in $G$. 

\bigskip

\begin{algorithmN1T}
[{Construction of Chevalley module $J_{G/H}$ for the transitive subgroups $G=dTm\leq S_d$}]
{}~\\
\begin{verbatim}
Norm1TorusJ:= function(d,m)
    local I,M1,M2,M,f,Sd,T;
    I:=IdentityMat(d-1);
    Sd:=SymmetricGroup(d);
    T:=TransitiveGroup(d,m);
    M1:=Concatenation(List([2..d-1],x->I[x]),[-List([1..d-1],One)]);
    if d=2 then
        M:=[M1];
    else
        M2:=Concatenation([I[2],I[1]],List([3..d-1],x->I[x]));
        M:=[M1,M2];
    fi;
    f:=GroupHomomorphismByImages(Sd,Group(M),GeneratorsOfGroup(Sd),M);
    return Image(f,T);
end;
\end{verbatim}
\end{algorithmN1T}

\bigskip

\begin{example}[{$[J_{G/H}]^{fl}=0$ for $G=5T4\simeq A_5$}]
By using Method II as in Algorithm \ref{Alg6}, 
we may verify that $[J_{G/H}]^{fl}=0$ for $G=5T4\simeq A_5$ and $H=A_4$. 

\bigskip

\begin{verbatim}
gap> Read("crystcat.gap");
gap> Read("FlabbyResolution.gap");

gap> J54:=Norm1TorusJ(5,4);
<matrix group with 2 generators>
gap> StructureDescription(J54);
"A5"
gap> CrystCatZClass(J54);
[ 4, 31, 3, 2 ]
gap> IsInvertibleF(J54);
true
gap> Rank(FlabbyResolution(J54).actionF.1); # F is of rank 16
16

gap> mis:=SearchCoflabbyResolutionBase(TransposedMatrixGroup(J54),5);;
gap> Set(List(mis,Length))-4; # Method III could not apply 
[ 16, 21, 26, 31, 36, 41, 46, 51, 56, 61, 66, 71, 76, 81 ]

gap> ll:=PossibilityOfStablyPermutationF(J54);
[ [ 1, -2, -1, 0, 0, 1, 1, 1, -1, 0 ], [ 0, 0, 0, 0, 1, 1, -1, 0, 0, -1 ] ]
gap> l:=ll[Length(ll)];
[ 0, 0, 0, 0, 1, 1, -1, 0, 0, -1 ]
gap> bp:=StablyPermutationFCheckP(J54,Nlist(l),Plist(l));;
gap> Length(bp);
11
gap> Length(bp[1]); # rank of the both sides of (10) is 22
22
gap> rs:=RandomSource(IsMersenneTwister);
<RandomSource in IsMersenneTwister>
gap> rr:=List([1..10000],x->List([1..11],y->Random(rs,[-1..2])));;    
gap> Filtered(rr,x->Determinant(x*bp)^2=1);                           
[ [ 2, 0, 0, -1, -1, 0, -1, 1, 0, 1, 1 ] ]
gap> p:=last[1]*bp;
[ [ 2, 0, 0, 0, 0, 0, 0, 0, 2, 0, 0, 0, 0, 0, -1, -1, -1, 0, 0, 0, -1, -1 ], 
  [ 0, 2, 0, 0, 0, 0, 0, 2, 0, 0, 0, 0, -1, 0, 0, 0, -1, -1, 0, -1, 0, -1 ], 
  [ 0, 0, 2, 0, 0, 0, 0, 0, 0, 0, 0, 2, 0, -1, 0, -1, 0, -1, 0, -1, -1, 0 ], 
  [ 0, 0, 0, 2, 0, 0, 0, 0, 0, 0, 2, 0, 0, -1, -1, 0, -1, 0, -1, -1, 0, 0 ], 
  [ 0, 0, 0, 0, 2, 2, 0, 0, 0, 0, 0, 0, -1, 0, -1, 0, 0, -1, -1, 0, -1, 0 ], 
  [ 0, 0, 0, 0, 0, 0, 2, 0, 0, 2, 0, 0, -1, -1, 0, -1, 0, 0, -1, 0, 0, -1 ], 
  [ -1, 0, -1, 1, 1, 0, -1, 1, -1, -1, 0, -1, 0, 1, 0, 1, 0, 0, 0, 0, 1, 1 ], 
  [ 1, -1, 1, 1, -1, -1, -1, -1, 0, -1, 0, 0, 1, 0, 0, 0, 0, 1, 1, 0, 0, 1 ], 
  [ -1, -1, 0, -1, 0, 1, 1, -1, -1, 0, -1, 1, 0, 0, 1, 0, 1, 0, 0, 1, 0, 1 ], 
  [ 1, 1, -1, -1, 0, 1, -1, 0, 0, -1, -1, -1, 0, 1, 0, 1, 0, 0, 1, 1, 0, 0 ], 
  [ -1, 0, 0, -1, -1, -1, 1, 1, -1, 0, -1, 1, 0, 0, 1, 0, 1, 0, 1, 0, 1, 0 ], 
  [ 0, -1, -1, 0, -1, -1, 0, -1, 1, 1, 1, -1, 1, 0, 0, 0, 0, 1, 0, 1, 1, 0 ], 
  [ -1, -1, 0, 0, 0, 1, -1, -1, -1, -1, 1, 1, 1, 0, 0, 1, 1, 0, 0, 0, 0, 1 ], 
  [ -1, 1, 1, -1, -1, -1, 0, 0, -1, 1, -1, 0, 0, 0, 1, 0, 1, 0, 1, 0, 1, 0 ], 
  [ -1, 0, -1, 1, -1, -1, 1, 1, -1, 0, 0, -1, 0, 0, 1, 1, 0, 1, 0, 0, 1, 0 ], 
  [ 0, -1, 0, 0, -1, -1, -1, -1, 1, -1, 1, 1, 1, 0, 0, 0, 0, 1, 1, 0, 0, 1 ], 
  [ 0, -1, -1, -1, 1, 0, 0, -1, 1, 1, -1, -1, 0, 1, 0, 0, 1, 1, 0, 1, 0, 0 ], 
  [ 0, 0, 0, -1, -1, -1, -1, 1, 1, -1, -1, 1, 1, 1, 1, 0, 0, 0, 1, 0, 0, 0 ], 
  [ -1, 1, -1, 0, -1, -1, 0, 0, -1, 1, 1, -1, 0, 0, 1, 1, 0, 1, 0, 0, 1, 0 ], 
  [ 1, -1, -1, 1, -1, -1, 1, -1, 0, 0, 0, -1, 1, 0, 0, 0, 0, 1, 0, 1, 1, 0 ], 
  [ 0, 0, -1, -1, 1, 0, -1, 1, 1, -1, -1, -1, 0, 1, 0, 1, 0, 0, 1, 1, 0, 0 ], 
  [ 1, -1, -1, -1, 0, 1, 1, -1, 0, 0, -1, -1, 0, 1, 0, 0, 1, 1, 0, 1, 0, 0 ] ]
gap> Determinant(p);
-1
gap> StablyPermutationFCheckMat(J54,Nlist(l),Plist(l),p);
true
\end{verbatim}
\end{example}

\bigskip

\begin{example}[{$[J_{G/H}]^{fl}=0$ for $G=6T3\simeq D_6$}]
By using Method II as in Algorithm \ref{Alg6}, 
we may verify that $[J_{G/H}]^{fl}=0$ for $G=6T3\simeq D_6$ and $H=C_2$. 

\bigskip

\begin{verbatim}
gap> Read("caratnumber.gap");
gap> Read("FlabbyResolution.gap");

gap> J63:=Norm1TorusJ(6,3);
<matrix group with 2 generators>
gap> StructureDescription(J63);
"D12"
gap> CaratZClass(J63);
[ 5, 391, 4 ]
gap> IsInvertibleF(J63);
true
gap> Rank(FlabbyResolution(J63).actionF.1); # F is of rank 13
13
gap> mis:=SearchCoflabbyResolutionBase(TransposedMatrixGroup(J63),3);; # Method III
gap> List(mis,Length);
[ 24, 20, 24, 18, 24, 20, 24, 20, 18, 14, 21, 17, 21, 24, 20, 24, 24, 20, 
  24, 18, 30, 26, 24, 20, 18, 20, 14, 24, 20, 18 ]
gap> mi:=mis[Length(mis)-3]; # (new) F is of rank 9 (=14-5)
[ [ -1, 0, 0, 0, 0 ], [ -1, 1, -1, 1, -1 ], [ -1, 1, 0, 0, 0 ], [ 0, -1, 1, 0, 0 ], 
  [ 0, 0, -1, 1, 0 ], [ 0, 0, 0, -1, 1 ], [ 0, 0, 0, 0, -1 ], [ 0, 0, 0, 0, 1 ], 
  [ 0, 0, 0, 1, -1 ], [ 0, 0, 1, -1, 0 ], [ 0, 1, -1, 0, 0 ], [ 1, -1, 0, 0, 0 ], 
  [ 1, -1, 1, -1, 1 ], [ 1, 0, 0, 0, 0 ] ]
gap> ll:=PossibilityOfStablyPermutationFFromBase(J63,mi);
[ [ 1, -1, 0, -1, 0, 2, 0, 1, 1, -1, -1 ], [ 0, 0, 1, 0, 1, 0, -1, 0, 0, 1, -1 ] ]
gap> l:=ll[Length(ll)];
[ 0, 0, 1, 0, 1, 0, -1, 0, 0, 1, -1 ]
gap> bp:=StablyPermutationFCheckPFromBase(J63,mi,Nlist(l),Plist(l));;
gap> Length(bp);
17
gap> Length(bp[1]); # rank of the both sides of (10) is 11
11
gap> rs:=RandomSource(IsMersenneTwister);                             
<RandomSource in IsMersenneTwister>
gap> rr:=List([1..5000],x->List([1..17],y->Random(rs,[-1..1])));;
gap> Filtered(rr,x->Determinant(x*bp)^2=1);                      
[ [ 0, 1, -1, 0, -1, 1, -1, 0, 1, 1, 0, -1, 1, -1, 0, 0, 1 ] ]
gap> p:=last[1]*bp;                                                 
[ [ 0, 1, 1, 0, 1, 0, -1, 0, 0, -1, -1 ], 
  [ 1, 0, 0, 1, 0, 1, 0, -1, -1, 0, -1 ], 
  [ 1, -1, 0, 1, -1, 1, 1, 0, -1, 1, -1 ], 
  [ 0, 0, 0, 0, 0, 0, 0, -1, 0, -1, 1 ], 
  [ 0, 1, 1, -1, 1, -1, -1, 1, 1, 0, -1 ], 
  [ 1, 0, -1, 1, -1, 1, 1, -1, 0, 1, -1 ], 
  [ -1, 1, 1, -1, 1, 0, -1, 1, 1, 0, -1 ], 
  [ 1, -1, -1, 1, 0, 1, 1, -1, 0, 1, -1 ], 
  [ -1, 1, 1, 0, 1, -1, -1, 1, 1, 0, -1 ], 
  [ -1, 1, 1, 0, 1, -1, 0, 1, 1, -1, -1 ], 
  [ 1, -1, -1, 1, 0, 1, 1, 0, -1, 1, -1 ] ]
gap> Determinant(p);
1
gap> StablyPermutationFCheckMatFromBase(J63,mi,Nlist(l),Plist(l),p);
true
\end{verbatim}
\end{example}

\bigskip

\begin{example}[{The flabby class $[J_{Q_8}]^{fl}$ is not invertible but 
flabby and coflabby}]\label{exQ8}
Let $J_{Q_8}$ be the Chevalley module of the quaternion 
group $Q_8$ of order $8$. The rank of $J_{Q_8}$ is $7$. 
Let $8T5$ be the 5th transitive subgroup of $S_8$. 
Then $8T5\simeq Q_8$. 
By Theorems \ref{th13-1}, \ref{thEM73} (ii) and \ref{thCTS77}, 
the flabby class $[J_{Q_8}]^{fl}$ of $J_{Q_8}$ 
is not invertible but flabby and coflabby. 
Using {\tt FlabbyResolution} as in Algorithm \ref{Alg1}, 
we may obtain the flabby class $[J_{Q_8}]^{fl}=[F]$ 
of $J_{Q_8}$ where $F$ is of rank $9$ which satisfies 
$0\rightarrow M_{Q_8}\rightarrow P\rightarrow F\rightarrow 0$.

\bigskip

\begin{verbatim}
gap> Read("FlabbyResolution.gap");

gap> J85:=Norm1TorusJ(8,5);                 
<matrix group with 2 generators>
gap> StructureDescription(J85);
"Q8"
gap> F:=FlabbyResolution(J85).actionF; 
<matrix group with 2 generators>
gap> Rank(F.1); # F is of rank 9
9

gap> mis:=SearchCoflabbyResolutionBase(TransposedMatrixGroup(J85),5);;
gap> Set(List(mis,Length))-7; # Method III could not apply
[ 9, 11, 13, 15, 17, 19, 21, 23, 25, 27, 29, 33 ]

gap> IsFlabby(F);
true
gap> IsCoflabby(F);
true
gap> IsInvertibleF(J85);
false
\end{verbatim}
\end{example}

\bigskip

\begin{example}[{$[J_{G/H}]^{fl}$ for $G=7Tm$ and $G=11Tm$}]
By Theorems \ref{th13-1}, \ref{th13-2}, \ref{th15}, \ref{thS} and \ref{thA}, 
and Lemma \ref{lemp3}, we may obtain 
the stably rational classification of 
norm one tori $R_{K/k}^{(1)}(\bG_m)$ of dimensions $6$ and $10$ 
via Algorithm N1T for 
$[J_{G/H}]^{fl}$ with $G=7Tm\leq S_7$ $(1\leq m\leq 7)$ and $G=11Tm\leq S_{11}$ 
$(1\leq m\leq 8)$ respectively. 
The result is as follows:\\

\begin{theorem}[Stably rational classification of 
norm one tori of dimensions $6$ and $10$]
~\\
{\rm (i)} $R_{K/k}^{(1)}(\bG_m)$ is stably $k$-rational 
for $7T1\simeq C_7$ and $7T2\simeq D_7$;\\
{\rm (ii)} $R_{K/k}^{(1)}(\bG_m)$ is not stably but retract $k$-rational 
for $7T3\simeq F_{21}$, $7T4\simeq F_{42}$, $7T5\simeq {\rm PSL}(2,7)$, 
$7T6\simeq A_7$ and $7T7\simeq S_7$;\\
{\rm (iii)} $R_{K/k}^{(1)}(\bG_m)$ is stably $k$-rational 
for $11T1\simeq C_{11}$ and $11T2\simeq D_{11}$;\\
{\rm (iv)} $R_{K/k}^{(1)}(\bG_m)$ is not stably but retract $k$-rational 
for $11T3\simeq F_{55}$, $11T4\simeq F_{110}$, $11T5\simeq {\rm PSL}(2,11)$, 
$11T6\simeq M_{11}$, $11T7\simeq A_{11}$ and $11T8\simeq S_{11}$ where 
$M_{11}$ is the Mathieu group of degree $11$.
\end{theorem}
\end{example}

\bigskip

\section{Tate cohomology: GAP computations}\label{seTate}

In this section, we provide some algorithms of GAP 
for computing the Tate cohomology $\widehat H^n(G,M_G)$ by 
using the GAP package HAP (\cite{HAP}). 
We will use this for showing the main theorem (Theorem \ref{th1M}) 
in Section \ref{seProof1}. 
The following algorithms are available from 
{\tt http://math.h.kyoto-u.ac.jp/\~{}yamasaki/Algorithm/} 
as {\tt Hn.gap}.\\

\bigskip

\noindent 
{\tt TateCohomology(G,n)} returns the Tate cohomology group 
$\widehat H^n(G,M_G)$ for $n\in\bZ$.\\

\bigskip

\begin{algorithmTC}[Tate cohomology]
{}~\\
\begin{verbatim}
LoadPackage("HAP");

CochainComplexMatrixGroup:= function(M,n)
    local IA,G,A,R,TR;
    IA:=RegularActionHomomorphism(TransposedMatrixGroup(M));
    G:=Image(IA);
    A:=InverseGeneralMapping(IA);
    R:=ResolutionFiniteGroup(G,n);
    TR:=HomToIntegralModule(R,A);
    return TR;
end;

ChainComplexMatrixGroup:= function(M,n)
    local IA,G,A,R,TR;
    IA:=RegularActionHomomorphism(TransposedMatrixGroup(M));
    G:=Image(IA);
    A:=InverseGeneralMapping(IA);
    R:=ResolutionFiniteGroup(G,n);
    TR:=TensorWithIntegralModule(R,A);
    return TR;
end;

CochainComplexRightCosets:= function(g,h,n)
    local gg,hg,og,gp,A,R,TR;
    gg:=GeneratorsOfGroup(g);
    hg:=SortedList(RightCosets(g,h));
    og:=Length(hg);
    gp:=List(gg,x->PermutationMat(Permutation(x,hg,OnRight),og));
    A:=GroupHomomorphismByImages(g,Group(gp),gg,gp);
    R:=ResolutionFiniteGroup(g,n);
    TR:=HomToIntegralModule(R,A);
    return TR;
end;

ChainComplexRightCosets:= function(g,h,n)
    local gg,hg,og,gp,A,R,TR;
    gg:=GeneratorsOfGroup(g);
    hg:=SortedList(RightCosets(g,h));
    og:=Length(hg);
    gp:=List(gg,x->PermutationMat(Permutation(x,hg,OnRight),og));
    A:=GroupHomomorphismByImages(g,Group(gp),gg,gp);
    R:=ResolutionFiniteGroup(g,n);
    TR:=TensorWithIntegralModule(R,A);
    return TR;
end;

TateCohomology:= function(g,n)
    local m,s,r,TR;
    if n=0 then
        m:=Sum(g);
        s:=SmithNormalFormIntegerMat(m);
        r:=Rank(s);
        return List([1..r],x->s[x][x]);
    elif n>0 then
        TR:=CochainComplexMatrixGroup(g,n+1);
        return Cohomology(TR,n);
    else
        TR:=ChainComplexMatrixGroup(g,-n);
        return Homology(TR,-n-1);
    fi;
end;
\end{verbatim}
\end{algorithmTC}

\bigskip

\begin{example}[Tate cohomology $\widehat H^n(G,M_G)$ 
for the group $C_2^3$ of the GAP ID $(3,3,3,3)$]
Let $G\simeq C_2^3$ be the group of the GAP ID $(3,3,3,3)$. 
We may obtain the Tate cohomologies $\widehat H^n(G,M_G)$ 
for $-7\leq n\leq 7$ as follows:\\

\renewcommand{\arraystretch}{1.2}

\begin{tabular}{l|cccccccc} 
$n$ & $-7$ & $-6$ & $-5$ & $-4$ & $-3$ & $-2$ & $-1$ & $0$ \\\hline
$\widehat H^n(C_2^3,M_{C_2^3})$ 
& $(\bZ/2\bZ)^{\oplus 13}$& $(\bZ/2\bZ)^{\oplus 9}$
& $(\bZ/2\bZ)^{\oplus 7}$ & $(\bZ/2\bZ)^{\oplus 4}$
& $(\bZ/2\bZ)^{\oplus 3}$ & $\bZ/2\bZ$ & $\bZ/2\bZ$
& $0$ \\
\end{tabular}

\begin{tabular}{l|ccccccc} 
$n$ & $1$ & $2$ & $3$ & $4$ & $5$ & $6$ & $7$\\\hline
$\widehat H^n(C_2^3,M_{C_2^3})$ & 
$(\bZ/2\bZ)^{\oplus 2}$ & $(\bZ/2\bZ)^{\oplus 3}$
& $(\bZ/2\bZ)^{\oplus 6}$ & $(\bZ/2\bZ)^{\oplus 8}$ 
& $(\bZ/2\bZ)^{\oplus 12}$ & $(\bZ/2\bZ)^{\oplus 15}$ 
& $(\bZ/2\bZ)^{\oplus 20}$\\
\end{tabular}\\

\renewcommand{\arraystretch}{1}

\bigskip

\begin{verbatim}
gap> Read("Hn.gap");

gap> G:=MatGroupZClass(3,3,3,3);;
gap> List([-7..7],i->TateCohomology(G,i));
[ [ 2, 2, 2, 2, 2, 2, 2, 2, 2, 2, 2, 2, 2 ], [ 2, 2, 2, 2, 2, 2, 2, 2, 2 ], 
  [ 2, 2, 2, 2, 2, 2, 2 ], [ 2, 2, 2, 2 ], [ 2, 2, 2 ], [ 2 ], [ 2 ], [  ], 
  [ 2, 2 ], [ 2, 2, 2 ], [ 2, 2, 2, 2, 2, 2 ], [ 2, 2, 2, 2, 2, 2, 2, 2 ], 
  [ 2, 2, 2, 2, 2, 2, 2, 2, 2, 2, 2, 2 ], 
  [ 2, 2, 2, 2, 2, 2, 2, 2, 2, 2, 2, 2, 2, 2, 2 ], 
  [ 2, 2, 2, 2, 2, 2, 2, 2, 2, 2, 2, 2, 2, 2, 2, 2, 2, 2, 2, 2 ] ]
gap> List(last,Length);
[ 13, 9, 7, 4, 3, 1, 1, 0, 2, 3, 6, 8, 12, 15, 20 ]
\end{verbatim}
\end{example}

\bigskip

\begin{example}[{The group ${\rm Indmf}(6,10,1)$} is not a subgroup of the 
$17$ irreducible maximal finite groups but is indecomposable]\label{ex6101}
Let $G={\rm Indmf}(6,10,1)$ be the indecomposable maximal finite group 
of the CARAT ID $(6,5517,4)$. 
We will confirm that $G$ is not a subgroup of all the 
$17$ irreducible maximal finite groups ${\rm Imf}(6,i,j)$ of dimension $6$ 
but is indecomposable by using the Tate cohomology $\widehat H^{-1}(G)$, $\widehat H^1(G)$ 
and $\widehat H^2(G)$ where $\widehat H^i(G)=\widehat H^i(G,M_G)$ 
and $M_G$ is the corresponding $G$-lattice as in Definition \ref{defMG}. 

Let ${\rm Sy}_2(G)$ be a $2$-Sylow subgroup of $G$. 
We have the normalizer $N_{\GL(6,\bZ)}({\rm Sy}_2(G))$ of ${\rm Sy}_2(G)$ in $\GL(6,\bZ)$ 
is ${\rm Sy}_2(G)$. 
Suppose that there exists a group $G^\prime$ such that 
$G<G^\prime$ with even index. 
Then there exists a $2$-Sylow subgroup ${\rm Sy}_2(G^\prime)$ of $G^\prime$ 
such that ${\rm Sy}_2(G)<{\rm Sy}_2(G^\prime)$. 
But this is impossible because $N_{\GL(6,\bZ)}({\rm Sy}_2(G))={\rm Sy}_2(G)$. 
Hence such a group $G<G^\prime$ with even index does not exist. 

Only two groups $G^\prime={\rm Imf}(6,3,1)$ and $G^{\prime\prime}={\rm Imf}(6,3,2)$ 
out of the $17$ groups ${\rm Imf}(6,i,j)$ may have 
a subgroup $G$ with odd index.  
However, these two does not occur because 
$\widehat H^1({\rm Sy}_2(G))=\bZ/2\bZ\oplus\bZ/2\bZ$, 
$\widehat H^1({\rm Sy}_2(G^\prime))=\bZ/2\bZ$ and 
$\widehat H^1({\rm Sy}_2(G^{\prime\prime}))=\bZ/2\bZ$. 
This implies that $G$ is not a subgroup of all the 
$17$ irreducible maximal finite groups ${\rm Imf}(6,i,j)$ of dimension $6$. 

Next, we will show that $G$ is indecomposable. 
Over the field $\bC$ of complex numbers, $G$ splits into 
$2$ irreducible direct summands of degree $3$. 
Indeed, we also see that the $6$ groups of the CARAT IDs 
$(6,5517,1)$, $(6,5517,3)$, $(6,5517,5)$, $(6,5517,6)$, $(6,5517,8)$ and $(6,5517,9)$ 
are in the same $\bQ$-class of $G$ but splits into $2$ irreducible 
direct summands of degree $3$ in $\GL(3,\bZ)$. 
This implies that $G$ splits into them in $\GL(3,\bQ)$. 

We will confirm that $G$ is indecomposable in $\GL(6,\bZ)$. 
Because of the order of $G$ is $2304=48^2$, 
we see that $G$ should be a direct product of $2$ groups 
among $3$ groups $G_1={\rm Imf}(3,1,1)$, $G_2={\rm Imf}(3,1,2)$ and 
$G_3={\rm Imf}(3,1,3)$ which are isomorphic to the group $C_2\times S_4$ of order $48$. 
By computing the Tate cohomologies $\widehat H^i$, we have 
\begin{align*}
&\widehat H^{-1}(G)=\bZ/2\bZ,& 
&\widehat H^{-1}(G_1\times G_1)=(\bZ/2\bZ)^{\oplus 2},& 
&\widehat H^{-1}(G_2\times G_2)=(\bZ/2\bZ)^{\oplus 2},& 
&\widehat H^{-1}(G_3\times G_3)=0,\\
& & 
&\widehat H^{-1}(G_1\times G_2)=(\bZ/2\bZ)^{\oplus 2},& 
&\widehat H^{-1}(G_1\times G_3)=\bZ/2\bZ,& 
&\widehat H^{-1}(G_2\times G_3)=\bZ/2\bZ,\\
&\widehat H^{1}(G)=\bZ/2\bZ,& 
&\widehat H^{1}(G_1\times G_1)=(\bZ/2\bZ)^{\oplus 2},& 
&\widehat H^{1}(G_2\times G_2)=0,& 
&\widehat H^{1}(G_3\times G_3)=(\bZ/2\bZ)^{\oplus 2},\\
& &
&\widehat H^{1}(G_1\times G_2)=\bZ/2\bZ,& 
&\widehat H^{1}(G_1\times G_3)=(\bZ/2\bZ)^{\oplus 2},& 
&\widehat H^{1}(G_2\times G_3)=\bZ/2\bZ,\\
&\widehat H^{2}(G)=(\bZ/2\bZ)^{\oplus 4},& 
&\widehat H^{2}(G_1\times G_1)=(\bZ/2\bZ)^{\oplus 8},& 
&\widehat H^{2}(G_2\times G_2)=(\bZ/2\bZ)^{\oplus 2},& 
&\widehat H^{2}(G_3\times G_3)=(\bZ/2\bZ)^{\oplus 8},\\
& &
&\widehat H^{2}(G_1\times G_2)=(\bZ/2\bZ)^{\oplus 5},& 
&\widehat H^{2}(G_1\times G_3)=(\bZ/2\bZ)^{\oplus 8},& 
&\widehat H^{2}(G_2\times G_3)=(\bZ/2\bZ)^{\oplus 5}.
\end{align*}
Hence $G$ is indecomposable in $\GL(6,\bZ)$.

\bigskip

\begin{verbatim}
gap> Read("caratnumber.gap");
gap> Read("FlabbyResolution.gap");
gap> Read("Hn.gap");
gap> Read("KS.gap");

gap> G:=CaratMatGroupZClass(6,5517,4);; # G=Indmf(6,10,1)
gap> Order(G); # #G=2304=48^2=2^8*3^2
2304
gap> imf6:=AllImfMatrixGroups(6);
[ ImfMatrixGroup(6,1,1), ImfMatrixGroup(6,1,2), ImfMatrixGroup(6,1,3), 
  ImfMatrixGroup(6,2,1), ImfMatrixGroup(6,3,1), ImfMatrixGroup(6,3,2), 
  ImfMatrixGroup(6,4,1), ImfMatrixGroup(6,4,2), ImfMatrixGroup(6,5,1), 
  ImfMatrixGroup(6,6,1), ImfMatrixGroup(6,6,2), ImfMatrixGroup(6,6,3), 
  ImfMatrixGroup(6,7,1), ImfMatrixGroup(6,7,2), ImfMatrixGroup(6,8,1), 
  ImfMatrixGroup(6,9,1), ImfMatrixGroup(6,9,2) ]
gap> List(imf6,x->Order(x)/2304);
[ 20, 20, 20, 9/2, 45, 45, 35/8, 35/8, 7/24, 5/48, 5/48, 5/48, 2, 2, 10, 1/8, 1/8 ]
gap> G2:=SylowSubgroup(G,2); # G2 is a 2-sylow group of G
<group of 6x6 matrices of size 256 in characteristic 0>
gap> Normalizer(GL(6,Integers),G2);
<matrix group with 14 generators>
gap> Order(last); # (Normalizer of G2)=G2
256
gap> H1(G2);
[ 1, 1, 1, 1, 2, 2 ]
gap> H1(SylowSubgroup(ImfMatrixGroup(6,3,1),2));
[ 1, 1, 1, 1, 1, 2 ]
gap> H1(SylowSubgroup(ImfMatrixGroup(6,3,2),2));
[ 1, 1, 1, 1, 1, 2 ]

gap> Set(AllDirectProductIndmfMatrixGroups([3,3]),CaratZClass);
[ [ 6, 5517, 1 ], [ 6, 5517, 3 ], [ 6, 5517, 5 ], 
  [ 6, 5517, 6 ], [ 6, 5517, 8 ], [ 6, 5517, 9 ] ]

gap> ig:=IrreducibleRepresentations(G);;
gap> List(ig,x->Sum(G,y->Trace(y)*Trace(Image(x,y))));   
[ 0, 0, 0, 0, 0, 0, 0, 0, 0, 0, 0, 0, 0, 0, 0, 0, 0, 0, 0, 0, 
  0, 0, 0, 0, 0, 0, 0, 0, 0, 0, 0, 0, 0, 0, 0, 0, 0, 0, 0, 0, 
  0, 0, 0, 2304, 0, 0, 0, 0, 0, 0, 0, 0, 0, 0, 0, 0, 0, 0, 0, 0, 
  0, 0, 0, 0, 2304, 0, 0, 0, 0, 0, 0, 0, 0, 0, 0, 0, 0, 0, 0, 0, 
  0, 0, 0, 0, 0, 0, 0, 0, 0, 0, 0, 0, 0, 0, 0, 0, 0, 0, 0, 0 ]
gap> Filtered([1..Length(ig)],x->last[x]>0);
[ 44, 65 ]
gap> ig[44];
CompositionMapping( [ (2,6)(5,8)(7,11), (5,8), (1,12)(2,6)(4,5,9,8)(7,11), (6,7), 
  (1,2,6)(3,4,8)(5,10,9)(7,12,11), (1,2,6)(7,12,11), (4,9)(5,8), (1,12)(6,7), (2,11)(6,7), 
  (3,10)(4,9) ] -> 
[ [ [ -1, 0, 0 ], [ 0, 1, 0 ], [ 0, 0, 1 ] ], [ [ -1, 0, 0 ], [ 0, 1, 0 ], [ 0, 0, 1 ] ], 
  [ [ 0, 1, 0 ], [ -1, 0, 0 ], [ 0, 0, 1 ] ], [ [ 1, 0, 0 ], [ 0, 1, 0 ], [ 0, 0, 1 ] ], 
  [ [ 0, 0, 1 ], [ 1, 0, 0 ], [ 0, 1, 0 ] ], [ [ 1, 0, 0 ], [ 0, 1, 0 ], [ 0, 0, 1 ] ], 
  [ [ -1, 0, 0 ], [ 0, -1, 0 ], [ 0, 0, 1 ] ], [ [ 1, 0, 0 ], [ 0, 1, 0 ], [ 0, 0, 1 ] ], 
  [ [ 1, 0, 0 ], [ 0, 1, 0 ], [ 0, 0, 1 ] ], [ [ 1, 0, 0 ], [ 0, -1, 0 ], [ 0, 0, -1 ] ] ], 
  <action isomorphism> )
gap> ig[65];
CompositionMapping( [ (2,6)(5,8)(7,11), (5,8), (1,12)(2,6)(4,5,9,8)(7,11), (6,7), 
  (1,2,6)(3,4,8)(5,10,9)(7,12,11), (1,2,6)(7,12,11), (4,9)(5,8), (1,12)(6,7), (2,11)(6,7), 
  (3,10)(4,9) ] -> 
[ [ [ 1, 0, 0 ], [ 0, 0, 1 ], [ 0, 1, 0 ] ], [ [ 1, 0, 0 ], [ 0, 1, 0 ], [ 0, 0, 1 ] ], 
  [ [ -1, 0, 0 ], [ 0, 0, 1 ], [ 0, 1, 0 ] ], [ [ 1, 0, 0 ], [ 0, -1, 0 ], [ 0, 0, 1 ] ], 
  [ [ 0, 0, 1 ], [ 1, 0, 0 ], [ 0, 1, 0 ] ], [ [ 0, 0, 1 ], [ 1, 0, 0 ], [ 0, 1, 0 ] ], 
  [ [ 1, 0, 0 ], [ 0, 1, 0 ], [ 0, 0, 1 ] ], [ [ -1, 0, 0 ], [ 0, -1, 0 ], [ 0, 0, 1 ] ], 
  [ [ 1, 0, 0 ], [ 0, -1, 0 ], [ 0, 0, -1 ] ], [ [ 1, 0, 0 ], [ 0, 1, 0 ], [ 0, 0, 1 ] ] ], 
  <action isomorphism> )
gap> Imf3:=AllImfMatrixGroups(3);
[ ImfMatrixGroup(3,1,1), ImfMatrixGroup(3,1,2), ImfMatrixGroup(3,1,3) ]
gap> List(Imf3,Size);
[ 48, 48, 48 ]
gap> G11:=DirectProductMatrixGroup([Imf3[1],Imf3[1]]);;
gap> G12:=DirectProductMatrixGroup([Imf3[1],Imf3[2]]);;
gap> G13:=DirectProductMatrixGroup([Imf3[1],Imf3[3]]);;
gap> G22:=DirectProductMatrixGroup([Imf3[2],Imf3[2]]);;
gap> G23:=DirectProductMatrixGroup([Imf3[2],Imf3[3]]);;
gap> G33:=DirectProductMatrixGroup([Imf3[3],Imf3[3]]);;
gap> List([G11,G22,G33,G12,G13,G23],CaratZClass);
[ [ 6, 5517, 1 ], [ 6, 5517, 9 ], [ 6, 5517, 6 ], 
  [ 6, 5517, 5 ], [ 6, 5517, 3 ], [ 6, 5517, 8 ] ]
gap> List([G,G11,G22,G33,G12,G13,G23],x->TateCohomology(x,-1));
[ [ 2 ], [ 2, 2 ], [ 2, 2 ], [  ], [ 2, 2 ], [ 2 ], [ 2 ] ]
gap> List([G,G11,G22,G33,G12,G13,G23],x->TateCohomology(x,1));
[ [ 2 ], [ 2, 2 ], [  ], [ 2, 2 ], [ 2 ], [ 2, 2 ], [ 2 ] ]
gap> List([G,G11,G22,G33,G12,G13,G23],x->TateCohomology(x,2));
[ [ 2, 2, 2, 2 ], [ 2, 2, 2, 2, 2, 2, 2, 2 ], [ 2, 2 ], [ 2, 2, 2, 2, 2, 2, 2, 2 ], 
  [ 2, 2, 2, 2, 2 ], [ 2, 2, 2, 2, 2, 2, 2, 2 ], [ 2, 2, 2, 2, 2 ] ]
gap> List(last,Length);
[ 4, 8, 2, 8, 5, 8, 5 ]
\end{verbatim}
\end{example}

\bigskip

\section{Proof of Theorem \ref{th1M}}\label{seProof1}
Let $G$ be a finite subgroup of $\GL(4,\bZ)$ and $M=M_G$ be 
the corresponding $G$-lattice of rank $4$ as in Definition \ref{defMG}. \\


{\bf Step 1.} The case where $M_G$ is decomposable. \\

We assume that $M_G\simeq M_1\oplus M_2$ is decomposable with 
$\rank M_1\geq \rank M_2\geq 1$. 
Thus $\rank M_1\leq 3$ and $\rank M_2\leq 2$. 
By Theorem \ref{thVo}, we have $[M_2]^{fl}=0$. 
It follows from Lemma \ref{lemp1} that $[M_G]^{fl}=[M_1]^{fl}$. 
By Theorem \ref{thKu} and Lemma \ref{lemp1}, 
$[M_1]^{fl}\neq 0\iff [M_1]^{fl}\ {\rm is\ not\ invertible}\iff 
G/N_1$ is conjugate to one of the $15$ groups as in Table $1$ 
where $N_1=\{\sigma\in G\mid \sigma(v)=v\ {\rm for\ any}\ v\in M_1\}$. 
There exist exactly $64$ such $G$-lattices $M_G\simeq M_1\oplus M_2$ 
with $[M_G]^{fl}\neq 0$ 
(equivalently $[M_G]^{fl}$ is not invertible in this case) 
whose GAP IDs $\mathcal{N}_{31}$ are computed in Example \ref{exN} 
(see also Table $3$). \\

{\bf Step 2.} $[M_G]^{fl}=0$ for $G={\rm Imf}(4,2,1)$, 
${\rm Imf}(4,3,1)$ and ${\rm Imf}(4,4,1)$. \\

We assume that $M_G$ is indecomposable. 
There exist $295$ indecomposable $G$-lattices $M_G$ 
of rank $4$ (see Example \ref{exKS1}). 
Let ${\rm Imf}(4,i,j)\leq \GL(4,\bZ)$ be the 
$j$-th $\bZ$-class of the $i$-th $\bQ$-class of the 
irreducible maximal finite group of dimension $4$
which corresponds to 
the GAP command {\tt ImfMatrixGroup(4,i,j)}. 
There exist exactly $6$ such maximal groups 
${\rm Imf}(4,1,1)$, ${\rm Imf}(4,2,1)$, ${\rm Imf}(4,3,1)$, 
${\rm Imf}(4,3,2)$, ${\rm Imf}(4,4,1)$ and ${\rm Imf}(4,5,1)$ 
of order $1152$, $288$, $240$, $240$, $384$ and $144$ respectively. 
They also form the maximal indecomposable finite groups of dimension $4$ 
(see Subsection \ref{ssIndmf}). 

By using {\tt flfl} in Algorithm \ref{Alg3}, 
we get $[M_G]^{fl}=0$ for $G={\rm Imf}(4,4,1)$ (see Example \ref{ex3}). 

We may also verify that $[M_G]^{fl}=0$ for $G={\rm Imf}(4,2,1)$ and 
${\rm Imf}(4,3,1)$ by using 
{\tt StablyPermutationFCheckP} in Algorithm \ref{Alg6}: Method II 
and 
{\tt StablyPermutationFCheck} in Algorithm \ref{Alg5}: Method I (1) 
respectively 
(see Example \ref{ex421} and Example \ref{ex5m11}). 
Hence, by Lemma \ref{lemp3}, 
$[M_H]^{fl}=0$ for any subgroup $H$ of 
${\rm Imf}(4,2,1)$, ${\rm Imf}(4,3,1)$ and ${\rm Imf}(4,4,1)$.\\

{\bf Step 3.} Determination of all $G$-lattices $M_G$ 
whose flabby class $[M_G]^{fl}$ is not invertible. \\

Assume that $M_G$ is indecomposable. 
By Step 2, $[M_H]^{fl}$ is invertible for 
any subgroups $H$ of ${\rm Imf}(4,2,1)$, 
${\rm Imf}(4,3,1)$ and ${\rm Imf}(4,4,1)$. 
We will check whether 
$[M_H]^{fl}$ is invertible for subgroups 
$H$ of ${\rm Imf}(4,1,1)$, 
${\rm Imf}(4,3,2)$ and ${\rm Imf}(4,5,1)$. 

There exist $182$, $36$ and $50$ indecomposable $G$-lattices 
$M_G$ where $G$ is (a conjugacy class of) a 
subgroup of ${\rm Imf}(4,1,1)$ 
${\rm Imf}(4,3,2)$ and ${\rm Imf}(4,5,1)$ respectively. 
Some of them are also subgroups of ${\rm Imf}(4,2,1)$, 
${\rm Imf}(4,3,1)$ or ${\rm Imf}(4,4,1)$ as in Step 2. 
Then we should treat 
$166$, $8$ and $27$ indecomposable $G$-lattices 
$M_G$ where $G$ is a subgroup of ${\rm Imf}(4,1,1)$ 
${\rm Imf}(4,3,2)$ and ${\rm Imf}(4,5,1)$ respectively 
but not a subgroup of ${\rm Imf}(4,2,1)$, 
${\rm Imf}(4,3,1)$ and ${\rm Imf}(4,4,1)$. 
We denote the set of the corresponding GAP IDs of $M_G$ 
in each case by $\widetilde{\mathcal{U}}_{411}$, 
$\widetilde{\mathcal{U}}_{432}$ and 
$\widetilde{\mathcal{U}}_{451}$, i.e. 
$\# \widetilde{\mathcal{U}}_{411}=166$, 
$\# \widetilde{\mathcal{U}}_{432}=8$, 
$\# \widetilde{\mathcal{U}}_{451}=27$, 
(see Example \ref{exp2} below). 

By using {\tt IsInvertibleF} in Algorithm \ref{Alg2}, 
we may determine all of the $G$-lattices $M_G$ whose flabby 
class $[M_G]^{fl}$ is not invertible. 
There exist $137$, $0$, $27$ $G$-lattices $M_G$ with $[M_G]^{fl}$ 
not invertible 
where the GAP ID of $G$ is $l\in\widetilde{\mathcal{U}}_{411}$, 
$\widetilde{\mathcal{U}}_{432}$, $\widetilde{\mathcal{U}}_{451}$ 
respectively. 
Note that 
$[M_G]^{fl}$ is invertible for $G={\rm Imf}(4,3,2)$.  
Because the $12$ $G$-lattices of them are overlapped, 
there exist $152$ $G$-lattices $M_G$ with $[M_G]^{fl}$ 
not invertible (see Example \ref{exp2} below and Table $4$). \\

{\bf Step 4.} Determination whether $[M_G]^{fl}=0$.\\

We should determine whether $[M_G]^{fl}=0$ 
for the remaining cases $\mathcal{U}_{411}$ and $\mathcal{U}_{432}$ of 
$\widetilde{\mathcal{U}}_{411}$ and $\widetilde{\mathcal{U}}_{432}$ 
as in Step $3$. 
We see that $\# \mathcal{U}_{411}=29$, $\# \mathcal{U}_{432}=8$ and 
$\mathcal{U}_{411}$ and $\mathcal{U}_{432}$ are given by 
\begin{align*}
\mathcal{U}_{411}=
\{& \text{(4,5,1,7),(4,5,1,11),(4,5,1,13),(4,6,1,10),(4,6,1,12),
(4,6,2,11),(4,12,1,6),(4,12,1,7),}\\
& \text{(4,12,3,10),(4,12,3,12),(4,12,3,13),(4,12,4,13),(4,13,1,6),
(4,13,2,6),(4,13,3,6),(4,13,4,6),}\\
& \text{(4,13,6,6),(4,13,7,12),(4,24,1,4),(4,24,1,6),(4,24,3,4),
(4,24,3,6),(4,24,4,6),(4,25,1,5),}\\
& \text{(4,25,3,5),(4,25,4,5),(4,25,7,5),(4,25,8,5),(4,33,2,1)}\},\\
\mathcal{U}_{432}=
\{& \text{(4,31,1,3),(4,31,1,4),(4,31,2,2),(4,31,3,2),(4,31,4,2),
(4,31,5,2),(4,31,6,2),(4,31,7,2)}\}.
\end{align*}

We will show that there exist exactly $7$ groups $G$ 
of the GAP IDs $(4,33,2,1)$ $\in \mathcal{U}_{411}$ and 
$(4,31,1,3)$, $(4,31,1,4)$, $(4,31,2,2)$, $(4,31,4,2)$, 
$(4,31,5,2)$ and $(4,31,7,2)$ $\in \mathcal{U}_{432}$
such that $[M_G]^{fl}\neq 0$.

First we treat the case of $\mathcal{U}_{411}$. 
Each group $G$ of the GAP ID 
$l\in\mathcal{U}_{411}$ is contained in 
at least one of the groups of 
the 21st, 27th, 28th and 29th GAP IDs, i.e. 
$(4,24,3,4)$, $(4,25,7,5)$, $(4,25,8,5)$ and $(4,33,2,1)$, 
in $\mathcal{U}_{411}$ (see Example \ref{exp3} below). 
By using {\tt flfl} in Algorithm \ref{Alg3}, 
(resp. {\tt StablyPermutationFCheckP} in 
Algorithm \ref{Alg6}: Method II), 
we may confirm that $[M_G]^{fl}=0$ for the 21st and 27th 
(resp. 28th) groups $G$ of 
the GAP ID $(4,24,3,4)$ and $(4,25,7,5)$ 
(resp. $(4,25,8,5)$) (see Example \ref{exp3} below). 
None of the groups $G$ of the GAP ID $l\in\mathcal{U}_{411}$ 
is contained in the 29th group $G$ other than itself. 
By Lemma \ref{lemp3},
$[M_G]^{fl}=0$ for all groups $G$ of the GAP ID 
$l\in\mathcal{U}_{411}$ except for the 29th group $G$ of the 
GAP ID $(4,33,2,1)$.
By using {\tt PossibilityOfStablyPermutationF} 
in Algorithm \ref{Alg5}, 
we may see that $[M_G]^{fl}\neq 0$ 
for the 29th group $G$ (see Example \ref{exp3} below). 

Next we will consider the case of $\mathcal{U}_{432}$. 
By using {\tt PossibilityOfStablyPermutationF} 
in Algorithm \ref{Alg5}, 
we see that $[M_G]^{fl}\neq 0$ 
for the 1st and the 2nd groups $G$ 
of the GAP IDs $(4,31,1,3)$ and $(4,31,1,4)$ 
(see Example \ref{exal4} and Example \ref{exp4} below). 
Because the 3rd, 5th, 6th and 8th groups $G$ 
of the GAP IDs $(4,31,2,2), (4,31,4,2), (4,31,5,2)$ and $(4,31,7,2)$ 
contain the 1st or 2nd ones, 
it follows from Lemma \ref{lemp3} that $[M_G]^{fl}\neq 0$. 
We may confirm that $[M_G]^{fl}=0$ for the 4th and 7th 
groups $G$ of the GAP IDs $(4,31,3,2)$ and $(4,31,6,2)$ 
by using {\tt StablyPermutationFCheckP} in 
Algorithm \ref{Alg6}: Method II 
(see Example \ref{exp4} below). \\


{\bf Step 5.} $[M_G]^{fl}=0$ if and only if $[M_G]^{fl}$ is 
of finite order in $\cC(G)/\cS(G)$. \\

We should show that 
if $[M_G]^{fl}\neq 0$, then $[M_G]^{fl}$ is of infinite order in $\cC(G)/\cS(G)$. 
Note that $[M_G]^{fl}$ is of finite order if and only if 
there exist permutation $G$-lattices $P$, $P^\prime$ 
such that $[(M_G)^{\oplus r}]^{fl}\oplus P\simeq P^\prime$ 
for some $r\geq 1$. 
Thus if $[M_G]^{fl}$ is not invertible, then $[M_G]^{fl}$ is of 
infinite order. 
Hence it is enough to verify that 
$[(M_G)^{\oplus r}]^{fl}\neq 0$ for the $7$ cases as in Table $2$, i.e. 
the $7$ cases where $[M_G]^{fl}$ is not zero but invertible. 
By Step 4, we should show that 
$[(M_G)^{\oplus r}]^{fl}\neq 0$ for only the $3$ groups 
$G$ of the GAP IDs $(4,31,1,3)$, $(4,31,1,4)$ and $(4,33,2,1)$ 
because the remaining $4$ groups contain the $1$st or the $2$nd one. 

By using {\tt PossibilityOfStablyPermutationF} as in Algorithm \ref{Alg4}, 
for $2$ groups $G\simeq F_{20}$ of the GAP IDs $(4,31,1,3)$ and $(4,31,1,4)$, 
we get the flabby class $[F]=[M_G]^{fl}$ with rank $16$ and see that 
the only possibility of the isomorphism (\ref{eqisom}) is 
\begin{align}
\bZ[G]\oplus\bZ[G/C_2]\oplus\bZ[G/C_5]\simeq \bZ[G/D_5]\oplus F^{\oplus 2}\label{eqF201}
\end{align}
(the rank of the both sides is $20+10+4=2+2\times 16=34$). 
By using the algorithm {\tt TateCohomology} as in Section \ref{seTate}, 
however, for the group $G$ of the GAP ID $(4,31,1,3)$, 
we see that the $2$nd (Tate) cohomologies of the both sides are 
$\widehat H^2(G_1,M_{G_1})=\bZ/2\bZ$ and $\widehat H^2(G_2,M_{G_2})=\bZ/10\bZ$ 
where $G_1$ (resp. $G_2$) is the matrix representation group 
of the action of $G$ on the left (resp. right) hand side of (\ref{eqF201}). 
Thus the isomorphism (\ref{eqF201}) is impossible. 
Hence $[F^{\oplus r}]\neq 0$ for any $r\geq 1$. 
Similarly, for the group $G$ of the GAP ID $(4,31,1,4)$, 
we see that $\widehat H^2(G_1,M_{G_1})=\bZ/5\bZ\oplus \bZ/10\bZ$ and 
$\widehat H^2(G_2,M_{G_2})=\bZ/10\bZ$ (see Example \ref{exp5} below). 

For the group $G\simeq C_3\rtimes C_8$ of the GAP ID $(4,33,2,1)$, 
by using {\tt PossibilityOfStablyPermutationF} as in Algorithm \ref{Alg4}, 
the only possibility of (\ref{eqisom}) is 
\begin{align}
\bZ[G]\oplus\bZ[G/C_2]\oplus\bZ[G/C_3]\simeq\bZ[G/C_6]\oplus F^{\oplus 2}\label{eqG24}
\end{align}
where $F$ is of rank $20$ with $[F]=[M_G]^{fl}$ 
(the rank of the both sides is $24+12+8=4+2\times 20=44$). 
However, the Tate cohomologies of the both sides 
$\widehat H^n(G_1,M_{G_1})\simeq \widehat H^n(G_2,M_{G_2})$ coincide 
for $-7\leq n\leq 7$ (see Example \ref{exp5} below). 
Thus we need another observation. 

Assume that the isomorphism (\ref{eqG24}) holds. 
Let $G_1$ (resp. $G_2$) be the matrix representation group 
of the action of $G$ on the left (resp. right) hand side of (\ref{eqG24}). 
By using {\tt StablyPermutationFCheckGen} as in Algorithm \ref{Alg6}: Method II, 
we may obtain $G_1$ and $G_2$ of rank $44$ (see Example \ref{exp5} below). 
We choose the generators $a,b$ of $G$ such that 
$G=\langle a,b\mid a^3=b^8=1, b^{-1}ab=a^{-1}\rangle\simeq 
\langle a\rangle\rtimes \langle b\rangle$. 
We take the reduction $\overline{G_1}=\langle\overline{a_1},\overline{b_1}\rangle$, 
$\overline{G_2}=\langle\overline{a_2},\overline{b_2}\rangle\leq\GL(44,\bF_3)$ 
over the field $\bF_3$ of $3$ elements. 
Because the subgroup $\langle a\rangle\simeq C_3$ is normal in $G$, 
$v_i=\{ (\overline{a_i-I_{44}})^2\mid a_i\in G_i \}$ is $\overline{G_i}$-invariant $(i=1,2)$. 
Consider the action of $b$ on the $\overline{G_i}$-invariant space $v_i$ over the 
field $\bF_9$ of $9$ elements 
and compare the dimensions of each eigen spaces of the action matrix 
of $b$ on $v_i$ $(i=1,2)$ (we see dim$(v_1)=$ dim$(v_2)=12$). 
Then we get the (multi-)set of such dimensions 
$\{ 2, 2, 2, 2, 2, 0, 2, 0 \}$ and $\{ 2, 2, 1, 2, 1, 1, 2, 1 \}$ 
for $v_1$ and $v_2$ respectively (see Example \ref{exp5} below). 
This yields a contradiction. 
Hence the isomorphism (\ref{eqG24}) is impossible.\\


{\bf Step 6.} $[M_G]^{fl}=-[J_{S_5/S_4}]^{fl}$ 
for the group $G\simeq S_5$ of the GAP ID $(4,32,5,2)$.\\

Let $G_1$, $G_2\simeq S_5\leq \GL(4,\bZ)$ be the groups of the GAP IDs 
$(4,31,4,2)$ and $(4,31,5,2)$. 
Note that $M_{G_1}\simeq J_{S_5/S_4}$. 
Let $M_G\simeq M_{G_1}\oplus M_{G_2}$ be the $S_5$-lattice of rank $8$. 
By Step 4, $[M_{G_1}]^{fl}$ and $[M_{G_2}]^{fl}$ is not zero but invertible, 
and we will show that $[M_G]^{fl}=[F]=0$. 

By using {\tt StablyPermutationFCheckGen} as in Algorithm \ref{Alg6}: Method II, 
it is possible that 
\begin{align}
\bZ[G/H_4]\oplus\bZ[G/H_{13}]\oplus F\simeq 
\bZ[G/H_8]\oplus\bZ[G/H_9]\oplus\bZ[G/H_{10}]\oplus\bZ[G/H_{11}]\label{eqS552}
\end{align}
where $H_4\simeq C_3$, $H_{13}\simeq D_5$, $H_8\simeq C_5$, 
$H_9\simeq C_6$, $H_{10}\simeq H_{11}\simeq S_3$ 
(the rank of the both sides is $40+12+32=24+20+20+20=84$). 
However, we could not establish the isomorphism (\ref{eqS552}). 

As in Example \ref{ex52} (Algorithm \ref{Alg5}: Method I (2)), 
after adding $\bZ$ to the both sides of (\ref{eqS552}), 
we may confirm that the isomorphism 
\begin{align*}
\bZ[G/C_3]\oplus\bZ[G/D_5]\oplus F\oplus\bZ\simeq 
\bZ[G/C_5]\oplus\bZ[G/C_6]\oplus\bZ[G/S_3^{(1)}]\oplus\bZ[G/S_3^{(2)}]\oplus\bZ
\end{align*}
holds (see Example \ref{exp6} below)\\

{\bf Step 7.} $[M_G]^{fl}=-[J_{F_{20}/C_4}]^{fl}$ 
for the group $G\simeq F_{20}$ of the GAP ID $(4,31,1,4)$.\\

Let $G_1^\prime$, $G_2^\prime\simeq F_{20}\leq \GL(4,\bZ)$ 
be the groups of the GAP IDs 
$(4,31,1,3)$ and $(4,31,1,4)$. 
Note that $M_{G_1^\prime}\simeq J_{F_{20}/C_4}$. 
Let $M_G\simeq M_{G_1^\prime}\oplus M_{G_2^\prime}$ be the $F_{20}$-lattice of rank $8$. 
By Step 4, $[M_{G_1^\prime}]^{fl}$ and $[M_{G_2^\prime}]^{fl}$ is 
not zero but invertible. 
We should show that $[M_G]^{fl}=[F]=0$. 
Because $G_1^\prime$ (resp. $G_2^\prime$) is a subgroup of 
$G_1$ (resp. $G_2$) where $G_1$ (resp. $G_2$) is the group $S_5$ 
as in Step 6, the assertion follows from Step 6 and Lemma \ref{lemp3}. 

We will give another (direct) proof below. 
By using {\tt StablyPermutationFCheckGen} as in Algorithm \ref{Alg6}: Method II, 
it is possible that 
\begin{align}
\bZ[G]\oplus\bZ[G/C_2]\oplus\bZ[G/C_5]\simeq \bZ[G/D_5]\oplus F\label{eqF2012}
\end{align}
(the rank of the both sides is $20+10+4=2+32=34$). 
And, after some efforts, we may confirm that the isomorphism (\ref{eqF2012}) 
actually holds (see Example \ref{exp7} below). 
\qed

\bigskip

\begin{example}[{Determination whether $[M_G]^{fl}$ is invertible}]
\label{exp2}
{}~\\
\begin{verbatim}
gap> Read("crystcat.gap");
gap> Read("caratnumber.gap");
gap> Read("FlabbyResolution.gap");
gap> Read("KS.gap");

gap> imf411sub:=Set(ConjugacyClassesSubgroups2(ImfMatrixGroup(4,1,1)),
> x->CrystCatZClass(Representative(x)));;
gap> imf421sub:=Set(ConjugacyClassesSubgroups2(ImfMatrixGroup(4,2,1)),
> x->CrystCatZClass(Representative(x)));;
gap> imf431sub:=Set(ConjugacyClassesSubgroups2(ImfMatrixGroup(4,3,1)),
> x->CrystCatZClass(Representative(x)));;
gap> imf432sub:=Set(ConjugacyClassesSubgroups2(ImfMatrixGroup(4,3,2)),
> x->CrystCatZClass(Representative(x)));;
gap> imf441sub:=Set(ConjugacyClassesSubgroups2(ImfMatrixGroup(4,4,1)),
> x->CrystCatZClass(Representative(x)));;
gap> imf451sub:=Set(ConjugacyClassesSubgroups2(ImfMatrixGroup(4,5,1)),
> x->CrystCatZClass(Representative(x)));;

gap> ind4:=LatticeDecompositions(4)[NrPartitions(4)];;
gap> imf411ind:=Intersection(imf411sub,ind4);;
gap> imf421ind:=Intersection(imf421sub,ind4);;
gap> imf431ind:=Intersection(imf431sub,ind4);;
gap> imf432ind:=Intersection(imf432sub,ind4);;
gap> imf441ind:=Intersection(imf441sub,ind4);;
gap> imf451ind:=Intersection(imf451sub,ind4);;
gap> List([imf411ind,imf421ind,imf431ind,imf432ind,imf441ind,imf451ind],Length)
[ 182, 50, 36, 36, 45, 50 ]

gap> U411t:=Difference(imf411ind,Union([imf421ind,imf431ind,imf441ind]));;
gap> U432t:=Difference(imf432ind,Union([imf421ind,imf431ind,imf441ind]));;
gap> U451t:=Difference(imf451ind,Union([imf421ind,imf431ind,imf441ind]));;
gap> List([U411t,U432t,U451t],Length);
[ 166, 8, 27 ]

gap> imf411n:=Filtered(U411t,x->IsInvertibleF(MatGroupZClass(x[1],x[2],x[3],x[4]))=false);;
gap> imf432n:=Filtered(U432t,x->IsInvertibleF(MatGroupZClass(x[1],x[2],x[3],x[4]))=false);;
gap> imf451n:=Filtered(U451t,x->IsInvertibleF(MatGroupZClass(x[1],x[2],x[3],x[4]))=false);;
gap> List([imf411n,imf432n,imf451n],Length);
[ 137, 0, 27 ]

gap> N4:=Union(imf411n,imf451n);;
gap> Length(N4);
152
gap> Intersection(imf411n,imf451n);
[ [ 4, 22, 1, 1 ], [ 4, 22, 2, 1 ], [ 4, 22, 3, 1 ], [ 4, 22, 4, 1 ], 
  [ 4, 22, 5, 1 ], [ 4, 22, 5, 2 ], [ 4, 22, 6, 1 ], [ 4, 22, 7, 1 ], 
  [ 4, 22, 8, 1 ], [ 4, 22, 9, 1 ], [ 4, 22, 10, 1 ], [ 4, 22, 11, 1 ] ]
gap> Length(last);
12

gap> U411:=Difference(U411t,imf411n);
[ [ 4, 5, 1, 7 ], [ 4, 5, 1, 11 ], [ 4, 5, 1, 13 ], [ 4, 6, 1, 10 ], [ 4, 6, 1, 12 ], 
  [ 4, 6, 2, 11 ], [ 4, 12, 1, 6 ], [ 4, 12, 1, 7 ], [ 4, 12, 3, 10 ], [ 4, 12, 3, 12 ], 
  [ 4, 12, 3, 13 ], [ 4, 12, 4, 13 ], [ 4, 13, 1, 6 ], [ 4, 13, 2, 6 ], [ 4, 13, 3, 6 ], 
  [ 4, 13, 4, 6 ], [ 4, 13, 6, 6 ], [ 4, 13, 7, 12 ], [ 4, 24, 1, 4 ], [ 4, 24, 1, 6 ], 
  [ 4, 24, 3, 4 ], [ 4, 24, 3, 6 ], [ 4, 24, 4, 6 ], [ 4, 25, 1, 5 ], [ 4, 25, 3, 5 ], 
  [ 4, 25, 4, 5 ], [ 4, 25, 7, 5 ], [ 4, 25, 8, 5 ], [ 4, 33, 2, 1 ] ]
gap> U432:=Difference(U432t,imf432n);
[ [ 4, 31, 1, 3 ], [ 4, 31, 1, 4 ], [ 4, 31, 2, 2 ], [ 4, 31, 3, 2 ], 
  [ 4, 31, 4, 2 ], [ 4, 31, 5, 2 ], [ 4, 31, 6, 2 ], [ 4, 31, 7, 2 ] ]
gap> U451:=Difference(U451t,imf451n);
[  ]
gap> List([U411,U432,U451],Length);
[ 29, 8, 0 ]
\end{verbatim}
\end{example}

\bigskip

\begin{example}[{Determination whether $[M_G]^{fl}=0$ for 
the GAP ID $l\in\mathcal{U}_{411}$}]\label{exp3}
{}~\\
\begin{verbatim}
gap> U411g:=List(U411,x->MatGroupZClass(x[1],x[2],x[3],x[4]));;
gap> U411sub:=List(U411g,x->Set(ConjugacyClassesSubgroups2(x),
> y->CrystCatZClass(Representative(y))));;
gap> U411inc:=List(U411sub,x->Set(Intersection(x,U411),y->Position(U411,y)));
[ [ 1 ], [ 2 ], [ 3 ], [ 2, 4 ], [ 3, 5 ], [ 2, 6 ], [ 7 ], [ 8 ], [ 1, 7, 9 ], 
  [ 2, 8, 10 ], [ 3, 8, 11 ], [ 3, 8, 12 ], [ 8, 13 ], [ 8, 14 ], [ 2, 3, 15 ],
  [ 2, 3, 16 ], [ 2, 3, 4, 5, 8, 10, 11, 13, 16, 17 ], 
  [ 2, 3, 5, 6, 8, 10, 12, 14, 15, 18 ], [ 1, 19 ], [ 3, 20 ], 
  [ 1, 7, 9, 19, 21 ], [ 3, 8, 11, 20, 22 ], [ 3, 8, 12, 20, 23 ], 
  [ 3, 5, 20, 24 ], [ 2, 3, 16, 20, 25 ], [ 2, 3, 15, 20, 26 ], 
  [ 2, 3, 4, 5, 8, 10, 11, 13, 16, 17, 20, 22, 24, 25, 27 ], 
  [ 2, 3, 5, 6, 8, 10, 12, 14, 15, 18, 20, 23, 24, 26, 28 ], [ 29 ] ]

gap> U411[21]; # checking [M]^{fl}=0
[ 4, 24, 3, 4 ]
gap> G:=MatGroupZClass(4,24,3,4);; # G=S4
gap> flfl(G);
[  ]

gap> U411[27]; # checking [M]^{fl}=0
[ 4, 25, 7, 5 ]
gap> G:=MatGroupZClass(4,25,7,5);; # G=C2xS4
gap> flfl(G);
[  ]

gap> U411[28]; # checking [M]^{fl}=0
[ 4, 25, 8, 5 ]
gap> G:=MatGroupZClass(4,25,8,5);; # G=C2xS4
gap> Rank(FlabbyResolution(G).actionF.1); # F is of rank 10
10

gap> mis:=SearchCoflabbyResolutionBase(TransposedMatrixGroup(G),5);;
gap> Set(List(mis,Length))-4; # Method III could not apply
[ 10, 16, 18, 22, 24, 26, 28, 30, 32, 34, 36, 38, 40, 42, 44, 46, 50 ]

gap> ll:=PossibilityOfStablyPermutationF(G);
[ [ 1, 0, 0, 0, 0, -2, -1, 0, 0, 0, 0, 0, 0, 0, 0, 0, 0, 2, 0, 0, 1, 1, 0, -1, 0, 
    0, 0, 0, 1, 1, -1, 1, -1, -1 ], 
  [ 0, 1, 0, 0, 0, 0, 0, 0, 0, 0, 0, 0, 0, 0, 0, -2, -1, 0, 0, 0, 1, 0, 0, 0, 0, 
    0, 0, 2, 1, 1, 0, 0, -1, -1 ], 
  [ 0, 0, 0, 0, 1, -1, 0, 0, 0, 0, 0, 0, 0, 0, 0, 0, 0, 1, -1, 0, 0, 1, 0, -1, 0, 
    0, 0, 0, 0, 0, -1, 1, 0, 0 ], 
  [ 0, 0, 0, 0, 0, 0, 0, 1, 0, 0, 0, 0, 0, 0, 0, 0, 0, 0, 0, 0, 1, -1, 0, -1, 0, 
    0, -1, 0, 1, 1, 1, 1, -1, -1 ], 
  [ 0, 0, 0, 0, 0, 0, 0, 0, 0, 0, 0, 0, 0, 0, 0, 0, 0, 0, 0, 1, 1, 0, 0, 0, 0, 
    0, 0, 0, -1, 0, 0, 0, 1, -1 ] ]
gap> l:=ll[Length(ll)];
[ 0, 0, 0, 0, 0, 0, 0, 0, 0, 0, 0, 0, 0, 0, 0, 0, 0, 0, 0, 1, 
  1, 0, 0, 0, 0, 0, 0, 0, -1, 0, 0, 0, 1, -1 ]
gap> bp:=StablyPermutationFCheckP(G,Nlist(l),Plist(l));;
gap> Length(bp);
20
gap> Length(bp[1]); # rank of the both sides of (10) is 13
13
gap> rs:=RandomSource(IsMersenneTwister);
<RandomSource in IsMersenneTwister>
gap> rr:=List([1..1000],x->List([1..20],y->Random(rs,[0,1])));;
gap> Filtered(rr,x->Determinant(x*bp)^2=1);
[ [ 0, 1, 1, 1, 1, 1, 0, 1, 1, 1, 0, 0, 1, 0, 0, 1, 0, 0, 0, 0 ],
  [ 0, 1, 0, 1, 1, 1, 0, 1, 0, 1, 0, 1, 0, 0, 1, 1, 0, 0, 0, 1 ] ]
gap> p:=last[1]*bp;
[ [ 0, 1, 1, 1, 0, 1, 1, 1, 1, 1, 1, 1, 1 ],
  [ 1, 0, 1, 0, 1, 1, 1, 1, 1, 1, 1, 1, 1 ],
  [ 1, 1, 0, 1, 1, 0, 1, 1, 1, 1, 1, 1, 1 ],
  [ 0, 0, 0, 1, 1, 1, 1, 1, 1, 1, 1, 1, 1 ],
  [ 1, 1, 1, 0, 0, 0, 1, 1, 1, 1, 1, 1, 1 ],
  [ 0, 0, 1, 0, 0, 2, 1, 0, 0, 0, 0, 0, 0 ],
  [ 0, 2, 0, 1, 0, 0, 1, 0, 0, 0, 0, 0, 0 ],
  [ 1, 0, 0, 0, 2, 0, 0, 1, 0, 0, 0, 0, 0 ],
  [ 0, 0, 1, 0, 0, 2, 0, 0, 0, 0, 0, 1, 0 ],
  [ 0, 0, 2, 0, 0, 1, 0, 1, 0, 0, 0, 0, 0 ],
  [ 2, 0, 0, 0, 1, 0, 0, 0, 1, 0, 0, 0, 0 ],
  [ 1, 0, 0, 0, 2, 0, 0, 0, 0, 0, 1, 0, 0 ],
  [ 2, 0, 0, 0, 1, 0, 0, 0, 0, 1, 0, 0, 0 ] ]
gap> Determinant(p);
1
gap> StablyPermutationFCheckMat(G,Nlist(l),Plist(l),p);
true

gap> U411[29]; # checking [M]^{fl}: non-zero
[ 4, 33, 2, 1 ]
gap> G:=MatGroupZClass(4,33,2,1);; # G=C3:C8
gap> PossibilityOfStablyPermutationF(G);
[ [ 1, 1, 1, 0, -1, 0, 0, 0, -2 ] ]

gap> i411:=[U411[29]];
[ [ 4, 33, 2, 1 ] ]
\end{verbatim}
\end{example}

\bigskip

\begin{example}[{Determination whether $[M_G]^{fl}=0$ for 
the GAP ID $l\in\mathcal{U}_{432}$}]\label{exp4}
{}~\\
\begin{verbatim}
gap> U432g:=List(U432,x->MatGroupZClass(x[1],x[2],x[3],x[4]));;
gap> U432sub:=List(U432g,x->Set(ConjugacyClassesSubgroups2(x),
> y->CrystCatZClass(Representative(y))));;
gap> U432inc:=List(U432sub,x->Set(Intersection(x,U432),y->Position(U432,y)));
[ [ 1 ], [ 2 ], [ 1, 2, 3 ], [ 4 ], [ 1, 4, 5 ], [ 2, 4, 6 ], [ 4, 7 ],
  [ 1, 2, 3, 4, 5, 6, 7, 8 ] ]

gap> U432[1]; # checking [M]^{fl}: non-zero
[ 4, 31, 1, 3 ]
gap> G:=MatGroupZClass(4,31,1,3);; # G=F20
gap> PossibilityOfStablyPermutationF(G);
[ [ 1, 1, 0, 1, -1, 0, -2 ] ]

gap> U432[2]; # checking [M]^{fl}: non-zero
[ 4, 31, 1, 4 ]
gap> G:=MatGroupZClass(4,31,1,4);; # G=F20
gap> PossibilityOfStablyPermutationF(G);
[ [ 1, 1, 0, 1, -1, 0, -2 ] ]

gap> U432[4]; # checking [M]^{fl}=0
[ 4, 31, 3, 2 ]
gap> G:=MatGroupZClass(4,31,3,2);; # G=A5
gap> Rank(FlabbyResolution(G).actionF.1); # F is of rank 16
16

gap> mis:=SearchCoflabbyResolutionBase(TransposedMatrixGroup(G),5);;
gap> Set(List(mis,Length))-4; # Method III could not apply
[ 16, 21, 26, 31, 36, 41, 46, 51, 56, 61, 66, 71, 76, 81 ]

gap> ll:=PossibilityOfStablyPermutationF(G);
[ [ 1, -2, -1, 0, 0, 1, 1, 1, -1, 0 ], [ 0, 0, 0, 0, 1, 1, -1, 0, 0, -1 ] ]
gap> l:=ll[Length(ll)];
[ 0, 0, 0, 0, 1, 1, -1, 0, 0, -1 ]
gap> bp:=StablyPermutationFCheckP(G,Nlist(l),Plist(l));;
gap> Length(bp); 
11
gap> Length(bp[1]); # rank of the both sides of (10) is 22
22
gap> rs:=RandomSource(IsMersenneTwister);                          
<RandomSource in IsMersenneTwister>
gap> rr:=List([1..100000],x->List([1..11],y->Random(rs,[-1..2])));;
gap> Filtered(rr,x->Determinant(x*bp)^2=1);                        
[ [ 2, 0, 2, -1, 0, 0, 1, 0, 0, 1, 1 ], [ 2, 0, 0, 1, -1, 1, 0, 1, 0, 1, 1 ] ]
gap> p:=last[1]*bp;
[ [ 2, 0, 0, 2, 0, 0, 0, 0, 0, 0, 0, 0, 2, -1, 2, 2, 2, -1, 2, -1, -1, -1 ], 
  [ 0, 2, 0, 0, 0, 2, 0, 0, 0, 0, 0, 0, 2, 2, -1, 2, -1, 2, -1, -1, -1, 2 ], 
  [ 0, 0, 2, 0, 0, 0, 0, 0, 0, 2, 0, 0, -1, 2, 2, 2, -1, -1, -1, 2, 2, -1 ], 
  [ 0, 0, 0, 0, 2, 0, 0, 0, 0, 0, 2, 0, 2, -1, -1, -1, 2, 2, -1, 2, 2, -1 ], 
  [ 0, 0, 0, 0, 0, 0, 2, 0, 0, 0, 0, 2, -1, 2, -1, -1, 2, -1, 2, 2, -1, 2 ], 
  [ 0, 0, 0, 0, 0, 0, 0, 2, 2, 0, 0, 0, -1, -1, 2, -1, -1, 2, 2, -1, 2, 2 ], 
  [ 0, 0, 0, 1, 0, 0, 0, 1, 0, 1, 0, 0, 0, 0, 1, 1, 0, 0, 1, 0, 1, 0 ], 
  [ 0, 0, 0, 1, 0, 0, 1, 1, 0, 0, 0, 0, 0, 0, 1, 0, 1, 0, 1, 0, 0, 1 ], 
  [ 0, 1, 1, 0, 0, 0, 1, 0, 0, 0, 0, 0, 0, 1, 0, 1, 0, 0, 0, 1, 0, 1 ], 
  [ 1, 0, 0, 0, 0, 0, 0, 0, 1, 0, 0, 1, 0, 0, 1, 0, 1, 0, 1, 0, 0, 1 ], 
  [ 0, 0, 0, 1, 1, 0, 1, 0, 0, 0, 0, 0, 1, 0, 0, 0, 1, 0, 1, 1, 0, 0 ], 
  [ 0, 0, 1, 0, 1, 0, 1, 0, 0, 0, 0, 0, 0, 1, 0, 0, 1, 0, 0, 1, 1, 0 ], 
  [ 1, 0, 1, 0, 0, 0, 0, 0, 1, 0, 0, 0, 0, 0, 1, 1, 0, 0, 1, 0, 1, 0 ], 
  [ 0, 0, 1, 0, 1, 0, 0, 0, 1, 0, 0, 0, 0, 0, 1, 0, 0, 1, 0, 1, 1, 0 ], 
  [ 0, 0, 0, 0, 1, 1, 0, 0, 1, 0, 0, 0, 1, 0, 0, 0, 0, 1, 0, 0, 1, 1 ], 
  [ 0, 0, 0, 0, 0, 1, 0, 0, 0, 1, 0, 1, 0, 1, 0, 1, 0, 0, 0, 1, 0, 1 ], 
  [ 1, 0, 0, 0, 0, 0, 0, 0, 0, 0, 1, 1, 1, 0, 0, 0, 1, 0, 1, 1, 0, 0 ], 
  [ 0, 0, 0, 0, 0, 0, 0, 0, 0, 1, 1, 1, 0, 1, 0, 0, 1, 0, 0, 1, 1, 0 ], 
  [ 0, 0, 0, 1, 1, 1, 0, 0, 0, 0, 0, 0, 1, 0, 0, 1, 1, 1, 0, 0, 0, 0 ], 
  [ 1, 1, 1, 0, 0, 0, 0, 0, 0, 0, 0, 0, 1, 1, 1, 1, 0, 0, 0, 0, 0, 0 ], 
  [ 0, 0, 0, 0, 0, 0, 0, 1, 0, 1, 1, 0, 0, 0, 1, 0, 0, 1, 0, 1, 1, 0 ], 
  [ 0, 0, 0, 0, 0, 1, 0, 0, 1, 0, 0, 1, 0, 1, 0, 0, 0, 1, 1, 0, 0, 1 ] ]
gap> Determinant(p);
1
gap> StablyPermutationFCheckMat(G,Nlist(l),Plist(l),p);            
true

gap> U432[7]; # checking [M]^{fl}=0
[ 4, 31, 6, 2 ]
gap> G:=MatGroupZClass(4,31,6,2);; # G=C2xA5
gap> Rank(FlabbyResolution(G).actionF.1); # F is of rank 16
16

gap> mis:=SearchCoflabbyResolutionBase(TransposedMatrixGroup(G),5);;
gap> Set(List(mis,Length))-4; # Method III could not apply
[ 16, 26, 36, 46, 56, 66, 76, 86, 96, 106, 116, 126, 136, 146, 156, 166 ]

gap> ll:=PossibilityOfStablyPermutationF(G);
[ [ 1, 0, 0, -2, 0, 0, 0, 0, -1, 0, -1, 1, 0, 0, 1, 0, -1, -1, 1, 2, 1, -2, 1 ], 
  [ 0, 1, 0, 0, 0, 0, 0, -2, 0, 0, -1, 0, 0, 0, 0, 0, 0, 1, 1, 1, 0, -1, 0 ], 
  [ 0, 0, 1, -1, 0, 0, 0, 0, 0, 0, 0, 1, 0, 0, 0, -1, -1, -1, 1, 1, 1, -1, 0 ], 
  [ 0, 0, 0, 0, 1, 0, 0, 0, -1, -1, -1, -1, 0, 0, 1, 1, 0, 1, -1, 0, 0, 0, 1 ], 
  [ 0, 0, 0, 0, 0, 0, 0, 0, 0, 0, 0, 0, 0, 1, 0, 0, 0, 1, -1, 0, 0, 0, -1 ] ]
gap> l:=ll[Length(ll)];
[ 0, 0, 0, 0, 0, 0, 0, 0, 0, 0, 0, 0, 0, 1, 0, 0, 0, 1, -1, 0, 0, 0, -1 ]
gap> bp:=StablyPermutationFCheckP(G,Nlist(l),Plist(l));;
gap> Length(bp);
11
gap> Length(bp[1]); # rank of the both sides of (10) is 22
22
gap> rs:=RandomSource(IsMersenneTwister);                       
<RandomSource in IsMersenneTwister>
gap> rr:=List([1..100000],x->List([1..11],y->Random(rs,[-1..2])));;
gap> Filtered(rr,x->Determinant(x*bp)^2=1);
[ [ 2, 1, 0, 1, 0, 0, -1, 0, 0, 0, -1 ], [ 2, 0, 1, 0, 1, -1, 0, 1, 0, 1, 1 ], 
  [ 2, 0, 1, 0, 1, -1, 0, 1, 0, 1, 1 ], [ 2, 0, 0, 1, 1, -1, -1, 0, 0, -1, -1 ], 
  [ 2, 1, 2, -1, 0, 0, -1, 0, -1, 1, 2 ], [ 2, 1, -1, 0, 0, 0, 0, -1, 0, 0, 1 ] ]
gap> p:=last[1]*bp;                        
[ [ 2, 1, 1, 1, 1, 1, 2, 1, 1, 1, 1, 1, 0, 1, 0, 1, 0, 0, 0, 1, 1, 1 ], 
  [ 1, 2, 1, 1, 1, 1, 1, 1, 1, 1, 1, 2, 1, 0, 0, 1, 1, 0, 1, 0, 1, 0 ], 
  [ 1, 1, 2, 1, 2, 1, 1, 1, 1, 1, 1, 1, 0, 0, 1, 0, 1, 0, 1, 1, 0, 1 ], 
  [ 1, 1, 1, 2, 1, 1, 1, 1, 1, 2, 1, 1, 1, 1, 0, 0, 0, 1, 1, 1, 0, 0 ], 
  [ 1, 1, 1, 1, 1, 2, 1, 1, 2, 1, 1, 1, 1, 0, 1, 1, 0, 1, 0, 0, 0, 1 ], 
  [ 1, 1, 1, 1, 1, 1, 1, 2, 1, 1, 2, 1, 0, 1, 1, 0, 1, 1, 0, 0, 1, 0 ], 
  [ 0, 0, -1, 0, 0, 0, 0, -1, 0, 0, 0, -1, 0, 0, 0, 0, -1, 0, 0, 0, 0, 0 ], 
  [ -1, 0, 0, 0, -1, 0, 0, 0, -1, 0, 0, 0, 0, 0, 0, 0, 0, 0, 0, 0, 0, -1 ], 
  [ -1, 0, 0, 0, -1, 0, 0, 0, 0, -1, 0, 0, 0, 0, 0, 0, 0, 0, 0, -1, 0, 0 ], 
  [ 0, 0, 0, -1, 0, 0, 0, 0, -1, 0, -1, 0, 0, 0, 0, 0, 0, -1, 0, 0, 0, 0 ], 
  [ 0, 0, -1, -1, 0, 0, 0, 0, 0, 0, 0, -1, 0, 0, 0, 0, 0, 0, -1, 0, 0, 0 ], 
  [ 0, 0, 0, -1, 0, 0, 0, 0, -1, 0, 0, -1, -1, 0, 0, 0, 0, 0, 0, 0, 0, 0 ], 
  [ 0, 0, -1, 0, 0, -1, 0, -1, 0, 0, 0, 0, 0, 0, -1, 0, 0, 0, 0, 0, 0, 0 ], 
  [ 0, 0, -1, 0, 0, -1, -1, 0, 0, 0, 0, 0, 0, 0, 0, 0, 0, 0, 0, 0, 0, -1 ], 
  [ 0, -1, 0, 0, 0, -1, -1, 0, 0, 0, 0, 0, 0, 0, 0, -1, 0, 0, 0, 0, 0, 0 ], 
  [ 0, 0, 0, -1, 0, 0, -1, 0, 0, 0, -1, 0, 0, -1, 0, 0, 0, 0, 0, 0, 0, 0 ], 
  [ 0, 0, 0, 0, -1, 0, 0, 0, -1, 0, -1, 0, 0, 0, -1, 0, 0, 0, 0, 0, 0, 0 ], 
  [ -1, 0, 0, 0, 0, 0, 0, -1, 0, -1, 0, 0, 0, -1, 0, 0, 0, 0, 0, 0, 0, 0 ], 
  [ 0, -1, 0, 0, -1, 0, 0, 0, 0, -1, 0, 0, 0, 0, 0, 0, 0, 0, -1, 0, 0, 0 ], 
  [ -1, 0, 0, 0, 0, 0, 0, -1, 0, 0, 0, -1, 0, 0, 0, 0, 0, 0, 0, 0, -1, 0 ], 
  [ 0, -1, 0, 0, 0, -1, 0, 0, 0, -1, 0, 0, -1, 0, 0, 0, 0, 0, 0, 0, 0, 0 ], 
  [ 0, -1, 0, 0, 0, 0, -1, 0, 0, 0, -1, 0, 0, 0, 0, 0, 0, 0, 0, 0, -1, 0 ] ]
gap> Determinant(p);
-1
gap> StablyPermutationFCheckMat(G,Nlist(l),Plist(l),p);
true

gap> i432:=Difference(U432,[U432[4],U432[7]]);
[ [ 4, 31, 1, 3 ], [ 4, 31, 1, 4 ], [ 4, 31, 2, 2 ], [ 4, 31, 4, 2 ],
  [ 4, 31, 5, 2 ], [ 4, 31, 7, 2 ] ]

gap> I4:=Union(i411,i432);
[ [ 4, 31, 1, 3 ], [ 4, 31, 1, 4 ], [ 4, 31, 2, 2 ], [ 4, 31, 4, 2 ],
  [ 4, 31, 5, 2 ], [ 4, 31, 7, 2 ], [ 4, 33, 2, 1 ] ]
gap> List(I4,x->StructureDescription(MatGroupZClass(x[1],x[2],x[3],x[4])));
[ "C5 : C4", "C5 : C4", "C2 x (C5 : C4)", "S5", "S5", "C2 x S5", "C3 : C8" ]
gap> Length(I4);
7
\end{verbatim}
\end{example}

\bigskip

\begin{example}
[{$[M_G]^{fl}=0$ if and only if $[M_G]^{fl}$ is 
of finite order in $\cC(G)/\cS(G)$ for the groups $G$ of 
the GAP IDs $(4,31,1,3)$, $(4,31,1,4)$ and $(4,33,2,1)$}]
\label{exp5}
{}~\\
\begin{verbatim}
gap> Read("Hn.gap");

gap> G:=MatGroupZClass(4,31,1,3);; # G=F20
gap> Rank(FlabbyResolution(G).actionF.1); # F is of rank 16
16
gap> ll:=PossibilityOfStablyPermutationF(G);
[ [ 1, 1, 0, 1, -1, 0, -2 ] ]
gap> l:=ll[1];
[ 1, 1, 0, 1, -1, 0, -2 ]
gap> List(ConjugacyClassesSubgroups2(G),x->StructureDescription(Representative(x)));
[ "1", "C2", "C4", "C5", "D10", "C5 : C4" ]
gap> gg:=StablyPermutationFCheckGen(G,Nlist(l),Plist(l));;
gap> G1:=Group(gg[1]);; 
gap> G2:=Group(gg[2]);;
gap> [Length(G1.1),Length(G2.1)]; # rank of the both sides is 34
34
gap> TateCohomology(G1,2);
[ 2 ]
gap> TateCohomology(G2,2);
[ 10 ]

gap> G:=MatGroupZClass(4,31,1,4);; # G=F20
gap> ll:=PossibilityOfStablyPermutationF(G); # F is of rank 16
[ [ 1, 1, 0, 1, -1, 0, -2 ] ]
gap> l:=ll[1];
[ 1, 1, 0, 1, -1, 0, -2 ]
gap> List(ConjugacyClassesSubgroups2(G),x->StructureDescription(Representative(x)));
[ "1", "C2", "C4", "C5", "D10", "C5 : C4" ]
gap> gg:=StablyPermutationFCheckGen(G,Nlist(l),Plist(l));;
gap> G1:=Group(gg[1]);;
gap> G2:=Group(gg[2]);;
gap> [Length(G1.1),Length(G2.1)]; # rank of the both sides is 34
34
gap> TateCohomology(G1,2);
[ 5, 10 ]
gap> TateCohomology(G2,2);
[ 10 ]

gap> G:=MatGroupZClass(4,33,2,1);; # G=C3:C8
gap> Rank(FlabbyResolution(G).actionF.1); # F is of rank 20
20
gap> ll:=PossibilityOfStablyPermutationF(G);
[ [ 1, 1, 1, 0, -1, 0, 0, 0, -2 ] ]
gap> l:=ll[1];
[ 1, 1, 1, 0, -1, 0, 0, 0, -2 ]
gap> List(ConjugacyClassesSubgroups2(G),x->StructureDescription(Representative(x)));
[ "1", "C2", "C3", "C4", "C6", "C8", "C12", "C3 : C8" ]
gap> gg:=StablyPermutationFCheckGen(G,Nlist(l),Plist(l));;

gap> G1:=Group(gg[1]);
<matrix group with 2 generators>
gap> G2:=Group(gg[2]);
<matrix group with 2 generators>
gap> [Length(G1.1),Length(G2.1)]; # rank of the both sides is 44
[ 44, 44 ]
gap> List([-7..7],i->TateCohomology(G1,i)); # Tate cohomologies are coincide 
[ [  ], [ 6 ], [  ], [ 6 ], [  ], [ 6 ], [ 0, 0, 0 ], [ 1, 1, 6 ], 
  [  ], [ 6 ], [  ], [ 6 ], [  ], [ 6 ], [  ] ]
gap> List([-7..7],i->TateCohomology(G2,i));
[ [  ], [ 6 ], [  ], [ 6 ], [  ], [ 6 ], [ 0, 0, 0 ], [ 1, 1, 6 ], 
  [  ], [ 6 ], [  ], [ 6 ], [  ], [ 6 ], [  ] ]

gap> Length(GeneratorsOfGroup(G)); # number of generators of G is 2
2
gap> List([G.1,G.2],Order); # order of two generators are 8, 12
[ 8, 12 ]
gap> a:=G.2^4;
[ [ -1, -1, -1, 1 ], [ 1, 0, 0, -1 ], [ 0, 0, -1, 1 ], [ 0, 0, -1, 0 ] ]
gap> b:=G.1;
[ [ 0, 0, -1, 0 ], [ -1, 0, 0, 0 ], [ 1, 1, 1, -2 ], [ 0, 1, 0, -1 ] ]
gap> List([a,b],Order);
[ 3, 8 ]
gap> a^b=a^2;
true
gap> G=Group(a,b);
true

gap> G1_3:=Group(gg[1]*Z(3));;
gap> G2_3:=Group(gg[2]*Z(3));;
gap> v1:=VectorSpace(GF(3^2),(G1_3.2^4-IdentityMat(44)*Z(3)^0)^2);
<vector space over GF(3^2), with 44 generators>
gap> v2:=VectorSpace(GF(3^2),(G2_3.2^4-IdentityMat(44)*Z(3)^0)^2);
<vector space over GF(3^2), with 44 generators>
gap> List([v1,v2],Dimension);
[ 12, 12 ]

gap> Eigenvalues(GF(3^2),G1_3.1);
[ Z(3), Z(3)^0, Z(3^2)^5, Z(3^2)^6, Z(3^2)^7, Z(3^2), Z(3^2)^2, Z(3^2)^3 ]
gap> Eigenvalues(GF(3^2),G2_3.1);
[ Z(3), Z(3)^0, Z(3^2)^5, Z(3^2)^6, Z(3^2)^7, Z(3^2), Z(3^2)^2, Z(3^2)^3 ]
gap> e1:=Eigenspaces(GF(3^2),G1_3.1);
[ <vector space over GF(3^2), with 7 generators>, 
  <vector space over GF(3^2), with 7 generators>, 
  <vector space over GF(3^2), with 4 generators>, 
  <vector space over GF(3^2), with 7 generators>, 
  <vector space over GF(3^2), with 4 generators>, 
  <vector space over GF(3^2), with 4 generators>, 
  <vector space over GF(3^2), with 7 generators>, 
  <vector space over GF(3^2), with 4 generators> ]
gap> e2:=Eigenspaces(GF(3^2),G2_3.1);
[ <vector space over GF(3^2), with 7 generators>, 
  <vector space over GF(3^2), with 7 generators>, 
  <vector space over GF(3^2), with 4 generators>, 
  <vector space over GF(3^2), with 7 generators>, 
  <vector space over GF(3^2), with 4 generators>, 
  <vector space over GF(3^2), with 4 generators>, 
  <vector space over GF(3^2), with 7 generators>, 
  <vector space over GF(3^2), with 4 generators> ]
gap> List(e1,x->Dimension(Intersection(x,v1)));
[ 2, 2, 2, 2, 2, 0, 2, 0 ]
gap> List(e2,x->Dimension(Intersection(x,v2)));                  
[ 2, 2, 1, 2, 1, 1, 2, 1 ]
\end{verbatim}
\end{example}

\bigskip

\begin{example}[{$[M_G]^{fl}+[J_{S_5/S_4}]^{fl}=0$ 
for the group $G\simeq S_5$ of the GAP ID $(4,31,5,2)$}]
\label{exp6}
{}~\\
\begin{verbatim}
gap> ip:=InverseProjection([[4,31,4,2],[4,31,5,2]]);
[ <matrix group of size 120 with 3 generators>, 
  <matrix group of size 14400 with 6 generators>, 
  <matrix group of size 7200 with 5 generators> ]
gap> G:=ip[1]; # G=S5
<matrix group of size 120 with 3 generators>
gap> Rank(FlabbyResolution(G).actionF.1); # F is of rank 32
32
gap> ll:=PossibilityOfStablyPermutationF(G);
[ [ 1, 0, 0, 0, 0, -4, 0, -1, 1, 0, -3, 0, 0, -1, 0, 4, 4, 1, -4, 1 ], 
  [ 0, 1, 0, 0, 0, -1, -1, 0, 0, 0, -1, 0, 0, 0, 1, 1, 1, 0, -1, 0 ], 
  [ 0, 0, 1, 0, 0, -2, 0, 0, 1, 0, -1, 0, -1, -1, 0, 2, 2, 1, -2, 0 ], 
  [ 0, 0, 0, 1, 0, 0, 0, -1, -1, -1, -1, 0, 1, 0, 0, 0, 0, 0, 0, 1 ], 
  [ 0, 0, 0, 0, 1, -2, 2, 0, 2, 1, -2, -2, -1, -2, -2, 2, 4, 1, -2, 0 ] ]
gap> l:=ll[4];
[ 0, 0, 0, 1, 0, 0, 0, -1, -1, -1, -1, 0, 1, 0, 0, 0, 0, 0, 0, 1 ]
gap> Length(l);
20
gap> [l[4],l[13],l[20],l[8],l[9],l[10],l[11]];
[ 1, 1, 1, -1, -1, -1, -1 ]
gap> ss:=List(ConjugacyClassesSubgroups2(G),x->StructureDescription(Representative(x)));
[ "1", "C2", "C2", "C3", "C2 x C2", "C2 x C2", "C4", "C5", "C6", "S3", 
  "S3", "D8", "D10", "A4", "D12", "C5 : C4", "S4", "A5", "S5" ]
gap> [ss[4],ss[13],ss[8],ss[9],ss[10],ss[11]];       
[ "C3", "D10", "C5", "C6", "S3", "S3" ]
gap> l2:=IdentityMat(Length(l))[Length(l)-1];
[ 0, 0, 0, 0, 0, 0, 0, 0, 0, 0, 0, 0, 0, 0, 0, 0, 0, 0, 1, 0 ]
gap> bp:=StablyPermutationFCheckP(G,Nlist(l)+l2,Plist(l)+l2);; 
gap> Length(bp);
85
gap> Length(bp[1]); # rank of the both sides is 85
85

# after some efforts we may get 
gap> n:=[ 
> -1, 1, 1, 1, -1, -1, -1, 0, 0, 1, 1, 1, 0, 0, -1, 0, 0, 0, -1, 0, 0, -1, -1, -1, 0, 
> 0, 0, 0, 0, -1, 0, 0, 0, 0, -1, 0, -1, 0, 1, 0, -1, -1, 1, 0, 0, -1, 0, 0, 1, 1, 
> 1, 0, 0, 0, 0, 1, 0, 1, 1, 0, 0, 0, 1, -1, -1, 0, -1, -1, -1, 1, 0, 1, -1, 1, -1, 
> -1, 1, 1, 1, 1, 5, 2, -3, 3, -2 ]
gap> p:=n*bp;;
gap> StablyPermutationFCheckMat(G,Nlist(l)+l2,Plist(l)+l2,p);
true
\end{verbatim}
\end{example}

\bigskip

\begin{example}[{$[M_G]^{fl}+[J_{F_{20}/C_4}]^{fl}=0$ 
for the group $G\simeq F_{20}$ of the GAP ID $(4,31,1,4)$}]
\label{exp7}
{}~\\
\begin{verbatim}
gap> ip:=InverseProjection([[4,31,1,3],[4,31,1,4]]);
[ <matrix group of size 20 with 3 generators>, 
  <matrix group of size 400 with 4 generators>, 
  <matrix group of size 200 with 5 generators>, 
  <matrix group of size 100 with 4 generators>, 
  <matrix group of size 100 with 4 generators> ]
gap> G:=ip[1];; # G=F20
gap> Rank(FlabbyResolution(G).actionF.1); # F is of rank 32
32
gap> ll:=PossibilityOfStablyPermutationF(G);
[ [ 1, 1, 0, 1, -1, 0, -1 ] ]
gap> l:=ll[1];
[ 1, 1, 0, 1, -1, 0, -1 ]
gap> List(ConjugacyClassesSubgroups2(G),x->StructureDescription(Representative(x)));
[ "1", "C2", "C4", "C5", "D10", "C5 : C4" ]
gap> bp:=StablyPermutationFCheckP(G,Nlist(l),Plist(l));;
gap> Length(bp);
62
gap> Length(bp[1]); # rank of the both sides of (10) is 34
34

# after some efforts we may get 
gap> n:=[ 
> 1, 0, 0, 1, 0, 1, 1, 0, 1, 1, 1, 1, 1, -1, 0, 1, 0, -1, 1, 1, 0, 0, -1, -1, 0, 
> 1, 1, -1, 0, 1, -1, -1, 0, 0, 1, -1, 0, 1, 1, 0, 0, 0, 0, 1, 1, -1, 0, -1, 1, -1, 
> 1, 0, 1, 1, 0, -1, -1, 1, -1, -1, 0, -1 ];
gap> p:=n*bp;;
gap> Determinant(p);
1
gap> StablyPermutationFCheckMat(G,Nlist(l),Plist(l),p);
true
\end{verbatim}
\end{example}

\bigskip

\section{Proof of Theorem \ref{th2M}}\label{seProof2}
Let $G$ be a finite subgroup of $\GL(5,\bZ)$ and $M=M_G$ be 
the corresponding $G$-lattice of rank $5$ as in Definition \ref{defMG}. \\


{\bf Step 1.} The case where $M_G$ is decomposable. \\

We assume that $M_G\simeq M_1\oplus M_2$ is decomposable with 
$\rank M_1\geq \rank M_2\geq 1$. 
Hence $\rank M_1\leq 4$ and $\rank M_2\leq 2$. 
It follows from Theorem \ref{thVo} and Lemma \ref{lemp1} 
that $[M_2]^{fl}=0$ and $[M_G]^{fl}=[M_1]^{fl}$. 

There exist exactly $25$ $G$-lattices $M_G\simeq M_1\oplus M_2$ 
with $[M_G]^{fl}\neq 0$ but invertible 
whose GAP IDs $\mathcal{I}_{41}$ are computed in Example \ref{exN} 
(see also Table $11$). 

There exist exactly $245$ (resp. $849$, $768$) $G$-lattices 
$M_G\simeq M_1\oplus M_2\oplus M_3$ 
with rank $M_1=3$, rank $M_2=1$, rank $M_3=1$ 
(resp. 
$M_G\simeq M_1\oplus M_2$ 
with rank $M_1=3$, rank $M_2=2$, 
$M_G\simeq M_1\oplus M_2$ 
with rank $M_1=4$, rank $M_2=1$) 
and $[M_G]^{fl}\neq 0$ 
whose GAP IDs $\mathcal{N}_{311}$ 
(resp. 
$\mathcal{N}_{32}$, $\mathcal{N}_{41}$) 
are computed in Example \ref{exN} 
(see also Tables $12$, $13$ and $14$). \\

{\bf Step 2.}  Determination of all $G$-lattices $M_G$ 
whose flabby class $[M_G]^{fl}$ is not invertible. \\

We assume that $M_G$ is indecomposable. 
There exist $1452$ indecomposable $G$-lattices $M_G$ 
of rank $5$ (see Example \ref{exKS1}). 
By using {\tt IsInvertibleF} as in Algorithm \ref{Alg2}, 
we see that there exist exactly $1141$ $G$-lattices $M_G$ 
with $[M_G]^{fl}$ not invertible (see Example \ref{exq1} below). 
In the next step, we will show that all the 
remaining $311$ cases satisfy $[M_G]^{fl}=0$.\\

{\bf Step 3.} Verification of $[M_G]^{fl}=0$ for all 
the remaining $311$ cases.\\

First, we see that for 
all the remaining $311$ $G$-lattices $M_G$ 
$G$ is a subgroup of at least one of 
the $18$ groups of the CARAT IDs as in Table $9$ 
(see Example \ref{exq2} below). 
Hence by Lemma \ref{lemp3} it is enough to show that $[M_G]^{fl}=0$ 
for the $18$ groups $G$ in Table $9$.\\

\begin{center}
Table $9$: the maximal $18$ groups 
in the remaining $311$ cases\vspace*{2mm}\\
\begin{tabular}{llll} 
CARAT ID & $G$ & $\#G$ & Method\\\hline
$(5,942,1)$ & ${\rm Imf}(5,1,1)$ & $3840$ & {\tt flfl}\\
$(5,953,4)$ & $S_6$ & $720$ & {\tt flfl}\\
$(5,726,4)$ & 
& $384$ & {\tt flfl}\\
$(5,919,4)$ & $C_2\times S_5$ & $240$ &  Method III\\
$(5,801,3)$ & $C_2\times (S_3^2\rtimes C_2)$ & $144$ &  Method III\\
$(5,655,4)$ & $D_4^2\rtimes C_2$ & $128$ & {\tt flfl}\\
$(5,911,4)$ & $S_5$ & $120$ & Method I (1)\\
$(5,946,2)$ & $S_5$ & $120$ & Method I (2)\\
$(5,946,4)$ & $S_5$ & $120$ & Method II $+\alpha$
\end{tabular}\ \ 
\begin{tabular}{llll} 
CARAT ID & $G$ & $\#G$ & Method\\\hline
$(5,947,2)$ & $S_5$ & $120$ & Method III $+\alpha$\\
$(5,337,12)$ & $D_4\times S_3$ & $48$ & Method III\\
$(5,341,6)$ & $D_4\times S_3$ & $48$ & Method III\\
$(5,531,13)$ & $C_2\times S_4$ & $48$ & Method III\\
$(5,533,8)$ & $C_2\times S_4$ & $48$ & Method III\\
$(5,623,4)$ & $C_2\times S_4$ & $48$ & Method III\\
$(5,245,12)$ & $C_2^2\times S_3$ & $24$ & Method III\\
$(5,81,42)$ & $C_2\times D_4$ & $16$ & {\tt flfl}\\
$(5,81,48)$ & $C_2\times D_4$ & $16$ & {\tt flfl}
\end{tabular}
\end{center}

\bigskip

Using {\tt flfl} as in Algorithm \ref{Alg3} once or twice, 
we see that $[M_G]^{fl}=0$ for the $6$ groups $G$ of 
the CARAT IDs $(5,942,1)$, $(5,953,4)$, $(5,726,4)$, $(5,655,4)$, 
$(5,81,42)$ and $(5,81,48)$ 
(see Example \ref{exq3} below). 

Using functions in Algorithm \ref{Alg7}: Method III, 
we may confirm that $[M_G]^{fl}=0$ for the $8$ 
groups $G$ of the CARAT IDs $(5,919,4)$, $(5,801,3)$, 
$(5,337,12)$, $(5,341,6)$, $(5,531,13)$, $(5,533,8)$, $(5,623,4)$ and 
$(5,245,12)$ (see Example \ref{exq4} below). 
Thus there are $4$ remaining cases which 
are groups $G\simeq S_5$ of order $120$ 
of the CARAT IDs $(5,911,4)$, $(5,946,2)$, $(5,946,4)$, $(5,947,2)$. 

For two group $G$ of the CARAT IDs $(5,911,4)$ and $(5,946,2)$, 
by Method I (1) and Method I (2) respectively, we have that $[M_G]^{fl}=0$ 
(see Example \ref{exq5} below and Example \ref{ex52}). 
Indeed, for the group $G$ of the CARAT ID $(5,911,4)$, 
we see that $M_G$ is flabby and coflabby and hence 
stably permutation by Theorem \ref{thfac}. 
This implies $[M_G]^{fl}=0$. 
We also see that for the group $G$ of the CARAT ID $(5,946,2)$, 
the CARAT ID of $F$ with $[F]=[M_G]^{fl}$ is $(5,911,4)$. 
Hence we confirm that $[M_G]^{fl}=0$. 

For the group $G\simeq S_5$ of the CARAT ID $(5,947,2)$, 
using Algorithm \ref{Alg7}: Method III, 
we may find the flabby class $[M_G]^{fl}=[F]$ where $F$ is of rank $21$. 
However, the rank of the both sides of (\ref{eqisom}) becomes 
$81$. 
Thus we take another $F$ of rank $25$ 
(see Example \ref{exq6} below). 
Then the rank of the both sides of (\ref{eqisom}) becomes $55$. 
By using {\tt PossibilityOfStablyPermutationFFromBase}, 
it is possible that 
\begin{align}
\bZ[G/H_6]\oplus\bZ[G/H_{11}]\oplus\bZ[G/H_{17}]\simeq 
\bZ[G/H_{10}]\oplus\bZ[G/H_{15}]\oplus F\label{eq59472}
\end{align}
where 
$H_6\simeq C_2^2$, $H_{11}\simeq S_3$, $H_{17}\simeq S_4$, 
$H_{10}\simeq S_3$ and $H_{15}\simeq D_6$ (the rank of the both sides 
is $30+20+5=20+10+25=55$). 
We could not establish the isomorphism (\ref{eq59472}). 
However, as in Example \ref{ex52} (Algorithm \ref{Alg5}: Method I (2)), 
after adding $\bZ$ to the both sides of (\ref{eq59472}), 
we may confirm that the isomorphism 
\begin{align*}
\bZ[G/H_6]\oplus\bZ[G/H_{11}]\oplus\bZ[G/H_{17}]\oplus\bZ \simeq 
\bZ[G/H_{10}]\oplus\bZ[G/H_{15}]\oplus F\oplus\bZ
\end{align*}
holds (see Example \ref{exq6}). 

For the group $G\simeq S_5$ of the CARAT ID $(5,946,4)$, 
Method III does not work well. 
We may obtain $[M_G]^{fl}=[F]$ 
with $F$ rank $17$ by {\tt FlabbyResolution(G).actionF}. 
Using {\tt PossibilityOfStablyPermutationF}, it turns out that 
the isomorphism 
\begin{align}
&\bZ[G/H_6]\oplus\bZ[H_9]\oplus\bZ[G/H_{11}]\oplus\bZ[G/H_{16}]\oplus
\bZ[G/H_{17}]^{\oplus 2}\oplus\bZ[G/H_{18}]\label{eq59464}\\
&\simeq 
\bZ[G/H_7]\oplus\bZ[G/H_{10}]\oplus\bZ[G/H_{14}]\oplus\bZ\oplus F\nonumber
\end{align}
may occur where 
$H_6\simeq C_2^2$, $H_9\simeq H_{11}\simeq S_3$, 
$H_{16}\simeq F_{20}$, $H_{17}\simeq S_4$, $H_{18}\simeq A_5$, 
$H_7\simeq C_4$, $H_{10}\simeq C_6$, 
$H_{14}\simeq A_4$ and $H_{15}\simeq D_6$ 
(the rank of the both sides is $30+20+20+6+2\times 5+2=30+20+10+10+1+17=88$). 
After some efforts, we see that the isomorphism 
(\ref{eq59464}) actually holds (see Example \ref{exq7} below).\\

{\bf Step 4.} $[M_G]^{fl}=0$ if and only if $[M_G]^{fl}$ is 
of finite order in $\cC(G)/\cS(G)$. \\

We should show that if $G$ is one of the $25$ groups 
as in Table $11$, i.e. the $25$ cases where $[M_G]^{fl}$ is not zero but invertible, 
then $[(M_G)^{\oplus r}]^{fl}\neq 0$ for any $r\geq 1$ 
(see also Step 5 of Section \ref{seProof1}). 
By Remark \ref{rem25}, $[M_G]^{fl}=[M_1]^{fl}$ where 
$M_G\simeq M_1\oplus M_2$, $M_1$ is a $G/N_1$-lattice of rank $4$ which 
is one of the $7$ cases as in Table $2$. 
Note that $[M_1]^{fl}=\rho_G(M_1)=0$ if and only if 
$\rho_{G/N_1}(M_1)=0$ by Lemma \ref{lemp1}. 
Hence the assertion follows from Step 5 of Section \ref{seProof1}. 
\qed

\bigskip

\begin{example}[{Determination of all the cases where 
$[M_G]^{fl}$ is not invertible}]
\label{exq1}
{}~\\
\begin{verbatim}
gap> Read("caratnumber.gap");
gap> Read("FlabbyResolution.gap");
gap> Read("KS.gap");

gap> ind5:=LatticeDecompositions(5:Carat)[NrPartitions(5)];;
gap> Length(ind5);
1452
gap> N5:=Filtered(ind5,x->IsInvertibleF(CaratMatGroupZClass(x[1],x[2],x[3]))=false);;
gap> Length(N5);
1141
\end{verbatim}
\end{example}

\bigskip

\begin{example}[{The maximal $18$ groups 
in the remaining $311$ cases }]
\label{exq2}
{}~\\
\begin{verbatim}
gap> U5:=Difference(ind5,N5);;
gap> Length(U5);
311
gap> gg18:=[[5,942,1],[5,953,4],[5,726,4],[5,919,4],[5,801,3],[5,655,4], 
> [5,911,4],[5,946,2],[5,946,4],[5,947,2],[5,337,12],[5,341,6],
> [5,531,13],[5,533,8],[5,623,4],[5,245,12],[5,81,42],[5,81,48]];;
gap> Length(gg18);
18
gap> gg18sub:=Union(Set(gg18,y->Set(ConjugacyClassesSubgroups2(y),
> x->CaratZClass(Representative(x)))));
gap> Difference(e5a,gg18sub);
[  ]
\end{verbatim}
\end{example}

\bigskip

\begin{example}[{Verification of $[M_G]^{fl}=0$ for the $6$ 
groups $G$ of the CARAT ID $(5,942,1)$, 
$(5,953,4)$, $(5,726,4)$, $(5,655,4)$, $(5,81,42)$ and $(5,81,48)$: 
Method {\tt flfl}}]\label{exq3}
{}~\\
\begin{verbatim}
gap> flfl(CaratMatGroupZClass(5,942,1)); # G=Imf(5,1,1)
[  ]
gap> flfl(CaratMatGroupZClass(5,953,4)); # G=S6
[  ]
gap> flfl(CaratMatGroupZClass(5,726,4)); # #G=384
[  ]
gap> flfl(flfl(CaratMatGroupZClass(5,655,4))); # G=(D4xD4):C2
[  ]
gap> flfl(flfl(CaratMatGroupZClass(5,81,42))); # G=C2xD4
[  ]
gap> flfl(flfl(CaratMatGroupZClass(5,81,48))); # G=C2xD4
[  ]
\end{verbatim}
\end{example}

\bigskip

\begin{example}[{Verification of $[M_G]^{fl}=0$ for the $8$ 
groups $G$ of the CARAT IDs $(5,919,4)$, $(5,801,3)$, 
$(5,337,12)$, $(5,341,6)$, $(5,531,13)$, $(5,533,8)$, $(5,623,4)$ and 
$(5,245,12)$}: Method III]\label{exq4}
{}~\\
\begin{verbatim}
gap> G:=CaratMatGroupZClass(5,919,4);; # G=C2xS5 
gap> Rank(FlabbyResolution(G).actionF.1); # F is of rank 27
27
gap> mis:=SearchCoflabbyResolutionBase(TransposedMatrixGroup(G),3);; # Method III
gap> List(mis,Length);                                              
[ 42, 32, 42, 82, 72, 42, 32, 32, 42, 32, 12 ]
gap> mi:=mis[Length(mis)]; # (new) F is of rank 7 (=12-5)
[ [ -1, 0, 0, 0, -1 ], [ -1, 0, 1, 1, -2 ], [ -1, 1, 1, 0, -1 ], [ 0, -1, 0, 1, 0 ], 
  [ 0, 0, -1, 0, 1 ], [ 0, 0, 0, -1, 1 ], [ 0, 0, 0, 1, -1 ], [ 0, 0, 1, 0, -1 ], 
  [ 0, 1, 0, -1, 0 ], [ 1, -1, -1, 0, 1 ], [ 1, 0, -1, -1, 2 ], [ 1, 0, 0, 0, 1 ] ]
gap> ll:=PossibilityOfStablyPermutationFFromBase(G,mi);;
gap> Length(ll);
18
gap> l:=ll[Length(ll)];
[ 0, 0, 0, 0, 0, 0, 0, 0, 0, 0, 0, 0, 0, 0, 0, 0, 0, 0, 0, 0, 0, 0, 0, 0, 0, 
  0, 0, 0, 0, 0, 0, 0, 0, 0, 0, 0, 0, 0, 0, 0, 0, 0, 0, 0, 0, 0, 0, 0, 0, 0, 
  0, 1, 0, 0, 0, 1, 0, -1 ]
gap> StablyPermutationFCheckFromBase(G,mi,Nlist(l),Plist(l)); 
[ [ 0, 0, 0, -1, 0, 0, -1 ], 
  [ 1, 1, 1, 1, 1, 1, 3 ], 
  [ -1, 0, 0, 0, 0, 0, -1 ], 
  [ 0, -1, 0, 0, 0, 0, -1 ], 
  [ 0, 0, -1, 0, 0, -1, 0 ], 
  [ 0, 0, 0, 0, -1, -1, 0 ], 
  [ 0, 0, 0, 0, -1, 0, -1 ] ]

gap> G:=CaratMatGroupZClass(5,801,3);; # #G=144
gap> Rank(FlabbyResolution(G).actionF.1); # F is of rank 25
25
gap> mis:=SearchCoflabbyResolutionBase(TransposedMatrixGroup(G),3);; # Method III
gap> List(mis,Length);                                              
[ 48, 48, 32, 32, 66, 66, 50, 50, 48, 48, 32, 32, 32, 48, 42, 42, 30, 32, 48, 42, 
  30, 32, 26, 26, 14, 32, 26, 14 ]
gap> mi:=mis[Length(mis)]; # (new) F is of rank 9 (=14-5)
[ [ -1, -1, 0, 1, 0 ], [ -1, -1, 1, 1, -1 ], [ -1, 0, 0, 0, 0 ], [ 0, -1, 0, 0, 0 ], 
  [ 0, 0, -1, 0, 0 ], [ 0, 0, 0, -1, 0 ], [ 0, 0, 0, 0, -1 ], [ 0, 0, 0, 0, 1 ], 
  [ 0, 0, 0, 1, 0 ], [ 0, 0, 1, 0, 0 ], [ 0, 1, 0, 0, 0 ], [ 1, 0, 0, 0, 0 ], 
  [ 1, 1, -1, -1, 1 ], [ 1, 1, 0, -1, 0 ] ]
gap> ll:=PossibilityOfStablyPermutationFFromBase(G,mi);;
gap> Length(ll);
32              
gap> l:=ll[Length(ll)];
[ 0, 0, 0, 0, 0, 0, 0, 0, 0, 0, 0, 0, 0, 0, 0, 0, 0, 0, 0, 0, 0, 0, 0, 0, 0, 
  0, 0, 0, 0, 0, 0, 0, 0, 0, 0, 0, 0, 0, 0, 0, 0, 0, 0, 0, 0, 0, 0, 0, 0, 0, 
  0, 0, 0, 0, 0, 0, 0, 0, 0, 0, 0, 0, 0, 0, 0, 0, 1, 1, 0, 0, 0, 0, 0, 0, 0, 
  0, 0, 0, 0, -1, 0, 0, 0, 0, 0, 1, -1 ]
gap> bp:=StablyPermutationFCheckPFromBase(G,mi,Nlist(l),Plist(l));; 
gap> Length(bp);
16
gap> Length(bp[1]); # rank of the both sides of (10) is 11
11
gap> rr:=Filtered(Tuples([0,1],16),x->DeterminantMat(x*bp)^2=1);;
gap> Length(rr);
104
gap> rr[1];
[ 0, 1, 0, 1, 0, 0, 0, 1, 1, 0, 1, 0, 1, 0, 0, 0 ]
gap> p:=rr[1]*bp;
[ [ 0, 0, 1, 1, 0, 1, 0, 1, 1, 0, 0 ], 
  [ 1, 1, 0, 0, 1, 0, 1, 0, 0, 1, 0 ], 
  [ 0, 0, 0, 0, 0, 0, 0, 0, 1, 1, 1 ], 
  [ 0, 1, 0, 0, 1, 0, 1, 0, 0, 0, 0 ], 
  [ 0, 0, 0, 1, 0, 1, 0, 1, 0, 0, 0 ], 
  [ 0, 0, 1, 0, 0, 1, 0, 1, 0, 0, 0 ], 
  [ 1, 0, 0, 0, 1, 0, 1, 0, 0, 0, 0 ], 
  [ 0, 0, 1, 1, 0, 0, 0, 0, 1, 0, 0 ], 
  [ 1, 1, 0, 0, 0, 0, 0, 0, 0, 1, 0 ], 
  [ 1, 1, 0, 0, 0, 0, 1, 0, 0, 0, 0 ], 
  [ 0, 0, 1, 1, 0, 0, 0, 1, 0, 0, 0 ] ]
gap> StablyPermutationFCheckMatFromBase(G,mi,Nlist(l),Plist(l),p);
true

gap> G:=CaratMatGroupZClass(5,337,12);; # G=D4xS3
gap> Rank(FlabbyResolution(G).actionF.1); # F is of rank 17
17
gap> mis:=SearchCoflabbyResolutionBase(TransposedMatrixGroup(G),3);; # Method III
gap> List(mis,Length);
[ 34, 34, 22, 22, 22, 22, 22, 22, 22, 22, 22, 14, 14, 22, 22, 22, 22, 22, 22, 14, 
  14, 22, 22, 22, 22, 22, 22, 16, 16, 16, 16, 16, 16, 12, 10, 12, 10 ]
gap> mi:=mis[Length(mis)]; # (new) F is of rank 5 (=10-5)
[ [ -1, 1, 1, 1, -1 ], [ -1, 1, 1, 2, -1 ], [ -1, 1, 2, 1, -1 ], [ 0, -1, -1, -1, 1 ], 
  [ 0, 0, -1, -1, 2 ], [ 0, 0, 1, 1, -2 ], [ 0, 0, 1, 1, -1 ], [ 1, -1, -2, -1, 1 ], 
  [ 1, -1, -1, -2, 1 ], [ 1, 0, -1, -1, 1 ] ]
gap> ll:=PossibilityOfStablyPermutationFFromBase(G,mi);             
[ [ 1, -1, -1, -1, -1, 0, 1, -1, 1, -2, 2 ], [ 0, 0, 0, 0, 0, 1, 0, 1, 0, 0, -1 ] ]
gap> l:=ll[Length(ll)];
[ 0, 0, 0, 0, 0, 1, 0, 1, 0, 0, -1 ]
gap> StablyPermutationFCheckFromBase(G,mi,Nlist(l),Plist(l)); 
[ [ 1, 1, 1, 0, 2 ], 
  [ -1, 0, 0, 0, -1 ], 
  [ 0, -1, 0, 0, -1 ], 
  [ 0, 0, 0, 1, -1 ], 
  [ 0, 0, -1, -1, 0 ] ]

gap> G:=CaratMatGroupZClass(5,341,6);; # G=D4xS3
gap> Rank(FlabbyResolution(G).actionF.1); # F is of rank 14
14
gap> mis:=SearchCoflabbyResolutionBase(TransposedMatrixGroup(G),3);; # Method III
gap> List(mis,Length);
[ 19, 19, 19, 19, 11, 19, 19, 11, 19, 19, 19, 19, 10, 10, 10, 11, 7, 8, 8 ]
gap> mi:=mis[Length(mis)-2]; # (new) F is of rank 2 (=7-5)
[ [ -1, 0, 1, 1, -1 ], [ -1, 1, 1, 1, -1 ], [ -1, 1, 1, 1, 0 ], [ 0, 0, 0, 1, -1 ], 
  [ 0, 0, 1, 0, -1 ], [ 0, 0, 1, 1, -1 ], [ 0, 1, 1, 1, -1 ] ]
gap> FlabbyResolutionFromBase(G,mi).actionF; # (new) F is permutation
Group([ [ [ 1, 0 ], [ 0, 1 ] ], [ [ 1, 0 ], [ 0, 1 ] ], [ [ 0, 1 ], [ 1, 0 ] ] ])

gap> G:=CaratMatGroupZClass(5,531,13);; # G=C2xS4
gap> Rank(FlabbyResolution(G).actionF.1); # F is of rank 15
15
gap> mis:=SearchCoflabbyResolutionBase(TransposedMatrixGroup(G),3);; # Method III
gap> List(mis,Length);
[ 18, 18, 14, 14, 14, 14, 12, 12, 10, 12, 6, 7, 12, 6, 7 ]
gap> mi:=mis[Length(mis)-1]; # (new) F is trivial of rank 1 (=6-5)
[ [ 0, 0, 0, 1, 0 ], [ 0, 0, 1, 1, 1 ], [ 0, 1, 0, 1, 1 ], 
  [ 1, -1, -1, 0, -1 ], [ 1, 0, 0, 0, -1 ], [ 1, 0, 0, 0, 0 ] ]
gap> FlabbyResolutionFromBase(G,mi).actionF; # (new) F is trivial of rank 1
Group([ [ [ 1 ] ], [ [ 1 ] ] ])

gap> G:=CaratMatGroupZClass(5,533,8);; # G=C2xS4
gap> Rank(FlabbyResolution(G).actionF.1); # F is of rank 44
44
gap> mis:=SearchCoflabbyResolutionBase(TransposedMatrixGroup(G),3);; # Method III
gap> List(mis,Length);
[ 29, 29, 15, 15, 15, 15 ]
gap> mi:=mis[Length(mis)]; # (new) F is of rank 10 (=15-5)
[ [ -1, 0, -1, 0, 1 ], [ -1, 0, -1, 1, 0 ], [ 0, 0, -1, 0, 1 ], [ 0, 0, -1, 1, 0 ], 
  [ 0, 0, 0, 0, 1 ], [ 0, 0, 0, 1, 0 ], [ 0, 1, -1, 0, 1 ], [ 0, 1, -1, 0, 2 ], 
  [ 0, 1, -1, 1, 0 ], [ 0, 1, -1, 1, 1 ], [ 0, 1, -1, 2, 0 ], [ 0, 1, 0, 0, 1 ], 
  [ 0, 1, 0, 1, 0 ], [ 1, 1, 0, 0, 1 ], [ 1, 1, 0, 1, 0 ] ]
gap> ll:=PossibilityOfStablyPermutationFFromBase(G,mi);;
gap> Length(ll);
5
gap> l:=ll[Length(ll)];
[ 0, 0, 0, 0, 0, 0, 0, 0, 0, 0, 0, 0, 0, 0, 0, 0, 0, 0, 0, 1, 0, 1, 0, 0, 0, 0, 0, 
  0, -1, 0, 0, 0, 1, -1 ]
gap> bp:=StablyPermutationFCheckPFromBase(G,mi,Nlist(l),Plist(l));;
gap> Length(bp);
20
gap> Length(bp[1]); # rank of the both sides of (10) is 13
13
gap> rr:=Filtered(Tuples([0,1],20),x->DeterminantMat(x*bp)^2=1);;
gap> Length(rr);
2448
gap> rr[1];
[ 0, 0, 1, 0, 0, 1, 0, 0, 0, 0, 1, 0, 0, 0, 0, 1, 0, 1, 0, 1 ]
gap> p:=rr[1]*bp;
[ [ 0, 0, 1, 1, 0, 0, 0, 0, 0, 0, 0, 0, 1 ], 
  [ 1, 0, 0, 0, 0, 1, 0, 0, 0, 0, 0, 0, 1 ], 
  [ 0, 1, 0, 0, 1, 0, 0, 0, 0, 0, 0, 0, 1 ], 
  [ 0, 0, 0, 0, 1, 0, 0, 0, 0, 0, 0, 1, 0 ], 
  [ 0, 0, 1, 0, 0, 0, 0, 0, 0, 0, 1, 0, 0 ], 
  [ 0, 0, 0, 1, 0, 0, 0, 1, 0, 0, 0, 0, 0 ], 
  [ 1, 0, 0, 0, 0, 0, 1, 0, 0, 0, 0, 0, 0 ], 
  [ 0, 0, 0, 0, 0, 1, 0, 0, 0, 1, 0, 0, 0 ], 
  [ 0, 1, 0, 0, 0, 0, 0, 0, 1, 0, 0, 0, 0 ], 
  [ 0, 0, 0, 0, 0, 1, 0, 0, 1, 0, 0, 0, 0 ], 
  [ 1, 1, 1, 0, 0, 0, 0, 0, 0, 0, 0, 0, 1 ], 
  [ -2, -2, -2, -2, -2, -2, -1, -1, -1, -1, -1, -1, -2 ], 
  [ 0, 0, 0, 1, 0, 0, 1, 0, 0, 0, 0, 0, 0 ] ]
gap> StablyPermutationFCheckMatFromBase(G,mi,Nlist(l),Plist(l),p);
true

gap> G:=CaratMatGroupZClass(5,623,4);; # G=C2xS4                                 
gap> Rank(FlabbyResolution(G).actionF.1); # F is of rank 13
13
gap> mis:=SearchCoflabbyResolutionBase(TransposedMatrixGroup(G),3);; # Method III
gap> List(mis,Length);
[ 42, 30, 24, 33, 21, 15, 36, 38, 42, 33, 42, 36, 30, 24, 18, 26, 20, 30, 24, 21, 
  15, 24, 18, 12 ]
gap> mi:=mis[Length(mis)]; # (new) F is of rank 7 (=12-5)
[ [ -1, -1, 1, 0, 0 ], [ -1, 0, 1, 0, 0 ], [ -1, 1, -1, 1, -1 ], [ -1, 1, 0, 1, -1 ],
  [ 0, -1, 1, -1, 1 ], [ 0, 0, 1, -1, 1 ], [ 0, 1, -1, 1, -1 ], [ 0, 1, 0, 1, -1 ], 
  [ 1, -1, 0, -1, 1 ], [ 1, -1, 0, 0, 1 ], [ 1, 0, -1, -1, 0 ], [ 1, 0, -1, 0, 0 ] ]
gap> ll:=PossibilityOfStablyPermutationFFromBase(G,mi);
[ [ 0, 2, 1, -1, -1 ] ]
gap> l:=ll[1];
[ 0, 2, 1, -1, -1 ]
gap> bp:=StablyPermutationFCheckPFromBase(G,mi,Nlist(l),Plist(l));;
gap> Length(bp);
14
gap> Length(bp[1]); # rank of the both sides of (10) is 8
8
gap> rr:=Filtered(Tuples([0,1],14),x->DeterminantMat(x*bp)^2=1);;
gap> Length(rr);
224
gap> rr[1];
[ 0, 1, 1, 1, 0, 0, 0, 0, 0, 0, 0, 1, 0, 1 ]
gap> p:=rr[1]*bp;
[ [ 0, 0, 0, 1, 1, 1, 1, 1 ], 
  [ 1, 0, 0, 0, 0, 0, 0, 0 ], 
  [ 0, 0, 1, 0, 1, 0, 0, 1 ],
  [ 0, 1, 0, 0, 0, 0, 0, 0 ], 
  [ 0, 0, 1, 1, 0, 0, 1, 0 ], 
  [ 0, 1, 0, 0, 0, 1, 1, 0 ],
  [ 1, 0, 0, 0, 0, 1, 0, 1 ], 
  [ 0, 1, 0, 1, 0, 0, 0, 1 ] ]
gap> StablyPermutationFCheckMatFromBase(G,mi,Nlist(l),Plist(l),p);
true

gap> G:=CaratMatGroupZClass(5,245,12);; # G=C2^2xS3
gap> Rank(FlabbyResolution(G).actionF.1); # F is of rank 16
16
gap> mis:=SearchCoflabbyResolutionBase(TransposedMatrixGroup(G),3);; # Method III
gap> List(mis,Length);
[ 11 ]
gap> mi:=mis[1]; # (new) F is of rank 6 (=11-5)
[ [ -1, -1, 1, -2, 0 ], [ -1, 0, 1, -2, -1 ], [ 0, 0, -1, 1, 1 ], [ 0, 0, 0, -1, -1 ], 
  [ 0, 0, 0, 1, 1 ], [ 0, 0, 1, -2, -1 ], [ 0, 0, 1, -1, -1 ], [ 0, 1, -1, 2, 1 ], 
  [ 0, 1, 0, 0, 0 ], [ 1, 1, -1, 2, 0 ], [ 1, 1, -1, 2, 1 ] ]
gap> ll:=PossibilityOfStablyPermutationFFromBase(G,mi);;
gap> Length(ll);                                                    
8
gap> l:=ll[Length(ll)];
[ 0, 0, 0, 0, 0, 0, 0, 0, 0, 0, 0, 0, 0, 0, 0, 0, 0, 0, 0, 0, 0, 0, 0, 1, 0, 
  1, 0, 0, 1, 0, 0, -1, -1 ]
gap> StablyPermutationFCheckFromBase(G,mi,Nlist(l),Plist(l));
[ [ -1, -1, -1, 0, 0, -1, -1 ], 
  [ 1, 0, 1, 0, 0, 0, 1 ], 
  [ 1, 1, 1, 0, -1, 1, 0 ], 
  [ 0, 1, 1, 0, 0, 1, 0 ], 
  [ 1, 1, 0, 0, 0, 0, 1 ], 
  [ 1, 1, 0, 0, 0, 1, 0 ], 
  [ 1, 1, 1, -1, 0, 1, 0 ] ]
\end{verbatim}
\end{example}

\bigskip

\begin{example}[{Verification of $[M_G]^{fl}=0$ for two groups 
$G\simeq S_5$ of the CARAT IDs $(5,911,4)$ and $(5,946,2)$}]\label{exq5}
{}~\\
\begin{verbatim}
gap> G:=CaratMatGroupZClass(5,911,4);; # G=S5
gap> Rank(FlabbyResolution(G).actionF.1); # F is of rank 11
11
gap> mis:=SearchCoflabbyResolutionBase(TransposedMatrixGroup(G),3);; # Method III
gap> List(mis,Length);
[ 36, 16, 16, 16, 16, 6 ]
gap> mi:=mis[Length(mis)]; # (new) F is of rank 1 (=6-5)
[ [ 0, 0, -1, 0, 1 ], [ 0, 0, 0, -1, 1 ], [ 0, 1, 0, -1, 0 ], 
  [ 1, -1, -1, 0, 1 ], [ 1, 0, -1, -1, 2 ], [ 1, 0, 0, 0, 1 ] ]
gap> FlabbyResolutionFromBase(G,mi).actionF; # (new) F is trivial
Group([ [ [ 1 ] ], [ [ 1 ] ] ])

gap> G:=CaratMatGroupZClass(5,946,2);; # G=S5
gap> CaratZClass(FlabbyResolution(G).actionF);
[ 5, 911, 4 ]
\end{verbatim}
\end{example}

\bigskip

\begin{example}[{Verification of $[M_G]^{fl}=0$ for the group 
$G\simeq S_5$ of the CARAT ID $(5,947,2)$}]\label{exq6}
{}~\\
\begin{verbatim}
gap> G:=CaratMatGroupZClass(5,947,2);; # G=S5
gap> Rank(FlabbyResolution(G).actionF.1); # F is of rank 45
45

gap> mis:=SearchCoflabbyResolutionBase(TransposedMatrixGroup(G),3);; # Method III
gap> List(mis,Length);
[ 51, 81, 51, 51, 41, 105, 75, 75, 65, 45, 96, 66, 66, 56, 36, 96, 96, 86, 66, 66, 
  56, 36, 56, 36, 26, 120, 120, 110, 90, 90, 80, 60, 80, 60, 50, 30 ]
gap> mi26:=mis[25];; # (new) F is of rank 21 (=26-5)
gap> ll:=PossibilityOfStablyPermutationFFromBase(G,mi26);;
gap> Length(ll);
5
gap> l:=ll[Length(ll)];
[ 0, 0, 0, 0, 0, 1, -1, 0, 1, -1, 1, 0, 0, 0, -1, 1, 1, 0, 0, -1 ]
gap> bp:=StablyPermutationFCheckPFromBase(G,mi26,Nlist(l),Plist(l));;
gap> Length(bp[1]); # but the rank of the both sides of (10) is 81
81

gap> mi:=mis[Length(mis)]; # (new) F is of rank 25 (=30-5)
[ [ -2, -1, -1, 0, 1 ], [ -2, 0, -1, 1, 1 ], [ -1, -2, 0, -1, 1 ], [ -1, -1, -1, -1, 1 ], 
  [ -1, -1, -1, 0, 0 ], [ -1, -1, 0, -1, 0 ], [ -1, 0, 0, 1, 1 ], [ -1, 1, -1, 1, 1 ], 
  [ -1, 1, -1, 2, 0 ], [ -1, 1, 0, 1, 0 ], [ 0, -2, 1, -1, 1 ], [ 0, -1, 1, 0, 1 ], 
  [ 0, 0, -1, -1, 1 ], [ 0, 0, -1, 1, -1 ], [ 0, 0, 1, -1, -1 ], [ 0, 1, -1, 0, 1 ], 
  [ 0, 1, -1, 2, -1 ], [ 0, 1, 1, 0, -1 ], [ 1, -1, 1, -1, 1 ], [ 1, -1, 1, 0, 0 ], 
  [ 1, -1, 2, -1, 0 ], [ 1, 0, 0, -1, 1 ], [ 1, 0, 0, 1, -1 ], [ 1, 0, 2, -1, -1 ], 
  [ 1, 1, -1, 0, 0 ], [ 1, 1, -1, 1, -1 ], [ 1, 1, 0, -1, 0 ], [ 1, 1, 0, 1, -2 ], 
  [ 1, 1, 1, -1, -1 ], [ 1, 1, 1, 0, -2 ] ]
gap> ll:=PossibilityOfStablyPermutationFFromBase(G,mi);
[ [ 1, 0, 0, -1, 0, 0, -4, 0, 1, -2, 2, 0, -1, -1, 0, 4, 4, 1, -4, 0 ],
  [ 0, 1, 0, 0, 0, 0, -1, 0, 0, -2, 1, 0, 0, 0, 0, 1, 2, 0, -1, -1 ],
  [ 0, 0, 1, 0, 0, 0, -2, 0, 0, -1, 1, 0, -1, -1, 0, 2, 2, 1, -2, 0 ],
  [ 0, 0, 0, 0, 1, 0, -2, 0, 1, 0, 0, -2, -1, -2, 0, 2, 2, 1, -2, 2 ],
  [ 0, 0, 0, 0, 0, 1, 0, 0, 0, -1, 1, 0, 0, 0, -1, 0, 1, 0, 0, -1 ] ]
gap> l:=ll[Length(ll)];
[ 0, 0, 0, 0, 0, 1, 0, 0, 0, -1, 1, 0, 0, 0, -1, 0, 1, 0, 0, -1 ]
gap> Length(l);
20
gap> [l[6],l[11],l[17],l[10],l[15],l[20]];
[ 1, 1, 1, -1, -1, -1 ]
gap> ss:=List(ConjugacyClassesSubgroups2(G),x->StructureDescription(Representative(x)));
[ "1", "C2", "C2", "C3", "C2 x C2", "C2 x C2", "C4", "C5", "C6", "S3", "S3", 
  "D8", "D10", "A4", "D12", "C5 : C4", "S4", "A5", "S5" ]
gap> Length(ss);
19
gap> [ss[6],ss[11],ss[17],ss[10],ss[15]];  
[ "C2 x C2", "S3", "S4", "S3", "D12" ]
gap> bp:=StablyPermutationFCheckPFromBase(G,mi,Nlist(l),Plist(l));;
gap> Length(bp); 
40
gap> Length(bp[1]); # rank of the both sides of (10) is 55
55

gap> l2:=IdentityMat(Length(l))[Length(l)-1];
[ 0, 0, 0, 0, 0, 0, 0, 0, 0, 0, 0, 0, 0, 0, 0, 0, 0, 0, 1, 0 ]
gap> bp:=StablyPermutationFCheckPFromBase(G,mi,Nlist(l)+l2,Plist(l)+l2);;
gap> Length(bp); 
47
gap> Length(bp[1]); # rank of the both sides of (10) is 56
56

# after some efforts we may get 
gap> n:=[ 
> 1, 0, -1, 0, -1, -1, -1, 1, 0, 0, -1, 1, -1, 0, -1, 1, 1, 1, 1, 0, 0, 1, 0, 1, 0, 
> -1, -1, -1, 0, -1, 1, 2, 0, -1, 0, 0, 1, 1, 1, -1, -1, 0, -1, 0, 1, 0, 1 ];
gap> p:=n*bp;;
gap> Determinant(p);
-1
gap> StablyPermutationFCheckMatFromBase(G,mi,Nlist(l)+l2,Plist(l)+l2,p));
true
\end{verbatim}
\end{example}

\bigskip

\begin{example}[{Verification of $[M_G]^{fl}=0$ for the group 
$G\simeq S_5$ of the CARAT ID $(5,946,4)$}]\label{exq7}
{}~\\
\begin{verbatim}
gap> G:=CaratMatGroupZClass(5,946,4);; # G=S5
gap> Rank(FlabbyResolution(G).actionF.1); # F is of rank 17
17

gap> mis:=SearchCoflabbyResolutionBase(TransposedMatrixGroup(G),5);;
gap> Set(List(mis,Length))-5; # Method III could not apply
[ 17, 32, 35, 47, 50, 62, 65, 77, 80, 92, 95 ]

gap> ll:=PossibilityOfStablyPermutationF(G);
[ [ 1, 0, 0, -1, 0, 0, -4, 0, 1, -2, 2, 0, -1, -1, 0, 4, 4, 1, -4, 0 ], 
  [ 0, 1, 0, 0, 0, 0, -2, 0, 1, -2, 1, 0, 0, -1, 0, 2, 3, 1, -2, -1 ], 
  [ 0, 0, 1, 0, 0, 0, -2, 0, 0, -1, 1, 0, -1, -1, 0, 2, 2, 1, -2, 0 ], 
  [ 0, 0, 0, 0, 1, 0, 0, 0, -1, 0, 0, -2, -1, 0, 0, 0, 0, -1, 0, 2 ], 
  [ 0, 0, 0, 0, 0, 1, -1, 0, 1, -1, 1, 0, 0, -1, -1, 1, 2, 1, -1, -1 ] ]
gap> l:=ll[Length(ll)];
[ 0, 0, 0, 0, 0, 1, -1, 0, 1, -1, 1, 0, 0, -1, -1, 1, 2, 1, -1, -1 ]
gap> Length(l);
20
gap> [l[6],l[9],l[11],l[16],l[17],l[18],l[7],l[10],l[14],l[15],l[19],l[20]];
[ 1, 1, 1, 1, 2, 1, -1, -1, -1, -1, -1, -1 ]
gap> ss:=List(ConjugacyClassesSubgroups2(G),x->StructureDescription(Representative(x)));
[ "1", "C2", "C2", "C3", "C2 x C2", "C2 x C2", "C4", "C5", "S3", "C6", "S3", 
  "D8", "D10", "A4", "D12", "C5 : C4", "S4", "A5", "S5" ]
gap> [ss[6],ss[9],ss[11],ss[16],ss[17],ss[18],ss[7],ss[10],ss[14],ss[15],ss[19]];
[ "C2 x C2", "S3", "S3", "C5 : C4", "S4", "A5", "C4", "C6", "A4", "D12", "S5" ]
gap> bp:=StablyPermutationFCheckP(G,Nlist(l),Plist(l));;
gap> Length(bp);
122
gap> Length(bp[1]); # rank of the both sides of (10) is 88
88

# after some efforts we may get 
gap> n:=[ 
> -1, 1, 1, 0, 0, 1, 0, 1, 0, 0, 1, 1, 1, -1, 1, -1, 0, -1, 0, 0, -1, 1, 0, -1, 1, 
> 1, 1, 0, 0, 0, -1, -1, 0, -1, 1, 0, 1, 1, -1, 0, 1, 1, 0, -1, 1, 0, 1, 1, -1, 0, 
> 1, 1, -1, 0, -1, 1, -1, -1, 0, 1, 1, 0, 1, -1, 1, -1, -1, 0, 1, -1, 0, 0, 0, -1, 1, 
> 0, -1, -1, -1, -1, 0, -1, -1, 1, 1, 1, 0, 2, -2, 4, 0, 1, 3, -1, -1, -1, -1, -1, 0, -1,
> -1, -1, -1, 1, 0, 1, -1, 0, -1, 1, 0, -1, -1, 1, -1, 0, 0, -1, 1, 1, -1, 1 ];;
gap> p:=n*bp;;
gap> Determinant(p);
-1
gap> StablyPermutationFCheckMat(G,Nlist(l),Plist(l),p);
true
\end{verbatim}
\end{example}

\bigskip

\section{Proof of Theorem \ref{th3}}\label{seProof3}

\begin{theorem}[{Yamasaki \cite[Lemma 4.3]{Yam12}}]\label{thY}
Let $k$ be a field of {\rm char} $k\neq 2$ and $k(x,y,z)$ be the rational 
function field over $k$ with variables $x,y,z$. 
Let $\sigma$ be a $k$-involution on $k(x,y,z)$ defined by 
\begin{align*}
\sigma : x\mapsto -x,\quad y\mapsto \frac{a}{y},\quad  z\mapsto \frac{-bx^2+c}{z}
\quad (a,b,c\in k^{\times}). 
\end{align*}
{\rm (i)} $k(x,y,z)^{\langle\sigma\rangle}=k(z_0,z_1,z_2,z_3)$ 
where 
\[
z_0^2=(z_1^2-a)(z_2^2-b)(z_3^2-c).
\]
{\rm (ii)} The fixed field $k(x,y,z)^{\langle\sigma\rangle}$ is $k$-rational 
if and only if $[k(\sqrt{a},\sqrt{b},\sqrt{c}) : k]\leq 2$ or 
$[k(\sqrt{a},\sqrt{b},\sqrt{c}) : k]=4$ with $abc\not\in k^{\times 2}$. 
In particular, if $k(x,y,z)^{\langle\sigma\rangle}$ is not $k$-rational, 
then $k(x,y,z)^{\langle\sigma\rangle}$ is not retract $k$-rational. 
\end{theorem}

\begin{example}[Another proof of Theorem \ref{thY}: 
not retract $k$-rational cases] 
Assume that $[k(\sqrt{a},\sqrt{b},\sqrt{c}):k]=4$ and $abc\in k^{\times 2}$. 
We will show that $k(x,y,z)^{\langle\sigma\rangle}$ is not 
retract $k$-rational by using {\tt IsInvertibleF} in Algorithm \ref{Alg2}. 
This also implies that $k(x,y,z)^{\langle\sigma\rangle}$ is not 
retract $k$-rational when $[k(\sqrt{a},\sqrt{b},\sqrt{c}):k]=8$.

We may assume that $c=ab$. 
Put $\alpha=\sqrt{a}$, $\beta=\sqrt{b}$ and $L=k(\alpha,\beta)$.
Then $L=k(\sqrt{a},\sqrt{b},\sqrt{c})$ and $[L:k]=4$. 
Put $y':=\frac{y-\alpha}{y+\alpha}$, 
$z':=\frac{z-\beta x-\alpha\beta}{z+\beta x+\alpha\beta}$. 
Then $L(x,y,z)=L(x,y',z')$ and $\sigma$ acts on $L(x,y',z')$ by 
$\sigma : x\mapsto -x$, $y'\mapsto -y'$, $z'\mapsto -z'$.
Put $y_1:=x^2$, $y_2:=xy'$, $y_3:=xz'$. 
Then $k(x,y,z)^{\langle\sigma\rangle}
=(L(x,y,z)^{\langle\sigma\rangle})^{\langle\rho_a,\rho_b\rangle}
=L(y_1,y_2,y_3)^{\langle\rho_a,\rho_b\rangle}$ is $L$-rational 
where 
\begin{align*}
\rho_a &: \alpha\mapsto -\alpha,\ y_1\mapsto y_1,\ y_2\mapsto 
\frac{y_1}{y_2},\ y_3\mapsto \frac{y_1(y_3+\alpha)}{y_1+\alpha y_3},\\
\rho_b &: \beta\mapsto -\beta,\ y_1\mapsto y_1,\ y_2\mapsto y_2,\ 
y_3\mapsto \frac{y_1}{y_3}. 
\end{align*}

Let $G={\rm Gal}(L/k)=\langle\rho_a,\rho_b\rangle\simeq C_2\times C_2$. 
We consider the $G$-lattice $M=\langle y_1,y_2,y_3,t_1,t_2,t_3\rangle$ 
of rank $6$ where $(t_1,t_2,t_3)=(y_1-a, y_1+\alpha y_3, y_3+\alpha)$. 
The action of $G$ on $L(M)$ is given by 
\begin{align*}
\rho_a &: \alpha\mapsto -\alpha,\  y_1\mapsto y_1,\ 
y_2\mapsto \frac{y_1}{y_2},\ y_3\mapsto \frac{y_1t_3}{t_2},\ 
t_1\mapsto t_1,\ t_2\mapsto \frac{y_1t_1}{t_2},\ 
t_3\mapsto \frac{y_3t_1}{t_2},\\ 
\rho_b &: \beta\mapsto -\beta,\ y_1\mapsto y_1,\ 
y_2\mapsto y_2,\ 
y_3\mapsto \frac{y_1}{y_3},\ 
t_1\mapsto t_1,\ 
t_2\mapsto \frac{y_1t_3}{y_3},\ 
t_3\mapsto \frac{t_2}{y_3}.
\end{align*} 
The actions of $\rho_a$ and $\rho_b$ on $M$ are represented as matrices
\vspace*{2mm}
\begin{center}
{\scriptsize
$\left(
\begin{array}{ccccccccc}\vspace*{-0.5mm}
\!\! 1 &\!\! 0 &\!\! 0 &\!\! 0 &\!\! 0 &\!\! 0 \!\! \\\vspace*{-0.5mm}
\!\! 1 &\!\! -1 &\!\! 0 &\!\! 0 &\!\! 0 &\!\! 0 \!\! \\\vspace*{-0.5mm}
\!\! 1 &\!\! 0 &\!\! 0 &\!\! 0 &\!\! -1 &\!\! 1 \!\! \\\vspace*{-0.5mm}
\!\! 0 &\!\! 0 &\!\! 0 &\!\! 1 &\!\! 0 &\!\! 0 \!\! \\\vspace*{-0.5mm}
\!\! 1 &\!\! 0 &\!\! 0 &\!\! 1 &\!\! -1 &\!\! 0 \!\! \\\vspace*{-0.5mm}
\!\! 0 &\!\! 0 &\!\! 1 &\!\! 1 &\!\! -1 &\!\! 0 \!\!
\end{array}
\right)$}, 
{\scriptsize 
$\left(
\begin{array}{ccccccccc}\vspace*{-0.5mm}
\!\! 1 &\!\! 0 &\!\! 0 &\!\! 0 &\!\! 0 &\!\! 0 \!\! \\\vspace*{-0.5mm}
\!\! 0 &\!\! 1 &\!\! 0 &\!\! 0 &\!\! 0 &\!\! 0 \!\! \\\vspace*{-0.5mm}
\!\! 1 &\!\! 0 &\!\! -1 &\!\! 0 &\!\! 0 &\!\! 0 \!\! \\\vspace*{-0.5mm}
\!\! 0 &\!\! 0 &\!\! 0 &\!\! 1 &\!\! 0 &\!\! 0 \!\! \\\vspace*{-0.5mm}
\!\! 1 &\!\! 0 &\!\! -1 &\!\! 0 &\!\! 0 &\!\! 1 \!\! \\\vspace*{-0.5mm}
\!\! 0 &\!\! 0 &\!\! -1 &\!\! 0 &\!\! 1 &\!\! 0 \!\!
\end{array}
\right)$}. \vspace*{2mm}
\end{center}

By Algorithm \ref{Alg2}, we see that $[M]^{fl}$ is not invertible. 
Hence $L(M)^G$ is not retract $k$-rational. 
It follows from \cite[Theorem 2.10]{Yam12} that 
$k(x,y,z)^{\langle\sigma\rangle}$ is not retract $k$-rational 
(cf. \cite[Case 3 in the proof of Lemma 4.3]{Yam12}, in particular, we do not need 
to enlarge $M$ in order to vanish $H^1$).

\bigskip

\begin{verbatim}
gap> Read("FlabbyResolution.gap");

gap> v1:=[
> [1,0,0,0,0,0],
> [1,-1,0,0,0,0],
> [1,0,0,0,-1,1],
> [0,0,0,1,0,0],
> [1,0,0,1,-1,0],
> [0,0,1,1,-1,0]];;
gap> v2:=[
> [1,0,0,0,0,0],
> [0,1,0,0,0,0],
> [1,0,-1,0,0,0],
> [0,0,0,1,0,0],
> [1,0,-1,0,0,1],
> [0,0,-1,0,1,0]];;
gap> IsInvertibleF(Group(v1,v2)); 
false
\end{verbatim}
\end{example}

\bigskip


We generalize Theorem \ref{thY} as follows: 

\begin{theorem}\label{th3}
Let $k$ be a field of {\rm char} $k\neq 2$ and $k(x,y,z)$ be 
the rational function field over $k$ with variables $x,y,z$. 
Let $\sigma_{a,b,c,d}$ be a $k$-involution on $k(x,y,z)$ defined by 
\begin{align*}
\sigma_{a,b,c,d} : x\mapsto -x,\quad y\mapsto \frac{-ax^2+b}{y},\quad  z\mapsto 
\frac{-cx^2+d}{z}
\quad (a,b,c,d\in k^{\times})
\end{align*}
and $m=[k(\sqrt{a},\sqrt{b},\sqrt{c},\sqrt{d}):k]$.\\ 
{\rm (i)} $k(x,y,z)^{\langle\sigma_{a,b,c,d}\rangle}=k(t_1,t_2,t_3,t_4)$ 
where $t_1,t_2,t_3,t_4$ satisfy the relation 
\begin{align}
(t_1^2-a)(t_4^2-d)=(t_2^2-b)(t_3^2-c).\label{eq1}
\end{align}
{\rm (ii)} $k(x,y,z)^{\langle\sigma_{a,b,c,d}\rangle}$ 
is $k$-isomorphic to 
$k(x,y,z)^{\langle\sigma_{\tau(a),\tau(b),\tau(c),\tau(d)}\rangle}$ 
for $\tau\in D_4$ where $D_4=\langle (abdc),(ab)(cd)\rangle$ 
is the permutation group on the set $\{a,b,c,d\}$ which is isomorphic to 
the dihedral group of order $8$.\\
{\rm (iii)} If one of the following conditions holds, then 
$k(x,y,z)^{\langle\sigma_{a,b,c,d}\rangle}$ 
is not retract $k$-rational:\\
{\rm (C15)} $m=4$, 
{\rm (1)} $ab, acd\in k^{\times 2}$;
{\rm (2)} $bd, abc\in k^{\times 2}$;
{\rm (3)} $cd, abd\in k^{\times 2}$;
{\rm (4)} $ac, bcd\in k^{\times 2}$;\\
{\rm (C16)} $m=4$, 
{\rm (1)} $ad, abc\in k^{\times 2}$;
{\rm (2)} $bc, abd\in k^{\times 2}$;\\
{\rm (C18)} $m=8$, 
{\rm (1)} $ab\in k^{\times 2}$;
{\rm (2)} $ac\in k^{\times 2}$;
{\rm (3)} $bd\in k^{\times 2}$;
{\rm (4)} $cd\in k^{\times 2}$;\\
{\rm (C19)} $m=8$, 
{\rm (1)}  $ad\in k^{\times 2}$;
{\rm (2)} $bc\in k^{\times 2}$;\\
{\rm (C20)} $m=8$, 
{\rm (1)} $abc\in k^{\times 2}$;
{\rm (2)} $bcd\in k^{\times 2}$;
{\rm (3)} $abd\in k^{\times 2}$;
{\rm (4)} $acd\in k^{\times 2}$;\\
{\rm (C21)} $m=8$, $abcd\in k^{\times 2}$;\\
{\rm (C22)} $m=16$.
\end{theorem}

\bigskip

\begin{proof}[Proof of Theorem \ref{th3}]
We prove the assertion (i). 
Put 
\begin{align*}
&t_1:=\frac{1}{2x}\left(y-\frac{-ax^2+b}{y}\right),\ 
t_2:=\frac{1}{2}\left(y+\frac{-ax^2+b}{y}\right),\\ 
&t_3:=\frac{1}{2x}\left(z-\frac{-cx^2+d}{z}\right),\ 
t_4:=\frac{1}{2}\left(z+\frac{-cx^2+d}{z}\right).
\end{align*}
Then we see that $k(t_1,t_2,t_3,t_4)\subset k(x,y,z)^{\langle\sigma_{a,b,c,d}\rangle}$. 
It follows from the equalities 
\[
y=t_2+t_1x,\ z=t_4+t_3x,\ x^2(t_3^2-c)-(t_4^2-d)=0
\]
that $[k(x,y,z):k(t_1,t_2,t_3,t_4)]\leq 2$. 
Hence we get $k(x,y,z)^{\langle\sigma_{a,b,c,d}\rangle}=k(t_1,t_2,t_3,t_4)$. 
The relation $(t_1^2-a)(t_4^2-d)=(t_2^2-b)(t_3^2-c)$ may be obtained 
by the direct calculation. 
The assertion (ii) follows from (i). 
We will prove the assertion (iii). \\

The case {\rm (C15)}: $m=4$. By (ii), we should show only the case 
(1) $ab,acd\in k^{\times 2}$. 
Define $Y:=\frac{ay}{ax+\sqrt{ab}}$. Then 
$k(x,y,z)=k(x,Y,z)$ and $\sigma_{a,b,c,d}$ acts on $k(x,Y,z)$ by 
\begin{align*}
\sigma_{a,b,c,d} : x\mapsto -x,\quad Y\mapsto \frac{a}{Y},\quad  
z\mapsto \frac{-cx^2+d}{z}.
\end{align*}
By Theorem \ref{thY} (ii), $k(x,y,z)^{\langle\sigma_{a,b,c,d}\rangle}$ is 
not retract $k$-rational.\\

The case {\rm (C16)}: $m=4$. By (ii), we should treat only the case 
(1) $ad, abc\in k^{\times 2}$. 
Let $L=k(\alpha,\beta,\gamma,\delta)$ 
where $\alpha^2=a$, $\beta^2=b$, $\gamma^2=c$, $\delta^2=d$. 
Then $L=k(\alpha,\beta)$ and $[L:k]=4$. 
Put $y':=(\alpha x+\beta)/y$, $z':=(\gamma x+\delta)/z$. 
Then $L(x,y,z)=L(x,y',z')$ and $\sigma_{a,b,c,d}$ acts on $L(x,y',z')$ by 
\[
\sigma_{a,b,c,d} : x\mapsto x,\ y'\mapsto \frac{1}{y'},\ z'\mapsto \frac{1}{z'}. 
\]
We put 
\[
y_1:=x^2,\ 
y_2:=x\ \frac{1-y'}{1+y'}=x\ \frac{y-\alpha x-\beta}{y+\alpha x+\beta},\ 
y_3:=x\ \frac{1-z'}{1+z'}=x\ \frac{z-\gamma x-\delta}{y+\gamma x+\delta}.
\]
By the assumptions $ad, abc\in k^{\times 2}$, 
there exist $e, f\in k^\times$ such that $\gamma=\alpha\beta e$ 
and $\delta=\alpha f$. 
Then $k(x,y,z)^{\langle\sigma_{a,b,c,d}\rangle}
=(L(x,y,z)^{\langle\sigma_{a,b,c,d}\rangle})^{\langle\rho_a,\rho_b\rangle}
=L(y_1,y_2,y_3)^{\langle\rho_a,\rho_b\rangle}$, where 
\begin{align*}
\rho_a &: \alpha\mapsto -\alpha,\ y_1\mapsto y_1,\ y_2\mapsto 
\frac{\alpha y_1+\beta y_2}{\alpha y_2+\beta},\ y_3\mapsto \frac{y_1}{y_3},\\
\rho_b &: \beta\mapsto -\beta,\ y_1\mapsto y_1,\ y_2\mapsto \frac{y_1(\alpha y_2+\beta)}{\alpha y_1+\beta y_2},\ y_3\mapsto 
\frac{\beta e y_1+f y_3}{\beta e y_3+f}. 
\end{align*}
Let $G={\rm Gal}(L/k)=\langle\rho_a,\rho_b\rangle\simeq C_2\times C_2$. 
We consider the $G$-lattice $M=\langle y_1,y_2,y_3,t_1,t_2,t_3,t_4,u_1,u_2\rangle$ 
of rank $9$ where 
\[
(t_1,t_2,t_3,t_4,u_1,u_2)
=(\alpha y_1+\beta y_2, \alpha y_2+\beta, 
\beta e y_1+f y_3,\beta e y_3+f, ay_1-b,-be^2y_1+f^2).
\]
The action of $G$ on $L(M)$ is given by 
\begin{align*}
\rho_a &: \alpha\mapsto -\alpha,\  
y_2\mapsto \frac{t_1}{t_2},\ y_3\mapsto \frac{y_1}{y_3},\ 
t_1\mapsto -\frac{u_1y_2}{t_2},\ t_2\mapsto -\frac{u_1}{t_2},\ 
t_3\mapsto \frac{y_1t_4}{y_3},\ t_4\mapsto \frac{t_3}{y_3},\\ 
\rho_b &: \beta\mapsto -\beta,\ 
y_2\mapsto \frac{y_1t_2}{t_1},\ 
y_3\mapsto \frac{t_3}{t_4},\ 
t_1\mapsto \frac{y_1u_1}{t_1},\ 
t_2\mapsto \frac{u_1y_2}{t_1},\ 
t_3\mapsto \frac{u_2y_3}{t_4},\ 
t_4\mapsto \frac{u_2}{t_4} 
\end{align*}
where $y_1,u_1,u_2$ are invariants under the action of $G$. 
The actions of $\rho_a$ and $\rho_b$ on $M$ are represented 
as matrices

\begin{center}
{\scriptsize
$\left(
\begin{array}{ccccccccc}\vspace*{-0.5mm}
\!\! 1 &\!\! 0 &\!\! 0 &\!\! 0 &\!\! 0 &\!\! 0 &\!\! 0 &\!\! 0 &\!\! 0 \!\! \\\vspace*{-0.5mm}
\!\! 0 &\!\! 0 &\!\! 0 &\!\! 1 &\!\! -1 &\!\! 0 &\!\! 0 &\!\! 0 &\!\! 0 \!\! \\\vspace*{-0.5mm}
\!\! 1 &\!\! 0 &\!\! -1 &\!\! 0 &\!\! 0 &\!\! 0 &\!\! 0 &\!\! 0 &\!\! 0 \!\! \\\vspace*{-0.5mm}
\!\! 0 &\!\! 1 &\!\! 0 &\!\! 0 &\!\! -1 &\!\! 0 &\!\! 0 &\!\! 1 &\!\! 0 \!\! \\\vspace*{-0.5mm}
\!\! 0 &\!\! 0 &\!\! 0 &\!\! 0 &\!\! -1 &\!\! 0 &\!\! 0 &\!\! 1 &\!\! 0 \!\! \\\vspace*{-0.5mm}
\!\! 1 &\!\! 0 &\!\! -1 &\!\! 0 &\!\! 0 &\!\! 0 &\!\! 1 &\!\! 0 &\!\! 0 \!\! \\\vspace*{-0.5mm}
\!\! 0 &\!\! 0 &\!\! -1 &\!\! 0 &\!\! 0 &\!\! 1 &\!\! 0 &\!\! 0 &\!\! 0 \!\! \\\vspace*{-0.5mm}
\!\! 0 &\!\! 0 &\!\! 0 &\!\! 0 &\!\! 0 &\!\! 0 &\!\! 0 &\!\! 1 &\!\! 0 \!\! \\\vspace*{-0.5mm}
\!\! 0 &\!\! 0 &\!\! 0 &\!\! 0 &\!\! 0 &\!\! 0 &\!\! 0 &\!\! 0 &\!\! 1\!\!
\end{array}
\right)$}, 
{\scriptsize 
$\left(
\begin{array}{ccccccccc}\vspace*{-0.5mm}
\!\! 1 &\!\! 0 &\!\! 0 &\!\! 0 &\!\! 0 &\!\! 0 &\!\! 0 &\!\! 0 &\!\! 0 \!\! \\\vspace*{-0.5mm}
\!\! 1 &\!\! 0 &\!\! 0 &\!\! -1 &\!\! 1 &\!\! 0 &\!\! 0 &\!\! 0 &\!\! 0 \!\! \\\vspace*{-0.5mm}
\!\! 0 &\!\! 0 &\!\! 0 &\!\! 0 &\!\! 0 &\!\! 1 &\!\! -1 &\!\! 0 &\!\! 0 \!\! \\\vspace*{-0.5mm}
\!\! 1 &\!\! 0 &\!\! 0 &\!\! -1 &\!\! 0 &\!\! 0 &\!\! 0 &\!\! 1 &\!\! 0 \!\! \\\vspace*{-0.5mm}
\!\! 0 &\!\! 1 &\!\! 0 &\!\! -1 &\!\! 0 &\!\! 0 &\!\! 0 &\!\! 1 &\!\! 0 \!\! \\\vspace*{-0.5mm}
\!\! 0 &\!\! 0 &\!\! 1 &\!\! 0 &\!\! 0 &\!\! 0 &\!\! -1 &\!\! 0 &\!\! 1 \!\! \\\vspace*{-0.5mm}
\!\! 0 &\!\! 0 &\!\! 0 &\!\! 0 &\!\! 0 &\!\! 0 &\!\! -1 &\!\! 0 &\!\! 1 \!\! \\\vspace*{-0.5mm}
\!\! 0 &\!\! 0 &\!\! 0 &\!\! 0 &\!\! 0 &\!\! 0 &\!\! 0 &\!\! 1 &\!\! 0 \!\! \\\vspace*{-0.5mm}
\!\! 0 &\!\! 0 &\!\! 0 &\!\! 0 &\!\! 0 &\!\! 0 &\!\! 0 &\!\! 0 &\!\! 1\!\!
\end{array}
\right)$}. 
\end{center}

By Algorithm \ref{Alg2}, we see that $[M]^{fl}$ is not invertible 
(see Example \ref{exf} below). 
Hence $L(M)^G$ is not retract $k$-rational. 
It follows from \cite[Theorem 2.10]{Yam12} that 
$k(x,y,z)^{\langle\sigma_{a,b,c,d}\rangle}$ is not retract $k$-rational. \\

The cases (C18), (C19) and (C20). 
By the result of (C15) and (C16), $k(x,y,z)^{\langle\sigma_{a,b,c,d}\rangle}$ is 
not retract rational over some quadratic extension of $k$, hence 
not retract $k$-rational. \\

The case (C21): $m=8$, $abcd\in k^{\times 2}$. 
Let $L=k(\sqrt{a},\sqrt{b},\sqrt{c},\sqrt{d})$. 
We assume $[L:k]=8$ and $abcd\in k^{\times 2}$, 
and hence $d=abce^2$ for some $e\in k^\times$. 
We put $\alpha:=\sqrt{a}$, $\beta:=\sqrt{b}$, $\gamma:=\sqrt{c}$. 
Then $L=k(\alpha,\beta,\gamma)$. 
We first see that $L(x,y,z)^{\langle\sigma_{a,b,c,d}\rangle}$ is $L$-rational as follows. 
Put 
\[
y_1:=x^2,\ y_2:=x\,\frac{y-\alpha x-\beta}{y+\alpha x+\beta},\ 
y_3:=x\,\frac{z-\gamma x-\alpha\beta\gamma e}{z+\gamma x+\alpha\beta\gamma e}.
\]
Then $L(x,y,z)=L(x,y_2,y_3)$ and the action of $\sigma_{a,b,c,d}$ on $L(x,y_2,y_3)$ is 
given by $\sigma_{a,b,c,d} : x\mapsto -x, y_2\mapsto y_2, y_3\mapsto y_3$. 
Hence the field $L(x,y,z)^{\langle\sigma_{a,b,c,d}\rangle}=L(y_1,y_2,y_3)$ is 
$L$-rational. 
Let $G=\langle \rho_a,\rho_b,\rho_c\rangle\simeq C_2\times C_2\times C_2$ 
be the Galois group ${\rm Gal}(L/k)$ of $L/k$ where 
\begin{align*}
\rho_a &: \alpha\mapsto -\alpha,\ \beta\mapsto \beta,\ \gamma\mapsto \gamma,\\ 
\rho_b &: \alpha\mapsto \alpha,\ \beta\mapsto -\beta,\ \gamma\mapsto \gamma,\\ 
\rho_c &: \alpha\mapsto \alpha,\ \beta\mapsto \beta,\ \gamma\mapsto -\gamma. 
\end{align*}
Then we get $k(x,y,z)^{\langle\sigma_{a,b,c,d}\rangle}=(L(x,y,z)^{\langle\sigma_{a,b,c,d}\rangle})^{G}
=L(y_1,y_2,y_3)^{G}$. 
The action of $G$ on $L(y_1,y_2,y_3)$ is given by 
\begin{align*}
\rho_a &: \alpha\mapsto -\alpha,\ y_1\mapsto y_1,\ 
y_2\mapsto \frac{\alpha y_1+\beta y_2}{\alpha y_2+\beta},\ 
y_3\mapsto \frac{y_1(y_3+\alpha\beta e)}{y_1+\alpha\beta e y_3},\\
\rho_b &: \beta\mapsto -\beta,\ y_1\mapsto y_1,\ 
y_2\mapsto \frac{y_1(\alpha y_2+\beta)}{\alpha y_1+\beta y_2},\ 
y_3\mapsto \frac{y_1(y_3+\alpha\beta e)}{y_1+\alpha\beta e y_3},\\
\rho_c &: \gamma\mapsto -\gamma,\ y_1\mapsto y_1,\ 
y_2\mapsto y_2,\ y_3\mapsto \frac{y_1}{y_3}.
\end{align*}
We consider the $G$-lattice $M=\langle y_1,y_2,y_3,t_1,t_2,t_3,t_4,u_1,u_2\rangle$ 
of rank $9$ where 
\[
(t_1,t_2,t_3,t_4,u_1,u_2)=(\alpha y_1+\beta y_2, \alpha y_2+\beta, y_3+\alpha\beta e, 
y_1+\alpha\beta e y_3, b-ay_1, y_1-abe^2).
\]
The action of $G$ on $L(M)$ is given by 
\begin{align*}
\rho_a &: \alpha\mapsto -\alpha,\ 
y_2\mapsto \frac{t_1}{t_2},\ y_3\mapsto \frac{y_1t_3}{t_4},\ 
t_1\mapsto \frac{u_1y_2}{t_2},\ t_2\mapsto \frac{u_1}{t_2},\ 
t_3\mapsto \frac{u_2y_3}{t_4},\ t_4\mapsto \frac{y_1u_2}{t_4},\\ 
\rho_b &: \beta\mapsto -\beta,\ 
y_2\mapsto \frac{y_1t_2}{t_1},\ y_3\mapsto \frac{y_1t_3}{t_4},\ 
t_1\mapsto -\frac{y_1u_2}{t_1},\ t_2\mapsto -\frac{u_1y_2}{t_1},\ 
t_3\mapsto \frac{u_2y_3}{t_4},\ t_4\mapsto \frac{y_1u_2}{t_4},\\
\rho_c &: \gamma\mapsto -\gamma,\ 
y_2\mapsto y_2,\ y_3\mapsto \frac{y_1}{y_3},\ t_1\mapsto t_1,\ 
t_2\mapsto t_2,\ t_3\mapsto \frac{t_4}{y_3},\ t_4\mapsto \frac{y_1t_3}{y_3}
\end{align*}
where $y_1,u_1,u_2$ are invariants under the action of $G$. 
The actions of $\rho_a$, $\rho_b$ and $\rho_c$ on $M$ 
are represented as matrices
\begin{center}
{\scriptsize
$\left(
\begin{array}{ccccccccc}\vspace*{-0.5mm}
\!\! 1 &\!\! 0 &\!\! 0 &\!\! 0 &\!\! 0 &\!\! 0 &\!\! 0 &\!\! 0 &\!\! 0 \!\! \\\vspace*{-0.5mm}
\!\! 0 &\!\! 0 &\!\! 0 &\!\! 1 &\!\! -1 &\!\! 0 &\!\! 0 &\!\! 0 &\!\! 0 \!\! \\\vspace*{-0.5mm}
\!\! 1 &\!\! 0 &\!\! 0 &\!\! 0 &\!\! 0 &\!\! 1 &\!\! -1 &\!\! 0 &\!\! 0 \!\! \\\vspace*{-0.5mm}
\!\! 0 &\!\! 1 &\!\! 0 &\!\! 0 &\!\! -1 &\!\! 0 &\!\! 0 &\!\! 1 &\!\! 0 \!\! \\\vspace*{-0.5mm}
\!\! 0 &\!\! 0 &\!\! 0 &\!\! 0 &\!\! -1 &\!\! 0 &\!\! 0 &\!\! 1 &\!\! 0 \!\! \\\vspace*{-0.5mm}
\!\! 0 &\!\! 0 &\!\! 1 &\!\! 0 &\!\! 0 &\!\! 0 &\!\! -1 &\!\! 0 &\!\! 1 \!\! \\\vspace*{-0.5mm}
\!\! 1 &\!\! 0 &\!\! 0 &\!\! 0 &\!\! 0 &\!\! 0 &\!\! -1 &\!\! 0 &\!\! 1 \!\! \\\vspace*{-0.5mm}
\!\! 0 &\!\! 0 &\!\! 0 &\!\! 0 &\!\! 0 &\!\! 0 &\!\! 0 &\!\! 1 &\!\! 0 \!\! \\\vspace*{-0.5mm}
\!\! 0 &\!\! 0 &\!\! 0 &\!\! 0 &\!\! 0 &\!\! 0 &\!\! 0 &\!\! 0 &\!\! 1\!\!
\end{array}
\right)$}, 
{\scriptsize 
$\left(
\begin{array}{ccccccccc}\vspace*{-0.5mm}
\!\! 1 &\!\! 0 &\!\! 0 &\!\! 0 &\!\! 0 &\!\! 0 &\!\! 0 &\!\! 0 &\!\! 0 \!\! \\\vspace*{-0.5mm}
\!\! 1 &\!\! 0 &\!\! 0 &\!\! -1 &\!\! 1 &\!\! 0 &\!\! 0 &\!\! 0 &\!\! 0 \!\! \\\vspace*{-0.5mm}
\!\! 1 &\!\! 0 &\!\! 0 &\!\! 0 &\!\! 0 &\!\! 1 &\!\! -1 &\!\! 0 &\!\! 0 \!\! \\\vspace*{-0.5mm}
\!\! 1 &\!\! 0 &\!\! 0 &\!\! -1 &\!\! 0 &\!\! 0 &\!\! 0 &\!\! 1 &\!\! 0 \!\! \\\vspace*{-0.5mm}
\!\! 0 &\!\! 1 &\!\! 0 &\!\! -1 &\!\! 0 &\!\! 0 &\!\! 0 &\!\! 1 &\!\! 0 \!\! \\\vspace*{-0.5mm}
\!\! 0 &\!\! 0 &\!\! 1 &\!\! 0 &\!\! 0 &\!\! 0 &\!\! -1 &\!\! 0 &\!\! 1 \!\! \\\vspace*{-0.5mm}
\!\! 1 &\!\! 0 &\!\! 0 &\!\! 0 &\!\! 0 &\!\! 0 &\!\! -1 &\!\! 0 &\!\! 1 \!\! \\\vspace*{-0.5mm}
\!\! 0 &\!\! 0 &\!\! 0 &\!\! 0 &\!\! 0 &\!\! 0 &\!\! 0 &\!\! 1 &\!\! 0 \!\! \\\vspace*{-0.5mm}
\!\! 0 &\!\! 0 &\!\! 0 &\!\! 0 &\!\! 0 &\!\! 0 &\!\! 0 &\!\! 0 &\!\! 1\!\!
\end{array}
\right)$}, 
{\scriptsize $\left(
\begin{array}{ccccccccc}\vspace*{-0.5mm}
\!\! 1 &\!\! 0 &\!\! 0 &\!\! 0 &\!\! 0 &\!\! 0 &\!\! 0 &\!\! 0 &\!\! 0\!\! \\\vspace*{-0.5mm}
\!\! 0 &\!\! 1 &\!\! 0 &\!\! 0 &\!\! 0 &\!\! 0 &\!\! 0 &\!\! 0 &\!\! 0\!\! \\\vspace*{-0.5mm}
\!\! 1 &\!\! 0 &\!\! -1 &\!\! 0 &\!\! 0 &\!\! 0 &\!\! 0 &\!\! 0 &\!\! 0\!\! \\\vspace*{-0.5mm}
\!\! 0 &\!\! 0 &\!\! 0 &\!\! 1 &\!\! 0 &\!\! 0 &\!\! 0 &\!\! 0 &\!\! 0\!\! \\\vspace*{-0.5mm}
\!\! 0 &\!\! 0 &\!\! 0 &\!\! 0 &\!\! 1 &\!\! 0 &\!\! 0 &\!\! 0 &\!\! 0\!\! \\\vspace*{-0.5mm}
\!\! 0 &\!\! 0 &\!\! -1 &\!\! 0 &\!\! 0 &\!\! 0 &\!\! 1 &\!\! 0 &\!\! 0\!\! \\\vspace*{-0.5mm}
\!\! 1 &\!\! 0 &\!\! -1 &\!\! 0 &\!\! 0 &\!\! 1 &\!\! 0 &\!\! 0 &\!\! 0\!\! \\\vspace*{-0.5mm}
\!\! 0 &\!\! 0 &\!\! 0 &\!\! 0 &\!\! 0 &\!\! 0 &\!\! 0 &\!\! 1 &\!\! 0\!\! \\\vspace*{-0.5mm}
\!\! 0 &\!\! 0 &\!\! 0 &\!\! 0 &\!\! 0 &\!\! 0 &\!\! 0 &\!\! 0 &\!\! 1\!\!
\end{array}
\right)$}.
\end{center}
By Algorithm \ref{Alg2}, we obtain that $[M]^{fl}$ is not invertible 
(see Example \ref{exf} below). 
Hence $L(M)^G$ is not retract $k$-rational. 
By \cite[Theorem 2.10]{Yam12}, 
$k(x,y,z)^{\langle\sigma_{a,b,c,d}\rangle}$ is not retract $k$-rational. \\

The case (C22): $m=16$. 
By the result of (C21), $k(x,y,z)^{\langle\sigma_{a,b,c,d}\rangle}$ is 
not retract rational over a quadratic extension of $k$, hence 
not retract $k$-rational. 
\end{proof}

\bigskip

\begin{example}[{$[M]^{fl}$ is not invertible for the cases  (C16) and (C21)}]\label{exf}
The following GAP computation confirms that 
$[M]^{fl}$ is not invertible hence 
$L(M)^{G}$ is not retract $k$-rational as 
in the cases (C16) and (C21) of 
the proof of Theorem \ref{th3} above. 

\bigskip

\begin{verbatim}
gap> Read("FlabbyResolution.gap");

gap> v1:=[
> [1,0,0,0,0,0,0,0,0],
> [0,0,0,1,-1,0,0,0,0],
> [1,0,-1,0,0,0,0,0,0],
> [0,1,0,0,-1,0,0,1,0],
> [0,0,0,0,-1,0,0,1,0],
> [1,0,-1,0,0,0,1,0,0],
> [0,0,-1,0,0,1,0,0,0],
> [0,0,0,0,0,0,0,1,0],
> [0,0,0,0,0,0,0,0,1]];;
gap> v2:=[
> [1,0,0,0,0,0,0,0,0],
> [1,0,0,-1,1,0,0,0,0],
> [0,0,0,0,0,1,-1,0,0],
> [1,0,0,-1,0,0,0,1,0],
> [0,1,0,-1,0,0,0,1,0],
> [0,0,1,0,0,0,-1,0,1],
> [0,0,0,0,0,0,-1,0,1],
> [0,0,0,0,0,0,0,1,0],
> [0,0,0,0,0,0,0,0,1]];;
gap> IsInvertibleF(Group(v1,v2));
false

gap> v1:=[
> [1,0,0,0,0,0,0,0,0],
> [0,0,0,1,-1,0,0,0,0],
> [1,0,0,0,0,1,-1,0,0],
> [0,1,0,0,-1,0,0,1,0],
> [0,0,0,0,-1,0,0,1,0],
> [0,0,1,0,0,0,-1,0,1],
> [1,0,0,0,0,0,-1,0,1],
> [0,0,0,0,0,0,0,1,0],
> [0,0,0,0,0,0,0,0,1]];;
gap> v2:=[
> [1,0,0,0,0,0,0,0,0],
> [1,0,0,-1,1,0,0,0,0],
> [1,0,0,0,0,1,-1,0,0],
> [1,0,0,-1,0,0,0,1,0],
> [0,1,0,-1,0,0,0,1,0],
> [0,0,1,0,0,0,-1,0,1],
> [1,0,0,0,0,0,-1,0,1],
> [0,0,0,0,0,0,0,1,0],
> [0,0,0,0,0,0,0,0,1]];;
gap> v3:=[
> [1,0,0,0,0,0,0,0,0],
> [0,1,0,0,0,0,0,0,0],
> [1,0,-1,0,0,0,0,0,0],
> [0,0,0,1,0,0,0,0,0],
> [0,0,0,0,1,0,0,0,0],
> [0,0,-1,0,0,0,1,0,0],
> [1,0,-1,0,0,1,0,0,0],
> [0,0,0,0,0,0,0,1,0],
> [0,0,0,0,0,0,0,0,1]];;
gap> IsInvertibleF(Group(v1,v2,v3));                                    
false
\end{verbatim}
\end{example}

\bigskip

\begin{lemma}\label{lem3}
Let $k$ be a field of {\rm char} $k\neq 2$ and $k(x,y,z)$ be 
the rational function field over $k$ with variables $x,y,z$. 
Let $\sigma_{a,b,c,d}$ be a $k$-involution on $k(x,y,z)$ defined by 
\begin{align*}
\sigma_{a,b,c,d} : x\mapsto -x,\quad y\mapsto \frac{-ax^2+b}{y},\quad  z\mapsto 
\frac{-cx^2+d}{z}
\quad (a,b,c,d\in k^{\times})
\end{align*}
and $m=[k(\sqrt{a},\sqrt{b},\sqrt{c},\sqrt{d}):k]$. 
If one of the following conditions holds, then 
$k(x,y,z)^{\langle\sigma_{a,b,c,d}\rangle}$ 
is $k$-rational:\\
{\rm (C1)} $m=1$;\\
{\rm (C2)} $m=2$, 
{\rm (1)} $a,b,c\in k^{\times 2}$;
{\rm (2)} $a,b,d\in k^{\times 2}$;
{\rm (3)} $a,c,d\in k^{\times 2}$;
{\rm (4)} $b,c,d\in k^{\times 2}$;\\
{\rm (C3)} $m=2$, 
{\rm (1)} $a,b,cd\in k^{\times 2}$;
{\rm (2)} $b,d,ac\in k^{\times 2}$;
{\rm (3)} $d,c,ab\in k^{\times 2}$;
{\rm (4)} $c,a,bd\in k^{\times 2}$;\\
{\rm (C5)} $m=2$, 
{\rm (1)} $a,bd,cd\in k^{\times 2}$;
{\rm (2)} $b,cd,ac\in k^{\times 2}$;
{\rm (3)} $d,ac,ab\in k^{\times 2}$;
{\rm (4)} $c,ab,bd\in k^{\times 2}$;\\
{\rm (C6)} $m=2$, $ab,ac,ad\in k^{\times 2}$;\\
{\rm (C7)} $m=4$, 
{\rm (1)} $a, b\in k^{\times 2}$;
{\rm (2)} $b, d\in k^{\times 2}$;
{\rm (3)} $d, c\in k^{\times 2}$;
{\rm (4)} $c, a\in k^{\times 2}$;\\
{\rm (C10)} $m=4$, 
{\rm (1)} $a, bd\in k^{\times 2}$;
{\rm (2)} $b, dc\in k^{\times 2}$;
{\rm (3)} $d, ac\in k^{\times 2}$;
{\rm (4)} $c, ab\in k^{\times 2}$;\\
\hspace*{20.6mm}
{\rm (5)} $a, cd\in k^{\times 2}$;
{\rm (6)} $b, ac\in k^{\times 2}$;
{\rm (7)} $d, ab\in k^{\times 2}$;
{\rm (8)} $c, bd\in k^{\times 2}$;\\
{\rm (C12)} $m=4$, 
{\rm (1)} $ab,cd\in k^{\times 2}$;
{\rm (2)} $bd,ac\in k^{\times 2}$;\\
{\rm (C13)} $m=4$, 
{\rm (1)}  $ab,ac\in k^{\times 2}$;
{\rm (2)} $bd,ab\in k^{\times 2}$;
{\rm (3)} $cd,bd\in k^{\times 2}$;
{\rm (4)} $ac,cd\in k^{\times 2}$.
\end{lemma}
\begin{proof}
The case (C7): $m=4$. 
By Theorem \ref{th3} (ii), 
we should show only the case (1) $a,b\in k^{\times 2}$.
Take $\alpha, \beta\in k$ with $\alpha^2=a$ and $\beta^2=b$. 
The equation (\ref{eq1}) becomes 
\[
\frac{(t_1+\alpha)(t_1-\alpha)}{(t_2+\beta)(t_2-\beta)}=\frac{t_3^2-c}{t_4^2-d}.
\]
Define $T_1:=(t_1+\alpha)/(t_2+\beta)$, $T_2:=(t_1-\alpha)/(t_2-\beta)$. 
Then $k(x,y,z)^{\langle\sigma_{a,b,c,d}\rangle}=k(t_1,t_2,t_3,t_4)
=k(T_1,T_2,t_3,t_4)=k(T_1,t_3,t_4)$ is $k$-rational. 

The cases (C10), (C12) and (C13): $m=4$. 
By Theorem \ref{th3} (ii), 
we may assume that $ab\in k^{\times 2}$. 
Define $Y:=\frac{ay}{ax+\sqrt{ab}}$. Then 
$k(x,y,z)=k(x,Y,z)$ and $\sigma_{a,b,c,d}$ acts on $k(x,Y,z)$ by 
\begin{align*}
\sigma_{a,b,c,d} : x\mapsto -x,\quad Y\mapsto \frac{a}{Y},\quad  
z\mapsto \frac{-cx^2+d}{z}.
\end{align*}
By Theorem \ref{thY} (ii), $k(x,y,z)^{\langle\sigma_{a,b,c,d}\rangle}$ 
is $k$-rational.

The cases (C1), (C2), (C3), (C5) and (C6).  
By the results of (C10), (C12) and (C13), 
$k(x,y,z)^{\langle\sigma_{a,b,c,d}\rangle}$ 
is $k$-rational.
\end{proof}
\begin{remark}
We do not know whether $k(x,y,z)^{\langle\sigma_{a,b,c,d}\rangle}$ 
is $k$-rational (resp. stably $k$-rational, retract $k$-rational) 
for the following cases:\\
{\rm (C4)} $m=2$, 
{\rm (1)} $a, d, bc\in k^{\times 2}$;
{\rm (2)} $b, c, ad\in k^{\times 2}$;\\
{\rm (C8)} $m=4$, 
{\rm (1)} $a,d\in k^{\times 2}$;
{\rm (2)} $b, c\in k^{\times 2}$;\\
{\rm (C9)} $m=4$, 
{\rm (1)} $a, bc\in k^{\times 2}$;
{\rm (2)} $b, ad\in k^{\times 2}$;
{\rm (3)} $d, bc\in k^{\times 2}$;
{\rm (4)} $c, ad\in k^{\times 2}$;\\
{\rm (C11)} $m=4$, 
{\rm (1)}  $a,bcd\in k^{\times 2}$;
{\rm (2)} $b,acd\in k^{\times 2}$;
{\rm (3)} $d,abc\in k^{\times 2}$;
{\rm (4)} $c,abd\in k^{\times 2}$;\\
{\rm (C14)} $m=4$, $ad,bc\in k^{\times 2}$;\\
{\rm (C17)} $m=8$, 
{\rm (1)} $a\in k^{\times 2}$;
{\rm (2)} $b\in k^{\times 2}$;
{\rm (3)} $d\in k^{\times 2}$;
{\rm (4)} $c\in k^{\times 2}$.

For example, for the case of (C4) $m=2$, {\rm (1)} $a, d, bc\in k^{\times 2}$, 
$k(x,y,z)^{\langle\sigma_{a,b,c,d}\rangle}$ may be obtained as follows. 
Let $L=k(\alpha,\beta,\gamma,\delta)$ 
where $\alpha^2=a$, $\beta^2=b$, $\gamma^2=c$, $\delta^2=d$. 
Then $L=k(\beta)$ and 
$L(x,y,z)=L(x,y',z')$ where 
$y'=(\alpha x+\beta)/y$ and $z'=(\gamma x+\delta)/z$. 
We see that $\sigma_{a,b,c,d}$ acts 
on $L(x,y',z')$ by 
\[
\sigma_{a,b,c,d} : x\mapsto x,\ y'\mapsto \frac{1}{y'},\ z'\mapsto \frac{1}{z'}. 
\]
Hence $L(x,y',z')^{\langle\sigma_{a,b,c,d}\rangle}=L(y_1,y_2,y_3)$ where
\[
y_1=x^2,\ 
y_2=x\ \frac{1-y'}{1+y'}=x\ \frac{y-\alpha x-\beta}{y+\alpha x+\beta},\ 
y_3=x\ \frac{1-z'}{1+z'}=x\ \frac{z-\gamma x-\delta}{y+\gamma x+\delta}.
\]
By the assumption $bc\in k^{\times 2}$, 
there exists $e\in k^\times$ such that $\gamma=\beta e$.  
Then $k(x,y,z)^{\langle\sigma_{a,b,c,d}\rangle}
=(L(x,y,z)^{\langle\sigma_{a,b,c,d}\rangle})^{\langle\rho_b\rangle}$ 
$=$ $L(y_1,y_2,y_3)^{\langle\rho_b\rangle}$ and 
\begin{align*}
\rho_b : \beta\mapsto -\beta,\ 
y_1\mapsto y_1,\ 
y_2\mapsto \frac{y_1(\alpha y_2+\beta)}
{\alpha y_1+\beta y_2},\ y_3\mapsto 
\frac{\beta e y_1+\delta y_3}{\beta e y_3+\delta}. 
\end{align*}
Define
\begin{align*}
z_1:=\beta e\left(\frac{\beta e y_1+\delta y_3}
{\beta e y_3+\delta}-y_3\right)-2\delta,\ 
z_2:=e\left(\frac{\beta e y_1+\delta y_3}
{\beta e y_3+\delta}+y_3\right),\ 
z_3:=4be^2(\alpha y_1+\beta y_2).
\end{align*}
Then $L(y_1,y_2,y_3)=L(z_1,z_2,z_3)$ and 
\begin{align*}
\rho_b : \beta\mapsto -\beta,\ 
z_1\mapsto z_1,\ 
z_2\mapsto z_2,\ 
z_3\mapsto 
\frac{f(z_1,z_2)}{z_3}
\end{align*}
where $f(z_1,z_2)=(z_1^2-bz_2^2-4d)(a(z_1^2-bz_2^2-4d)+4b^2e^2)$. 
Define
\[
t_1:=\frac{1}{2}
\left(z_3+\frac{f(z_1,z_2)}{z_3}\right),\ 
t_2:=\frac{1}{2\beta}
\left(z_3-\frac{f(z_1,z_2)}{z_3}\right).
\]
Then $k(x,y,z)^{\langle\sigma_{a,b,c,d}\rangle}=
k(t_1,t_2,z_1,z_2)$ where 
\[
t_1^2-bt_2^2=f(z_1,z_2)=
(z_1^2-bz_2^2-4d)(a(z_1^2-bz_2^2-4d)+4b^2e^2).
\]
However, we do not know whether 
$k(x,y,z)^{\langle\sigma_{a,b,c,d}\rangle}$ is $k$-rational (resp. 
stably $k$-rational, retract $k$-rational). 
\end{remark}

\bigskip

\section{Application of Theorem \ref{th3}}\label{seApp}

Let $G$ be a finite group acting on the rational function field 
$k(x_g\,|\,g\in G)$ by $k$-automorphisms $h(x_g)=x_{hg}$ for any $g,h\in G$. 
We denote the fixed field $k(x_g\,|\,g\in G)^G$ by $k(G)$. 
Noether's problem asks whether $k(G)$ is $k$-rational 
(see a survey paper of Swan \cite{Swa83} for Noether's problem for abelian groups). 
The unramified Brauer group $Br_{v,k}(L)$ of $L$ over $k$ is defined to be 
$Br_{v,k}(L)=\cap_R Br(R)$ where $Br(R)$ is the image of the injective map 
$Br(R)\rightarrow Br(L)$ of Brauer groups and $R$ runs over all discrete valuation 
domain with $k\subset R$ and the quotient field of $R$ is $L$ (cf. \cite{Sal84b}). 

Let $F$ be an algebraically closed field with ${\rm gcd}\{{\rm char}\,k,|G|\}=1$. 
By \cite[Theorem 12]{Sal90}, \cite[Theorem 1.3$'$]{Bog90}, 
$Br_{v,F}(F(G))=\cap_A {\rm Ker}({\rm res} : H^2(G,\mu)\rightarrow H^2(A,\mu))$ 
where $\mu\subset F^{\times}$ denotes the subgroup of roots of unity, 
${\rm res}$ is the restriction map of cohomology groups and 
$A$ runs over abelian subgroups of $G$ of rank $1$ or $2$. 
If $G$ is a $2$-group, this is valid not only over an algebraically closed field 
but also over a quadratically closed field $F$ (i.e. a field of ${\rm char}\, F \neq 2$ 
satisfying $\sqrt{a}\in F$ for any $a\in F$). 
We denote $Br_{v,F}(F(G))$ by $B_0(G)$ (cf. \cite{BMP04}, \cite{Kun10}, \cite{CHKK10}). 
By \cite[Proposition 3.2]{Sal84b} and \cite[Proposition 2.2]{Sal87}, 
if $L_1/L_2$ $(\supset k)$ is retract rational, then $Br_{v,k}(L_1)\simeq Br_{v,k}(L_2)$ 
and hence if $F(G)$ is retract $F$-rational, then $B_0(G)=0$. 

Saltman \cite{Sal84b} showed that for any prime $p$ 
there exists a meta-abelian $p$-group $G$ of order $p^9$ such that $B_0(G)\neq 0$. 
In particular, $F(G)$ is not retract $F$-rational. 
Moreover, Bogomolov \cite{Bog88} obtained that when char $F=0$ 
there exists a $p$-group $G$ of order $p^6$ such that $B_0(G)\neq 0$ 
for any prime $p$. 
Recently, Moravec \cite{Mor12} proved that there exist exactly 
$3$ groups of order $3^5$ such that $B_0(G)\neq 0$ by GAP computations. 
Hoshi, Kang and Kunyavskii \cite{HKK13} showed that 
there exist exactly $1+{\rm gcd}\{4,p-1\}+{\rm gcd}\{3,p-1\}$ 
groups of order $p^5$ such that $B_0(G)\neq 0$ for 
any prime $p\geq 5$.

In the case where $p=2$, by Chu and Kang \cite{CK01} and 
Chu, Hu, Kang and Prokhorov \cite{CHKP08}, 
if $G$ is a group of order $\leq 32$, then $F(G)$ is $F$-rational. 
Moreover Chu, Hu, Kang and Kunyavskii \cite{CHKK10} investigated 
the case where $G$ is a group of order $64$ as follows. 
There exist exactly $267$ non-isomorphic groups of order $64$. 

Let $C_n$ be the cyclic group of order $n$, $Z(G)$ be the center of the group $G$ 
and $[G,G]$ be the commutator subgroup of the group $G$. 
Let $G(n,i)$ be the $i$-th group of order $n$ in GAP \cite{GAP}. 

\begin{theorem}[{\cite[Theorem 1.8]{CHKK10}}]\label{thCHKK1}
Let $G$ be a group of order $64$ and $F$ be a quadratically closed field. 
Then the following conditions are equivalent:\\
{\rm (i)} $B_0(G)\neq 0$;\\
{\rm (ii)} $Z(G)\simeq C_2^2, [G,G]\simeq C_4\times C_2, G/[G,G]\simeq C_2^3$, 
$G$ has no abelian subgroup of index $2$, and $G$ has no faithful $4$-dimensional 
representation over $\bC$;\\
{\rm (iii)} $G$ is isomorphic to one of the nine groups $G(64,i)$ where 
$i=149$, $150$, $151$, $170$, $171$, $172$, $177$, $178$, $182$. 
\end{theorem}
\begin{theorem}[{\cite[Theorem 1.9]{CHKK10}}]\label{thCHKK2}
Let $G$ be a group of order $64$ and $F$ be a quadratically closed field. 
If $B_0(G)=0$, then $F(G)$ is $F$-rational except possibly 
for groups $G$ which is isomorphic to one of the five groups $G(64,i)$ where 
$241\leq i\leq 245$. 
\end{theorem}
\begin{theorem}[{\cite[Proposition 6.3]{CHKK10}}]\label{thCHKK3}
Let $G=G(64,i)$ where $241\leq i\leq 245$, 
$F$ be a quadratically closed field 
and $H=\langle f_1,f_2\rangle\simeq C_2\times C_2$ act on the rational function field 
$F(x_1,x_2,x_3,y_1,y_2,y_3)$ by $F$-automorphisms defined by 
\begin{align*}
f_1 &: x_1\mapsto \frac{1}{x_1},\ x_2\mapsto \frac{1}{x_1x_3},\ x_3\mapsto \frac{x_1}{x_2},\ 
y_1\mapsto y_1,\ y_2\mapsto \frac{y_1}{y_2},\ y_3\mapsto \frac{1}{y_1y_3},\\
f_2 &: x_1\mapsto \frac{1}{x_1},\ x_2\mapsto x_3,\ x_3\mapsto x_2,\ 
y_1\mapsto \frac{1}{y_1},\ y_2\mapsto y_3,\ y_3\mapsto y_2.
\end{align*}
Let $L=F(x_1,x_2,x_3,y_1,y_2,y_3)^H$. 
Then there is an $F$-injective homomorphism $\varphi : L\rightarrow F(G)$ so that 
$F(G)$ is a rational extension over $\rho(L)$.
\end{theorem}

Let $k(x_1,\ldots,x_n)$ be the rational function field over $k$ with $n$ variables 
$x_1,\ldots,x_n$. 
Let $H$ be a finite subgroup of $\GL(n,\bZ)$ acting on $k(x_1,\ldots,x_n)$ by 
$k$-automorphisms 
\begin{align}
\sigma(x_j)=c_j(\sigma)\prod_{i=1}^n x_i^{a_{i,j}},\ \sigma=[a_{i,j}]_{1\leq i,j\leq n}\in H,\ 
c_j(\sigma)\in k^\times,\ {\rm for}\ 1\leq j\leq n,\label{acts2}
\end{align}
where $k^\times=k\setminus\{0\}$. 
This action of $H$ on $k(x_1,\ldots,x_n)$ is said to be monomial. 
If $c_j(\sigma)=1$ for any $\sigma\in G$ and any $1\leq j\leq n$, 
the action is said to be purely monomial. 
The rationality problem with respect to monomial action, 
i.e. whether $k(x_1,\ldots,x_n)^G$ is $k$-rational, 
is determined up to conjugacy in $\GL(n,\bZ)$ and 
solved affirmatively by Hajja \cite{Haj83}, \cite{Haj87} when $n=2$.
For $n=3$, the problem under purely monomial actions 
was solved affirmatively by Hajja and Kang \cite{HK92}, \cite{HK94} 
and Hoshi and Rikuna \cite{HR08} 
(see also \cite{KP10}, \cite{HK10}, \cite{Yam12} and \cite{HKY11} 
for non-purely monomial case). 

It is remarked in \cite[page 2537]{CHKK10} that although 
$L_0:=F(x_1,x_2,x_3)^H=F(t_1,t_2,t_3)$ is $F$-rational, 
the field $L$ as in Theorem \ref{thCHKK3} may be regarded as 
$L=L_0(\alpha,\beta)(y_1,y_2,y_3)^H$ 
the function field of a $3$-dimensional algebraic torus defined over $L_0$ 
and split over biquadratic Galois extension $F(x_1,x_2,x_3)=L_0(\alpha,\beta)$ of $L_0$ 
for some $\alpha,\beta$ with $\alpha^2,\beta^2\in L_0$. 
By Theorem \ref{thKu}, the field $L$ is not stably rational over 
$L_0=F(t_1,t_2,t_3)$. 
We restate Theorem \ref{thKu} of Kunyavskii after adopting the definition 
of the action of $G$ via $(\ref{acts2})$ (we should take $G$ instead of ${^t}G$ 
and hence the corresponding GAP IDs may change, cf. Theorem \ref{thKu}).
\begin{theorem}[{Kunyavskii \cite{Kun90}}]\label{thKu2}
Let $L/k$ be a Galois extension and $G\simeq 
{\rm Gal}(L/k)$ be a finite subgroup of $\GL(3,\bZ)$ 
which acts on $L(x_1,x_2,x_3)$ via $(\ref{acts2})$. 
Then $L(x_1,x_2,x_3)^G$ is not $k$-rational if and only if 
$G$ is conjugate to one of the $15$ groups which are given 
as in {\rm Table} $10$.  
Moreover, if $L(x_1,x_2,x_3)^G$ is not $k$-rational, 
then it is not retract $k$-rational. 
\end{theorem}
\vspace*{2mm}
\begin{center}
Table $10$: $L(M)^G$ not retract $k$-rational, 
rank $M=3$, $M$: indecomposable ($15$ cases)\vspace*{2mm}\\
\fontsize{8pt}{10pt}\selectfont
\begin{tabular}{lll} 
$G$ in \cite{Kun90} & GAP ID& $G$ \\\hline
$U_1$ & $(3,3,1,4)$ & $C_2^2$\\
$U_2$ & $(3,3,3,4)$ & $C_2^3$\\
$U_3$ & $(3,4,4,2)$ & $D_4$\\
$U_4$ & $(3,4,6,4)$ & $D_4$\\
$U_5$ & $(3,7,1,3)$ & $A_4$
\end{tabular}\quad
\begin{tabular}{lll}
$G$ in \cite{Kun90} & GAP ID & $G$ \\\hline
$U_6$ & $(3,4,7,2)$ & $D_4\times C_2$\\
$U_7$ & $(3,7,2,3)$ & $A_4\times C_2$\\
$U_8$ & $(3,7,3,2)$ & $S_4$\\
$U_9$ & $(3,7,3,3)$ & $S_4$\\
$U_{10}$ & $(3,7,4,3)$ & $S_4$
\end{tabular}\quad
\begin{tabular}{lll} 
$G$ in \cite{Kun90} & GAP ID & $G$ \\\hline
$U_{11}$ & $(3,7,5,2)$ & $S_4\times C_2$\\
$U_{12}$ & $(3,7,5,3)$ & $S_4\times C_2$\\
$W_1$ & $(3,4,3,2)$ & $C_4\times C_2$\\
$W_2$ & $(3,3,3,3)$ & $C_2^3$\\
$W_3$ & $(3,7,2,2)$ & $A_4\times C_2$
\end{tabular}
\end{center}
\vspace*{4mm}

Let $H=\langle f_1,f_2\rangle\simeq C_2\times C_2$ 
act on $k(x_1,x_2,x_3,y_1,y_2,y_3)$ by $k$-automorphisms 
as the same in Theorem \ref{thCHKK3}. 
The purely monomial actions of $H$ 
on $k(x_1,x_2,x_3)$ and on $k(y_1,y_2,y_3)$ correspond to the same conjugacy class 
of the GAP ID $(3,3,1,4)$. 
Indeed, if we define 
$(X_1,X_2,X_3,Y_1,Y_2,Y_3):=(x_1/x_2,1/(x_1x_3),x_3,1/(y_1y_3),y_1/y_2,y_3)$, 
then 
$k(x_1,x_2,x_3,y_1,y_2,y_3)=k(X_1,X_2,X_3,Y_1,Y_2,Y_3)$ and the action of 
$H$ on $k(X_1,X_2,X_3,Y_1,Y_2,Y_3)$ is given by 
\begin{align*}
\sigma_1 &: X_1\mapsto X_3,\ X_2\mapsto \frac{1}{X_1X_2X_3},\ X_3\mapsto X_1,\ 
Y_1\mapsto Y_3,\ Y_2\mapsto \frac{1}{Y_1Y_2Y_3},\ Y_3\mapsto Y_1,\\
\sigma_2 &: X_1\mapsto X_2,\ X_2\mapsto X_1,\ X_3\mapsto \frac{1}{X_1X_2X_3},\ 
Y_1\mapsto Y_2,\ Y_2\mapsto Y_1,\ Y_3\mapsto \frac{1}{Y_1Y_2Y_3}. 
\end{align*}
We define
\begin{align*}
s_1:=\frac{1+X_1X_3}{1-X_1X_3},\ s_2:=\frac{1+X_2X_3}{1-X_2X_3},\ s_3:=X_3,\ 
s_4:=\frac{1+Y_1Y_3}{1-Y_1Y_3},\ s_5:=\frac{1+Y_2Y_3}{1-Y_2Y_3},\ s_6:=Y_3.
\end{align*}
Then $k(X_1,X_2,X_3,Y_1,Y_2,Y_3)=k(s_1,\ldots,s_6)$ 
and the actions of $\sigma_1$ and $\sigma_2$ on $k(s_1,\ldots,s_6)$ are given by 
\begin{align*}
\sigma_1 :&s_1\mapsto s_1, s_2\mapsto -s_2, s_3\mapsto \frac{s_1-1}{s_3(s_1+1)},
s_4\mapsto s_4, s_5\mapsto -s_5, s_6\mapsto \frac{s_4-1}{s_6(s_4+1)},\\
\sigma_2 :&s_1\mapsto -s_1, s_2\mapsto -s_2, 
s_3\mapsto \frac{s_3(s_1+1)(s_2+1)}{(s_1-1)(s_2-1)}, 
s_4\mapsto -s_4, s_5\mapsto -s_5, s_6\mapsto \frac{s_6(s_4+1)(s_5+1)}{(s_4-1)(s_5-1)}.
\end{align*}
We also define 
\begin{align*}
t_1:=s_1s_2,\ t_2:=\frac{(s_1s_2+1)s_3}{(s_1-1)(s_2-1)},\ 
t_3:=\frac{(s_4s_5+1)s_6}{(s_4-1)(s_5-1)},\  
t_4:=s_2^2,\ t_5:=\frac{s_1}{s_4},\ t_6:=\frac{s_2}{s_5}.
\end{align*}
Then we have $k(s_1,\ldots,s_6)=k(t_1,t_2,t_3,s_2,t_5,t_6)$ 
and $\sigma_2 : s_2\mapsto -s_2,$ $t_i\mapsto t_i$, $1\leq i\leq 6$.  
It follows that $k(x_1,x_2,x_3,y_1,y_2,y_3)^{\langle \sigma_2\rangle}$ $=$ $k(t_1,\ldots,t_6)$ and 
the action of $\sigma_1$ on $k(t_1,\ldots,t_6)$ is given by 
\begin{align*}
\sigma_1 :&\ t_1\mapsto -t_1,\ t_2\mapsto \frac{(t_1^2-1)t_4}{t_2(t_1^2-t_4)(t_4-1)},\ 
t_3\mapsto \frac{t_4(t_5^2t_6^2-t_1^2)}{t_3(t_6^2-t_4)(t_1^2-t_4t_5^2)},\ 
t_4\mapsto t_4,\ t_5\mapsto t_5,\ t_6\mapsto t_6.
\end{align*}
Define 
\begin{align*}
u_1:=\frac{t_1}{t_4},\ u_2:=\frac{t_1+1}{t_2},\ u_3:=\frac{t_5t_6+t_1}{t_3},\ 
u_4:=t_4,\ u_5:=t_5,\ u_6:=t_6.
\end{align*}
Then we get $k(t_1,\ldots,t_6)=k(u_1,\ldots,u_6)$ and the action of $\sigma_1$ on 
$k(u_1,\ldots,u_6)$ is given by 
\begin{align}
\sigma_1 :&\ u_1\mapsto -u_1,\ u_2\mapsto -\frac{(u_4-1)(u_1^2u_4-1)}{u_2},\ 
u_3\mapsto -\frac{(u_4-u_6^2)(u_1^2u_4-u_5^2)}{u_3},\label{equ}\\ 
&\ u_4\mapsto u_4,\ u_5\mapsto u_5,\ u_6\mapsto u_6.\nonumber
\end{align}
As a consequence of Theorem \ref{th3}, we have:
\begin{proposition}\label{prop1}
Let $k$ be a field of ${\rm char}\, k\neq 2$ and 
$\langle\sigma_1,\sigma_2\rangle\simeq C_2\times C_2$ act on 
$k(X_1,X_2,X_3,Y_1,Y_2,Y_3)$ by $k$-automorphisms 
\begin{align*}
\sigma_1 &: X_1\mapsto X_3,\ X_2\mapsto \frac{1}{X_1X_2X_3},\ X_3\mapsto X_1,\ 
Y_1\mapsto Y_3,\ Y_2\mapsto \frac{1}{Y_1Y_2Y_3},\ Y_3\mapsto Y_1,\\
\sigma_2 &: X_1\mapsto X_2,\ X_2\mapsto X_1,\ X_3\mapsto \frac{1}{X_1X_2X_3},\ 
Y_1\mapsto Y_2,\ Y_2\mapsto Y_1,\ Y_3\mapsto \frac{1}{Y_1Y_2Y_3}. 
\end{align*}
There exist algebraically independent elements 
$u_4,u_5,u_6\in k(X_1,X_2,X_3,Y_1,Y_2,Y_3)^{\langle\sigma_1,\sigma_2\rangle}$
over $k$ and algebraically dependent elements 
$z_1,z_2,z_3,z_4\in k(X_1,X_2,X_3,Y_1,Y_2,Y_3)^{\langle\sigma_1,\sigma_2\rangle}$ 
over $k(u_4,u_5,u_6)$ which satisfy the following conditions:\\
{\rm (i)} 
$k(X_1,X_2,X_3,Y_1,Y_2,Y_3)^{\langle\sigma_1,\sigma_2\rangle}=k(z_1,z_2,z_3,z_4,u_4,u_5,u_6)$ 
and 
\begin{align*}
(z_1^2-a)(z_4^2-d)=(z_2^2-b)(z_3^2-c)
\end{align*}
where $a=u_4(u_4-1), b=u_4-1, c=u_4(u_4-u_6^2), d=u_5^2(u_4-u_6^2)$;\\
{\rm (ii)} 
$k(z_1,z_2,z_3,z_4,u_4,u_5,u_6)$ is 
not retract rational over $k(u_4,u_5,u_6)$. 
\end{proposition}
\begin{proof}
If we take a field $k(u_4,u_5,u_6)$ as the base field, 
the assertion (i) follows from (\ref{equ}) and Theorem \ref{th3} (i). 
For (ii), we see that $abcd\in k(u_4,u_5,u_6)^{\times 2}$ 
and $m=[k(u_4,u_5,u_6)(\sqrt{a},\sqrt{b},\sqrt{c},\sqrt{d}):k(u_4,u_5,u_6)]=8$. 
Thus (ii) follows from Theorem \ref{th3} (iii). 
\end{proof}
Furthermore, we define
\begin{align*}
v_1:=u_1,\ v_2:=\frac{u_2}{1-u_1^2u_4},\ v_3:=\frac{(u_1^2-1)u_3}{1-u_1^2u_4},\ 
v_4:=\frac{u_4-1}{1-u_1^2u_4},\ 
v_5:=u_5,\ v_6:=u_6. 
\end{align*}
Then $k(X_1,X_2,X_3,Y_1,Y_2,Y_3)^{\langle\sigma_2\rangle}
=k(u_1,\ldots,u_6)=k(v_1,\ldots,v_6)$ and 
\begin{align*}
\sigma_1 :&v_1\mapsto -v_1, v_2\mapsto \frac{v_4}{v_2}, 
v_3\mapsto -\frac{(v_1^2v_4v_5^2-v_1^2v_4-v_1^2+v_5^2)(v_1^2v_4v_6^2+v_6^2-v_4-1)}{v_3}, 
v_4\mapsto v_4, v_5\mapsto v_5, v_6\mapsto v_6.
\end{align*}
We also put 
\begin{align*}
w_1:=v_1,\ w_2:=v_3,\ w_3:=\frac{1}{2}\left(v_2+\frac{v_4}{v_2}\right),\ 
w_4:=\frac{1}{2v_1}\left(v_2-\frac{v_4}{v_2}\right),\ 
w_5:=v_5,\ w_6:=v_6. 
\end{align*}
Then $k(v_1,\ldots,v_6)=k(w_1,\ldots,w_6)$ and 
\begin{align*}
\sigma_1 :&w_1\mapsto -w_1, w_2\mapsto 
-\frac{\left(w_4^2(w_5^2-1)w_1^4+(w_3^2-w_3^2w_5^2+1)w_1^2-w_5^2\right)
\left(w_4^2w_6^2w_1^4-(w_4^2+w_3^2w_6^2)w_1^2+w_3^2-w_6^2+1\right)}{w_2},\\ 
&w_3\mapsto w_3, 
w_4\mapsto w_4, w_5\mapsto w_5, w_6\mapsto w_6.
\end{align*}
Define $t=w_1^2$. 
Then $k(X_1,X_2,X_3,Y_1,Y_2,Y_3)^{\langle\sigma_1,\sigma_2\rangle}
=k(u,v,t,w_3,w_4,w_5,w_6)$ with the relation
\begin{align*}
u^2-tv^2=-\left(w_4^2(w_5^2-1)t^2+(w_3^2-w_3^2w_5^2+1)t-w_5^2\right)
\left(w_4^2w_6^2t^2-(w_4^2+w_3^2w_6^2)t+w_3^2-w_6^2+1\right)
\end{align*}
(cf. \cite[Section 6]{HKK14}). 

We do not know whether 
$k(X_1,X_2,X_3,Y_1,Y_2,Y_3)^{\langle\sigma_1,\sigma_2\rangle}$ 
is $k$-rational (resp. 
stably $k$-rational, retract $k$-rational). 
We may also obtain the following description of 
$k(X_1,X_2,X_3,Y_1,Y_2,Y_3)^{\langle\sigma_1,\sigma_2\rangle}$ 
although the proof is omitted.

\begin{theorem}
The field 
$k(X_1,X_2,X_3,Y_1,Y_2,Y_3)^{\langle\sigma_1,\sigma_2\rangle}=k(m_0,\ldots,m_6)$ where
\begin{align*}
m_0^2=(4m_3+m_3m_4^2+m_4^2)(m_3-m_5^2+1)(m_1^2m_3+m_6^2-1)(4m_3+m_1^2m_2^2m_3+m_2^2m_6^2).
\end{align*}
\end{theorem}

\newpage


\section{Tables for the stably rational classification of algebraic $k$-tori of dimension $5$}\label{tables}
%
\begin{center}
\vspace*{4mm}
Table $11$: $L(M)^G$ not stably but retract $k$-rational, 
$M=M_1\oplus M_2$,\\ rank $M_1=4$, rank $M_2=1$, 
$M_i$ : indecomposable ($25$ cases)\vspace*{2mm}\\
{\tiny 
\begin{tabular}{ll}
CARAT & $G$ \\\hline
(5,692,1) & $C_3\rtimes C_8$\\
(5,693,1) & $C_3\rtimes C_8$\\
(5,736,1) & $C_2\times (C_3\rtimes C_8)$\\
(5,911,1) & $S_5$\\
(5,912,1) & $S_5$
\end{tabular}
\begin{tabular}{ll}
CARAT & $G$\\\hline
(5,914,1) & $S_5$\\
(5,916,1) & $S_5$\\
(5,917,1) & $F_{20}$\\
(5,917,5) & $F_{20}$\\
(5,918,1) & $F_{20}$\\
\end{tabular}
\begin{tabular}{ll}
CARAT & $G$\\\hline
(5,918,5) & $F_{20}$\\
(5,919,1) & $C_2\times S_5$\\
(5,921,1) & $C_2\times S_5$\\
(5,922,1) & $C_2\times S_5$\\
(5,923,1) & $C_2\times S_5$
\end{tabular}
\begin{tabular}{ll}
CARAT & $G$\\\hline
(5,924,1) & $C_2\times S_5$\\
(5,925,1) & $C_2\times S_5$\\
(5,926,1) & $C_2\times F_{20}$\\
(5,926,3) & $C_2\times F_{20}$\\
(5,927,1) & $C_2\times F_{20}$
\end{tabular}
\begin{tabular}{ll}
CARAT & $G$\\\hline
(5,928,1) & $C_2\times F_{20}$\\
(5,928,3) & $C_2\times F_{20}$\\
(5,929,1) & $C_2\times F_{20}$\\
(5,930,1) & $C_2^2\times S_5$\\
(5,932,1) & $C_2^2\times F_{20}$
\end{tabular}
}\vspace*{4mm}
\end{center}
\begin{center}
Table $12$: $L(M)^G$ not retract $k$-rational, 
$M=M_1\oplus M_2\oplus M_3$,\\ rank $M_1=3$, rank $M_2=$ rank $M_3=1$, 
$M_i$ : indecomposable ($245$ cases)\vspace*{2mm}\\
{\tiny 
\begin{tabular}{l}
CARAT \\\hline
(5,11,4)\\
(5,12,4)\\
(5,12,5)\\
(5,14,4)\\
(5,16,4)\\
(5,16,5)\\
(5,17,4)\\
(5,17,5)\\
(5,17,7)\\
(5,17,12)\\
(5,18,7)\\
(5,19,5)\\
(5,20,4)\\
(5,20,5)\\
(5,20,7)\\
(5,20,12)\\
(5,21,4)\\
(5,21,5)\\
(5,21,7)\\
(5,22,5)\\
(5,22,12)\\
(5,23,7)\\
(5,23,15)\\
(5,24,5)\\
(5,24,7)
\end{tabular}
\begin{tabular}{l}
CARAT \\\hline
(5,24,8)\\
(5,24,10)\\
(5,24,20)\\
(5,25,3)\\
(5,25,5)\\
(5,25,9)\\
(5,25,10)\\
(5,26,3)\\
(5,26,5)\\
(5,26,8)\\
(5,26,12)\\
(5,26,22)\\
(5,27,3)\\
(5,27,6)\\
(5,28,3)\\
(5,28,9)\\
(5,29,3)\\
(5,29,6)\\
(5,30,5)\\
(5,30,10)\\
(5,31,4)\\
(5,31,6)\\
(5,31,7)\\
(5,31,12)\\
(5,31,19)
\end{tabular}
\begin{tabular}{l}
CARAT \\\hline
(5,32,8)\\
(5,32,10)\\
(5,32,26)\\
(5,40,2)\\
(5,44,2)\\
(5,45,2)\\
(5,46,2)\\
(5,47,2)\\
(5,47,3)\\
(5,48,2)\\
(5,48,3)\\
(5,48,9)\\
(5,48,10)\\
(5,49,2)\\
(5,49,3)\\
(5,49,10)\\
(5,50,2)\\
(5,50,3)\\
(5,51,3)\\
(5,52,2)\\
(5,52,3)\\
(5,52,9)\\
(5,53,2)\\
(5,53,3)\\
(5,54,2)
\end{tabular}
\begin{tabular}{l}
CARAT \\\hline
(5,55,2)\\
(5,55,3)\\
(5,56,2)\\
(5,56,3)\\
(5,59,3)\\
(5,60,2)\\
(5,60,10)\\
(5,61,2)\\
(5,62,10)\\
(5,63,2)\\
(5,63,9)\\
(5,63,10)\\
(5,64,3)\\
(5,65,3)\\
(5,65,9)\\
(5,65,10)\\
(5,66,2)\\
(5,66,3)\\
(5,67,2)\\
(5,68,2)\\
(5,69,3)\\
(5,70,3)\\
(5,71,2)\\
(5,71,3)\\
(5,71,6)
\end{tabular}
\begin{tabular}{l}
CARAT \\\hline
(5,72,2)\\
(5,72,21)\\
(5,72,22)\\
(5,73,2)\\
(5,73,21)\\
(5,73,22)\\
(5,74,2)\\
(5,75,2)\\
(5,75,21)\\
(5,76,2)\\
(5,76,6)\\
(5,76,28)\\
(5,76,29)\\
(5,76,32)\\
(5,77,2)\\
(5,77,6)\\
(5,78,2)\\
(5,78,3)\\
(5,78,22)\\
(5,79,2)\\
(5,79,3)\\
(5,79,21)\\
(5,79,22)\\
(5,80,2)\\
(5,80,3)
\end{tabular}
\begin{tabular}{l}
CARAT \\\hline
(5,81,3)\\
(5,81,6)\\
(5,81,29)\\
(5,82,3)\\
(5,82,6)\\
(5,83,2)\\
(5,83,3)\\
(5,84,3)\\
(5,85,3)\\
(5,86,2)\\
(5,87,2)\\
(5,88,2)\\
(5,88,3)\\
(5,88,21)\\
(5,88,22)\\
(5,89,2)\\
(5,89,3)\\
(5,90,2)\\
(5,90,21)\\
(5,90,24)\\
(5,91,2)\\
(5,91,3)\\
(5,92,2)\\
(5,93,2)\\
(5,94,2)
\end{tabular}
\begin{tabular}{l}
CARAT \\\hline
(5,94,14)\\
(5,95,2)\\
(5,96,2)\\
(5,97,2)\\
(5,98,6)\\
(5,99,2)\\
(5,99,3)\\
(5,99,32)\\
(5,100,2)\\
(5,100,3)\\
(5,101,2)\\
(5,102,2)\\
(5,103,6)\\
(5,502,6)\\
(5,503,1)\\
(5,503,4)\\
(5,504,1)\\
(5,504,4)\\
(5,505,1)\\
(5,506,1)\\
(5,507,1)\\
(5,507,4)\\
(5,508,1)\\
(5,508,4)\\
(5,509,1)
\end{tabular}
\begin{tabular}{l}
CARAT \\\hline
(5,510,1)\\
(5,510,4)\\
(5,511,1)\\
(5,511,4)\\
(5,512,1)\\
(5,512,4)\\
(5,513,1)\\
(5,514,1)\\
(5,514,4)\\
(5,515,1)\\
(5,515,4)\\
(5,516,1)\\
(5,516,4)\\
(5,517,1)\\
(5,517,4)\\
(5,518,1)\\
(5,518,8)\\
(5,519,1)\\
(5,519,16)\\
(5,520,16)\\
(5,521,16)\\
(5,522,1)\\
(5,522,16)\\
(5,523,1)\\
(5,523,9)
\end{tabular}
\begin{tabular}{l}
CARAT \\\hline
(5,524,1)\\
(5,524,9)\\
(5,525,1)\\
(5,526,1)\\
(5,526,8)\\
(5,527,1)\\
(5,527,9)\\
(5,528,1)\\
(5,528,9)\\
(5,529,1)\\
(5,529,9)\\
(5,530,1)\\
(5,530,9)\\
(5,531,1)\\
(5,532,1)\\
(5,533,1)\\
(5,533,9)\\
(5,534,1)\\
(5,534,8)\\
(5,535,1)\\
(5,536,1)\\
(5,536,9)\\
(5,537,1)\\
(5,537,9)\\
(5,538,1)
\end{tabular}
\begin{tabular}{l}
CARAT \\\hline
(5,538,9)\\
(5,539,1)\\
(5,539,8)\\
(5,540,1)\\
(5,540,9)\\
(5,541,1)\\
(5,541,9)\\
(5,542,1)\\
(5,542,9)\\
(5,543,1)\\
(5,544,1)\\
(5,544,8)\\
(5,545,1)\\
(5,545,8)\\
(5,546,1)\\
(5,546,8)\\
(5,547,1)\\
(5,547,8)\\
(5,548,1)\\
(5,548,9)\\ \\ \\ \\ \\ \\
\end{tabular}
}
\end{center}
\newpage
\begin{center}
Table $13$: $L(M)^G$ not retract $k$-rational, 
$M=M_1\oplus M_2$,\\ 
rank $M_1=3$, rank $M_2=2$, $M_i$ : indecomposable ($849$ cases)\vspace*{2mm}\\
{\tiny
\begin{tabular}{l}
CARAT \\\hline
(5,14,8)\\
(5,16,8)\\
(5,16,9)\\
(5,17,11)\\
(5,17,16)\\
(5,18,19)\\
(5,20,11)\\
(5,20,16)\\
(5,21,11)\\
(5,21,16)\\
(5,23,14)\\
(5,23,26)\\
(5,24,16)\\
(5,24,19)\\
(5,24,25)\\
(5,25,15)\\
(5,25,18)\\
(5,25,24)\\
(5,25,26)\\
(5,26,17)\\
(5,26,20)\\
(5,26,25)\\
(5,26,35)\\
(5,26,37)\\
(5,27,9)\\
(5,27,13)\\
(5,28,18)\\
(5,28,24)\\
(5,29,9)\\
(5,29,13)\\
(5,30,18)\\
(5,30,26)\\
(5,31,17)\\
(5,31,23)\\
(5,31,30)\\
(5,31,38)\\
(5,31,42)\\
(5,32,17)\\
(5,32,24)\\
(5,32,50)\\
(5,35,2)\\
(5,36,2)\\
(5,36,5)\\
(5,38,5)\\
(5,38,8)\\
(5,39,2)\\
(5,39,4)\\
(5,47,7)\\
(5,48,7)\\
(5,48,14)\\
(5,49,7)\\
(5,49,14)\\
(5,50,7)\\
(5,51,7)\\
(5,52,7)\\
(5,53,7)\\
(5,55,7)\\
(5,56,7)\\
(5,59,7)\\
(5,60,14)\\
(5,63,14)\\
(5,64,7)\\
(5,65,7)\\
(5,65,14)\\
(5,66,7)\\
(5,69,7)\\
(5,70,7)\\
(5,71,11)\\
(5,71,15)\\
(5,71,18)\\
(5,71,20)\\
(5,71,21)\\
(5,72,8)\\
(5,72,12)\\
(5,72,16)\\
(5,72,27)\\
(5,72,31)\\
(5,72,33)\\
(5,72,35)\\
(5,73,8)\\
(5,73,12)\\
(5,73,16)\\
(5,73,27)\\
(5,73,31)\\
(5,73,33)
\end{tabular}
\begin{tabular}{l}
CARAT \\\hline
(5,73,35)\\
(5,74,12)\\
(5,75,12)\\
(5,75,31)\\
(5,76,11)\\
(5,76,15)\\
(5,76,20)\\
(5,76,21)\\
(5,76,37)\\
(5,76,41)\\
(5,76,44)\\
(5,76,46)\\
(5,76,47)\\
(5,77,11)\\
(5,77,15)\\
(5,77,20)\\
(5,77,21)\\
(5,78,8)\\
(5,78,11)\\
(5,78,14)\\
(5,78,27)\\
(5,78,33)\\
(5,79,8)\\
(5,79,11)\\
(5,79,14)\\
(5,79,16)\\
(5,79,27)\\
(5,79,30)\\
(5,79,33)\\
(5,79,35)\\
(5,80,8)\\
(5,80,11)\\
(5,80,14)\\
(5,81,18)\\
(5,81,20)\\
(5,81,44)\\
(5,82,18)\\
(5,82,20)\\
(5,83,8)\\
(5,83,12)\\
(5,83,14)\\
(5,83,16)\\
(5,84,14)\\
(5,85,8)\\
(5,85,14)\\
(5,85,16)\\
(5,86,5)\\
(5,86,8)\\
(5,87,5)\\
(5,87,8)\\
(5,87,10)\\
(5,88,8)\\
(5,88,12)\\
(5,88,14)\\
(5,88,16)\\
(5,88,27)\\
(5,88,31)\\
(5,88,33)\\
(5,88,35)\\
(5,89,8)\\
(5,89,12)\\
(5,89,14)\\
(5,89,16)\\
(5,90,8)\\
(5,90,12)\\
(5,90,15)\\
(5,90,27)\\
(5,90,31)\\
(5,90,33)\\
(5,90,34)\\
(5,91,8)\\
(5,91,11)\\
(5,91,14)\\
(5,91,16)\\
(5,92,12)\\
(5,93,5)\\
(5,93,8)\\
(5,94,5)\\
(5,94,8)\\
(5,94,10)\\
(5,94,17)\\
(5,94,20)\\
(5,94,22)\\
(5,95,5)\\
(5,95,8)
\end{tabular}
\begin{tabular}{l}
CARAT \\\hline
(5,95,10)\\
(5,96,5)\\
(5,96,8)\\
(5,96,10)\\
(5,97,5)\\
(5,97,8)\\
(5,97,10)\\
(5,98,20)\\
(5,99,15)\\
(5,99,18)\\
(5,99,47)\\
(5,100,11)\\
(5,100,15)\\
(5,100,18)\\
(5,101,12)\\
(5,102,8)\\
(5,102,12)\\
(5,103,20)\\
(5,104,2)\\
(5,105,2)\\
(5,105,4)\\
(5,105,8)\\
(5,107,4)\\
(5,109,2)\\
(5,109,4)\\
(5,109,8)\\
(5,109,10)\\
(5,110,4)\\
(5,111,2)\\
(5,111,8)\\
(5,111,10)\\
(5,112,2)\\
(5,113,2)\\
(5,115,2)\\
(5,116,2)\\
(5,116,4)\\
(5,116,11)\\
(5,116,17)\\
(5,116,19)\\
(5,116,23)\\
(5,118,12)\\
(5,118,16)\\
(5,119,2)\\
(5,119,7)\\
(5,120,2)\\
(5,120,10)\\
(5,122,2)\\
(5,122,13)\\
(5,123,2)\\
(5,123,4)\\
(5,123,11)\\
(5,123,13)\\
(5,124,2)\\
(5,124,4)\\
(5,125,2)\\
(5,125,4)\\
(5,127,2)\\
(5,127,4)\\
(5,127,8)\\
(5,127,10)\\
(5,128,2)\\
(5,128,8)\\
(5,128,10)\\
(5,129,4)\\
(5,130,2)\\
(5,130,4)\\
(5,131,2)\\
(5,131,7)\\
(5,132,2)\\
(5,132,7)\\
(5,132,9)\\
(5,132,13)\\
(5,134,2)\\
(5,134,7)\\
(5,135,2)\\
(5,136,2)\\
(5,136,10)\\
(5,136,12)\\
(5,136,16)\\
(5,139,2)\\
(5,139,8)\\
(5,139,10)\\
(5,140,4)\\
(5,140,10)\\
(5,140,14)
\end{tabular}
\begin{tabular}{l}
CARAT \\\hline
(5,140,16)\\
(5,140,20)\\
(5,140,22)\\
(5,141,4)\\
(5,141,10)\\
(5,141,14)\\
(5,141,16)\\
(5,141,20)\\
(5,141,22)\\
(5,142,2)\\
(5,142,4)\\
(5,142,11)\\
(5,142,13)\\
(5,142,17)\\
(5,142,19)\\
(5,142,23)\\
(5,142,25)\\
(5,143,2)\\
(5,143,4)\\
(5,143,8)\\
(5,143,10)\\
(5,143,14)\\
(5,143,16)\\
(5,143,20)\\
(5,143,22)\\
(5,144,4)\\
(5,144,8)\\
(5,144,10)\\
(5,145,2)\\
(5,145,4)\\
(5,145,8)\\
(5,145,10)\\
(5,146,4)\\
(5,146,8)\\
(5,146,10)\\
(5,147,2)\\
(5,148,2)\\
(5,148,7)\\
(5,149,2)\\
(5,149,10)\\
(5,149,12)\\
(5,149,16)\\
(5,150,2)\\
(5,150,4)\\
(5,151,2)\\
(5,151,7)\\
(5,151,9)\\
(5,151,13)\\
(5,152,2)\\
(5,153,2)\\
(5,153,4)\\
(5,153,8)\\
(5,153,10)\\
(5,154,2)\\
(5,154,7)\\
(5,154,9)\\
(5,154,13)\\
(5,155,2)\\
(5,157,2)\\
(5,157,4)\\
(5,157,8)\\
(5,158,2)\\
(5,158,4)\\
(5,158,8)\\
(5,158,10)\\
(5,159,2)\\
(5,159,4)\\
(5,160,2)\\
(5,224,4)\\
(5,227,5)\\
(5,228,3)\\
(5,232,4)\\
(5,232,9)\\
(5,233,4)\\
(5,233,6)\\
(5,234,5)\\
(5,234,7)\\
(5,235,3)\\
(5,235,4)\\
(5,236,3)\\
(5,236,4)\\
(5,237,3)\\
(5,239,4)\\
(5,239,5)\\
(5,240,4)
\end{tabular}
\begin{tabular}{l}
CARAT \\\hline
(5,240,6)\\
(5,240,14)\\
(5,240,15)\\
(5,241,4)\\
(5,241,5)\\
(5,241,9)\\
(5,241,11)\\
(5,242,4)\\
(5,242,14)\\
(5,244,4)\\
(5,244,6)\\
(5,244,14)\\
(5,244,15)\\
(5,246,5)\\
(5,246,7)\\
(5,246,17)\\
(5,246,18)\\
(5,247,3)\\
(5,247,7)\\
(5,248,3)\\
(5,248,4)\\
(5,248,7)\\
(5,248,8)\\
(5,249,4)\\
(5,249,5)\\
(5,249,9)\\
(5,249,10)\\
(5,251,4)\\
(5,251,5)\\
(5,251,9)\\
(5,251,10)\\
(5,253,4)\\
(5,254,5)\\
(5,254,6)\\
(5,256,3)\\
(5,256,4)\\
(5,257,4)\\
(5,257,5)\\
(5,257,9)\\
(5,257,11)\\
(5,258,3)\\
(5,258,4)\\
(5,258,7)\\
(5,258,8)\\
(5,259,3)\\
(5,260,3)\\
(5,260,4)\\
(5,261,3)\\
(5,261,4)\\
(5,261,7)\\
(5,261,8)\\
(5,262,4)\\
(5,262,5)\\
(5,264,4)\\
(5,264,5)\\
(5,265,4)\\
(5,265,5)\\
(5,265,9)\\
(5,265,10)\\
(5,266,3)\\
(5,266,4)\\
(5,273,2)\\
(5,274,2)\\
(5,275,2)\\
(5,276,2)\\
(5,278,2)\\
(5,279,2)\\
(5,280,2)\\
(5,280,4)\\
(5,281,2)\\
(5,282,2)\\
(5,285,2)\\
(5,288,2)\\
(5,289,2)\\
(5,290,2)\\
(5,290,4)\\
(5,291,2)\\
(5,291,4)\\
(5,294,2)\\
(5,294,4)\\
(5,295,2)\\
(5,295,4)\\
(5,296,2)\\
(5,296,4)\\
(5,299,2)
\end{tabular}
\begin{tabular}{l}
CARAT \\\hline
(5,299,4)\\
(5,302,2)\\
(5,302,10)\\
(5,303,2)\\
(5,303,6)\\
(5,304,2)\\
(5,304,10)\\
(5,305,2)\\
(5,305,6)\\
(5,305,10)\\
(5,305,12)\\
(5,306,4)\\
(5,306,10)\\
(5,309,2)\\
(5,309,4)\\
(5,309,6)\\
(5,309,8)\\
(5,310,2)\\
(5,310,6)\\
(5,311,2)\\
(5,311,4)\\
(5,312,2)\\
(5,312,6)\\
(5,315,4)\\
(5,315,8)\\
(5,316,2)\\
(5,317,2)\\
(5,317,4)\\
(5,318,2)\\
(5,319,2)\\
(5,320,2)\\
(5,322,2)\\
(5,323,2)\\
(5,323,4)\\
(5,324,2)\\
(5,325,2)\\
(5,326,2)\\
(5,326,4)\\
(5,327,2)\\
(5,328,2)\\
(5,328,4)\\
(5,331,2)\\
(5,332,2)\\
(5,333,2)\\
(5,333,4)\\
(5,334,2)\\
(5,334,4)\\
(5,335,2)\\
(5,335,4)\\
(5,336,2)\\
(5,336,4)\\
(5,336,6)\\
(5,336,10)\\
(5,337,2)\\
(5,337,6)\\
(5,338,2)\\
(5,338,4)\\
(5,339,2)\\
(5,339,4)\\
(5,340,2)\\
(5,340,4)\\
(5,340,6)\\
(5,340,10)\\
(5,342,2)\\
(5,342,6)\\
(5,343,2)\\
(5,343,4)\\
(5,343,6)\\
(5,343,8)\\
(5,344,2)\\
(5,344,6)\\
(5,345,2)\\
(5,345,4)\\
(5,346,2)\\
(5,346,4)\\
(5,346,6)\\
(5,346,8)\\
(5,347,2)\\
(5,347,4)\\
(5,347,6)\\
(5,347,8)\\
(5,348,2)\\
(5,348,4)\\
(5,348,6)\\
(5,348,8)
\end{tabular}
\begin{tabular}{l}
CARAT \\\hline
(5,349,2)\\
(5,349,4)\\
(5,350,2)\\
(5,350,4)\\
(5,350,6)\\
(5,350,8)\\
(5,351,2)\\
(5,351,4)\\
(5,352,2)\\
(5,353,2)\\
(5,353,4)\\
(5,354,2)\\
(5,355,2)\\
(5,357,4)\\
(5,359,2)\\
(5,359,6)\\
(5,360,2)\\
(5,360,4)\\
(5,360,6)\\
(5,360,8)\\
(5,361,2)\\
(5,361,4)\\
(5,361,6)\\
(5,361,8)\\
(5,362,2)\\
(5,362,4)\\
(5,363,2)\\
(5,363,4)\\
(5,364,2)\\
(5,364,4)\\
(5,364,6)\\
(5,364,8)\\
(5,365,2)\\
(5,365,6)\\
(5,368,2)\\
(5,368,4)\\
(5,369,2)\\
(5,369,4)\\
(5,369,6)\\
(5,369,8)\\
(5,370,2)\\
(5,371,2)\\
(5,372,2)\\
(5,372,4)\\
(5,373,2)\\
(5,373,4)\\
(5,373,6)\\
(5,373,8)\\
(5,374,2)\\
(5,374,4)\\
(5,376,4)\\
(5,377,2)\\
(5,377,4)\\
(5,378,2)\\
(5,378,4)\\
(5,379,2)\\
(5,380,2)\\
(5,380,4)\\
(5,380,6)\\
(5,380,8)\\
(5,381,2)\\
(5,381,4)\\
(5,382,2)\\
(5,383,2)\\
(5,384,2)\\
(5,384,4)\\
(5,385,2)\\
(5,518,2)\\
(5,518,10)\\
(5,519,4)\\
(5,519,17)\\
(5,520,17)\\
(5,521,17)\\
(5,522,4)\\
(5,522,17)\\
(5,523,2)\\
(5,523,12)\\
(5,524,2)\\
(5,524,12)\\
(5,525,2)\\
(5,526,2)\\
(5,526,10)\\
(5,527,2)\\
(5,527,12)\\
(5,528,2)
\end{tabular}
\begin{tabular}{l}
CARAT \\\hline
(5,528,12)\\
(5,529,2)\\
(5,529,12)\\
(5,530,2)\\
(5,530,12)\\
(5,531,2)\\
(5,532,2)\\
(5,533,2)\\
(5,533,12)\\
(5,534,2)\\
(5,534,10)\\
(5,535,2)\\
(5,536,2)\\
(5,536,12)\\
(5,537,2)\\
(5,537,12)\\
(5,538,2)\\
(5,538,12)\\
(5,539,2)\\
(5,539,10)\\
(5,540,2)\\
(5,540,12)\\
(5,541,2)\\
(5,541,12)\\
(5,542,2)\\
(5,542,12)\\
(5,543,2)\\
(5,544,2)\\
(5,544,10)\\
(5,545,2)\\
(5,545,10)\\
(5,546,2)\\
(5,546,10)\\
(5,547,2)\\
(5,547,10)\\
(5,548,2)\\
(5,548,12)\\
(5,549,3)\\
(5,549,4)\\
(5,550,5)\\
(5,550,6)\\
(5,550,8)\\
(5,550,9)\\
(5,551,5)\\
(5,551,6)\\
(5,551,8)\\
(5,551,9)\\
(5,552,5)\\
(5,552,6)\\
(5,552,8)\\
(5,552,9)\\
(5,553,3)\\
(5,553,4)\\
(5,554,3)\\
(5,554,4)\\
(5,555,3)\\
(5,556,3)\\
(5,556,4)\\
(5,557,3)\\
(5,557,4)\\
(5,558,3)\\
(5,558,4)\\
(5,559,3)\\
(5,559,4)\\
(5,560,3)\\
(5,561,3)\\
(5,561,4)\\
(5,562,3)\\
(5,562,4)\\
(5,563,3)\\
(5,564,3)\\
(5,564,4)\\
(5,565,5)\\
(5,565,6)\\
(5,565,8)\\
(5,565,9)\\
(5,566,3)\\
(5,567,3)\\
(5,567,4)\\
(5,568,5)\\
(5,568,6)\\
(5,568,8)\\
(5,568,9)\\
(5,569,5)\\
(5,569,6)
\end{tabular}
\begin{tabular}{l}
CARAT \\\hline
(5,569,8)\\
(5,569,9)\\
(5,570,5)\\
(5,570,6)\\
(5,570,8)\\
(5,570,9)\\
(5,571,5)\\
(5,571,8)\\
(5,572,3)\\
(5,573,3)\\
(5,573,4)\\
(5,574,5)\\
(5,574,6)\\
(5,574,8)\\
(5,574,9)\\
(5,575,3)\\
(5,575,4)\\
(5,576,3)\\
(5,577,3)\\
(5,577,4)\\
(5,578,3)\\
(5,578,4)\\
(5,579,3)\\
(5,579,4)\\
(5,580,1)\\
(5,581,1)\\
(5,581,3)\\
(5,582,1)\\
(5,582,3)\\
(5,583,1)\\
(5,583,3)\\
(5,584,1)\\
(5,584,3)\\
(5,585,1)\\
(5,585,3)\\
(5,586,1)\\
(5,587,1)\\
(5,587,3)\\
(5,588,1)\\
(5,588,3)\\
(5,589,1)\\
(5,589,3)\\
(5,589,4)\\
(5,589,6)\\
(5,590,1)\\
(5,591,1)\\
(5,591,3)\\
(5,591,4)\\
(5,591,6)\\
(5,592,1)\\
(5,592,3)\\
(5,592,4)\\
(5,592,6)\\
(5,593,1)\\
(5,593,3)\\
(5,593,4)\\
(5,593,6)\\
(5,594,1)\\
(5,594,3)\\
(5,594,4)\\
(5,594,6)\\
(5,595,1)\\
(5,595,3)\\
(5,595,4)\\
(5,595,6)\\
(5,596,1)\\
(5,596,3)\\
(5,597,1)\\
(5,597,3)\\
(5,597,4)\\
(5,597,6)\\
(5,598,1)\\
(5,598,3)\\
(5,599,1)\\
(5,599,3)\\
(5,599,4)\\
(5,599,6)\\
(5,600,1)\\
(5,600,3)\\
(5,600,4)\\
(5,600,6)\\
(5,601,1)\\
(5,602,1)\\
(5,602,3)\\
(5,603,1)
\end{tabular}
\begin{tabular}{l}
CARAT \\\hline
(5,603,3)\\
(5,603,4)\\
(5,603,6)\\
(5,604,1)\\
(5,604,5)\\
(5,605,1)\\
(5,606,1)\\
(5,606,5)\\
(5,607,1)\\
(5,607,5)\\
(5,608,1)\\
(5,608,5)\\
(5,608,7)\\
(5,608,10)\\
(5,609,1)\\
(5,609,3)\\
(5,610,1)\\
(5,610,3)\\
(5,611,1)\\
(5,611,3)\\
(5,612,1)\\
(5,612,3)\\
(5,612,4)\\
(5,612,6)\\
(5,613,1)\\
(5,613,3)\\
(5,613,4)\\
(5,613,6)\\
(5,614,1)\\
(5,614,3)\\
(5,615,1)\\
(5,616,1)\\
(5,616,3)\\
(5,617,1)\\
(5,617,3)\\
(5,617,4)\\
(5,617,6)\\
(5,618,1)\\
(5,618,3)\\
(5,618,4)\\
(5,618,6)\\
(5,619,1)\\
(5,619,3)\\
(5,620,1)\\
(5,621,1)\\
(5,621,5)\\
(5,622,1)\\
(5,622,5)\\
(5,623,1)\\
(5,623,5)\\
(5,623,7)\\
(5,623,10)\\
(5,624,1)\\
(5,624,5)\\
(5,624,7)\\
(5,624,10)\\
(5,625,1)\\
(5,626,1)\\
(5,626,3)\\
(5,627,1)\\
(5,627,3)\\
(5,628,1)\\
(5,628,3)\\
(5,629,1)\\
(5,630,1)\\
(5,630,3)\\
(5,631,1)\\
(5,631,3)\\
(5,632,1)\\
(5,633,1)\\
(5,634,1)\\
(5,634,3)\\
(5,635,1)\\
(5,635,3)\\
(5,635,4)\\
(5,635,6)\\
(5,636,1)\\
(5,636,3)\\
(5,636,4)\\
(5,636,6)\\
(5,637,1)\\
(5,637,3)\\
(5,638,1)\\
(5,638,5)\\
\\
\end{tabular}
}
\end{center}
\newpage
\begin{center}
Table $14$: $L(M)^G$ not retract $k$-rational,\\ 
$M=M_1\oplus M_2$, rank $M_1=4$, 
rank $M_2=1$, $M_i$ : indecomposable ($768$ cases)\vspace*{2mm}\\
{\tiny
\begin{tabular}{l}
CARAT \\\hline
(5,17,6)\\
(5,17,10)\\
(5,17,13)\\
(5,17,15)\\
(5,18,18)\\
(5,18,21)\\
(5,19,10)\\
(5,20,6)\\
(5,20,10)\\
(5,20,13)\\
(5,20,15)\\
(5,21,6)\\
(5,21,10)\\
(5,21,15)\\
(5,22,10)\\
(5,22,15)\\
(5,23,13)\\
(5,23,21)\\
(5,23,24)\\
(5,24,9)\\
(5,24,18)\\
(5,24,21)\\
(5,24,24)\\
(5,25,6)\\
(5,25,7)\\
(5,25,12)\\
(5,25,14)\\
(5,25,17)\\
(5,25,22)\\
(5,25,23)\\
(5,26,6)\\
(5,26,9)\\
(5,26,16)\\
(5,26,19)\\
(5,26,24)\\
(5,26,27)\\
(5,26,31)\\
(5,26,34)\\
(5,27,4)\\
(5,27,8)\\
(5,27,12)\\
(5,28,4)\\
(5,28,11)\\
(5,28,17)\\
(5,28,22)\\
(5,28,23)\\
(5,29,4)\\
(5,29,8)\\
(5,29,12)\\
(5,30,7)\\
(5,30,14)\\
(5,30,17)\\
(5,30,22)\\
(5,30,25)\\
(5,31,8)\\
(5,31,9)\\
(5,31,16)\\
(5,31,22)\\
(5,31,25)\\
(5,31,29)\\
(5,31,34)\\
(5,31,37)\\
(5,32,12)\\
(5,32,16)\\
(5,32,21)\\
(5,32,23)\\
(5,32,30)\\
(5,32,42)\\
(5,32,53)\\
(5,47,5)\\
(5,47,6)\\
(5,48,5)\\
(5,48,6)\\
(5,48,12)\\
(5,48,13)\\
(5,49,5)\\
(5,49,6)
\end{tabular}
\begin{tabular}{l}
CARAT \\\hline
(5,49,13)\\
(5,50,5)\\
(5,50,6)\\
(5,51,6)\\
(5,52,5)\\
(5,52,6)\\
(5,52,13)\\
(5,53,5)\\
(5,53,6)\\
(5,54,6)\\
(5,55,5)\\
(5,55,6)\\
(5,56,5)\\
(5,56,6)\\
(5,59,6)\\
(5,60,6)\\
(5,60,13)\\
(5,61,6)\\
(5,62,14)\\
(5,63,6)\\
(5,63,12)\\
(5,63,13)\\
(5,64,6)\\
(5,65,6)\\
(5,65,12)\\
(5,65,13)\\
(5,66,5)\\
(5,66,6)\\
(5,67,6)\\
(5,68,6)\\
(5,69,6)\\
(5,70,6)\\
(5,71,5)\\
(5,71,8)\\
(5,71,10)\\
(5,71,13)\\
(5,71,14)\\
(5,71,17)\\
(5,72,7)\\
(5,72,10)\\
(5,72,11)\\
(5,72,24)\\
(5,72,26)\\
(5,72,29)\\
(5,72,30)\\
(5,73,7)\\
(5,73,10)\\
(5,73,11)\\
(5,73,24)\\
(5,73,26)\\
(5,73,29)\\
(5,73,30)\\
(5,74,7)\\
(5,74,10)\\
(5,74,11)\\
(5,75,7)\\
(5,75,10)\\
(5,75,11)\\
(5,75,26)\\
(5,75,29)\\
(5,75,30)\\
(5,76,8)\\
(5,76,13)\\
(5,76,14)\\
(5,76,17)\\
(5,76,31)\\
(5,76,34)\\
(5,76,36)\\
(5,76,39)\\
(5,76,40)\\
(5,76,43)\\
(5,77,8)\\
(5,77,13)\\
(5,77,14)\\
(5,77,17)\\
(5,78,5)\\
(5,78,7)
\end{tabular}
\begin{tabular}{l}
CARAT \\\hline
(5,78,10)\\
(5,78,13)\\
(5,78,26)\\
(5,78,29)\\
(5,78,32)\\
(5,79,5)\\
(5,79,7)\\
(5,79,10)\\
(5,79,13)\\
(5,79,24)\\
(5,79,26)\\
(5,79,29)\\
(5,79,32)\\
(5,80,5)\\
(5,80,7)\\
(5,80,10)\\
(5,80,13)\\
(5,81,10)\\
(5,81,13)\\
(5,81,14)\\
(5,81,17)\\
(5,81,39)\\
(5,81,43)\\
(5,82,10)\\
(5,82,13)\\
(5,82,14)\\
(5,82,17)\\
(5,83,5)\\
(5,83,7)\\
(5,83,10)\\
(5,83,11)\\
(5,84,7)\\
(5,84,10)\\
(5,84,13)\\
(5,85,7)\\
(5,85,10)\\
(5,85,13)\\
(5,86,4)\\
(5,86,7)\\
(5,87,4)\\
(5,87,7)\\
(5,88,5)\\
(5,88,7)\\
(5,88,10)\\
(5,88,11)\\
(5,88,24)\\
(5,88,26)\\
(5,88,29)\\
(5,88,30)\\
(5,89,5)\\
(5,89,7)\\
(5,89,10)\\
(5,89,11)\\
(5,90,4)\\
(5,90,10)\\
(5,90,11)\\
(5,90,23)\\
(5,90,26)\\
(5,90,29)\\
(5,90,30)\\
(5,91,5)\\
(5,91,7)\\
(5,91,10)\\
(5,91,13)\\
(5,92,4)\\
(5,92,10)\\
(5,92,11)\\
(5,93,4)\\
(5,93,7)\\
(5,94,4)\\
(5,94,7)\\
(5,94,16)\\
(5,94,19)\\
(5,95,4)\\
(5,95,7)\\
(5,96,4)\\
(5,96,7)
\end{tabular}
\begin{tabular}{l}
CARAT \\\hline
(5,97,4)\\
(5,97,7)\\
(5,98,14)\\
(5,98,17)\\
(5,99,5)\\
(5,99,13)\\
(5,99,14)\\
(5,99,17)\\
(5,99,42)\\
(5,99,46)\\
(5,100,5)\\
(5,100,13)\\
(5,100,14)\\
(5,100,17)\\
(5,101,4)\\
(5,101,10)\\
(5,101,11)\\
(5,102,4)\\
(5,102,10)\\
(5,102,11)\\
(5,103,14)\\
(5,103,17)\\
(5,104,3)\\
(5,105,3)\\
(5,105,9)\\
(5,106,3)\\
(5,107,3)\\
(5,108,3)\\
(5,108,9)\\
(5,109,3)\\
(5,109,9)\\
(5,110,3)\\
(5,110,9)\\
(5,111,3)\\
(5,111,9)\\
(5,112,3)\\
(5,113,3)\\
(5,114,3)\\
(5,115,3)\\
(5,116,3)\\
(5,116,12)\\
(5,116,18)\\
(5,116,24)\\
(5,117,3)\\
(5,117,11)\\
(5,117,17)\\
(5,118,3)\\
(5,118,11)\\
(5,118,17)\\
(5,119,3)\\
(5,119,8)\\
(5,120,3)\\
(5,120,11)\\
(5,120,17)\\
(5,121,3)\\
(5,121,11)\\
(5,122,3)\\
(5,122,14)\\
(5,123,3)\\
(5,123,12)\\
(5,124,3)\\
(5,125,3)\\
(5,126,3)\\
(5,127,3)\\
(5,127,9)\\
(5,128,3)\\
(5,128,9)\\
(5,129,3)\\
(5,130,3)\\
(5,131,3)\\
(5,131,8)\\
(5,131,14)\\
(5,132,3)\\
(5,132,8)\\
(5,132,14)\\
(5,133,3)\\
(5,133,11)
\end{tabular}
\begin{tabular}{l}
CARAT \\\hline
(5,133,17)\\
(5,134,3)\\
(5,134,8)\\
(5,135,3)\\
(5,136,3)\\
(5,136,11)\\
(5,136,17)\\
(5,137,3)\\
(5,138,3)\\
(5,138,9)\\
(5,139,3)\\
(5,139,9)\\
(5,140,3)\\
(5,140,9)\\
(5,140,15)\\
(5,140,21)\\
(5,141,3)\\
(5,141,9)\\
(5,141,15)\\
(5,141,21)\\
(5,142,3)\\
(5,142,12)\\
(5,142,18)\\
(5,142,24)\\
(5,143,3)\\
(5,143,9)\\
(5,143,15)\\
(5,143,21)\\
(5,144,3)\\
(5,144,9)\\
(5,145,3)\\
(5,145,9)\\
(5,146,3)\\
(5,146,9)\\
(5,147,3)\\
(5,148,3)\\
(5,148,8)\\
(5,149,3)\\
(5,149,11)\\
(5,149,17)\\
(5,150,3)\\
(5,151,3)\\
(5,151,8)\\
(5,151,14)\\
(5,152,3)\\
(5,153,3)\\
(5,153,9)\\
(5,154,3)\\
(5,154,8)\\
(5,154,14)\\
(5,155,3)\\
(5,156,3)\\
(5,157,3)\\
(5,157,9)\\
(5,158,3)\\
(5,158,9)\\
(5,159,3)\\
(5,160,3)\\
(5,161,3)\\
(5,162,3)\\
(5,395,2)\\
(5,404,3)\\
(5,407,4)\\
(5,408,4)\\
(5,409,3)\\
(5,409,9)\\
(5,410,5)\\
(5,410,9)\\
(5,411,2)\\
(5,412,2)\\
(5,413,3)\\
(5,414,5)\\
(5,423,4)\\
(5,424,5)\\
(5,424,9)\\
(5,425,2)\\
(5,426,3)
\end{tabular}
\begin{tabular}{l}
CARAT \\\hline
(5,430,2)\\
(5,431,2)\\
(5,432,2)\\
(5,433,2)\\
(5,434,3)\\
(5,435,3)\\
(5,436,3)\\
(5,444,6)\\
(5,445,6)\\
(5,446,6)\\
(5,447,6)\\
(5,448,2)\\
(5,449,2)\\
(5,450,2)\\
(5,450,5)\\
(5,457,2)\\
(5,458,2)\\
(5,459,2)\\
(5,464,2)\\
(5,465,5)\\
(5,468,7)\\
(5,469,2)\\
(5,484,2)\\
(5,485,2)\\
(5,486,2)\\
(5,487,2)\\
(5,488,2)\\
(5,489,2)\\
(5,490,2)\\
(5,501,3)\\
(5,503,3)\\
(5,503,5)\\
(5,504,5)\\
(5,505,3)\\
(5,505,5)\\
(5,507,5)\\
(5,508,5)\\
(5,509,5)\\
(5,510,3)\\
(5,510,5)\\
(5,511,3)\\
(5,511,5)\\
(5,512,5)\\
(5,513,3)\\
(5,513,5)\\
(5,514,3)\\
(5,514,5)\\
(5,515,3)\\
(5,515,5)\\
(5,516,5)\\
(5,517,3)\\
(5,517,5)\\
(5,518,4)\\
(5,518,9)\\
(5,519,2)\\
(5,519,3)\\
(5,519,11)\\
(5,520,3)\\
(5,520,12)\\
(5,521,3)\\
(5,522,2)\\
(5,522,3)\\
(5,523,4)\\
(5,523,5)\\
(5,523,10)\\
(5,523,11)\\
(5,524,4)\\
(5,524,10)\\
(5,524,11)\\
(5,525,4)\\
(5,525,9)\\
(5,526,4)\\
(5,526,9)\\
(5,527,4)\\
(5,527,5)\\
(5,527,10)\\
(5,527,11)\\
\end{tabular}
\begin{tabular}{l}
CARAT \\\hline
(5,528,4)\\
(5,528,10)\\
(5,528,11)\\
(5,529,4)\\
(5,529,10)\\
(5,529,11)\\
(5,530,5)\\
(5,530,10)\\
(5,530,11)\\
(5,531,5)\\
(5,531,11)\\
(5,532,4)\\
(5,532,9)\\
(5,533,10)\\
(5,533,11)\\
(5,534,4)\\
(5,534,9)\\
(5,535,4)\\
(5,535,10)\\
(5,535,11)\\
(5,536,4)\\
(5,536,10)\\
(5,536,11)\\
(5,537,4)\\
(5,537,10)\\
(5,537,11)\\
(5,538,5)\\
(5,538,10)\\
(5,538,11)\\
(5,539,4)\\
(5,539,9)\\
(5,540,4)\\
(5,540,5)\\
(5,540,10)\\
(5,540,11)\\
(5,541,5)\\
(5,541,10)\\
(5,541,11)\\
(5,542,4)\\
(5,542,5)\\
(5,542,10)\\
(5,542,11)\\
(5,543,4)\\
(5,543,9)\\
(5,544,4)\\
(5,544,9)\\
(5,545,4)\\
(5,545,9)\\
(5,546,4)\\
(5,546,9)\\
(5,547,4)\\
(5,547,9)\\
(5,548,4)\\
(5,548,10)\\
(5,548,11)\\
(5,646,1)\\
(5,646,2)\\
(5,647,1)\\
(5,647,2)\\
(5,648,1)\\
(5,649,1)\\
(5,650,1)\\
(5,651,1)\\
(5,651,5)\\
(5,652,1)\\
(5,653,1)\\
(5,653,5)\\
(5,654,1)\\
(5,655,1)\\
(5,656,1)\\
(5,657,1)\\
(5,657,5)\\
(5,658,1)\\
(5,659,1)\\
(5,659,5)\\
(5,660,1)\\
(5,661,1)
\end{tabular}
\begin{tabular}{l}
CARAT \\\hline
(5,662,1)\\
(5,663,1)\\
(5,664,1)\\
(5,665,3)\\
(5,666,3)\\
(5,667,3)\\
(5,668,1)\\
(5,669,1)\\
(5,670,1)\\
(5,671,1)\\
(5,672,1)\\
(5,673,1)\\
(5,674,1)\\
(5,675,1)\\
(5,676,1)\\
(5,677,1)\\
(5,678,1)\\
(5,679,1)\\
(5,680,1)\\
(5,681,1)\\
(5,682,1)\\
(5,683,1)\\
(5,683,5)\\
(5,684,1)\\
(5,684,2)\\
(5,685,1)\\
(5,685,2)\\
(5,686,1)\\
(5,686,2)\\
(5,687,1)\\
(5,687,2)\\
(5,688,1)\\
(5,688,5)\\
(5,689,1)\\
(5,689,5)\\
(5,690,1)\\
(5,691,1)\\
(5,694,1)\\
(5,695,1)\\
(5,696,1)\\
(5,696,2)\\
(5,697,1)\\
(5,698,1)\\
(5,699,1)\\
(5,700,1)\\
(5,701,1)\\
(5,702,1)\\
(5,703,1)\\
(5,704,1)\\
(5,704,8)\\
(5,705,1)\\
(5,705,5)\\
(5,706,1)\\
(5,706,5)\\
(5,707,1)\\
(5,708,1)\\
(5,709,1)\\
(5,710,1)\\
(5,711,1)\\
(5,712,1)\\
(5,713,1)\\
(5,714,1)\\
(5,715,1)\\
(5,716,1)\\
(5,717,1)\\
(5,718,1)\\
(5,719,1)\\
(5,720,1)\\
(5,721,1)\\
(5,722,1)\\
(5,723,1)\\
(5,723,5)\\
(5,724,1)\\
(5,724,2)\\
(5,725,1)\\
(5,725,2)\\
(5,726,1)
\end{tabular}
\begin{tabular}{l}
CARAT \\\hline
(5,726,5)\\
(5,727,1)\\
(5,727,5)\\
(5,728,1)\\
(5,728,2)\\
(5,729,1)\\
(5,729,2)\\
(5,730,1)\\
(5,731,1)\\
(5,732,1)\\
(5,733,1)\\
(5,734,1)\\
(5,735,1)\\
(5,737,1)\\
(5,738,1)\\
(5,739,1)\\
(5,740,1)\\
(5,741,1)\\
(5,741,5)\\
(5,742,1)\\
(5,742,5)\\
(5,743,1)\\
(5,743,5)\\
(5,744,1)\\
(5,745,1)\\
(5,745,2)\\
(5,746,1)\\
(5,746,2)\\
(5,747,1)\\
(5,748,1)\\
(5,749,1)\\
(5,750,1)\\
(5,750,5)\\
(5,751,1)\\
(5,751,5)\\
(5,752,1)\\
(5,752,5)\\
(5,753,1)\\
(5,754,1)\\
(5,754,5)\\
(5,755,1)\\
(5,755,5)\\
(5,756,1)\\
(5,757,1)\\
(5,758,1)\\
(5,759,1)\\
(5,760,1)\\
(5,760,5)\\
(5,761,1)\\
(5,761,5)\\
(5,762,1)\\
(5,763,1)\\
(5,763,8)\\
(5,764,1)\\
(5,764,8)\\
(5,765,1)\\
(5,766,1)\\
(5,767,1)\\
(5,768,1)\\
(5,768,8)\\
(5,769,1)\\
(5,770,1)\\
(5,771,1)\\
(5,772,1)\\
(5,772,2)\\
(5,773,3)\\
(5,774,3)\\
(5,775,1)\\
(5,776,1)\\
(5,777,1)\\
(5,778,1)\\
(5,779,1)\\
(5,780,1)\\
(5,781,1)\\
(5,782,1)\\
(5,783,1)\\
(5,784,1)
\end{tabular}
\begin{tabular}{l}
CARAT \\\hline
(5,785,1)\\
(5,785,2)\\
(5,786,1)\\
(5,786,2)\\
(5,797,3)\\
(5,798,2)\\
(5,801,4)\\
(5,801,5)\\
(5,805,3)\\
(5,806,3)\\
(5,808,2)\\
(5,809,2)\\
(5,810,3)\\
(5,810,4)\\
(5,811,2)\\
(5,811,3)\\
(5,812,3)\\
(5,812,4)\\
(5,822,3)\\
(5,822,4)\\
(5,823,3)\\
(5,823,5)\\
(5,836,2)\\
(5,846,3)\\
(5,846,4)\\
(5,847,2)\\
(5,848,2)\\
(5,851,2)\\
(5,851,3)\\
(5,852,4)\\
(5,852,5)\\
(5,853,4)\\
(5,853,6)\\
(5,854,2)\\
(5,854,5)\\
(5,854,6)\\
(5,855,2)\\
(5,855,5)\\
(5,855,7)\\
(5,856,3)\\
(5,856,5)\\
(5,856,7)\\
(5,856,8)\\
(5,867,2)\\
(5,869,2)\\
(5,869,5)\\
(5,869,6)\\
(5,870,4)\\
(5,870,5)\\
(5,876,3)\\
(5,877,3)\\
(5,878,2)\\
(5,879,2)\\
(5,882,2)\\
(5,887,2)\\
(5,887,3)\\
(5,887,5)\\
(5,888,3)\\
(5,888,4)\\
(5,889,4)\\
(5,889,5)\\
(5,890,4)\\
(5,890,5)\\
(5,891,3)\\
(5,891,5)\\
(5,892,3)\\
(5,892,5)\\
(5,893,3)\\
(5,893,4)\\
(5,894,3)\\
(5,894,4)\\
(5,895,2)\\
(5,895,3)\\
(5,896,2)\\
(5,896,3)\\ \\ \\
\end{tabular}
}
\end{center}
\newpage
\begin{center}
Table $15$: $L(M)^G$ not retract $k$-rational, 
rank $M=5$, $M$ : indecomposable ($1141$ cases)\vspace*{2mm}\\
{\tiny
\begin{tabular}{l}
CARAT \\\hline
(5,17,14)\\
(5,17,17)\\
(5,18,23)\\
(5,19,17)\\
(5,20,14)\\
(5,20,17)\\
(5,21,14)\\
(5,21,17)\\
(5,22,17)\\
(5,23,23)\\
(5,24,23)\\
(5,24,27)\\
(5,25,8)\\
(5,25,16)\\
(5,25,19)\\
(5,25,20)\\
(5,25,21)\\
(5,25,25)\\
(5,25,27)\\
(5,25,28)\\
(5,25,29)\\
(5,25,30)\\
(5,25,32)\\
(5,25,33)\\
(5,26,10)\\
(5,26,18)\\
(5,26,21)\\
(5,26,26)\\
(5,26,28)\\
(5,26,29)\\
(5,26,30)\\
(5,26,33)\\
(5,26,36)\\
(5,26,38)\\
(5,26,39)\\
(5,26,40)\\
(5,26,41)\\
(5,26,42)\\
(5,27,5)\\
(5,27,10)\\
(5,27,11)\\
(5,27,14)\\
(5,27,15)\\
(5,27,16)\\
(5,28,19)\\
(5,28,20)\\
(5,28,21)\\
(5,28,25)\\
(5,28,28)\\
(5,28,29)\\
(5,28,30)\\
(5,29,5)\\
(5,29,10)\\
(5,29,11)\\
(5,29,14)\\
(5,29,15)\\
(5,29,16)\\
(5,30,19)\\
(5,30,20)\\
(5,30,21)
\end{tabular}
\begin{tabular}{l}
CARAT \\\hline
(5,30,24)\\
(5,30,27)\\
(5,30,28)\\
(5,30,30)\\
(5,30,32)\\
(5,31,10)\\
(5,31,18)\\
(5,31,24)\\
(5,31,26)\\
(5,31,27)\\
(5,31,31)\\
(5,31,32)\\
(5,31,33)\\
(5,31,36)\\
(5,31,39)\\
(5,31,40)\\
(5,31,43)\\
(5,31,44)\\
(5,31,45)\\
(5,31,46)\\
(5,31,48)\\
(5,32,18)\\
(5,32,25)\\
(5,32,34)\\
(5,32,35)\\
(5,32,36)\\
(5,32,41)\\
(5,32,44)\\
(5,32,51)\\
(5,32,54)\\
(5,32,57)\\
(5,32,60)\\
(5,36,3)\\
(5,36,7)\\
(5,38,6)\\
(5,39,5)\\
(5,71,12)\\
(5,71,16)\\
(5,71,19)\\
(5,71,22)\\
(5,71,23)\\
(5,71,24)\\
(5,71,25)\\
(5,71,26)\\
(5,71,28)\\
(5,71,29)\\
(5,72,9)\\
(5,72,13)\\
(5,72,15)\\
(5,72,17)\\
(5,72,18)\\
(5,72,19)\\
(5,72,28)\\
(5,72,32)\\
(5,72,34)\\
(5,72,36)\\
(5,72,37)\\
(5,72,38)\\
(5,72,40)\\
(5,72,41)
\end{tabular}
\begin{tabular}{l}
CARAT \\\hline
(5,73,9)\\
(5,73,13)\\
(5,73,15)\\
(5,73,17)\\
(5,73,18)\\
(5,73,19)\\
(5,73,28)\\
(5,73,32)\\
(5,73,34)\\
(5,73,36)\\
(5,73,37)\\
(5,73,38)\\
(5,74,13)\\
(5,74,15)\\
(5,74,17)\\
(5,74,19)\\
(5,75,13)\\
(5,75,15)\\
(5,75,17)\\
(5,75,19)\\
(5,75,32)\\
(5,75,34)\\
(5,75,36)\\
(5,75,38)\\
(5,75,40)\\
(5,76,12)\\
(5,76,16)\\
(5,76,19)\\
(5,76,22)\\
(5,76,23)\\
(5,76,24)\\
(5,76,25)\\
(5,76,26)\\
(5,76,38)\\
(5,76,42)\\
(5,76,45)\\
(5,76,48)\\
(5,76,49)\\
(5,76,50)\\
(5,76,51)\\
(5,76,53)\\
(5,76,54)\\
(5,76,55)\\
(5,77,12)\\
(5,77,16)\\
(5,77,19)\\
(5,77,22)\\
(5,77,23)\\
(5,77,24)\\
(5,77,25)\\
(5,77,26)\\
(5,78,9)\\
(5,78,12)\\
(5,78,15)\\
(5,78,17)\\
(5,78,18)\\
(5,78,19)\\
(5,78,28)\\
(5,78,31)\\
(5,78,34)
\end{tabular}
\begin{tabular}{l}
CARAT \\\hline
(5,78,36)\\
(5,78,38)\\
(5,78,40)\\
(5,79,9)\\
(5,79,12)\\
(5,79,15)\\
(5,79,17)\\
(5,79,18)\\
(5,79,19)\\
(5,79,28)\\
(5,79,31)\\
(5,79,34)\\
(5,79,36)\\
(5,79,37)\\
(5,79,38)\\
(5,79,40)\\
(5,79,41)\\
(5,80,9)\\
(5,80,12)\\
(5,80,15)\\
(5,80,17)\\
(5,80,18)\\
(5,80,19)\\
(5,81,16)\\
(5,81,19)\\
(5,81,22)\\
(5,81,23)\\
(5,81,24)\\
(5,81,26)\\
(5,81,45)\\
(5,81,52)\\
(5,81,55)\\
(5,82,16)\\
(5,82,19)\\
(5,82,22)\\
(5,82,23)\\
(5,82,24)\\
(5,82,26)\\
(5,83,9)\\
(5,83,13)\\
(5,83,15)\\
(5,83,17)\\
(5,83,18)\\
(5,83,19)\\
(5,83,21)\\
(5,83,22)\\
(5,84,12)\\
(5,84,15)\\
(5,84,17)\\
(5,84,19)\\
(5,85,9)\\
(5,85,12)\\
(5,85,15)\\
(5,85,17)\\
(5,85,18)\\
(5,85,19)\\
(5,86,6)\\
(5,86,9)\\
(5,86,11)\\
(5,86,12)
\end{tabular}
\begin{tabular}{l}
CARAT \\\hline
(5,87,6)\\
(5,87,9)\\
(5,87,11)\\
(5,87,12)\\
(5,88,9)\\
(5,88,13)\\
(5,88,15)\\
(5,88,17)\\
(5,88,18)\\
(5,88,19)\\
(5,88,28)\\
(5,88,32)\\
(5,88,34)\\
(5,88,36)\\
(5,88,37)\\
(5,88,38)\\
(5,88,40)\\
(5,88,41)\\
(5,89,9)\\
(5,89,13)\\
(5,89,15)\\
(5,89,17)\\
(5,89,18)\\
(5,89,19)\\
(5,90,9)\\
(5,90,13)\\
(5,90,16)\\
(5,90,17)\\
(5,90,18)\\
(5,90,19)\\
(5,90,28)\\
(5,90,32)\\
(5,90,35)\\
(5,90,36)\\
(5,90,37)\\
(5,90,38)\\
(5,91,9)\\
(5,91,12)\\
(5,91,15)\\
(5,91,17)\\
(5,91,18)\\
(5,91,19)\\
(5,92,13)\\
(5,92,17)\\
(5,92,18)\\
(5,92,19)\\
(5,93,6)\\
(5,93,9)\\
(5,93,11)\\
(5,93,12)\\
(5,94,6)\\
(5,94,9)\\
(5,94,11)\\
(5,94,12)\\
(5,94,18)\\
(5,94,21)\\
(5,94,23)\\
(5,94,24)\\
(5,95,6)\\
(5,95,9)
\end{tabular}
\begin{tabular}{l}
CARAT \\\hline
(5,95,11)\\
(5,95,12)\\
(5,96,6)\\
(5,96,9)\\
(5,96,11)\\
(5,96,12)\\
(5,97,6)\\
(5,97,9)\\
(5,97,11)\\
(5,97,12)\\
(5,98,23)\\
(5,98,27)\\
(5,98,29)\\
(5,99,16)\\
(5,99,19)\\
(5,99,23)\\
(5,99,24)\\
(5,99,25)\\
(5,99,27)\\
(5,99,28)\\
(5,99,48)\\
(5,99,55)\\
(5,99,58)\\
(5,100,12)\\
(5,100,16)\\
(5,100,19)\\
(5,100,23)\\
(5,100,24)\\
(5,100,25)\\
(5,100,27)\\
(5,100,28)\\
(5,101,13)\\
(5,101,17)\\
(5,101,18)\\
(5,101,20)\\
(5,101,21)\\
(5,102,9)\\
(5,102,13)\\
(5,102,17)\\
(5,102,18)\\
(5,102,20)\\
(5,102,21)\\
(5,103,23)\\
(5,103,27)\\
(5,103,29)\\
(5,104,4)\\
(5,104,5)\\
(5,105,5)\\
(5,105,6)\\
(5,105,11)\\
(5,105,12)\\
(5,106,6)\\
(5,107,5)\\
(5,107,6)\\
(5,108,6)\\
(5,108,12)\\
(5,108,14)\\
(5,109,5)\\
(5,109,6)\\
(5,109,11)
\end{tabular}
\begin{tabular}{l}
CARAT \\\hline
(5,109,12)\\
(5,109,14)\\
(5,109,15)\\
(5,110,5)\\
(5,110,6)\\
(5,110,12)\\
(5,110,14)\\
(5,111,5)\\
(5,111,6)\\
(5,111,11)\\
(5,111,12)\\
(5,111,14)\\
(5,111,15)\\
(5,112,4)\\
(5,112,5)\\
(5,112,7)\\
(5,112,8)\\
(5,113,4)\\
(5,113,5)\\
(5,114,5)\\
(5,115,5)\\
(5,115,6)\\
(5,116,5)\\
(5,116,6)\\
(5,116,8)\\
(5,116,9)\\
(5,116,14)\\
(5,116,15)\\
(5,116,20)\\
(5,116,21)\\
(5,116,26)\\
(5,116,27)\\
(5,117,5)\\
(5,117,7)\\
(5,117,14)\\
(5,117,19)\\
(5,118,5)\\
(5,118,7)\\
(5,118,13)\\
(5,118,14)\\
(5,118,18)\\
(5,118,19)\\
(5,119,4)\\
(5,119,5)\\
(5,119,9)\\
(5,119,10)\\
(5,119,12)\\
(5,119,13)\\
(5,120,4)\\
(5,120,5)\\
(5,120,7)\\
(5,120,8)\\
(5,120,13)\\
(5,120,14)\\
(5,120,19)\\
(5,120,21)\\
(5,121,5)\\
(5,121,7)\\
(5,121,13)\\
(5,122,4)
\end{tabular}
\begin{tabular}{l}
CARAT \\\hline
(5,122,5)\\
(5,122,8)\\
(5,122,9)\\
(5,122,10)\\
(5,122,15)\\
(5,122,16)\\
(5,123,5)\\
(5,123,6)\\
(5,123,8)\\
(5,123,9)\\
(5,123,14)\\
(5,123,15)\\
(5,124,5)\\
(5,124,6)\\
(5,125,5)\\
(5,125,6)\\
(5,126,6)\\
(5,127,5)\\
(5,127,6)\\
(5,127,11)\\
(5,127,12)\\
(5,128,5)\\
(5,128,6)\\
(5,128,11)\\
(5,128,12)\\
(5,129,5)\\
(5,129,6)\\
(5,130,5)\\
(5,130,6)\\
(5,131,4)\\
(5,131,5)\\
(5,131,10)\\
(5,131,11)\\
(5,131,16)\\
(5,131,18)\\
(5,132,4)\\
(5,132,5)\\
(5,132,10)\\
(5,132,11)\\
(5,132,15)\\
(5,132,16)\\
(5,132,18)\\
(5,132,19)\\
(5,133,5)\\
(5,133,7)\\
(5,133,14)\\
(5,133,19)\\
(5,134,4)\\
(5,134,5)\\
(5,134,9)\\
(5,134,10)\\
(5,135,4)\\
(5,135,5)\\
(5,136,4)\\
(5,136,5)\\
(5,136,7)\\
(5,136,8)\\
(5,136,13)\\
(5,136,14)\\
(5,136,18)
\end{tabular}
\begin{tabular}{l}
CARAT \\\hline
(5,136,19)\\
(5,137,5)\\
(5,138,6)\\
(5,138,12)\\
(5,139,5)\\
(5,139,6)\\
(5,139,11)\\
(5,139,12)\\
(5,140,5)\\
(5,140,6)\\
(5,140,11)\\
(5,140,12)\\
(5,140,17)\\
(5,140,18)\\
(5,140,23)\\
(5,140,24)\\
(5,141,5)\\
(5,141,6)\\
(5,141,11)\\
(5,141,12)\\
(5,141,17)\\
(5,141,18)\\
(5,141,23)\\
(5,141,24)\\
(5,141,26)\\
(5,141,27)\\
(5,142,5)\\
(5,142,6)\\
(5,142,8)\\
(5,142,9)\\
(5,142,14)\\
(5,142,15)\\
(5,142,20)\\
(5,142,21)\\
(5,142,26)\\
(5,142,27)\\
(5,143,5)\\
(5,143,6)\\
(5,143,11)\\
(5,143,12)\\
(5,143,17)\\
(5,143,18)\\
(5,143,23)\\
(5,143,24)\\
(5,144,5)\\
(5,144,6)\\
(5,144,11)\\
(5,144,12)\\
(5,145,5)\\
(5,145,6)\\
(5,145,11)\\
(5,145,12)\\
(5,146,5)\\
(5,146,6)\\
(5,146,11)\\
(5,146,12)\\
(5,147,4)\\
(5,147,5)\\
(5,148,4)\\
(5,148,5)
\end{tabular}
\begin{tabular}{l}
CARAT \\\hline
(5,148,9)\\
(5,148,10)\\
(5,148,12)\\
(5,148,13)\\
(5,149,4)\\
(5,149,5)\\
(5,149,7)\\
(5,149,8)\\
(5,149,13)\\
(5,149,14)\\
(5,149,18)\\
(5,149,19)\\
(5,150,5)\\
(5,150,6)\\
(5,151,4)\\
(5,151,5)\\
(5,151,10)\\
(5,151,11)\\
(5,151,15)\\
(5,151,16)\\
(5,151,18)\\
(5,151,19)\\
(5,152,4)\\
(5,152,5)\\
(5,153,5)\\
(5,153,6)\\
(5,153,11)\\
(5,153,12)\\
(5,154,4)\\
(5,154,5)\\
(5,154,10)\\
(5,154,11)\\
(5,154,15)\\
(5,154,16)\\
(5,155,4)\\
(5,155,5)\\
(5,156,5)\\
(5,157,5)\\
(5,157,6)\\
(5,157,11)\\
(5,157,12)\\
(5,158,5)\\
(5,158,6)\\
(5,158,11)\\
(5,158,12)\\
(5,159,5)\\
(5,159,6)\\
(5,160,4)\\
(5,160,5)\\
(5,160,7)\\
(5,161,5)\\
(5,161,7)\\
(5,162,5)\\
(5,162,7)\\
(5,224,9)\\
(5,227,11)\\
(5,232,14)\\
(5,233,9)\\
(5,233,10)\\
(5,234,11)
\end{tabular}
}
\end{center}
\begin{center}
Table $15$ (continued): $L(M)^G$ not retract $k$-rational, 
rank $M=5$, $M$ : indecomposable ($1141$ cases)\vspace*{2mm}\\
{\tiny
\begin{tabular}{l}
CARAT \\\hline
(5,234,12)\\
(5,240,9)\\
(5,240,10)\\
(5,241,14)\\
(5,241,15)\\
(5,242,9)\\
(5,244,9)\\
(5,244,10)\\
(5,246,11)\\
(5,246,12)\\
(5,257,14)\\
(5,257,15)\\
(5,274,4)\\
(5,275,4)\\
(5,276,4)\\
(5,290,6)\\
(5,291,6)\\
(5,294,6)\\
(5,302,4)\\
(5,303,4)\\
(5,304,4)\\
(5,305,4)\\
(5,305,8)\\
(5,306,12)\\
(5,316,4)\\
(5,326,6)\\
(5,335,6)\\
(5,336,8)\\
(5,336,12)\\
(5,337,8)\\
(5,338,6)\\
(5,339,6)\\
(5,340,8)\\
(5,340,12)\\
(5,342,4)\\
(5,372,6)\\
(5,404,5)\\
(5,409,5)\\
(5,410,7)\\
(5,413,6)\\
(5,414,8)\\
(5,424,7)\\
(5,426,6)\\
(5,434,5)\\
(5,435,6)\\
(5,436,5)\\
(5,465,7)\\
(5,501,5)\\
(5,518,6)\\
(5,518,7)\\
(5,518,11)\\
(5,518,12)\\
(5,519,5)\\
(5,519,7)\\
(5,519,8)
\end{tabular}
\begin{tabular}{l}
CARAT \\\hline
(5,519,15)\\
(5,520,7)\\
(5,520,9)\\
(5,520,15)\\
(5,521,7)\\
(5,521,9)\\
(5,522,5)\\
(5,522,7)\\
(5,522,8)\\
(5,523,7)\\
(5,523,8)\\
(5,523,13)\\
(5,523,14)\\
(5,523,16)\\
(5,523,17)\\
(5,524,8)\\
(5,524,13)\\
(5,524,14)\\
(5,525,6)\\
(5,525,7)\\
(5,525,11)\\
(5,525,12)\\
(5,526,6)\\
(5,526,7)\\
(5,526,11)\\
(5,526,12)\\
(5,527,7)\\
(5,527,8)\\
(5,527,13)\\
(5,527,14)\\
(5,527,16)\\
(5,527,17)\\
(5,528,7)\\
(5,528,8)\\
(5,528,13)\\
(5,528,15)\\
(5,528,16)\\
(5,528,17)\\
(5,529,8)\\
(5,529,13)\\
(5,529,15)\\
(5,529,16)\\
(5,530,7)\\
(5,530,8)\\
(5,530,13)\\
(5,530,15)\\
(5,530,16)\\
(5,530,17)\\
(5,531,8)\\
(5,531,15)\\
(5,531,17)\\
(5,532,6)\\
(5,532,7)\\
(5,532,11)\\
(5,532,12)
\end{tabular}
\begin{tabular}{l}
CARAT \\\hline
(5,533,13)\\
(5,533,14)\\
(5,534,6)\\
(5,534,7)\\
(5,534,11)\\
(5,534,12)\\
(5,535,7)\\
(5,535,8)\\
(5,535,13)\\
(5,535,14)\\
(5,536,7)\\
(5,536,8)\\
(5,536,13)\\
(5,536,14)\\
(5,537,8)\\
(5,537,13)\\
(5,537,14)\\
(5,538,8)\\
(5,538,13)\\
(5,538,14)\\
(5,539,6)\\
(5,539,7)\\
(5,539,11)\\
(5,539,12)\\
(5,540,7)\\
(5,540,8)\\
(5,540,13)\\
(5,540,14)\\
(5,540,16)\\
(5,540,17)\\
(5,541,8)\\
(5,541,13)\\
(5,541,14)\\
(5,542,7)\\
(5,542,8)\\
(5,542,13)\\
(5,542,14)\\
(5,543,6)\\
(5,543,7)\\
(5,543,11)\\
(5,543,12)\\
(5,544,6)\\
(5,544,7)\\
(5,544,11)\\
(5,544,12)\\
(5,545,6)\\
(5,545,7)\\
(5,545,11)\\
(5,545,12)\\
(5,546,6)\\
(5,546,7)\\
(5,546,11)\\
(5,546,12)\\
(5,547,6)\\
(5,547,7)
\end{tabular}
\begin{tabular}{l}
CARAT \\\hline
(5,547,11)\\
(5,547,12)\\
(5,548,7)\\
(5,548,8)\\
(5,548,13)\\
(5,548,14)\\
(5,549,2)\\
(5,549,5)\\
(5,550,2)\\
(5,550,4)\\
(5,550,7)\\
(5,550,10)\\
(5,551,2)\\
(5,551,4)\\
(5,551,7)\\
(5,551,10)\\
(5,552,2)\\
(5,552,4)\\
(5,552,7)\\
(5,552,10)\\
(5,553,2)\\
(5,553,5)\\
(5,554,2)\\
(5,554,5)\\
(5,555,2)\\
(5,555,5)\\
(5,556,2)\\
(5,556,5)\\
(5,557,2)\\
(5,557,5)\\
(5,558,2)\\
(5,558,5)\\
(5,559,2)\\
(5,559,5)\\
(5,560,2)\\
(5,560,5)\\
(5,561,2)\\
(5,561,5)\\
(5,562,2)\\
(5,562,5)\\
(5,563,2)\\
(5,563,5)\\
(5,564,2)\\
(5,564,5)\\
(5,565,2)\\
(5,565,4)\\
(5,565,7)\\
(5,565,10)\\
(5,566,2)\\
(5,566,5)\\
(5,567,2)\\
(5,567,5)\\
(5,568,2)\\
(5,568,4)\\
(5,568,7)
\end{tabular}
\begin{tabular}{l}
CARAT \\\hline
(5,568,10)\\
(5,569,2)\\
(5,569,4)\\
(5,569,7)\\
(5,569,10)\\
(5,570,2)\\
(5,570,4)\\
(5,570,7)\\
(5,570,10)\\
(5,571,2)\\
(5,571,4)\\
(5,571,7)\\
(5,571,10)\\
(5,572,2)\\
(5,572,5)\\
(5,573,2)\\
(5,573,5)\\
(5,574,2)\\
(5,574,4)\\
(5,574,7)\\
(5,574,10)\\
(5,575,2)\\
(5,575,5)\\
(5,576,2)\\
(5,576,5)\\
(5,577,2)\\
(5,577,5)\\
(5,578,2)\\
(5,578,5)\\
(5,579,2)\\
(5,579,5)\\
(5,580,2)\\
(5,604,2)\\
(5,604,4)\\
(5,605,2)\\
(5,605,4)\\
(5,606,2)\\
(5,607,2)\\
(5,607,6)\\
(5,608,2)\\
(5,608,6)\\
(5,621,2)\\
(5,621,4)\\
(5,622,2)\\
(5,622,4)\\
(5,623,2)\\
(5,623,6)\\
(5,624,2)\\
(5,624,4)\\
(5,624,6)\\
(5,624,9)\\
(5,625,2)\\
(5,625,4)\\
(5,638,2)\\
(5,638,4)
\end{tabular}
\begin{tabular}{l}
CARAT \\\hline
(5,639,2)\\
(5,641,4)\\
(5,642,2)\\
(5,643,2)\\
(5,651,2)\\
(5,651,4)\\
(5,651,6)\\
(5,651,8)\\
(5,652,2)\\
(5,652,4)\\
(5,653,2)\\
(5,653,4)\\
(5,653,6)\\
(5,653,8)\\
(5,654,2)\\
(5,654,4)\\
(5,655,2)\\
(5,656,2)\\
(5,657,2)\\
(5,657,4)\\
(5,657,6)\\
(5,657,8)\\
(5,658,2)\\
(5,658,4)\\
(5,659,2)\\
(5,659,4)\\
(5,659,6)\\
(5,659,8)\\
(5,660,2)\\
(5,660,4)\\
(5,664,2)\\
(5,664,4)\\
(5,665,4)\\
(5,666,4)\\
(5,667,2)\\
(5,667,4)\\
(5,668,2)\\
(5,669,2)\\
(5,669,4)\\
(5,670,2)\\
(5,670,3)\\
(5,670,7)\\
(5,671,2)\\
(5,671,3)\\
(5,671,7)\\
(5,672,2)\\
(5,673,2)\\
(5,673,4)\\
(5,674,2)\\
(5,675,2)\\
(5,675,4)\\
(5,683,2)\\
(5,683,4)\\
(5,684,3)\\
(5,685,3)
\end{tabular}
\begin{tabular}{l}
CARAT \\\hline
(5,685,5)\\
(5,686,3)\\
(5,686,5)\\
(5,687,3)\\
(5,688,2)\\
(5,688,4)\\
(5,689,2)\\
(5,696,3)\\
(5,697,2)\\
(5,697,4)\\
(5,700,2)\\
(5,701,2)\\
(5,701,4)\\
(5,702,2)\\
(5,702,4)\\
(5,703,2)\\
(5,704,2)\\
(5,704,3)\\
(5,704,5)\\
(5,704,9)\\
(5,705,2)\\
(5,705,6)\\
(5,705,7)\\
(5,706,2)\\
(5,706,6)\\
(5,706,7)\\
(5,706,8)\\
(5,706,10)\\
(5,706,12)\\
(5,707,2)\\
(5,708,2)\\
(5,708,4)\\
(5,709,2)\\
(5,709,3)\\
(5,709,5)\\
(5,710,2)\\
(5,711,2)\\
(5,712,2)\\
(5,712,4)\\
(5,713,2)\\
(5,714,2)\\
(5,714,4)\\
(5,715,2)\\
(5,716,2)\\
(5,716,4)\\
(5,717,2)\\
(5,717,4)\\
(5,718,2)\\
(5,718,4)\\
(5,719,2)\\
(5,719,4)\\
(5,720,2)\\
(5,720,4)\\
(5,721,2)\\
(5,721,4)
\end{tabular}
\begin{tabular}{l}
CARAT \\\hline
(5,723,2)\\
(5,723,4)\\
(5,724,3)\\
(5,724,5)\\
(5,725,3)\\
(5,725,5)\\
(5,726,2)\\
(5,727,2)\\
(5,727,4)\\
(5,728,3)\\
(5,728,5)\\
(5,729,3)\\
(5,729,5)\\
(5,741,2)\\
(5,741,4)\\
(5,742,2)\\
(5,743,2)\\
(5,743,4)\\
(5,749,2)\\
(5,749,4)\\
(5,750,2)\\
(5,750,6)\\
(5,751,2)\\
(5,751,4)\\
(5,751,6)\\
(5,751,8)\\
(5,752,2)\\
(5,752,4)\\
(5,752,6)\\
(5,752,8)\\
(5,753,2)\\
(5,754,2)\\
(5,754,6)\\
(5,755,2)\\
(5,755,4)\\
(5,755,6)\\
(5,755,8)\\
(5,756,2)\\
(5,757,2)\\
(5,757,4)\\
(5,758,2)\\
(5,759,2)\\
(5,759,4)\\
(5,760,2)\\
(5,760,6)\\
(5,761,2)\\
(5,761,4)\\
(5,761,6)\\
(5,761,8)\\
(5,762,2)\\
(5,763,2)\\
(5,763,3)\\
(5,763,5)\\
(5,763,9)\\
(5,764,2)
\end{tabular}
\begin{tabular}{l}
CARAT \\\hline
(5,764,3)\\
(5,764,5)\\
(5,764,7)\\
(5,764,9)\\
(5,764,11)\\
(5,765,2)\\
(5,765,4)\\
(5,766,2)\\
(5,766,4)\\
(5,767,2)\\
(5,767,4)\\
(5,768,2)\\
(5,768,3)\\
(5,768,5)\\
(5,768,7)\\
(5,768,9)\\
(5,768,11)\\
(5,769,2)\\
(5,769,4)\\
(5,770,2)\\
(5,770,4)\\
(5,772,3)\\
(5,772,5)\\
(5,773,4)\\
(5,774,4)\\
(5,785,3)\\
(5,786,3)\\
(5,786,5)\\
(5,801,6)\\
(5,822,5)\\
(5,823,4)\\
(5,846,5)\\
(5,852,6)\\
(5,853,5)\\
(5,854,7)\\
(5,855,6)\\
(5,856,4)\\
(5,869,7)\\
(5,870,6)\\
(5,889,6)\\
(5,890,6)\\
(5,891,4)\\
(5,892,4)\\
(5,933,2)\\
(5,933,3)\\
(5,934,2)\\
(5,934,3)\\
(5,935,2)\\
(5,935,3)\\
(5,936,2)\\
(5,936,3)\\
(5,937,2)\\
(5,937,3)\\
(5,938,2)\\
(5,938,3)
\end{tabular}
\begin{tabular}{l}
CARAT \\\hline
(5,939,2)\\
(5,939,3)\\
(5,940,2)\\
(5,940,3)\\
(5,941,2)\\
(5,941,3)\\
(5,942,2)\\
(5,942,3)\\
(5,943,2)\\
(5,943,3)\\
(5,944,2)\\
(5,944,3)\\
(5,945,2)\\
(5,945,3)\\
(5,946,1)\\
(5,946,3)\\
(5,947,1)\\
(5,947,3)\\
(5,948,1)\\
(5,948,2)\\
(5,948,3)\\
(5,948,4)\\
(5,949,1)\\
(5,949,2)\\
(5,949,3)\\
(5,949,4)\\
(5,950,1)\\
(5,950,2)\\
(5,950,3)\\
(5,950,4)\\
(5,951,1)\\
(5,951,2)\\
(5,951,3)\\
(5,952,1)\\
(5,952,3)\\
(5,953,1)\\
(5,953,2)\\
(5,953,3)\\
(5,954,1)\\
(5,954,2)\\
(5,954,3)\\
(5,954,4)\\
(5,955,1)\\
(5,955,2)\\
(5,955,3)\\
(5,955,4)\\
\\ \\ \\ \\ \\ \\ \\ \\ \\
\end{tabular}
}
\end{center}
%
\newpage
\begin{center}
Table $16$: 
stably rational classification of 
algebraic $k$-tori of dimension $5$
\vspace*{2mm}\\
\fontsize{6pt}{8pt}\selectfont
\begin{tabular}{lllll}
CARAT & \# & $[s,r,u]$ & $G(n,i)$ & \\\hline
(5,1)&$1$&[1,0,0]&$G(1,1)$&$\{1\}$\\
(5,2)&$1$&[1,0,0]&$G(2,1)$&$C_2$\\
(5,3)&$2$&[2,0,0]&$G(2,1)$&$C_2$\\
(5,4)&$2$&[2,0,0]&$G(2,1)$&$C_2$\\
(5,5)&$2$&[2,0,0]&$G(4,2)$&$C_2^2$\\
(5,6)&$3$&[3,0,0]&$G(2,1)$&$C_2$\\
(5,7)&$3$&[3,0,0]&$G(2,1)$&$C_2$\\
\end{tabular}
\ \ 
\begin{tabular}{lllll}
CARAT & \# & $[s,r,u]$ & $G(n,i)$ & \\\hline
(5,81)&$55$&[34,0,21]&$G(16,11)$&$C_2\times D_4$\\
(5,82)&$26$&[12,0,14]&$G(16,11)$&$C_2\times D_4$\\
(5,83)&$22$&[4,0,18]&$G(16,10)$&$C_4\times C_2^2$\\
(5,84)&$19$&[10,0,9]&$G(16,10)$&$C_4\times C_2^2$\\
(5,85)&$19$&[6,0,13]&$G(16,10)$&$C_4\times C_2^2$\\
(5,86)&$12$&[3,0,9]&$G(16,10)$&$C_4\times C_2^2$\\
(5,87)&$12$&[2,0,10]&$G(16,10)$&$C_4\times C_2^2$\\
\end{tabular}
\end{center}

\begin{center}
Table $16$ (continued):  
stably rational classification of 
algebraic $k$-tori of dimension $5$
\vspace*{2mm}\\\
\fontsize{6pt}{8pt}\selectfont
\begin{tabular}{lllll}
CARAT & \# & $[s,r,u]$ & $G(n,i)$ & \\\hline
(5,161)&$8$&[5,0,3]&$G(8,2)$&$C_4\times C_2$\\
(5,162)&$8$&[5,0,3]&$G(8,2)$&$C_4\times C_2$\\
(5,163)&$1$&[1,0,0]&$G(12,5)$&$C_6\times C_2$\\
(5,164)&$2$&[2,0,0]&$G(3,1)$&$C_3$\\
(5,165)&$1$&[1,0,0]&$G(6,2)$&$C_6$\\
(5,166)&$1$&[1,0,0]&$G(6,2)$&$C_6$\\
(5,167)&$2$&[2,0,0]&$G(6,2)$&$C_6$\\
\end{tabular}
\ \ 
\begin{tabular}{lllll}
CARAT & \# & $[s,r,u]$ & $G(n,i)$ & \\\hline
(5,241)&$15$&[9,0,6]&$G(24,14)$&$C_2^2\times S_3$\\
(5,242)&$15$&[12,0,3]&$G(24,14)$&$C_2^2\times S_3$\\
(5,243)&$15$&[15,0,0]&$G(24,14)$&$C_2^2\times S_3$\\
(5,244)&$15$&[9,0,6]&$G(24,14)$&$C_2^2\times S_3$\\
(5,245)&$18$&[18,0,0]&$G(24,14)$&$C_2^2\times S_3$\\
(5,246)&$18$&[12,0,6]&$G(24,14)$&$C_2^2\times S_3$\\
(5,247)&$8$&[6,0,2]&$G(24,14)$&$C_2^2\times S_3$\\
\end{tabular}
\end{center}

\begin{center}
Table $16$ (continued):  
stably rational classification of 
algebraic $k$-tori of dimension $5$
\vspace*{2mm}\\\
\fontsize{6pt}{8pt}\selectfont
\begin{tabular}{lllll}
CARAT & \# & $[s,r,u]$ & $G(n,i)$ & \\\hline
(5,321)&$2$&[2,0,0]&$G(48,45)$&$C_6\times D_4$\\
(5,322)&$2$&[1,0,1]&$G(48,45)$&$C_6\times D_4$\\
(5,323)&$4$&[2,0,2]&$G(48,45)$&$C_6\times D_4$\\
(5,324)&$2$&[1,0,1]&$G(48,45)$&$C_6\times D_4$\\
(5,325)&$2$&[1,0,1]&$G(48,44)$&$C_{12}\times C_2^2$\\
(5,326)&$6$&[3,0,3]&$G(48,35)$&$C_2\times C_4\times S_3$\\
(5,327)&$2$&[1,0,1]&$G(48,42)$&$C_2^2\times (C_3\rtimes C_4)$\\
\end{tabular}
\ \ 
\begin{tabular}{lllll}
CARAT & \# & $[s,r,u]$ & $G(n,i)$ & \\\hline
(5,401)&$2$&[2,0,0]&$G(144,192)$&$C_2^2\times S_3^2$\\
(5,402)&$2$&[2,0,0]&$G(144,192)$&$C_2^2\times S_3^2$\\
(5,403)&$1$&[1,0,0]&$G(144,192)$&$C_2^2\times S_3^2$\\
(5,404)&$5$&[3,0,2]&$G(18,5)$&$C_6\times C_3$\\
(5,405)&$2$&[2,0,0]&$G(18,5)$&$C_6\times C_3$\\
(5,406)&$2$&[2,0,0]&$G(18,5)$&$C_6\times C_3$\\
(5,407)&$4$&[3,0,1]&$G(18,3)$&$C_3\times S_3$\\
\end{tabular}
\end{center}

\begin{center}
Table $16$ (continued):  
stably rational classification of 
algebraic $k$-tori of dimension $5$
\vspace*{2mm}\\\
\fontsize{6pt}{8pt}\selectfont
\begin{tabular}{lllll}
CARAT & \# & $[s,r,u]$ & $G(n,i)$ & \\\hline
(5,481)&$3$&[3,0,0]&$G(72,46)$&$C_2\times S_3^2$\\
(5,482)&$3$&[3,0,0]&$G(72,46)$&$C_2\times S_3^2$\\
(5,483)&$3$&[3,0,0]&$G(72,46)$&$C_2\times S_3^2$\\
(5,484)&$5$&[4,0,1]&$G(72,46)$&$C_2\times S_3^2$\\
(5,485)&$4$&[3,0,1]&$G(72,46)$&$C_2\times S_3^2$\\
(5,486)&$4$&[3,0,1]&$G(72,46)$&$C_2\times S_3^2$\\
(5,487)&$4$&[3,0,1]&$G(72,46)$&$C_2\times S_3^2$\\
\end{tabular}
\ \ 
\begin{tabular}{lllll}
CARAT & \# & $[s,r,u]$ & $G(n,i)$ & \\\hline
(5,561)&$5$&[1,0,4]&$G(48,31)$&$C_4\times A_4$\\
(5,562)&$5$&[1,0,4]&$G(48,30)$&$A_4\rtimes C_4$\\
(5,563)&$5$&[2,0,3]&$G(48,30)$&$A_4\rtimes C_4$\\
(5,564)&$5$&[1,0,4]&$G(96,197)$&$D_4\times A_4$\\
(5,565)&$10$&[2,0,8]&$G(96,197)$&$D_4\times A_4$\\
(5,566)&$5$&[2,0,3]&$G(96,197)$&$A_4\times D_4$\\
(5,567)&$5$&[1,0,4]&$G(96,196)$&$C_2\times C_4\times A_4$\\
\end{tabular}
\end{center}

\begin{center}
Table $16$ (continued):  
stably rational classification of 
algebraic $k$-tori of dimension $5$
\vspace*{2mm}\\\
\fontsize{6pt}{8pt}\selectfont
\hspace*{-7mm}
\begin{tabular}{lllll}
CARAT & \# & $[s,r,u]$ & $G(n,i)$ & \\\hline
(5,641)&$4$&[3,0,1]&$G(16,7)$&$D_8$\\
(5,642)&$2$&[1,0,1]&$G(16,7)$&$D_8$\\
(5,643)&$2$&[1,0,1]&$G(32,39)$&$C_2\times D_8$\\
(5,644)&$2$&[2,0,0]&$G(8,1)$&$C_8$\\
(5,645)&$2$&[2,0,0]&$G(8,1)$&$C_8$\\
(5,646)&$2$&[0,0,2]&$G(1152,157511)$&\\
(5,647)&$2$&[0,0,2]&$G(1152,157478)$&\\
\end{tabular}
\ \ 
\begin{tabular}{lllll}
CARAT & \# & $[s,r,u]$ & $G(n,i)$ & \\\hline
(5,721)&$4$&[1,0,3]&$G(32,40)$&$C_2\times QD_8$\\
(5,722)&$1$&[0,0,1]&$G(384,18101)$&
$C_2\times ((\SL(2,3)\rtimes C_4)\rtimes C_2)$\\
(5,723)&$5$&[1,0,4]&$G(384,5833)$&\\
(5,724)&$5$&[1,0,4]&$G(384,5602)$&\\
(5,725)&$5$&[1,0,4]&$G(384,5602)$&\\
(5,726)&$5$&[2,0,3]&$G(384,5602)$&\\
(5,727)&$5$&[1,0,4]&$G(384,20090)$&\\
\end{tabular}
\end{center}

\begin{center}
Table $16$ (continued):  
stably rational classification of 
algebraic $k$-tori of dimension $5$
\vspace*{2mm}\\\
\fontsize{6pt}{8pt}\selectfont
\begin{tabular}{lllll}
CARAT & \# & $[s,r,u]$ & $G(n,i)$ & \\\hline
(5,801)&$6$&[3,0,3]&$G(144,186)$&$C_2\times (S_3^2\rtimes C_2)$\\
(5,802)&$2$&[2,0,0]&$G(144,115)$&$(C_2\times (C_3^2\rtimes C_4))\rtimes C_2$\\
(5,803)&$2$&[2,0,0]&$G(144,115)$&$(C_2\times (C_3^2\rtimes C_4))\rtimes C_2$\\
(5,804)&$2$&[2,0,0]&$G(144,115)$&$(C_2\times (C_3^2\rtimes C_4))\rtimes C_2$\\
(5,805)&$3$&[2,0,1]&$G(144,186)$&$C_2\times (S_3^2\rtimes C_2)$\\
(5,806)&$3$&[2,0,1]&$G(144,186)$&$C_2\times (S_3^2\rtimes C_2)$\\
(5,807)&$2$&[2,0,0]&$G(144,115)$&$(C_2\times (C_3^2\rtimes C_4))\rtimes C_2$\\
\end{tabular}
\begin{tabular}{lllll}
CARAT & \# & $[s,r,u]$ & $G(n,i)$ & \\\hline
(5,881)&$1$&[1,0,0]&$G(72,21)$&$(C_3\times (C_3\rtimes C_4))\rtimes C_2$\\
(5,882)&$2$&[1,0,1]&$G(72,46)$&$C_2\times S_3^2$\\
(5,883)&$1$&[1,0,0]&$G(72,23)$&$(C_6\times S_3)\rtimes C_2$\\
(5,884)&$1$&[1,0,0]&$G(72,23)$&$(C_6\times S_3)\rtimes C_2$\\
(5,885)&$1$&[1,0,0]&$G(72,23)$&$(C_6\times S_3)\rtimes C_2$\\
(5,886)&$1$&[1,0,0]&$G(72,23)$&$(C_6\times S_3)\rtimes C_2$\\
(5,887)&$5$&[2,0,3]&$G(72,46)$&$C_2\times S_3^2$\\
\end{tabular}
\end{center}


\end{document}